\documentclass[10pt,leqno]{amsart}
\usepackage{amsmath,amsthm,amsfonts,eucal,eufrak}
\usepackage{latexsym}
\usepackage{amssymb}
\usepackage[dvips]{epsfig}
\usepackage[active]{srcltx}
\renewcommand{\thesection}{\arabic{section}}

\newtheorem{theorem}{Theorem}[section]
\newtheorem*{thma}{Theorem A}
\newtheorem*{thmb}{Theorem B}
\newtheorem*{thmc}{Theorem C}
\newtheorem*{thmd}{Theorem D}
\newtheorem{lemma}[theorem]{Lemma}
\newtheorem{prop}[theorem]{Proposition}
\newtheorem{defi}[theorem]{Definition}

\newtheorem{corollary}[theorem]{Corollary}
\newtheorem{remark}[theorem]{Remark}

\renewcommand{\theequation}{\thesection .\arabic{equation}}
\let\subs\subsection
\renewcommand\subsection{\setcounter{equation}{0}
\gdef\theequation{\thesubsection \arabic{equation}}\subs}
\let\sect\section
\renewcommand\section{\setcounter{equation}{0}
\gdef\theequation{\thesection .\arabic{equation}}\sect}

\newcommand{\cA}{{\mathcal{A}}}

\newcommand{\cD}{{\mathcal{D}}}

\newcommand{\cH}{{\mathcal{H}}}

\newcommand{\cN}{{\mathcal{N}}}
\newcommand{\cR}{{\mathcal{R}}}

\newcommand{\cM}{{\mathcal{M}}}

\newcommand{\cS}{{\mathcal{S}}}

\newcommand{\cK}{{\mathcal{K}}}
\newcommand{\cL}{{\mathcal{L}}}
\newcommand{\cP}{{\mathcal{P}}}

\newcommand{\IC}{{\mathbb{C}}}

\newcommand{\IR}{{\mathbb{R}}}

\newcommand{\IZ}{{\mathbb{Z}}}
\newcommand{\zv}{\IZ^\nu}
\newcommand{\rv}{\IR^\nu}

\newcommand{\uE}{{\underline{E}}}

\newcommand{\be}{\begin{equation}}
\newcommand{\ee}{\end{equation}}
\newcommand{\nn}{\nonumber}

\newcommand{\diam}{\mathop{\rm{diam}}}

\newcommand{\dist}{\mathop{\rm{dist}}}

\newcommand{\Res}{\mathop{\rm{Res}}}
\newcommand{\Ree}{\mathop{\rm{Re}}}

\newcommand{\imm}{\mathop{\rm{Im}}}
\newcommand{\spec}{\mathop{\rm{spec}}}
\newcommand{\sgn}{\mathop{\rm{sgn}}}

\newcommand{\La}{\Lambda}
\newcommand{\la}{\langle}
\newcommand{\ra}{\rangle}
\newcommand{\ve}{\varepsilon}

\newcommand{\hle}{H_{\La,\ve}}
\newcommand{\vp}{\varphi}
\newcommand{\ka}{\kappa}

\def\Ga{\Gamma}

\def\omap{(\omega_1, \omega_2, \dots, \omega_\nu)}

\def\xumap{(x_0, u_0)}

\def\wmap{\omega = (\omega_1,\dots, \omega_\nu)}
\def\zero{{(0)}}

\def\zo{z^{(0)}}
\def\ro{R^{(0)}}

\def\one{{(1)}}
\def\two{{(2)}}

\def\es{{(s)}}
\def\ar{{(r)}}
\def\esone{{(s-1)}}

\def\two{{(2)}}

\begin{document}

\smallskip

\title{On the Inverse Spectral Problem for the Quasi-Periodic Schr\"odinger Equation}

\author{David Damanik}

\address{Department of Mathematics, Rice University, 6100 S. Main St. Houston TX 77005-1892, U.S.A.}

\email{damanik@rice.edu}

\author{Michael Goldstein}

\address{Department of Mathematics, University of Toronto, Bahen Centre, 40 St. George St., Toronto, Ontario, CANADA M5S 2E4}

\email{gold@math.toronto.edu}

\thanks{The first author was partially supported by a Simons Fellowship and NSF grants DMS--0800100 and DMS--1067988. The second author was partially supported by a Guggenheim Fellowship and by an NSERC grant.}

\begin{abstract}
We study the quasi-periodic Schr\"{o}dinger equation
$$
-\psi''(x) + V(x) \psi(x) = E \psi(x), \qquad x \in \IR
$$
in the regime of ``small'' $V$. Let $(E_m',E''_m)$, $m \in \zv$, be the standard labeled gaps in the spectrum. Our main result says that if $E''_m - E'_m \le \ve \exp(-\kappa_0 |m|)$ for all $m \in \zv$, with $\ve$ being small enough, depending on $\kappa_0 > 0$ and the frequency vector involved, then the Fourier coefficients of $V$ obey $|c(m)| \le \ve^{1/2} \exp(-\frac{\kappa_0}{2} |m|)$ for all $m \in \zv$. On the other hand we prove that if $|c(m)| \le \ve \exp(-\kappa_0 |m|)$ with $\ve$ being small enough, depending on $\kappa_0 > 0$ and the frequency vector involved, then $E''_m - E'_m \le 2 \ve \exp(-\frac{\kappa_0}{2} |m|)$.
\end{abstract}

\maketitle

\tableofcontents

\section{Introduction and Statement of the Main Result}\label{sec.1}

In the last thirty five years after the classical pioneering work \cite{DiSi} by Dinaburg and Sinai the theory of quasi-periodic Schr\"odinger equations has been extensively developed. Despite that there are still a number of basic problems which seem to be hard to access. Here are a few such problems:

\textsc{Problem 1}\quad\textit{Consider the  Schr\"odinger equations}

\begin{equation} \label{eq:1-1P}
-\psi''(x) +  [c_1\cos (2\pi x)+c_2\cos(2\pi\alpha x)] \psi(x) = E\psi(x)\ ,\qquad x \in \IR\ ,
\end{equation}
\textit{where $\alpha$ is irrational with ``nice'' Dipohantine properties and $c_1,c_2$ are constants. Describe the eigenfunctions and the instability intervals of the equation.}

\smallskip

\textsc{Problem 2}\quad\textit{Find all functions of the form }
\begin{equation} \label{eq:1.UP}
V(x) = \sum_{n,m \in \IZ}\, c(m,n) e^{2\pi i( m +\alpha n)x}\ ,
\end{equation}
\textit{such that the equation}
\begin{equation} \label{eq:1-1PI}
-\psi''(x) +  V(x) \psi(x) = E\psi(x)
\end{equation}
\textit{has the same instability intervals as equation \eqref{eq:1-1P}.}

\smallskip

\textsc{Problem 3}\quad\textit{Give a sufficient condition for a subset $\mathfrak{S}\subset \IR$ to be the spectrum of the equation \eqref{eq:1-1PI} with some function $V$ as in \eqref{eq:1.UP}.}

\smallskip

\textsc{Problem 4}\quad\textit{Solve the KdV equation }
\begin{equation} \label{eq:1.KdV}
\partial_t u + \partial_x^3 u + u\partial_x u = 0
\end{equation}
\textit{with the initial data}
\begin{equation} \label{eq:1-1KdV}
u(x,0) = c_1 \cos (2\pi x) + c_2 \cos (2 \pi \alpha x).
\end{equation}

\smallskip

Here are two comments regarding these problems.

$(1)$ It is known that for $c_1,c_2$ small, all generalized eigenfunctions are Floquet-like. On the other hand, for $c_1,c_2$ large, there is a collection of exponentially decaying eigenfunctions with eigenvalues which are dense in a Cantor set of positive measure. The problem is to find a method that will work for all values of $c_1,c_2$. In the discrete case, Avila has recently made significant progress in this direction in a series of papers \cite{Av1, Av2, Av3}.

$(2)$ We state the problems for the function $c_1 \cos (2\pi x) + c_2 \cos (2\pi\alpha x)$ just for the sake of simplicity of the statement. In fact the problems are as hard for this function as for any quasi-periodic function
\begin{equation} \label{eq:1.UU}
V(x) = \sum_{n \in \zv}\, c(n) e (x n \omega),
\end{equation}
$\wmap \in \rv$, $n\omega = \sum n_j \omega_j$, $e(x) := \exp(2\pi ix)$. In this work we study the latter case.

\bigskip

Let us state the main results of this work. We consider the Schr\"odinger equation
\begin{equation} \label{eq:1-1}
-\psi''(x) +  V(x) \psi(x) = E \psi(x), \qquad x \in \IR,
\end{equation}
where $V(x)$ is a real quasi-periodic function as in \eqref{eq:1.UU}. We assume that the Fourier coefficients $c(m)$ obey
\begin{equation}\label{eq:1-fouriercoeff}
|c(m)| \le \ve \exp(-\kappa_0 |m|)
\end{equation}
with $\ve$ being small. We assume that the vector $\omega$ satisfies the following Diophantine condition:
\begin{equation}\label{eq:1-8}
|n\omega| \ge a_0 |n|^{-b_0}, \quad n \in \zv \setminus \{0\}
\end{equation}
with some $0 < a_0 < 1$, $\nu-1 < b_0 < \infty$. Set
\begin{equation}\label{eq:1K.1}
\begin{split}
k_n = -n\omega/2, \quad n \in \zv \setminus \{0\}, \quad \mathcal{K}(\omega) = \{ k_n : n \in \zv \setminus \{0\} \}, \\
\mathfrak{J}_n = ( k_n - \delta(n), k_n + \delta(n) ), \quad \delta(n) = a_0 (1 + |n|)^{-b_0-3}, \quad n \in \zv \setminus \{0\}, \\
\mathfrak{R}(k) = \{ n \in \zv \setminus \{0\} : k \in \mathfrak{J}_n \}, \quad \mathfrak{G} = \{ k : |\mathfrak{R}(k)| < \infty \},
\end{split}
\end{equation}
where $a_0,b_0$ are as in the Diophantine condition \eqref{eq:1-8}. Let $k \in \mathfrak{G}$ be such that $|\mathfrak{R}(k)| > 0$. Due to the Diophantine condition, one can enumerate the points of $\mathfrak{R}(k)$ as $n^{(\ell)}(k)$, $\ell = 0, \dots, \ell(k)$, $1 + \ell(k) = |\mathfrak{R}(k)|$, so that $|n^{(\ell)}(k)| < |n^{(\ell+1)}(k)|$; see Lemma~\ref{lem:10resetdiscr1} in Section~\ref{sec.10}. Set
\begin{equation}\label{eq:1mjdefi}
\begin{split}
T_{m}(n) = m - n ,\quad m, n \in \mathbb{Z}^\nu, \\
\mathfrak{m}^{(0)}(k) = \{ 0, n^{(0)}(k) \}, \quad \mathfrak{m}^{(\ell)}(k) = \mathfrak{m}^{(\ell-1)}(k) \cup T_{n^{(\ell)}(k)}(\mathfrak{m}^{(\ell-1)}(k)), \quad \ell = 1, \dots, \ell(k).
\end{split}
\end{equation}

\begin{thma}
There exists $\ve_0 = \ve_0(\ka_0, a_0, b_0) > 0$ such that if $\ve \le \ve_0$, then for any $k \in \mathfrak{G} \setminus \frac{\omega}{2}(\zv\setminus \{0\})$, there exist $E(k) \in \mathbb{R}$ and $\vp(k) := (\vp(n;k))_{n \in \zv}$ such that the following conditions hold:

$(a)$ $\vp(0 ;k) = 1$,
\begin{equation} \label{eq:1-17evdecay1A}
\begin{split}
|\vp(n ;k)| \le \ve^{1/2} \sum_{m \in \mathfrak{m}^{(\ell)}} \exp \Big( -\frac{7}{8} \kappa_0 |n-m| \Big), \quad \text{ $n \notin \mathfrak{m}^{(\ell(k))}(k)$}, \\
|\vp(m;k)| \le 2, \quad \text{for any $m \in \mathfrak{m}^{(\ell(k))}(k)$.}
\end{split}
\end{equation}

$(b)$ The function
$$
\psi(k, x) = \sum_{n \in \zv} \vp(n; k) e(x (n\omega+k))
$$
is well-defined and obeys equation \eqref{eq:1-1} with $E = E(k)$, that is,
\begin{equation}\label{eq:1.sco}
H \psi(k,x) \equiv  - \psi''(k,x) + V(x) \psi(k,x) = E(k) \psi(k,x).
\end{equation}

$(c)$
$$
E(k) = E(-k), \quad \varphi(n;-k) = \overline{\varphi(-n; k)},\quad \psi(-k, x)=\overline{\psi(k, x)},
$$
\begin{equation}\label{eq:1Ekk1EGT11}
\begin{split}
(k^0)^2 (k - k_1)^2  < E(k) - E(k_1) < 2k (k - k_1) + 2 \ve \sum_{k_1 < k_{n} < k} \delta(n), \quad 0 < k - k_1 < 1/4, \; k_1 > 0,
\end{split}
\end{equation}
where $k^\zero:=\min(\ve_0, k/1024)$.

$(d)$ The spectrum of $H$ consists of the following set,
$$
\cS = [E(0) , \infty) \setminus \bigcup_{m \in \zv \setminus \{0\} : E^-( k_m) < E^+( k_m)} (E^-( k_m), E^+( k_m)),
$$
where
$$
E^\pm(k_m) = \lim_{k \to k_m \pm 0, \; k \in \mathfrak{G} \setminus \mathcal{K}(\omega)} E(k), \quad \text{ for $k_m>0$.}
$$
\end{thma}

One of the central results of the current work is the following:

\begin{thmb}
$(1)$ The gaps $(E^-(k_m), E^+(k_m))$ in Theorem~A obey $E^+(k_m) - E^-(k_m) \le 2 \ve \exp(-\frac{\kappa_0}{2} |m|)$.

$(2)$ Using the notation from Theorem~A, there exists $\ve^\zero$ such that if the gaps $(E^-(k_m), E^+(k_m))$ obey $E^+(k_m) - E^-(k_m) \le \ve \exp(-\kappa |m|)$ with $\ve < \ve^\zero$, $\kappa>4\kappa_0$, then, in fact, the Fourier coefficients $c(m)$ obey $|c(m)| \le \ve^{1/2} \exp(-\frac{\kappa}{2} |m|)$.
\end{thmb}

\begin{remark} 
$(1)$ In a companion paper, \cite{DG}, we apply Theorem~B to establish the existence of a global solution of the KdV equation
\begin{equation} \label{eq:1.KdVP1}
\partial_t u +\partial_x^3 u + u\partial_x u=0
\end{equation}
with small quasi-periodic initial data. This application is \textbf{the main objective of Theorem B}. We would like to explain in this remark why the estimate in part $2$ of Theorem~B is crucial for the existence of a global solution of \eqref{eq:1.KdVP1} with quasi-periodic data. Recall the following fundamental result by P.~Lax \cite{Lax}: Let $u(t,x)$ be a function defined for $0 \le t < t_0$, $x \in \mathbb{R}$ such that $\partial^\alpha_x u$ exist and are continuous and bounded in both variables for $0 \le \alpha \le 3$. Assume that $u$ obeys equation \eqref{eq:1.KdVP1}. Consider the Schr\"odinger operators
\begin{equation} \label{eq:1-1u}
[H_t \psi](x) = -\psi''(x) +  u(t,x) \psi(x), \qquad x \in \IR,
\end{equation}
Then $\sigma(H_t) = \sigma(H_0)$ for all $t$.  Assume that
\begin{equation} \label{eq:1-1KdV1}
u(t,x) = \sum_{n \in \zv} \, c(t,n) e^{i x n \omega},
\end{equation}
with
$$
|c(t,n)| \le \ve \exp(-\kappa_1 |n|) \text{for all $0 \le t < t_0$},
$$
where $\ve \le \ve_0(a_0,b_0,\kappa_1)$. Assume in addition that for $t = 0$, the estimates are better: $|c(0,n)| \le \ve \exp(-\kappa_0 |n|)$, $\ve \le \ve^\zero(a_0,b_0,\kappa_1)$. Applying Theorems~A and B, one concludes that in fact $|c(t,n)| \le \ve^{1/2} \exp(-\frac{\kappa_0}{2} |n|)$. In other words, there is no blow up of the estimates for the Fourier coefficients of the solution. Thus, due to Theorems~A,B, to prove the existence of a \textbf{global} solution of the KdV equation \eqref{eq:1.KdVP1} with quasi-periodic initial data
\begin{equation} \label{eq:1-1KdV1A}
u_0(x) = \sum_{n \in \zv}\, c_0(n) e^{i x n \omega},
\end{equation}
with $|c_0(n)| \le \ve \exp(-\kappa_0 |n|)$, $\ve \le \ve^\zero$, one only has to establish the existence of a \textbf{local} solution.

$(2)$ An estimate similar in spirit to the one in the first part of Theorem~B was established by Hadj Amor \cite{HA}.

$(3)$ The problem of ``keeping the exponential decay of the Fourier coefficients in check'' is also well known in the KAM theory of perturbations of integrable PDE's; see for instance the paper \cite{EK} by Eliasson and Kuksin on periodic non-linear Schr\"{o}dinger equations.
\end{remark}

The existence of solutions $\psi(k,x)$ as in Theorem~A was discovered for a large set of $k$'s in the paper \cite{DiSi} by Dinaburg and Sinai. Such solutions are called Floquet-Bloch or just Floquet solutions and the parameter $k$ is called quasi-momentum. In \cite{El}, Eliasson proved the existence of Floquet solutions for $k \in \mathfrak{G}$ and also the fact that the spectrum is purely absolutely continuous.

Our approach is completely different from Eliasson's approach. \emph{We prove exponential localization of the eigenfunctions of the dual operator.} The duality underlying this approach is called Aubry duality. In \cite{BoJi}, Bourgain and Jitomirskaya used this approach to study discrete quasi-periodic Schr\"{o}dinger operators for small values of the coupling constant; see also \cite{Bo}. Let us introduce the dual operators for \eqref{eq:1-1}. Given $k \in \IR$ and a function $\vp(n)$, $n \in \zv$ such that $|\varphi(n)| \le C_\varphi|n|^{-\nu-1}$, where $C_\varphi$ is a constant, set
\begin{equation} \label{eq:5-5}
y_{\vp, k}(x) = \sum_{n \in \zv}\, \vp(n) e \bigl((n\omega + k)x\bigr).
\end{equation}
The function $y_{\vp, k}(x)$ satisfies equation \eqref{eq:1-1} if and only if
\begin{equation} \label{eq:1-6}
(2\pi)^2 (n\omega + k)^2 \vp(n) +  \sum_{m \in \zv}\, c(n - m)\vp(m) = E\vp(n)
\end{equation}
for any $n \in \zv$.  Put
\begin{equation} \label{eq:1-7}
\begin{split}
h(m,n; k) & = (2\pi)^2(m\omega + k)^2 \quad \text{if } m = n, \\
h(m,n; k) & = c(n-m) \quad \text{if } m \neq n.
\end{split}
\end{equation}
Consider $H_{ k} = \bigl(h(m, n; k)\bigr)_{m, n \in \zv}$.

\begin{thmc}
There exists $\ve_0 = \ve_0(\ka_0, a_0, b_0) > 0$ such that for $\ve \le \ve_0$ and any $k \in \mathfrak{G} \setminus \frac{\omega}{2}\zv$, there exists $E(k) \in \mathbb{R}$ and $\vp(k) := (\vp(n;k))_{n \in \zv}$ such that the following conditions hold:

$(1)$ $\vp(0 ;k) = 1$,
\begin{equation} \label{eq:1-17evdecay1}
\begin{split}
|\vp(n;k)| \le \ve^{1/2} \sum_{m \in \mathfrak{m}^{(\ell)}} \exp \Big( -\frac{7}{8} \kappa_0 |n-m| \Big), \quad \text{ $n \notin \mathfrak{m}^{(\ell(k))}(k)$}, \\
|\vp(m;k)| \le 2, \quad \text{for any $m \in \mathfrak{m}^{(\ell(k))}(k)$,}
\end{split}
\end{equation}
\begin{equation} \label{eq:1philim}
H_{k} \vp(k) = E(k) \vp(k).
\end{equation}

$(2)$
\begin{equation}\label{eq:1EsymmetryT}
E(k) = E(-k), \quad \vp(n ;-k) = \overline{\vp(-n ;k)},
\end{equation}
\begin{equation}\label{eq:1Ekk1EGT}
\begin{split}
(k^\zero)^2 (k - k_1)^2  < E(k) - E(k_1) < 2k (k - k_1) + 2 \ve \sum_{k_1 < k_{n} < k}(\delta(n))^{1/8} , \quad \quad 0 < k - k_1 < 1/4, \; k_1 > 0,
\end{split}
\end{equation}
where $k^\zero := \min(\ve_0, k/1024)$.

$(3)$ Set $E^\pm(k_{n}) = \lim_{k \to k_n \pm 0, \; k \in \mathfrak{G}} E( \ve,k)$. Assume that $E^+(k_{n^\zero}) > E^-(k_{n^\zero})$. Let $E^-(k_{n^\zero}) < E < E^+(k_{n^\zero})$. Then $(E - H_{k})$ is invertible for every $k$.
\end{thmc}

Let us give a short description of our method and the central technical difficulty we resolve. The proof of Theorem~C is built upon an abstract multi-scale analysis scheme. We estimate the Green function $(E - H_\La)(m,n)$ of the  matrix $H_\La := \bigl( h(m,n;k) \bigr)_{m,n \in \La}$, $\La \subset \zv$ moving up on the ``size scale'' of $\La$. This approach was introduced in the theory of Anderson localization in the fundamental work \cite{FrSp} of Fr\"ohlich and Spencer on random potentials and later by Fr\"ohlich, Spencer and Wittwer in \cite{FSW} for quasi-periodic potentials. Our multi-scale scheme is based on the Schur complement formula:
\begin{equation}\label{eq:1schurfor}
\begin{bmatrix} \cH_{1} & \Gamma_{1,2} \\[5pt] \Gamma_{2,1} & \cH_{2}\end{bmatrix}^{-1} = \begin{bmatrix} \cH_1^{-1} + \cH_1^{-1} \Gamma_{1,2} \tilde H_2^{-1} \Gamma_{2,1} \cH_1^{-1} & - \cH_1^{-1} \Gamma_{1,2} \tilde H_2^{-1} \\[5pt] -\tilde H_2^{-1} \Gamma_{2,1} \cH_1^{-1} & \tilde H_2^{-1} \end{bmatrix},
\end{equation}
with
\begin{equation}\label{eq:1schurforH2}
\tilde H_2^{-1} = (\cH_2 - \Gamma_{2,1} \cH_1^{-1} \Gamma_{1,2})^{-1}.
\end{equation}
The main piece here is $\tilde H_2^{-1}$. \textit{The iteration of \eqref{eq:1schurforH2} over the scales builds up a ``continued-fraction-function'' of the spectral parameter $E$ and the quasi-momentum $k$}. To estimate $\tilde H_2^{-1}$ on a given scale, say $s$, one has to study the roots of the determinant of $\cH_2 - \Gamma_{2,1} \cH_1^{-1} \Gamma_{1,2}$ which is the previous continued-fraction-function. These roots are close to $E^{(s-1)}_{\La'}(k)$ -- the eigenvalues of the matrix of the previous scale set $\La'$ parameterized against $k$. The major problem here is that $E^{(s-1)}_{\La'}(k)$ and $E^{(s-1)}_{\La''}(k)$ can be ``extremely'' close for a finite (if $k \in \mathfrak{G}$), but large number of times. These are the so-called essential resonances. The eigenfunction $\varphi(n;k)$ in fact ``typically'' assumes values $\approx 1$ for all $n \in \mathfrak{m}^{(0)}(k)$; see\eqref{eq:1mjdefi}. The sets ``around'' $n \in \mathfrak{m}^{(0)}(k)$ produce these resonance effects. This fact gives an idea of the complexity of the central technical problem one faces in this approach. The advantage of this approach is that it eventually gives a system of equations relating the gaps in the spectrum and the Fourier coefficients. The central technical tool we develop to resolve the resonance problem consists of ``continued-fraction-functions'' and their roots. This is done in Section~\ref{sec.4}. To give the reader an idea what this is about, consider the problem for the simplest ``continued-fraction-function'':
\begin{equation} \label{eq:1-1fractios1}
E - a_1(\ve,k,E) - \frac{b(\ve,k,E)^2}{u - a_2(\ve,k,E)} = 0.
\end{equation}
The new variable $\ve$ is introduced here by considering $\ve c(n)$ instead of $c(n)$. This variable plays a crucial role since we build the solutions via analytic continuation in $\ve$, starting at $\ve=0$. Note that the fact that the numerator $b^2$ here is non-negative reflects the self-adjointness of the problem, which is also crucial for the derivation. Technically, the problem here is that $a_1$ and $a_2$ can be arbitrarily close due to resonances. A direct application of the implicit function theorem to
\begin{equation} \label{eq:1-1fract2}
\chi(\ve,k,E) := \bigl(E - a_1(\ve,k,E)\bigr)\bigl(E - a_2(\ve,k,E)\bigr) - b(\ve,k,E)^2 = 0
\end{equation}
fails ($\partial_E \chi$ may have zeros). What comes to the rescue is that \textit{the symmetries in the structure of $H_\La$, with $\La$ built appropriately, allow for the comparison}
\begin{equation} \label{eq:1-1frac3}
a_1(\ve,k,E) > a_2(\ve,k,E)
\end{equation}
\textit{for all values of $\ve \in (-\ve_0,\ve_0)$ and for $k$, $E$ close to the ones in question}. Due to this fact, one has two solutions $E^+(\ve,k) > E^-(\ve,k)$. For $k = -\frac{m\omega}{2}$, these are very close to the two edges of the corresponding gap. One of the crucial estimates we develop says that the margins  $E^{(s_2)}(k) - E^{(s_1)}(k)$ can be estimated via the quantities $|k + \frac{m\omega}{2}|$. \textit{The symmetries in the structure of $H_\La$ with $\La$ built appropriately play a crucial role in this}. Let us mention here that the next level ``continued-fraction-functions'' look as follows,
\begin{equation} \label{eq:1-1fractios3}
f = f_1 - \frac{b_1^2}{f_2} ,
\end{equation}
where $f_1, f_2$ are like in \eqref{eq:1-1fractios1}. We are interested in the solution of the equation
$f=0$.
An important detail here is that although $f_1,f_2$ are assumed to be ``small on the next scale,'' their derivatives are of magnitude $\sim 1$. This accommodates the above mentioned fact that the eigenfunction $\varphi(n;k)$ assumes values $\sim 1$ at all resonant points involved. In general the construction iterates a large number of times.

Let us say a few words about these symmetries. Given $k$, we define an increasing sequence $\La^\es_k$ of subsets of $\zv$, $\bigcup_s \La^\es = \zv$, which allows us to analyze inductively the eigenvalues $E(\La^\es_k,k)$ and the eigenvectors. The construction of the sets $\La^\es_k$ requires involved combinatorial arguments. The set $\La^\es_k$  is a relatively ``small'' perturbation of a union of two ``large'' cubes, one centered at the origin and another at $n^{(\ell)}(k)$; see \eqref{eq:1K.1}. The boundary of the set is of ``fractal nature'' built on the scale basis. The purpose of this ``fractal'' boundary is as follows. We need the set $\La^\es_k$ to be invariant under the map $T(n) = n^{(\ell)}(k) - n$, and at the same time we want the boundary $\partial \La^\es_k$ to avoid each subset $m + \La^{(s')}_{k+m\omega}$ with $s' < s$ and with $|E(\La^{(s')}_{k+m\omega},k+m\omega) - E(\La^\esone_k,k)|$ being ``small.''

Finally, let us say a few words about the structure of the work. First of all we split the technical difficulties into two major parts. In the first one, we develop a general theory of matrices which by definition have the needed structures. These matrices do not depend on the quasi-momentum $k$. We start with a general multi-scale analysis scheme and then inductively introduce more and more complex matrices under consideration. This is done in Sections~\ref{sec.2}--\ref{sec.3} and \ref{sec.5}--\ref{sec.6}. As already mentioned, in Section~\ref{sec.4} we develop the necessary theory of ``continued-fraction-functions.'' In the second part, which consists of Sections~\ref{sec.7}--\ref{sec.10}, we verify that the matrices dual to the quasi-periodic Schr\"{o}dinger equation, with  appropriate $\La^\es_k$, fit into the definitions from Sections~\ref{sec.2}--\ref{sec.6}. Finally, in Section~\ref{sec.11} we prove the main theorems.

\begin{remark}
The fact that our presentation separates the general theory from the application to small quasi-periodic potentials with Diophantine frequency vector also has the additional benefit that in subsequent applications of the general theory, one merely needs to verify all its necessary assumptions in a given situation. We envision a number of additional applications of the general theory, such as an extension of the quasi-periodic results beyond the case of small coupling, and more generally a version of them for suitable non-zero background potentials. We intend to address these additional applications in future works.
\end{remark}

\setcounter{section}{1}

\section{A General Multi-Scale Analysis Scheme Based on the Schur Complement Formula}\label{sec.2}

Let $\La \subseteq \IZ^\nu$ and let $\mathcal{H} = (\mathcal{H}(m,n))_{m,n \in \La}$ be an arbitrary matrix. For $\La' \subset \La$, denote by $P_{\Lambda'}$ the orthogonal projection onto the subspace $\IC^{\Lambda'}$ of all functions $\psi = \left\{\psi(n) : n \in \IZ^\nu \right\}$ vanishing off $\Lambda'$. The restriction of $\mathcal{H}$ to $\Lambda'$ is the operator $\mathcal{H}_{\La'}: \IC^{\Lambda'} \to \IC^{\Lambda'}$,
$$
\cH_{\La'} := P_{\Lambda'} \cH P_{\Lambda'}.
$$

Let $\Lambda_0 \subset \Lambda$ be an arbitrary subset and set $\Lambda_1 = \Lambda \setminus \Lambda_0$. Then,
$$
\cH = P_{\Lambda_0} \cH_{\La} P_{\Lambda_0} + P_{\Lambda_1} \cH_{\La} P_{\Lambda_1} + P_{\Lambda_0} \cH_{\La} P_{\Lambda_1} + P_{\Lambda_1} \cH_{\La} P_{\Lambda_0}.
$$
By viewing $\IC^\Lambda$ as $\IC^{\Lambda_1} \oplus \IC^{\Lambda_0}$, one has the following matrix representation of $H_{\La}$,
\be\label{eq:2schurfor1b}
\cH_{\La} = \begin{bmatrix} \cH_{\Lambda_1} & \Gamma_{\Lambda_1, \Lambda_0} \\[5pt] \Gamma_{\Lambda_0, \Lambda_1} & \cH_{\Lambda_0} \end{bmatrix},
\end{equation}
where
$$
\Gamma_{\Lambda_i, \Lambda_j} (k, \ell) = \cH(k,\ell),\quad k\in \La_i,\ell\in \La_j.
$$
Recall the following fact, known as the \emph{Schur complement formula}.

\begin{lemma}\label{lem:2schur} Let
\be\label{eq:2schurfor1a}
\cH = \begin{bmatrix} \cH_{1} & \Gamma_{1,2} \\[5pt] \Gamma_{2,1} & \cH_{2}\end{bmatrix},
\end{equation}
where $\cH_j$ is an $(N_j \times N_j)$-matrix, $j = 1, 2$, and $\Gamma_{i,j}$ is an $(N_i \times N_j)$-matrix. Assume that $\cH_1$ is invertible. Let $\tilde H_2 = \cH_2 - \Gamma_{2,1} \cH_1^{-1} \Gamma_{1,2}$. Then, $\cH$ is invertible if and only if $\tilde H_2$ is invertible; and in this case, we have
\begin{equation}\label{eq:2schurfor}
\cH^{-1} = \begin{bmatrix} \cH_1^{-1} + \cH_1^{-1} \Gamma_{1,2} \tilde H_2^{-1} \Gamma_{2,1} \cH_1^{-1} & - \cH_1^{-1}
\Gamma_{1,2} \tilde H_2^{-1} \\[5pt] -\tilde H_2^{-1} \Gamma_{2,1} \cH_1^{-1} & \tilde H_2^{-1} \end{bmatrix}.
\end{equation}
\end{lemma}

\begin{defi}\label{def:aux1}
\begin{itemize}
\item[(1)] For each $m$, let $\gamma(m) := (m)$ be the sequence which consists of one point $m$. Set $\Ga(m, m; 1, \La) := \{\gamma(m)\}$,
$\Ga(m,n;1,\La) := \emptyset$ for $n \not= m$,
\begin{equation}\label{eq:auxtraject}
\begin{split}
\Ga(k,\La) = \{ \gamma = (n_1,\dots ,n_k) : n_j \in \La, \quad n_{j+1} \neq n_j \}, \; k \ge 2, \\
\Ga(m,n;k,\La) = \{ \gamma \in \Ga(k,\La), \quad n_1 = m, n_k = n \}, \; m, n \in \La, \; k \ge 2, \\
\Ga_1(m,n;\La) = \bigcup_{k \ge 1} \Ga(m,n;k,\La), \quad \Ga_1(\La) = \bigcup_{m, n \in \La} \Ga_1(m,n;\La).
\end{split}
\end{equation}

Let $\gamma = (n_1, \dots, n_k)$, $\gamma' = (n'_1, \dots, n'_\ell)$, $n_i, n'_j \in \IZ^\nu$. Set
\begin{equation}\label{eq:auxgammagamma}
\gamma \cup \gamma' := \begin{cases} (n_1, \dots , n_k, n'_1, \dots, n'_\ell) & \text{if $n_k\neq n'_1$}, \\ (n_1, \dots, n_k, n'_2, \dots, n'_\ell) & \text{if $n_k = n'_1$}. \end{cases}
\end{equation}

\item[(2)] Let $w(m,n)$, $D(m)$ be functions obeying $w(m,n)\ge 0$, $D(m) \ge 1$, $m,n \in \La$. For $\gamma = (n_1, \dots, n_k)$, set
\begin{equation}\label{eq:auxtrajectweightO}
w_{D,\kappa_0} (\gamma) := \Big[ \prod_{1 \le j \le k-1} w(n_j,n_{j+1}) \Big] \exp \Big( \sum_{1 \le j \le k} D(n_j) \Big).
\end{equation}
Wherever we apply $w_{D,\kappa_0} (\gamma_1\cup\gamma_2)$, we assume that we are in the second case in \eqref{eq:auxgammagamma}. For that matter,  $w_{D,\kappa_0} (\gamma_1\cup\gamma_2)= w_{D,\kappa_0} (\gamma_1) w_{D,\kappa_0} (\gamma_2)$.

Let $0 < \kappa_0 < 1$. We always assume that $w(m,m)=1$ and
\begin{equation}\label{eq:2.weighdecaycond}
w(m,n)\le \exp(-\kappa_0|m-n|),
\end{equation}
and  we set
\begin{equation}\label{eq:auxtrajectweight}
\begin{split}
W_{D,\kappa_0} (\gamma) := \exp \Big( -\kappa_0 \|\gamma\| + \sum_{1 \le j \le k} D(n_j) \Big), \\
\|\gamma\| := \sum_{1 \le i \le k-1} |n_i - n_{i+1}|, \quad \bar D(\gamma) := \max_j D(n_j).
\end{split}
\end{equation}
Here, $\|\gamma\| = 0$ if $k = 1$. Obviously, $w_{D,\kappa_0} (\gamma)\le W_{D,\kappa_0} (\gamma)$.

\item[(3)] Let $T \ge 8$. We say that $\gamma = (n_1, \dots, n_{k})$, $n_j \in \La$, $k \ge 1$ belongs to $\Ga_{D, T, \kappa_0} (n_1, n_{k}; k, \La)$ if the following condition holds:
\begin{equation}\label{eq:auxtrajectweight5}
\min (D(n_{i}), D(n_{j})) \le T \| (n_{i}, \dots, n_{j}) \|^{1/5} \quad \text{for any $i < j$ such that $\min (D(n_{i}), D(n_{j})) \ge 4 T \kappa_0^{-1}$}.
\end{equation}
Note that $\Ga_{D, T, \kappa_0} (n_1, n_{1}; 1, \La) = \{ (n_1) \}$. Set $\Ga_{D, T, \kappa_0} (m, n; \La) = \bigcup_k \Ga_{D, T, \kappa_0} (m, n; k, \La)$, $\Ga_{D, T, \kappa_0} (\La) = \bigcup_{m,n} \Ga_{D, T, \kappa_0} (m, n; \La)$.

\item[(4)] Set
\begin{equation}\label{eq:auxtrajectweight1}
\begin{split}
s_{D, T, \kappa_0; k, \La} (m, n) = \sum_{\gamma \in \Ga_{D, T, \kappa_0} (m, n; k, \La)} w_{D, \kappa_0} (\gamma),\\
S_{D, T, \kappa_0; k, \La} (m, n) = \sum_{\gamma \in \Ga_{D, T, \kappa_0} (m, n; k, \La)} W_{D, \kappa_0} (\gamma).
\end{split}
\end{equation}
Note that $s_{D, T, \kappa_0; 1, \La} (m, m) = S_{D, T, \kappa_0; 1, \La} (m, m)=\exp(D(m))$.

\item[(5)] Let $\La \subset \bar \La\subset \IZ^\nu$. Set $\mu_{\La,\bar\La}(m) := \dist (m, \bar\La \setminus \La)$. We say that the function $D(m)$, $m \in \La$ belongs to $\mathcal{G}_{\La,\bar\La, T, \kappa_0}$ if the following condition holds:
\begin{equation}\label{eq:auxDcond}
D(m) \le T \mu_{\La,\bar\La}(m)^{1/5} \quad \text{for any $m$ such that $D(m) \ge 4 T \kappa_0^{-1}$}.
\end{equation}

\item[(6)] Let $D \in \mathcal{G}_{\La,\bar\La, T, \kappa_0}$. We say that $\gamma = (n_1, \dots, n_{k})$, $n_j \in \La$, $k \ge 1$ belongs to $\Ga_{D, T, \kappa_0} (n_1, n_{k}; k, \La, \mathfrak{R})$ if the following conditions hold:
\begin{equation}\label{eq:auxtrajectweight5NNNNN}
\begin{split}
\min (D(n_{i}), D(n_{j})) \le T \| (n_{i}, \dots, n_{j}) \|^{1/5} \\
\text{for any $i < j$ such that $\min (D(n_{i}), D(n_{j})) \ge 4 T \kappa_0^{-1}$}, \quad \text{unless $j = i + 1$}.
\end{split}
\end{equation}
Moreover,
\begin{equation}\label{eq:auxtrajectweight5NNNNN1}
\begin{split}
\text{if $\min (D(n_{i}), D(n_{i+1})) > T |(n_{i} - n_{i+1})|^{1/5}$} \quad \text{for some $i$, then } \\
\min (D(n_{j'}), D(n_{i})) \le T \| (n_{j'}, \dots, n_{i}) \|^{1/5}, \quad \min (D(n_{i}), D(n_{j''})) \le T \| (n_{i}, \dots, n_{j''}) \|^{1/5},\\
\min (D(n_{j'}), D(n_{i+1})) \le T \|(n_{j'}, \dots, n_{i+1}) \|^{1/5}, \quad \min (D(n_{i+1}), D(n_{j''})) \le T \| (n_{i+1}, \dots, n_{j''}) \|^{1/5},\\
\text{for any $j' < i < i+1 < j''$.}
\end{split}
\end{equation}
Set $\Ga_{D, T, \kappa_0} (m, n; \La, \mathfrak{R}) = \bigcup_{k} \Ga_{D, T, \kappa_0} (m, n; k, \La, \mathfrak{R})$, $\Ga_{D, T, \kappa_0} (\La, \mathfrak{R}) = \bigcup_{m,n} \Ga_{D, T, \kappa_0} (m, n; \La, \mathfrak{R})$. Given $\gamma = (n_1, \dots, n_{k}) \in \Ga_{D, T, \kappa_0} (n_1, n_{k}; k, \La, \mathfrak{R}) \setminus \Ga_{D, T, \kappa_0} (n_1, n_{k}; k, \La)$, we denote by $\cP(\gamma)$ the set of all $i$ for which $\min (D(n_{i}), D(n_{i+1})) \ge T |(n_{i} - n_{i+1})|^{1/5}$. Set
\begin{equation}\label{eq:auxtrajectweight111111}
\begin{split}
s_{D, T, \kappa_0; k, \La, \mathfrak{R}} (m,n) = \sum_{\gamma \in \Ga_{D, T, \kappa_0} (m, n; k, \La, \mathfrak{R})} w_{D,\kappa_0} (\gamma),\\
S_{D, T, \kappa_0; k, \La, \mathfrak{R}} (m,n) = \sum_{\gamma \in \Ga_{D, T, \kappa_0} (m, n; k, \La, \mathfrak{R})} W_{D,\kappa_0} (\gamma).
\end{split}
\end{equation}

\end{itemize}
\end{defi}

\begin{remark}\label{rem:2withRnoR}$(1)$. \textbf{Everywhere in this section the set} $\bar \La$ \textbf{is fixed. For this reason we suppress } $\bar \La$ \textbf{from the notation. We always assume in all statements that each subset $\La \subset \IZ^\nu$ under consideration is
a subset of} $\bar \La$. \textbf{The complement $\La^c$ always means} $\bar \La \setminus \La$. When we apply the statements from the current section in Sections~\ref{sec.3} and \ref{sec.5}, we will assume $\bar\La = \IZ^\nu$. On the other hand, we will use different subsets in the role of $\bar\La$ starting from Section~\ref{sec.6}. Note for that matter that $\mathcal{G}_{\La,\bar\La, T, \kappa_0} \subset \mathcal{G}_{\La,\bar\La_1, T, \kappa_0}$ if $\La\subset \bar \La_1\subset \bar \La$.

$(2)$ The sets of trajectories $\Ga_{D, T, \kappa_0} (n_1, n_{k}; k, \La)$, $\Ga_{D, T, \kappa_0} (n_1, n_{k}; k, \La, \mathfrak{R})$ are designed to estimate the inverse for two different types of matrices we study in this work. We introduce these two types of matrices in Section~\ref{sec.3} and Section~\ref{sec.5}, respectively. We estimate the inverse via the functions
\begin{equation}\label{eq:auxtrajectweightsumdef}
\begin{split}
s_{D, T, \kappa_0, \ve_0; \La} (m,n) := \sum_{k \ge 1} \ve_0^{k-1} s_{D, T, \kappa_0; k, \La} (m,n), \\
s_{D, T, \kappa_0, \ve_0; \La, \mathfrak{R}} (m,n) := \sum_{k \ge 1} \ve_0^{k-1} s_{D, T, \kappa_0; k, \La, \mathfrak{R}} (m,n), \\
S_{D, T, \kappa_0, \ve_0; \La} (m,n) := \sum_{k \ge 1} \ve_0^{k-1} S_{D, T, \kappa_0; k, \La} (m,n), \\
S_{D, T, \kappa_0, \ve_0; \La, \mathfrak{R}} (m,n) := \sum_{k \ge 1} \ve_0^{k-1} S_{D, T, \kappa_0; k, \La, \mathfrak{R}} (m,n),
\end{split}
\end{equation}
respectively. One of the important properties of these functions is ``functoriality'' with respect to the Schur complement formula. The precise meaning of this ``functoriality'' is formulated in Lemma~\ref{lem:aux5}. Its derivation is based on the mutiplicativity property of the functions $w_{D,\kappa_0}(\gamma)$, $W_{D,\kappa_0}(\gamma)$ with respect to the operation $\gamma_1 \cup \gamma_2$.

$(3)$ In Sections~\ref{sec.3}--\ref{sec.10} we will use the function  $W_{D,\kappa_0}(\gamma)$ and the corresponding sums. We will use the function $w_{D,\kappa_0}(\gamma)$ and the corresponding sums in Section~\ref{sec.11}.
\end{remark}

\begin{lemma}\label{lem:auxweight}
Let $\gamma = (n_1, \dots, n_{k}) \in \Ga_{D, T, \kappa_0} (n_1, n_{k}; k, \La, \mathfrak{R})$. Set $M = 4 T \kappa_0^{-1}$, $t_D(\gamma) := \frac{\log \bar D(\gamma)}{\log M}$, $\vartheta_t = \sum_{0 < s \le t} 2^{-5s}$.

If $t_D(\gamma) \le 5$, then $W_{D,\kappa_0}(\gamma) \le \exp(-\kappa_0 \|\gamma\| + k M^5)$. Otherwise, with $\ell$ chosen such that $D(n_\ell) = \bar D(\gamma)$, we have
\begin{equation}\label{eq:auxtrajectweight2}
\begin{split}
W_{D,\kappa_0}(\gamma) \le \begin{cases} e^{-\kappa_0 (1 - \vartheta_{t_D(\gamma)}) \| \gamma \| + \bar D(\gamma)} & \text{if  $\ell, \ell-1 \notin \cP(\gamma)$ and $\max_{j \neq \ell} D(n_j) < \frac{D(n_{\ell})}{M^2}$}, \\
e^{-\kappa_0 (1 - \vartheta_{t_D(\gamma)+1}) \| \gamma \|} & \text{ if $\ell, \ell-1 \notin \cP(\gamma)$ and  $\max_{j \neq \ell} D(n_j) \ge \frac{D(n_{\ell})}{M^2}$}, \\
e^{-\kappa_0 (1 - \vartheta_{t_D(\gamma)}) \| \gamma \| + 2 \bar D(\gamma)} & \text{if $\ell \in \cP(\gamma)$ or $\ell-1 \in \cP(\gamma)$ and $\max_{j \notin \{\ell-1, \ell\}} D(n_j) < \frac{D(n_{\ell})}{M^2}$}, \\
e^{-\kappa_0 (1 - \vartheta_{t_D(\gamma) + 1}) \| \gamma \|} & \text{ if $\ell \in \cP(\gamma)$ or $\ell-1 \in \cP(\gamma)$ and $\max_{j \notin \{\ell-1, \ell\}} D(n_j) \ge \frac{D(n_{\ell})}{M^2}$}. \end{cases}
\end{split}
\end{equation}
Here, by convention, a maximum taken over the empty set is set to be $-\infty$.
\end{lemma}

\begin{proof}
The verification of the estimate goes by induction in $k = 1, 2, \dots$. The estimate obviously holds for $k = 1$. Note also that if $t_D(\gamma) \le 5$, the estimate holds for obvious reasons. So, we assume henceforth that $t_D(\gamma) > 5$. Assume that the estimate \eqref{eq:auxtrajectweight2} holds for any trajectory $\gamma' = (n'_1, \dots, n'_t)$ with $t \le k-1$, $k \ge 2$. Recall that $\ell$ is chosen such that $D(n_\ell) = \bar D(\gamma)$. There are several cases to be considered.

Case $(I)$. Assume first that $\ell-1, \ell \notin \cP(\gamma)$. Assume also that $2 \le \ell \le k-1$, so that $\gamma_1 = (n_1, \dots, n_{\ell-1})$, $\gamma_2 = (n_{\ell+1}, \dots, n_{k})$ are defined. Let $\ell_i$ be such that $D(n_{\ell_i}) = \bar D(\gamma_i)$, $i = 1, 2$. Note that $\gamma = (n_1, \dots, n_{k}) \in \Ga_{D,T}(n_1, n_{k}; \La, \mathfrak{R})$ implies $T\| (n_{\ell_1}, \dots, n_{\ell}) \|^{1/5} \ge D(n_{\ell_1})$, since otherwise $\ell_1 = \ell-1 \in \cP(\gamma)$. In particular, $(\| \gamma_1 \| + |n_{\ell-1} - n_\ell|)^{1/5} \ge T^{-1} D(n_{\ell_1}) = T^{-1} M^{t_D(\gamma_1)} \ge M^{t_D(\gamma_1)-1}$. Let us consider the following sub-cases.

$(a)$ Assume that $M^2\max_{j < \ell, j \neq \ell_1} D(n_j) < D(n_{\ell_1}) < M^{-2} D(n_{\ell})$. This implies in particular $D(n_{\ell_1}) > M^2$, that is, $t_D(\gamma_1) > 2$. In particular, $4 t_D(\gamma_1) - 5 > t_D(\gamma_1) + 1$. It implies also that $t_D(\gamma_1) + 2 < t_D(\gamma)$. Recall that due to the inductive assumption, we have $W_{D,\kappa_0}(\gamma_1) \le \exp(-\kappa_0 (1 - \vartheta_{t_D(\gamma_1)}) \|\gamma_1\| + 2 D(n_{\ell_1}))$. Hence,
\begin{equation}\label{eq:auxtrajectweight16}
\begin{split}
W_{D,\kappa_0}(\gamma_1) \exp(-\kappa_0 |n_{\ell-1} - n_\ell|) \le \exp(-\kappa_0 (1 - \vartheta_{t_D(\gamma_1)}) (\|\gamma_1\| + |n_{\ell-1} - n_\ell|) + 2 D(n_{\ell_1})) \\
\le \exp(-\kappa_0 (1 - \vartheta_{t_D(\gamma_1)}) (\|\gamma_1\| + |n_{\ell-1} - n_\ell|) + 2 T (\|\gamma_1\| + |n_{\ell-1} - n_\ell|)^{1/5})\\
\le \exp(-[\kappa_0 (1 - \vartheta_{t_D(\gamma_1)} - 2 T \kappa_0^{-1} (\|\gamma_1\| + |n_{\ell-1} - n_\ell|)^{-4/5})] (\|\gamma_1\| + |n_{\ell-1} - n_\ell|)) \\
\le \exp(-[\kappa_0 (1 - \vartheta_{t_D(\gamma_1)} - 2 T \kappa_0^{-1} M^{-4(t_D(\gamma_1)-1)})] (\|\gamma_1\| + |n_{\ell-1} - n_\ell|)) \\
\le \exp(-\kappa_0 (1 - \vartheta_{t_D(\gamma_1)} - 4^{-4 t_D(\gamma_1) + 5}) (\|\gamma_1\| + |n_{\ell-1} - n_\ell|)) \\
\le \exp(-\kappa_0 (1 - \vartheta_{t_D(\gamma_1) + 1}) (\|\gamma_1\| + |n_{\ell-1} - n_\ell|)) \le \exp(-\kappa_0 (1 - \vartheta_{t_D(\gamma)}) (\|\gamma_1\| + |n_{\ell-1} - n_\ell|)).
\end{split}
\end{equation}

$(b)$ Assume that $D(n_{\ell_1}) \le M^{-2} \max_{j < \ell, j \neq \ell_1} D(n_j)$, $D(n_{\ell_1}) < M^{-2} D(n_{\ell})$. Once again, $t_D(\gamma_1) + 2 < t_D(\gamma)$. Due to the inductive assumption, this time one has $W_{D,\kappa_0}(\gamma_1) \le \exp(-\kappa_0 (1 - \vartheta_{t_D(\gamma_1) + 1}) \|\gamma_1\|)$. So,
\begin{equation}\label{eq:auxtrajectweight17}
\begin{split}
W_{D,\kappa_0}(\gamma_1) \exp(-\kappa_0 |n_{\ell-1} - n_\ell|) \le \exp(-\kappa_0 (1 - \vartheta_{t_D(\gamma_1) + 1}) (\|\gamma_1\| + |n_{\ell-1} - n_\ell|) \\
\le \exp(-\kappa_0 (1 - \vartheta_{t_D(\gamma)}) (\|\gamma_1\| + |n_{\ell-1} - n_\ell|)).
\end{split}
\end{equation}

$(c)$ Assume $D(n_{\ell_1}) < M^{-2} D(n_{\ell})$. Obviously, $(a)$ or $(b)$ applies. Thus, in any event, we have $W_{D,\kappa_0}(\gamma_1) \exp(-\kappa_0 |n_{\ell-1} - n_\ell|) \le \exp(-\kappa_0 (1 - \vartheta_{t_D(\gamma)}) (\|\gamma_1\| + |n_{\ell-1} - n_\ell|))$.

$(d)$ Assume  $D(n_{\ell_1})\ge M^{-2} D(n_{\ell})$. Since we assumed that $t_D(\gamma) > 5$, we have $D(n_{\ell_1}) > M^{-2} M^5 = M^3$. So, $t_D(\gamma_1) > 3$. In particular, $4 t_D(\gamma_1) - 7 > t_D(\gamma_1) + 1$. Using the inductive assumption, we obtain
\begin{equation}\label{eq:auxtrajectweight18}
\begin{split}
W_{D,\kappa_0}(\gamma_1) \exp(-\kappa_0 |n_{\ell-1} - n_\ell| + 2 D(n_\ell)) \\
\le \exp(-\kappa_0 (1 - \vartheta_{t_D(\gamma_1)}) (\|\gamma_1\| + |n_{\ell-1} - n_\ell|) + (2 + 2 M^2) D(n_{\ell_1})) \\
\le \exp(-\kappa_0 (1 - \vartheta_{t_D(\gamma_1)} - \kappa_0^{-1} (2 + 2 M^2) M^{-4(t_D(\gamma_1) - 1)}) (\|\gamma_1\| + |n_{\ell-1} - n_\ell|)) \\
< \exp(-\kappa_0 (1 - \vartheta_{t_D(\gamma_1)} - 4^{-4t_D(\gamma_1) + 7}) (\|\gamma_1\| + |n_{\ell-1} - n_\ell|)) \\
\le \exp(-\kappa_0 (1 - \vartheta_{t_D(\gamma_1) + 1}) (\|\gamma_1\| + |n_{\ell-1} - n_\ell|)) \le \exp(-\kappa_0 (1 - \vartheta_{t_D(\gamma) + 1}) (\|\gamma_1\| + |n_{\ell-1} - n_\ell|)).
\end{split}
\end{equation}

Now we prove the statement in case $(I)$. Obviously, the cases $(c)$ and $(d)$ complement each other. Note also that since $\ell \notin \cP$, one can similarly identify the cases $(a)$--$(d)$ for $\gamma_2$ and establish estimates similar to \eqref{eq:auxtrajectweight16}--\eqref{eq:auxtrajectweight18}. Assume that case $(c)$ applies for both $\gamma_1$ and $\gamma_2$. Combining the estimates for $\gamma_1$ and $\gamma_2$, we obtain the desired estimate in the first line case in \eqref{eq:auxtrajectweight2}. Assume now that we have case $(d)$ for $\gamma_1$ and case $(c)$ for $\gamma_2$. Then,
\begin{equation}\label{eq:auxtrajectweight19}
\begin{split}
W_{D,\kappa_0}(\gamma) = W_{D,\kappa_0}(\gamma_1) \exp(-\kappa_0 |n_{\ell-1} - n_\ell| + D(n_\ell)) \exp(-\kappa_0 |n_{\ell+1} - n_\ell|) W_{D,\kappa_0}(\gamma_2) \\
\le \exp(-\kappa_0 (1 - \vartheta_{t_D(\gamma) + 1}) (\|\gamma_1\| + |n_{\ell-1} - n_\ell| - \kappa_0 (1 - \vartheta_{t_D(\gamma_2)}) (\|\gamma_2\| + |n_{\ell-1} - n_\ell|)) \\
\le \exp(-\kappa_0 (1 - \vartheta_{t_D(\gamma) + 1}) \|\gamma\|),
\end{split}
\end{equation}
which is the estimate in the second line case in \eqref{eq:auxtrajectweight2}. The same estimate holds if we have case $(c)$ for $\gamma_1$ and case $(d)$ for $\gamma_2$. Finally, assume we have case $(d)$ for both $\gamma_1$ and $\gamma_2$. Since $D(x) \ge 1$ for any $x$, it follows that
\begin{equation}\label{eq:auxtrajectweight20}
\begin{split}
W_{D,\kappa_0}(\gamma) = W_{D,\kappa_0}(\gamma_1) \exp(-\kappa_0 |n_{\ell-1} - n_\ell| + D(n_\ell)) \exp(-\kappa_0 |n_{\ell+1} - n_\ell|) W_{D,\kappa_0}(\gamma_2) \\
\le \exp(-\kappa_0 (1 - \vartheta_{t_D(\gamma) + 1}) (\|\gamma_1\| + |n_{\ell-1} - n_\ell| - \kappa_0 (1 - \vartheta_{t_D(\gamma_2) + 1}) (\|\gamma_2\| + |n_{\ell-1} - n_\ell|)) \\
\le \exp(-\kappa_0 (1 - \vartheta_{t_D(\gamma) + 1}) \|\gamma\|),
\end{split}
\end{equation}
which is again the estimate in the second line case in \eqref{eq:auxtrajectweight2}.

This finishes the verification in the case $(I)$ with $2 \le \ell \le k-1$. One can see that the estimates hold for the rest of sub-cases in the case $(I)$.

Case $(II)$. Assume now that $\ell \in \cP(\gamma)$. Then, in particular, $\ell+1 \le k$. Assume in addition that $2 \le \ell \le k-2$, so that $\gamma_1 = (n_1,\dots ,n_{\ell-1})$, $\gamma'_2 = (n_{\ell+2},\dots ,n_{k})$ are defined. Due to \eqref{eq:auxtrajectweight5NNNNN1} in Definition~\ref{def:aux1}, the arguments from case $(I)$ apply to $\gamma_1$. For the same reason very similar arguments apply also to $\gamma'_2$. The estimates for $\gamma'_2$ are as follows:
\begin{equation}\label{eq:auxtrajectweight19II}
\begin{split}
W_{D,\kappa_0}(\gamma'_2) e^{-\kappa_0||(n_{\ell},n_{\ell+1},n_{\ell+2})||} \le e^{-\kappa_0 (1-\vartheta_{t_D(\gamma)}) (\|\gamma_1\| + |n_{\ell-1}-n_\ell|)}, \text{if $\bar D(\gamma'_2) < M^{-2} D(n_{\ell})$}, \\
W_{D,\kappa_0}(\gamma'_2) e^{-\kappa_0||(n_{\ell},n_{\ell+1},n_{\ell+2})|| + 2D(n_\ell)} \le e^{-\kappa_0 (1-\vartheta_{t_D(\gamma)}+1) (\|\gamma_1\| + |n_{\ell-1}-n_\ell|)}, \text{if $\bar D(\gamma'_2)\ge M^{-2}D(n_{\ell})$}.
\end{split}
\end{equation}
Combining the estimates for $\gamma_1$ and $\gamma'_2$, one obtains the desired estimate in the third and forth line cases in \eqref{eq:auxtrajectweight2}. One can see that the estimates hold for the rest of sub-cases in the case $(II)$.

Case $(III)$. Finally, assume that $\ell-1 \in \cP(\gamma)$. Then, in particular $\ell-1 \ge 1$. Assume in addition that $3 \le \ell \le k-1$, so that $\gamma'_1 = (n_1,\dots ,n_{\ell-2})$, $\gamma_2 = (n_{\ell+1},\dots ,n_{k})$ are defined. The argument for this case is completely similar to the one in case $(II)$.
\end{proof}

\begin{corollary}\label{cor:auxweight1}
Let $D \in \mathcal{G}_{\La,T,\kappa_0}$, $\gamma \in \Ga_{D,T,\kappa_0}(m,n;k,\La,\mathfrak{R})$, $k \ge 1$. Then,
\begin{equation}\label{eq:auxtrajectweight30}
\begin{split}
W_{D,\kappa_0}(\gamma) \le \exp(-\kappa_0\|\gamma\| + k (4T \kappa_0^{-1})^5) \le \exp(-\frac{7}{8} \kappa_0 |m-n|) \exp(-\frac{1}{8} \kappa_0 \|\gamma\| + k (4T \kappa_0^{-1})^5) \quad \text{if $t_D(\gamma) \le 5$}, \\
W_{D,\kappa_0}(\gamma) < \exp(-\frac{15}{16} \kappa_0 \|\gamma\| + 2 \bar D(\gamma)) \\
\le \min \Big[ \exp(-\frac{7}{8} \kappa_0 |m-n| + 2 T (\min \mu_\La(m),\mu_\La(n))^{1/5}) \exp(-\frac{1}{16} \kappa_0 \|\gamma\| + 2 T \|\gamma\|^{1/5}), \\
\exp(-\frac{15}{16} \kappa_0 |m-n| + 2 \bar D) \Big], \quad \bar D := \max_x D(x) \quad \text{if $t_D(\gamma) > 5$}.
\end{split}
\end{equation}
\end{corollary}

\begin{proof}
If $t_D(\gamma) \le 5$, then the estimate follows from Lemma~\ref{lem:auxweight} since $\|\gamma\| \ge |m-n|$. Assume that $t_D(\gamma) > 5$. Let $\ell$ be such that $D(n_\ell) = \bar D(\gamma)$. Recall that $D(n_\ell) \le T \mu_\La(n_\ell)^{1/5}$. Furthermore, $\mu_\La(n_\ell) \le \mu_\La(m) + |m - n_\ell| \le \mu_\La(m) + \|\gamma\|$. So, $\bar D(\gamma) \le T(\mu_\La(m) + \|\gamma\|)^{1/5} \le T(\mu_\La(m)^{1/5} + \|\gamma\|^{1/5})$. Similarly, $\bar D(\gamma) \le T(\mu_\La(n)^{1/5} + \|\gamma\|^{1/5})$. Note also that $1 - \vartheta_t > 1 - 1/31 > 15/16$ for any $t$. Due to  Lemma~\ref{lem:auxweight}, one has in any event
\begin{equation}\label{eq:auxtrajectweight30abc}
\begin{split}
W_{D,\kappa_0}(\gamma) \le \exp(-\kappa_0 (1 - \vartheta_{t_D(\gamma)}) \|\gamma\| + 2 \bar D(\gamma)) < \exp(-\frac{15}{16} \kappa_0 \|\gamma\| + 2 \bar D(\gamma)) \\
\le \exp(-\frac{7}{8} \kappa_0 |m-n| + 2 T (\min \mu_\La(m), \mu_\La(n))^{1/5}) \exp(-\frac{1}{16} \kappa_0 \|\gamma\| + 2 T \|\gamma\|^{1/5}).
\end{split}
\end{equation}
It follows also from Lemma~\ref{lem:auxweight} that $W_{D,\kappa_0}(\gamma) \le \exp(-\frac{15}{16} \kappa_0 |m-n| + 2 \bar D)$.
\end{proof}

To proceed we need the following elementary estimate.

\begin{lemma}\label{lem:2gammasum}
$(1)$ For any $\alpha, B, k > 0$ and $0 < \ve_0 \le \min (2^{-12 \nu - 4} \alpha^{4\nu}, \exp(-8B))$, we have
\begin{equation}\label{eq:auxtrajectweight21a}
\begin{split}
\sum_{\gamma \in \Ga(m,n;k,\zv)} \exp(-\alpha \|\gamma\|) < (8 \alpha^{-1})^{(k-1) \nu}, \\
\sum_{k \ge 2} \ve_0^{k-1} \exp(kB) \sum_{\gamma \in \Ga(m,n;k,\zv)} \exp(-\alpha \|\gamma\|) < \ve_0^{\frac{1}{2}}.
\end{split}
\end{equation}

$(2)$ For any $C, T > 1$ and $\ve_0 \le \min (\exp(-8TC^{1/5}), 2^{-4(\nu+1)} (C+1)^{-4\nu}$, we have
\begin{equation}\label{eq:auxtrajectweight21ab}
\sum_{k \ge 2}\ve_0^{k-1} \sum_{\gamma \in \Ga(m,n;k,\La), \|\gamma\| \le C} \exp(2 T \|\gamma\|^{1/5}) \le \sum_{k \ge 2} \ve_0^{k-1} \exp(2 T C^{1/5}) (2 (C+1))^{(k-1) \nu} \le \ve_0^{\frac{1}{2}}.
\end{equation}
\end{lemma}

\begin{proof}
One has
\begin{equation}\label{eq:auxtrajectweight21aAA}
\begin{split}
\sum_{\gamma \in \Ga(m,n;k,\zv)} \exp(-\alpha \|\gamma\|) \le \bigl( \sum_{r \in \IZ^\nu} \exp(-\alpha |r|) \bigr)^{k-1} < \bigl( 2 \sum_{r \in \IZ, r \ge 0} \exp(-\alpha r) \bigr)^{(k-1) \nu} \\
= (2 (1 - \exp(-\alpha))^{-1})^{(k-1) \nu} < (8 \alpha^{-1})^{(k-1) \nu}, \\
\sum_{k \ge 2} \ve_0^{k-1} \exp(kB) \sum_{\gamma \in \Ga(m,n;k,\La,\mathfrak{R})} \exp(-\alpha \|\gamma\|) \le \sum_{k \ge 2} \ve_0^{k-1} \exp(kB) (8 \alpha^{-1})^{(k-1) \nu} \le \ve_0^{\frac{1}{2}}.
\end{split}
\end{equation}
This verifies $(1)$. Part $(2)$ follows from $(1)$.
\end{proof}

\begin{lemma}\label{lem:auxweight1}
Let $D \in \mathcal{G}_{\La,T,\kappa_0}$. Let $0 < \ve_0 \le \min(2^{-24 \nu - 4} \kappa_0^{4 \nu}, \exp(-(8 T \kappa_0^{-1})^5), 2^{-10 (\nu+1)} T^{-8 \nu})$. Then,
\begin{equation}\label{eq:auxtrajectweightsumest8}
\begin{split}
S_{D,T,\kappa_0,\ve_0;\La,\mathfrak{R}}(m,n) \le \min \big[ 3 \ve_0^{1/2} \exp(-\frac{7}{8} \kappa_0 |m-n| + 2 T (\min \mu_\La(m), \mu_\La(n))^{1/5}), \\
2 \ve_0^{1/2} \exp(-\frac{1}{4} \kappa_0 |m-n| + 2 \bar D)\big] \quad \text{if $m \neq n$}, \\
S_{D,T,\kappa_0,\ve_0;\La,\mathfrak{R}}(m,m) \le \min \big[ \exp(D(m)) + 3 \ve_0^{1/2} \exp(2 T \mu_\La(m)^{1/5}), 2 \exp(2 \bar D) \big].
\end{split}
\end{equation}
\end{lemma}

\begin{proof}
Let $m \neq n$. Using \eqref{eq:auxtrajectweight30}, one obtains
\begin{equation}\label{eq:auxtrajectweight21}
\begin{split}
S_{D,T,\kappa_0,;k,\La,\mathfrak{R}}(m,n) \le \exp(-\frac{7}{8} \kappa_0 |m-n| + k (4 T \kappa_0^{-1})^5) \sum_{\gamma \in \Ga(m,n;k,\La,\mathfrak{R})} \exp(-\frac{1}{8} \kappa_0 \|\gamma\|) \\
+ \exp(-\frac{7}{8} \kappa_0 |m-n| + 2 T (\min \mu_\La(m),\mu_\La(n))^{1/5}) \times \\
\big[ \sum_{\gamma \in \Ga(m,n;k,\La,\mathfrak{R})} \exp(-\frac{1}{8} \kappa_0 \|\gamma\|) + \sum_{\gamma \in \Ga(m,n;k,\La,\mathfrak{R}), \|\gamma\| \le 2^5 (T \kappa_0^{-1})^{3/2}} \exp(2 T \|\gamma\|^{1/5}) \big].
\end{split}
\end{equation}
Combining \eqref{eq:auxtrajectweight21}, \eqref{eq:auxtrajectweight21a}, and \eqref{eq:auxtrajectweight21ab}, one obtains
$$
S_{D,T,\kappa_0,\ve_0;\La,\mathfrak{R}}(m,n) \le 3 \ve_0^{1/2} \exp(-\frac{7}{8} \kappa_0 |m-n| + 2 T (\min \mu_\La(m),\mu_\La(n))^{1/5}).
$$
The derivation of the other estimates is completely similar.
\end{proof}

\begin{remark}\label{rem:2withRnoRest}
In the last lemma we estimate the functions $S_{D,T,\kappa_0,\ve_0;\La,\cR}(m,n)$ only. Clearly, $S_{D,T,\kappa_0,\ve_0;\La}(m,n)\le S_{D,T,\kappa_0,\ve_0;\La,\mathfrak{R}}(m,n)$.
\end{remark}

Later in this work we will need also the following estimates:

\begin{lemma}\label{lem:auxweight1iterated}
Let $D \in \mathcal{G}_{\La, T, \kappa_0}$. Let $0 < \ve_0 \le \min (2^{-24 \nu-4} \kappa_0^{4 \nu}, \exp(-(8 T \kappa_0^{-1})^5), 2^{-10(\nu + 1)} T^{-8\nu})$. Let $0 \le a(m,n) \le 1$, $m, n \in \La$ be arbitrary. Then, for any $m_0, n_0 \in \La^c$, we have
\begin{equation}\label{eq:auxtrajectweightsumest8iterted}
\begin{split}
\mathfrak{G}_{D, T, \kappa_0, \ve_0; \La, \mathfrak{R}}(m_0,n_0) := \ve_0 \sum_{m, n \in \La} a(m_0,m) \exp(-\kappa_0 |m_0 - m| + |m_0 - m|^{1/5}) S_{D, T, \kappa_0, \ve_0; \La, \mathfrak{R}}(m,n) \\
\exp(-\kappa_0 |n_0 - n| + |n_0 - n|^{1/5}) a(n,n_0) < \ve_0^{1/2} \exp(-\kappa_0 |m_0 - n_0|/4), \\
\mathfrak{Q}_{D, T, \kappa_0, \ve_0; \La, \mathfrak{R}}(m_0) := \ve_0 \mathfrak{G}_{D, T, \kappa_0, \ve_0; \La, \mathfrak{R}}(m_0,m_0) < \ve_0^{3/2},
\end{split}
\end{equation}
\begin{equation}\label{eq:auxtrajectweightsumest8iterted1}
\begin{split}
\mathfrak{D}^\one_{D, T, \kappa_0, \ve_0; \La, \mathfrak{R}}(m_0,n_0) := \ve_0 \sum_{m_i, n_i \in \La} a(m_0,m_1) \exp(-\kappa_0 |m_0 - m_1| + |m_0 - m_1|^{1/5}) \\
S_{D, T, \kappa_0, \ve_0; \La, \mathfrak{R}}(m_1,n_1) \exp(-\kappa_0 |n_1 - m_2| + |n_1 - m_2|^{1/5} + |m_2 - n_0|^{1/5}) S_{D, T, \kappa_0, \ve_0; \La, \mathfrak{R}}(m_2,n_2) \\
\exp(-\kappa_0 |n_0 - n_2| + |n_0 - n_2|^{1/5}) a(n_2,n_0) < \ve_0^{1/2} \exp(-\kappa_0 |m_0 - n_0|/8),
\end{split}
\end{equation}
\begin{equation}\label{eq:auxtrajectweightsumest8iterted2}
\begin{split}
\mathfrak{D}^\two_{D, T, \kappa_0, \ve_0; \La, \mathfrak{R}}(m_0,n_0) := \ve_0 \sum_{m_i, n_i \in \La} a(m_0,m_1) \exp(-\kappa_0 |m_0 - m_1| + |m_0 - m_1|^{1/5}) \\
S_{D, T, \kappa_0, \ve_0; \La, \mathfrak{R}}(m_1,n_1) \exp(-\kappa_0 |n_1 - m_2| + |m_2 - n_1|^{1/5} + |m_2 - m_0|^{1/5}) S_{D, T, \kappa_0, \ve_0; \La, \mathfrak{R}}(m_2,n_2) \\
\exp(-\kappa_0 |m_3 - n_2| + |m_3 - n_2|^{1/5} +|m_3 - m_0|^{1/5}) S_{D, T, \kappa_0, \ve_0; \La, \mathfrak{R}}(m_3,n_3) \\
\exp(-\kappa_0 |n_3 - m_4| + |m_4 - n_3|^{1/5} + |m_4 - m_0|^{1/5}) S_{D, T, \kappa_0, \ve_0; \La, \mathfrak{R}}(m_4,n_4) \\
\exp(-\kappa_0 |n_0 - n_4| + |n_0 - n_4|^{1/5}) a(n_4,n_0) < \ve_0^{1/2} \exp(-\kappa_0 |m_0 - n_0|/16).
\end{split}
\end{equation}
\end{lemma}

\begin{proof}
Using \eqref{eq:auxtrajectweightsumest8} from Lemma~\ref{lem:auxweight1} and \eqref{eq:auxDcond} from Definition~\ref{def:aux1}, one obtains
\begin{equation}\label{eq:4101rem1iterations}
\begin{split}
\mathfrak{Q}_{D, T, \kappa_0, \ve_0; \La, \mathfrak{R}}(m_0) \le \varepsilon_0^2 \sum_{m, n \in \La, m \neq n} 3 \ve_0^{1/2} \times \\
\exp(-\frac{1}{4} \kappa_0 |m - n| - \kappa_0 |m_0 - m| - \kappa_0 |m_0 - n| + T(\min \mu_{\La}(m), \mu_{\La}(n))^{1/5} \\
+ |m - n|^{1/5} + |m_0 - m|^{1/5} + |m_0 - n|^{1/5}) + \varepsilon_0^2 \sum_{m \in \La} \exp(-2 \kappa_0 |m_0 - m| + 2 |m_0 - m|^{1/5}) \\
[\exp(T \mu_{\La}(m)^{1/5}) + 3 \ve_0^{1/2} \exp(2 T \mu_{\La}(m)^{1/5})] \le \varepsilon_0^2 \sum_{m, n \in \La, m \neq n} 3 \ve_0^{1/2} \times \\
\exp(-\frac{1}{4} \kappa_0 |m - n| - \kappa_0 |m_0 - m| - \kappa_0 |m_0 - n| \\
+ T (\min (|m_0 - m|, |m_0 - n|)^{1/5} + |m - n|^{1/5} + |m_0 - m|^{1/5} + |m_0 - n|^{1/5}) \\
+ \varepsilon_0^2 \sum_{m \in \La} \exp(-2 \kappa_0 |m_0 - m| + 2 |m_0 - m|^{1/5}) [\exp(T |m_0 - m|^{1/5}) + 3 \ve_0^{1/2} \exp(2 T |m_0 - m|^{1/5})].
\end{split}
\end{equation}
Combining this estimate with \eqref{eq:auxtrajectweight21a}, one obtains \eqref{eq:auxtrajectweightsumest8iterted}. The derivation of the rest of the estimates is completely similar.
\end{proof}

\begin{lemma}\label{lem:auxDfunctionsrules}
Assume that $\La = \La_1 \cup \La_2$, $\La_1 \cap \La_2 = \emptyset$. Then, $\mu_\La(m) \ge \mu_{\La_j}(m)$ if $m \in \La_j$. In particular, let $D_j \in \mathcal{G}_{\La_j,T,\kappa_0}$, $j = 1, 2$. Set $D(m) := D_j(m)$ if $m \in \La_j$. Then, $D \in \mathcal{G}_{\La,T,\kappa_0}$.
\end{lemma}

\begin{proof}
Let $m \in \La_j$. It follows from the definition of the functions $\mu_{\La'}$ that $\mu_\La(m) \ge \mu_{\La_j}(m)$. The second statement follows from the first one, just due to the definition of $\mathcal{G}_{\La,T,\kappa_0}$.
\end{proof}

\begin{lemma}\label{lem:auxDfunctionsrules1}
Assume that $\La = \La_1 \cup \La_2$,  $\La_1 \cap \La_2 = \emptyset$. Let $D_j \in \mathcal{G}_{\La_j,T,\kappa_0}$, $j = 1,2$. Set $D(m) := D_j(m)$ if $m \in \La_j$. Let $m,n \in \La_1$,
\begin{equation}\label{eq:auxtrajectweight10}
\begin{split}
\gamma = \gamma_1 \cup \gamma_2 \dots \cup \gamma_{2t+1}, \quad \sigma = \gamma_1 \cup \gamma_2 \dots \cup \gamma_{2t} \\
\gamma_{2i+1} = (n_{1,2i+1}, \dots, n_{k_{2i+1},2i+1}) \in \Ga_{D_1,T}(n_{0,2i+1}, n_{k_{2i+1}, 2i+1}; \La_1, \mathfrak{R}), \quad n_{1,1} = m, n_{k_{2t+1,2t+1}} = n \\
\gamma_{2i} = (n_{1,2i}, \dots, n_{k_{2i},2i}) \in \Ga_{D_2,T}(n_{0,2i}, n_{k_{2i}, 2i+1}; \La_2, \mathfrak{R}), n_{k_{t}} = n.
\end{split}
\end{equation}
Then,
\begin{itemize}

\item[(1)] $\gamma, \sigma \in \Ga_{D,T,\kappa_0}(\La,\mathfrak{R})$.

\item[(2)] If $\gamma_1 \cup \gamma_2 \dots \cup \gamma_{2t+1} = \gamma'_1 \cup \gamma'_2 \dots \cup \gamma'_{2t'+1}$, $t, t' \ge 0$, then $t = t'$, $\gamma_j = \gamma'_j$. Similarly, if $\gamma_1 \cup \gamma_2 \dots \cup \gamma_{2t} = \gamma'_1 \cup \gamma'_2 \dots \cup \gamma'_{2t'}$, $t, t' \ge 0$, then $t = t'$, $\gamma_j = \gamma'_j$.

\end{itemize}
\end{lemma}

\begin{proof}
$(1)$ We verify the statement for $\gamma$. The verification for $\sigma$ is completely similar. Re-denote $\gamma$ as $\gamma = (n_0, n_1, \dots, n_k)$. We need to verify conditions \eqref{eq:auxtrajectweight5NNNNN}, \eqref{eq:auxtrajectweight5NNNNN1} for any $i<j$ such that $\min(D(n_i),D(n_j)) \ge 4T \kappa_0^{-1}$. Clearly these conditions hold if $n_{i}, \dots, n_{j}$ are consecutive points in some $\gamma_h$. Assume that $n_{i} \in \La_1$, $n_j \in \La_2$. Assume also that $D(n_i),D(n_j) \ge 4T \kappa_0^{-1}$. One has
\begin{equation}\label{eq:auxtrajectweight11}
\begin{split}
D_1(n_{i}) \le T \mu_{\La_1} (n_{i})^{1/5}, D_2(n_{j}) \le T \mu_{\La_2}(n_{j})^{1/5}, \\
\|(n_{i},\dots,n_{j})\| > |n_{i} - n_{j}| \ge \max (\mu_{\La_1}(n_{i}), \mu_{\La_2}(n_{j})), \\
\max(D(n_i),D(n_j)) \le T \|(n_{i}, \dots, n_{j})\|^{1/5}.
\end{split}
\end{equation}
So, conditions \eqref{eq:auxtrajectweight5NNNNN}, \eqref{eq:auxtrajectweight5NNNNN1} hold in this case. Assume now that $n_{i} \in \La_1$, $n_h \in \La_2$, $n_j \in \La_1$, $i < h < j$. Assume also that $D(n_i),D(n_h),D(n_j) \ge 4T \kappa_0^{-1}$. Then, due to \eqref{eq:auxtrajectweight11}, one has $D(n_i) \le T \|(n_{i},\dots,n_{h})\|^{1/5}$, $D(n_j) \le T \|(n_{h}, \dots, n_{j})\|^{1/5}$. This of course implies conditions \eqref{eq:auxtrajectweight5NNNNN}, \eqref{eq:auxtrajectweight5NNNNN1} in this case. The verification for the rest of the cases is completely similar. This finishes the first statement.

$(2)$ The proof goes by induction in $\max(t,t') = 0, 1, \dots$. We will prove the statement regarding $\gamma = \gamma'$. The proof for $\sigma = \sigma'$ is completely similar. If $t,t' = 0$, then the statement is trivial. Assume that the statement holds if $\max(t,t') \le s-1$, where $s \ge 1$. If $t \ge 1$, then $\gamma_1 \cup \gamma_2 \dots \cup \gamma_{2t+1} \notin \Ga(m,n;\La_1)$ since $\La_1 \cap \La_2 = \emptyset$. So, one can assume $t,t' \ge 1$. Note that $n_{1,2t+1}, \dots, n_{k_{2t+1},2t+1} \in \La_1$, $n'_{1,2t'+1}, \dots, n'_{k_{2t'+1},2t'+1} \in \La_1$, $n_{k_{2t},2t} \in \La_2$, $n'_{k_{2t'},2t'} \in \La_2$. Since $\gamma = \gamma'$, one concludes that $k_{2t+1} = k_{2t'+1}$ and $n_{i,2t+1} = n'_{i,2t'+1}$ for all $i$. This implies $\gamma_1 \cup \gamma_2 \dots \cup \gamma_{2t} = \gamma'_1 \cup \gamma_2' \dots \cup \gamma_{2t'}'$. Repeating this argument, one concludes that $k_{2t} = k_{2t'}$ and $n_{i,2t} = n'_{i,2t'}$ for all $i$. This implies $\gamma_1 \cup \gamma_2 \dots \cup \gamma_{2t-1} = \gamma_1' \cup \gamma_2' \dots \cup \gamma_{2t'-1}'$. Due to the inductive assumption, one has then $t-1 = t'-1$, $\gamma_j = \gamma'_j$, $1 \le j \le t-1$. This finishes the proof.
\end{proof}

\begin{lemma}\label{lem:Rtraject}
$(1)$ Assume that $\La = \La_1 \cup \La_2$, $\La_1 \cap \La_2 = \emptyset$. Let $D_1 \in \mathcal{G}_{\La_1,T,\kappa_0}$. Let $D_2(x) \ge 1$, $x \in \La_2$ be such that $D_2(x) \le T \mu_\La(x)^{1/5}$ for any $x \in \La_2$. Set $D(m) := D_j(m)$ if $m \in \La_j$, $m \in \La_j$.
Then, $D \in \mathcal{G}_{\La,T,\kappa_0}$.

$(2)$ Let $m, n \in \La_2$ and $\gamma_i = (n_{1,i}, \dots, n_{k_i,i}) \in \Ga_{D,T,\kappa_0}(\La_1,\mathfrak{R})$, $i = 1, 2$, be arbitrary. Set $\gamma' = (m,n)$ if $m \neq n$, $\gamma' = (m)$ if $m = n$, $\gamma'_1 = (m)$, $\gamma'_2 = n$. Then, $\gamma_1 \cup \gamma', \gamma' \cup \gamma_2, \gamma_1 \cup \gamma' \cup \gamma_2, \gamma'_1 \cup \gamma_1 \cup \gamma'_2 \in \Ga_{D,T,\kappa_0}(\La,\mathfrak{R})$.
\end{lemma}

\begin{proof}
The first part is clear. For the second part, we cannot just refer to Lemma~\ref{lem:auxDfunctionsrules1} since it may happen that $D_2 \notin \mathcal{G}_{\La_2,T,\kappa_0}$. However, a part of the argument from the proof of Lemma~\ref{lem:auxDfunctionsrules1} still works. We need to verify conditions \eqref{eq:auxtrajectweight5NNNNN}, \eqref{eq:auxtrajectweight5NNNNN1} for any $i<j$. We will do this for $\gamma := (n_1, \dots, n_k) := \gamma_1 \cup \gamma' \cup \gamma_2$ with $m \neq n$. The verification for the rest of the cases is similar. If $i < j \le k_1$ or $k_1+3 \le i < j \le k$, then \eqref{eq:auxtrajectweight5NNNNN}, \eqref{eq:auxtrajectweight5NNNNN1} hold since $\gamma_i \in \Ga_{D,T,\kappa_0}(\La_1,\mathfrak{R})$. The argument from the proof of Lemma~\ref{lem:auxDfunctionsrules1} still works in the following cases: $(a)$ $i \le k_1$,  $k_1+3 \le j \le k$, $(b)$ $i \le k_1$, $k_1+1 \le j \le k_1+2$, $(c)$ $k_1+1 \le i \le k_1+2$, $k_1+3 \le j \le k$ since $D_1 \in \mathcal{G}_{\La_1,T,\kappa_0}$. Let $k_1+1 \le i < j \le k_1+2$. Then $i = k_1+1$, $j = k_1+2$, that is, $j = i+1$. Obviously, in this case \eqref{eq:auxtrajectweight5NNNNN1} holds.
\end{proof}

\begin{lemma}\label{lem:aux5}
Let $\La = \La_1 \cup \La_2$, $\La_1 \cap \La_2 = \emptyset$. Let $D_j \in \mathcal{G}_{\La_j,T,\kappa_0}$, $j = 1,2$. Let
$$
\cH_\Lambda = \begin{bmatrix} \cH_{\Lambda_1} & \Gamma_{1,2}\\[5pt] \Gamma_{2,1} & \cH_{\Lambda_2}\end{bmatrix}.
$$
$\Gamma_{i,j} := \Gamma_{\Lambda_i, \Lambda_j} (k, \ell) = \cH(k,\ell),\quad k\in \La_i,\ell\in \La_j$. Assume that the following conditions hold:
\begin{enumerate}
\item[(i)]
$$
\ve_0 w(m,n) := \ve_0 w_\La (m,n) := |\mathcal{H}_\Lambda(m,n)| \le \varepsilon_0 \exp( - \kappa_0 |m-n| ) , \quad m \not= n,
$$
$0 < \kappa_0 < 1$, $0 < \ve_0 \le \min(2^{-24\nu-4} \kappa_0^{4\nu}, \exp(-(8T\kappa_0^{-1})^5, 2^{-10(\nu+1)} T^{-8\nu})$.

\item[(ii)] The matrix $\cH_{\La_j}$ is invertible; moreover,
\begin{equation}\label{eq:aux00H1inverse1new}
|\cH_{\Lambda_j}^{-1}(m,n)| \le s_{D_j,T,\kappa_0,\ve_0;\La_j,\mathfrak{R}}(m,n).
\end{equation}

Then, $\tilde H_2 := [\cH_{\La_2} - \Gamma_{2,1} \cH_{\La_1}^{-1} \Gamma_{1,2}](m,n)$ is invertible, $\cH_{\Lambda}$ is invertible, and
\begin{equation}\label{eq:aux00c011kappad2statement}
\begin{split}
|\cH_{\La}^{-1} (m,n)| \le \sum_{k \ge 1} \ve_0^{k-1} \sum_{\gamma \in \Ga_{D_1,T,\kappa_0}(m,n;k,\La_1,\mathfrak{R})} w_{D_1,\kappa_0}(\gamma) \\
+ \sum_{q \ge 3} \ve_0^{q-1} \sum_{\gamma \in \Ga^{(1,2)}_{D,T,\kappa_0}(m,n;q,\La)} w_{D,\kappa_0}(\gamma) \le s_{D,T,\kappa_0,\ve_0;k,\La,\mathfrak{R}}(m,n), \quad m,n \in \La_1, \\
|\cH_{\La}^{-1} (m,n)| = |\tilde H_2^{-1}(m,n)| \le \sum_{k \ge 1} \ve_0^{k-1} \sum_{\gamma \in \Ga_{D,T,\kappa_0}(m,n;k,\La_2,\mathfrak{R})} w_{D,\kappa_0}(\gamma) \\
+ \sum_{k \ge 3} \ve_0^{k-1} \sum_{\gamma \in \Ga^{(2,1)}_{D,T,\kappa_0}(m,n;k,\La,\mathfrak{R})} w_{D,\kappa_0}(\gamma) \le s_{D,T,\kappa_0,\ve_0;k,\La,\mathfrak{R}}(m,n), \quad m,n \in \La_2, \\
|\cH_{\Lambda}^{-1}(m,n)| \\
\le \sum_{k \ge 3} \ve_0^{k-1} \sum_{\gamma \in \Ga^{(p,q,odd)}_{D,T,\kappa_0}(m,n;k,\La,\mathfrak{R})} w_{D,\kappa_0}(\gamma) \le s_{D,T,\kappa_0,\ve_0;k,\La,\mathfrak{R}}(m,n), \quad m \in \La_p, n \in \La_q, p \neq q,
\end{split}
\end{equation}
where $\Ga^{(p,q)}_{D,T,\kappa_0}(m,n;k,\La,\mathfrak{R}) = \bigcup_{t \ge 1} \Ga^{(p,q,t)}_{D,T,\kappa_0}(m,n;k,\La,\mathfrak{R})$, $\Ga^{(p,q,t)}_{D,T,\kappa_0} (m,n;k,\La,\mathfrak{R})$ stands for the set of all $\gamma \in \Ga_{D,T,\kappa_0}(m,n;k,\La,\mathfrak{R})$ such that  $\gamma = \gamma_1 \cup \gamma'_1 \dots \cup \gamma_{t+1}$ with $\gamma_j \in \Ga_{D_p,T,\kappa_0}(\La_p,\mathfrak{R})$, $\gamma'_i \in \Ga_{D_q,T,\kappa_0}(\La_q)$, $p \neq q$, $\Ga^{(p,q,odd)}_{D,T,\kappa_0} (m,n;k,\La,\mathfrak{R}) = \bigcup_{t \ge 1} \Ga^{(p,q,odd,t)}_{D,T,\kappa_0}(m,n;k,\La,\mathfrak{R})$, $\Ga^{(p,q,odd,t)}_{D,T,\kappa_0} (m,n;k,\La,\mathfrak{R})$ stands for the set of all $\gamma \in \Ga_{D,T,\kappa_0}(m,n;k,\La,\mathfrak{R})$ such that $\gamma = \gamma_1 \cup \gamma'_1 \dots \cup \gamma'_{t}$ with $\gamma_j \in \Ga_{D_p,T,\kappa_0}(\La_p)$, $\gamma'_i \in \Ga_{D_q,T,\kappa_0}(\La_q)$, $p \neq q$, and $D(m) = D_j(m)$ if $m \in \La_j$.
\end{enumerate}
\end{lemma}

\begin{proof}
Let $m, n \in \La_2$. For any $t \ge 1$, one has
\begin{equation}\label{eq:aux00c01a1part1}
\begin{split}
|[\cH_{\La_2}^{-1} \bigl(\Gamma_{2,1} \cH_{\La_1}^{-1} \Gamma_{1,2}\cH_{\La_2}^{-1}\bigr)^{t} ](m,n)| \le \sum_{n_{i} \in \La_2; n'_{i} \in \Lambda_1, i = 1, \dots, t} \\
\ve_0^{2t} |\cH_{\La_2}^{-1}(m,n_1)| \exp(-\kappa_0 |n_1-n'_1|) |\cH_{\La_1}^{-1}(n'_1,n'_2)| \dots |\cH_{\La_2}^{-1}(n_{t},n)| \le \sum_{n_{i} \in \La_2; n'_{i} \in \Lambda_1, i = 1,\dots,t} \\
\sum_{k_i, k'_j \ge 1, j = 1,\dots} \ve_0^{(\sum_{j}k_j) + (\sum_i k'_i)-1} \sum_{\gamma_1 \in \Ga_{D_2,T,\kappa_0} (m,n_1;k_1,\La_2,\mathfrak{R})} \sum_{\gamma'_1 \in \Ga_{D_1,T,\kappa_0}(n_2,n_3;k'_1,\La_1,\mathfrak{R})} \dots \sum_{\gamma_{t+1} \in \Ga_{D_2,T,\kappa_0}(n_t,n;k_{t+1},\La_2,\mathfrak{R})} \\
w_{D_2,\kappa_0}(\gamma_1) \exp(-\kappa_0 |n_1-n_2|) W_{D_1,\kappa_0} (\gamma'_1) \dots w_{D,\kappa_0} (\gamma_{t+1}) \\
= \sum_{n_{i} \in \La_2; n'_{i} \in \Lambda_1, i = 1, \dots, t} \\
\sum_{k_i, k'_j \ge 1, j = 1, \dots} \ve_0^{(\sum_{j}k_j) + (\sum_i  k'_i)-1} \sum_{\gamma_1 \in \Ga_{D_2,T,\kappa_0}(m,n_1;k_1,\La_2,\mathfrak{R})} \sum_{\gamma'_1 \in \Ga_{D_1,T,\kappa_0}(n_2,n_3;k'_1,\La_1,\mathfrak{R})} \dots \sum_{\gamma_{t+1} \in \Ga_{D_2,T,\kappa_0} (n_t,n;k_{t+1},\La_2,\mathfrak{R})} \\
w_{D,\kappa_0} (\gamma_1 \cup \gamma'_1 \dots \cup \gamma_{t+1}).
 \end{split}
\end{equation}
Combing \eqref{eq:aux00c01a1part1} with  Lemma~\ref{lem:auxDfunctionsrules1}, one obtains
\begin{equation}\label{eq:aux00c01a1part2}
|[\cH_{\La_2}^{-1} \bigl(\Gamma_{2,1} \cH_{\La_1}^{-1} \Gamma_{1,2}\cH_{\La_2}^{-1}\bigr)^{t} ](m,n)| \le \sum_{k \ge 1} \ve_0^{k-1} \sum_{\gamma \in \Ga^{(2,1,t)}_{D,T,\kappa_0}(m,n;k,\La,\mathfrak{R})} w_{D,\kappa_0}(\gamma).
\end{equation}
Due to Lemma~\ref{lem:auxDfunctionsrules1}, $\Ga^{(2,1,t)}_{D,T,\kappa_0}(m,n;k,\La,\mathfrak{R}) \cap \Ga^{(2,1,t')}_{D,T,\kappa_0} (m,n;k',\La,\mathfrak{R}) = \emptyset$, unless $t = t'$. Hence,
\begin{equation}\label{eq:aux00c01a1part3}
\sum_{t \ge 1}|[\cH_{\La_2}^{-1} \bigl(\Gamma_{2,1} \cH_{\La_1}^{-1} \Gamma_{1,2}\cH_{\La_2}^{-1}\bigr)^{t} ](m,n)| \le \sum_{k \ge 1} \ve_0^{k-1} \sum_{\gamma \in \Ga^{(2,1)}_{D,T,\kappa_0}(m,n;k,\La,\mathfrak{R})} w_{D,\kappa_0}(\gamma).
\end{equation}
Note that $\Ga^{(2,1)}_{D,T,\kappa_0}(m,n;k,\La,\mathfrak{R}) \cap \Ga_{D_2,T,\kappa_0}(m,n;k,\La_2,\mathfrak{R}) = \emptyset$. Thus,
\begin{equation}\label{eq:aux00c01a1part4}
\begin{split}
\tilde H_2^{-1} := [(\cH_{\La_2} - \Gamma_{2,1} \cH_{\La_1}^{-1} \Gamma_{1,2})^{-1}](m,n)| \le |[\cH_{\La_2}^{-1}](m,n)| + \\
\sum_{t \ge 1}|[\cH_{\La_2}^{-1} \bigl(\Gamma_{2,1} \cH_{\La_1}^{-1} \Gamma_{1,2}\cH_{\La_2}^{-1}\bigr)^{t} ](m,n)| \le \sum_{k \ge 1} \ve_0^{k-1} \sum_{\gamma \in \Ga_{D_2,T,\kappa_0}(m,n;k,\La_2,\mathfrak{R})} w_{D,\kappa_0}(\gamma) \\
+ \sum_{k \ge 1} \ve_0^{k-1} \sum_{\gamma \in \Ga^{(2,1)}_{D,T,\kappa_0}(m,n;k,\La,\mathfrak{R})} w_{D,\kappa_0}(\gamma) \le s_{D,T,\kappa_0,\ve_0;k,\La,\mathfrak{R}}(m,n).
\end{split}
\end{equation}
Due to the Schur complement formula, $\cH_{\La}$ is invertible and $[\cH_{\La}^{-1}](m,n) = [\tilde H_2^{-1}](m,n)|$. This finishes the proof of the statement when $m, n \in \La_2$.

Let now $m, n \in \La_1$. Using the Schur complement formula and \eqref{eq:aux00c01a1part2}, one obtains
\begin{equation}\label{eq:aux00c01a1Part4}
\begin{split}
|\cH_{\La}^{-1} (m,n)| = |[\cH_{\La_1,\ve}^{-1}](m,n)| + |[\cH_{\La_1,\ve}^{-1} \Gamma_{1,2} \tilde H_2^{-1} \Gamma_{2,1} \cH_{\La_1,\ve}^{-1}](m,n)| \\
\le \sum_{k \ge 1} \ve_0^{k-1} \sum_{\gamma \in \Ga_{D_1,T,\kappa_0}(m,n;k,\La_1,\mathfrak{R})} w_{D_1,\kappa_0}(\gamma) \\
+ \sum_{k, \ell, \tilde k \ge 1} \ve_0^{k+\ell+\tilde k-1} \sum_{n_1, n_4 \in \La_1; n_2, n_3 \in \La_2} \sum_{\gamma \in \Ga_{D_1,T,\kappa_0}(m,n_1;k,\La_1,\mathfrak{R})} \\
\big[\sum_{\lambda \in \Ga_{D_2,T,\kappa_0}(n_2,n_3;\ell,\La_2)} + \sum_{t \ge 1} \sum_{\lambda \in \Ga^{(2,1,t)}_{D,T,\kappa_0}(n_2,n_3;\ell,\La,\mathfrak{R})}\big] \sum_{\tilde \gamma \in \Ga_{D_1,T,\kappa_0}(n_4,n;k,\La_1,\mathfrak{R})} \\
w_{D,\kappa_0}(\gamma) w(n_1,n_2) w_{D,\kappa_0}(\lambda) w(n_3,n_4) w_{D,\kappa_0}(\tilde\gamma).
\end{split}
\end{equation}
Note that here
$$
w_{D,\kappa_0}(\gamma) w(n_1,n_2) w_{D,\kappa_0}(\lambda) w(n_3,n_4) w_{D,\kappa_0}(\tilde \gamma) = w_{D,\kappa_0}(\gamma \cup \lambda \cup \tilde \gamma).
$$
Let $\gamma \in \Ga_{D_1,T,\kappa_0}(k,\La_1,\mathfrak{R})$, $\tilde \gamma \in \Ga_{D_1,T,\kappa_0}(\tilde k,\La_1,\mathfrak{R})$, $\lambda \in \Ga^{(2,1,t)}_{D,T,\kappa_0}(\ell,\La,\mathfrak{R})$. One has $\lambda = \lambda_1 \cup \lambda'_1 \dots \cup \lambda_{t+1}$ with $\lambda_j \in \Ga_{D_2,T,\kappa_0}(m,n;\La_2,\mathfrak{R})$, $\lambda'_i \in \Ga_{D_1,T,\kappa_0}(\La_1,\mathfrak{R})$. Therefore, $\gamma \cup \lambda \cup \tilde \gamma \in \Ga^{(1,2,t+1)}_{D,T,\kappa_0}(m,n;k+\ell+\tilde k,\La,\mathfrak{R})$. Furthermore, let $\gamma' \in \Ga_{D_1,T,\kappa_0}(k',\La_1,\mathfrak{R})$, $\tilde \gamma' \in \Ga_{D_1,T,\kappa_0}(\tilde k',\La_1,\mathfrak{R})$, $\sigma \in \Ga^{(2,1,t')}_{D,T,\kappa_0}(\ell',\La,\mathfrak{R})$, $\sigma = \sigma_1 \cup \sigma'_1 \dots \cup \sigma_{t'+1}$ with $\sigma_j \in \Ga_{D_2,T,\kappa_0}(m,n;\La_2,\mathfrak{R})$, $\sigma'_i \in \Ga_{D_1,T,\kappa_0}(\La_1,\mathfrak{R})$. If $\gamma \cup \lambda \cup \tilde \gamma = \gamma' \cup \sigma \cup \tilde \gamma'$, then, due to Lemma~\ref{lem:auxDfunctionsrules1}, $t = t'$, $\gamma = \gamma'$, $\tilde \gamma = \tilde \gamma'$, $\lambda_1 = \sigma_1, \lambda'_1 = \sigma'_1, \dots \lambda_{t+1} = \sigma_{t+1}$. If $\lambda \in \Ga_{D_2,T,\kappa_0}(n_2,n_3; \ell,\La_2,\mathfrak{R})$, then $\gamma \cup \lambda \cup \tilde \gamma \in\Ga^{(1,2,1)}_{D,T,\kappa_0}(m,n;k+\ell+\tilde k,\La,\mathfrak{R})$. Therefore,
\begin{equation}\label{eq:aux00c01a1Part4AAAA}
\begin{split}
|\cH_{\La}^{-1} (m,n)| \le \sum_{k \ge 1} \ve_0^{k-1} \sum_{\gamma \in \Ga_{D_1,T,\kappa_0}(m,n;k,\La_1,\mathfrak{R})} w_{D_1,\kappa_0}(\gamma)
+ \sum_{t \ge 1} \sum_{q \ge 3} \ve^{q-1} \sum_{\gamma \in \Ga^{(1,2,t)}_{D,T,\kappa_0}(m,n;q,\La,\mathfrak{R})} w_{D,\kappa_0}(\gamma) \\
= \sum_{k \ge 1} \ve_0^{k-1} \sum_{\gamma \in \Ga_{D_1,T,\kappa_0}(m,n;k,\La_1,\mathfrak{R})} w_{D_1,\kappa_0}(\gamma) + \sum_{q \ge 3} \ve^{q-1} \sum_{\gamma \in \Ga^{(1,2)}_{D,T,\kappa_0}(m,n;q,\La,\mathfrak{R})} w_{D,\kappa_0}(\gamma),
\end{split}
\end{equation}
where $\Ga^{(1,2)}_{D,T,\kappa_0}(m,n;k,\La,\mathfrak{R}) = \bigcup_{t \ge 1} \Ga^{(1,2,t)}_{D,T,\kappa_0}(m,n;k,\La,\mathfrak{R})$. Note that $\Ga^{(1,2)}_{D,T,\kappa_0}(m,n;k,\La,\mathfrak{R}) \cap \Ga_{D_1,T,\kappa_0}(m,n;k,\La_1,\mathfrak{R}) = \emptyset$. Thus,
\begin{equation}\label{eq:aux00c01a1part4ZXXZ}
|\cH_{\La}^{-1} (m,n)|\le s_{D,T,\kappa_0,\ve_0;k,\La,\mathfrak{R}}(m,n).
\end{equation}
This finishes the proof of the statement when $m,n \in \La_1$. The proof for the cases $m \in \La_1$, $n \in \La_2$ and $m \in \La_2$, $n \in \La_1$ is completely similar.
\end{proof}

\begin{lemma}\label{lem:aux5AABBCC}
Let $\La \subset \IZ^\nu$. Assume that
$$
\ve_0 w(m,n) := \ve_0 w_\La (m,n) := |\mathcal{H}_\Lambda(m,n)| \le \varepsilon_0 \exp( - \kappa_0 |m-n| ), \quad m \not= n,
$$
$0 < \kappa_0 < 1$, $0 < \ve_0 \le \min(2^{-24\nu-4} \kappa_0^{4\nu}, \exp(-2^5 T \kappa_0^{-1}), 2^{-10(\nu+1)} T^{-8\nu})$. Let $\La_1\cup \La_2 = \La$, $\La_1 \cap \La_2 = \emptyset$,
$$
\cH_\Lambda = \begin{bmatrix} \cH_{\Lambda_1} & \Gamma_{1,2}\\[5pt] \Gamma_{2,1} & \cH_{\Lambda_2}\end{bmatrix}
$$
Assume that

$(i)$ The matrix $\cH_{\La_1}$ is invertible and there exists $D_1 \in \mathcal{G}_{\La_1,T,\kappa_0}$ such that
\begin{equation}\label{eq:aux00H1inverse1newNEW}
|\cH_{\Lambda_1}^{-1}(m,n)| \le s_{D_1,T,\kappa_0,\ve_0;\La_1,\mathfrak{R}}(m,n).
\end{equation}

$(ii)$ $\tilde H_2 := \cH_{\La_2} - \Gamma_{2,1} \cH_{\La_1}^{-1} \Gamma_{1,2}$ obeys $|\det \tilde H_2|^{-1} \le \exp(D_0)$, where $D_0 \le T \min_{x \in \La_2} \mu_\La(x)^{1/5}$.

Set $D(x) = D_1(x)$ if $x \in \La_1$, $D(x) = D_0$ if $x \in \La_2$. Then, $D \in \mathcal{G}_{\La,T,\kappa_0}$, $\cH_\La$ is invertible and
\begin{equation}\label{eq:aux00c011OP}
|\cH_{\Lambda}^{-1}(m,n)| \le s_{D,T,\kappa_0,\ve_0;\La,\mathfrak{R}}(m,n).
\end{equation}
\end{lemma}

\begin{proof}
Note that condition $(i)$ implies in particular $D \in \mathcal{G}_{\La,T,\kappa_0}$. Furthermore,
\begin{equation}\label{eq:2-10acmm5NNN}
|\tilde H_2(m,n)| \le |\cH(m,n)| + \sum_{m',n' \in \La_1} |\cH(m,m')| s_{D_1,T,\kappa_0,\ve_0;\La_1,\mathfrak{R}}(m',n') |\cH(n',n)|.
\end{equation}
Let $m',n' \in \La_1$ and $\gamma \in \Ga_{D,T,\kappa_0}(\La_1,\mathfrak{R})$ be arbitrary. Set $\gamma_1' = (m)$ if $\gamma'' = (n)$. Then, due to Lemma~\ref{lem:Rtraject}, $\gamma'_1 \cup \gamma \cup \gamma'_2 \in \Ga_{D,T,\kappa_0}(\La_1,\mathfrak{R})$. Using Cramer's rule, condition $(i)$, and \eqref{eq:2-10acmm5NNN}, one obtains
\begin{equation}\label{eq:2-10acmm5QP}
|\tilde H_2^{-1}(m,n)| \le s_{D,T,\kappa_0,\ve_0;\La,\mathfrak{R}}(m,n).
\end{equation}
Similarly, let $m,n \in \La_1$. Using the Schur complement formula, Lemma~\ref{lem:Rtraject}, \eqref{eq:aux00H1inverse1new}, and \eqref{eq:2-10acmm5NNN}, one obtains
\begin{equation}\label{eq:aux00c014RR}
|\cH_{\La}^{-1} (m,n)| = |[\cH_{\La_1}^{-1}](m,n)| + |[\cH_{\La_1}^{-1} \Gamma_{1,2} \tilde H_2^{-1} \Gamma_{2,1} \cH_{\La_1}^{-1}](m,n)| \\
\le s_{D,T,\kappa_0,\ve_0;\La,\mathfrak{R}}(m,n).
\end{equation}
The same estimate holds for $m \in \La_1$, $n \in \La_2$.
\end{proof}

\begin{lemma}\label{lem:aux6}
Assume that the following conditions hold:
\begin{equation}\label{eq:auxbasicpertcondrep1}
\begin{split}
\ve_0 w(m,n) := |\cH(m,n)| \le \ve_0 \exp(- \kappa_0 |m-n|), \quad m,n \in \La, m \neq n \\
\min_{m \in \La} |\cH(m,m)| \ge \exp(-4T\kappa_0^{-1}),
\end{split}
\end{equation}
$0 < \ve_0 \le \min(2^{-24\nu-4} \kappa_0^{4\nu}, \exp(-(8T \kappa_0^{-1})^5), 2^{-10(\nu+1)}T^{-8\nu})$. Then, $\cH_\La$ is invertible and
\begin{equation}\label{eq:aux00basicpertestimate}
\left|\cH_\La^{-1}(m,n)\right| \le s_{D,T,\kappa_0,\ve_0;k,\La}(m,n).
\end{equation}
\end{lemma}

\begin{proof}
Set $D(m) = 4T\kappa_0^{-1}$, $m \in \La$. Note that $D \in \mathcal{G}_{\La,T,\kappa_0}$. Set also $A(m,n) = \cH(m,m)\delta_{m,n}$, $B(m,n) = \cH(m,n) - A(m,n)$, $m,n \in \La$. Then, $A$ is invertible with  $|A^{-1}(m,m)| \le \exp(4T\kappa_0^{-1})$ and $A^{-1}(m,n) = 0$ if $m \neq n$. Just as in \eqref{eq:aux00c01a1part1}--\eqref{eq:aux00c01a1part3}, one obtains
\begin{equation}\label{eq:aux00c01a1partAB}
\sum_{t \ge 1} |[A^{-1} \bigl(B A^{-1}\bigr)^{t} ](m,n)| \le \sum_{k \ge 1} \ve_0^{k-1} \sum_{\gamma \in \Ga_{D,T,\kappa_0}(m,n;k,\La)} w_{D,\kappa_0}(\gamma).
\end{equation}
Hence,
\begin{equation}\label{eq:aux00c01a1partAB1}
\left|\cH_\La^{-1}(m,n)\right| \le |A^{-1}(m,n)| + \sum_{t \ge 1} |[A^{-1} \bigl(B A^{-1}\bigr)^{t} ](m,n)| = s_{D,T,\kappa_0,\ve_0;k,\La}(m,n).
\end{equation}
\end{proof}

Now it is very easy to derive the main result of this section which is the ``general multi-scale analysis scheme based on the Schur complement formula'' mentioned in the section title.

\begin{prop}\label{prop:aux1}
Let $(\cH(x,y))_{x,y \in \La}$, $\La \subset \IZ^\nu$ be a matrix, which obeys
$$
\ve_0 w(m,n) := |\cH(x,y)| \le \ve_0 \exp(-\kappa_0 |x-y|)
$$
for any $x \neq y$, $0 < \ve_0 \le \min(2^{-24 \nu-4} \kappa_0^{4\nu}, \exp(-(8T \kappa_0^{-1})^5), 2^{-10(\nu+1)} T^{-8\nu})$. Let $\La_{j}$,  $j \in J$ be subsets of $\La$, $\La_i \cap \La_j = \emptyset$ if $i \neq j$. Let $D_j \in \mathcal{G}_{\La_j,T,\kappa_0}$. Assume that the following conditions hold:
\begin{itemize}

\item[(a)] Each $\cH_{\La_{j}}$ is invertible and
\begin{equation}\label{eq:aux102}
|\cH_{\La_j}^{-1}(m,n)| \le s_{D_j,T,\kappa_0,\ve_0;k,\La_j,\mathfrak{R}}(m,n), \text {for any $m,n \in \La_j$ and any $j$}.
\end{equation}

\item[(b)] For each $n \notin \bigcup_{j \in J} \La_{j}$, $|\cH(n,n)| \ge \exp(-4T \kappa_0^{-1})$.

\end{itemize}
Then,
\begin{equation}\label{eq:aux201}
|\cH_\La^{-1}(m,n)| \le s_{D,T,\kappa_0,\ve_0;k,\La,\mathfrak{R}}(m,n),
\end{equation}
where $D(m) = D_j(m)$ if $m \in \La_j$ for some $j$, and $D(m) = 4T \kappa_0^{-1}$ otherwise.
\end{prop}

\begin{proof}
Note that $D \in \mathcal{G}_{\La,T,\kappa_0}$. Set $\La_0 := \La \setminus \cup_{j \in J} \La_{j}$. Due to Lemma~\ref{lem:aux6},
\begin{equation}\label{eq:auxlambdas1}
|\cH_{\Lambda_{0}}^{-1}(m,n)| \le s_{D,T,\kappa_0,\ve_0;k,\La_0}(m,n).
\end{equation}
Applying repeatedly Lemma~\ref{lem:aux5}, one obtains the statement.
\end{proof}

\begin{remark}\label{rem:2withRnoRest1}
In the last three lemmas and in Proposition~\ref{prop:aux1} we analyze the cases based on the functions $s_{D,T,\kappa_0,\ve_0;\La,\mathfrak{R}}(m,n)$ only. The analysis of the cases based on the functions $s_{D,T,\kappa_0,\ve_0;\La}(m,n)$ is completely similar.
\end{remark}

\begin{lemma}\label{lem:2Qfunction}
Assume that $(\mathcal{H}(m,n))_{m, n \in \mathbb{Z}^\nu}$ obeys
$$
|\mathcal{H}(m,n)| \le \varepsilon_0 \exp( - \kappa_0 |m-n| ), \quad m \not= n,\quad m, n \in \mathbb{Z}^\nu.
$$
Given $\La$ such that $\cH_{\La}$ is invertible, set
\begin{equation}\label{eq:2abrem1a1Qdefinition}
G(m_0, n_0, \La) := \sum_{m, n \in \La} \cH(m_0,m) \cH_{\La}^{-1} (m,n) \cH(n,n_0), \quad m_0, n_0 \in \mathbb{Z}^\nu.
\end{equation}
Assume that $\La_j$, $\La$ are such that all conditions of Lemma~\ref{lem:aux5} hold. Assume also that $R := \dist(\{m_0, n_0\}, \La_2)$ obeys $\max ( \max_{x \in \La} D(x), 4 T \kappa_0^{-1} ) \le \kappa_0 R/8$. Then,
\begin{equation}\label{eq:2QQ0comparison}
|G(m_0, n_0, \La) - G(m_0, n_0, \La_1)| \le 4 |\ve_0|^{3/2} \exp(-\frac {\kappa_0}{4} R).
\end{equation}
\end{lemma}

\begin{proof}
We write
\begin{equation}\label{eq:2abrem1a1Qsplit}
\begin{split}
G(m_0, n_0, \La) = [ \sum_{m, n \in \La_1} + \sum_{m, n \in \La_2} + \sum_{m \in \La_1, n \in \La_2} + \sum_{m \in \La_2, n \in \La_1}] \\
\cH(m_0,m) \cH_{\La}^{-1} (m,n) \cH(n, n_0) := Q_{1,1} + Q_{2,2} + Q_{1,2} + Q_{2,1}.
\end{split}
\end{equation}
Due to the Schur complement formula \eqref{eq:2schurfor}, one has for $x, y \in \La_1$,
\begin{equation}\label{eq:2aschurforKKll}
\cH_{\La}^{-1} (x,y) = [ \cH_{\La_1}^{-1}] (x,y) + [\cH_{\La_1}^{-1} \Gamma_{1,2} \hat H_2^{-1} \Gamma_{2,1}\cH_{\La_1}^{-1} ] (x,y),
\end{equation}
where $\hat H_2 := \cH_{\La_2} - \Ga_{2,1} \cH_{\La_1}^{-1} \Ga_{1,2}$. This implies
\begin{equation}\label{eq:2rem1a1KKLLN}
\begin{split}
|G(m_0, n_0, \La) - G(m_0, n_0, \La_1)| \le  \sum_{m, n \in \La_1} \cH(m_0, m) \cH(n, n_0) \times \\
[ \cH_{\La_1}^{-1} \Gamma_{1,2} \tilde H_2^{-1} \Gamma_{2,1} \cH_{\La_1}^{-1} ] (m,n) \big| + |Q_{2,2}| + |Q_{1,2}| + |Q_{2,1}| := |\tilde Q_{1,1}| + |Q_{2,2}| + |Q_{1,2}| + |Q_{2,1}|.
\end{split}
\end{equation}
Since all conditions of Lemma~\ref{lem:aux5} hold, one can invoke \eqref{eq:aux00c011kappad2statement}. Using the estimate \eqref{eq:aux00c011kappad2statement} combined with the estimate \eqref{eq:auxtrajectweight2}, one obtains
\begin{equation}\label{eq:2auxAAAAKKLL}
\begin{split}
|\tilde Q_{1,1}| \le \ve_0^2 \sum_{m, n \in \La_{m_0}} \exp(-\kappa_0 |m_0 - m| - \kappa_0 |n - n_0|) \times \\
\sum_{q \ge 3} \ve_0^{q-1} \sum_{\gamma \in \Ga^{(1,2)}_{D, T, \kappa_0} (m, n; q, \La,\mathfrak{R})} w_{D, \kappa_0} (\gamma) \le \ve_0^2 \sum_{m, n \in \La_1} \exp(-\kappa_0 |m_0 - m| - \kappa_0 |n - n_0|) \times \\
\sum_{q \ge 3} \ve_0^{q-1} \sum_{\gamma \in \Ga^{(1,2)}_{D, T, \kappa_0} (m, n; q, \La,\mathfrak{R})} \exp(-\frac{1}{2} \kappa_0 \|\gamma\| +
\max(\bar D(\gamma), 4 T\kappa_0^{-1})),
\end{split}
\end{equation}
where $\Ga^{(p,q)}_{D, T, \kappa_0} (m, n; k, \La,\mathfrak{R}) = \bigcup_{t \ge 1} \Ga^{(p,q,t)}_{D, T, \kappa_0} (m, n; k, \La,\mathfrak{R})$, $\Ga^{(p,q,t)}_{D, T, \kappa_0} (m, n; k, \La,\mathfrak{R})$ stands for the set of all $\gamma \in \Ga_{D, T, \kappa_0} (m, n; k, \La,\mathfrak{R})$ such that $\gamma = \gamma_1 \cup \gamma'_1 \dots \cup \gamma_{t+1}$ with $\gamma_j \in \Ga_{D,T,\kappa_0} (\La_p,\mathfrak{R})$, $\gamma'_i \in \Ga_{D,T,\kappa_0} (\La_q,\mathfrak{R})$, $p \neq q$. Note that for any $m, n$ and any $\gamma \in \Ga^{(1,2)}_{D, T, \kappa_0} (m, n; q, \La,\mathfrak{R})$, we have $|m_0 - m| + \|\gamma\| + |n - n_0| \ge 2 \dist(\{m_0, n_0\}, \La_2) = 2 R$. Due to the assumptions of the lemma, $\max(\bar D(\gamma), 4 T\kappa_0^{-1})) \le \kappa_0 R/8$. Combining these estimates, one obtains
\begin{equation}\label{eq:2auxAAAAKKLL1}
\begin{split}
|\tilde Q_{1,1}| \le |\ve_0|^2 \sum_{q \ge 3} |\ve_0|^{q-1} \sum_{\gamma \in \Ga(m_0, n_0; q+2, \La,\mathfrak{R}), \quad \|\gamma\| \ge 2R} \exp(-\frac{1}{4} \kappa_0 \|\gamma\|) \\
\le |\ve_0|^2 \sum_{q \ge 3} |\ve_0|^{q-1} \exp(-\frac{\kappa_0}{4} R) \sum_{\gamma \in \Ga(m_0, n_0; q+2, \La,\mathfrak{R})} \exp(-\frac{1}{8} \kappa_0 \|\gamma\|) \le |\ve_0|^{3/2} \exp(-\frac {\kappa_0}{4} R).
\end{split}
\end{equation}
The estimation of the rest of the terms in \eqref{eq:2rem1a1KKLLN} is completely similar.
\end{proof}

\begin{remark}\label{rem:2withRnoRag}
We remark here that Lemma~\ref{lem:2Qfunction} applies to any $m_0, n_0 \in \IZ^\nu$, regardless of whether $\bar \La = \IZ^\nu$ or $\bar \La \neq \IZ^\nu$ in Remark~\ref{rem:2withRnoR}, provided of course the conditions of the lemma hold. The same applies to Lemma~\ref{lem:2Qfunctionderiv} below.
\end{remark}

\begin{lemma}\label{lem:7differentiationB}
$(1)$ Let $H_\xi = (h(x,y;\xi))_{x, y \in \La}$ be a matrix-function, $\xi \in U \subset \mathbb{R}^d$. Assume that $H_\xi^{-1}$ exists for all $\xi$. If $H_\xi$ is $C^1$-smooth, then $H_\xi^{-1}$ is $C^1$-smooth, and
\begin{equation}\label{eq:7ResderivAAAa}
\partial_{\xi_j} H_\xi^{-1} = H_\xi^{-1} (\partial_{\xi_j} H_\xi) H_\xi^{-1}.
\end{equation}
If $H_\xi$ is $C^2$-smooth, then $H_\xi^{-1}$ is $C^2$-smooth, and
\begin{equation}\label{eq:7ResderivAAB}
\partial^2_{\xi_i, \xi_j} H_\xi^{-1} = H_\xi^{-1} (\partial_{\xi_i} H_\xi) H_\xi^{-1} (\partial_{\xi_j} H_\xi) H_\xi^{-1} + H_\xi^{-1} (\partial^2_{\xi_i,\xi_j} H_\xi) H_\xi^{-1} + H_\xi^{-1} (\partial_{\xi_j} H_\xi) H_\xi^{-1} (\partial_{\xi_i} H_\xi) H_\xi^{-1}.
\end{equation}

$(2)$ Assume that for any $\xi$ and any $x, y \in \La$, we have
\begin{equation}\label{eq:3Hinvestimatestatement1kK}
|H_\xi^{-1} (x,y)| \le s_{D(\cdot; \La), T, \kappa_0, \ve_0; \La,\mathfrak{R}} (x,y), \quad x, y \in \La,
\end{equation}
where $D \in \mathcal{G}_{\La, T, \kappa_0}$. Assume that  $h(m,n;\xi)$ are $C^2$-smooth and for $m \neq n$ obey $|\partial^\alpha h(m,n;\xi)| \le B \exp(-\kappa_0 |m - n|)$ for $|\alpha| \le 2$, where $B > 0$ is a constant. Furthermore, assume that there is $m_0 \in \La$ such that $|\partial^\alpha h(m,m;\xi)| \le B' \exp(\kappa_0 |m - m_0|^{1/5})$ for any $m \in \La$, $0< |\alpha| \le 2$, where $B'>0$ is a constant. Finally, assume that $|h(m,m;\xi)| \le B''$ for any $m \in \La$, where $B'' > 0$ is a constant. Set $B_0 = \max (1,B,B',B'')$. Then, for any $|\beta | \le 2$, and any $n \in \La$, we have
\begin{equation}\label{eq:2HinvestIM}
|\partial^\beta H_{\xi}^{-1} (m,n)| \le (3 B_0)^{|\beta|} \exp(|\beta| \kappa_0 |m - m_0|^{1/5}) \mathfrak{D}^{|\beta|}_{D(\cdot), T, \kappa_0, \ve_0; \La} (m,n), \quad x, y \in \La;
\end{equation}
compare \eqref{eq:auxtrajectweightsumest8iterted} in Lemma~\ref{lem:auxweight1iterated}.
\end{lemma}

\begin{proof}
$(1)$ To verify \eqref{eq:7ResderivAAAa}, assume for convenience $d = 1$, $\xi \in (\xi_1, \xi_2)$. Let $\xi_0 \in (\xi_1, \xi_2)$. For sufficiently small $|\xi - \xi_0|$, one has $\|H_\xi - H_{\xi_0}\| < M (\xi_0) |\xi - \xi_0|$, where $M(\xi_0) = 1 + \|\partial_\xi H_\xi|_{\xi = \xi_0}\|$. In particular, $\|H_\xi - H_{\xi_0}\| \|H_{\xi_0}^{-1}\| < 1/2$ for  sufficiently small $|\xi - \xi_0|$. Hence,
\begin{equation}\label{eq:7Resderiv}
\begin{split}
H_\xi^{-1} - H_{\xi_0}^{-1} = \sum_{t \ge 1} H_{\xi_0}^{-1} [ (H_{\xi_0} - H_\xi) H_{\xi_0}^{-1} ]^t \\
= H_{\xi_0}^{-1} (H_{\xi_0} - H_\xi) H_{\xi_0}^{-1} + R(\xi, \xi_0), \\
\|R(\xi, \xi_0)\| \le \sum_{t \ge 2} \|H_{\xi_0}^{-1}\|^{t+1} \|H_\xi - H_{\xi_0}\|^t \\
\le \|H_{\xi_0}^{-1}\|^{3} \|H_\xi - H_{\xi_0}\|^2 \sum_{t \ge 0} 2^{-t} \le C(\xi_0) (\xi-\xi_0)^2,
\end{split}
\end{equation}
where $C(\xi_0) = M(\xi_0)^2 \|H_{\xi_0}^{-1}\|^{3}$. This implies \eqref{eq:7ResderivAAAa}. The derivation of the rest of the identities is similar.

$(2)$ This part follows from part $(1)$ combined with Lemma~\ref{lem:auxweight1iterated}.
\end{proof}

\begin{lemma}\label{lem:2Qfunctionderiv}
Let $H_\xi = (h(x,y;\xi))_{x, y \in \mathbb{Z}^\nu}$, $\xi \in U \subset \mathbb{R}^d$ be as in part $(2)$ of Lemma~\ref{lem:7differentiationB}. Given $\La' \subset \La$, set $\cH_{\La'} = \cH_{\La',E,\xi} = (E - h(m,n;\xi))_{m, n \in \La'}$. Provided $\cH_{\La}^{-1}$ exists, set
\begin{equation}\label{eq:2abrem1a1Qdefinitionk}
G(m_0, n_0, \La; E, \xi) := \sum_{m, n \in \La} h(m_0, m; \xi) \cH_{\La}^{-1} (m, n) h(n, n_0; \xi), \quad m_0, n_0 \in \mathbb{Z}^\nu.
\end{equation}
Assume that $\La = \La_1 \cup \La_2$, $\La_1 \cap \La_2 = \emptyset$, and for any $E \in (E',E'')$, $\xi$, $\cH_{\La}$, $\cH_{\La_j}$ obey all conditions of  Lemma~\ref{lem:aux5}. Assume also that $R := \dist(\{m_0, n_0\}, \La_2)$ obeys $\max ( \max_{x \in \La} D(x), 4 T \kappa_0^{-1} ) \le \kappa_0 R/8$. Finally, assume that $(E', E'') \subset (-B_0, B_0)$. Then, for any multi-index $|\beta| \le 2$, we have
\begin{equation}\label{eq:2QQ0comparison1}
|\partial^\beta G(m_0, n_0, \La; E, \xi) - \partial^\beta G(m_0, n_0, \La_1; E, \xi)| \le 780 B_0^2 \ve_0^{3/2} \exp(-\frac {\kappa_0}{4} R).
\end{equation}
\end{lemma}

\begin{proof}
We use the notation from the proof of Lemma~\ref{lem:2Qfunction} with $E,\xi$ being suppressed. Using the notation from \eqref{eq:2rem1a1KKLLN}, one obtains
\begin{equation}\label{eq:2rem1a1KKLLNPartial}
|\partial^\beta \bigl( G(m_0, n_0, \La) - G(m_0, n_0, \La_1) \bigr)| \le |\partial^\beta \tilde Q_{1,1}| + |\partial^\beta Q_{2,2}| + |\partial^\beta Q_{1,2}| + |\partial^\beta Q_{2,1}|.
\end{equation}
Using \eqref{eq:2HinvestIM} from part $(2)$ of Lemma~\ref{lem:7differentiationB}, one obtains
\begin{equation}\label{eq:3HinvestIMappl}
\begin{split}
|\partial^\beta \cH_{\La}^{-1} (m, n)|, |\partial^\beta \tilde H_2^{-1} (m, n) | \le (3 B_0)^{|\beta|} \mathfrak{D}^{|\beta|}_{D(\cdot), T, \kappa_0, \ve_0; \La} (m,n), \\
|\partial^\beta \bigl(\cH_{\La}^{-1} \Gamma_{1,2} \tilde H_2^{-1} \Gamma_{2,1}\cH_{\La_1}^{-1} (m,n) \bigr)| \le (195 B_0)^{|\beta|} \mathfrak{D}^{|\beta|}_{D(\cdot), T, \kappa_0, \ve_0; \La} (m,n).
\end{split}
\end{equation}
Now, using \eqref{eq:3HinvestIMappl} just like in \eqref{eq:2auxAAAAKKLL}, \eqref{eq:2auxAAAAKKLL1}, one obtains $|\partial^\beta \tilde Q_{1,1}| \le 195 B_0 \ve_0^{3/2} \exp(-\frac {\kappa_0}{4} R)$. The estimation of the rest of the terms in \eqref{eq:2rem1a1KKLLNPartial} is similar.
\end{proof}

\section{Eigenvalues and Eigenvectors of Matrices with Inessential Resonances of Arbitrary Order}\label{sec.3}

Let $\La$ be a non-empty subset of $\zv$. Let $v(n)$, $n \in \La$, $h_0(m, n)$, $m, n \in \La$, $m \ne n$ be some complex functions. Consider $H_{\La,\ve} = \bigl(h(m, n; \ve)\bigr)_{m, n \in \La}$, where $\ve \in \IC$,
\begin{alignat}{2}
h(n, n; \ve) & = v(n), & \quad & n \in \La, \label{eq:2-1} \\[6pt]
h(m, n; \ve) & = \ve h_0(m, n), & & m, n \in \La,\ m \ne n. \nn
\end{alignat}
Assume that the following conditions are valid,
\begin{gather}
v(n) = \overline{v(n)}, \label{eq:2-2} \\[6pt]
h_0(m,n) = \overline{h_0(n,m)}, \label{eq:2-3} \\[6pt]
|h_0(m, n)| \le B_1 \exp (-\ka_0 |m - n|)\ ,\quad m, n \in \La,\ m \ne n, \label{eq:2-4}
\end{gather}
where $0<B_1 < \infty$, $\ka_0 >0$ are constants, $\big| (z_1, z_2, \dots, z_\nu) \big| = \sum_j |z_j|$, $z_j \in \IC$. \textit{For convenience we always assume that $0<B_1 \le 1$, $0 < \ka_0 \le 1/2$}.

Take an arbitrary $m_0 \in \La$. For $\ve = 0$, the matrix $H_{\La,\ve}$ has an eigenvalue $E_0 = v(m_0)$, and $\vp_0(n) = \delta_{m_0, n}$, $n \in \La$ is the corresponding eigenvector. Assume that
\begin{equation}\label{eq:2-5}
\inf \left\{ |v(n) - v(m_0)|: n \in \La,\ n \ne m_0 \right\} \ge \delta_0 > 0.
\end{equation}
In this case, elementary perturbation theory yields the following:

\medskip
\noindent
{\it There exist $\ve_0 > 0$ and analytic functions $E(\ve)$, $\vp(n, \ve)$ defined in the disc $\cD(0, \ve_0) = \left\{ \ve \in \IC: |\ve| < \ve_0\right\}$, $n \in \La$ such that}
\begin{gather}
\sum_{n \in \La}\ |\vp(n, \ve)|^2 = 1\ ,\quad \text{for}\ \ve \in (-\ve_0, \ve_0), \label{eq:2-6} \\[6pt]
H_\ve \vp(n, \ve) = E(\ve) \vp(n, \ve)\label{eq:2-7}, \\[6pt]
E(0) = E_0\ ,\quad \vp(n, 0) = \vp_0(n).\label{eq:2-8}
\end{gather}

Let $\hle = \bigl(h(m, n;\ve)\bigr)_{m, n \in \La}$ be defined as in \eqref{eq:2-1} {\it In this section we will analyze some cases where the basic non-resonance condition \eqref{eq:2-5} does not hold for the matrix $\hle$, but it does hold for some smaller matrices $H_{\La',\ve}$, $\La'\subset \La$. More specifically, we will assume that there is some structure of such smaller matrices. This idea leads to an inductive definition of classes of matrices $\cN^{(s')} \bigl( m'_0, \La'; \delta_0 \bigr)$, which we introduce here.}

\bigskip

The idea of analytic continuation in the parameter $\ve$ is absolutely crucial in the further development of the method. This development addresses the so-called cases of pairs of resonances. In this section we establish all estimates related to the analytic dependence on the parameter $\ve$ needed later in the applications. On the other hand, the analytic dependence itself helps to avoid certain ambiguities in the very definitions in this section. Let us first recall Rellich's theorem on the analytic dependence of the eigenvalues of self-adjoint matrices:

$\bullet$ Let $A_\ve = (a_{m,n}(\ve))_{1 \le m, n \le N}$ be an analytic matrix function defined in a neighborhood of the interval $\ve_1 < \ve < \ve_2$. Assume that for $\ve \in (\ve_1,\ve_2)$, the matrix $A_\ve$ is self-adjoint. Then, there exist real analytic functions $E_n(\ve)$, $\ve \in (\ve_1, \ve_2)$, such that for each $\ve$, $\spec A_\ve = \{ E_n(\ve) : 1 \le n \le N \}$. In particular, assume that for some $\ve^\zero$, the matrix has a simple eigenvalue $E^\zero$. Then, there is unique $E_{n_0} (\ve)$ such that $E_{n_0} (\ve^\zero) = E^\zero$.

\begin{defi}\label{def:4-1}
Assume that $H_{\La,\ve}$ obeys \eqref{eq:2-1}--\eqref{eq:2-5},
\begin{equation} \label{eq:2epsilon0}
|\ve| < \ve_0, \quad \ve_0 := \ve_0(\delta_0,\kappa_0) := (\bar \ve_0)^3, \quad \bar \ve_0 := \min(2^{-24\nu-4}\kappa_0^{4\nu}, \delta_0^{2^9}, 2^{-10(\nu+1)} (4 \kappa_0 \log \delta_0^{-1})^{-8\nu}).
\end{equation}

For these values of $\ve$, we say that $H_{\La,\ve}$ belongs to the class $\cN^{(1)} \bigl( m_0, \La; \delta_0 \bigr)$.

Let $0 < \beta_0 < 1$ be a constant. We assume that $\log \delta_0^{-1} > 2^{32} \beta_0^{-1} \log \kappa_0^{-1}$. Introduce the following quantities:
\be\label{eq:3basicparameters}
R^{(1)} := \bigl( \delta_0 \bigr)^{-4\beta_0}, \quad R^{(u)} := \bigl( \delta_0^{(u-1)} \bigr)^{-\beta_0}, \; u = 2,3,\dots, \quad \delta_0^{(u)} = \exp \bigl( - (\log R^{(u)})^2 \bigr), \; u = 1,2,\dots .
\end{equation}

Assume that the classes $\cN^{(s')} \bigl( m'_0, \La'; \delta_0 \bigr)$ are already defined for $s' = 1, \dots, s-1$, where $s\ge 2$.

Assume that $H_{\La,\ve}$ obeys \eqref{eq:2-1}--\eqref{eq:2-4}. Let $m_0 \in \IZ^\nu$. Assume that there exist subsets $\cM^{(s')} (\La) \subset \La$, $s' = 1, \dots, s-1$, some of which may be empty, and a collection of subsets $\La^{(s')} (m) \subset \La$, $m \in \cM^{(s')}$, such that the following conditions hold:
\begin{itemize}

\item[(a)] $m_0 \in \cM^{(s-1)} (\La)$, $m \in \La^{(s')}(m)$ for any $m \in \cM^{(s')} (\La)$, $s' \le s-1$.

\item[(b)] $\cM^{(s')} (\La) \cap \cM^{(s'')} (\La) = \emptyset$ for any $s' < s''$. For any $(m', s') \neq (m'', s'')$, we have
$$
\Lambda^{(s')} (m') \cap \Lambda^{(s'')} (m'') = \emptyset.
$$

\item[(c)] For any $s' = 1, \dots, s-1$ and any $m \in \cM^{(s')}(\La)$, the matrix $H_{\La^{(s')} (m), \ve}$ belongs to $\cN^{(s')} \bigl( m, \La^{(s')}(m); \delta_0 \bigr)$. Note that, in particular, this means that for the set $\La^{(s')} (m)$, a system of subsets $\cM^{(s')} (\La^{(s')} (m)) \subset \La^{(s')}(m)$, $s'' = 1, \dots, s'$, and $\La^{(s'')} (m) \subset \La^{(s')} (m)$, $m \in \cM^{(s')} (\La^{(s')} (m))$ is defined so that all the conditions stated above and below are valid for $H_{\La^{(s')} (m), \ve}$ in the role of $\hle$, $s'$ in the role of $s$, and $m$ in the role of $m_0$.

\item[(d)]
$$
\bigl( m' + B(R^{(s')}) \bigr) \subset \Lambda^{(s')} (m'), \quad \text {for any $m' \in \cM^{(s')} (\La)$, $s' < s$}.
$$

$$
\bigl( m_0 + B(R^{(s)}) \bigr) \subset \Lambda.
$$

\item[(e)] For any $n \in \La \setminus \{ m_0 \}$, we have $v(n) \neq v(m_0)$. So, $E^{(s)} (m_0, \La; 0) := v(m_0)$ is a simple eigenvalue of $H_{\La,0}$. Let $E^{(s)} \bigl( m_0, \La; \ve \bigr)$, $\ve \in \IR$, be the real analytic function such that $E^{(s)} \bigl( m_0, \La; \ve \bigr) \in \spec H_{\La,\ve}$ for any $\ve$, $E^{(s)} \bigl( m_0, \La; 0 \bigr) = v(m_0)$. Similarly, for any $m \in \cM^{(s')}(\La)$, and $n \in \La^{(s')} (m) \setminus \{m\}$, we have $v(n) \neq v(m)$. So, $E^{(s')} (m, \La^{(s')} (m); 0) := v(m)$ is a simple eigenvalue of $H_{\La^{(s')} (m), 0}$. Let $E^{(s')} \bigl( m, \La^{(s')} (m); \ve \bigr)$, $\ve \in \IR$, be the real analytic function such that $E^{(s')} \bigl( m, \La^{(s')} (m); \ve \bigr) \in \spec H_{\La^{(s')} (m), \ve} $ for any $\ve$, $E^{(s')} \bigl( m, \La^{(s')} (m); 0 \bigr) = v(m)$. Set
$$
\ve_s = \ve_0 - \sum_{1 \le s' \le s} \delta_0^{(s')}, \; s\ge 1.
$$
If $s = 1$, we will show in Proposition~\ref{prop:4-4} that $E^{(1)} \bigl( m_0, \La; \ve \bigr)$ can be extended analytically in
the disk $|\ve| < \ve_0$. For $s = 2$, it is required by the current definition that for all complex $\ve$, $|\varepsilon| < \varepsilon_{0}$, we have
\begin{equation}\label{eq:4-3}
3 \delta_0^{(1)} \le \big| E^{(1)} \bigl( m, \La^{(1)}(m); \ve \bigr) - E^{(1)} \bigl( m_0, \La^{(1)}(m_0); \ve \bigr) \big| \le  \delta^\zero_0 := \delta_0/8.
\end{equation}
We show in Proposition~\ref{prop:4-4} that in this case, $E^{(2)} \bigl( m_0, \La; \ve \bigr)$ can be extended analytically in the disk $|\ve| < \ve_2$. Using induction we prove in Proposition~\ref{prop:4-4} that this is true for all $s$. For $s \ge 3$, we require that for all  $\ve \in \mathbb{C}$ with $|\varepsilon| < \varepsilon_{s-2}$, we have
\begin{equation}\label{eq:4-3sge3}
\begin{split}
3 \delta_0^{(s-1)} \le \big| E^{(s-1)} \bigl( m, \La^{(s-1)} (m); \ve \bigr) - E^{(s-1)} \bigl( m_0, \La^{(s-1)}(m_0); \ve \bigr) \big| <  \delta_0^{(s-2)} ,\quad \text{for $m \neq m_0$}, \\
\frac{\delta_0^{(s')}}{2} \le \big| E^{(s')} \bigl( m, \La^{(s')} (m); \ve \bigr) - E^{(s-1)} \bigl( m_0, \La^{(s-1)} (m_0); \ve \bigr) \big| < \delta_0^{(s'-1)}, \quad \text{for $s' = 1, \dots, s-2$}.
\end{split}
\end{equation}

\item[(f)] For $s = 1$, we have $|v(n) - v(m_0)| \ge \delta_0/4$ for every $m \neq m_0$. For $s \ge 2$, we have $|v(n) - v(m_0)| \ge (\delta_0)^4$ for every $n \in \Lambda \setminus \bigl( \bigcup_{1 \le s' \le s-1} \bigcup_{m \in \cM{(s')}} \La^{(s')} (m) \bigr)$.

\end{itemize}

In this case we say that $H_{\La, \varepsilon}$ belongs to the class $\cN^{(s)} \bigl( m_0, \La; \delta_0 \bigr)$. We call $m_0$ the principal point. We set $s(m_0)=s$. We call $m_0$ the principal point. We call $\La^{(s-1)}(m_0)$ the $(s-1)$-set for $m_0$.
\end{defi}

\begin{remark}\label{rem:3.Rs}
Note that in particular
$$
\kappa_0 (R^\es)^{1/16} > \log (\delta_0^\es)^{-1}, \delta^\es < (\delta^\esone)^{8}, \quad \delta_0^\one < \ve_0/2.
$$
\end{remark}

\begin{prop}\label{prop:4-4}
Let $E^{(s')} (m, \La^{(s')}(m); \ve)$ be the same as in Definition~\ref{def:4-1}, $m \in \cM^{(s')}$, $s' = 1, \dots, s-1$. The following statements hold:
\begin{itemize}

\item[(1)] Define inductively the functions $D(\cdot; \La^{(s')} (m))$, $1 \le s' \le s-1$, $m \in \cM{(s')}$, $D(\cdot; \La)$ by setting for $s = 1$, $D(x; \La) = 4 \log \delta_0^{-1}$ for $x \in \La \setminus \{m_0\}$, $D(m_0; \La) := 4 \log (\delta^\one)^{-1}$; and by setting for $s \ge 2$, $D(x; \La) = D(x; \La^{(s')} (m))$ if $x \in \La^{(s')} (m)$ for some $s' \le s-1$ and some $m \in \cM{(s')} \setminus \{m_0\}$, $D(x; \La) = D(x; \La^{(s-1)} (m_0))$ if $x \in \La^{(s-1)} (m_0) \setminus \{m_0\}$, $D(m_0; \La) = 2 \log (\delta^{(s)}_0)^{-1}$, $D(x; \La) = 4 \log \delta_0^{-1}$ if $x \in \Lambda \setminus \bigl( \bigcup_{1 \le s'\le s} \bigcup_{m \in \cM{(s')}} \La^{(s')} (m) \bigr)$. Then, $D(\cdot; \La^{(s')} (m)) \in \mathcal{G}_{\La^{(s')} (m), T, \kappa_0}$, $1 \le s' \le s-1$, $m \in \cM{(s')}$, $D(\cdot; \La) \in \mathcal{G}_{\La, T, \kappa_0}$, $T = 4 \kappa_0 \log \delta_0^{-1}$, $\max_{x \neq m_0} D(x; \La) \le 4 \log (\delta^{(s-1)}_0)^{-1}$. We will denote by $D(\cdot; \La \setminus \{m_0\})$ the restriction of $D(\cdot; \La)$ to $\La \setminus \{m_0\}$.

\item[(2)] For $s = 1$, the matrix $(E - H_{\La \setminus \{m_0\}, \ve})$ is invertible for any $|\ve| < \bar \ve_0$, $|E - v(m_0)| < \delta_0/4$. For $s \ge 2$, $|\ve| < \ve_{s-2}$, and $\big|E - E^{(s-1)} (m_0, \La^{(s-1)}(m_0); \ve) \big| < 2 \delta^{(s-1)}_0$, the matrices $(E - H_{\La^{(s')} (m), \ve})$, $s' \le s-1$, $m \in \cM^{(s')}$, $m \neq m_0$ and the matrices $(E - H_{\La^{(s-1)} (m_0) \setminus \{m_0\}, \ve})$, $(E - H_{\La \setminus \{m_0\}, \ve})$ are invertible. Moreover,
\begin{equation}\label{eq:3Hinvestimatestatement1}
\begin{split}
|[(E - H_{\La^{(s')} (m), \ve})^{-1}] (x,y)| \le s_{D(\cdot; \La^{(s')} (m)), T, \kappa_0, |\ve|; \La^{(s')} (m)} (x,y), \\
|[(E - H_{\La^{(s-1)} (m_0) \setminus \{m_0\}, \ve})^{-1}] (x,y)| \le s_{D(\cdot; \La^{(s-1)} (m_0) \setminus \{m_0\}), T, \kappa_0, |\ve|; \La^{(s-1)} (m_0) \setminus \{m_0\}} (x,y), \\
|[(E - H_{\La \setminus \{m_0\}, \ve})^{-1}] (x,y)| \le s_{D(\cdot; \La \setminus \{m_0\}), T, \kappa_0, |\ve|; \La \setminus \{m_0\}} (x,y).
\end{split}
\end{equation}

\item[(3)] Set $\La_{m_0} := \Lambda \setminus \{m_0\}$. The functions
\begin{equation}\label{eq:4-10ac}
\begin{split}
K^{(s)} (m, n, \La_{m_0}; \ve, E) = (E - H_{\La_{m_0}, \ve})^{-1} (m,n), \quad m, n \in \La_{m_0}, \\
Q^{(s)} (m_0, \La; \ve, E) = \sum_{m', n' \in \La_{m_0}} h(m_0, m'; \ve) K^{(s)} (m', n'; \La_{m_0}; \ve, E) h(n', m_0; \ve), \\
F^{(s)} (m_0, n, \La_{m_0}; \ve, E) = \sum_{m \in \La_{m_0}} K^{(s)} (n, m, \La_{m_0}; \ve, E) h(m, m_0; \ve), \quad n \in \La_{m_0}
\end{split}
\end{equation}
are well-defined and analytic in the following domain,
\begin{equation}\label{eq:4.domain}
\begin{split}
|\ve| < \bar \ve_0, \quad |E - v(m_0)| < \delta_0/4, \quad \text { in case $s = 1$}, \\
|\ve| < \ve_{s-2} := \ve_0 - \sum_{1 \le s' \le s-2} \delta^{(s')}_0, \quad \big| E - E^{(s-1)}(m_0, \La^{(s-1)} (m_0); \ve) \big| < 2 \delta^{(s-1)}_0, \quad s \ge 2.
\end{split}
\end{equation}
The following estimates hold,
\begin{equation}\label{eq:4-12acACC}
\begin{split}
& \big| Q^{(s)} (m_0, \La; \ve, E) - Q^{(s-1)} \bigl( m_0, \La^{(s-1)}(m_0); \ve, E \bigr) \big| \le 4 |\ve|^{3/2} \exp \left( -\kappa_0 R^{(s-1)} \right) \le |\ve| (\delta_0^{(s-1)})^6 \\
& \text{for} \quad |\ve| < \ve_{s-2}, \quad \big| E - E^{(s-1)} (m_0, \La^{(s-1)} (m_0); \ve) \big| < 2 \delta^{(s-1)}_0, \quad s \ge 2,\\
& \big| \partial_\ve Q^{(1)} (m_0, \La; \ve, E) \big| \le |\ve|^{1/2}, \quad \big| \partial_E^\alpha Q^{(1)} (m_0, \La; \ve, E) \big| \le |\ve| \\
&\text {for $\alpha \le 2$ and any $|\ve| < \ve_0$, $|E - v(m_0) | < \delta_0/8$ },\\
& \big| \partial_\ve Q^{(s)} (m_0, \La; \ve, E) \big| \le |\ve|^{1/2}, \quad \big| \partial^\alpha_E Q^{(s)} (m_0, \La; \ve, E) \big| \le |\ve| \\
&\text {for $\alpha \le 2$ and any $|\ve| < \ve_{s-1}$, $\big|E - E^{(s-1)} (m_0, \La^{(s-1)} (m_0); \ve) \big| < 3 \delta^{(s-1)}_0/2$ },
 \end{split}
\end{equation}

\begin{equation}\label{eq:4-11acACC}
\begin{split}
\big| F^{(s)} (m_0, n, \La; \ve, E) \big| \le 4 |\ve|^{1/2} \exp \left( -\frac{7\kappa_0}{8} |n - m_0| \right), \\
\big| F^{(s)} (m_0, n, \La; \ve, E) - F^{(s-1)} \bigl( m_0, n, \La^{(s-1)}(m_0); \ve, E \bigr) \big| \le  |\ve|^{1/2} \exp \left( -\kappa_0 R^{(s-1)} \right), \\
|\ve| < \ve_{s-2}, \quad \big| E - E^{(s-1)} (m_0, \La^{(s-1)}(m_0); \ve) \big| < 2 \delta^{(s-1)}_0, \quad s \ge 2, \\
\big| F^{(1)} (m_0, n, \La; \ve, E) \big| \le 4 |\ve|^{1/2} \exp \left( -\frac{7\kappa_0}{8} |n - m_0| \right), \\
\big| \partial_\ve F^{(1)} (m_0, n, \La; \ve, E) \big| \le \bar \ve_0^{-1/2}, \quad \big| \partial^\alpha_E F^{(1)}(m_0, n, \La; \ve, E) \big| \le |\ve|^{1/2} \\
\text{for $\alpha \le 2$ and any $|\ve| < \ve_0$, $|E - v(m_0) | < \delta_0/8$ }, \\
\big| \partial_\ve F^{(s)} (m_0, n, \La; \ve, E) \big| \le \bar \ve_0^{-1/2}, \quad \big| \partial^\alpha_E F^{(s)} (m_0, n, \La; \ve, E) \big| \le |\ve|^{1/2} \\
\text{for $\alpha \le 2$ and any } |\ve| < \ve_{s-1}, \, \big| E - E^{(s-1)}(m_0, \La^{(s-1)}(m_0); \ve) \big| < 3 \delta^{(s-1)}_0/2.
\end{split}
\end{equation}

\item[(4)] For $s = 1$ and $|\ve| < \ve_0$, the equation
\begin{equation}\label{eq:4-16}
E = v(m_0) + Q^{(s)} (m_0, \La; \ve, E)
\end{equation}
has a unique solution $E = E^{(1)}(m_0, \La; \ve)$ in the disk $\big| E - v (m_0) \big | < \delta_0/8$. For $s \ge 2$ and $|\ve| < \ve_{s-1}$, the equation~\eqref{eq:4-16} has a unique solution $E = E^{(s)}(m_0, \La; \ve)$ in the disk $\big| E - E^{(s-1)} (m_0, \La^{(s-1)}(m_0); \ve\bigr) \big | < 3 \delta^{(s-1)}_0/2$. This solution is a simple zero of $\det(E - H_{\La, \ve})$. Furthermore, $\det(E - H_{\La, \ve})$ has no other zeros in the disk $|E - E^{(s-1)} (m, \La^{(s-1)}(m); \ve)| < 2 \delta^{(s-1)}_0$. The function $E^{(s)} (m_0, \La; \ve)$ is analytic in the disk $|\ve| < \ve_{s-1}$ and obeys
\begin{equation}\label{eq:4-17AAAA}
\begin{split}
\big| E^{(s)} (m_0, \La; \ve) - E^{(s-1)} \bigl( m_0, \La^{(s-1)} (m_0); \ve \bigr) \big| < |\ve| (\delta_0^{(s-1)})^5, \\
\big| E^{(s)} (m_0, \La; \ve) - v(m_0)| \big| < |\ve|.
\end{split}
\end{equation}
Finally,
\begin{equation}\label{eq:4-13acAAAA}
\big| v(m_0) + Q^{(s)} (m_0, \La; \ve, E) - E^{(s)} \bigl( m_0, \La; \ve \bigr) \big| \le  |\ve| \big| E - E^{(s)} \bigl( m_0, \La; \ve \bigr) \big|.
\end{equation}

\item[(5)] For $s = 1$, $|\ve| < \ve_0$, and $(\delta_0^\one)^4 < \big| E - E^{(1)} (m_0, \La; \ve)\big | < \delta_0/16$, the matrix $(E - H_{\La, \ve})$ is invertible. For $s \ge 2$, $|\ve| < \ve_{s-1}$, and $(\delta_0^\es)^4 < \big| E - E^{(s)}(m_0, \La; \ve) \big| < 2 \delta^{(s-1)}_0$, the matrix $(E - H_{\La, \ve})$ is invertible. Moreover,
$$
|[(E - H_{\La, \ve})^{-1}] (x,y)| \le S_{D(\cdot; \La), T, \kappa_0, |\ve|; k, \La} (x,y).
$$

\item[(6)] The vector $\vp^\es(\La; \ve) := (\vp^{(s)} (n, \La; \ve))_{n \in \La}$, given by $\vp^{(s)} (m_0, \La; \ve) = 1$ and $\vp^{(s)} (n, \La; \ve) = - F^{(s)}(m_0, n, \La; \ve, E^{(s)} (m_0, \La; \ve))$ for $n \ne m_0$, obeys
\begin{equation}\label{eq:4-21ACC}
H_{\La, \ve} \vp^{(s)} (\La; \ve) = E^{(s)}(m_0, \La; \ve) \vp^{(s)}(\La; \ve),
\end{equation}
\begin{equation}\label{eq:4-11evdecay}
\begin{split}
|\vp^{(s)} (n, \La; \ve) | \le 4 |\ve|^{1/2} \exp \left( -\frac{7\kappa_0}{8} |n - m_0| \right), \quad n \neq m_0, \\
\vp^{(s)} (m_0, \La; \ve) = 1.
\end{split}
\end{equation}
Furthermore, let $P(m_0, \La; \ve)$ be the Riesz projector onto the one-dimensional subspace $\mathbb{C} \vp^\es (\La; \ve)$ $($see \eqref{eq:4ResRes}$)$, and let $\delta_{m_0} := (\delta_{m_0,x})_{x \in \La}$. Then, $\|P(m_0, \La; \ve) \delta_{m_0}\| \ge 2/3$. Finally,
\begin{equation}\label{eq:4-11acACCev}
|\vp^\es (n, \La; \ve) - \vp^\esone (n, \La^\esone (m_0); \ve) \le 2|\ve| (\delta_0^{(s-1)})^5, \quad n \in \La^\esone(m_0).
\end{equation}

\end{itemize}
\end{prop}

\begin{proof}
The proof of $(1)$--$(5)$ goes simultaneously by induction over $s$, starting with $s = 1$.

We will prove now $(1)$--$(5)$ in the case $s = 1$. Let $E \in \IC$ be such that
\begin{equation}\label{eq:4-spectralgap}
|E - v(m_0)| \le \delta_0/4.
\end{equation}
Set
\begin{equation}\label{eq:4inddef0}
\cH_{\La_{m_0}} = E - H_{\La_{m_0}, \ve}.
\end{equation}
Clearly, $D(\cdot; \La_{m_0}) \in \mathcal{G}_{\La_{m_0}, T, \kappa_0}$. Set $D(x; \La) = D(x; \La_{m_0}) := 2 \log \delta_0^{-1}$ if $x \in \La_{m_0}$, and $D(m_0; \La) = 2 \log (\delta^{(1)}_0)^{-1}$. Due to condition $(d)$ in Definition~\ref{def:4-1}, one has $\mu_\La (m_0) \ge R^{(1)}$. So,
$$
D(m_0; \La) = 2 \log (\delta^{(1)}_0)^{-1} = 2 \log \exp \bigl( (\log R^{(1)})^{2} \bigr) = (\log R^{(1)})^{2} < (R^\one)^{1/5} < \mu_\La (m_0)^{1/5}.
$$
Hence, $D(\cdot; \La) \in \mathcal{G}_{\La, T, \kappa_0}$. This finishes the proof of $(1)$ in case $s = 1$.

One has $|\cH_{\La_{m_0}}(n,n)| \ge \delta_0/4$ for each $n \in \La_{m_0}$. Due to Lemma~\ref{lem:aux6}, we have for $|\ve| \le \bar \ve_0$,
$$
|\cH_{\La_{m_0}}^{-1}(m,n)| \le s_{D(\cdot;\La_{m_0}),T,\kappa_0,|\ve|;\La_{m_0}}(m,n).
$$
Due to \eqref{eq:auxtrajectweightsumest8iterted} from Lemma~\ref{lem:auxweight1iterated}, one has $|Q^{(1)}(m_0, \La; \ve, E)| < |\varepsilon|^{3/2}$. It follows from Cramer's rule that $K^{(1)}(m,n, \La_{m_0}; \ve, E)$ is analytic wherever it is defined. Thus, $K^{(1)}(m,n, \La_{m_0}; \ve, E)$, $Q^{(1)}(m_0, \La; \ve, E)$ are analytic in the domain
\begin{equation}\label{eq:4domain1}
|\ve| < \bar \ve_0, \quad |E - v(m_0)| < \delta_0/4.
\end{equation}
Using Cauchy estimates for analytic functions, one obtains
\begin{equation}\label{eq:4derivest}
|\partial_\ve Q^{(1)}(m_0, \La; \ve, E)| < \frac{1}{2} \bar \ve_0^{-1} |\varepsilon|^{3/2} < |\ve|^{1/2},
\end{equation}
provided
\begin{equation}\label{eq:4domain1N}
|\ve| < \ve_0, \quad |E - v(m_0| < \delta_0/8.
\end{equation}
The verification of the rest of \eqref{eq:4-12acACC} and \eqref{eq:4-11acACC} with $s=1$ is completely similar. This finishes the proof of $(2)$ and $(3)$ in the case $s=1$.

Let $|E - v(m_0)| < \delta_0/4$. Due to the Schur complement formula, $\cH_\La := E - H_{\La,\ve}$ is invertible if and only if
\begin{equation}\label{eq:4characteristicequation1}
\tilde H_2 = E - v(m_0) - \sum_{m,n \in \La_{m_0}} (-\ve h(m_0,m)) K^{(1)}(m,n, \La_{m_0}; \ve, E) (-\ve h(n,m_0)) \neq 0.
\end{equation}
Moreover,
\begin{equation}\label{eq:4schurfor}
\cH_\La^{-1} = \begin{bmatrix} \cH_{\La_{m_0}}^{-1} + \cH_{\La_{m_0}}^{-1} \Gamma_{1,2} \tilde H_2^{-1} \Gamma_{2,1} \cH_{\La_{m_0}}^{-1} & - \cH_{\La_{m_0}}^{-1} \Gamma_{1,2} \tilde H_2^{-1} \\[5pt] -\tilde H_2^{-1} \Gamma_{2,1} \cH_{\La_{m_0}}^{-1} & \tilde H_2^{-1}\end{bmatrix}.
\end{equation}
In other words, if $|E - v(m_0)| < \delta_0/4$, then $E \in \spec H_{\La,\ve}$ if and only if it obeys
\begin{equation}\label{eq:4chracteristic1}
E = v(m_0) + Q^{(1)}(m_0, \La; \ve, E).
\end{equation}
To solve the equation \eqref{eq:4chracteristic1}, we invoke part $(2)$ of Lemma~\ref{em:implicitanalyticont} with
\begin{equation}\label{eq:4equationdomain1}
\begin{split}
\phi_0(z) \equiv 0 ,\quad z_0 = 0, \quad \sigma_0  = \bar \ve_0, \\
f(z,w)  = Q^{(1)} (m_0, \La; z,v(m_0)+ w), \quad \big| w -\phi_0(z) \big| < \delta_0/4, \quad  \quad \rho_0  = \delta_0/4.
\end{split}
\end{equation}
Note that
\begin{equation}\label{eq:4equationequivalent}
f(z,w) = w \quad \text{if and only if $E = v(m_0) + w$ obeys equation \eqref{eq:4chracteristic1}}.
\end{equation}
One has
\begin{equation}\label{eq:4equationauxestimate}
|f(z,w) - \phi_0(z)| = |Q^{(1)} (m_0, \La; z,v(m_0)+ w)| < |\ve|^{3/2} \quad \text{for any $| w -\phi_0(z) \big| < \delta_0/4$}.
\end{equation}
Using Cauchy inequalities for the derivatives, one obtains
\begin{equation}\label{eq:4equationauxestimate2}
|\partial_w f (z, w)| < (\delta_0/8)^{-1}|\ve|^{3/2} <1/2 \quad \text{for any $| w -\phi_0(z) \big| < \delta_0/8$}.
\end{equation}
As in part $(2)$ of Lemma~\ref{em:implicitanalyticont}, set
\begin{align*}
M_0 & = \sup_z |\phi_0(z)| + \rho_0 + \sup_{z,w} |f(z,w)|, \\
M_1 & = \max(1,M_0), \quad \varepsilon_1 = \frac {\sigma_0^2 \rho_0^2}{10^{10} M_1^3 (1 + \log (\max (100, M_1)))^2}.
\end{align*}
We have $M_0 < 1$, $M_1 = 1$. This implies $\varepsilon_1 > \bar \ve_0^3 =\ve_0 > |\ve|^{3/2}$. On the other hand, $|f(z,\phi_0(z)) - \phi_0(z)| < |\ve|^{3/2}$. Thus, conditions $(\alpha)$, $(\beta)$ from the part $(2)$ of Lemma~\ref{em:implicitanalyticont} both hold. Therefore the equation $f(z,w) = w$ has a unique solution $w = w(z)$. Set $E^{(1)} (m_0, \La; z) := v(m_0) + w(z)$. The function $E^{(1)} (m_0, \La; z)$ is defined and analytic for $|z - z_0| < \ve_0 $. Moreover,
\begin{equation}\label{eq:4equationsolved1}
\begin{split}
v(m_0) + Q^{(1)} (m_0, \La; z,E^{(1)} (m_0, \La; z)) = E^{(1)} (m_0, \La; z), \\
|E^{(1)} (m_0, \La; z) - v(m_0)| < |\ve|^{3/2} < |\ve|(\delta^{(0}_0)^6.
\end{split}
\end{equation}

Clearly, $E^{(1)}(m_0, \La; \ve)$ is a  zero of $\det(E - H_{\La,\ve})$, and at $\ve = 0$ it obeys $E^{(1)}(m_0, \La^{(1)}(m_0);0) = v(m_0)$. Furthermore, $\det(E - H_{\La,\ve})$ has no other zeros in the disk $|E - v(m_0)| < 3 \delta^\zero_0/2$. Note that
\begin{equation}\label{eq:4equationsolved1AAAAA}
\begin{split}
|E - v(m_0) - Q^{(1)}(m_0, \La; \ve, E)| = |w - f(\ve,w)| = |[w - f(\ve,w)] - [w_0 - f(\ve,w_0)]| \\
\ge \min \partial_w [w - f(\ve,w)] |w - w_0| \ge \frac{1}{2} |w - w_0| = \frac{1}{2} |E - E^{(1)} (m_0, \La; \ve))| \\
w_0 = E^{(1)} (m_0, \La; \ve) - v(m_0), \quad w = E - v(m_0).
\end{split}
\end{equation}
Combining \eqref{eq:4schurfor} with \eqref{eq:4equationsolved1AAAAA}, one concludes that $\|(E - H_{\La,\ve})^{-1}\| \le C_{H_{\La,\ve}} |E - E^{(1)} (m_0, \La; \ve))|^{-1}$. Hence, $|\det(E - H_{\La,\ve})|^{-1} \le C_{H_{\La,\ve}} |E - E^{(1)} (m_0, \La; \ve))|^{-1}$. Therefore, $E^{(1)}(m_0, \La; \ve)$ is a simple zero of $\det(E - H_{\La,\ve})$.

To verify \eqref{eq:4-13acAAAA} note that
\begin{equation} \label{eq:4-13acACAC}
\begin{split}
\big| v(m_0) + Q^{(1)} (m_0, \La; \ve, E) - E^{(1)} \bigl(m_0, \La; \ve \bigr) \big | = \big| Q^{(1)} (m_0, \La; \ve, E) - Q^{(1)} (m_0, \La; \ve,E^{(1)} \bigl(m_0, \La; \ve \bigr)) \big | \\
\le [\sup |\partial_E Q^{(1)} (m_0, \La; \ve, E)|] |E - E^{(1)} \bigl(m_0, \La; \ve \bigr)|.
\end{split}
\end{equation}
Recall that $|\partial_E Q^{(1)} (m_0, \La; \ve, E)| < |\ve|$ if $|E - v(m_0)| < \delta_0/8$. This proves $(4)$ in case $s = 1$.

Let $(\delta_0^\one)^4 \le |E - E^{(1)}(m_0, \La; \ve)| \le \delta_0/16$. To verify $(5)$, we apply Lemma~\ref{lem:aux5AABBCC} with $\La_2 := \{ m_0 \}$, $\La_1 := \Lambda_{m_0}$. One has
$$
|v(m_0) - E| < |v(m_0) - E^{(1)}(m_0, \La; \ve)| + |E - E^{(1)}(m_0, \La; \ve)| \le \ve_0 + \delta_0/16 < \delta_0/8.
$$
Therefore, the matrix $\cH_{\La_{m_0}}$ is invertible and
\begin{equation}\label{eq:aux00H1inverse1newAAA}
|\cH_{\Lambda_{m_0}}^{-1}(m,n)| \le s_{D_1,T,\kappa_0,|\ve|;\La_{m_0}}(m,n).
\end{equation}
Furthermore, using \eqref{eq:4-13acAAAA}, one obtains
\begin{equation}\label{eq:aux00H1BBBB}
\begin{split}
|\tilde H_2| := |\cH(m_0,m_0) - \Gamma_{2,1} \cH_{\La_{m_0}}^{-1} \Gamma_{1,2}| = |E - v(m_0) - Q^{(1)} (m_0, \La; \ve, E)| \\
\ge |E - E^{(1)} \bigl(m_0, \La; \ve \bigr)| - |v(m_0) + Q^{(1)} (m_0, \La; \ve, E) - E^{(1)} \bigl(m_0, \La; \ve \bigr)| \\
\ge |E - E^{(1)} \bigl(m_0, \La; \ve \bigr)| - |\ve| |E - E^{(1)} \bigl(m_0, \La; \ve \bigr)| > \frac{1}{2} |E - E^{(1)} \bigl(m_0, \La; \ve \bigr)| > (\delta_0^\one)^4/2, \\
|\tilde H_2|^{-1} \le 2\exp(D(m_0;\La)).
\end{split}
\end{equation}
Thus, all conditions of Lemma~\ref{lem:aux5AABBCC} hold. So, $\cH_{\Lambda}$ is invertible and
\begin{equation}\label{eq:3aux00c011kappad2statement}
|\cH_{\Lambda}^{-1}(m,n)| \le s_{D,T,\kappa_0,\ve_0;\La}(m,n).
\end{equation}
This proves $(5)$ in case $s=1$.

\bigskip

Assume now that $s \ge 2$ and statements $(1)$--$(5)$ hold for any matrix of class $\cN^{(s')}(m,\Lambda^{(s')}(m); \delta_0)$ with $1 \le s' \le s-1$.

Note first of all the following. Let $\ve, E \in \mathbb{C}$ be such that $|\ve| < \ve_{s-2}$ and $|E^{(s-1)}(m_0,\La^\esone(m_0); \ve) - E| < 2 \delta^\esone$. Assume that $s \ge 3$. Let $m \in \cM^{(s-1)}$ be arbitrary, $m \neq m_0$. Then, using \eqref{eq:4-3sge3} from condition $(e)$ in Definition~\ref{def:4-1}, one obtains
\begin{equation}\label{eq:3EdomainstimateAAAA}
\begin{split}
|E^{(s-1)}(m, \La^{(s-1)}(m); \ve) - E| < \delta^{(s-2)}_0 + 2 \delta^\esone_0 < 2 \delta^{(s-2)}_0, \\
|E^{(s-1)}(m, \La^{(s-1)}(m); \ve) - E| > 3 \delta^{(s-1)}_0 - 2 \delta^\esone_0 = \delta^{(s-1)}_0 > (\delta^{(s-1)}_0)^2.
\end{split}
\end{equation}
Similarly, let  $1 \le s' \le s-2$, $m \in \cM{(s')}$ be arbitrary. Then, using \eqref{eq:4-3sge3} from condition $(e)$ in Definition~\ref{def:4-1}, one obtains
\begin{equation}\label{eq:3EdomainstimateAAAABCDE}
\begin{split}
|E^{(s')}(m, \La^{(s')}(m); \ve) - E| < \delta^{(s'-1)}_0 + 2 \delta^\esone < 3 \delta^{(s'-1)}_0/2, \\
|E^{(s')}(m, \La^{(s')}(m); \ve) - E| > \delta^{(s')}_0/2 - 2 \delta^\esone > \delta^{(s')}_0/4 > (\delta^{(s')}_0)^2.
\end{split}
\end{equation}
This means that the inductive assumption applies to $H_{\La^{(s')}(m),\ve}$ in the role of $H_{\La,\ve}$ and to the value $E$, so that $(1)$--$(5)$ hold. In particular, each $\cH_{\La^{(s'),\ve}(m),\ve} := E - H_{\La^{(s'),\ve}(m),\ve}$ $m \neq m_0$ is invertible. Furthermore, obviously, the inductive assumptions apply to $\cH_{\La^{(s-1)}(m_0),\ve}$ in the role of $H_{\La,\ve}$ and to the value $E$, so that $(1)$--$(4)$ hold. In particular, $\cH_{\La^{(s-1)}(m_0)\setminus\{m_0\},\ve} := E - H_{\La^{(s-1)}(m_0) \setminus \{m_0\},\ve}$ is invertible. Moreover,
\begin{equation}\label{eq:3HinvestimateAXYZ}
\begin{split}
|\cH_{\La^{(s')}(m),\ve}^{-1}(x,y)| \le s_{D(\cdot;\La^{(s')}(m)),T,\kappa_0,|\ve|;\La_{m_0}}(x,y), \\
|\cH_{\La^{(s-1)}(m_0) \setminus \{m_0\},\ve}^{-1}(x,y)| \le s_{D(\cdot;\La^{(s-1)}(m_0)\setminus\{m_0\}),T,\kappa_0,|\ve|; \La^{(s-1)}(m_0)\setminus\{m_0\}}(x,y).
\end{split}
\end{equation}
For $s = 2$, one arrives at the same conclusions using \eqref{eq:4-3} instead of \eqref{eq:4-3sge3}. Due to \eqref{eq:4-17AAAA},
$\big|E^{(s-1)}(m_0, \La^\esone(m_0); \ve) - v(m_0) \big| < |\ve| < \delta_0/64$. Recall also that $|v(n) - v(m_0)| \ge \delta_0/16$ for any $n \in \Lambda \setminus \bigl(\bigcup_{1 \le s' \le s} \bigcup_{m \in \cM{(s')}} \La^{(s')}(m) \bigr)$. This implies $|E - v(n)| \ge \delta_0/32 > \delta_0^2$ for any $n \in \Lambda \setminus \bigl(\bigcup_{1 \le s' \le s} \bigcup_{m \in \cM{(s')}} \La^{(s')}(m) \bigr)$ since $|E^{(s-1)}(m_0,\La^\esone(m_0); \ve) - E| < 2 \delta^\esone < \delta_0/64$. Let again $\La_{m_0} = \La \setminus \{m_0\}$, $D(x;\La_{m_0}) = D(x;\La^{(s')}(m))$ if $x \in \La^{(s')}(m)$ for some $s' \le s-1$ and some $m \in \cM{(s')}$, $m \neq m_0$ or if $x \in \La^{(s-1)}(m_0) \setminus \{m_0\}$, $D(x;\La_{m_0}) = 4 \log \delta_0^{-1}$; otherwise just as in part $(1)$ of the current proposition. Due to the inductive assumptions, $D(\cdot;\La^{(s')}(m)) \in \mathcal{G}_{\La^{(s')}(m),T,\kappa_0}$, $1 \le s' \le s-1$, $m \in \cM{(s')}$, $m \neq m_0$, and also $D(\cdot;\La^{(s-1)}(m_0) \setminus \{m_0\}) \in \mathcal{G}_{\La^{(s-1)}(m_0) \setminus \{m_0\},T,\kappa_0}$. Due to Lemma~\ref{lem:auxDfunctionsrules}, $D(\cdot;\La_{m_0}) \in \mathcal{G}_{\La_{m_0},T,\kappa_0}$. Due to condition $(d)$ in Definition~\ref{def:4-1}, one has $\mu_\La(m_0) \ge R^{(s)}$. So,
$$
D(m_0;\La) = 4 \log (\delta^{(s)}_0)^{-1} = 4 \log \exp \bigl((\log R^{(s)})^{2}\bigr) = 4 (\log R^{(s)})^{2} < (R^\es)^{1/5} < \mu_\La(m_0)^{1/5}.
$$
Hence, $D(\cdot;\La) \in \mathcal{G}_{\La,T,\kappa_0}$. Due to the inductive assumption, $\max_{x \neq m} D(x;\La^{(s')}(m)) \le \log 4 (\delta_0^{(s'-1)})^{-1}$ for any $s'$ and any $m \in \cM{(s')}$. Due to the definition, $D(m;\La^{(s')}(m)) = \log 4 (\delta_0^{(s'-1)})^{-1}$ if $m \in \cM{(s')}$. Thus, $\max_{x \neq m_0} D(x;\La^{(s')}(m)) \le \log 4(\delta_0^{(s-1)})^{-1}$. This finishes the proof of $(1)$,$(2)$.

Due to Proposition~\ref{prop:aux1}, $\cH_{\La_{m_0},\ve} = E - H_{\La_{m_0},\ve}$ is invertible and
\begin{equation}\label{eq:3Hinvestimate}
|\cH_{\La_{m_0},\ve}^{-1}(x,y)| \le s_{D(\cdot;\La_{m_0}),T,\kappa_0,\ve_0;\La_{m_0}}(x,y).
\end{equation}

Just like in the case $s = 1$ one concludes that
\begin{equation}\label{eq:4101rem1a1}
\begin{split}
|Q^{(s)}(m_0, \La; \ve, E)| \le |\varepsilon|^{3/2}, \\
\text{where $Q^{(s)}(m_0, \La; \ve, E) := \ve^2 \sum_{m,n \in \La_{m_0}} h(m_0,m) K^{(s)}(m,n, \La_{m_0}; \ve, E) h(n,m_0)$}, \\
K^{(s)}(x,y, \La_{m_0}; \ve, E) := \cH_{\La_{m_0}}^{-1}(x,y).
\end{split}
\end{equation}
The functions  $K^{(s)}(m,n, \La_{m_0}; \ve, E)$, $Q^{(s)}(m_0, \La; \ve, E)$ are analytic in the domain
\begin{equation}\label{eq:4domain1A1AA}
|\ve| < \ve_{s-2}, \quad |E^{(s-1)}(m_0,\La^\esone(m_0); \ve) - E| < 2 \delta^\esone.
\end{equation}

To verify the first estimate in \eqref{eq:4-12acACC}, we write
\begin{equation}\label{eq:4101rem1a1Qsplit}
\begin{split}
Q^{(s)}(m_0, \La; \ve, E) := \ve^2 [\sum_{m,n \in \La_1} + \sum_{m,n \in \La_2} + \sum_{m \in \La_1, n \in \La_2} + \sum_{m \in \La_2, n \in \La_1}] \\
h(m_0,m) \cH_{\La_{m_0}}^{-1}(m,n) h(n,m_0) := Q_{1,1} + Q_{2,2} + Q_{1,2} + Q_{2,1},
\end{split}
\end{equation}
where $\La_1 := \La^{(s-1)}(m_0) \setminus \{m_0\}$, $\La_2 := \La_{m_0} \setminus \La_1 = \La \setminus \La^{(s-1)}(m_0)$. We invoke the Schur complement formula \eqref{eq:2schurfor} with these $\La_1$, $\La_2$,
$$
\cH_{\La_{m_0}}^{-1} = \begin{bmatrix} \cH_1^{-1} + \cH_1^{-1} \Gamma_{1,2} \tilde H_2^{-1} \Gamma_{2,1} \cH_1^{-1} & - \cH_1^{-1}
\Gamma_{1,2} \tilde H_2^{-1} \\[5pt] -\tilde H_2^{-1} \Gamma_{2,1} \cH_1^{-1} & \tilde H_2^{-1}\end{bmatrix}.
$$
For $x,y \in \La_1$, one has
\begin{equation}\label{eq:4schurforKKll}
\cH_{\La_{m_0}}^{-1} (x,y) = [\cH_{\La^{(s-1)}(m_0) \setminus \{m_0\},\ve}^{-1}] (x,y) + [\cH_{\La^{(s-1)}(m_0) \setminus \{m_0\},\ve}^{-1} \Gamma_{1,2} \hat H_2^{-1} \Gamma_{2,1}\cH_{\La^{(s-1)}(m_0) \setminus \{m_0\},\ve}^{-1}](x,y),
\end{equation}
where $\hat H_2 := \cH_{\La \setminus \La^{(s-1)}(m_0),\ve} - \Ga_{2,1} \cH_{\La^{(s-1)}(m_0) \setminus \{m_0\},\ve}^{-1} \Ga_{1,2}$. This implies
\begin{equation}\label{eq:4101rem1a1KKLLN}
\begin{split}
|Q^{(s)}(m_0, \La; \ve, E) - Q^{(s-1)}(m_0, \La^{(s-1)}(m_0); \ve, E)| \le \ve^2 \sum_{m,n \in \La_1} \exp(-\kappa_0 |m_0-m| - \kappa_0 |n - m_0|) \times \\
|[\cH_{\La^{(s-1)}(m_0) \setminus \{m_0\},\ve}^{-1} \Gamma_{1,2} \tilde H_2^{-1} \Gamma_{2,1} \cH_{\La^{(s-1)}(m_0) \setminus \{m_0\},\ve}^{-1}](m,n)| + |Q_{2,2}| + |Q_{1,2}| + |Q_{2,1}| \\
:= R_{1,1} + |Q_{2,2}| + |Q_{1,2}| + |Q_{2,1}|.
\end{split}
\end{equation}
Once again, since all conditions of Lemma~\ref{lem:aux5AABBCC} hold, one can invoke \eqref{eq:aux00c011kappad2statement}.

Using the estimate \eqref{eq:aux00c011kappad2statement}, combined with the estimate \eqref{eq:auxtrajectweight2}, one obtains
\begin{equation}\label{eq:auxAAAAKKLL}
\begin{split}
R_{1,1} \le |\ve|^2 \sum_{m,n \in \La_1} \exp(-\kappa_0 |m_0-m| - \kappa_0 |n-m_0|) \times \\
\sum_{q \ge 3} |\ve|^{q-1} \sum_{\gamma \in \Ga^{(1,2)}_{D,T,\kappa_0}(m,n;q,\La)} W_{D,\kappa_0}(\gamma) \le |\ve|^2 \sum_{m,n \in \La_1} \exp(-\kappa_0 |m_0-m| - \kappa_0 |n-m_0|) \times \\
\sum_{q \ge 3} |\ve|^{q-1} \sum_{\gamma \in \Ga^{(1,2)}_{D,T,\kappa_0}(m,n;q,\La_{m_0})} \exp \left( -\frac{15}{16} \kappa_0 \|\gamma\| + \max(\bar D(\gamma),4T\kappa_0^{-1}) \right),
\end{split}
\end{equation}
where $\Ga^{(p,q)}_{D,T,\kappa_0}(m,n;k,\La) = \bigcup_{t \ge 1} \Ga^{(p,q,t)}_{D,T,\kappa_0}(m,n;k,\La)$, $\Ga^{(p,q,t)}_{D,T,\kappa_0}(m,n;k,\La)$ stands for the set of all $\gamma \in \Ga_{D,T,\kappa_0}(m,n;k,\La)$ such that $\gamma = \gamma_1 \cup \gamma'_1 \dots \cup \gamma_{t+1}$ with $\gamma_j \in \Ga_{D_p,T,\kappa_0}(\La_p)$, $\gamma'_i \in \Ga_{D_q,T,\kappa_0}(\La_q)$, $p \neq q$. Note that for any $m,n$ and any $\gamma \in \Ga^{(1,2)}_{D,T,\kappa_0}(m,n;q,\La_{m_0})$, we have
\begin{equation}\nn
\begin{split}
\bar D(\gamma) \le \max_{x \neq m_0} D(x) \le \log 4 (\delta_0^\esone)^{-1} = \bigl(R^\esone\bigr)^{1/4} + 2 \log2, \\
|m_0-m| + \|\gamma\| + |n-m_0| \ge 2 \mu_{\La^{(s-1)}(m_0)}(m_0) \ge 2 R^\esone
\end{split}
\end{equation}
(the second estimate here is due to condition $(d)$ in Definition~\ref{def:4-1}). Combining these estimates with \eqref{eq:auxAAAAKKLL}, one obtains
\begin{equation}\label{eq:auxAAAAKKLL1}
\begin{split}
R_{1,1} \le |\ve|^2 \sum_{q \ge 3} |\ve|^{q-1} \sum_{\gamma \in \Ga(m_0,m_0;q+2,\La), \quad \|\gamma\| \ge 2 R^\esone} \exp \left( -\frac{15}{16} \kappa_0 \|\gamma\| \right) \\
\le |\ve|^2 \sum_{q \ge 3} |\ve|^{q-1} \exp \left( -\kappa_0 R^{(s-1)} \right) \sum_{\gamma \in \Ga(m_0,m_0;q+2,\La)} \exp \left( -\frac{1}{4} \kappa_0 \|\gamma\| \right) \\
\le |\ve|^{3/2} \exp \left( -\kappa_0 R^{(s-1)} \right) < |\ve| (\delta_0^{(s-1)})^6
\end{split}
\end{equation}
(here we used \eqref{eq:auxtrajectweight21a} from Lemma~\ref{lem:2gammasum} and $|\ve| \le \ve_0$). The estimation of the rest of the terms in \eqref{eq:4101rem1a1KKLLN} is completely similar. So, the first estimate in \eqref{eq:4-12acACC} holds. The second estimate in \eqref{eq:4-12acACC} follows from the first one combined with the inductive assumption and Cauchy estimates for analytic functions. This finishes the verification of $(3)$.

Let us turn to part $(4)$. Suppose $|E^{(s-1)}(m_0,\La^\esone(m_0); \ve) - E| < 2 \delta^\esone$. Due to the Schur complement formula, $\cH_\La := E - H_\La$ is invertible if and only if
\begin{equation}\label{eq:4characteristicequation1s}
\tilde H_2 = E - v(m_0) - \sum_{m,n \in \La_{m_0}} (-\ve h(m_0,m)) K^{(s)}(m,n, \La_{m_0}; \ve, E) (-\ve h(n,m_0)) \neq 0.
\end{equation}
In this case,
\begin{equation}\label{eq:4schurforsss}
\cH_\La^{-1} = \begin{bmatrix} \cH_{\La_{m_0}}^{-1} + \cH_{\La_{m_0}}^{-1} \Gamma_{1,2} \tilde H_2^{-1} \Gamma_{2,1} \cH_{\La_{m_0}}^{-1} & - \cH_{\La_{m_0}}^{-1} \Gamma_{1,2} \tilde H_2^{-1} \\[5pt] -\tilde H_2^{-1} \Gamma_{2,1} \cH_{\La_{m_0}}^{-1} & \tilde H_2^{-1}\end{bmatrix}.
\end{equation}
In other words, if $|E^{(s-1)}(m_0,\La^\esone(m_0); \ve) - E| < \delta^\esone$, then $E \in \spec H_{\La,\ve}$ if and only if it obeys
\begin{equation}\label{eq:4chracteristic1sss}
E = v(m_0) + Q^{(s)}(m_0, \La; \ve, E).
\end{equation}
To solve the equation \eqref{eq:4chracteristic1sss}, we again invoke part $(2)$ of Lemma~\ref{em:implicitanalyticont}. We set
\begin{equation}\label{eq:4equationdomain1sss}
\begin{split}
\phi_0(z) := E^{(s-1)}(m_0,\La^\esone(m_0); \ve) - v(m_0), \quad z := \ve, \quad z_0 := 0, w := E - v(m_0), \quad \sigma_0 := \ve_{s-1}, \\
f(z,w)  = Q^{(s)} (m_0, \La; \ve,w+v(m_0)), \quad \big| w -\phi_0(z) \big| < 2 \delta^\esone, \quad \quad \rho_0 = \delta^\esone.
\end{split}
\end{equation}
Note that due to equation \eqref{eq:4-16} with $(s-1)$ in the role of $s$ and the first estimate in \eqref{eq:4-12acACC}, one has
\begin{align*}
|f & (z, \phi_0(z)) - \phi_0(z)| \\
& = |v(m_0) + Q^{(s)} (m_0, \La; z, E^{(s-1)}(m_0, \La^\esone(m_0); z)) - E^{(s-1)} (m_0, \La^\esone(m_0); z)| \\
& = |v(m_0) + Q^{(s)} (m_0, \La; z, E^{(s-1)}(m_0, \La^\esone(m_0); z)) \\
& \qquad \qquad - [|v(m_0) + Q^{(s-1)} (m_0, \La^\esone(m_0); z, E^{(s-1)} (m_0, \La^\esone(m_0); z))]| \\
& < |\ve| (\delta^{(s-1)}_0)^6.
\end{align*}
As in part $(2)$ of Lemma~\ref{em:implicitanalyticont}, set
\begin{align*}
M_0 & = \sup_z |\phi_0(z)| + \rho_0 + \sup_{z,w} |f(z,w)|, \\
M_1 & = \max(1,M_0), \quad \varepsilon_1  = \frac{\sigma_0^2 \rho_0^2}{10^{10} M_1^3 (1 + \log (\max (100, M_1)))^2}.
\end{align*}
One has $M_0 < \ve_0 + \rho_0 + \ve_0 < 1$. This implies $\varepsilon_1 > \frac{\delta^{(s-1)}_0)^5}{10^{12} } > (\delta^\esone_0)^6$. Due to part $(2)$ of Lemma~\ref{em:implicitanalyticont}, the equation $f(z,w) = w$ has a unique solution, which we denote by $w = E^{(s)} (m_0, \La; z) - v(m_0)$. The function $E^{(s)} (m_0, \La; z)$ is defined and analytic for $|z - z_0| < \sigma_0 - \ve_1$, and it obeys equation \eqref{eq:4chracteristic1sss}. Note that $\sigma_0 - \ve_1 > \ve_s$. Due to part $(2)$ of Lemma~\ref{em:implicitanalyticont}, one has
$$
|E^{(s)} (m_0, \La; z) - E^{(s-1)} (m_0, \La^\esone(m_0); z)| < 10^3 (1 + \log (\max (100, M_1)))^2 |f(z, \phi_0(z)) - \phi_0(z)| < |\ve|(\delta^{(s-1)}_0)^5.
$$
This validates \eqref{eq:4-17AAAA}. Next,
\begin{equation}\label{eq:4-13acAAAAvalid}
\begin{split}
\big| v(m_0) + Q^{(s)} (m_0, \La; \ve, E) - E^{(s)} \bigl(m_0, \La; \ve \bigr) \big | = \big| Q^{(s)} (m_0, \La; \ve, E) - Q^{(s)} (m_0, \La; \ve,E^{(s)} \bigl(m_0, \La; \ve \bigr)) \big | \\
\le  [\sup |\partial_E Q^{(s)} (m_0, \La; \ve, E)|]|E - E^{(s)} \bigl( n^\zero, \La^{(s)}; \ve \bigr) |<|\ve| |E - E^{(s)} \bigl( n^\zero, \La^{(s)}; \ve \bigr) |,
\end{split}
\end{equation}
which validates \eqref{eq:4-13acAAAA}.

The validation of part $(5)$ goes just the same way as for $s=1$. Thus, $(1)$--$(5)$ hold for any $s$.

We will now verify $(6)$. Since $E = E^{(s)}(m_0, \La; \ve)$ is a simple zero of $\det(E - H_{\La,\ve})$, the operator
\begin{equation}\label{eq:4ResRes}
P(m_0,\La;\ve) := \Res (E - H_{\La,\ve})^{-1}|_{E = E^{(s)}(m_0, \La; \ve)}
\end{equation}
is a one-dimensional projector on the eigenspace corresponding to $E^{(s)}(m_0, \La; \ve)$, which is called Riesz projector. Due to \eqref{eq:4schurforsss}, one has
\begin{equation}\label{eq:4ResRes1}
\begin{split}
(E - H_{\La,\ve})^{-1}(n,m_0) = - \cH_{\La_{m_0},\ve}^{-1} \Gamma_{1,2} \tilde H_2^{-1} \\
= - \sum_{m \in \La_{m_0}} \cH_{\La_{m_0},\ve}^{-1}(n,m) h(m,m_0;\ve) (E - v(m_0) - Q^{(s)}(m_0, \La; \ve, E))^{-1} \\
= - F^{(s)}(m_0,n, \La_{m_0}; \ve, E) (E - v(m_0) - Q^{(s)}(m_0, \La; \ve, E))^{-1}, \quad n \neq m_0, \\
(E - H_{\La,\ve})^{-1}(m_0,m_0) = (E - v(m_0) - Q^{(s)}(m_0, \La; \ve, E))^{-1}.
\end{split}
\end{equation}
Hence,
\begin{equation}\label{eq:4ResRes2}
P(m_0,\La;\ve) \delta_{m_0,\cdot} = \Res [(E - v(m_0) - Q^{(s)}(m_0, \La; \ve, E))^{-1} \vp^\es(\La;\ve,E)]|_{E = E^{(s)}(m_0, \La; \ve)},
\end{equation}
where $\vp^\es(\La;\ve,E) := (\vp^{(s)}(n, \La; \ve,E))_{n \in \La}$, $\vp^{(s)}(m_0, \La; \ve,E) = 1$, $\vp^{(s)}(n, \La; \ve,E) = -F^{(s)}
(m_0, n, \La; \ve, E)$, $n \ne m_0$. Recall that $(E - v(m_0) - Q^{(s)}(m_0, \La; \ve, E))^{-1}$ has a simple pole at $E = E^{(s)}(m_0, \La; \ve)$. Therefore,
\begin{equation}\label{eq:4ResRes3}
P(m_0,\La;\ve)\delta_{m_0,\cdot} = \Res [(E - v(m_0) - Q^{(s)}(m_0, \La; \ve, E))^{-1}]|_{E = E^{(s)}(m_0, \La; \ve)}\vp^\es(\La;\ve),
\end{equation}
where $\vp^\es(\La;\ve)$ is defined as in part $(6)$. Since $\Res [(E - v(m_0) - Q^{(s)}(m_0, \La; \ve, E))^{-1}]|_{E = E^{(s)}(m_0, \La; \ve)} \neq 0$, $P(m_0,\La;\ve) \delta_{m_0,\cdot} \neq 0$. Hence, $P(m_0,\La;\ve) \delta_{m_0,\cdot}$ is an eigenvector of $H_{\La,\ve}$ corresponding to $E^{(s)}(m_0, \La; \ve)$. Therefore, $\vp^{(s)}(\La; \ve)$ is an eigenvector of $H_{\La,\ve}$ corresponding to $E^{(s)}(m_0, \La; \ve)$. The estimate in \eqref{eq:4-11evdecay} follows from \eqref{eq:4-11acACC}. The identity in \eqref{eq:4-11evdecay} is just the definition of  $\vp^{(s)}(m_0, \La; \ve)$. To verify $\|P(m_0,\La;\ve) \delta_{m_0}\| \ge 2/3$, note that $|\partial_E (E - v(m_0) - Q^{(s)}(m_0, \La; \ve, E))| \le 3/2$. Hence,
\begin{equation}\label{eq:4ResRes3U}
\begin{split}
|\Res [(E - v(m_0) - Q^{(s)}(m_0, \La; \ve, E))^{-1}]|_{E = E^{(s)}(m_0, \La; \ve)}| \\
= |\partial_E (E - v(m_0) - Q^{(s)}(m_0, \La; \ve, E))|_{E = E^{(s)}(m_0, \La; \ve)}|^{-1} > 2/3.
\end{split}
\end{equation}
This implies the desired estimate. Finally, using \eqref{eq:4-11acACC} and \eqref{eq:4-17AAAA}, one obtains
\begin{equation}\label{eq:4-11acACCevP}
\begin{split}
|\vp^\es(n,\La;\ve) -\vp^\esone(n,\La^\esone(m_0);\ve)| \le \sup_{E} \big| F^{(s)}(m_0, n, \La; \ve, E) - F^{(s-1)} \bigl( m_0, n, \La^{(s-1)}(m_0); \ve, E \bigr) \big| \\
+ 2 \sup_{E,s'}  \big| \partial_E F^{(s')} \big| \big| E^{(s)}(m_0, \La; \ve) - E^{(s-1)} \bigl( m_0, \La^{(s-1)}(m_0); \ve \bigr) \big| \\
\le |\ve|^{1/2} \exp(-\kappa_0 R^{(s-1)}) + |\ve| (\delta_0^{(s-1)})^5 \le 2 |\ve| (\delta_0^{(s-1)})^5
\end{split}
\end{equation}
for any $n \in \La^\esone(m_0)$, as claimed in \eqref{eq:4-11acACCev}.
\end{proof}

Using the notation of \eqref{eq:2-1}--\eqref{eq:2-4}, assume that the functions $h(m, n,\ve)$, $m, n \in \La$ depend also on some parameter $k \in (k_1,k_2)$, that is, $h(m,n;\ve) = h(m, n;\ve,k)$. Let
\begin{equation}\label{eq:3.Hk}
H_{\La,\varepsilon,k} := \bigl(h(m, n; \ve,k)\bigr)_{m, n \in \La}.
\end{equation}
Assume that $H_{\La,\varepsilon,k}$ belongs to the class $\cN^{(s)}\bigl(m_0, \La; \delta_0\bigr)$. Denote by $K^{(s)}(m,n, \La_{m_0}; \ve,k, E)$, $Q^{(s)}(m_0,\La; \ve,k, E)$, $E^{(s)}(m_0, \La; \ve,k)$ the functions introduced in Proposition~\ref{prop:4-4} with $H_{\La,\varepsilon,k}$ in the role of $\hle$. Later in this work we will need estimates for the partial derivatives of these functions with respect to the parameter $k$.

\begin{lemma}\label{lem:7differentiation}
$(1)$ Let $H_k = (h(x,y;k))_{x,y \in \La}$ be a matrix-function, $k \in (k_1,k_2)$. Let $E \in \mathbb{C} \setminus \bigcup_{k \in (k_1,k_2)} \spec H_k$, so that $(E - H_k)^{-1}$ is well defined for $k \in (k_1,k_2)$. If $H_k$ is $C^1$-smooth, then $(E - H_k)^{-1}$ is a $C^1$-smooth function of $E,k$
\begin{equation}\label{eq:7ResderivAAA}
\begin{split}
\partial_k (E - H_k)^{-1} = (E - H_k)^{-1} \partial_k H_k (E - H_k)^{-1}, \\
\partial_E (E - H_k)^{-1} = - (E - H_k)^{-2}.
\end{split}
\end{equation}
If $H_k$ is $C^2$-smooth, then $(E - H_k)^{-1}$ is a $C^2$-smooth function of $E,k$ and
\begin{equation}\label{eq:7ResderivAAB-2}
\begin{split}
\partial^2_k (E - H_k)^{-1} = 2 (E - H_k)^{-1} \partial_k H_k (E - H_k)^{-1} \partial_k H_k (E - H_k)^{-1} + (E - H_k)^{-1} \partial^2_{k,k} H_k (E - H_k)^{-1}, \\
\partial^2_{E,k} (E - H_k)^{-1} = (E - H_k)^{-1} \partial_k H_k (E - H_k)^{-2} + (E - H_k)^{-2} \partial_k H_k (E - H_k)^{-1}, \\
\partial^2_E (E - H_k)^{-1} = 2 (E - H_k)^{-2}.
\end{split}
\end{equation}

$(2)$ Let $H_{\La,\varepsilon,k}$ be as in \eqref{eq:3.Hk}. Assume that for any $E \in (E',E'')$, we have
\begin{equation}\label{eq:3Hinvestimatestatement1kK-2}
|(E - H_{\La,\ve,k})^{-1}(x,y)| \le S_{D(\cdot;\La),T,\kappa_0,|\ve|;\La}(x,y), \quad x,y \in \La,
\end{equation}
where $D \in \mathcal{G}_{\La,T,\kappa_0}$. Assume also that $h(m,n;\ve,k)$ are $C^2$-smooth functions that for $m \neq n$ obey $|\partial^\alpha h(m,n;\ve,k)| \le B_0 \exp(-\kappa_0 |m - n|)$ for $|\alpha| \le 2$. Furthermore, assume that there is $m_0 \in \La$ such that $|\partial^\alpha h(m,m;\ve,k)| \le B_0 \exp(\kappa_0 |m - m_0|^{1/5})$ for any $m \in \La$, $0< |\alpha| \le 2$. Then, for any multi-index $|\beta |\le 2$, we have
\begin{equation}\label{eq:3HinvestIM}
|\partial^\beta (E - H_{\La,\ve,k})^{-1}(m,n)| \le (3B_0)^{|\beta|} \exp( |\beta| \kappa_0 |m - m_0|^{1/5}) \mathfrak{D}^{|\beta|}_{D(\cdot),T,\kappa_0,|\ve|;\La}(m,n), \quad m,n \in \La;
\end{equation}
see Lemma~\ref{lem:7differentiationB}.
\end{lemma}

\begin{proof}
$(1)$ Let $k_0 \in (k_1,k_2)$ be arbitrary. For sufficiently small $|k - k_0|$, one has $\|H_k - H_{k_0}\| < M(k_0) |k-k_0|$, where $M(k_0) = 1 + \|\partial_k H_k|_{k=k_0}\|$. In particular, $\|H_k - H_{k_0}\| \|(E - H_{k_0})^{-1}\| < 1/2$ for sufficiently small $|k - k_0|$. Hence,
\begin{equation}\label{eq:7Resderiv-2}
\begin{split}
(E - H_k)^{-1} - (E - H_{k_0})^{-1} = \sum_{t \ge 1} (E - H_{k_0})^{-1} [(H_{k_0} - H_k) (E - H_{k_0})^{-1}]^t \\
= (E - H_{k_0})^{-1} (H_{k_0} - H_k) (E - H_{k_0})^{-1} + R(k,k_0), \\
\|R(k,k_0)\| \le \sum_{t \ge 2} \|(E - H_{k_0})^{-1}\|^{t+1} \|H_k - H_{k_0}\|^t \\
\le \|(E - H_{k_0})^{-1}\|^{3} \|H_k - H_{k_0}\|^2 \sum_{t \ge 0} 2^{-t} \le C(k_0) (k - k_0)^2,
\end{split}
\end{equation}
where $C(k_0) = 2 M(k_0)^2 \|(E - H_{k_0})^{-1}\|^{3}$. This implies the first identity in \eqref{eq:7ResderivAAA}. The derivation of the other identities is similar.

$(2)$ This part follows from part $(1)$ combined with Lemma~\ref{lem:7differentiationB} and with the definitions in \eqref{eq:auxtrajectweightsumest8iterted} of Lemma~\ref{lem:auxweight1iterated}.
\end{proof}

\begin{lemma}\label{lem:5differentiation}
Assume that $H_{\La,\varepsilon,k} \in \cN^{(s)}\bigl(m_0, \La; \delta_0\bigr)$. Then,

$(1)$ If $h(m,n;\ve,k)$ are $C^t$-smooth functions of $k$, then $K^{(s)}(m,n, \La_{m_0}; \ve,k, E)$, $Q^{(s)}(m_0,\La; \ve,k, E)$, $E^{(s)}(m_0, \La; \ve,k)$ are $C^t$-smooth functions of all variables involved.

$(2)$ Assume that $h(m,n;\ve,k)$ obeys conditions in part $(2)$ of Lemma~\ref{lem:7differentiation}.
Then, for $\alpha=1,2$, we have
\begin{equation}\label{eq:3Hinvestimatestatement1kvard}
|\partial^\alpha_k Q^{(s)}(m_0,\La; \ve,k, E)|\le
(3B_0)^\alpha|\ve|\mathfrak{D}^\alpha_{D(\cdot;\La\setminus\{m_0,m_0\}),T,\kappa_0,|\ve|;\La\setminus \{m_0\}}(m_0)<(3B_0)^\alpha|\ve|^{3/2},
\end{equation}
\begin{equation}\label{eq:3Hinvestimatestatement1kvardE}
|\partial^\alpha_k E^{(s)}(m_0,\La; \ve,k)-\partial^\alpha_k v(m_0,k)| < (3B_0)^\alpha|\ve|^{3/2}.
\end{equation}
\end{lemma}

\begin{proof}
It follows from Lemma~\ref{lem:7differentiation} that $K^{(s)}(m,n, \La_{m_0}; \ve,k, E)$ is a $C^t$-smooth functions of all variables involved. Therefore, $Q^{(s)}(m_0,\La; \ve,k, E)$ is  $C^t$-smooth. Due to the implicit function theorem, $E^{(s)}(m_0, \La; \ve,k)$ is $C^t$-smooth.

Using \eqref{eq:7ResderivAAA} from Lemma~\ref{lem:7differentiation}, \eqref{eq:3Hinvestimatestatement1} from Proposition~\ref{prop:4-4}, and \eqref{eq:auxtrajectweightsumest8iterted} from Lemma~\ref{lem:auxweight1iterated}, one obtains
\begin{equation}\label{eq:3Hinvestimatestatement1kvard-2}
\begin{split}
| \partial_k Q^{(s)}(m_0,\La; \ve,k, E)| = \Big| \partial_k \sum_{m',n' \in \La_{m_0}} h(m_0,m_1;\ve,k) (E - H_{\La_{m_0},\ve,k})^{-1}(m_1,n_1) h(n_1,n_2;\ve,k) \Big| \\
\le 3 B_0 \mathfrak{D}^\one_{D(\cdot;\La\setminus\{m_0\}),T,\kappa_0,|\ve|;\La\setminus \{m_0\}}(m_0) < 3 B_0 |\ve|^{3/2}.
\end{split}
\end{equation}
This verifies \eqref{eq:3Hinvestimatestatement1kvard} for $\alpha = 1$. The verification for $\alpha = 2$ is completely similar.

Differentiating equation \eqref{eq:4-16}, one obtains
\begin{equation}\label{eq:4-16DDDD1}
\begin{split}
[\partial_k E^{(s)}(m_0, \La; \ve,k) -\partial_k v(m_0,k)] \bigl(1 -\partial_E Q^{(s)}(m_0, \La; \ve,k, E)|_{E = E^{(s)}(m_0, \La; \ve,k)} \bigr) \\
= \partial_E Q^{(s)}(m_0, \La; \ve,k, E)|_{E = E^{(s)}(m_0, \La; \ve,k)} \partial_k v(m_0,k).
\end{split}
\end{equation}
Combining \eqref{eq:4-16DDDD1} with \eqref{eq:3Hinvestimatestatement1kvard}, and taking into account the estimate for $|\partial_E Q^{(s)}(m_0,\La; \ve,k, E)|$ from \eqref{eq:4-12acACC}, one obtains the estimate \eqref{eq:3Hinvestimatestatement1kvardE} for $\alpha = 1$. The derivation for $\alpha = 2$ is completely similar.
\end{proof}

Let $H_{\La_j, \ve}$, $j = 1, 2$, be two matrices belonging to the class $\cN^\es(m_0, \La_j, \delta_0)$ with the same principal point $m_0$. Let $v(n,j)$ be the diagonal entries of $H_{\La_j,\ve}$. We assume that $v(n,1) = v(n,2)$ for $n \in \La_1 \cap \La_2$. Let $E^{(s)} \bigl( \La_j; \ve \bigr)$ be the eigenvalue defined by  Proposition~\ref{prop:4-4} with $H_{\La_j, \ve}$ in the role of $\hle$, $j=1,2$. One has the following:

\begin{corollary}\label{cor:5.twolambdas1}
\begin{equation}\label{eq:5.twolambdas1}
\big| E^{(s)} ( m_0, \La_1; \ve) - E^{(s)} \bigl( m_0, \La_2; \ve \bigr) \big| < |\ve| (\delta_0^{(s)})^5.
\end{equation}
\end{corollary}

\begin{proof}
Note first of all that $m_0 + B(R^\es) \subset \La_2 \cap \La_1$. Let $\vp^{(s)}(\La_1; \ve)$ be the vector defined in part $(6)$ of Proposition~\ref{prop:4-4} with $H_{\La_1, \ve}$ in the role of $\hle$. Set $\widetilde{\vp}^{(s)}(\La_2; \ve)(n) = \vp^{(s)}(\La_1; \ve)(n)$ if $n \in \La_2 \cap \La_1$, and $\widetilde{\vp}^{(s)}(\La_2; \ve)(n) = 0$ otherwise. Since $m_0 + B(R^\es) \subset \La_2 \cap \La_1$, one obtains using \eqref{eq:4-11evdecay} from Proposition~\ref{prop:4-4}, $\|(E^{(s)} ( \La_1; \ve) - H_{\La_2, \ve}) \widetilde{\vp}^{(s)}(\La_2; \ve)\| \le \exp(-\frac{\kappa_0}{8} R^\es)$. This, along with \eqref{eq:4-11evdecay} again for normalization purposes, implies
\begin{equation}\label{eq:4-17AAAAevclose}
\dist(E^{(s)}(m_0, \La_1;\ve), \spec H_{\La_2, \ve}) \le \exp \left(-\frac{\kappa_0}{16} R^\es \right).
\end{equation}
Recall that due to Definition~\ref{def:4-1} and \eqref{eq:4-17AAAA} from Proposition~\ref{prop:4-4}, there exists $\La^\esone_j$ such that $H_{\La^\esone_j, \ve} \in \cN^\esone(m_0, \La^\esone_j, \delta_0)$, and
\begin{equation}\label{eq:4-17AAAAevcloseA}
\big| E^{(s)}\bigl( m_0, \La_j; \ve\bigr) - E^{(s-1)}\bigl(m_0, \La^{(s-1)}_j; \ve \bigr) \big| < |\ve| (\delta_0^{(s-1)})^5, \quad j=1,2.
\end{equation}
Using induction and combining \eqref{eq:4-17AAAAevclose} with part $(4)$ of Proposition~\ref{prop:4-4} and with \eqref{eq:4-17AAAAevcloseA}, one obtains the statement.
\end{proof}

\section{Implicit Functions Defined by Continued-Fraction-Functions}\label{sec.4}

In this section and later in this paper we use the following notation:
\begin{align}
\cD \bigl( \zo, \ro \bigr),\ \zo  & = \bigl( \zo_1, \zo_2, \ldots, \zo_k \bigr),\ \zo_j \in \IC^1,\ j = 1, 2, \ldots, k \label{eq:6-1n}\\[6pt]
\ro & = \bigl( \ro_1, \ro_2, \ldots, \ro_k \bigr),\ \ro_j > 0 \label{eq:6-2n}
\end{align}
for the polydisk $\prod \limits_{1 \le j \le k} \cD \bigl(\zo_j, \ro_j\bigr) \subset \IC^k$, where $\cD(\zeta, r) = \{z \in \IC^1 : |z - \zeta| < r\},\ \zeta \in \IC^1,\ r > 0$;
\begin{equation}\label{eq6-5n}
S(\alpha, \beta; \rho) = \left\{z \in \IC^1: \Ree z \in (\alpha, \beta),\ |\imm z| < \rho\right\},
\end{equation}
$\alpha < \beta$; $\rho > 0$;
\begin{align}
\cL(g, \rho_1,\rho_0) & = \left\{ (z, w) \in \mathbb{C}^2: z \in S(\alpha, \beta; \rho_1),|w - g(\Re z)| < \rho_0\right\} , \label{eq:6-3n} \\[5pt]
\cL_\IR\bigl(g, \rho_0 \bigr) & = \cL(g,\rho_1, \rho_0) \cap (\IR\times \IR), \label{eq:6-4n}
\end{align}
where $g(x)$ is a real function defined on the interval $(\alpha,\beta)$ ($\cL_\IR\bigl(g, \rho_0\bigr)$ obviously does not depend on $\rho_1$);
\begin{align}
\cL(g,\cD,\rho) & = \left\{ (z, w) \in \mathbb{C}^2 : z \in \cD,|w - g(z)| < \rho\right\} , \label{eq:6-3nN}
\end{align}
where $g(z)$ is a complex function defined on the domain $\cD$.

We start with the following quantitative version of the implicit function theorem for complex analytic functions.

\begin{lemma}\label{standimplfunctheor}
Let $F(z,w)$ be an analytic function defined in the polydisk $\cP (z_0, w_0; r_0, r_0) := \cD (z_0, r_0) \times \cD(w_0, r_0)$. Assume that the following conditions hold: $(a)$ $ F(z_0,w_0) = 0$, $(b)$ $\tau := \big| \partial_w F_{|_{(z_0,w_0)}} \big| > 0$. Set $r = \tau^2 r_0^3 / (16 M_0)$, $r' = \tau r_0^2 / (2 M_0)$, where $M_0 := \sup_{\cP(z_0, w_0; r_0, r_0)} |F(z,w)|$. Then, for any $|z - z_0| < r$, there exists a unique $w = \phi(z)$, $|\phi(z)-w_0|<r'$ such that $F(z, \phi(z)) = 0$. Moreover, $\phi(z)$ is analytic in the disk $\cD (z_0, r) = \{z : |z - z_0| < r\}$.
\end{lemma}

\begin{proof}
Due to Cauchy estimates for the derivatives, one has $|\partial^2_{ww} F| \le 8 M_0 r_0^{-2}$ for any $(z,w) \in \cD (z_0, r_0) \times \cD (w_0, r_0 / 2)$. This implies
\begin{align*}
|F(z_0, w)| & = |F(z_0, w) - F(z_0, w_0)| \\
& \ge \big|\partial_w F_{|_{(z_0,w_0)}} \big| |w - w_0| - \frac{1}{2}\left(\sup_{|w - w_0| \le r_0/2}|\partial^2_{ww}F|\right) |w - w_0|^2 \\
& \ge \tau |w - w_0|/2
\end{align*}
for any $|w - w_0| \le \tau r_0^2 / 2 M_0$. This implies, in particular, that $w_0$ is the only root of $F(z_0, \cdot)$ in the disk $|w - w_0| \le \tau r_0^2 / 2 M_0$. Once again, due to Cauchy estimates for the derivatives, one has $|\partial_{z}F| \le 2 M_0 r_0^{-1}$ for any $(z,w) \in \cD (z_0, r_0/2) \times \cD(w_0, r_0)$. Hence, for any $|z - z_0| < r$ and any $|w - w_0| = \tau r_0^2 / (2 M_0)$, one has
\begin{align*}
|F(z,w) - F(z_0,w)| & \le \left(\sup_{|z-z_0| < r_0/2, |w-w_0| \le r_0}|\partial_{z}F|\right) |z - z_0| \\
& \le 2 M_0 r r_0^{-1} \\
& = \tau^2 r_0^2 / 8 \\
& = \tau |w - w_0| / 4 \le |F(z_0,w)|/2 < |F(z_0,w)|.
\end{align*}
Due to Rouch\'e's Theorem, the function $F(z,\cdot)$ has exactly one root in the disk $|w - w_0| \le \tau r_0^2 / 2 M_0$ for any $|z - z_0| < r$. Denote this root by $\phi(z)$. By the residue theorem with $r' = \tau r_0^2 / 2 M_0$, one has
$$
\frac{1}{2\pi i} \oint_{|w - w_0| = r'} w \frac{F_w(z,w)}{F(z,w)}\, dw = \phi(z)
$$
and the analyticity of $\phi(z)$ follows.
\end{proof}

We proceed with the derivation of a somewhat stronger version of this statement, where condition $(a)$ is being replaced by $(a')$ $|F(z_0,w_0)| \le \epsilon$ with sufficiently small $\epsilon$. For that we need the following version of the Harnack inequality.

\begin{lemma}\label{lem:6.harn}
Let $f(x)$ be analytic in $\cD (z_0, r_0)$ and non-vanishing in $\cD (z_0, r_1)$ with $0 < r_1 \le r_0$. Assume that
$$
K := \sup \bigl\{|f(z)|: z \in \cD (z_0, r_0) \bigr\} < \infty.
$$
Assume also that
\begin{equation}\label{harcond}
|f(z_0)| \ge K^{-1}.
\end{equation}
Then,
$$
|f(\zeta)| \le \exp(4) |f(z)|
$$
for any $z, \zeta \in \cD (z_0, r_2)$, $r_2 = (1 + \log (\max (100,K)))^{-2}r_1$.
\end{lemma}

\begin{proof}
Assume first that $K \ge 100$. The function $u(z) := \log K - \log |f(z)|$ is harmonic and non-negative in $\cD (z_0, r_0)$. Applying Harnack's inequality to it in $\cD(z_0, r_1)$ yields
\begin{align*}
[1 - 2 (1 + \log K)^{-2}] (\log K - \log |f(z_0)|) & \le \log K - \log |f(z)| \\
& \le [1 + 3 (1 + \log K)^{-2}] (\log K -\log |f(z_0)|)
\end{align*}
for any $z \in \cD (z_0, r_2)$. Hence, using \eqref{harcond} and $K \ge 100$, this implies that
$$
- 2 - \log |f(z_0)| \le - \log |f(z)| \le 2 - \log |f(z_0)|
$$
for any $z \in \cD (z_0, r_2)$, and the lemma follows.

Assume now that $K < 100$. Set $\tilde{f} (z) = \lambda f(z)$ and $\lambda = 100 / K$, so that
$$
\tilde{K} := \sup \bigl\{|\tilde {f}(z)|: z \in \cD (z_0, r_0) \bigr\} = 100.
$$
Then, $|\tilde{f} (z_0)| \ge 100/K^2 \ge 1/100 = 1/\tilde K$. Thus, $\tilde{f} (z)$ obeys the condition of the lemma with $\tilde K = 100$. By what we saw above, this implies
$$
|\tilde{f} (\zeta)| \le \exp(4) |\tilde{f} (z)|
$$
for any $z, \zeta \in \cD (z_0, r_2)$, $r_2 = (1 + \log 100)^{-2}r_1$. Replacing here $\tilde{f} (\cdot)$ by $\lambda f(\cdot)$, one obtains the statement.
\end{proof}

\begin{corollary}\label{cor:6.1}
Let $F(w)$ be an analytic function defined in the disk $\cD (w_0,r_0)$. Assume that $\tau_0 := |\partial_w F_{|_{w_0}}\big| > 0$. Assume also that $M_0 := \sup_{\cD(z_0,r_0)} |F(w)| < \infty$.
\begin{itemize}

\item[(1)]  If $|F(w_0)| < r_0 \tau_0 / (200 (1 + \log (\max (100, M_0)))^2)$, then there exists $w'_0 \in \cD (w_0,2r_1)$ with $r_1 = 100 (1 + \log (\max (100, M_0)))^2 \tau_0^{-1} |F(w_0)|$ such that $F(w'_0) = 0$.

\item[(2)] If $|F(w_0)| < (\min(1,r_0))^2 (\min(1,\tau_0))^2 /(200 \max(1,M_0) (1 + \log (\max (100, M_0))))^{2}$, then $F(w) \neq 0$ for any $w \in \cD (w_0,r'_0) \setminus \{w'_0\}$ with $r'_0 = \min(1,\tau_0) (\min(1,r_0))^2 / (8 \max(1,M_0))$. Moreover, $w'_0$ is a simple zero of $F$.

\end{itemize}
\end{corollary}

\begin{proof}
$(1)$ Set $r_2 = r_1 (1 + \log (\max (100, M_0)))^{-2}$. One has $r_1 < r_0/2$, $r_2 < r_1/8$. Set $M'_0 = \max_{|w-w_0| = r_1} |F(w)|$, $M''_0 = \max_{|w - w_0| = r_2} |F(w)|$. Due to the Cauchy inequality, one has $M''_0 \ge r_2 \tau_0$. So, $|F(w_1)| \ge r_2 \tau_0$ for some $|w_1 - w_0| = r_2$. Set $\lambda_0 = (r_2 \tau_0 M'_0)^{-1/2}$, $g(w) := \lambda_0 F(w)$. Then,
\begin{equation}\nn\begin{split}
\tilde {M}_0 := \sup_{w \in \cD (w_0,r_1)} |g(w)| = \lambda_0 M'_0, \\
|g(w_1)| \ge \lambda_0 r_2 \tau_0 = 1/(\lambda_0 M'_0) = 1/\tilde {M}_0.
\end{split}
\end{equation}
Note that,
$$
|g(w_1)| |g(w_0)|^{-1} = |F(w_1)| |F(w_0)|^{-1} \ge r_2 \tau_0 |F(w_0)|^{-1} = 100 > \exp(4).
$$
Due to Lemma~\ref{lem:6.harn}, $g(w)$ must vanish at some point $w'_0 \in \cD (w_1,r_1)$. Clearly, $w'_0 \in \cD (w_0,2r_1)$.

$(2)$ Assume now that $|F(w_0)| < \min(1,r_0)^2 \min(1,\tau_0)^2 (200 \max(1,M_0) (1 + \log (\max (100, M_0))))^{-2}$. Then, $r'_0 < r_0/2$. Using Cauchy inequalities, one gets for any $w \in \cD (w_0,r'_0)$,
$$
|\partial_w F| \ge |\partial_w F|_{w=w_0}|- \sup_{\cD(z_0,r_0/2)} |\partial^2_{w,w}F| |w - w_0| \ge \tau_0 - 4 r_0^{-2} M_0 r'_0 \ge \tau_0/2
$$
and
$$
|\partial^2_{w,w} F| \le 4 r_0^{-2} M_0.
$$
Let $|w' - w_0|< r'_0$. One has
\begin{align*}
|F(w')| & =|F(w') - F(w'_0)| \\
& \ge |\partial_w F|_{w=w'_0}| |w' - w'_0| - \frac{1}{2} \sup_{\cD (w_0,r'_0)} \left| \partial^2_{w,w} F \right| |w' - w'_0|^2 \\
& \ge \tau_0 |w' - w'_0|/2 - \frac{M_0}{r_0^2} |w' - w'_0|^2 \\
& \ge |w' - w'_0| \left( \frac{\tau_0}{2} -\frac{M_0}{r_0^2} r'_0 \right) \\
& > 0,
\end{align*}
provided that $w' \neq w'_0$. Moreover, this calculation shows that $w'_0$ is a simple zero of $F$.
\end{proof}

\begin{lemma}\label{em:implicitanalyticont}
\begin{itemize}

\item[(1)] Let $F(z,w)$ be an analytic function defined in the poly-disk $\cD (z_0,w_0; p_0,r_0) := \{ z \in \IC : |z - z_0| < p_0 \} \times \{ w \in \IC : |w - w_0| < r_0 \}$. Assume that the following conditions hold:

$(a)$ $\tau_0 := \big| \partial_w F_{|_{(z_0,w_0)}} \big| > 0$,

$(b)$
$$
|F(z_0,w_0)| \le \epsilon_1 := \frac{\tau_1^2(\min (1,r_0))^2}{10^8 M_1^2 (1 + \log (\max (100, M_0)))^2},
$$
$\tau_1 = \min(1,\tau_0)$, $M_1 = \max(1,M_0)$, where $M_0 := \sup_{\cP(z_0,w_0,p_0,r_0)} |F(z,w)|$. Set $r = \epsilon_1 (\min (1, p_0, r_0))^2/M_1$. Then, for any $|z - z_0| < r$, there exists a unique $\phi(z)=w$, $|w - w_0| < r_1 := 400 (1 + \log (\max (100, M_0)))^2 (\tau_1)^{-1} \epsilon_1$, such that $F(z,w) = 0$. For $z = z_0$, we have $|\phi(z_0) - w_0| < 400 (1 + \log (\max (100, M_0)))^2 (\tau_1)^{-1} |F(z_0,w_0)|$. Finally, $\phi(z)$ is analytic in the disk $\cD(z_0,r)$.

\item[(2)] Let $\phi_0(z)$ be an analytic function defined in the disk $\cD (z_0,\sigma_0)$, $0 < \sigma_0 \le 1$, and let $f(z,w)$ be an analytic function defined in the domain $\cL (\phi_0, \cD (z_0, \sigma_0), \rho_0)$, $0 < \rho_0 < 1$. Assume that the following conditions hold,

$(\alpha)$ $\sup_{z,w} |\partial_w f| \le 1/2$,

$(\beta)$
$$
|f(z, \phi_0(z)) - \phi_0(z)| < \epsilon_1 := \frac{\sigma_0^2 \rho_0^2}{10^{10} M_1^3 (1 + \log (\max (100, M_1)))^2}
$$
for any $z \in \cD (z_0,\ve_0)$, where $\ve_0\le \sigma_0-\rho_0$, $M_1 := \max(1,M_0)$, $M_0 := \sup_{z} |\phi_0(z)| + \rho_0 + \sup_{z,w} |f(z,w)|$.

Then, for any $|z - z_0| < \ve_0$, there exists $w = \phi(z)$,
$$
|\phi(z)-\phi_0(z)| < 10^3 (1 + \log (\max (100, M_1)))^2 |f(z,\phi_0(z)) - \phi_0(z)|,
$$
such that $f(z,\phi(z)) = \phi(z)$.  Furthermore, $\phi(z)$ is analytic in the disk $\cD(z_0,\ve_0-\ve_1)$. Finally, $w\neq f(z,w)$ if $|z - z_0| < \ve_0$, $|w-\phi_0(z)|<\rho_0$ and $w\neq \phi(z)$.

\end{itemize}
\end{lemma}

\begin{proof}
Let $|z - z_0| < r$. Due to Cauchy inequalities, $|\partial_z F|\le M_1/p_0$, $|\partial^2_{z,w}F| \le M_1/p_0r_0$. This implies
\begin{equation}\label{eqz0zcorrections}
|F(z,w_0)| \le \epsilon_1 + M_1 |z - z_0| / p_0 < 2 \epsilon_1, \quad |\partial_w F|_{z,w_0}| > \tau_0 - M_1 |z - z_0| / p_0 r_0 > \tau_1 / 2.
\end{equation}
Therefore, Corollary~\ref{cor:6.1} may be applied and it follows that there exists $w \in \cD (w_0, 2 \tilde r)$ with $\tilde r = 100 (1 + \log (\max (100, M_0)))^2 (\tau_1/2)^{-1} |F(z,w_0)| < r_1 / 2$ such that $F(z,w) = 0$. Moreover, $F(z,w') \neq 0$ for any $w' \in \cD (w_0,r'_0) \setminus \{w\}$, where $r'_0 = (\tau_1/2) \min(1,r_0)^2 / (8 M_1) > r_1$. Set $\phi(z) = w$. To finish part $(1)$ we have to show that $\phi$ is analytic. By the residue theorem,
$$
\frac{1}{2\pi i} \oint_{|w - w_0| = r_1} w \, \frac{F_w(z,w)}{F(z,w)} \, dw = \phi(z),
$$
and analyticity follows.

To prove part $(2)$, set $F(z,w) := w - f(z,w)$. Then, $|\partial_w F| \ge 1/2$ for any $(z,w)$. Furthermore, $\sup_{z,w} |F(z,w)| \le \sup_{z} |\phi_0(z)| + \rho_0 + \sup_{z,w} |f(z,w)| \le M_1$. Let $z'_0 \in \cD (z_0,\ve_0)$ be arbitrary, $w'_0 := \phi_0(z'_0)$. For $z \in \cD (z'_0,\rho_0/2)$, one has $|\partial_z\phi_0| \le 2 \rho_0^{-1} M_0$, due to Cauchy estimates. Hence, $|w - \phi_0(z)| < |w-w'_0| + |\phi_0(z) - \phi_0(z'_0)| < \sigma_0$, provided that $|w - w'_0| < \sigma_0 / 2 =: r'_0$, $|z - z'_0| < \sigma_0 \rho_0 / 4 M_1 =: p'_0$. So, the function $F(z,w)$ is well-defined and analytic in the poly-disk $\cP (z'_0,w'_0,p'_0,r'_0)$. Due to condition $(\beta)$, one has $|F(z'_0,w'_0)| < \epsilon_1$, where $\tau_1 = \min(\tau,1) \ge 1/2$. Due to part $(1)$ of the lemma, applied to the function $F(z,w)$ in the poly-disk $\cD(z'_0,w'_0;p'_0,r'_0)$, for any $|z - z'_0|<r$ with some $r>0$, there exists a unique $w = \phi(z)$ such that $\phi(z) = f(z,\phi(z))$ and $|\phi(z) - w'_0| < 10^3 (1 + \log (\max (100, M_1)))^2 |f(z,\phi_0(z)) - \phi_0(z)|$. Moreover, $\phi(z)$ is analytic in the disk $\cD(z'_0,r)$. Assume that $w_1= f(z,w_1)$ for some $|z - z_0| < \ve_0$, $|w_1-\phi_0(z)|<\rho_0$. Then
$$
|w_1 - \phi(z)| = \Big| \int_{\phi(z)}^{w_1} \partial_w f(z,w) \, dw \Big| \le (\sup |\partial_w f(z,w)|) |w_1 - \phi(z)| \le \frac{1}{2} |w_1 - \phi(z)|.
$$
Hence, $w_1 = \phi(z)$. This finishes part $(2)$.
\end{proof}

Let $a_1(x,u)$, $a_2(x, u)$, $b(x, u)$, $g(x)$ be real functions such that:
\begin{enumerate}

\item[(i)] $g(x)$ is a $C^2$-function on some interval $(-\alpha_0, \alpha_0)$.

\item[(ii)] $a_1(x, u)$, $a_2(x,u)$, $b^2(x,u)$ are $C^2$-functions in the domain $\cL_{\IR}\bigl(g, \rho_0\bigr)$, $\rho_0<1$.

\item[(iii)] $a_1(x,u) > a_2(x,u)$ for any $(x, u)$; $b(0, u) = 0$ for any $u \in \bigl(g(0) - \rho_0, g(0) + \rho_0\bigr)$.

\item[(iv)] $\big | a_i(x, u) - g(x) \big | < \rho_0/4$, for any $(x,u)$, $i = 1, 2$; $|b(x,u)| < \rho_0/4$ for any $x, u$.

\item[(v)] $|\partial_u\, a_i| < 1/2$ for any $(x,u)$, $i = 1, 2$; $|\partial_u\, b^2|< |b|/4$ for any $(x, u)$.

\end{enumerate}

Consider the following equation
\begin{equation}\label{eq:6-1}
\chi(x,u) := \bigl(u - a_1(x,u)\bigr)\bigl(u - a_2(x,u)\bigr) - b(x,u)^2 = 0.
\end{equation}

\begin{lemma}\label{lem:6-1}
For any $x \in (-\alpha_0, \alpha_0)$, the equation \eqref{eq:6-1} has exactly two solutions, $\zeta_+(x)$ and $\zeta_-(x)$. The functions $\zeta_+(x)$, $\zeta_-(x)$ are continuously differentiable on $(-\alpha_0, \alpha_0)$ and obey
\begin{equation} \label{eq:6-1''}
\max (a_1 \bigl(x, \zeta_+(x) \bigr), a_2 \bigl(x, \zeta_+(x) \bigr) + |b(x,\zeta_+(x))|)  \le \zeta_+(x) \le a_1 \bigl(x, \zeta_+(x) \bigr) + |b(x,\zeta_+(x))|,
\end{equation}
\begin{equation}\label{eq:6-1'''}
a_2 \bigl(x, \zeta_-(x)\bigr) - |b(x,\zeta_-(x))| \le \zeta_-(x) \le \min (a_2 \bigl( x, \zeta_-(x) \bigr), a_1 \bigl( x, \zeta_+(x) \bigr) - |b(x,\zeta_+(x))|),
\end{equation}
\begin{equation}\label{eq:6-1''''}
g(x) - \rho_0/2 \le \zeta_\pm(x) \le g(x) + \rho_0/2.
\end{equation}
\end{lemma}

\begin{proof}
Consider the following equations,
\begin{align}
u & = (1/2) \left[a_1(x,u) + a_2(x,u) + \bigl((a_1(x, u) - a_2(x,u))^2 + 4b^2(x,u)\bigr)^{1/2}\right] , \label{eq:6-2} \\[6pt]
u & = (1/2) \left[a_1(x,u) + a_2(x,u) - \bigl((a_1(x,u) - a_2(x,u))^2 + 4b^2(x,u)\bigr)^{1/2}\right] . \label{eq:6-3}
\end{align}
Note that $\chi(x,u) = 0$ if and only if \eqref{eq:6-2} or \eqref{eq:6-3} holds. Denote by $\vp_+(x,u)$ (resp., $\vp_-(x,u)$) the expression on the right-hand side of \eqref{eq:6-2} (resp., \eqref{eq:6-3}) and by $r(x,u)$ the square root in \eqref{eq:6-2} (and \eqref{eq:6-3}). Note the following relations,
\begin{align}
& \max \Bigl\{ \bigl( a_1(x,u) - a_2(x,u) \bigr),\ 2 |b(x,u)| \Bigr \} \le r(x,u) \le \bigl( a_1(x,u) - a_2(x,u) + 2|b(x,u)|\bigr), \label{eq:6-4}\\[6pt]
& \max \Bigl\{ a_1(x,u), (1/2) \bigl[ a_1(x,u) + a_2(x,u) + 2|b(x,u)| \bigr] \Bigr\} \le \vp_+(x,u) \le a_1(x,u) + |b(x,u)|, \label{eq:6-5}\\[6pt]
& a_2(x,u) - |b(x,u)| \le \vp_-(x,u) \le \min \Bigl\{ a_2(x,u), (1/2) \bigl[ (a_1(x,u) + a_2(x,u)) - 2|b(x,u)| \bigr] \Bigr\}. \label{eq:6-6}
\end{align}
Assume that $\chi(x_0, u_0) = 0$ for some $(x_0, u_0) \in \cL \bigl( g, \rho_0 \bigr)$. Then, either $u_0 = \vp_+(x_0, u_0)$ or $u_0 = \vp_-(x_0, u_0)$. Assume $u_0 = \vp_+(x_0, u_0)$. Then, due to \eqref{eq:6-5} and conditions (i)--(v), we obtain
\begin{equation}\label{eq:6-7}
\begin{split}
\partial_u \chi \Big |_{\xumap} & = \Bigl\{ (1 -\partial_u a_1)(u - a_2) + (1-\partial_u a_2)(u - a_1) - \partial_u b^2 \Bigr\} \Big |_{\xumap} \\
& \ge (1 - 1/2) \bigl( \vp_+ - a_2 \bigr) + (1 - 1/2) \bigl( \vp_+ - a_1 \bigr) - |b|/4 \Big |_{\xumap} \\
& \ge (1/4) \bigl( (a_1 - a_2)+|b| \bigr) \Big|_{\xumap} > 0.
\end{split}
\end{equation}
Thus $\chi(x, u)$ satisfies all conditions of the implicit function theorem in some neighborhood of $\xumap$.  Consider the equation
\begin{equation}\label{eq:6-8}
u = a_1(0, u),
\end{equation}
$u \in \bigl( g(0) - \rho_0, g(0) + \rho_0 \bigr)$. Due to condition~(iv), $a_1(0, u) \in I_0 = \bigl[ g(0) - \rho_0/4, g(0) + \rho_0/4 \bigr]$ for any $u \in \bigl( g(0) - \rho_0, g(0) + \rho_0 \bigr)$. Hence, $u \to a_1(0, u)$ maps $I_0$ into itself.  Since $|\partial_u a_1| < 1/2$, this map is contracting.  Therefore, the equation \eqref{eq:6-8} has a unique solution in $I_0$, which we denote by $\zeta_+(0)$. Clearly, $u_0 = \zeta_+(0)$ satisfies $u_0 = \vp_+(0, u_0)$. Due to \eqref{eq:6-7}, for any $x$ in some neighborhood of $x_0 = 0$, the equation \eqref{eq:6-1} has a unique solution $\zeta_+(x)$ belonging to some small neighborhood of $u_0$. Clearly, $\zeta_+(x) = \vp_+(x, \zeta_+(x))$.

Assume that $\chi(x_1, u_1) = 0$ for some $(x_1, u_1) \in \cL_{\IR} \bigl( g, (-\alpha_0, \alpha_0), \rho_0 \bigr)$. Then, due to \eqref{eq:6-5} and \eqref{eq:6-6},
\begin{equation}\label{eq:6-9}
a_2(x_1, u_1) - |b(x_1, u_1)| \le u_1 \le a_1(x_1, u_1) + |b(x_1, u_1)|.
\end{equation}
Combining \eqref{eq:6-9} with condition~(iv), we obtain
\begin{equation}\label{eq:6-10}
g(x_1) - \rho_0/2 \le u_1 \le g(x_1) + \rho_0/2.
\end{equation}
It follows from \eqref{eq:6-10} and the above arguments that, given $(\bar x, \bar u) \in \cL_{\IR} \bigl( g, (-\alpha_0, \alpha_0), \rho_0 \bigr)$ such that $\bar u = \vp_+ (\bar x, \bar u)$, there exists a unique $C^1$-function $\zeta_+(x)$ defined on $(-\alpha_0, \alpha_0)$ such that $\chi(x, \zeta_+(x)) = 0$, $\zeta_+(x) = \vp_+(x, \zeta_+(x))$, $\zeta_+(\bar x) = \bar u$. In a similar way we define $\zeta_-(x)$, $x \in (-\alpha_0, \alpha_0)$. Let $u_1 = \zeta_+(0)$ and $u_2 = \zeta_-(0)$. Then, $u_i = a_i(0, u_i)$, $i = 1, 2$. Since $u \to a_1(0, u)$ is a contraction, $|u_1 - u_2| > |a_1(0, u_1) - a_1(0, u_2)| = |u_1 - a_1(0, u_2)|$. Due to (iii), $a_1(0, u_2) > a_2(0, u_2) = u_2$. Therefore, $u_2 \ge u_1$ is impossible.  So, $\zeta_+(0) > \zeta_-(0)$. Since $\zeta_+(0) \ne \zeta_-(0)$, $\zeta_+(x) \ne \zeta_-(x)$ for any $x \in (-\alpha_0, \alpha_0)$.  Hence, $\zeta_+(x) > \zeta_-(x)$ for any $x \in (-\alpha_0, \alpha_0)$.

The estimates \eqref{eq:6-1''}, \eqref{eq:6-1'''} follow from \eqref{eq:6-5} and \eqref{eq:6-6}. The estimate \eqref{eq:6-1''''} follows from  \eqref{eq:6-1''}, \eqref{eq:6-1'''}.
\end{proof}

\begin{remark}\label{rem:zetazetasmalleras}
For our applications of Lemma~\ref{lem:6-1} we need to generalize its statement for some cases when the crucial conditions $|\partial_u a_i| < 1/2$ in {\rm (v)}, due to which we can apply the implicit function theorem, fail. In Definition~\ref{def:4a-functions} below we introduce inductively the classes of functions for which we need the statement. We need also to accommodate the case when the functions depend on some parameter $\theta$.
\end{remark}

Before we proceed with the definition of the cases mentioned in the previous remark we need the following lemma which solves the inequality $|\chi(x,u)|<\epsilon$ rather than the equation $\chi(x,u)=0$ from Lemma~\ref{lem:6-1}.
\begin{lemma}\label{4:generalquadratic}
Let $a_1 > a_2$ and $b$ be reals. Let $u$ be a solution of the quadratic inequality
\begin{equation}\label{eq:4generalquadraticinequ}
|(u - a_1)(u - a_2) - b^2| < (a_1 - a_2)^2/4.
\end{equation}
Let $\lambda = \lambda(u) := (a_1 - a_2)^{-2}[(u - a_1)(u - a_2) - b^2]$, $\gamma = \gamma(u) := (\sqrt{1 + 4\lambda}-1)/2$. Either
\begin{equation}\label{eq:4ineqdichotomy+}
u \ge \max \bigl\{ a_1 - |\gamma| (a_1 - a_2), (1/2)\bigl[a_1 + a_2 + 2 |b| \bigr] \bigr\} \ge a_2 + (1/2) \bigl(a_1 - a_2 \bigr) + |b|,
\end{equation}
or
\begin{equation}\label{eq:4ineqdichotomy-}
u \le \min \bigl\{ a_2 + |\gamma| (a_1-a_2), (1/2) \bigl[ (a_1 + a_2) - 2 |b| \bigr] \bigr\} \le a_1 - (1/2) \bigl( a_1 - a_2 \bigr) - |b|.
\end{equation}
In any event, $a_2 - |\gamma| (a_1 - a_2) - |b| \le u \le a_1 + |\gamma| (a_1 - a_2) + |b|$.
\end{lemma}

\begin{proof}
One has
\begin{equation}\label{eq:4a-functions11N}
u^2 - (a_1 + a_2) u + a_1 a_2 - b^2 - \lambda(a_1-a_2)^2 = 0.
\end{equation}
Therefore, $u$ obeys one of the following equations
\begin{align}
u & = \varphi_+(u,\lambda):=(1/2) \left[a_1 + a_2 + \bigl((a_1 - a_2)^2(1+4\lambda) + 4b^2\bigr)^{1/2}\right] , \label{eq:4-2NEWN} \\[6pt]
u & = \varphi_-(u,\lambda):=(1/2) \left[a_1 + a_2 - \bigl((a_1 - a_2)^2(1+4\lambda) + 4b^2\bigr)^{1/2}\right] . \label{eq:4-3NEWN}
\end{align}
Note that \eqref{eq:4generalquadraticinequ} implies, in particular, that $1+4\lambda>0$. One has
\begin{align}
& \max \bigl\{ a_1 + \gamma (a_1 - a_2), (1/2) \bigl[ a_1 + a_2 + 2 |b| \bigr] \bigr\} \le \vp_+ \le a_1 + \gamma (a_1 - a_2) + |b|, \label{eq:6-5NEW} \\[6pt]
& a_2 - \gamma (a_1 - a_2) - |b| \le \vp_- \le \min \bigl\{ a_2 - \gamma (a_1 - a_2), (1/2) \bigl[ (a_1 + a_2) - 2 |b| \bigr] \bigr\}, \label{eq:6-6NEW}
\end{align}
and the statement follows.
\end{proof}

\begin{remark}\label{rem:5.pmlemma}
In the definition below we refer to the cases in the last lemma as the $+$-case and the $-$-case, respectively.
\end{remark}

\begin{defi}\label{def:4a-functions}
$(1)$  Let $g_ 0(x)$ be a $C^2$-function on $(-\alpha_0, \alpha_0)$. Let $0 < \lambda \le 1$. Let $a_1(x,u,\theta)$, $a_2(x, u,\theta)$, $b^2(x, u,\theta)$ be $C^2$-functions which obey the conditions {\rm (i)--(iii)} before Lemma~\ref{lem:6-1}, $\rho_0 < 1/32$ for each fixed $\theta \in \Theta$. Assume in addition that $|u - a_i|, b^2, |\partial_{u,\theta}^{\alpha_1,\alpha_2} a_i|, |\partial_{u,\theta}^{\alpha_1,\alpha_2} b^2| < \lambda/64$ for any $x,u,\theta$ and any  $|(\alpha_1,\alpha_2)| \le 2$. Set
\begin{equation}\label{eq:4a-functions1}
\begin{split}
f(x,u,\theta,1) = u - a_1 - \frac{b^2}{u - a_2}, \quad f(x,u,\theta,2) = u - a_2 - \frac{b^2}{u - a_1}, \quad (x,u) \in \cL_{\IR} \bigl( g, \rho_0 \bigr),\\
\mathfrak{F}^{(1)}_{\mathfrak{g}^\one, \mathfrak{r}^\one, \lambda} (a_1,a_2,b^2) = \{ f(\cdot,j) : j = 1, 2 \}, \quad \mathfrak{g}^\one := g_0, \quad \mathfrak{r}^\one := \rho_0,\quad\quad\quad\quad\\ \mu^{(f(\cdot,1))} = (u-a_2), \quad \mu^{(f(\cdot,2))} = (u-a_1), \quad \chi^{(f(\cdot,i))} = \mu^{(f(\cdot,i))} f(\cdot,i) \quad\quad\quad\quad\quad\\
\tau^{(f(\cdot,i))}(x,u,\theta):= a_1(x,u,\theta) - a_2(x,u,\theta),\quad i=1,2.\quad\quad\quad\quad\quad\quad\quad\quad
\end{split}
\end{equation}
Here, $f(\cdot,1)$ is defined if $u - a_2(x,u,\theta) \neq 0$, and $f(\cdot,2)$ if $u - a_1(x,u,\theta) \neq 0$. Set $\mathfrak{F}^{(1)}_{\mathfrak{g}^\one, \mathfrak{r}^\one, \lambda} = \bigcup_{a_1,a_2,b^2} \mathfrak{F}^{(1)}_{\mathfrak{g}^\one, \mathfrak{r}^\one, \lambda} (a_1,a_2,b^2)$. With some abuse of notation we will write $f \in \mathfrak{F}^{(1)}_{\mathfrak{g}^\one, \mathfrak{r}^\one, \lambda} (f_1, f_2, b^2)$ for $f \in \mathfrak{F}^{(1)}_{\mathfrak{g}^\one, \mathfrak{r}^\one, \lambda} (a_1, a_2, b^2)$ with $f_i := u - a_i$, $i = 1, 2$.
Note here that we do not need $g$ to be dependent on $\theta$. Furthermore, our ``main'' variables are $x,u$, and we view $\theta$ as a parameter. Below we will sometimes drop $\theta$ from the notation. Requirements on $\theta$-dependent quantities are then implicitly assumed to hold for all $\theta$.

$(2)$ Let $g_ 0(x)$, $g_1(x)$ be $C^2$-functions on $(-\alpha_0, \alpha_0)$ and $0 < \rho_1 < \rho_0$. Assume that $\cL_{\IR} \bigl( g_{0}, \rho_{0} \bigr) \supset \cL_{\IR} \bigl(g_{1}, \rho_{1}\bigr)$. Assume also that $g_0(0) = g_1(0)$. Set $\mathfrak{g}^{(2)} = (g_0,g_{1})$, $\mathfrak{r}^{(2)} = (\rho_0,\rho_{1})$. Let $f_i \in \mathfrak{F}^{(1)}_{\mathfrak{g}^{(1)}, \mathfrak{r}^\one, \lambda} (a_{i,1}, a_{i,2}, b^2_i)$, $i = 1, 2$. Let $b(\cdot,\theta)$ be $C^2$-smooth in $\cL_{\IR} \bigl(g_{0}, \rho_{0} \bigr)$. Assume that the following conditions hold: $(a)$ $\chi^{(f_1)}<\chi^{(f_2)}$ for all $(x,u) \in \cL_{\IR} \bigl(g_{1}, \rho_{1}\bigr)$, $(b)$ $|f_i| < (\min_j \lambda \tau^{(f_j)})^{10}$ for all $(x,u) \in \cL_{\IR} \bigl( g_{1}, \rho_{1} \bigr)$, $(c)$ the inequality $|(u - a_{i,1})(u-a_{i,2})-b_i^2|<(a_{i,1}-a_{i,2})^2/4$, which holds for all $x,u$ due to condition $(b)$, is either in the $+$-case
for all $(x,u) \in \cL_{\IR} \bigl( g_{1}, \rho_{1} \bigr)$, $i = 1, 2$, or in the $-$-case for all $(x,u) \in \cL_{\IR} \bigl( g_{1}, \rho_{1} \bigr)$, $i = 1, 2$; furthermore, $f_i = (u - a_{i,1}) - b^2_i (u - a_{i,2})^{-1}$ in the $+$ case, respectively, $f_i = (u - a_{i,2}) - b^2_i (u - a_{i,1})^{-1}$
for all $(x,u) \in \cL_{\IR} \bigl( g_{1}, \rho_{1} \bigr)$, $i = 1, 2$, in the $-$-case, $(d)$ $|b| < (\min_j \lambda \tau^{(f_j)})^{10}$,
$|\partial_u b^2|,|\partial_\theta b^2| < (\min_j \lambda \tau^{(f_j)})^{10} |b|$, $|\partial^2_u b^2|,|\partial^2_\theta b^2| < (\min_j \lambda \tau^{(f_j)})^{10}$ for any $(x,u) \in \cL_{\IR} \bigl( g_{1}, \rho_{1} \bigr)$ and any $\theta$, $(e)$ $f_i (0,u) = u - g_1(0)$, $b(0,u) = 0$ for any $u$, $i = 1, 2$. Set
\begin{equation}\label{eq:4a-functionsN2}
\begin{split}
f(x,u,\theta,1) = f_1 - \frac{b^2}{f_2},\quad f(x,u,\theta,2) = f_2 - \frac{b^2}{f_1}, \quad\quad\quad \quad\quad\quad\quad\\
\mathfrak{F}^{(2)}_{\mathfrak{g}^{(2)},\mathfrak{r}^{(2)},\lambda} (f_1,f_2,b^2) = \{f(\cdot,j) : j = 1, 2\},\quad\quad\quad\quad\quad\quad\quad\quad\\
\mu^{(f(\cdot,1))}=\mu^{(f_1)}\mu^{(f_2)}f_2,\quad\mu^{(f(\cdot,2))}=\mu^{(f_2)}\mu^{(f_1)}f_1,\quad
\chi^{(f(\cdot,i))}=\mu^{(f(\cdot,i))}f(\cdot,i),\quad\quad\quad\quad\\
\tau^{(f(\cdot,i))}(x,u,\theta):= \chi^{(f_2)} - \chi^{(f_1)},\quad i=1,2.\quad\quad\quad\quad\quad\quad\quad\quad
\end{split}
\end{equation}
Let $f \in \mathfrak{F}^{(2)}_{\mathfrak{g}^{(2)}, \mathfrak{r}^{(2)}, \lambda}(f_1,f_2,b^2)$. We say that $f \in \mathfrak{F}^{(2,\pm)}_{\mathfrak{g}^{(2)}, \mathfrak{r}^{(2)}, \lambda}(f_1,f_2,b^2)$, according to the dichotomy in $(c)$. Set $\mathfrak{F}^{(2,\pm)}_{\mathfrak{g}^{(2)}, \mathfrak{r}^{(2)}, \lambda} = \bigcup_{f_1,f_2,b^2} \mathfrak{F}^{(2,\pm)}_{\mathfrak{g}^{(2)}, \mathfrak{r}^{(2)}, \lambda} (f_1,f_2,b^2)$, $\mathfrak{F}^{(2)}_{\mathfrak{g}^{(2)}, \mathfrak{r}^{(2)}, \lambda} = \mathfrak{F}^{(2,+)}_{\mathfrak{g}^{(2)}, \mathfrak{r}^{(2)}, \lambda} \cup \mathfrak{F}^{(2,-)}_{\mathfrak{g}^{(2)}, \mathfrak{r}^{(2)}, \lambda}$, $\sigma(f) = \pm 1$ if $f \in \mathfrak{F}^{(2,\pm)}_{\mathfrak{g}^{(2)}, \mathfrak{r}^{(2)}, \lambda}(f_1,f_2,b^2)$.  We introduce also the following sequence $\hat \sigma(f):=(\sigma(f))$, consisting just of one term.

$(3)$ We define the classes of functions $\mathfrak{F}^{(\ell)}_{\mathfrak{g}^{(\ell)},\lambda}$ inductively. Assume that $\mathfrak{F}^{(t)}_{\mathfrak{g}^{(t)}}$ are already defined for $t = 1, \dots, \ell-1$, where $\ell \ge 3$. Let $g_t(x)$ be a $C^2$-function on $(-\alpha_0, \alpha_0)$, $0< \rho_{t+1}  < \rho_{t} < 1$, $t = 0, \dots, \ell-1$. Assume that $\cL_{\IR}\bigl(g_{\ell-2}, \rho_{\ell-2}\bigr) \supset \cL_{\IR}\bigl(g_{\ell-1}, \rho_{\ell-1}\bigr)$. Set $\mathfrak{g}^{(t)} = (g_0, \dots, g_{t-1})$, $\mathfrak{r}^{(t)} = (\rho_0,\dots,\rho_{t-1})$. Let $f_i \in \mathfrak{F}^{(\ell-1)}_{\mathfrak{g}^{(\ell-1)},\mathfrak{r}^{(\ell-1)}, \lambda} (f_{i,1},f_{i,2},b^2_i)$,  $i=1,2$. Assume that the following conditions hold: $(a)$ $\chi^{(f_1)} < \chi^{(f_2)}$, for all $(x,u) \in \cL_{\IR}\bigl(g_{\ell-1}, \rho_{\ell-1}\bigr)$, $(b)$ $|f_i| < (\min_j \lambda \tau^{(f_j)})^{10}$ for all $(x,u) \in \cL_{\IR}\bigl(g_{\ell-1}, \rho_{\ell-1}\bigr)$, $(c)$ with $\chi^{(f_{i,j})} := u-a_{i,j}$, the inequality $|(u - a_{i,1})(u-a_{i,2})-\mu^{(f_i)}b_i^2|<(a_{i,1}-a_{i,2})^2/4$, which holds for all $x,u$ due to condition $(b)$, see the verification in \eqref{eq:4chipmderivation}, is either in the $+$-case for all $(x,u) \in \cL_{\IR}\bigl(g_{\ell-1}, \rho_{\ell-1}\bigr)$, $i = 1, 2$, or in the $-$-case for all $(x,u) \in \cL_{\IR}\bigl(g_{\ell-1}, \rho_{\ell-1}\bigr)$, $i = 1, 2$; furthermore, $f_i = f_{i,1} - b^2_i  f_{i,2}^{-1}$ in the $+$ case, respectively, $f_i = f_{i,2} - b^2_i a_{i,1}^{-1}$ for all $(x,u) \in \cL_{\IR}\bigl(g_{\ell-1}, \rho_{\ell-1}\bigr)$, $i = 1, 2$, in the $-$-case. $(d)$ $|b| < (\lambda \min_j \tau^{(f_{j})})^{10}$, $|\partial_u b^2|,|\partial_\theta b^2| < (\lambda \min_j \tau^{(f_{j})})^{10} |b|$, $|\partial^2_u b^2|, |\partial^2_\theta b^2| < (\lambda \min_j \tau^{(f_{j})})^{10}$, $(e)$ $\hat\sigma(f_1) = \hat\sigma(f_2)$. Here $\tau^{(f)}$, $\sigma(f)$ and $\hat \sigma(f)$ $)$ are defined inductively see part $(4)$ below. Set
\begin{equation}\label{eq:4a-functionsN2-2}
\begin{split}
f(x,u,\theta,1) = f_1 - \frac{b^2}{f_2},\quad f(x,u,\theta,2) = f_2 - \frac{b^2}{f_1}, \\
\mathfrak{F}^{(\ell)}_{\mathfrak{g}^{(\ell)},\mathfrak{r}^{(\ell)},\lambda} (f_1,f_2,b^2) = \{f(\cdot,j) : j = 1, 2\}.
\end{split}
\end{equation}
We say that $f \in \mathfrak{F}^{(\ell,\pm)}_{\mathfrak{g}^{(\ell)},\mathfrak{r}^{(\ell)}, \lambda} (f_1,f_2,b^2)$, according to the dichotomy in
$(c)$. Set $\mathfrak{F}^{(\ell,\pm)}_{\mathfrak{g}^{(\ell)}, \mathfrak{r}^{(\ell)}, \lambda} = \bigcup_{f_1,f_2,b^2} \mathfrak{F}^{(\ell,\pm)}_{\mathfrak{g}^{(\ell)}, \mathfrak{r}^{(\ell)}, \lambda} (f_1,f_2,b^2)$, $\mathfrak{F}^{(\ell)}_{\mathfrak{g}^{(\ell)}, \mathfrak{r}^{(\ell)}, \lambda} (f_1,f_2,b^2) = \mathfrak{F}^{(\ell,+)}_{\mathfrak{g}^{(\ell)}, \mathfrak{r}^{(\ell)}, \lambda} (f_1,f_2,b^2) \cup \mathfrak{F}^{(\ell,-)}_{\mathfrak{g}^{(\ell)}, \lambda} (f_1,f_2,b^2)$, $\mathfrak{F}^{(\ell)}_{\mathfrak{g}^{(\ell)}, \mathfrak{r}^{(\ell)}, \lambda} = \mathfrak{F}^{(\ell,+)}_{\mathfrak{g}^{(\ell)}, \mathfrak{r}^{(\ell)}, \lambda} \cup \mathfrak{F}^{(\ell,-)}_{\mathfrak{g}^{(\ell)}, \mathfrak{r}^{(\ell)}, \lambda}$.

$(4)$ Let $f \in \mathfrak{F}^{(1)}_{\mathfrak{g}^\one} (a_1,a_2,b^2)$. With $f_i := u - a_i$, we introduce for convenience $\chi^{(f_i)} := f_i$, $\mu^{(f_i)} := 1$, $\tau^{(f_i)} := 1$, $\sigma(f_i) := 1$, $i = 1, 2$.

Let $f \in \mathfrak{F}^{(\ell,\pm)}_{\mathfrak{g}^{(\ell)},\mathfrak{r}^{(\ell)},\lambda} (f_1,f_2,b^2)$. Set
\begin{equation}\label{eq:4a-bfunctions1a}
\begin{split}
\mu^{(f)} = \begin{cases} \mu^{(f_1)}\mu^{(f_2)}f_2 & \text {if} \quad f = f_1 - \frac{b^2}{f_2}, \\ \mu^{(f_1)} \mu^{(f_2)} f_1 & \text {if} \quad f = f_2 - \frac{b^2}{f_1}, \end{cases} \\
\chi^{(f)} = \mu^{(f)} f,\quad\quad\quad\quad\quad\quad\quad\quad\quad\quad\\
\tau^{(f)} = (\chi^{(f_2)}-\chi^{(f_1)}) \tau^{(f_1)}\tau^{(f_2)},\quad\quad\quad\quad\quad\\
 \sigma(f) = \pm \sigma(f_{1}) = \pm \sigma(f_{2}) \text { according to
 $f \in \mathfrak{F}^{(\ell,\pm)}_{\mathfrak{g}^{(\ell)},\mathfrak{r}^{(\ell)},\lambda} (f_1,f_2,b^2)$}.
\end{split}
\end{equation}
The sequence $\hat \sigma(f)$ is defined just by attaching $\sigma(f)$ to $\hat \sigma(f_i)$ from the left, that is, $\hat \sigma (f) = (\sigma(f),\hat \sigma (f_i))$. Due to condition $(e)$ in the definition in part $(3)$, the result does not depend on $i=1,2$.
\end{defi}

\begin{remark}\label{rem:4.1smoothnesfi}
$(1)$ The parameter $\lambda$ is introduced in the definitions above only for the sake of stability under small perturbations which we establish in Lemma~\ref{lem:4stable} at the very end of this section. Clearly, $\mathfrak{F}^{(\ell,\pm)}_{\mathfrak{g}^{(\ell)}, \lambda} \subset \mathfrak{F}^{(\ell,\pm)}_{\mathfrak{g}^{(\ell)},1}$. Everywhere in this section, with the exception of Lemma~\ref{lem:4stable}, we always assume $\lambda = 1$ and we suppress $\lambda$ from the notation.

We remark here also that the quantities $0 < \rho_{t}$, $t = 0, \dots, \ell-1$ do not enter any inequalities in Definition~\ref{def:4a-functions}. Let $0 < \rho_{t,1} \le \rho_t$, $t = 0, \dots, \ell-1$ be such that $\cL_{\IR} \bigl(g_{\ell-2}, \rho_{\ell-2,1} \bigr) \supset \cL_{\IR} \bigl(g_{\ell-1}, \rho_{\ell-1,1} \bigr)$. If $f \in \mathfrak{F}^{(\ell,\pm)}_{\mathfrak{g}^{(\ell)}, \mathfrak{r}^{(\ell)}, \lambda}$, then also $f \in \mathfrak{F}^{(\ell,\pm)}_{\mathfrak{g}^{(\ell)}, \mathfrak{r}^{(\ell,1)}, \lambda}$, where  $\mathfrak{r}^{(t,1)} = (\rho_{0,1}, \dots, \rho_{t-1,1})$. We suppress $\mathfrak{r}^{(\ell)}$ from the notation, everywhere except Lemma~\ref{lem:6-1ell}. We will use it later on, starting from Section~\ref{sec.6}.

$(2)$ Let $f \in \mathfrak{F}^{(\ell)}_{\mathfrak{g}^{(\ell)}} (f_1,f_2,b^2)$. We remark here that Definition~\ref{def:4a-functions} implies in particular that $f_i$ is a $C^2$-smooth function in $\cL_{\IR}\bigl(g_{\ell-1}, \rho_{\ell-1} \bigr)\times \Theta$.

$(3)$ Once again, note we do not need $g_i$ to be dependent on $\theta$ and, to simplify notations, we suppress $\theta$ wherever it does not cause ambiguity.
\end{remark}

\begin{lemma}\label{4.fcontinuedfrac}
Suppose $f \in \mathfrak{F}^{(\ell)}_{\mathfrak{g}^{(\ell)}} (f_1,f_2,b^2)$. Then, the following statements hold:

$(1)$  $\max_j |f_j|, |\tau^{(f)}|, |\mu^{(f)}|, |\chi^{(f)}| < 2^{-2^{2\ell}}$ for all $(x,u) \in \cL_{\IR} \bigl( g_{\ell-1}, \rho_{\ell-1} \bigr)$. Furthermore, $|\chi^{(f_i)}| < (\min_j |\tau^{(f_j)}|)^{10}$ for all $(x,u,\theta) \in \cL_{\IR} \bigl( g_{\ell-1}, \rho_{\ell-1} \bigr) \times \Theta$.

$(2)$ The functions $\mu^{(f)}$, $\chi^{(f)}$ are $C^2$-smooth, $|\partial^\alpha \mu^{(f)}|, |\partial^\alpha \chi^{(f)}| < 2^{-2^{2(\ell-1)} + 3}$, $|\alpha| \le 2$.

$(3)$ Let $\ell \ge 2$. Either $f_i \in \mathfrak{F}^{(\ell-1,+)}_{\mathfrak{g}^{(\ell-1)}} (f_{i,1}, f_{i,2}, b_i^2)$, $i = 1,2$, or $f_i \in \mathfrak{F}^{(\ell-1,-)}_{\mathfrak{g}^{(\ell-1)}} (f_{i,1}, f_{i,2}, b_i^2)$, $i = 1,2$. In the first case, $\chi^{(f_{i,1})}>-( \min_j \tau^{(f_{j})})^8(\chi^{(f_{i,2})} - \chi^{(f_{i,1})})$, $\chi^{(f_{i,2})}\ge (1/2)(\chi^{(f_{i,2})} - \chi^{(f_{i,1})})+(\prod_j\mu^{(f_{i,j})})^{1/2}|b_i|$
for all $(x,u) \in \cL_{\IR} \bigl( g_{\ell-1}, \rho_{\ell-1} \bigr)$, $i = 1,2$. In the second case, $\chi^{(f_{i,2})}<( \min_j \tau^{(f_{j})})^8(\chi^{(f_{i,2})} - \chi^{(f_{i,1})})$, $\chi^{(f_{i,1})}\le -(1/2)(\chi^{(f_{i,2})} - \chi^{(f_{i,1})})-(\prod_j\mu^{(f_{i,j})})^{1/2}|b_i|$
for all $(x,u) \in \cL_{\IR} \bigl( g_{\ell-1}, \rho_{\ell-1} \bigr)$, $i = 1,2$.

$(4)$ Let $\ell \ge 2$ and $f_i \in \mathfrak{F}^{(\ell-1)}_{\mathfrak{g}^{(\ell-1)}} (f_{i,1}, f_{i,2}, b_i^2)$. Then  $\sigma(f_{i,j}) = \sigma(f_{i',j'})$, for any $i, j, i', j'$.

$(5)$ $\sigma(f_{i}) \partial_u \chi^{(f_{i})} > (\tau^{(f_{{i}})})^2$ , $i = 1, 2$.

$(6)$ Assume $\chi^{(f)}(x_0,u_0) = 0$. Then, $\sgn f_1(x_0,u_0) \partial_u \chi^{(f)}|_{x_0,u_0} > (\tau^{(f)})^2|_{x_0,u_0}$.

$(7)$ $\partial^2_u \chi^{(f)} > (1/2) (\min_{i} \tau^{(f_{i})})^4$ for all $(x,u) \in \cL_{\IR} \bigl( g_{\ell-1}, \rho_{\ell-1} \bigr)$.

$(8)$ Let $f \in \mathfrak{F}^{(1)} (a_1,a_2,b^2)$. Assume that the following condition $(k^\zero)$ holds:

$(k^\zero)$ $\partial_\theta a_1 > k^\zero$, $\partial_\theta a_2 < -k^\zero$, $|f_i| < (k^\zero)^2/8$, $|\partial^2_\theta b^2| < (k^\zero)^2/8$, where $k^\zero > 0$ is a constant.

Then, $\partial^2_\theta \chi^{(f)} < -(k^\zero)^2 $. Furthermore, assume in addition that $\chi^{(f)}(x,u,\theta) = \chi^{(f)}(x,u,-\theta)$, $\theta \in \Theta = (-\theta_0,0) \cup (0,\theta_0)$. Then,
\begin{equation}\label{eq:4a-thetaest}
\partial_\theta \chi^{(f)} < -(k^\zero)^2 \theta \quad \text{if $\theta > 0$}, \quad \partial_\theta \chi^{(f)} > -(k^\zero)^2 \theta \quad \text{if $\theta < 0$}.
\end{equation}

$(9)$ Let $\ell \ge 2$. Assume that $|\partial_\theta \chi^{(f_{i})}| > (\tau^{(f_{i})})^4$, $i=1,2$,
and $\sgn (\partial_\theta \chi^{(f_{1})})=-\sgn (\partial_\theta \chi^{(f_{2})})$. Then, $\partial^2_\theta \chi^{(f)} \le - (\min_{j} \tau^{(f_{j})})^8$. Furthermore, assume in addition that $\chi^{(f)}(x,u,\theta) = \chi^{(f)}(x,u,-\theta)$, $\theta \in \Theta = (-\theta_0,0)\cup (0,\theta_0)$. Then,
\begin{equation}\label{eq:4a-thetaest1}
\partial_\theta \chi^{(f)} < -(\min_{j} \tau^{(f_{j})})^8 \theta \quad \text{if $\theta > 0$}, \quad \partial_\theta \chi^{(f)} >- (\min_{j} \tau^{(f_{j})})^8 \theta \quad \text{if $\theta < 0$}.
\end{equation}
\end{lemma}

\begin{proof}
$(1)$ Let us first consider the case $\ell = 1$ and assume $f \in \mathfrak{F}^{(1)}_{\mathfrak{g}^\one} (a_1, a_2, b^2)$. It follows from Definition~\ref{def:4a-functions} that $|f_i| = |u-a_i| < 1/16 = 2^{-2^{2}}$, $|\tau^{(f)}|, |\mu^{(f)}|, |\chi^{(f)}| \le 2^{-2^{2}}$ for any $(x,u) \in \cL_{\IR} \bigl( g_{0}, \rho_{0} \bigr)$. Recall that by convention $\chi^{(f_i)} := f_i$, $\tau^{(f_i)} := 1$. Therefore,
$|\chi^{(f_i)}| < (\min_j |\tau^{(f_j)}|)^{10}$ obviously holds. Using the notation from Definition~\ref{def:4a-functions}, assume now that $\ell \ge 2$ and $f \in \mathfrak{F}^{(\ell)}_{\mathfrak{g}^{(\ell)}} (f_1, f_2, b^2)$, $f_i \in \mathfrak{F}^{(\ell-1)}_{\mathfrak{g}^{(\ell-1)}}$, $i = 1, 2$. Then, $|f_i|, |b| < (\min_j\tau^{(f_{j})})^{10}$, $\tau^{(f)} \le 2(\max_j |f_j|) (\max_j \tau^{(f_{j})})^2$, $|\mu^{(f)}| \le (\max_j |f_j|) (\max_j |\mu^{(f_{j})}|)^2$, $|\chi^{(f)}| \le (\max_j |\chi^{(f_{j})}|)^2 + |b|^2  (\max_j |\mu^{(f_{j})}|)^2$, $|\chi^{(f_i)}| = |f_i| |\mu^{(f_{i})}| \le (\min_j\tau^{(f_{j})})^{10} |\mu^{(f_{i})}|$. Using induction one obtains the claim.

$(2)$ For $\ell = 1$, the statement follows from \eqref{eq:4a-functions1}
 and the conditions $|\partial^\alpha a_i|, |\partial^\alpha b^2| < 1/64$ for any $x,u$ and any $1 \le \alpha \le 2$ in Definition~\ref{def:4a-functions}. Using induction over $\ell = 1, \dots$ and
\eqref{eq:4a-bfunctions1a}, one proves the claim for any $\ell$.

$(3)$ Due to Definition~\ref{def:4a-functions}, $\hat \sigma(f_1) = \hat \sigma(f_2)$. Due to the definition of the sequences $\hat \sigma(\cdot)$
this implies that either $f_i \in \mathfrak{F}^{(\ell-1,+)}_{\mathfrak{g}^{(\ell-1)}} (f_{i,1}, f_{i,2}, b_i^2)$, $i = 1, 2$, or $f_i \in \mathfrak{F}^{(\ell-1,-)}_{\mathfrak{g}^{(\ell-1)}} (f_{i,1}, f_{i,2}, b_i^2)$, $i = 1, 2$. Assume $\ell\ge 2$, $f_i \in \mathfrak{F}^{(\ell-1,+)}_{\mathfrak{g}^{(\ell-1)}} (f_{i,1}, f_{i,2}, b^2_i)$, $i = 1, 2$. Recall that due to condition $(b)$ $|f_i| < (\min_j \tau^{(f_{j})})^{10} < (\chi^{(f_{i,2})} - \chi^{(f_{i,1})})^2/4$ for all $(x,u) \in \cL_{\IR} \bigl( g_{\ell-1}, \rho_{\ell-1} \bigr)$, $i = 1, 2$. As in Definition~\ref{def:4a-functions} set $a_{i,j} = u - \chi^{(f_{i,j})}$. Since $|\mu^{(f_{i,1})}||\mu^{(f_{i,2})}||f_{i,2}| < 1$, one has
\begin{equation}\label{eq:4chipmderivation}
\begin{split}
(a_{i,1}-a_{i,2})^2/4=
(\chi^{(f_{i,2})} - \chi^{(f_{i,1})})^2/4>|\mu^{(f_{i,1})}||\mu^{(f_{i,2})}||f_{i,2}|(\chi^{(f_{i,2})} - \chi^{(f_{i,1})})^2/4>\\
|\mu^{(f_{i,1})}||\mu^{(f_{i,2})}||f_{i,2}||f_i|=|\chi^{(f_{i,1})}\chi^{(f_{i,2})}-\prod_j\mu^{(f_{i,j})} b_i^2|=|(u - a_{i,1})(u-a_{i,2})-\prod_j\mu^{(f_{i,j})}b_i^2|\\
\end{split}
\end{equation}
Due to Definition~\ref{def:4a-functions} we are in the $+$-case in Lemma~\ref{4:generalquadratic}. So, \eqref{eq:4ineqdichotomy+} applies. In particular, \eqref{eq:4ineqdichotomy+} implies $\chi^{(f_{i,1})}= u - a_{i,1} \ge -|\gamma|(a_{i,1} - a_{i,2}) = -|\gamma|(\chi^{(f_{i,2})} - \chi^{(f_{i,1})})$,
$\chi^{(f_{i,2})} = u - a_{i,2} \ge (1/2)(a_{i,1} - a_{i,2})+(\prod_j\mu^{(f_{i,j})})^{1/2}|b_i|= (1/2)(\chi^{(f_{i,2})} - \chi^{(f_{i,1})})+(\prod_j\mu^{(f_{i,j})})^{1/2}|b_i|$ for all $(x,u) \in \cL_{\IR} \bigl( g_{\ell}, \rho_{\ell} \bigr)$, $i = 1, 2$. Here, $\gamma = (\sqrt{1+4\lambda} - 1)/2$, $\lambda = (a_{i,1} - a_{i,2})^{-2}[(u - a_{i,1})(u - a_{i,2}) - \prod_j\mu^{(f_{i,j})}b_i^2]$. One has due to conditions $(b)$ and $(d)$ in Definition~\ref{def:4a-functions} $|\lambda| < (a_{i,1} - a_{i,2})^{-2} (\min_j \tau^{(f_{j})})^{10}/2 < (\min_j \tau^{(f_{j})})^{8}/2$, $|\gamma| < 2|\lambda| < (\min_j \tau^{(f_{j})})^8$. This finishes the proof of the claim in the first case in $(3)$. The verification for the second case is completely similar. The verification in case $\ell=1$ is also completely similar and we omit it.

$(4)$ Due to Definition~\ref{def:4a-functions} $\hat \sigma(f_1) = \hat\sigma(f_2)$. That implies the statement in part $(4)$.

$(5)$ The proof goes by induction over $\ell = 1 \dots$. Let $f \in \mathfrak{F}^{(1)} (a_1,a_2,b^2)$, $f = (u - a_1) - b^2 (u - a_2)^{-1}$. Then, by convention, $f_i := u - a_i$, $\chi^{(f_i)} := f_i$, $\mu^{(f_i)} := 1$, $\tau^{(f_i)} := 1$, $\sigma(f_{i})= 1$, $i = 1, 2$. Hence, $\sigma(f_{i}) \partial_u \chi^{(f_i)} = 1 - \partial_u a_i > 1/2 = (\tau^{(f_i)})^2/2$, as claimed. Assume that the statement holds for any $h \in \mathfrak{F}^{(t)}_{\mathfrak{g}^{(t)}}$, $1 \le t \le \ell - 1$, $\ell \ge 2$. Assume, for instance, $f \in \mathfrak{F}^{(\ell)}_{\mathfrak{g}^{(\ell)}} (f_1, f_2, b^2)$, $f_i \in \mathfrak{F}^{(\ell-1,-)}_{\mathfrak{g}^{(\ell-1)}} (f_{i,1}, f_{i,2}, b^2_i)$, $i = 1, 2$. Due to the inductive assumption, $\sigma(f_{i,j}) \partial_u \chi^{(f_{i,j})} > (\tau(f_{{i,j}}))^2$, $i, j = 1, 2$. Due to  to $(4)$, $\sigma(f_{i,j}) = \sigma(f_{i',j'})$ for any $i, j, i', j'$. Due to part $(3)$, one has $\chi^{(f_{i,1})} \le -(1/2) (\chi^{(f_{i,2})}-\chi^{(f_{i,1})})$, $\chi^{(f_{i,2})}< (\min_j\tau^{(f_{j})})^8 (\chi^{(f_{i,2})}-\chi^{(f_{i,1})}) < (\min_j\tau^{(f_{j})})^8$, $i = 1, 2$. Due to part $(1)$, $ |\mu^{(f_{i,j})}|,|\mu^{(f_{i,j})}| < 2^{-2^{2(\ell-2)}}$, due to part $(2)$, $|\partial^\alpha \mu^{(f_{i,j})}|, |\partial^\alpha \chi^{(f_{i,j})}| < 2^{-2^{2(\ell-3)}}$. Finally, due to Definition~\ref{def:4a-functions}, one has $|b_i| < (\min_j\tau^{(f_{i,j})})^{10}$, $|\partial_u b_i^2| < (\min_j\tau^{(f_{i,j})})^{10} |b_i|$. Using all these estimates, one obtains
\begin{equation}\label{eq:4chisignident3}
\begin{split}
\sigma(f_{i}) \partial_u \chi^{(f_{i})} = \sigma(f_{i}) \partial_u [\chi^{(f_{i,1})} \chi^{(f_{i,2})} - \mu^{(f_{i,1})} \mu^{(f_{i,2})} b_i^2]=\\
= [\sigma(f_{i,1}) (\partial_u \chi^{(f_{i,1})})] [- \chi^{(f_{i,2})}] + [\sigma(f_{i,2}) (\partial_u \chi^{(f_{i,2})})] [-\chi^{(f_{i,1})}] \\
-[(\partial_u \mu^{(f_{i,1})}) \mu^{(f_{i,2}) }b_i^2 + (\partial_u \mu^{(f_{i,2})}) \mu^{(f_{i,1})} b_i^2 + (\partial_u b_i^2) \mu^{(f_{i,1})} \mu^{(f_{i,2})}] \\
\ge -[\sigma(f_{i,1}) (\partial_u \chi^{(f_{i,1})})]  (\min_j \tau^{(f_{j})})^8
 + [\sigma(f_{i,1})\partial_u \chi^{(f_{i,2})}]
(\chi^{(f_{i,2})}-\chi^{(f_{i,1})})/2\\
- [|\partial_u \mu^{(f_{i,1})}| |\mu^{(f_{i,2})}| b_i^2 + |\partial_u \mu^{(f_{i,2})}| |\mu^{(f_{i,1})}| b_i^2 + |\partial_u b_i^2| |\mu^{(f_{i,1})}| |\mu^{(f_{i,2})}|] \\
\ge -2^{-2^{2(\ell-2)}} (\min_j \tau^{(f_{j})})^8 +  (\chi^{(f_{i,2})}-\chi^{(f_{i,1})})(\tau^{(f_{i,2})})^2/2 \\
- 2^{-2^{\ell}}(\min_j\tau^{(f_{i,j})})^6 > (\chi^{(f_{i,2})}-\chi^{(f_{i,1})})(\tau^{(f_{i,2})})^2/4 > (\chi^{(f_{i,2})}-\chi^{(f_{i,1})})^2(\tau^{(f_{i,1})})^2 (\tau^{(f_{i,2})})^2 = (\tau^{(f_{i})})^2.
\end{split}
\end{equation}
This finishes the proof in case $f_i \in \mathfrak{F}^{(\ell-1,-)}_{\mathfrak{g}^{(\ell-1)}}$. The case $f_i \in \mathfrak{F}^{(\ell-1,+)}_{\mathfrak{g}^{(\ell-1)}}$ is completely similar.

$(6)$ Assume $\chi^{(f)}(x_0,u_0) = 0$. Set $a_i := u_0 - \chi^{(f_{i})}(x_0,u_0)$, $i=1,2$, $b:=(\prod_i\chi^{(f_{i})})^{1/2}b(x_0,u_0)$. Recall that $\chi = \prod_i\mu^{(f_{i})} f_i - b^2 \prod_i\mu^{(f_{i})}$. Due to part $(4)$, $\prod_i \mu^{(f_{i})}\neq 0$. Hence one has $\bigl(u_0 - a_1\bigr)\bigl(u_0 - a_2\bigr) - |b|^2 = 0$. One can apply Lemma~\ref{4:generalquadratic}. Assume for instance $u_0 - a_1(x_0,u_0) \ge 0$. Then \eqref{eq:4ineqdichotomy+} applies. Note that here $\lambda = 0$, $\gamma = 0$. So, $\chi^{(f_{1})}(x_0,u_0) > 0$, $\chi^{(f_{2})}(x_0,u_0) > [(1/2)(\chi^{(f_{2})}(x_0,u_0)-\chi^{(f_{1})} + (\prod_i\chi^{(f_{i})})^{1/2}|b|]|_{x_0,u_0}$. From this point, the derivation of the estimate in $(7)$ goes exactly the same way as in \eqref{eq:4chisignident3}. The case $u_0 - a_1(x_0,u_0) \le 0$ is completely similar.

$(7)$ Consider the case $\ell \ge 2$. Due to part $(5)$, $|\partial_u \chi^{(f_{i})}| > (\tau^{(f_{i})})^2$, $\sgn (\partial_u \chi^{(f_{1})}) = \sgn (\partial_u \chi^{(f_{2})})$. Due to part $(1)$, $|\chi^{(f_i)}| < (\min_j |\tau^{(f_j)}|)^{10}$. Due to part $(2)$, $|\partial^\alpha \mu^{(f)}|, |\partial^\alpha \chi^{(f)}| < 2^{-2^{2(\ell-1)}+3}$, $|\alpha| \le 2$. Finally, due to Definition~\ref{def:4a-functions}, one has $|\partial^\alpha_u b^2| < (\min_j\tau^{(f_{j})})^{10}$. Using these estimates, one obtains
\begin{equation}\label{eq:4chisignident5}
\begin{split}
\partial^2_u \chi^{(f)} \ge |\partial_u \chi^{(f_{1})}||\partial_u \chi^{(f_{2})}| - \big\{|\partial^2_u \chi^{(f_{1})}| |\chi^{(f_{2})}| + |\partial^2_u \chi^{(f_{2})}| |\chi^{(f_{1})}| + |\partial^2_u [\mu^{(f_{1})} \mu^{(f_{2})}b^2]| \big\} \\
\ge \prod_i (\tau^{(f_{i})})^2 - 2 \cdot 2^{-2^{2(\ell-2)}+3} (\min_j |\tau^{(f_j)}|)^{10} - 6 \cdot 2^{-2^{2(\ell-2)}+3} \cdot 2^{-2^{2(\ell-2)}+3} \cdot (\min_j |\tau^{(f_j)}|)^{10} \\
\ge (1/2)(\min_{j} \tau^{(f_{j})})^4.
\end{split}
\end{equation}
The estimation for $\ell = 1$ is similar.

$(8)$ Let $f \in \mathfrak{F}^{(1)} (a_1,a_2,b^2)$, $f = (u - a_1) - b^2 (u - a_2)^{-1}$. By convention, $f_i := u - a_i$, $\chi^{(f)} := f_1f_2-b^2$. Due to part $(2)$, $|\partial^2_\theta f_{1}| < 1$. Due to condition $(k^\zero)$, $\partial_\theta f_1 =  - \partial_\theta a_1 < -k^\zero$, $\partial_\theta f_2 =  - \partial_\theta a_2 > k^\zero$, and moreover,  $|f_i| < (k^\zero)^2/8$, $|\partial^2_\theta b^2| < (k^\zero)^2/8$. One has
\begin{equation}\label{eq:4chisignident5thetaA}
\partial^2_\theta \chi^{(f)} \le -2(k^\zero)^2 + |\partial^2_\theta f_{1}| |f_{2}| + |\partial^2_\theta f_{2}| |f_{1}| + |\partial^2_\theta b^2| < -(k^\zero)^2,
\end{equation}
as claimed. Assume now in addition that $\chi^{(f)}(x,u,\theta) = \chi^{(f)}(x,u,-\theta)$. This implies $\partial_\theta \chi^{(f)}|_{x,u,0} = 0$, and $\partial_\theta \chi^{(f)} < -(k^\zero)^2 \theta$ if $\theta > 0$, $\partial_\theta \chi^{(f)} > -(k^\zero)^2 \theta$ if $\theta < 0$.

$(9)$ The estimation is similar to the one in $(8)$. Recall that we assume here $\ell \ge 2$. Due to part $(1)$, $|\chi^{(f_i)}| < (\min_j |\tau^{(f_j)}|)^{10}$. Due to part $(2)$, $|\partial^\alpha \mu^{(f)}|, |\partial^\alpha \chi^{(f)}| < 2^{-2^{2(\ell-1)}+3}$, $|\alpha| \le 2$. Due to Definition~\ref{def:4a-functions}, $|\partial^\alpha_\theta b^2| < (\min_j\tau^{(f_{j})})^{10}$. Using these estimates and the assumption $\sgn (\partial_\theta \chi^{(f_{1})})=-\sgn (\partial_\theta \chi^{(f_{2})})$, one obtains
\begin{equation}\label{eq:4chisignident5theta}
\begin{split}
\partial^2_\theta \chi^{(f)} \le -2|\partial_\theta \chi^{(f_{1})}||\partial_\theta \chi^{(f_{2})}| + \big\{|\partial^2_\theta \chi^{(f_{1})}| |\chi^{(f_{2})}| + |\partial^2_\theta \chi^{(f_{2})}| |\chi^{(f_{1})}| + |\partial^2_\theta [\mu^{(f_{1})} \mu^{(f_{2})}b^2]| \big\} \\
\le -2 \prod_i (\tau^{(f_{i})})^4 + 2 \cdot 2^{-2^{2(\ell-2)}+3} (\min_j |\tau^{(f_j)}|)^6 + 6 \cdot 2^{-2^{2(\ell-2)}+3} \cdot 2^{-2^{2(\ell - 2)} + 3} \cdot (\min_j |\tau^{(f_j)}|)^6 \\
\le -(\min_{i} \tau^{(f_{i})})^8.
\end{split}
\end{equation}
The second statement in $(10)$ follows from the first one just like in $(8)$.
\end{proof}

We need the following elementary calculus statements.

\begin{lemma}\label{elemcalculusconv}
Let $f(u)$ be a $C^2$-function, $u \in (t_0 - \rho_0, t_0 + \rho_0)$. Assume that $\sigma_0 = \inf f'' > 0$.

$(0)$ The function $f$ has at most two zeros.

$(1)$ Assume that $\sgn (f'(v_1)) \sgn (f'(v_2)) \ge 0$ for some $v_1 < v_2$. Then,
\begin{align*}
(v_2 - v_1)^2 \le 2 \sigma_0^{-1} |f(v_1)-f(v_2)|.
\end{align*}

$(2)$ Let $|v_0 - t_0| < \frac{\rho_0}{2}$. Assume $-\frac{\sigma_0\rho_0}{2} < f'(v_0) < 0$. Then there exists $v_0 < u_0 \le v_0 + \sigma_0^{-1} |f'(v_0)|$ such that $f'(u_0) = 0$. Similarly, if $\frac{\sigma_0\rho_0}{2} > f'(v_0) > 0$, then there exists $v_0 > u_0 \ge v_0 - \sigma_0^{-1} |f'(v_0)|$ such that $f'(u_0) = 0$.

$(3)$ Let $|v_0 - t_0| < \frac{\rho_0}{2}$, $0 < \rho \le \rho_0$. Assume $-\frac{\sigma_1^2\rho^2}{256} < f(v_0) \le 0$, $f'(v_0) < 0$, $\sigma_1 := \min(\sigma_0, 1)$. Then there exists $t_0 - \rho_0 < v_0 - \frac{\rho}{8} < v \le v_0$ such that $f(v) = 0$. Similarly,
assume $-\frac{\sigma_1^2\rho^2}{256} < f(v_0) \le 0$, $f'(v_0) > 0$. Then there exists $v_0 \le v < v_0 + \frac{\rho}{8} < t_0 + \rho_0$
such that $f(v) = 0$.

Assume in addition that $\sup |f'| \le 1$.

$(4)$ Let $|v_0 - t_0| < \frac{\rho_0}{2}$, $0 < \rho \le \rho_0$. Assume $-\frac{\sigma_1^2\rho^2}{256} < f(v_0) \le 0$, $-\frac{\sigma_1^2\rho^2}{256} < f'(v_0) < 0$. Then there exist $t_0 - \rho_0 < v_0 - \frac{\rho}{8} < v_1 \le v_0 < v_2 < v_0 +\frac{\rho}{4} < t_0 + \rho_0$ such that $f(v_j) = 0$ $j = 1, 2$. Similarly, assume that $-\frac{\sigma_1^2\rho^2}{256} < f(v_0) \le 0$, $\frac{\sigma_1^2\rho^2}{256} > f'(v_0) > 0$. Then there exist $t_0 - \rho_0 < v_0 - \frac{\rho}{4} < v_1 < v_0 \le v_2 < v_0 + \frac{\rho}{8} < t_0 + \rho_0$ such that $f(v_j) = 0$, $j = 1, 2$.

$(5)$ If $f$ has two zeros $v_1 < v_2$, $|v_i - v_0| < \frac{\rho_0}{2}$, then $-f'({v_1}),f'(v_2) > \frac{\sigma_1^2(v_2 - v_1)^2}{256}$.
\end{lemma}

\begin{proof}
$(0)$ Follows from Rolle's Theorem.

$(1)$ Assume $f'(v_i) \ge 0$, $i = 1, 2$. Then,
\begin{align*}
f(v_2) - f(v_1) & = \int_{v_1}^{v_2} f'(v) \, dv = \int_{v_1}^{v_2} \Big[ \int_{v_1}^{v} f''(x) \, dx + f'(v_1) \Big]\, dv \\
& \ge \sigma_0 \int_{v_1}^{v_2} \int_{v_1}^{v} \, dx \, dv = \sigma_0 (v_2 - v_1)^2/2.
\end{align*}
Assume $f'(v_i) \le 0$, $i = 1, 2$. Then,
\begin{align*}
f(v_2) - f(v_1) & = \int_{v_1}^{v_2} f'(v) \, dv = \int_{v_1}^{v_2} \Big[ -\int_{v}^{v_2} f''(x) \, dx + f'(v_2) \Big]\, dv \\
& \le -\sigma_0 \int_{v_1}^{v_2} \int_{v_1}^{v} \, dx \, dv = -\sigma_0 (v_2 - v_1)^2/2.
\end{align*}

$(2)$ Let us verify the first statement. One has $f'(u) \ge f'(v_0) + \sigma_0 (u - v_0) > 0$ if $v_0 + \sigma_0^{-1} |f'(v_0)| < u < t_0 + \rho_0$. Hence there exists $v_0 < u_0 \le v_0 + \sigma_0^{-1} |f'(v_0)|$ such that $f'(u_0) = 0$. The verification of the second statement is completely similar.

$(3)$ Let us verify the first statement. Since $f'(v_0) < 0$, integrating like in part $(1)$, one obtains $f(u) \ge f(v_0) + \sigma_0 (u - v_0)^2/2$ for $t_0 - \rho_0 < u < v_0$. So, $f(v_0 - \rho/8) > 0$. Hence, there exists $v_0 - \rho/8 < v \le v_0$ such that $f(v) = 0$. The verification of the second statement is completely similar.

$(4)$ Let us verify the first statement. Using the notation from part $(2)$, one has $-\sigma_1 \rho^2/128 < f(v_0) - \sigma_0^{-1} |f'(v_0)| < f(v_0) - (u_0 - v_0) \le f(u_0) < 0$, since $\sup |f'| \le 1$. Note also that $u_0 \le v_0 + \sigma_0^{-1} |f'(v_0)| \le v_0 + \sigma_1 \rho^2/256 < t_0 + 33 \rho_0/64$. Like in part $(1)$, one obtains $f(u) > f(u_0) + \sigma_0 (u - u_0)^2/2 > 0$ if $u_0 + \rho/8 < u < t_0 + \rho_0$. Hence, there exists
$u_0 < v_2 < u_0 + \rho_0/8 < t_0 + \rho_0$ such that $f(v_2) = 0$. The existence of $v_1$ is due to part $(3)$. The verification of the second statement is completely similar.

$(5)$ Since $f'' > 0$, one has $f'(v_1) < 0$, $f'(v_2) > 0$. Set $\rho := v_2 - v_1$. Then, $\rho < \rho_0$. It follows from part $(4)$ that $-f'({v_1}),f'(v_2) > \frac{\sigma_1^2\rho^2}{256}$, since otherwise $f$ would have at least three zeros.
\end{proof}

\begin{lemma}\label{lem:6-1ell}
Let $f \in \mathfrak{F}^{(\ell)}_{\mathfrak{g}^{(\ell)}, \mathfrak{r}^{(\ell)}}$.

$(1)$ For any $x \in (-\alpha_0, \alpha_0)$, the equation  $\chi^{(f)} = 0$ has at most two solutions $\zeta_-(x) \le \zeta_+(x)$.

$(2)$ Let $\ell \ge 2$. Assume that the following conditions hold: $(a)$ $\zeta_+(0)$ and $\zeta_-(0)$ exist, $\chi^{(f_1)} (0, \zeta_+(0)) = 0$, $\chi^{(f_2)} (0, \zeta_-(0) = 0$, $(b)$ $|\chi^{(f_1)} (x, g_{\ell-1}(x))|, |\chi^{(f_2)} (x, g_{\ell-1}(x))| < (\tau_0)^6 \rho_{\ell-1}$ for all $x$, $(c)$ $|b| < (\tau_0)^6 \rho_{\ell - 1}$ for all $x, u$, where $\tau_0 := \inf_{x,u} (\min_{i} \tau^{(f_{i})})$. Then, $\zeta_+(x)$ and $\zeta_-(x)$ exist for all $x \in (-\alpha_0, \alpha_0)$. The functions $\zeta_+(x)$, $\zeta_-(x)$ are $C^2$-smooth on $(-\alpha_0, \alpha_0)$ and obey the estimates \eqref{eq:6-1''}, \eqref{eq:6-1'''}, where $a_i = u - f_i$, and also the following estimates:
\begin{equation}\label{eq:4zetaindomain}
|\zeta_\pm(x) - g_{\ell - 1} (x)| < \rho_{\ell - 1}/2,
\end{equation}

\bigskip

\begin{equation}\label{eq:4zetassplit}
\begin{split}
\partial_u \chi^{(f)}|_{x, \zeta_-(x)} < -(\tau^{(f)}|_{x, \zeta_-(x)})^2 < 0, \quad \partial_u \chi^{(f)}|_{x, \zeta_+(x)} > (\tau^{(f)}|_{x, \zeta_+(x)})^2 > 0, \\
\zeta_+(x) - \zeta_{-}(x) > \frac{1}{8} [-\partial_u \chi^{(f)}|_{x, \zeta_-(x)} + \partial_u \chi^{(f)}|_{x, \zeta_+(x)}], \\
-\partial_u \chi^{(f)}|_{x, \zeta_-(x)}, \partial_u \chi^{(f)}|_{x, \zeta_+(x)} \ge \frac{\sigma_1^2 (\zeta_+(x) - \zeta_-(x))^2}{256}, \\
|\chi^{(f)}(x,u)| \ge \min \bigl( \frac{\sigma_1}{2} (u - \zeta_-(x))^2, \frac{\sigma_1}{2} (u - \zeta_+(x))^2 \bigr),
\end{split}
\end{equation}
where $\sigma_1 := (1/8) (\inf_{x,u} (\min_{i} \tau^{(f_{i})}))^4$.
\end{lemma}

\begin{proof}
Due to part $(7)$ of Lemma~\ref{4.fcontinuedfrac}, $\partial^2_u \chi^{(f)} > 0$ everywhere. Therefore, for any $x \in (-\alpha_0, \alpha_0)$, the equation $\chi^{(f)} = 0$ has at most two solutions $\zeta_+(x)$ and $\zeta_-(x)$. Assume $\zeta_+(0)$ and $\zeta_-(0)$ exist. Due to part $(6)$ in Lemma~\ref{4.fcontinuedfrac}, $|\partial_u \chi^{(f)}| |_{x_0,u_0} > 0$, provided $\chi^{(f)}(x_0,u_0) = 0$. Therefore, $\zeta_+(x)$ and $\zeta_-(x)$ can be defined via continuation and the standard implicit function theorem as long as the point $(x, \zeta_\pm(x))$ does not leave the domain $\cL_{\IR} \bigl( g_{\ell-1}, \rho_{\ell-1} \bigr)$. Let us verify that as long as $\zeta_+(x)$ and $\zeta_-(x)$ are defined, \eqref{eq:4zetaindomain} holds. Recall that $\chi^{(f)} = \chi^{(f_{1})} \chi^{(f_{2})} - \mu^{(f_{1})} \mu^{(f_{2})} b^2$. Since $|\mu^{(f_{i})}| < 1$, one has $|\chi^{(f_{1})} (x, \zeta_\pm(x))| |\chi^{(f_{2})}(x, \zeta_\pm(x))| < (\tau_0)^{12} \rho_{\ell - 1}^2$, due to $(c)$ in the current lemma. Hence, $\min(|\chi^{(f_{1})} (x, \zeta_\pm(x))|, |\chi^{(f_{2})} (x, \zeta_\pm(x))|) < (\tau_0)^6 \rho_{\ell - 1}$. Recall that due to part $(5)$ of Lemma~\ref{4.fcontinuedfrac}, one has $|\partial_u \chi^{(f_{i})}| > (\tau^{(f_{{i}})})^2$. Combining this with condition $(b)$ in the current lemma, one concludes that \eqref{eq:4zetaindomain} holds. Note that \eqref{eq:4zetaindomain} says in particular that the point $(x, \zeta_\pm(x))$ does not leave the domain $\cL_{\IR} \bigl( g_{\ell - 1}, \rho_{\ell - 1} \bigr)$ ever. Hence $\zeta_+(x)$ and $\zeta_-(x)$ exist for all $x \in (-\alpha_0, \alpha_0)$. Due to the standard implicit function theorem, the functions $\zeta_+(x)$, $\zeta_-(x)$ are $C^2$-smooth. Recall that due to part $(7)$ in Lemma~\ref{4.fcontinuedfrac}, one has $\partial^2_u \chi^{(f)} > 0$. Therefore, $\partial_u \chi^{(f)}|_{x,\zeta_-(x)} \le 0$, $\partial_u \chi^{(f)}|_{x,\zeta_+(x)} \ge 0$. On the other hand, due to part $(6)$ in Lemma~\ref{4.fcontinuedfrac}, one has $|\partial_u \chi^{(f)}||_{x,\zeta_\pm(x)} > (\tau^{(f)})^2|_{x,\zeta_\pm(x)}$. Thus, $\partial_u \chi^{(f)}|_{x,\zeta_-(x)} < -(\tau^{(f)})^2|_{x,\zeta_-(x)}$, $\partial_u \chi^{(f)}|_{x,\zeta_+(x)} > (\tau^{(f)})^2|_{x,\zeta_+(x)}$, that is, the first line in \eqref{eq:4zetassplit} holds. Due to part $(2)$ in Lemma~\ref{4.fcontinuedfrac}, $|\partial^2_u \chi^{(f)}| < 8$ for any $\ell$. Therefore the second line in \eqref{eq:4zetassplit} holds.  Recall that due to part $(7)$ in Lemma~\ref{4.fcontinuedfrac}, $\partial^2_u \chi^{(f)} > \sigma_1$ everywhere. Therefore the third line in \eqref{eq:4zetassplit} holds due to part $(4)$ of Lemma~\ref{elemcalculusconv}. Finally, the last line in \eqref{eq:4zetassplit} is due to part $(1)$ of Lemma~\ref{elemcalculusconv}.

Recall that $\zeta_\pm(x)$ obeys the equation $(u-a_1)(u-a_2)-b^2=0$ with $a_i = u - \chi^{(f_i)}$, $i = 1, 2$. Let
\begin{align*}
 \varphi_+(x,u):=(1/2) \left[a_1 + a_2 + \bigl((a_1 - a_2)^2 + 4b^2\bigr)^{1/2}\right] , \\[6pt]
\varphi_-(x,u):=(1/2) \left[a_1 + a_2 - \bigl((a_1 - a_2)^2 + 4b^2\bigr)^{1/2}\right]
\end{align*}
be as in the proof of Lemma~\ref{lem:6-1}. Due to condition $(a)$ in the current lemma, one concludes that $\zeta_+(0) = a_1 (0, \zeta_+(0)) = \varphi_+(0, \zeta_+(0))$, $\zeta_-(0) = a_2 (0, \zeta_-(0)) = \varphi_-(0, \zeta_-(0))$. Note that due to part $(2)$ in Lemma~\ref{4.fcontinuedfrac}
the functions $\varphi_\pm(x,\zeta_+(x))$ are continuous. Since $\varphi_+(x,u) > \varphi_-(x,u)$, by continuity, $\zeta_\pm(x) = \varphi_\pm(x,\zeta_+(x))$ for all $x$. Now just as in the proof of Lemma~\ref{lem:6-1}, one verifies that $\zeta_+(x)$, $\zeta_-(x)$ obey the estimates \eqref{eq:6-1''}, \eqref{eq:6-1'''}.
\end{proof}

For our applications, we will also need a certain generalization of the last lemma in the case when condition $(c)$ fails, that is, $|b| \nless (\tau_0)^6 \rho_{\ell-1}$. This happens when $\rho_{\ell-1}$ is too small. The specific situation is as follows. Let $g_{t,\pm}(x)$ be $C^2$-functions on $(-\alpha_0, \alpha_0)$, $0 < \rho_{t+1} < \rho_{t} < 1$, $t = 0, \dots, \ell - 1$. Assume that $g_{t,-}(x) < g_{t,+}(x)$ for every $x$. Assume that $\cL_{\IR} \bigl(g_{\ell',\pm}, \rho_{\ell'} \bigr) \supset \cL_{\IR} \bigl( g_{\ell'+1,\pm}, \rho_{\ell'+1} \bigr)$, $\ell' = 0, 1, \dots$. Set $\mathfrak{g}^{(t)}_\pm = (g_{0,\pm}, \dots, g_{t-1,\pm})$. Using these notations assume that $f \in \mathfrak{F}^{(\ell)}_{\mathfrak{g}^{(\ell)}_-} (f_1, f_2, b)$ and also $f \in \mathfrak{F}^{(\ell)}_{\mathfrak{g}^{(\ell)}_+} (f_1, f_2, b)$. This means in particular that if $(x,u) \in \cL_{\IR} \bigl( g_{\ell-1,-}, \rho_{\ell-1} \bigr) \cap \cL_{\IR} \bigl(g_{\ell-1,+}, \rho_{\ell-1} \bigr)$, then $f(x,u)$, $f_1$, $f_2$, $b$, and also the rest of the functions involved in the definition are the same no matter which way one defines them. We use the notation $\chi^{(f)}(x,u)$ for the corresponding function. Note that it is well-defined and smooth in
$\cL_{\IR} \bigl(g_{\ell-1,-}, \rho_{\ell-1} \bigr) \cup \cL_{\IR} \bigl(g_{\ell-1,+}, \rho_{\ell-1} \bigr)$.

Assume that the following conditions hold: $(\alpha)$ $|\chi^{(f)}(x, g_{\ell-1,\pm}(x))| < \frac{\sigma_1^{13} \rho^8}{2^{83}}$, with $\sigma_1 := (1/8) (\inf_{x,u} (\min_{i} \tau^{(f_{i})}))^4$, $0 < \rho \le \rho_{\ell-1}$, $(\beta)$ $\prod_i \chi^{(f_i)}|_{0,g_{\ell-1,\pm} (0)} = 0$, $(\gamma)$ $g_{\ell-1,+}(x) - g_{\ell-1,-}(x) + \frac{\sigma_1^6\rho^4}{2^{39}} \ge \min \bigl( \frac{1}{8} [|\partial_u\chi^{(f)}|_{x,g_{\ell-1,+}(x)}| + |\partial_u\chi^{(f)}|_{x,g_{\ell-1,-}(x)}|], \rho_{\ell-1} \bigr)$, $(\delta)$ $\frac{\sigma_1^2\rho^2}{128} + \min (-\partial_u \chi^{(f)}|_{x,g_{\ell-1,-}(x)}, \partial_u \chi^{(f)}|_{x,g_{\ell-1,+}(x)}) \ge \min \bigl( \frac{\sigma_1^2 (g_{\ell-1,+}(x) - g_{\ell-1,-}(x))^2}{256}, \frac{\sigma_1^2 \rho^2}{64} \bigr)$.

\begin{lemma}\label{lem:6-1ellM}
For any $x \in (-\alpha_0, \alpha_0)$, the equation  $\chi^{(f)}(x,u) = 0$ has exactly two solutions $\zeta_-(x) < \zeta_-(x)$. The functions $\zeta_+(x)$, $\zeta_-(x)$ are $C^2$-smooth on $(-\alpha_0, \alpha_0)$, obey the estimates \eqref{eq:6-1''}, \eqref{eq:6-1'''}, where $a_i = u-f_i$, and also the following estimates,
\begin{equation}\label{eq:4zetaindomain11s}
|\zeta_\pm(x) - g_{\ell-1,\pm}(x)| < \frac{\sigma_1^2\rho^2}{2^{12}},
\end{equation}
\begin{equation}\label{eq:4zetassplit11}
\partial_u \chi^{(f)}|_{x,\zeta_-(x)} \le -(\tau^{(f)})^2 (x,\zeta_-(x)) < 0, \quad \partial_u \chi^{(f)}|_{x,\zeta_+(x)} \ge (\tau^{(f)})^2 (x,\zeta_+(x)) > 0,
\end{equation}
\begin{equation}\label{eq:4zetassplit1t}
\zeta_+(x) - \zeta_{-}(x) \ge \min \bigl( \frac{1}{8} [-\partial_u\chi^{(f)}|_{x,\zeta_-(x)} + \partial_u \chi^{(f)}|_{x,\zeta_+(x)}],
\rho_{\ell-1} \bigr),
\end{equation}
\begin{equation}\label{eq:4zetassplit1tututu}
- \partial_u \chi^{(f)}|_{x,\zeta_-(x)}, \partial_u \chi^{(f)}|_{x,\zeta_+(x)} \ge \min \bigl( \frac{\sigma_1^2 (\zeta_+(x) - \zeta_-(x))^2}{256}, \frac{\sigma_1^2 \rho^2}{128} \bigr),
\end{equation}
\begin{equation}\label{eq:4zetassplit1tatata}
|\chi^{(f)}(x,u)| \ge \min (\frac{\sigma_1}{2}(u - \zeta_-(x))^2, \frac{\sigma_1}{2}(u - \zeta_+(x))^2),
 \quad \text{if $\min(|u - \zeta_-(x)|, |u - \zeta_+(x)|) < \frac{\sigma_1^2\rho^2}{2^{11}}$}.
\end{equation}
\end{lemma}

\begin{proof}
Note that $\chi^{(f)}(0, g_{\ell-1,\pm}(0)) = \prod_i \chi^{(f_i)}|_{0,g_{\ell-1,\pm}(0)} = 0$. So, $\zeta_\pm(0)$ exist. Like in the proof of Lemma~\ref{lem:6-1ell}, $\zeta_\pm(x)$ can be defined via continuation, starting at $x = 0$, and the standard implicit function theorem, as long as the point $(x, \zeta_\pm(x))$ does not leave the domain $\cL_{\IR} \bigl( g_{\ell-1,-}, \rho_{\ell-1} \bigr) \cup \cL_{\IR} \bigl( g_{\ell-1,+}, \rho_{\ell-1} \bigr)$. Due to condition $(\beta)$, \eqref{eq:4zetaindomain11s} holds for $|x|$ sufficiently small.

Assume that $\zeta_+(x)$ and $\zeta_-(x)$ are defined and obey \eqref{eq:4zetaindomain11s} for all $x \in [0,x_0)$. The standard implicit function theorem arguments apply to show that  $\zeta_\pm(x)$ are well defined for $x \in [0,x_1)$ with  $x_1 - x_0 > 0$ being small. We claim that in fact \eqref{eq:4zetaindomain11s}, \eqref{eq:4zetassplit11} hold for any $x \in [0,x_1)$. Let $x \in [0,x_1)$ be arbitrary. Note first of all that since  $\zeta_-(0) < \zeta_+(0)$, the implicit function theorem arguments imply that $\zeta_-(x) < \zeta_+(x)$ for any  $x \in [0,x_1)$. Assume first $g_{\ell-1,+}(x) - g_{\ell-1,-}(x) < 2 \rho_{\ell-1}$. Then, $\chi^{(f)}(x,\cdot)$ is a $C^2$-smooth function defined in $(g_{\ell-1,-}(x) - \rho_{\ell-1}, g_{\ell-1,+}(x) + \rho_{\ell-1})$. Due to part $(7)$ of Lemma~\ref{4.fcontinuedfrac}, $\partial^2_u \chi^{(f)} > \sigma_1$ everywhere. Since $\chi^{(f)}(x, \zeta_\pm(x)) = 0$, $\zeta_-(x) < \zeta_+(x)$, one concludes that  $\partial_u \chi^{(f)}|_{x,\zeta_-(x)} < 0$, $\partial_u \chi^{(f)}|_{x,\zeta_+(x)} > 0$. Combined with part $(6)$ of Lemma~\ref{4.fcontinuedfrac}, this implies \eqref{eq:4zetassplit11}. Furthermore, $\chi^{(f)}(x,\cdot)$ has exactly two zeros. Due to part $(1)$ of Lemma~\ref{elemcalculusconv}, one concludes that $\min_{+,-} |\zeta_\pm(x) - g_{\ell-1,-}(x)| < \bigl( 2 \sigma_1^{-1} |\chi^{(f)} (x, g_{\ell-1,-}(x))| \bigr)^{1/2} < \frac{\sigma_1^6\rho^4}{2^{41}}$. Similarly, $\min_{+,-} |\zeta_\pm(x) - g_{\ell-1,-}(x)| < \frac{\sigma_1^6\rho^4}{2^{41}}$. Assume first $\max_{+,-} |\zeta_-(x) - g_{\ell-1,\pm}(x)| < \frac{\sigma_1^6 \rho^4}{2^{40}}$. Then, $g_{\ell-1,+}(x) - g_{\ell-1,-}(x) < \frac{\sigma_1^6 \rho^4}{2^{39}}$. Due to condition $(\gamma)$, one obtains $\frac{\sigma_1^6\rho^4}{2^{38}} > g_{\ell-1,+}(x) - g_{\ell-1,-}(x) + \frac{\sigma_1^6 \rho^4}{2^{39}} \ge 2^{-3} [|\partial_u \chi^{(f)}| |_{x,g_{\ell-1,+}(x)} + |\partial_u\chi^{(f)}| |_{x,g_{\ell-1,-}(x)}]$. In particular, $\frac{\sigma_1^6 \rho^4}{2^{35}} > |\partial_u \chi^{(f)}| |_{x,g_{\ell-1,+}(x)}$. Since $|\partial^2_u \chi^{(f)}| < 8$, one concludes $|\partial_u \chi^{(f)}| |_{x,\zeta_{-}(x)} < \frac{\sigma_1^6 \rho^4}{2^{34}}$. Due to part $(4)$ of Lemma~\ref{elemcalculusconv}, one concludes that $\zeta_+(x) - \zeta_{-}(x) < \frac{\sigma_1^2 \rho^2}{2^{13}}$. Since $\max_{+,-} |\zeta_-(x) - g_{\ell-1,\pm}(x)| < \frac{\sigma_1^6 \rho^4}{2^{40}}$, \eqref{eq:4zetaindomain11s} follows. Similarly, \eqref{eq:4zetaindomain11s} follows if $\max_{+,-} |\zeta_+(x) - g_{\ell-1,\pm}(x)| < \frac{\sigma_1^6 \rho^4}{2^{40}}$. Assume now $\max_{+,-} |\zeta_-(x) - g_{\ell-1,\pm}(x)| \ge \frac{\sigma_1^6 \rho^4}{2^{40}}$ and $\max_{+,-} |\zeta_+(x) - g_{\ell-1,\pm}(x)| \ge \frac{\sigma_1^6 \rho^4}{2^{40}}$. Since  $\zeta_-(x) < \zeta_+(x)$, $g_{\ell-1,-}(x) < g_{\ell-1,+}(x)$, $\min_{+,-} |\zeta_\pm(x) - g_{\ell-1,+}(x)| < \frac{\sigma_1^6 \rho^4}{2^{41}}$, $\min_{+,-}| \zeta_\pm(x) - g_{\ell-1,-}(x)| < \frac{\sigma_1^6 \rho^4}{2^{41}}$, one concludes that $|\zeta_\pm(x) - g_{\ell-1,\pm}(x)| < \frac{\sigma_1^6 \rho^4}{2^{41}}$. In particular, \eqref{eq:4zetaindomain11s} holds. This finishes the proof of the claim in case $g_{\ell-1,+}(x)- g_{\ell-1,-}(x) < 2\rho_{\ell-1}$.

Assume now  $g_{\ell-1,+}(x) - g_{\ell-1,-}(x) \ge 2 \rho_{\ell-1}$. In this case, due to condition $(\delta)$, $\min( - \partial_u \chi^{(f)}|_{x,g_{\ell-1,-}(x)}, \partial_u \chi^{(f)}|_{x,g_{\ell-1,+}(x)}) \ge \frac{\sigma_1^2 \rho^2}{128}$. Recall that $|\zeta_\pm(x) - g_{\ell-1,\pm}(x)| < \frac{\sigma_1^2 \rho^2}{2^{12}}$ and $|\partial^2_u \chi^{(f)}| < 8$. This implies in particular $-\partial_u \chi^{(f)}|_{x,\zeta_{-}(x)}, \partial_u \chi^{(f)}|_{x,\zeta_{+}(x)} > \frac{\sigma_1^2 \rho^2}{256}$. Combined with part $(6)$ of Lemma~\ref{4.fcontinuedfrac}, this implies \eqref{eq:4zetassplit11}. Since $g_{\ell-1,+}(x) - g_{\ell-1,-}(x) \ge 2 \rho_{\ell-1}$, it follows from part $(1)$ of Lemma~\ref{elemcalculusconv} that $|\zeta_\pm(x) - g_{\ell-1,\pm}(x)| < \frac{\sigma_1^6 \rho^4}{2^{41}}$. Thus, \eqref{eq:4zetaindomain11s} holds. This finishes the verification of the claim.

It follows from the claim that $\zeta_+(x)$ and $\zeta_-(x)$ can be defined for all $x$. These functions are $C^2$-smooth and obey \eqref{eq:4zetaindomain11s}, \eqref{eq:4zetassplit11}. Let us verify \eqref{eq:4zetassplit1t}. Assume first $g_{\ell-1,+}(x) - g_{\ell-1,-}(x) < 2 \rho_{\ell-1}$. Then, $\chi^{(f)}(x,\cdot)$ is a $C^2$-smooth function defined in $(g_{\ell-1,-}(x) - \rho_{\ell-1}, g_{\ell-1,+}(x) + \rho_{\ell-1})$. Therefore, \eqref{eq:4zetassplit1t} follows from \eqref{eq:4zetassplit11} since $|\partial^2_u \chi^{(f)}| < 8$. The estimate \eqref{eq:4zetassplit1tututu} follows from part $(5)$ of Lemma~\ref{elemcalculusconv}. The estimate \eqref{eq:4zetassplit1tatata} follows from part $(1)$ of Lemma~\ref{elemcalculusconv}, and in fact, in this case it holds for any $u$. Assume $g_{\ell-1,+}(x) - g_{\ell-1,-}(x) \ge 2 \rho_{\ell-1}$. In this case, \eqref{eq:4zetassplit1t} follows from \eqref{eq:4zetaindomain11s}.  Above we verified that $-\partial_u \chi^{(f)}|_{x,\zeta_{-}(x)}, \partial_u \chi^{(f)}|_{x,\zeta_{+}(x)} > \frac{\sigma_1^2 \rho^2}{256}$. Note also that $\frac{\sigma_1^2 (\zeta_+(x) - \zeta_-(x))^2}{256} > \frac{\sigma_1^2}{128}$. This verifies \eqref{eq:4zetassplit1tututu} for this case. Assume $|u - \zeta_-(x)| < \frac{\sigma_1^2 \rho^2}{2^{11}}$. Then, $\partial_u \chi^{(f)}|_{x,u} < -\frac{\sigma_1^2 \rho^2}{256} < 0$. So, part $(1)$ in Lemma~\ref{elemcalculusconv} applies and \eqref{eq:4zetassplit1tatata}  follows. The case $|u - \zeta_+(x)|) < \frac{\sigma_1^2 \rho^2}{256}$ is similar.
\end{proof}

\begin{lemma}\label{lem:4zetatheta}
Let $\zeta_\pm$ be as in  Lemma~\ref{lem:6-1ell} or as in Lemma~\ref{lem:6-1ellM}. If $\ell = 1$, assume that \eqref{eq:4a-thetaest} from $(9)$ of Lemma~\ref{4.fcontinuedfrac} holds. If $\ell \ge 2$, assume that \eqref{eq:4a-thetaest1} from $(10)$ of Lemma~\ref{4.fcontinuedfrac} holds. Then,
\begin{equation}\label{eq:4a-thetaest1AST}
\begin{split}
\partial_\theta \zeta_+ > (k^\zero)^2 \theta, \quad \partial_\theta \zeta_- < - (k^\zero)^2 \theta \quad \text{if $\ell = 1$, $\theta > 0$}, \\
\partial_\theta \zeta_+ > (\min_{j} \tau^{(f_{j})})^8 \theta, \quad \partial_\theta \zeta_- < - (\min_{j} \tau^{(f_{j})})^8 \theta \quad \text{if $\ell \ge 2$, $\theta > 0$}.
\end{split}
\end{equation}
\end{lemma}

\begin{proof}
Take an arbitrary $x_0$ and let $\theta > 0$. Set $u_0 = \zeta_+(x_0,\theta)$. Due to part $(7)$ of Lemma~\ref{4.fcontinuedfrac} one has
$\partial_u \chi^{(f)}|_{x_0,u_0} > (\tau^{(f)})^2|_{x_0,u_0}>0$. On the other hand, due to part $(2)$ of Lemma~\ref{4.fcontinuedfrac}, one has $|\partial^\alpha \chi^{(f)}| < 1$ for all $x,u,\theta$. Consider for instance the case $\ell \ge 2$ and $\theta > 0$. Then the assumption is that for $\theta > 0$,
\begin{equation}\label{eq:4a-thetaest1A}
\partial_\theta \chi^{(f)} < -(\min_{j} \tau^{(f_{j})})^8 \theta.
\end{equation}
Hence,
\begin{equation}\label{eq:4a-thetaest1AB}
\partial_\theta \zeta_+ = -\frac{\partial_\theta \chi}{\partial_u \chi}  > (\min_{j} \tau^{(f_{j})})^8 \theta,
\end{equation}
as claimed. The proof for the rest of the cases is similar.
\end{proof}

\begin{lemma}\label{lem:4stable}
Using the notation from Definition~\ref{def:4a-functions}, assume $f \in \mathfrak{F}^{(\ell)}_{\mathfrak{g}^{(\ell)},\lambda} (f_1, f_2, b^2)$. Let $r_i$, $h^2$ be $C^2$-functions of $(x,u) \in \cL_{\IR} \bigl( g_{\ell-1}, \rho_{\ell-1} \bigr)$. Let $\tilde f_i = f_i + r_i$, $\tilde b^2 = b^2 + h^2$. Assume that the following conditions hold for $(x,u) \in \cL_{\IR} \bigl( g_{\ell-1}, \rho_{\ell-1} \bigr)$: $(i)$ $\tilde f_1 < \tilde f_2$, $(ii)$ $|\partial^\alpha_u r_i|, |\partial^\alpha_u h^2| \le \min_j (\delta \lambda \tau^{(f_j)})^6$, $0 \le \alpha \le 2$, $i = 1, 2$, with some $\delta < (1 - \lambda)/4 \lambda$, $(iii)$ $h^2(0,u) = 0$. Set $\tilde f = \tilde f_1 - \tilde b^2/\tilde f_2$.
Then, $\tilde f \in \mathfrak{F}^{(\ell)}_{\mathfrak{g}^{(\ell)}, (1+4\delta)\lambda} (\tilde f_1, \tilde f_2, \tilde b^2)$, $\tau^{(\tilde f_j)} > (1 - \delta)\tau^{(f_j)}$.
\end{lemma}

\begin{proof}
The proof goes by induction in $\ell = 2, \dots$. Assume for instance $f \in \mathfrak{F}^{(2,+)}_{\mathfrak{g}^{(2)},\lambda} (f_1, f_2, b^2)$ and $u - a_{i,2} > 0$, $i = 1, 2$. One has in this case $\tilde f_i = f_i + r_i = (u - a_{i,1}) + r_i - b^2_i (u - a_{i,2})^{-1} = (u - \tilde a_{i,1}) - b^2_i (u - \tilde a_{i,2})^{-1}$, $\tilde a_{i,1} := a_{i,1} - r_i$, $\tilde a_{i,2} = a_{i,2}$, $\tau^{(\tilde f_j)} \ge \tau^{(f_j)} - |r_j| > (1 - \delta) \tau^{(f_j)}$, $|\tilde f_i| \le |f_i| + |r_i| < \min_j (\lambda \tau^{(f_j)})^6 + \min_j (\delta \lambda \tau^{(f_j)})^6 < \min_j ((1 + \delta) \lambda \tau^{(f_j)})^6 \le \min_j ((1 + 4\delta) \lambda \tau^{(\tilde f_j)})^6$. The verification of the rest of the conditions $(b)$--$(d)$ in part $(2)$ of Definition~\ref{def:4a-functions} is similar. Condition $(a)$ is due to condition $(i)$ in the current lemma. Let $f \in \mathfrak{F}^{(\ell)}_{\mathfrak{g}^{(\ell)},\lambda} (f_1, f_2, b^2)$, $\ell \ge 3$, $f_i \in \mathfrak{F}^{(\ell-1)}_{\mathfrak{g}^{(\ell-1)},\lambda} (f_{i,1}, f_{i,2}, b^2_i)$. Assume for instance  $f_{i,2} > 0$, $f_i = f_{i,1} - b^2_i f_{i,2}^{-1}$, $i = 1, 2$. Then, $\tilde f_i = \tilde f_{i,1} - b^2_i \tilde f_{i,2}^{-1}$, $\tilde f_{i,1} = f_{i,1} + r_i$, $\tilde f_{i,2} = f_{i,2}$. Since $\tau^{(f_j)} < \min_{i,k} \tau^{(f_{i,k})}$, one can verify all conditions in part $(3)$ of Definition~\ref{def:4a-functions} just like above. Induction is needed just to make sure that $\tilde f_i \in \mathfrak{F}^{(\ell-1)}_{\mathfrak{g}^{(\ell-1)},(1+4\delta)\lambda} (\tilde f_{i,1}, \tilde f_{i,2}, b^2_i)$.
\end{proof}

\section{Matrices with Ordered Pairs of Resonances}\label{sec.5}

Let us now return to the setting of Section~\ref{sec.3}. Let $\La$ be a subset of $\zv$. Let $v(n)$, $n \in \La$, $h_0(m, n)$, $m, n \in \La$, $m \ne n$ be some complex functions. Consider $H_{\La,\ve} = \bigl(h(m, n; \ve)\bigr)_{m, n \in \La}$, where $\ve \in \IC$,
\begin{alignat}{2}
h(n, n; \ve) & =v(n)\ , &\quad  & n \in \La, \label{eq:5-1N} \\[6pt]
h(m, n; \ve) & = \ve h_0(m, n)\ , && m, n \in \La,\ m \ne n. \nn
\end{alignat}
Assume that the following conditions are valid,
\begin{gather}
v(n) = \overline{v(n)}, \label{eq:5-2N} \\[6pt]
h_0(m,n) = \overline{h_0(n,m)}, \label{eq:5-3N}\\[6pt]
|h_0(m, n)| \le B_1 \exp (-\ka_0 |m - n|), \quad m, n \in \La,\ m \ne n, \label{eq:5-5N}
\end{gather}
where $0 < B_1 \le 1$, $0 < \ka_0 \le 1/2$.

\begin{defi}\label{def:8-1a}
Assume that $H_{\La,\ve}$ obeys \eqref{eq:5-1N}--\eqref{eq:5-5N}. Assume also that there exist $m_0^+, m_0^- \in \La$, $m_0^- \neq m_0^+$ such that $|v(m^+_0) - v(m_0^-)| < \delta_0^3$ and $|v(n) - v(m^+_0)| \ge \delta_0$ for any $n \in \La \setminus \{m^+_0, m_0^-\}$. Assume also that
\begin{equation} \label{eq:4-3AAAAABBBBBBCCCC}
\bigl(m^\pm_0 + B(R^{(1)})\bigr) \subset \Lambda.
\end{equation}
We say in this case that $H_{\La,\varepsilon} \in \widehat{OPR^{(1)}} \bigl( m^+_0, m^-_0, \La; \delta_0 \bigr)$.

Let $s \ge 2$. Let $m^+_0, m^-_0 \in \La$, $m_0^+ \neq m_0^-$. Assume that there exist subsets $\cM^{(s')} \subset \La$, $s' = 1, \ldots, s - 1$, some of which may be empty, and a collection of subsets $\La^{(s')}(m) \subset \La$, $m \in \cM^{(s')}$, defined only for those $s'$ for which $\cM^{(s')} \neq \emptyset$. Assume that $m^+_0, m^-_0 \in \cM^\esone$. Assume that all conditions in Definition~\ref{def:4-1} hold with $m_0 := m^+_0$ and with the following exception. The estimate \eqref{eq:4-3sge3} holds for any $m \neq m_0^-$, and moreover,
\begin{equation} \label{eq:4-3AAAAAmnotm0}
12(\delta_0^{(s-1)})^{1/8} \le \big| E^{(s-1)} \bigl(m, \La^{(s-1)}(m); \ve\bigr) - E^{(s-1)} \bigl(m^+_0, \La^{(s-1)}(m^+_0); \ve\bigr) \big| \le \delta_0^{(s-2)}.
\end{equation}
For $m = m^-_0$, we have
\begin{equation} \label{eq:4-3AAAAA}
\big| E^{(s-1)} \bigl(m^-_0, \La^{(s-1)}(m^-_0); \ve \bigr) - E^{(s-1)}\bigl(m^+_0, \La^{(s-1)}(m^+_0); \ve\bigr) \big| \le (\delta_0^{(s-1)})^{1/8}.
\end{equation}
Assume also that
\begin{equation} \label{eq:4-3AAAAABBBBBB}
\bigl(m^\pm_0 + B(R^{(s)})\bigr) \subset \Lambda.
\end{equation}
In this case, we say that $H_{\La,\varepsilon}$ belongs to the class $\widehat{OPR^{(s)}}\bigl(m^+_0, m^-_0, \La; \delta_0 \bigr)$. We set $s(m^\pm_0) = s$. We call $m^+_0, m^-_0$ the  principal points. We call $\La^{(s-1)}(m^\pm_0)$ the $(s-1)$-set for $m^\pm_0$.
\end{defi}

\begin{remark}\label{rem:5.1}
$(1)$ We remark here that if $H_{\La,\varepsilon} \in \widehat{OPR^{(s)}} \bigl( m^+_0, m^-_0, \La; \delta_0 \bigr)$, then some of the statements in Proposition~\ref{prop:4-4} still hold for obvious reasons, the lower estimate in \eqref{eq:4-3sge3} for $m = m_0^-$ does not affect these statements. In Proposition~\ref{prop:5-4I} below, these statements are made explicit.

$(2)$ Note that the classes $\widehat{OPR^{(s)}} \bigl( m^+_0, m^-_0, \La; \delta_0 \bigr)$ and $\cN^{(s)} \bigl( m_0, \La; \delta_0 \bigr)$ may intersect since \eqref{eq:4-3AAAAA} does not exclude such a possibility, that is, it is possible that one has
\begin{equation} \label{eq:4-3AAAAAU}
3 \delta_0^{(s-1)} \le \big| E^{(s-1)} \bigl( m^-_0, \La^{(s-1)}(m^-_0); \ve \bigr) - E^{(s-1)} \bigl( m^+_0, \La^{(s-1)}(m^+_0); \ve \bigr) \big| \le (\delta_0^{(s-1)})^{1/8}.
\end{equation}
In fact, in Section~\ref{sec.8} we will have examples for which this happens.
\end{remark}

\begin{prop}\label{prop:5-4I}
Let $H_{\La,\varepsilon} \in \widehat{OPR^{(s)}} \bigl( m^+_0, m^-_0, \La; \delta_0 \bigr)$. For any $m \in \cM^{(s')}$ and $n \in \La^{(s')} (m) \setminus \{m\}$, we have $v(n) \neq v(m)$, $s' = 1, \dots, s-1$. So, $E^{(s')} (m, \La^{(s')} (m); 0) := v(m)$ is a simple eigenvalue of $H_{\La^{(s')} (m), 0}$. Let $E^{(s')} \bigl(m, \La^{(s')}(m); \ve \bigr)$ be the analytic function such that $E^{(s')} \bigl( m, \La^{(s')} (m); \ve \bigr) \in \spec H_{\La^{(s')}(m), \ve}$ for any $\ve$, $E^{(s')} \bigl(m, \La^{(s')}(m); 0 \bigr) = v(m)$.
\begin{itemize}

\item[(1)] Define inductively the functions $D(\cdot; \La^{(s')} (m))$, $1 \le s'\le s-1$, $m \in \cM{(s')}$, $D(\cdot; \La)$, by setting:

    for $s = 1$, \quad $D(x; \La) = 4 \log \delta_0^{-1}$, $x \in \La \setminus \{m_0^\pm\}$, $D(m_0^\pm;\La) := 4\log (\delta^\one)^{-1}$,

    for $s > 1$, \quad $D(x;\La) = D(x; \La^{(s')} (m))$ if $x \in \La^{(s')}(m)$ for some $s' \le s-1$ and some $m \in \cM{(s')} \setminus \{ m_0^\pm \}$, or if $x \in \La^{(s-1)} (m_0^\pm) \setminus \{ m_0^\pm \}$, $D(m_0^\pm; \La) = 4 \log (\delta^{(s)}_0)^{-1}$, $D(x; \La) = 4 \log \delta_0^{-1}$ if $x \in \Lambda \setminus \bigl( \bigcup_{1 \le s'\le s-1} \bigcup_{m \in \cM{(s')}} \La^{(s')}(m) \bigr)$.

Then, $D(\cdot; \La^{(s')} (m)) \in \mathcal{G}_{\La^{(s')} (m), T, \kappa_0}$, $1 \le s'\le s-1$, $m \in \cM{(s')}$, $D(\cdot; \La) \in \mathcal{G}_{\La, T, \kappa_0}$, $T = 4 \kappa_0 \log \delta_0^{-1}$, $\max_{x \notin \{m^+_0,m_0^-\}} D(x) \le 4 \log (\delta^{(s-1)}_0)^{-1}$.

\item[(2)] If $s = 1$, the matrix $(E - H_{\La \setminus \{m_0^+,m_0^-\},\ve})$ is invertible for any complex $|\ve|<\ve_0$, $|E - v(m_0^+)| < \delta_0/4$.

Let $s \ge 2$. For any complex $|\ve| < \ve_{s-2}$, $\big| E - E^{(s-1)}(m_0^+, \La^{(s-1)} (m_0^+); \ve) \big| < 10 (\delta^{(s-1)}_0)^{1/8}$, each matrix $(E - H_{\La^{(s')} (m), \ve})$, $s' \le s-1$, $m \in \cM^{(s')}$, $m \notin \{m_0^+,m_0^- \}$ is invertible. The matrices $(E - H_{\La^{(s-1)} (m_0^\pm) \setminus \{ m_0^\pm \}, \ve})$ and the matrix  $(E - H_{\La \setminus \{m_0^+,m_0^-\}, \ve})$ are invertible. Here, $E^{(0)}(m', \La';0) := v(m')$ for any $\La'$ and any $m' \in \La'$. Moreover,
\begin{equation}\label{eq:5Hinvestimatestatement1}
\begin{split}
|[(E - H_{\La^{(s')}(m), \ve})^{-1}] (x,y)| \le s_{D(\cdot; \La^{(s')}(m)), T, \kappa_0, |\ve|; \La^{(s')}(m)} (x,y), \\
|[(E - H_{\La^{(s-1)}(m_0^\pm) \setminus \{ m_0^\pm \}, \ve})^{-1}] (x,y)| \le s_{D(\cdot; \La^{(s-1)} (m_0^\pm) \setminus \{ m_0^\pm \}), T, \kappa_0, |\ve|; \La^{(s-1)} (m_0^\pm) \setminus \{ m_0^\pm \}} (x,y), \\
|[(E - H_{\La \setminus \{ m_0^+, m_0^- \} , \ve})^{-1}] (x,y)| \le s_{D(\cdot; \La \setminus \{ m_0^+, m_0^- \}), T, \kappa_0, |\ve|; \La \setminus \{ m_0^+, m_0^- \}} (x,y).
\end{split}
\end{equation}

\end{itemize}
\end{prop}

\begin{lemma}\label{lem:5schur1}
Using the notation from Proposition~\ref{prop:5-4I}, the following statements hold.
\begin{itemize}

\item[(1)] The functions
\begin{equation} \label{eq:5-10ac}
\begin{split}
K^{(s)}(m, n, \La; \ve, E) = (E - H_{\La_{m_0^+, m_0^-}})^{-1} (m,n) , \quad m, n \in \La_{m_0^+, m_0^-} := \Lambda \setminus \{ m_0^+, m_0^- \}, \\
Q^{(s)} (m_0^\pm, \La; \ve, E) = \sum_{m', n' \in \La_{m^+_0, m_0^-}} h(m_0^\pm, m'; \ve) K^{(s)}(m', n'; \La; \ve, E) h(n', m_0^\pm; \ve), \\
G^\es(m^\pm_0, m^\mp_0, \La; \ve, E) = h(m_0^\pm, m_0^\mp; \ve) + \sum_{m', n' \in \La_{m^+_0,m_0^-}} h(m_0^\pm, m'; \ve) K^{(s)}(m', n'; \La; \ve, E) h(n', m_0^\mp;\ve)
\end{split}
\end{equation}
are well-defined and analytic in the following domain,
\begin{equation}\label{eq:6.domainOP}
\begin{split}
|\ve| < \bar \ve_0, \quad |E - v(m_0)| < \delta_0/4, \quad \text { in case $s = 1$}, \\
|\ve| < \ve_{s-1} := \ve_0 - \sum_{1 \le s' \le s-1} \delta^{(s')}_0, \quad \big| E - E^{(s-1)}(m^+_0, \La^{(s-1)}(m_0^+); \ve) \big| < 10 (\delta^{(s-1)}_0)^{1/8}, \quad s \ge 2, \\
\ve_0 := \bar \ve_0^3, \quad \bar \ve_0: = \min(2^{-24 \nu-4} \kappa_0^{4\nu}, \delta_0^{2^9}, 2^{-10(\nu+1)} ( (4 \kappa_0 \log \delta_0^{-1})^{-8\nu})).
\end{split}
\end{equation}
The following estimates hold with $0 \le \alpha \le 2$:
\begin{equation} \label{eq:5-12acACC}
\begin{split}
& \big| \partial^\alpha_E [Q^{(s)} (m_0^\pm, \La; \ve, E) - Q^{(s-1)} \bigl( m_0^\pm, \La^{(s-1)} (m_0^\pm); \ve,E \bigr) ] \big|\le 4 |\ve|^{3/2} \exp(-\frac{\kappa_0}{4} R^\esone) \le |\ve| (\delta_0^{(s-1)})^{12}, \\
& \big| \partial^\alpha_E Q^{(s)} (m_0^\pm, \La; \ve, E) \big| \le |\ve|, \\
& \big| \partial^\alpha_E G^\es( m^\pm_0, m^\mp_0, \La; \ve, E) \big| \le  4 |\ve|^{3/2} \exp(-\frac{\kappa_0}{4} |m_0^+ - m_0^-|) \le
4 |\ve|^{3/2} \exp(-\frac{\kappa_0}{4} R^\esone) \le |\ve| (\delta_0^{(s-1)})^{12}
\end{split}
\end{equation}

For $\ve, E \in \mathbb{R}$, the following identities hold:
\begin{equation} \label{eq:5-11selfadj}
\begin{split}
K^{(s)}(m, n, \La; \ve, E) = \overline{K^{(s)}(n, m, \La; \ve, E)}, \\
\overline{Q^\es(m^\pm_0, \La; \ve, E)} = Q^\es(m^\pm_0, \La; \ve, E), \quad G^\es(m^+_0, m^-_0, \La; \ve, E) = \overline{G^\es(m^-_0, m_0^+\La; \ve, E)}.
\end{split}
\end{equation}

\item[(2)] Let $|E - E^{(s-1)} (m_0^+, \La^\esone(m_0); \ve)| < 4 \delta^\esone$. Set $\cH_\La := E - \hle$. Let $\tilde H_2$ be as in the Schur complement formula with $\La_1 := \La_{m_0^+, m_0^-}$, $\La_2 := \La \setminus \La_1$. Then,
\begin{equation}\label{eq:5-13NNNNM}
\begin{split}
& \det \tilde H_2 = \chi(\ve, E) := \bigl( E - v(m_0^+) - Q^\es(m^+_0, \La; \ve, E) \bigr) \cdot \bigl( E - v(m_0^-) - Q^\es(m_0^-, \La; \ve, E) \bigr) \\
& \qquad - G^\es(m^+_0, m^-_0, \La; \ve, E) G^\es(m^-_0, m^+_0, \La; \ve, E).
\end{split}
\end{equation}
In particular, $E \in\spec H_{\La, \ve}$ if and only if $E$ obeys
\begin{equation} \label{eq:5-13NNNN}
\chi(\ve,E) = 0.
\end{equation}

\end{itemize}
\end{lemma}

\begin{proof}
The proof of all statements in $(1)$ is completely similar to the proof of $(3)$ in Proposition~\ref{prop:4-4}. The first identity in \eqref{eq:5-11selfadj} is due to the fact that $E - H_{\La_{m_0^+, m_0^-}}$ is self-adjoint if $\ve,E$ are real. Furthermore, one has
\begin{equation}\label{eq:5-10acselfadj}
\begin{split}
\overline{Q^{(s)} (m_0^\pm, \La; \ve, E)} = \overline{\sum_{m', n' \in \La_{m^\pm_0, m_0^-}} h(m_0^\pm, m'; \ve) K^{(s)} (m', n'; \La; \ve, E) h(n', m_0^\pm; \ve)} \\
= \sum_{m', n' \in \La_{m^\pm_0,m_0^-}} \overline{h(m_0^\pm, m'; \ve)} \quad \overline{K^{(s)} (m', n'; \La; \ve, E)} \quad \overline{h(n', m_0^\pm; \ve)}\\
= \sum_{m', n' \in \La_{m^\pm_0,m_0^-}} h(m_0^\pm, n'; \ve) K^{(s)}(n', m'; \La; \ve, E) h(m', m_0^\pm; \ve) = Q^{(s)} (m_0^\pm, \La; \ve, E).
\end{split}
\end{equation}
This verifies the second identity in \eqref{eq:5-11selfadj}. The verification of the third identity in \eqref{eq:5-11selfadj} is similar.

Due to the Schur complement formula with $\La_1 = \La_{m_0^+, m_0^-}$, $\La_2 = \La \setminus \La_1$, $\cH_\La = E - \hle$ is invertible if and only if
\begin{equation}\label{eq:5schurfor1a}
\tilde H_2 := \begin{bmatrix} E - v(m_0^+) - Q^\es(m^+_0, \La; \ve, E)\bigr) & - G^\es(m^+_0, m^-_0, \La; \ve, E) \\[5pt]
 - G^\es(m^-_0, m^+_0, \La; \ve, E) & E - v(m_0^-) - Q^\es(m^-_0, \La; \ve, E) \end{bmatrix}
\end{equation}
is invertible. Note that $\det \tilde H_2 = \chi(\ve, E)$. In particular, $E \in \spec H_{\La,\ve}$ if and only if it obeys \eqref{eq:5-13NNNN}.
\end{proof}

\begin{defi}\label{def:5-2a}
Using the notation of Lemma~\ref{lem:5schur1}, assume that for every $\ve \in (-\ve_{s-1},\ve_{s-1})$ and every $|E - E^\esone(m^+_0, \La^\esone (m^+_0); \ve)| < 10 (\delta^{(s-1)}_0)^{1/8}$, we have
\begin{equation} \label{eq:5-13NNNN1}
v(m_0^+) + Q^\es(m^+_0, \La; \ve, E) \ge v(m_0^-) + Q^\es(m_0^-,\La; \ve, E) + \tau^\zero,
\end{equation}
where $\tau^\zero > 0$. Then we say that $\hle \in OPR^{(s)} \bigl( m^+_0, m^-_0, \La; \delta_0, \tau^\zero \bigr)$. We always assume here for convenience that $\tau^\zero \le (\delta_0^\esone)^3$.
\end{defi}

\begin{prop}\label{prop:8-5n}
Assume $\hle \in OPR^{(s)} \bigl( m^+_0, m^-_0, \La; \delta_0, \tau^\zero \bigr)$.
\begin{itemize}

\item[(1)] For $\ve \in (-\ve_{s-1}, \ve_{s-1})$, $|E - E^\esone (m^+_0, \La^\esone (m^+_0); \ve)| < 8 (\delta^\esone_0)^{1/8}$, the equation
\begin{equation} \label{eq:7-13}
\begin{split}
& \chi(\ve,E) := \bigl( E - v(m_0^+) - Q^\es (m^+_0, \La; \ve, E) \bigr) \cdot \bigl( E - v(m_0^-) - Q^\es (m_0^-, \La ; \ve, E) \bigr) \\
& \qquad - \big |G^\es (m^+_0, m^-_0, \La; \ve, E) \big |^2 = 0
\end{split}
\end{equation}
has exactly two solutions $E = E^{(s, \pm)} (m_0^+, \La; \ve)$, obeying $E^{(s, -)} (m_0^+, \La; \ve) < E^{(s, +)}(m_0^+, \La; \ve)$,
\begin{equation} \label{eq.5Eestimates1AP}
|E^{(s, \pm)} (m_0^+, \La; \ve) - E^\esone (m^+_0, \La^\esone (m^+_0); \ve)| < 4|\ve| (\delta^\esone_0)^{1/8}.
\end{equation}
The functions $E^{(s, \pm)}(m^+_0, \La; \ve)$ are $C^2$-smooth on the interval $(-\ve_{s-1}, \ve_{s-1})$.

\item[(2)] The following estimates hold:
\begin{equation} \label{eq:8-13acnq0}
\begin{split}
|\partial^\alpha_E \chi| \le 8, \quad \text {for $\alpha \le 2$}, \quad \partial^2_E \chi > 1/8, \\
\partial_E \chi|_{\ve,E^{(s,-)}(m^+_0, \La; \ve)} < -(\tau^\zero)^2, \quad \partial_E \chi|_{\ve, E^{(s,+)}(m^+_0, \La; \ve)} > (\tau^\zero)^2, \\
E^{(s,+)} (m^+_0, \La; \ve) - E^{(s,-)} (m^+_0, \La; \ve) > \frac{1}{8} [- \partial_E \chi|_{\ve,E^{(s,-)} (m^+_0, \La; \ve)} + \partial_E \chi|_{\ve, E^{(s,+)} (m^+_0, \La; \ve)}], \\
- \partial_E \chi|_{\ve,E^{(s,-)} (m^+_0, \La; \ve)}, \partial_E \chi|_{\ve, E^{(s,+)} (m^+_0, \La; \ve)} > \frac{1}{2^{14}}
\bigl(E^{(s,+)} (m^+_0, \La; \ve) - E^{(s,-)} (m^+_0, \La; \ve) \bigr)^2, \\
|\chi(\ve, E)| \ge \frac{1}{8} \min \bigl( (E - E^{(s,-)} (m^+_0, \La; \ve))^2, (E - E^{(s,+)} (m^+_0, \La; \ve))^2 \bigr), \\
[a_1 (\ve, E) + |b (\ve, E)|]|_{E=E^{(s,+)} (m^+_0, \La; \ve)} \ge E^{(s,+)} (m^+_0, \La; \ve) \\
\ge \max ( a_1 (\ve, E),  a_2 (\ve, E) + |b (\ve, E)| )|_{E=E^{(s,+)} (m^+_0, \La; \ve)}, \\
[a_2 (\ve, E) - |b (\ve, E)| ]|_{E = E^{(s,-)} (m^+_0, \La; \ve)}\le E^{(s,-)} (m^+_0, \La; \ve) \\
\le \min ( a_2 (\ve, E),  a_1 (\ve, E) - |b (\ve, E)| )|_{E = E^{(s,-)} (m^+_0, \La; \ve)}.
\end{split}
\end{equation}
where
\begin{equation}
\begin{split}\nn
a_1 (\ve, E) = v(m_0^+) + Q^\es (m^+_0, \La; \ve, E), \quad\quad a_2(\ve, E) = v(m_0^-) + Q^\es (m_0^-, \La ; \ve, E), \\
b (\ve, E) = |b_1 (\ve,E)|, \quad b_1 (\ve, E) = G^\es (m^+_0, m^-_0, \La; \ve, E).
\end{split}
\end{equation}

\item[(3)] We have
\begin{equation} \label{eq:5specHEE}
\begin{split}
\spec H_{\La, \ve} \cap \{E : |E - E^\esone (m^+_0, \La^\esone (m^+_0); \ve)| < 8 (\delta^\esone_0)^{1/4}\} \\
= \{ E^{(s,+)} (m_0^+, \La; \ve), E^{(s, -)} (m_0^+, \La; \ve)\}, \\
E^{(s,\pm)} (m_0^+, \La; 0) = v(m^\pm_0).
\end{split}
\end{equation}

\item[(4)] Using the notation from part $(1)$ of Proposition~\ref{prop:5-4I}, for any
\begin{equation} \label{eq:8-13acnq0EDOM}
E^{(s,-)} \bigl( m^+_0, \La; \ve \bigr) - (\delta_0^{(s-1)})^{1/8} < E < E^{(s,+)} \bigl( m^+_0, \La; \ve \bigr) + (\delta_0^{(s-1)})^{1/8},
\end{equation}
the matrix $(E - H_{\La \setminus \{ m_0^+, m_0^- \}, \ve})$ is invertible and
\begin{equation}\label{eq:3Hinvestimatestatement1P}
|[(E - H_{\La \setminus \{ m_0^+, m_0^- \}, \ve })^{-1}] (x,y)| \le s_{D(\cdot; \La \setminus \{ m_0^+, m_0^- \}), T, \kappa_0, |\ve|; \La \setminus \{ m_0^+, m_0^- \}} (x,y).
\end{equation}
If
\begin{equation}\label{eq:5Esplitspecconddomain}
(\delta^\es_0)^4 < \min_\pm |E - E^{(s,\pm)} (n_0^+, \La; \ve)| < 6 (\delta^\esone_0)^{1/8},
\end{equation}
then the matrix $(E - H_{\La,\ve})$ is invertible. Moreover,
\begin{equation}\label{eq:5inverseestiMATE}
|[(E - H_{\La, \ve})^{-1}] (x,y)| \le s_{D(\cdot; \La), T, \kappa_0, |\ve|; k, \La, \mathfrak{R}} (x,y).
\end{equation}

\end{itemize}
\end{prop}

\begin{proof}
To prove $(1)$, we apply Lemma~\ref{lem:6-1ell}. Consider the case $s \ge 2$. Set
\begin{equation}
\begin{split}\nn
f_i = E - a_i, \quad f = f_1 - \frac{b^2}{f_2}, \quad g_0 (\ve) = E^\esone (m^+_0, \La^\esone (m^+_0); \ve), \quad \rho_0 = 10 (\delta^{(s-1)}_0)^{1/8}, \\
\tilde f_1 (\ve, E) = E - v(m^+_0)- Q^{(s-1)} (m_0^+, \La^{(s-1)} (m_0^+); \ve, E), \\
\tilde f_2 (\ve, E) = E - v(m^-_0)- Q^{(s-1)} (m_0^-, \La^{(s-1)} (m_0^-); \ve, E).
\end{split}
\end{equation}
We apply Lemma~\ref{lem:6-1ell} to $\chi^{(f)} = (E - a_1)(E- a_2) - b^2$. We also verify that $f \in \mathfrak{F}^{(1)}_{\mathfrak{g}^\one} (f_1, f_2, b^2)$. Let us verify conditions (i)--(iii) before Lemma~\ref{lem:6-1}. The functions $a_j$, $b$ are analytic in the complex domain $|\ve| < \ve_{s-1}$, $\big| E - E^{(s-1)}(m^+_0, \La^{(s-1)}(m_0^+); \ve) \big| < 10 (\delta^{(s-1)}_0)^{1/8}$, due to Lemma~\ref{lem:5schur1}. So, conditions (i) and (ii) hold. Due to the second identity in \eqref{eq:5-11selfadj}, $a_j$ assumes real values if $\ve, E$ are real. Due to Definition~\ref{def:5-2a}, we have $a_1 (\ve, E) - a_2 (\ve, E) \ge \tau^{(0)}$ for any $\ve \in (-\ve_{s-1},\ve_{s-1})$ and any $E \in (E^\esone(m^+_0, \La^\esone (m^+_0); \ve) - 10 (\delta^\esone_0)^{1/8}, E^\esone (m^+_0, \La^\esone (m^+_0); \ve) + 10 (\delta^\esone_0)^{1/8})$. One has also $G^\es (m^\pm_0, m^\mp_0, \La; 0, E) = 0$. Thus, both requirements in condition (iii) hold. Due to \eqref{eq:5-12acACC}, $|\partial^\alpha_E a_i| \le |\ve| < 1/64$, $i = 1, 2$, $\alpha = 1, 2$, $|\partial^\alpha b^2| < 4 |\ve|^{3/2} \exp(-\frac{\kappa_0}{4} R^\esone) < 1/64$, $\alpha \le 2$. Furthermore,
\begin{equation} \label{eq:5-12acACCAA}
\begin{split}
|\partial^\alpha_E [f_i - \tilde f_i]| \le \max_{+,-} \big| \partial^\alpha_E [Q^{(s)} (m_0^\pm, \La; \ve, E) - Q^{(s-1)} \bigl( m_0^\pm, \La^{(s-1)} (m_0^\pm); \ve, E \bigr) ] \bigr| \le 4 |\ve|^{3/2} \exp(-\frac{\kappa_0}{4} R^\esone), \quad \alpha \le 2, \\
\tilde f_1 (\ve, E^\esone (m^+_0, \La^\esone (m^+_0); \ve) = 0, \quad \tilde f_2(\ve, E^\esone (m^-_0, \La^\esone (m^-_0); \ve) = 0, \\
E^\esone (m^+_0, \La^\esone (m^+_0); \ve) - (\delta_0^{(s-1)})^{1/8} \le E^\esone (m^-_0, \La^\esone (m^-_0); \ve) < E^\esone (m^+_0, \La^\esone(m^+_0); \ve).
\end{split}
\end{equation}
Here we used \eqref{eq:5-12acACC}. Since $|\partial_E \tilde f_i| < 1$, \eqref{eq:5-12acACCAA} implies in particular
\begin{equation} \label{eq:5-12acACCAAqzeroref}
|E - a_i| = |f_i| < |E-E^\esone (m^+_0, \La^\esone (m^+_0); \ve)| < \rho_0 < 1/64.
\end{equation}
Moreover, all conditions in Definition~\ref{def:4a-functions} hold, and hence $f \in \mathfrak{F}^{(1)}_{\mathfrak{g}^\one} (f_1, f_2, b^2)$, $\chi = \chi^{(f)}$. Lemma~\ref{lem:6-1ell} implies parts $(1)$, $(2)$ of the current proposition.

The first identity in \eqref{eq:5specHEE} follows from part $(2)$ of Lemma~\ref{lem:5schur1}. The second identity in \eqref{eq:5specHEE} follows from the first one since $v(m^+_0), v(m^-_0)$ are the only eigenvalues of $H_{\La,0}$ which belong to the interval in the first line, and $v(m^+_0)) > v(m^-_0)$. This finishes part $(3)$.

We will now verify $(4)$. The estimate \eqref{eq:3Hinvestimatestatement1P} is due to \eqref{eq:3Hinvestimatestatement1}. For $E$ in the domain \eqref{eq:5Esplitspecconddomain}, we invoke Lemma~\ref{lem:aux5AABBCC} with $\La_2 = \{ m_0^+, m_0^- \}$. We need to verify conditions $(i)$, $(ii)$ in Lemma~\ref{lem:aux5AABBCC}. Condition $(i)$ holds due to \eqref{eq:3Hinvestimatestatement1P}. Let $\tilde H_2 := \cH_{\La_2} - \Gamma_{2,1} \cH_{\La_1}^{-1} \Gamma_{1,2}$. Recall that $\det \tilde H_2 = \chi(\ve,E)$, due to part $(2)$ of Lemma~\ref{lem:5schur1}.
Due to \eqref{eq:8-13acnq0}, one obtains
$$
D_0 := \log |\det \tilde H_2|^{-1} = \log |\chi(\ve,E)|^{-1} \le \frac{1}{4} \log (\delta^\esone)^{-1} + 3 \log 2 < D(m_0^\pm;\La);
$$
see the notation from part $(1)$ of Proposition~\ref{prop:5-4I}. Furthermore, due to condition \eqref{eq:4-3AAAAABBBBBB}, $\mu_\La (m_0^\pm) \ge R^\es$. Due to Remark~\ref{rem:3.Rs}, one obtains $D_0 < [\min ( \mu_\La(m_0^+), \mu_\La (m_0^-))]^{1/5}$. Thus, condition $(ii)$ in Lemma~\ref{lem:aux5AABBCC} holds. Due to Lemma~\ref{lem:aux5AABBCC}, \eqref{eq:5inverseestiMATE} holds. This finishes the case $s \ge 2$. The verification in case $s=1$ is completely similar.
\end{proof}

\begin{remark}\label{rem:7.sharperssoneestimate}
Here we want to comment on a stronger version of the estimate \eqref{eq.5Eestimates1AP} in the statement of the last proposition. Namely, in some of our applications we will consider cases where some additional conditions hold. Namely, the sets  $\La^\esone (m^\pm_0)$ will obey
\begin{equation} \label{eq.5biggerR}
\La^\esone (m^\pm_0) \supset m^\pm_0 + B(R)
\end{equation}
with $R > R^\esone$. Furthermore,
\begin{equation} \label{eq:4-3AAAAANEW}
\big| E^{(s-1)} \bigl(m^-_0, \La^{(s-1)}(m^-_0); \ve \bigr) - E^{(s-1)} \bigl( m^+_0, \La^{(s-1)}(m^+_0); \ve \bigr) \big| \le \exp(-R),
\end{equation}
compare with \eqref{eq:4-3AAAAA}. In this case, a revision of the proof of \eqref{eq.5Eestimates1AP} shows that the following stronger estimate holds,
\begin{equation} \label{eq.5Eestimates1APNEW}
|E^{(s, \pm)} (m_0^+, \La; \ve) - E^\esone (m^+_0, \La^\esone (m^+_0); \ve)| < 2|\ve| \exp(-\frac{\kappa_0}{2} R).
\end{equation}
\end{remark}

\begin{defi}\label{def:7-6}
Let $\hle$ be as in \eqref{eq:5-1N}--\eqref{eq:5-5N}. Let $s > 0$, $q > 0$ be integers. Assume that the classes of matrices $OPR^{(s,s')} \bigl( \tilde m^+_0, \tilde m^-_0, \tilde \La; \delta_0, \tau_0 \bigr)$ are defined for $s \le s' \le s + q - 1$, starting with $OPR^{(s,s)} \bigl( \tilde m^+_0, \tilde m^-_0, \tilde \La; \delta_0, \tau_0 \bigr) := OPR^{(s)} \bigl( \tilde m^+_0, \tilde m^-_0, \tilde \La; \delta_0, \tau_0 \bigr)$ being as in Definition~\ref{def:5-2a}. Let $m^+_0$, $m^-_0 \in \La$. Assume that there are subsets $\cM^{(s',+)} = \left\{ m^+_j : j \in J^{(s')} \right\}$, $\cM^{(s',-)} = \left\{ m^-_j : j \in J^{(s')} \right\}$, $\La^{(s')} (m_j^+) = \La^{(s')} (m_j^{-})$, $j \in J^{(s')}$, with $s \le s' \le s + q - 1$, and also subsets $\cM^{(s')}$, $\La^{(s')}(m)$, $m \in \cM^{(s')}$, $1 \le s'\le s + q - 1$ such that the following conditions are valid:
\begin{enumerate}

\item[(i)] $ m^\pm_0 \in \cM^{(s + q - 1, \pm)}$, $($ so, by convention, $0 \in J^{(s + q - 1)}$ $)$, $m \in \La^{(s')}(m) \subset \La$ for any $m$.

\item[(ii)]
\begin{equation} \nn
\begin{split}
\cM^{(\mathfrak{s}')} (\La) \cap \cM^{(\mathfrak{s}'')} (\La) = \emptyset, \quad \text {for any possible superscript indices $\mathfrak{s}' \neq \mathfrak{s}''$}, \\
\La^{(s')} (m') \cap \La^{(s'')} (m'') = \emptyset, \quad \text{ unless $s' = s''$, and $m' = m''$ or $m' = m^\pm_j$, $m'' = m^\mp_j$.}
\end{split}
\end{equation}

\item[(iii)] For $\tau^\zero > 0$ and any $m^+_j \in \cM^{(s',+)}$, $s' \ge s$, $H_{\La^{(s')} (m_j^+), \ve} \in OPR^{(s,s')} \bigl( m^+_j, m^-_j, \La^{(s')} (m_j^+); \delta_0, \tau^\zero \bigr)$. For any $m \in \cM^{(s')}$, $H_{\La^{(s')}(m), \ve} \in \cN^{(s')}(m,\La^{(s')}(m),\delta_0)$.

\item[(iv)] Let $\delta^{(s')}_0$, $R^{(s')}$ be as in Definition~\ref{def:4-1}. Then,
\begin{equation} \nn
\begin{split}
\bigl( m' + B(R^{(s')}) \bigr) \subset \Lambda^{(s')} (m'), \quad \text {for any $m'$, $s'$}, \\
\bigl( m_j^\pm + B(R^{(s')}) \bigr) \subset \Lambda^{(s')}(m^+_j), \quad \text {for any $j$, $s \le s' < s + q$}, \\
\bigl( m_0^\pm + B(R^{(s+q)}) \bigr) \subset \Lambda.
\end{split}
\end{equation}

\item[(v)] Given $m^+_j \in \cM^{(s',+)}$, let $E^{(s',\pm)} \bigl(m^+_j, \La^{(s')}(m^+_j); \ve \bigr)$, $Q^{(s')} \bigl(m^\pm_j, \La^{(s')} (m^+_j) ; \ve,E \bigr)$, etc.\ be the functions defined for the matrix $H_{\La^{(s')}(m^+_j), \ve}$. $($ Here, $E^{(s,\pm)} \bigl( m^+_j, \La^{(s)} (m^+_j); \ve \bigr)$ are just as in Proposition~\ref{prop:8-5n}. Below in Proposition~\ref{rem:con1smalldenomnn} we will give the construction of these functions for $s' > s$, which justifies the use of these functions in our inductive definition. $)$ Similarly, given $m \in \cM^{(s')}$, let $E^{(s')} \bigl( m, \La^{(s')}(m); \ve \bigr)$ be the functions defined for the matrix $H_{\La^{(s')} (m), \ve} \in \cN^{(s')} (m, \La^{(s')} (m), \delta_0)$. For each $m^+_j \in \cM^{(s',+)}$, $m^+_j \notin \{ m^+_0, m^-_0 \}$, $s \le s' < s+q$, any  $\ve \in (-\ve_{s-1}, \ve_{s-1})$, we have
\begin{align}
3 \delta^{(s+q-1)}_0 \le |E^{(s+q-1,\pm)} \bigl( m^+_j, \La^{(s+q-1)}(m^+_j); \ve \bigr) - E^{(s+q-1,\pm)} \bigl( m^+_0, \La^{(s+q-1)}(m^+_0); \ve \bigr)| \le \delta^{(s+q-2)}_0 \label{eq:5EVsplitdefs1a}, \\
3 \delta^{(s+q-1)}_0 \le |E^{(s+q-1,\mp)} \bigl(m^+_j, \La^{(s+q-1)}(m^+_j); \ve \bigr) - E^{(s+q-1,\pm)} \bigl( m_0^+, \La^{(s+q-1)}(m_0^+); \ve \bigr)| \label{eq:5EVsplitdefAs1b}, \\
\frac{\delta^{(s')}_0}{2} \le |E^{(s',\pm)} \bigl( m^+_j, \La^{(s')}(m^+_j); \ve \bigr) - E^{(s+q-1,\pm )} \bigl( m^+_0, \La^{(s+q-1)}(m^+_0); \ve \bigr)| \le \delta^{(s'-1)}_0 \label{eq:5EVsplitdefsq} , \quad \text{for $s \le s' < s+q-1$,} \\
\frac{\delta^{(s')}_0}{2} \le |E^{(s',\mp)} \bigl( m^+_j, \La^{(s')}(m^+_j); \ve \bigr) - E^{(s+q-1,\pm)} \bigl( m_0^+, \La^{(s+q-1)}(m_0^+); \ve \bigr) | \label{eq:5EVsplitdefA}, \quad \text{for $s \le s' < s+q-1$.}
\end{align}

Furthermore, for any $m \in \cM^{(s')}$, $1 \le s' \le s+q-1$ and any $\ve \in (-\ve_{s-1}, \ve_{s-1})$, we have
$$
\frac{\delta^{(s')}}{2} \le |E^{(s')} \bigl( m, \La^{(s')}(m); \ve \bigr) - E^{(s+q-1,+)} \bigl( m^+_0, \La^{(s+q-1)}(m^+_0); \ve \bigr)| \le \delta^{(s'-1)}_0.
$$

\item[(vi)] $|v(n) - v(m_0^+)| \ge 2 \delta_0^4$ for any $n \in \Lambda \setminus \bigl( \big[ \bigcup_{1 \le s' \le s+q-1} \bigcup_{m \in \cM{(s')}} \La^{(s')}(m) \big] \cup \big[ \bigcup_{s \le s' \le s+q-1} \bigcup_{j \in J^{(s')}} \La^{(s')}(m^+_j) \big] \bigr)$.

\item[(vii)] In Proposition~\ref{rem:con1smalldenomnn} we will show inductively that the functions
\begin{equation} \label{eq:5-10acOPQDEF}
\begin{split}
K^{(s+q)}(m, n, \La; \ve, E) = (E - H_{\La_{m_0^+, m_0^-}})^{-1} (m,n) , \quad m, n \in \La_{m_0^+, m_0^-} := \Lambda \setminus \{m_0^+, m_0^-\}, \\
Q^{(s+q)}(m_0^\pm, \La; \ve, E) = \sum_{m', n' \in \La_{m^\pm_0,m_0^-}} h(m_0^\pm, m'; \ve) K^{(s+q)}(m', n'; \La; \ve, E) h(n', m_0^\pm; \ve)
\end{split}
\end{equation}
are well-defined for any $\ve \in (-\ve_{s-1}, \ve_{s-1})$ and any
$$
E \in \bigcup_\pm(E^{(s+q-1,\pm)} \bigl(m^+_0, \La^{(s+q-1)} (m^+_0);\ve) - 2\delta^{(s+q-1)}_0,E^{(s+q-1,\pm)} \bigl(m^+_0, \La^{(s+q-1)} (m^+_0); \ve) + 2 \delta^{(s+q-1)}_0).
$$
We require that for these $\ve, E$ and with $\tau^\zero$ from {\rm (iii)}, we have
\begin{equation} \label{eq:5-13OPQ}
v(m_0^+) + Q^{(s+q)}(m^+_0, \La,E) \ge v(m_0^-) + Q^{(s+q)}(m_0^-,\La; \ve, E) + \tau^\zero.
\end{equation}
\end{enumerate}

Then we say that $\hle \in OPR^{(s,s+q)} \bigl( m^+_0, m^-_0, \La; \delta_0, \tau^\zero \bigr)$. We set $s(m_0^\pm) = s+q$. We call $m^+_0$, $m^-_0$ the principal points. We call $\La^{(s+q-1)}(m^\pm_0)$ the $(s+q-1)$-set for $m^\pm_0$.
\end{defi}

\begin{prop}\label{rem:con1smalldenomnn}
For each $q$ and any $\hle \in OPR^{(s,s+q)} \bigl( m^+_0, m^-_0, \La; \delta_0, \tau^\zero \bigr)$, one can define the functions $E^{(s+q,\pm)} (m_0^+, \La; \ve \bigr)$ so that the following conditions hold.
\begin{itemize}

\item[(0)] $E^{(s+q,\pm)} \bigl(m^+_0, \La; \ve \bigr)$ are $C^2$-smooth in $\ve \in (-\ve_{s-1}, \ve_{s-1})$.

\item[(1)] Let $D(\cdot; \La^{(s')}(m))$, $1 \le s' \le s+q-1$, $m \in \cM^{(s')}$ be defined as in Proposition~\ref{prop:4-4}. Define inductively the functions $D(\cdot; \La^{(s')}(m_j^+))$, $s \le s' \le s+q-1$, $j \in J{(s')}$, and the function
    $D(\cdot;\La)$ as follows. For $s' = s$, let $D(\cdot; \La^{(s')}(m_j^+))$ be just $D(\cdot; \La)$ from Proposition~\ref{prop:5-4I}
    with $\La^{(s')}(m_j^+)$ in the role of $\La$ and $m_j^+$ in the role of $m_0^+$. Similarly, for $s' > s$, let $D(\cdot; \La^{(s')}(m_j^+))$ be just $D(\cdot; \La)$ from the current proposition with $\La^{(s')}(m_j^+)$ in the role of $\La$ and $m_j^+$ in the role of $m_0^+$. Set $D(x; \La) = D(x; \La^{(s')}(m))$ if $x \in \La^{(s')}(m)$ for some $s' \le s-1$, or if $x \in \La^{(s')}(m)$, $m = m_j^+$, $j \in  J^{(s')}$, $s' \ge s$, $m^+_j \notin \{m^+_0,m_0^-\}$. Set $D(x; \La) = 4 \log \delta_0^{-1}$ if $x \in \Lambda \setminus \bigl( \big[ \bigcup_{1 \le s' \le s+q-1} \bigcup_{m \in \cM^{(s')}} \La^{(s')}(m) \big] \cup \big[ \bigcup_{s \le s'\le s+q-1} \bigcup_{j \in J^{(s')}} \La^{(s')}(m^+_j) \big] \bigr)$. Finally, set $D(m_0^\pm; \La) = D_0 := 4 \log (\delta^{(s+q)}_0)^{-1}$.

Then, $D(\cdot;\La) \in \mathcal{G}_{\La, T, \kappa_0}$, $T = 4 \kappa_0 \log \delta_0^{-1}$, and
$$
\max_{x \notin \{ m_0^+,m_0^-\}} D(x) \le 4 \log (\delta^{(s+q-1)}_0)^{-1}, \quad \max_{x \in \La} D(x) \le
    4 \log (\delta^{(s+q)}_0)^{-1}.
$$

\item[(2)] Let $q \ge 1$, $\cL^{(s+q-1,\pm)} := \cL_{\IR} \bigl( E^{(s+q-1,\pm)} \bigl( m^+_0, \La^{(s+q-1)}(m^+_0); \ve \bigr), 2 \delta_0^{(s+q-1)} \bigr)$. For any $(\ve,E) \in \cL^{(s+q-1,+)} \cup \cL^{(s+q-1,-)}$, the matrix $(E - H_{\La \setminus \{m_0^+, m_0^-\},\ve})$ is invertible. Moreover,
\begin{equation}\label{eq:3Hinvestimatestatement1PQ}
|[(E - H_{\La \setminus \{m_0^+,m_0^-\},\ve})^{-1}] (x,y)| \le s_{D(\cdot; \La \setminus \{m_0^+,m_0^-\}), T, \kappa_0, |\ve|; \La \setminus \{m_0^+,m_0^-\}, \mathfrak{R}}(x,y).
\end{equation}

\item[(3)] The functions
\begin{equation} \label{eq:5-10acOPQ}
\begin{split}
K^{(s+q)}(m, n, \La; \ve, E) = (E - H_{\La_{m_0^+, m_0^-}})^{-1} (m,n), \quad m, n \in \La_{m_0^+,m_0^-} := \Lambda \setminus \{m_0^+,m_0^-\}, \\
Q^{(s+q)}(m_0^\pm, \La; \ve, E) = \sum_{m', n' \in \La_{m^\pm_0,m_0^-}} h(m_0^\pm, m'; \ve) K^{(s+q)}(m', n'; \La; \ve, E) h(n', m_0^\pm; \ve), \\
G^{(s+q)}(m^\pm_0, m^\mp_0, \La; \ve, E) = h(m^\pm_0, m_0^\mp, \ve) + \sum_{m', n' \in \La_{m_0^+, m_0^-}} h(m_0^\pm, m'; \ve) K^{(s+q)}(m', n'; \La; \ve, E) h(n', m_0^\mp; \ve)
\end{split}
\end{equation}
are well-defined and $C^2$-smooth in $\cL^{(s+q-1,+)} \cup \cL^{(s+q-1,-)}$. These functions obey the following estimates for $(\ve,E) \in \cL^{(s+q-1,+)} \cup \cL^{(s+q-1,-)}$ and $0 \le \alpha \le 2$: 2 $|\ve|^{3/2}$ and some more regarding $G$
\begin{equation} \label{eq:5-12acACCOPQ}
\begin{split}
& \big| \partial^\alpha_E Q^{(s+q)}(m_0^\pm, \La; \ve, E) - \partial^\alpha_E Q^{(s+q-1)} \bigl( m_0^\pm, \La^{(s+q-1)}(m_0^+); \ve,E \bigr) \big | \\
& \le 4 |\ve|^{3/2} \exp(-\kappa_0 R^{(s+q-1)}) \le |\ve| (\delta_0^{(s+q-1)})^{12}, \\
& \big |\partial^\alpha_E G^{(s+q)}(m^\pm_0, m^\mp_0, \La; \ve, E) - \partial^\alpha_EG^{(s+q-1)} \bigl( m_0^\pm, m^\mp_0, \La^{(s+q-1)}(m_0^+); \ve, E \bigr) \big | \\
& \le 4 |\ve|^{3/2} \exp(-\kappa_0 R^{(s+q-1)}) \le |\ve| (\delta_0^{(s+q-1)})^{12}, \\
& \big| \partial^\alpha_E Q^{(s+q)}(m_0^\pm, \La; \ve, E) \big| \le |\ve|, \quad |E - v(m_0^\pm) - Q^{(s+q)}(m_0^\pm, \La; \ve, E)| < |\ve| \\
& \big| \partial^\alpha_EG^{(s+q)}(m^\pm_0, m^\mp_0, \La; \ve, E) \big| \le 8|\ve|^{3/2} \exp(-\frac{7\kappa_0}{8} |m_0^+ - m_0^-|)
\le |\ve| (\delta_0^{(s-1)})^{12}.
\end{split}
\end{equation}

The following identities hold:
\begin{equation} \label{eq:5-11selfadjOPQ}
\overline{Q^{(s+q)}(m^\pm_0, \La; \ve, E)} = Q^{(s+q)}(m^\pm_0, \La; \ve, E), \quad G^{(s+q)}(m^+_0, m^-_0, \La; \ve, E) = \overline{G^{(s+q)}(m^-_0, m_0^ + \La; \ve, E)}.
\end{equation}

\item[(4)] Let $(\ve,E) \in \cL^{(s+q-1,+)} \cup \cL^{(s+q-1,-)}$. Then, $E \in \spec H_{\La,\ve}$ if and only if $E$ obeys
\begin{equation} \label{eq:5-13NNNNOPQ}
\begin{split}
& \chi(\ve,E) := \bigl(E - v(m_0^+) - Q^{(s+q)}(m^+_0, \La; \ve, E) \bigr) \cdot \bigl( E - v(m_0^-) - Q^{(s+q)}(m_0^-, \La; \ve, E) \bigr) \\
& \qquad - G^{(s+q)}(m^+_0, m^-_0, \La; \ve, E) G^{(s+q)}(m^-_0, m^+_0, \La; \ve, E) = 0.
\end{split}
\end{equation}

\item[(5)] For $\ve \in (-\ve_{s-1}, \ve_{s-1})$, the equation
\begin{equation} \label{eq:8-13nn}
\chi(\ve,E) = 0
\end{equation}
has exactly two solutions $E = E^{(s+q, \pm)}(m_0^+, \La; \ve)$, obeying $E^{(s+q, -)}(m_0^+, \La; \ve) < E^{(s+q, +)}(m_0^+, \La; \ve)$ and
\begin{equation} \label{eq.5Eestimates1APqq}
|E^{(s+q, \pm)}(m_0^+, \La; \ve) - E^{(s+q-1,\pm)}(m^\pm_0, \La^{(s+q-1)}(m^+_0); \ve)| < |\ve| (\delta^{(s+q-1)}_0)^3.
\end{equation}
The functions $E^{(s+q, \pm)}(m^+_0, \La; \ve)$ are $C^2$-smooth on the interval $(-\ve_{s-1}, \ve_{s-1})$. The following estimates hold:
\begin{equation} \label{eq:8-13acnq0A}
\begin{split}
|\partial^\alpha_E \chi| \le 8, \quad \text {for $\alpha \le 2$}, \quad \partial^2_E \chi > 1/8, \\
\partial_E \chi|_{\ve, E^{(s+q,-)}(m^+_0, \La; \ve)} < - (\tau^\zero)^2, \quad \partial_E \chi|_{\ve,E^{(s+q,+)}(m^+_0, \La; \ve)} > (\tau^\zero)^2, \\
E^{(s+q,+)}(m^+_0, \La; \ve) - E^{(s+q,-)}(m^+_0, \La; \ve) \\
> \min \Bigl( \frac{1}{8}[ - \partial_E \chi|_{\ve,E^{(s,-)}(m^+_0, \La; \ve)} + \partial_E \chi|_{\ve,E^{(s,+)}(m^+_0, \La; \ve)}], 2 \delta^{(s+q-1)}_0 \Bigr), \\
- \partial_E \chi|_{\ve,E^{(s+q,-)}(m^+_0, \La; \ve)}, \partial_E \chi|_{\ve,E^{(s+q,+)}(m^+_0, \La; \ve)} \\
> \min \bigl( \frac{1}{2^{14}} \bigl( E^{(s+q,+)}(m^+_0, \La; \ve) - E^{(s+q,-)}(m^+_0, \La; \ve) \bigr)^2, \frac{(\delta^{(s+q-1)}_0)^2}{2^{11}} \bigr), \\
|\chi(\ve,E)| \ge \frac{1}{16} \min_{+,-} (E - E^{(s,\pm)}(m^+_0, \La; \ve))^2 \quad \text{if $\min_{+,-} |E - E^{(s+q,\pm)}(m^+_0, \La; \ve)| < \frac{(\delta^{(s+q-1)}_0)^2}{2^{12}}$}, \\
E^{(s+q,+)}(m^+_0, \La; \ve) - v(m_0^-) - Q^{(s+q)}(m^-_0, \La; \ve, E^{(s+q,+)}(m^+_0, \La; \ve) ) \\
\ge \max (\tau^\zero/2, |G^{(s+q)}(m^+_0, m^-_0, \La; \ve, E^{(s+q,+)}(m^+_0, \La; \ve)|) \\
E^{(s+q,-)}(m^+_0, \La; \ve) - v(m_0^+) - Q^{(s+q)}(m^+_0, \La; \ve, E^{(s+q,-)}(m^+_0, \La; \ve) ) \\
\le - \max (\tau^\zero/2, |G^{(s+q)}(m^+_0, m^-_0, \La; \ve, E^{(s+q,-)}(m^+_0, \La; \ve)|), \\
[ a_1 (\ve, E) + |b (\ve, E)| ]|_{E = E^{(s+q,+)} (m^+_0, \La; \ve)} \ge E^{(s+q,+)} (m^+_0, \La; \ve) \\
\ge \max ( a_1 (\ve, E),  a_2 (\ve, E) + |b (\ve, E)| )|_{E = E^{(s+q,+)} (m^+_0, \La; \ve)}, \\
[ a_2 (\ve, E) - |b (\ve, E)| ]|_{E = E^{(s+q,+)} (m^+_0, \La; \ve)}\le E^{(s+q,-)} (m^+_0, \La; \ve) \\
\le \min ( a_2 (\ve, E),  a_1 (\ve, E) - |b (\ve, E)| )|_{E = E^{(s+q,-)} (m^+_0, \La; \ve)}.
\end{split}
\end{equation}
where
\begin{equation}
\begin{split}\nn
a_1 (\ve, E) = v(m_0^+) + Q^{(s+q)} (m^+_0, \La; \ve, E), \quad\quad a_2(\ve, E) = v(m_0^-) + Q^{(s+q)}(m_0^-, \La ; \ve, E), \\
b (\ve, E) = |b_1 (\ve,E)|, \quad b_1 (\ve, E) = G^{(s+q)} (m^+_0, m^-_0, \La; \ve, E).
\end{split}
\end{equation}

\item[(6)]
\begin{equation} \label{eq:5specHEEAAA}
\begin{split}
\spec H_{\La, \ve} \cap \{ E : \min_\pm |E - E^{(s+q-1,\pm)}(m^+_0,\La^\esone(m^+_0); \ve)| < 8 (\delta^{(s+q-1)}_0)^{1/4} \} \\
= \{ E^{(s+q,+)}(m_0^+, \La; \ve), E^{(s+q, -)}(m_0^+, \La; \ve) \}, \\
E^{(s+q,\pm)}(m_0^+, \La; 0) = v(m^\pm_0).
\end{split}
\end{equation}

Let
\begin{equation}\label{eq:5Esplitspecconddomainq}
(\delta^{(s+q)}_0)^4 < \min_\pm |E - E^{(s+q-1,\pm)} \bigl( m^+_0, \La^{(s+q-1)}(m^+_0); \ve \bigr)| < (\delta_0^{(s+q-1)})^{1/2}, \quad E \in \IR.
\end{equation}
Then the matrix $(E - H_{\La,\ve})$ is invertible. Moreover, with $D(x;\La)$ as in part $(1)$,
\begin{equation}\label{eq:5inverseestiMATEq}
|[(E - H_{\La,\ve})^{-1}] (x,y)| \le S_{D(\cdot; \La), T, \kappa_0, |\ve|; k, \La, \mathfrak{R}} (x,y).
\end{equation}

\item[(7)] Set
\begin{equation}\label{eq:5Hevectors1PQPM}
\begin{split}
\beta^\pm = \frac{G^{(s+q)}(m^\mp_0, m^\pm_0, \La; \ve, E^{(s+q,\pm)}(m^+_0, \La; \ve))}{E^{(s+q,\pm)}(m^+_0, \La; \ve) - v(m_0^\mp) - Q^{(s+q)}(m^\mp_0, \La; \ve, E^{(s+q,\pm)}(m^+_0, \La; \ve))}, \\
\varphi^{(s+q,\pm)}(n, \La; \ve) = - \sum_{x \in \La \setminus \{m_0^+, m^-_0\}} (E^{(s+q,\pm)}(m^+_0, \La; \ve) - H_{\La \setminus \{ m_0^+, m^-_0 \}})^{-1} (n,x) \times \\
[h(x, m^\pm_0; \ve) + h(x, m^\mp_0; \ve) \beta^\pm], \quad n \notin \{ m_0^+, m_0^- \}, \\
\vp^{(s+q,\pm)}(m^\pm_0, \La; \ve) = 1, \quad \vp^{(s+q,\pm)}(m^\mp_0, \La; \ve) = \beta^\pm.
\end{split}
\end{equation}

\end{itemize}
Then the vector $\vp^{(s+q,\pm)}(\La; \ve) := (\vp^{(s+q,\pm)}(n, \La; \ve))_{n \in \La}$ is well-defined and obeys $\hle \vp^{(s+q,\pm)}(\La; \ve) = E^{(s+q,\pm)}(m^+_0, \La; \ve) \vp^{(s+q,\pm)}(\La; \ve)$,
\begin{equation}\label{eq:5evdecaY}
|\vp^{(s+q,\pm)}(n, \La; \ve)| \le |\ve|^{1/3} \Big[ \exp \Big( -\frac{7\kappa_0}{8} |n - m^+_0|) + \exp(-\frac{7\kappa_0}{8} |n - m^-_0| \Big) \Big], \quad n \notin \{m^+_0,m^-_0\},
\end{equation}
$|\vp^{(s+q,\pm)}(m^\mp_0,\La; \ve)| \le 1$.
\end{prop}

\begin{proof}
The proof of all statements goes simultaneously by induction over $q = 0, 1, \dots$. For $q = 0$, all statements except $(7)$ are due to Proposition~\ref{prop:5-4I}, Lemma~\ref{lem:5schur1}, and Proposition~\ref{prop:8-5n}. We discuss $(7)$ for $q \ge 1$; the derivation for $q = 0$ is completely similar. Let $q \ge 1$. Assume that the statements hold for any $q' \le q-1$ in the role of $q$. The derivation of $(1)$--$(4)$ is completely similar to the derivation of these properties in Proposition~\ref{prop:5-4I}, Lemma~\ref{lem:5schur1}, and Proposition~\ref{prop:8-5n}. We discuss the proof of these statements very briefly. A very important difference in $(5)$ is that this time we invoke Lemma~\ref{lem:6-1ellM} instead of Lemma~\ref{lem:6-1}.

Note first of all the following. Let $(\ve,E) \in \cL^{(s+q-1,+)}$. Let $j \in J^{(s+q-1)} \setminus \{0\}$ be arbitrary. Then, using conditions \eqref{eq:5EVsplitdefs1a}, \eqref{eq:5EVsplitdefAs1b} in Definition~\ref{def:7-6}, one obtains
\begin{equation}\label{eq:3EdomainstimateAAAAOPsq}
\begin{split}
|E^{(s+q-1,+)}(m^+_j, \La^{(s+q-1)}(m); \ve) - E| \le |E^{(s+q-1,+)}(m^+_j, \La^{(s+q-1)}(m); \ve) - E^{(s+q-1,+)}(m^+_0, \La^{(s+q-1)}(m^+_0); \ve)| \\
+ |E^{(s+q-1,+)}(m^+_0, \La^{(s+q-1)}(m^+_0); \ve) - E| < \delta^{(s+q-2)} + 2 \delta^{(s+q-1)} < 3 \delta^{(s+q-2)}/2, \\
|E^{(s+q-1,+)}(m^+_j, \La^{(s+q-1)}(m^+_j); \ve) - E| \ge |E^{(s+q-1,+)}(m^+_j, \La^{(s+q-1)}(m^+_j); \ve) - E^{(s+q-1,+)}(m^+_0, \La^{(s+q-1)}(m^+_0); \ve)|\\
- |E^{(s+q-1,+)}(m^+_0, \La^{(s+q-1)}(m^+_0); \ve) - E| \ge 3 \delta^{(s+q-1)} - 2 \delta^{(s+q-1)} > (\delta^{(s+q-1)})^4, \\
|E^{(s+q-1,-)}(m^+_j, \La^{(s+q-1)}(m^+_j); \ve) - E| \ge |E^{(s+q-1,-)}(m^+_j, \La^{(s+q-1)}(m^+_j); \ve) - E^{(s+q-1,+)}(m^+_0, \La^{(s+q-1)}(m^+_j); \ve) | \\
- |E^{(s+q-1,+)}(m^+_0, \La^{(s+q-1)}(m^+_j); \ve) - E| \ge 3 \delta^{(s+q-1)}-2 \delta^{(s+q-1)} > (\delta^{(s+q-1)})^4.
\end{split}
\end{equation}
Let $s \le s' \le s+q-2$, $j \in J^{(s')}$ be arbitrary. Then, using conditions \eqref{eq:5EVsplitdefsq}, \eqref{eq:5EVsplitdefA} in Definition~\ref{def:7-6}, one obtains
\begin{equation}\label{eq:3EdomainstimateAAAAOP}
\begin{split}
|E^{(s',+)}(m^+_j, \La^{(s')}(m); \ve) - E| \le |E^{(s',+)}(m^+_j, \La^{(s')}(m); \ve) - E^{(s+q-1,+)}(m^+_0, \La^{(s+q-1)}(m^+_0); \ve)| \\
+ |E^{(s+q-1,+)}(m^+_0, \La^{(s+q-1)}(m^+_0); \ve) - E| \le \delta^{(s'-1)} + 2 \delta^{(s+q-1)} < 3 \delta^{(s'-1)}/2, \\
|E^{(s',+)}(m^+_j, \La^{(s')}(m); \ve) - E| \ge |E^{(s',+)}(m^+_j, \La^{(s')}(m); \ve) - E^{(s+q-1,+)}(m^+_0, \La^{(s+q-1)}(m^+_0); \ve)| \\
- |E^{(s+q-1,+)}(m^+_0, \La^{(s+q-1)}(m^+_0); \ve) - E| \ge \frac{\delta^{(s')}}{2} - 2 \delta^{(s+q-1)} > (\delta^{(s')})^4, \\
|E^{(s',-)}(m^+_j, \La^{(s')}(m); \ve) - E| \ge |E^{(s',-)}(m^+_j, \La^{(s')}(m); \ve) - E^{(s+q-1,+)}(m^+_0, \La^{(s+q-1)}(m^+_0); \ve) | \\
- |E^{(s+q-1,+)}(m^+_0, \La^{(s+q-1)}(m^+_0); \ve) - E| \ge \frac{\delta^{(s')}}{2} - 2 \delta^{(s+q-1)} > (\delta^{(s')})^4.
\end{split}
\end{equation}

Similar estimates hold if $(\ve,E) \in \cL^{(s+q-1,-)}$. \emph{For this reason, the inductive assumption applies to $H_{\La^{(s')}(m^+_j),\ve}$ in the role of $\hle$ and to $(\ve,E)$ so that $(1)$--$(6)$ of the current proposition hold for $H_{\La^{(s')}(m^+_j),\ve}$}. In particular, for any $(\ve,E) \in \cL^{(s+q-1,+)} \cup \cL^{(s+q-1,-)}$, each matrix $(E - H_{\La^{(s')}(m),\ve})$ is invertible for any $s \le s' \le s+q-2$ and any $m$, $m \notin \{m_0^+, m_0^-\}$. The matrix $(E - H_{\La^{(s+q-1)}(m_0^+) \setminus \{ m_0^+, m_0^- \},\ve})$ is also invertible. Furthermore,
\begin{equation}\label{eq:3Hinvestimatestatement1NNNN}
\begin{split}
|[(E - H_{\La^{(s')}(m),\ve})^{-1}] (x,y)| \le S_{D(\cdot; \La^{(s')}(m)), T, \kappa_0, |\ve|; \La^{(s')}(m), \mathfrak{R}} (x,y), \\
|[(E - H_{\La^{(s+q-1)}(m_0^+) \setminus \{ m_0^+, m_0^- \}, \ve})^{-1}] (x,y)| \\
\le s_{D(\cdot; \La^{(s+q-1)}(m_0^+) \setminus \{ m_0^+, m_0^- \}), T, \kappa_0, |\ve|; \La^{(s+q-1)}(m_0^+) \setminus \{ m_0^+, m_0^- \}, \mathfrak{R}} (x,y).
\end{split}
\end{equation}
Similar estimates can be shown for $s' < s$ and for $s' = s+q-1$.

Recall also that $|v(n) - v(m_0^+)| \ge 2 \delta_0^4$ for any
$$
n \in \Lambda \setminus \bigl( \big[ \bigcup_{1 \le s' \le s+q-1} \bigcup_{m \in \cM{(s')}} \La^{(s')}(m) \big] \cup \big[ \bigcup_{s \le s' \le s+q-1} \bigcup_{j \in J^{(s')}} \La^{(s')}(m^+_j) \big] \bigr),
$$
due to condition (vi) in Definition~\ref{def:7-6}. This implies $|E - v(n)| \ge \delta_0^4$ for any such $n$.

Taking into account condition (iv) in Definition~\ref{def:7-6} and Remark~\ref{rem:3.Rs}, one obtains $D(m_0^\pm) \le T \mu_\La (m_0^\pm)^{1/3}$. Just as in the proof of Proposition~\ref{prop:8-5n}, one concludes that $D(\cdot; \La) \in \mathcal{G}_{\La, T, \kappa_0}$. Furthermore, due to Proposition~\ref{prop:aux1}, $\cH_{\La_{m_0^+, m_0^-}} := E - H_{\La_{m_0^+, m_0^-}, \ve}$ is invertible. Moreover,
\begin{equation}\label{eq:3Hinvestimate-2}
|\cH_{\La_{m_0^+, m_0^-}}^{-1} (x,y)| \le s_{D(\cdot; \La_{m_0^+, m_0^-}), T, \kappa_0, |\ve|; \La_{m_0^+, m_0^-}, \mathfrak{R}} (x,y).
\end{equation}
Thus, in particular, parts $(1)$, $(2)$ of the current proposition hold.

The estimates in \eqref{eq:5-12acACCOPQ} are due to Lemma~\ref{lem:2Qfunctionderiv}. This finishes the verification of $(1)$--$(4)$.

As we have mentioned, to verify $(5)$ we invoke Lemma~\ref{lem:6-1ellM}. For $(\ve,E) \in \cL^{(s+q-1,+)} \cup \cL^{(s+q-1,\pm)}$, set
\begin{equation}\begin{split}\label{eq:5gpm2}
a_1 (\ve,E) = v(m_0^+) + Q^{(s+q)}(m^+_0, \La; \ve, E), \; a_2(\ve,E) = v(m_0^-) + Q^{(s+q)}(m_0^-,\La; \ve, E), \\
b_1 (\ve,E) = G^{(s+q)}(m^+_0, m^-_0, \La; \ve, E), \; b(\ve,E) = |b_1(\ve,E)|, \\
f_i = E - a_i, \quad f = f_1 - b^2 f_2^{-1}, \\
g_{0,\pm}(\ve) = E^{(s+q-1,\pm)}(m^+_0, \La^{(s+q-1)}(m_0+); \ve), \quad \rho_0 = 2 \delta^{(s+q-1)}_0, \\
\tilde f_1 = E - v(m_0^+) - Q^{(s+q-1)}(m_0^+, \La^{(s+q-1)}(m_0^+); \ve, E), \\
\tilde f_2 = E - v(m_0^-) - Q^{(s+q-1)}(m_0^-, \La^{(s+q-1)}(m_0^+); \ve, E), \\
\tilde b^2 = |G^{(s+q-1)}(m_0^+, m_0^-, \La^{(s+q-1)}(m_0^+); \ve, E)|^2, \quad \chi_1(\ve,E) = \tilde f_1 \tilde f_2 - \tilde b^2.
\end{split}
\end{equation}
Due to \eqref{eq:5-12acACCOPQ}, one has $|\partial_E a_j|, |\partial^2_{E} a_j|, |b_1|, |\partial_{E} b_1|, |\partial^2_{E} b_1|, |E-a_j| < 1/64$. So, $f \in \mathfrak{F}^{(1)}_{\mathfrak{g}^{(1)}_-}(f_1, f_2, b)$, and also $f \in \mathfrak{F}^{(1)}_{\mathfrak{g}^{(1)}_+}(f_1, f_2, b)$, as required in Lemma~\ref{lem:6-1ellM}. Note that $\chi = \chi^{(f)}$. Let us verify conditions $(\alpha)$--$(\delta)$ needed for an application of Lemma~\ref{lem:6-1ellM}. We set $\rho := \rho_0$. Using \eqref{eq:5-12acACCOPQ}, one obtains $|\partial^\alpha_E \chi(\ve,E) - \partial^\alpha_E \chi_1(\ve,E)| \le 4 |\ve| (\delta_0^{(s+q-1)})^{12}$, $\alpha \le 2$. Recall that $\chi_1(\ve, E^{(s+q-1,\pm)}(m^\pm_0, \La^{(s+q-1)}(m^+_0); \ve)) = 0$, and \eqref{eq:8-13acnq0A} applies to $\chi_1$ in the role of $\chi$ and $q-1$ in the role of $q$. This implies conditions $(\alpha)$, $(\gamma)$, and $(\delta)$ with $\sigma_1 := (1/8) (\inf_{x,u} (\min_{i} \tau^{(f_{i})}))^4 = 1/8$. Note that $\chi_1(0,E) = \chi(0,E)$. This implies condition $(\beta)$. Part $(5)$ follows straight from  Lemma~\ref{lem:6-1ellM} $($the last two lines in \eqref{eq:8-13acnq0A} are due to \eqref{eq:6-1''}, \eqref{eq:6-1'''}, respectively$)$.

$(6)$ The proof of this part goes word for word as the proof of part $(3)$ of Proposition~\ref{prop:8-5n}.

$(7)$ Let $0 < | E - E^{(s,\pm)} \bigl( m^+_0, \La; \ve \bigr) | < (\delta_0^{(s+q-1)})^{1/2}$. We invoke the Schur complement formula \eqref{eq:2schurfor1a} with $\cH_\La = E - H_{\La,\ve}$, $\La_2 := \{m_0^+, m_0^-\}$, $\La_1 := \La \setminus \La_2$. Provided $\chi(\ve,E) \neq 0$, one has (see \eqref{eq:5schurfor1a})
$$
\cH_\Lambda = \begin{bmatrix} \cH_{\Lambda_1} & \Gamma_{1,2}\\[5pt] \Gamma_{2,1} & \cH_{\Lambda_2}\end{bmatrix},
$$
\begin{equation}\label{eq:3Hinvestimatest1PQPM-2}
\tilde H_2^{-1} = \frac{1}{\chi(\ve,E)} \begin{bmatrix} (E - v(m_0^-) - Q^{(s+q)}(m^-_0, \La; \ve, E)) & - G^{(s+q)}(m^+_0, m^-_0, \La; \ve, E) \\[5pt] - G^{(s+q)}(m^-_0, m^+_0, \La; \ve, E) & (E - v(m_0^+) - Q^{(s+q)}(m^+_0, \La; \ve, E)) \end{bmatrix},
\end{equation}
\begin{equation}\label{eq:3Hinvestimatest1PQPM-3}
\begin{split}
(E - H_{\La,\ve})^{-1} (n, m^\pm_0) = - [\cH_1^{-1} \Gamma_{1,2} \tilde H_2^{-1}] (n, m^\pm_0) = - \frac{1}{\chi(\ve,E)} \sum_{x \in \La \setminus \{m_0^+, m^-_0\}} \\
(E - H_{\La \setminus \{ m_0^+, m^-_0 \},\ve })^{-1} (n,x) [h(x, m^\pm_0; \ve) (E - v(m_0^\mp) - Q^{(s+q)}(m^\mp_0, \La; \ve, E)) + \\
h(x, m^\mp_0; \ve) G^{(s+q)}(m^\mp_0, m^\pm_0, \La; \ve, E)], \quad n \notin \{ m_0^+, m_0^-\} \\
(E - H_{\La,\ve})^{-1} (m^\pm_0, m^\pm_0) = \tilde H_2^{-1} (m^\pm_0, m^\pm_0) = \frac{1}{\chi(\ve,E)} (E - v(m_0^\mp) - Q^{(s+q)}(m^\mp_0, \La; \ve, E)), \\
(E - H_{\La,\ve})^{-1} (m^\pm_0, m^\mp_0) = \tilde H_2^{-1}(m^\pm_0, m^\mp_0) = \frac{1}{\chi(\ve,E)}
 G^{(s+q)}(m^+_0, m^-_0, \La; \ve, E).
\end{split}
\end{equation}
Note also that due to part (7) of Lemma~\ref{4.fcontinuedfrac}, $\partial_E \chi |_{E^{(s+q,\pm)}(m^+_0, \La; \ve)} \not= 0$. Set
\begin{equation}\label{eq:5muPM}
\alpha^\pm := \frac{E^{(s+q,\pm)}(m^+_0, \La; \ve) - v(m_0^\mp) - Q^{(s+q)}(m^\mp_0, \La; \ve, E^{(s+q,\pm)} (m^+_0, \La; \ve))}{\partial_E \chi |_{E^{(s+q,\pm)}(m^+_0, \La; \ve)}}.
\end{equation}
It follows from  \eqref{eq:8-13acnq0A} that $\alpha^\pm \neq 0$, $|\beta^\pm| \le 1$. One has
\begin{equation}\label{eq:5Hinvestimatest1PQPM}
\begin{split}
Res[(E-H_{\La,\ve})^{-1} (n, m^\pm_0)]|_{E = E^{(s+q,\pm)}(m^+_0, \La; \ve)} = - \alpha^\pm \sum_{x \in \La \setminus \{ m_0^+, m^-_0 \}} \\
(E^{(s+q,\pm)}(m^+_0, \La; \ve) - H_{\La \setminus \{ m_0^+, m^-_0 \},\ve })^{-1} (n,x) [h(x, m^\pm_0; \ve) + h(x, m^\mp_0; \ve) \beta^\pm], \quad n \in \La \setminus \{m_0^+, m^-_0\}, \\
Res[(E - H_{\La,\ve})^{-1}(m^\pm_0, m^\pm_0)]|_{E = E^{(s+q,\pm)}(m^+_0, \La; \ve)} = \alpha^\pm, \\
Res[(E - H_{\La,\ve})^{-1}(m^\pm_0, m^\mp_0)]|_{E = E^{(s+q,\pm)}(m^+_0, \La; \ve)} = \alpha^\pm\beta^\pm, \\
Res[(E - H_{\La,\ve})^{-1}\delta_{m^\pm_0}]|_{E = E^{(s+q,\pm)}(m^+_0, \La; \ve)} = \alpha^\pm \vp^{(s+q)}(\cdot, \La; \ve).
\end{split}
\end{equation}
This implies $\hle \vp^{(s+q,\pm)}(\La; \ve) = E^{(s+q,\pm)}(m^+_0, \La; \ve) \vp^{(s+q,\pm)}(\La; \ve)$. Combining \eqref{eq:5Hevectors1PQPM} with \eqref{eq:3Hinvestimatestatement1PQ} and with the estimate \eqref{eq:auxtrajectweightsumest8} from Lemma~\ref{lem:auxweight1}, one obtains \eqref{eq:5evdecaY}.
\end{proof}

Using the notation from Proposition~\ref{rem:con1smalldenomnn}, assume that the functions $h(m, n,\ve)$, $m, n \in \La$ depend also on some parameter $k \in (k_1, k_2)$, that is, $h(m, n; \ve) = h(m, n; \ve, k)$. Assume that $H_{\La, \varepsilon, k} := \bigl( h(m, n; \ve, k) \bigr)_{m, n \in \La}\in OPR^{(s,s+q)}\bigl(m^+_0, m^-_0, \La; \delta_0, \tau^\zero \bigr)$. Let $Q^{(s+q)}(m_0^\pm, \La; \ve, k, E)$, $G^{(s+q)}(m^\pm_0, m^\mp_0, \La; \ve, k, E)$, $E^{(s+q,\pm)}(m_0^+, \La; \ve, k)$ be the functions introduced in Proposition~\ref{rem:con1smalldenomnn} with $H_{\La,\varepsilon,k}$ in the role of $\hle$.

\begin{lemma}\label{lem:6differentiation}
$(1)$ If $h(m, n; \ve, k)$ are $C^t$-smooth functions of $k$, then $Q^{(s+q)}(m_0^\pm, \La; \ve, k, E)$, $G^{(s+q)}(m^\pm_0, m^\mp_0, \La; \ve, k, E)$, and $E^{(s+q,\pm)}(m_0^+, \La; \ve,k)$ are $C^t$-smooth functions of all variables involved.

$(2)$ Assume also that $h(m,n;\ve,k)$ are $C^2$-smooth functions and for $m \neq n$ obey $|\partial^\alpha h(m,n;\ve,k)| \le B_0 \exp(-\kappa_0 |m - n|)$ for $|\alpha| \le 2$. Furthermore, assume that $|\partial^\alpha h(m,m;\ve,k)| \le B_0 \exp(\kappa_0 |m - m_0^+|^{1/5})$ for any $m \in \La$, $0< |\alpha| \le 2$. Then, for $|\alpha|\le 2$, we have
\begin{equation}\label{eq:6Hinvestimatestatement1kvard}
\begin{split}
|\partial^\alpha (E - H_{\La \setminus \{ m_0^+, m_0^- \}, k})^{-1}] (x,y)| \le (3 B_0)^\alpha \mathfrak{D}^\alpha_{D(\cdot; \La \setminus \{ m_0^\pm, m^-_\mp \}), T, \kappa_0, |\ve|; \La \setminus \{ m_0^+, m^-_0 \} } (x,y), \\
|\partial^\alpha Q^{(s+q)}(m_0^\pm, \La; \ve, k, E)| \le (3 B_0)^\alpha |\ve| \mathfrak{D}^\alpha_{D(\cdot; \La \setminus \{ m_0^+, m^-_0 \} ), T, \kappa_0, |\ve|; \La \setminus \{ m_0^+, m^-_0 \} } (m_0^\pm, m_0^\pm) < (3 B_0)^\alpha |\ve|^{3/2}, \\
|\partial^\alpha G^{(s+q)}(m_0^\pm, m^\mp_0, \La; \ve, k, E)| \le (3 B_0)^\alpha \mathfrak{D}^\alpha_{D(\cdot; \La \setminus \{ m_0^\pm, m^-_\mp \}), T, \kappa_0, |\ve|; \La \setminus \{ m_0^+, m^-_0 \} } (m_0^\pm, m_0^\mp) \\
<(3B_0)^\alpha|\ve|^{1/2}\exp(-\kappa_0|m^+_0-m^-_0/16|),
\end{split}
\end{equation}
\begin{equation}\label{eq:6Hinvestimatestatement1kvardE}
|\partial^\alpha E^{(s+q,\pm)} (m_0^+, \La; \ve, k, E) - \partial^\alpha v(m_0,k)| < (3 B_0)^\alpha |\ve|^{3/2}.
\end{equation}
Here, $\mathfrak{D}^\alpha_{D(\cdot; \La \setminus \{x,y\}), T, \kappa_0, |\ve|; \La \setminus \{ m_0^+, m^-_0 \} } (m_0^\pm, m^\mp_0)$ is defined as in Lemma~\ref{lem:auxweight1iterated}.
\end{lemma}

The proof of this statement is completely similar to the proof of Lemma~\ref{lem:5differentiation} and we skip it.

\section{Self-Adjoint Matrices with a Graded System of Ordered Pairs of Resonances}\label{sec.6}

This section is to large extent an ``upgrade'' of Section~\ref{sec.5}. We explain only the new ingredients. We skip the rest of the proofs because up to the new notation, they go almost word for word as the proofs of the corresponding statements in Section~\ref{sec.5}. Lemma~\ref{lem:6-9.1} explains how the main transition to the ``upgraded'' case goes. After that, ultimately the main difference is that we use Lemma~\ref{lem:6-1ell} instead of Lemma~\ref{lem:6-1}, and in Lemma~\ref{lem:6-1ellM}, we have this time $\ell > 1$.

\begin{defi}\label{def:9-6}
Let $s > 0$, $q > 0$ be integers and let $\tau^\zero > (\delta^{(s+q-1)}_0)^{1/4}$. Using the notation from Definition~\ref{def:7-6}, assume that conditions $(i)$--$(iv)$ and $(vi)$ of Definition~\ref{def:7-6} hold. Assume also that there exists $m^+_{j_0} \in \cM^{(s+q-1,+)}$ such that
\begin{align}
(\delta^{(s+q-1)}_0)^{1/2} \le |E^{(s+q-1,\pm)}\bigl(m^+_j, \La^{(s+q-1)}(m^+_j); \ve \bigr) - E^{(s+q-1,\pm)}\bigl(m^+_0, \La^{(s+q-1)}(m^+_0); \ve\bigr)| \le \delta^{(s+q-2)}_0 \label{eq:6EVsplitdefs1a}, \quad j \neq j_0, \\
(\delta^{(s+q-1)}_0)^{1/2} \le |E^{(s+q-1,-)} \bigl(m^+_j, \La^{(s+q-1)}(m); \ve\bigr) - E^{(s+q-1,\pm)}\bigl(m_0^+, \La^{(s+q-1)}(m_0^+); \ve \bigr)| \label{eq:6EVsplitdefAs1b}, \quad j \neq j_0, \\
|E^{(s+q-1,\pm)}\bigl(m^+_{j_0}, \La^{(s+q-1)}(m^+_{j_0}); \ve \bigr) - E^{(s+q-1,\pm)}\bigl(m^+_0, \La^{(s+q-1)}(m^+_0); \ve \bigr)| \le (\delta_0^{(s+q-1)})^{5/8}. \label{eq:6-9EVsplitdefs1a}
\end{align}
Assume that the rest of condition $(v)$ of Definition~\ref{def:7-6} holds.

Due to Proposition~\ref{prop:6-4} below, the functions
\begin{equation}\label{eq:9-10acOPQDEF}
\begin{split}
K^{(s+q,\pm)}(m,n, \La; \ve, E) = (E - H_{\La_{m_0^\pm,m_{j_0}^\pm}})^{-1} (m,n), \quad m,n \in \La_{m_0^\pm,m_{j_0}^\pm} := \Lambda \setminus \{m_0^\pm,m_{j_0}^\pm\}, \\
Q^{(s+q,\pm)}(m,\La; \ve, E) = \sum_{m',n' \in \La_{m^\pm_0,m_{j_0}^\pm}} h(m,m';\ve) K^{(s+q,\pm)} (m',n'; \La; \ve, E) h(n',m;\ve), \quad m \in \{m_0^+,m^+_{j_0}\}
\end{split}
\end{equation}
are well-defined for any $\ve \in (-\ve_{s+q-2}, \ve_{s+q-2})$ and any
\begin{equation} \nn
\begin{split}
E \in \mathcal{E}^\pm(\ve) := (E^{(s+q-1,\pm)} \bigl(m^+_0, \La^{(s+q-1)} (m^+_0);\ve) - 2 \delta^{(s+q-1)}_0,E^{(s+q-1,\pm)} \bigl(m^+_0, \La^{(s+q-1)} (m^+_0);\ve) + 2 \delta^{(s+q-1)}_0).
\end{split}
\end{equation}
We require that for all $\ve$, we have
\begin{equation} \label{eq:6-9-13OPQ}
\begin{split}
v(m_0^+) + Q^{(s+q)}(m^+_0, \La,E) \ge v(m_{j_0}^+)+Q^{(s+q)}(m_{j_0}^+,\La; \ve, E) + \tau^\one, \quad E \in \mathcal{E}^+(\ve),
\end{split}
\end{equation}
or, alternatively,
\begin{equation} \label{eq:6-9-13OPQminus}
\begin{split}
v(m_0^-) + Q^{(s+q)}(m^-_0, \La,E) \ge v(m_{j_0}^-) + Q^{(s+q)}(m_{j_0}^-,\La; \ve, E) + \tau^\one, \quad E \in \mathcal{E}^-(\ve),
\end{split}
\end{equation}
where $\tau^\one > 0$. We introduce the following notation:
\begin{equation} \label{eq:6-9-13PAKA}
\mathfrak{m}^\one := \mathfrak{m} := ((m^+_0,m^-_0), (m^+_{j_0},m^-_{j_0})), \quad \quad \quad s^\zero := s\quad, s^\one = s+q, \quad \mathfrak{s}^\one := \mathfrak{s} := (s^\zero,s^\one), \quad \tau = (\tau^\zero,\tau^\one).
\end{equation}
With some abuse of notation, we set
\begin{equation} \label{eq:6-9-13PAKAZ}
\begin{split}
m^+ := m^+_0, \quad m^- := m^+_{j_0} \quad \text{if \eqref{eq:6-9-13OPQ} holds}, \\
\text {or alternatively} \\
m^+ := m^-_0, \quad m^- := m^-_{j_0} \quad \text{if \eqref{eq:6-9-13OPQminus} holds.}
\end{split}
\end{equation}
We say that $\hle \in GSR^{[\mathfrak{s}]}\bigl( \mathfrak{m},m^+,m^-, \La; \delta_0,\tau\bigr)$ . We set $s(m^\pm) = s^\one$. We call $m^+,m^-$ the principal points. We call $\La^{(s^\one-1)}(m^\pm)$ the $(s^\one-1)$-set for $m^\pm$.
\end{defi}

\begin{remark}\label{rem:6plusminusclasses}
We introduce the cases \eqref{eq:6-9-13OPQ} and \eqref{eq:6-9-13OPQminus} to address all possible cases for our applications in Section~\ref{sec.10}. In fact, in Section~\ref{sec.10} these cases exclude each other. This property is inessential for the development in the current section. For this reason, we do not include it in the definitions, and we consider these two cases as two alternatives.
\end{remark}

\begin{remark}\label{rem:6barLanobarLa}
\textbf{In the next statement we will use for the first time part $(5)$ of Definition~\ref{def:aux1} with }$\bar \La \neq \IZ^\nu$, see Remark~\ref{rem:2withRnoR}. Here and later in this work, $\mathcal{G}_{\La',T,\kappa_0} := \mathcal{G}_{\La', \IZ^\nu, T, \kappa_0}$. Recall also that $\mathcal{G}_{\La, \bar \La, T, \kappa_0} \subset \mathcal{G}_{\La,\bar\La_1, T, \kappa_0}$ if $\bar \La_1 \subset \bar \La$.
\end{remark}

\begin{lemma}\label{lem:6-9.1}
$(1)$ Let $H_{\La', \ve} \in OPR^{(s,s+q')} \bigl(m^+_0,m^-_0, \La'; \delta_0,\tau^\zero \bigr)$. Assume that $(\delta^{(s+q')}_0)^{1/4} < \tau^\zero$. Let $\cH_{\La'} := E - H_{\La',\ve}$, $\La_1 = \La' \setminus \{m^+_0,m^-_0\}$, $\La_2 = \{m^\mp_0\}$, $\cH_j := \cH_{\La_j}$, $\Gamma_{i,j}(k,\ell) := \Gamma_{\Lambda_i, \Lambda_j} (k, \ell) := \cH(k,\ell), \quad k \in \La_i, \ell \in \La_j$. For $|E - E^{(s+q',\pm)}(m_0^+, \La'; \ve)| < (\delta^{(s+q')}_0)^{1/4}/8$, the quantity $\tilde H^\pm_2 = \tilde H^\pm_2 (m_0^\mp,m_0^\mp) = [\cH_2 - \Gamma_{2,1} \cH_1^{-1} \Gamma_{1,2}](m_0^\mp,m_0^\mp)$ is well defined, $\tilde H^\pm_2 = E - v(m_0^\mp) - Q^{(s+q')}(m_0^\mp,\La'; \ve,E)$, and
\begin{equation}\label{eq:6aux00c01a1PSTATE}
\pm\tilde H^\pm_2(m_0^\mp,m^\mp_0)\ge (\delta^{(s+q')}_0)^{1/4}/4.
\end{equation}
Furthermore, set $D(x;\La'\setminus\{m^\pm_0\}) = D(x;\La')$, $x \in \La' \setminus \{m^+_0,m^-_0\}$, where $D(x;\La'\setminus\{m^+_0,m^-_0\})$ is defined as in Proposition~\ref{rem:con1smalldenomnn} with $\La'$ in the role of $\La$, $D(m^\mp_0;\La'\setminus \{m^\pm_0\}) = 4 \log (\delta^{(s+q'-1)}_0)^{-1}$. Then, $D(\cdot;\La'\setminus \{m^\pm_0\}) \in \mathcal{G}_{\La'\setminus \{m^\pm_0\},\IZ^\nu \setminus \{m^\pm_0\},T,\kappa_0}$. The matrix $E - H_{\Lambda'\setminus \{m^\pm_0\}}$ is invertible and
\begin{equation}\label{eq:6aux00c01a1PSTATEIII}
|[(E - H_{\La' \setminus \{m^\pm_0\},\ve})^{-1}](x,y)| \le s_{D(\cdot;\La \setminus \{m^\pm_0\}),T,\kappa_0,|\ve|;\La' \setminus \{m^\pm_0\},\mathfrak{R}}(x,y).
\end{equation}

$(2)$ Set
\begin{equation} \label{eq:6-10acOPQDEFQQ}
\hat Q^{(s+q')}(m_0^\pm,\La'; \ve, E) = \sum_{m,n\in \La'\setminus{m^\pm_0}} h(m_0^\pm,m;\ve) (E - H_{\La' \setminus \{m^\pm_0\},\ve})^{-1} (m,n) h(n,m_0^\pm;\ve).
\end{equation}
Then,
\begin{equation}\label{eq:6Qformula}
\hat Q^{(s+q')}(m_0^\pm,\La'; \ve, E) = Q^{(s+q')}(m_0^\pm,\La'; \ve, E) + \frac{|G^{(s+q')}(m_0^\pm,m^\mp_0,\La'; \ve, E)|^2}{E - v(m_0^\mp) - Q^{(s+q')}(m_0^\mp,\La'; \ve,E )}.
\end{equation}

$(3)$ Using the notation from part $(2)$, set $g_0(\ve,E) = g_{0,\pm}(\ve,E) = E^{(s+q'-1,\pm)}(m_0^+, \La'; \ve)$, $\rho_0 = 2 \delta^{(s+q'-1)}_0$,
\begin{equation}\label{eq:6Qfformula}
\begin{split}
f_\pm(\ve,E) = E - v(m_0^\pm) - Q^{(s+q')}(m_0^\pm,\La'; \ve,E ), \quad b^2(\ve,E) = |G^{(s+q')}(m_0^\pm,m^\mp_0,\La'; \ve, E)|^2, \\
f(\ve,E) = f_+(\ve,E) - \frac{b^2(\ve,E)}{f_-(\ve,E)}.
\end{split}
\end{equation}
Then, $f \in \mathfrak{F}^{(1)}_{\mathfrak{g}^\one,1/2}(f_+,f_-,b^2)$, $\tau^{(f)}\ge \tau_0$; see Definition~\ref{def:4a-functions}.
\end{lemma}

\begin{proof}
$(1)$ Let $|E - E^{(s+q',+)}(m_0^+, \La'; \ve)| < (\delta^{(s+q')}_0)^{1/4}/8$. Due to part $(2)$ of Proposition~\ref{rem:con1smalldenomnn}, $Q^{(s+q')}(m_0^-,\La'; \ve,E )$ is well-defined. One can see that $\tilde H^+_2 = E - v(m_0^-) - Q^{(s+q')}(m_0^-,\La'; \ve,E )$. Due to \eqref{eq:8-13acnq0A} from Proposition~\ref{rem:con1smalldenomnn}, one obtains
\begin{equation}\label{eq:6-9gpm2N}
\tilde H_2^+|_{E = E^{(s+q', +)}(m_0^+, \La'; \ve)} = E^{(s+q', +)}(m_0^+, \La'; \ve) - v(m_0^-) - Q^{(s+q')}(m_0^-,\La'; \ve,E^{(s+q', +)}(m_0^+, \La'; \ve) ) \ge \tau^\zero/2.
\end{equation}
Combining \eqref{eq:6-9gpm2N} with the estimates \eqref{eq:5-12acACCOPQ} from Proposition~\ref{rem:con1smalldenomnn} and taking into account that $(\delta^{(s+q')}_0)^{1/4} < \tau^\zero$, one obtains \eqref{eq:6aux00c01a1PSTATE} for $\tilde H^+_2$. The derivation for $\tilde H^-_2$ is completely similar. Furthermore, due to \eqref{eq:3Hinvestimatestatement1PQ} from Proposition~\ref{rem:con1smalldenomnn}, one has
\begin{equation}\label{eq:6aux00c01a1PSTATEAGA}
|[(E - H_{\La' \setminus \{m^+_0,m^-_0\},\ve})^{-1}](x,y)| \le s_{D(\cdot;\La \setminus \{m^+_0,m^-_0\}),T,\kappa_0,|\ve|;\La' \setminus \{m^+_0,m^-_0\},\mathfrak{R}}(x,y).
\end{equation}
Due to Proposition~\ref{prop:aux1}, the estimate \eqref{eq:6aux00c01a1PSTATEIII} follows from \eqref{eq:6aux00c01a1PSTATE} combined with \eqref{eq:6aux00c01a1PSTATEAGA}. The case $|E - E^{(s+q',-)}(m_0^+, \La'; \ve)| < (\delta^{(s+q')}_0)^{1/4}/8$ is completely similar.

$(2)$ One has
\begin{equation}\label{eq:6Qfoursplit}
\begin{split}
\hat Q^{(s+q')}(m_0^\pm,\La'\setminus \{m^\pm_0\}; \ve, E) = [\sum_{m,n \in \La_1} + \sum_{m,n \in \La_2} + \sum_{m \in \La_1, n \in \La_2} + \sum_{m \in \La_2, n \in \La_1}] \\
h(m_0^\pm,m;\ve) (E - H_{\La' \setminus \{m^\pm_0\},\ve})^{-1}(m,n) h(n,m_0^\pm;\ve) := \sum_{1 \le j \le 4} \mathfrak{S}_j.
\end{split}
\end{equation}
Using the Schur complement formula and the definition \eqref{eq:5-10acOPQ}, one obtains
\begin{equation}\label{eq:6applySchur1}
\begin{split}
\mathfrak{S}_1 = \sum_{m,n \in \La_1} h(m_0^\pm,m;\ve) [\cH_1^{-1} + \cH_1^{-1} \Gamma_{1,2} \tilde H_2^{-1} \Gamma_{2,1}\cH_1^{-1}](m,n)
h(n,m_0^\pm;\ve), \\
= Q^{(s+q')}(m_0^\pm,\La'; \ve, E) + (\tilde H_2(m^\mp_0,m^\mp_0))^{-1} |\mathfrak{M}_1|^2, \\
\mathfrak{M}_1 = \sum_{m,n \in \La_1} h(m_0^\pm,m;\ve) \cH_1^{-1}(m,n) h(n,m^\mp_0;\ve);
\end{split}
\end{equation}
see Lemma~\ref{lem:2schur}. Similarly,
\begin{equation}\label{eq:6applySchur2}
\begin{split}
\mathfrak{S}_2 = (\tilde H_2(m^\mp_0,m^\mp_0))^{-1} |h(m^\mp_0,m_0^\pm;\ve)|^2, \\
\mathfrak{S}_3 = (\tilde H_2(m^\mp_0,m^\mp_0))^{-1} \mathfrak{M}_1 h(m^\mp_0,m_0^\pm;\ve), \quad \mathfrak{S}_4 = (\tilde H_2(m^\mp_0,m^\mp_0))^{-1} \overline{\mathfrak{M}_1 h(m^\mp_0,m_0^\pm;\ve)}.
\end{split}
\end{equation}
Due to \eqref{eq:6applySchur1}, \eqref{eq:6applySchur2}, and the definition \eqref{eq:5-10acOPQ}, one has
\begin{equation}\label{eq:6QfoursplitBAK}
\sum_{1 \le j \le 4} \mathfrak{S}_j = Q^{(s+q')}(m_0^\pm,\La'; \ve, E) + (\tilde H_2(m^\mp_0,m^\mp_0))^{-1} |G^{(s+q')}(m_0^\pm, m^\mp_0, \La'; \ve, E)|^2,
\end{equation}
as claimed in \eqref{eq:6Qformula}.

$(3)$ Due to \eqref{eq:5-13OPQ} in Definition~\ref{def:7-6}, one has $f_+ > f_-$ for any $\ve, E$. It follows from the estimates \eqref{eq:5-12acACCOPQ} in Proposition~\ref{rem:con1smalldenomnn} that $f \in \mathfrak{F}^{(1)}_{\mathfrak{g}^\one,1/2}(f_+,f_-,b^2)$. Furthermore, due to Definition~\ref{def:7-6}, one has $\tau^{(f)} \ge \tau_0$.
\end{proof}

\begin{prop}\label{prop:6-4}
Using the notation from Definition~\ref{def:9-6}, the following statements hold.

$(1)$ Let $D(\cdot;\La(m))$ be as in Proposition~\ref{rem:con1smalldenomnn}. Set $D(x;\La) = D(x;\La\setminus\{m^+,m^-\}) = D(x;\La\setminus \mathfrak{m}) = D(x;\La(m))$ if $x \in \La(m) \setminus \mathfrak{m}$, $D(x;\La) = D(x;\La\setminus\{m^+,m^-\}) = 4 \log (\delta^{(s+q-1)}_0)^{-1}$ if $x \in \mathfrak{m} \setminus\{m^+,m^-\}$, and $D(x;\La) = 4 \log (\delta^{(s+q)}_0)^{-1}$ if $x \in \{m^+,m^-\}$. Then, $D(\cdot;\La\setminus \mathfrak{m}) \in \mathcal{G}_{\La\setminus \mathfrak{m},T,\kappa_0}$, $D(\cdot;\La \setminus \{m^+,m^-\}) \in \mathcal{G}_{\La \setminus \{m^+,m^-\}, \IZ^\nu \setminus \{m^+,m^-\},T,\kappa_0}$, $D(\cdot;\La) \in \mathcal{G}_{\La,T,\kappa_0}$.

$(2)$ Set $g_1 = E^{(s^\one-1,\pm)} \bigl(m^+, \La^{(s^\one-1)}(m^+); \ve \bigr)$ if \eqref{eq:6-9-13OPQ} or \eqref{eq:6-9-13OPQminus} holds, respectively. Set
\begin{equation}\label{eq:6-11domainnOPQU}
\cL^{(s^\one-1)} := \Big\{ (\ve,E) : \ve \in (-\ve_{s-1},\ve_{s-1}), |E - g_1(\ve)| < \frac{(\delta_0^{(s^\one-1)})^{1/2}}{2} \Big\}.
\end{equation}
For any $(\ve,E) \in \cL^{(s^\one-1)}$, the matrix $(E - H_{\La \setminus \{m^+,m^-\},\ve})$ is invertible and
\begin{equation}\label{eq:6-9.Hinvestimatestatement1PQ}
|[(E - H_{\La \setminus \{m^+,m^-\},\ve})^{-1}](x,y)| \le s_{D(\cdot;\La\setminus\{m^+,m^-\}),T,\kappa_0,|\ve|;\La\setminus
\{m^+,m^-\},\mathfrak{R}}(x,y).
\end{equation}

$(3)$ Set
\begin{equation} \label{eq:6-10acU}
\begin{split}
Q^{(s^\one)}(m^\pm,\La; \ve, E) = \sum_{m,n \in \La \setminus \{m^+,m^-\}} h(m^\pm,m;\ve) (E - H_{\La \setminus \{m^+,m^-\}})^{-1} (m,n) h(n,m^\pm;\ve), \\
G^{(s^\one)}(m^\pm,m^\mp, \La\setminus\{m^+,m^-\}; \ve, E) = h(m^\pm,m^\mp;\ve) \\
+ \sum_{m,n \in \La \setminus \{m^+,m^-\}} h(m^\pm,m;\ve) (E - H_{\La \setminus \{m^+,m^-\}})^{-1}(m,n) h(n,m^\mp;\ve).
\end{split}
\end{equation}
Let  $\hat Q^{(s^\one-1)} \bigl(m^\pm, \La^{(s^\one-1)}(m^\pm); \ve,E \bigr)$ be defined as in Lemma~\ref{lem:6-9.1} with $\La' = \La^{(s^\one-1)}(m^\pm)$. Then, for $\alpha \le 2$,
\begin{equation} \label{eq:6ACCOPQ}
\begin{split}
& \big| \partial^\alpha_E Q^{(s^\one)}(m^\pm, \La; \ve, E) - \partial^\alpha_E \hat Q^{(s^\one-1)}\bigl(m^\pm, \La^{(s^\one-1)}(m^\pm); \ve,E \bigr) \big | \\
& \le 4 |\ve|^{3/2} \exp \left( -\kappa_0 R^{(s^\one-1)} \right) \le |\ve| (\delta_0^{(s^\one-1)})^{12}, \\
& \big| \partial^\alpha_E G^{(s^\one)}(m^+, m^-, \La;\ve,E) \big| \le 4 |\ve|^{3/2} \exp \left( -\frac{7\kappa_0}{8} |m^+ - m^-| \right) \\
& \le 4 |\ve|^{3/2} \exp \left( -\kappa_0 R^{(s^\one-1)} \right) \le |\ve| (\delta_0^{(s^\one-1)})^{12}.
\end{split}
\end{equation}
Furthermore, set $\rho_0 = \delta^{(s^\one-2)}_0$, $\rho_1 = (\delta^{(s^\one-1)}_0)^{1/4}/8$, $g_0(\ve,E) = g_{0,\pm}(\ve,E) = E^{(s^\one-2,\pm)}(m^+, \La(m^+); \ve)$, $g_1(\ve,E) = g_{1,\pm}(\ve,E) = E^{(s^\one-1,\pm)}(m^+, \La(m^+); \ve)$ if \eqref{eq:6-9-13OPQ} or \eqref{eq:6-9-13OPQminus} holds, respectively, and
\begin{equation}\label{eq:6QfformulaUP}
\begin{split}
f_1(\ve,E) = E - v(m^+) - Q^{(s^\one)}(m^+,\La; \ve,E ), \quad f_2(\ve,E) = E - v(m^-) - Q^{(s^\one)}(m^-,\La; \ve,E ), \\
b^2(\ve,E) = |G^{(s^\one)}(m^\pm,m^\mp,\La; \ve, E)|^2, \quad f(\ve,E) = f_1(\ve,E) - \frac{b^2(\ve,E)}{f_2(\ve,E)}.
\end{split}
\end{equation}
Then, $f \in \mathfrak{F}^{(2,\pm)}_{\mathfrak{g}^{(2)},3/4}(f_1,f_2,b^2)$ if \eqref{eq:6-9-13OPQ} or \eqref{eq:6-9-13OPQminus} holds, respectively, $\tau^{(f_j)} > \tau^\zero/2$, $\tau^{(f)} \ge \tau^\one(\tau^\zero)^2/4$; see Definition~\ref{def:4a-functions} .

$(4)$ Let $(\ve,E) \in \cL^{(s^\one-1)}$. Then, $E \in \spec H_{\La,\ve}$ if and only if $E$ obeys
\begin{equation} \label{eq:6-13NNNNOPQU}
\begin{split}
& \chi(\ve,E) := \bigl( E - v(m^+) - Q^{(s^\one)}(m^+, \La; \ve, E) \bigr) \cdot \bigl( E - v(m^-) - Q^{(s^\one)}(m^-,\La; \ve, E) \bigr) \\
& \qquad - |G^{(s^\one)}(m^+, m^-, \La; \ve, E)|^2 = 0.
\end{split}
\end{equation}

$(5)$ Let $f$ be as in part $(3)$ and let $\chi^{(f)}$ be as in Definition~\ref{def:4a-functions}. Then, $\chi(\ve,E) = 0$ if and only if $\chi^{(f)} = 0$. For $\ve \in (-\ve_{s-1}, \ve_{s-1})$, the equation
\begin{equation}\label{eq:8-13nn-2}
\chi^{(f)}(\ve,E) = 0
\end{equation}
has exactly two solutions $E^{(s^\one, +)}(m^+, \La; \ve) > E^{(s^\one, -)}(m^+, \La; \ve)$,
\begin{equation} \label{eq.6Eestimates1APN1}
|E^{(s^\one, \pm)}(m^+, \La; \ve) - g_1| < 4(\delta^{s^\one-1}_0)^{1/8}.
\end{equation}
The functions $E^{(s^\one, \pm)}(m^+, \La; \ve)$ are smooth on the interval $(-\ve_{s-1}, \ve_{s-1})$. The following estimates hold:
\begin{equation} \label{eq:6-13acnq0N1}
\begin{split}
|\partial^\alpha_E \chi^{(f)}| \le 8 \quad \text{for $\alpha \le 2$}, \quad \partial^2_E \chi^{(f)} > 1/8, \\
\partial^\alpha_E \chi^{(f)}|_{\ve,E^{(s^\one,-)}(m^+, \La; \ve)} < -(\tau^{(f)})^2, \quad \partial^\alpha_E \chi^{(f)}|_{\ve,E^{(s^\one,+)}(m^+, \La; \ve)} > (\tau^{(f)})^2, \\
E^{(s^\one,+)}(m^+, \La; \ve) - E^{(s^\one,-)}(m^+, \La; \ve) > \frac{1}{8}[ -\partial^\alpha_E \chi^{(f)}|_{\ve,E^{(s^\one,-)}(m^+, \La; \ve)} + \partial^\alpha_E \chi^{(f)}|_{\ve,E^{(s^\one,+)}(m^+, \La; \ve)}], \\
-\partial^\alpha_E \chi^{(f)}|_{\ve,E^{(s^\one,-)}(m^+, \La; \ve)}, \partial^\alpha_E \chi^{(f)}|_{\ve,E^{(s^\one,+)}(m^+, \La; \ve)} >
\frac{1}{2^{14}} \bigl( E^{(s^\one,+)}(m^+, \La; \ve) - E^{(s^\one,-)}(m^+, \La; \ve) \bigr)^2, \\
|\chi^{(f)}(\ve,E)| \ge \frac{1}{8} \min \bigl( (E - E^{(s^\one,-)}(m^+, \La; \ve))^2, (E - E^{(s^\one,+)}(m^+, \La; \ve))^2 \bigr), \\
[a_1 (\ve, E) + |b (\ve, E)|]|_{E = E^{(s^\one,+)} (m^+, \La; \ve)} \ge E^{(s^\one,+)} (m^+, \La; \ve) \\
\ge \max ( a_1 (\ve, E),  a_2 (\ve, E) + |b (\ve, E)| )|_{E = E^{(s^\one,+)} (m^+, \La; \ve)}, \\
[a_2 (\ve, E) - |b (\ve, E)|]|_{E = E^{(s^\one,+)} (m^+, \La; \ve)} \le E^{(s^\one,-)} (m^+, \La; \ve) \\
\le \min ( a_2 (\ve, E),  a_1 (\ve, E) - |b (\ve, E)| )|_{E = E^{(s^\one,-)} (m^+, \La; \ve)}.
\end{split}
\end{equation}
where
\begin{equation}
\begin{split}\nn
a_1 (\ve, E) = v(m^+) + Q^{(s^\one)} (m^+, \La; \ve, E), \quad \quad a_2(\ve, E) = v(m^-) + Q^{(s^\one)}(m^-, \La ; \ve, E), \\
b (\ve, E) = |b_1 (\ve,E)|, \quad b_1 (\ve, E) = G^{(s^\one)} (m^+, m^-, \La; \ve, E).
\end{split}
\end{equation}

$(6)$
\begin{equation} \label{eq:6specHEEN1}
\begin{split}
\spec H_{\La,\ve} \cap \{E : |E - E^{(s^\one-1)}(m^+_0,\La^{(s^\one-1)}(m^+_0);\ve)| < 8 (\delta^{(s^\one-1)}_0)^{1/4}\} \\
= \{ E^{(s^\one,+)}(m_0^+, \La; \ve), E^{(s^\one, -)}(m_0^+, \La; \ve) \}, \\
E^{(s^\one,\pm)}(m_0^+, \La;0) = v(m^\pm_0).
\end{split}
\end{equation}
Let
\begin{equation}\label{eq:6EsplitspecconddomainqU}
(\delta^{(s^\one)}_0)^4 < \min_\pm |E - E^{(s^\one,\pm)}(m^+, \La; \ve)| < (\delta^{(s^\one-1)}_0)^{1/2}, \quad E \in \IR.
\end{equation}
Then, the matrix $(E - H_{\La,\ve})$ is invertible. Moreover,
\begin{equation}\label{eq:5inverseestiMATEq-2}
|[(E - H_{\La,\ve})^{-1}](x,y)| \le s_{D(\cdot;\La),T,\kappa_0,|\ve|;k,\La,\mathfrak{R}}(x,y),
\end{equation}
where $D(x,\La)$ is as in part $(1)$.
\end{prop}

\begin{proof}
$(1)$ The verification of this part is the same as for part $(1)$ of Proposition~\ref{rem:con1smalldenomnn}.

$(2)$ Let $(\ve,E) \in \cL^{(s^\one-1)}$. Assume that \eqref{eq:6-9-13OPQ} holds. If \eqref{eq:6-9-13OPQminus} holds, the arguments are completely similar. Let  $j \in J^{(s^\one+q-1)} \setminus \{0,j_0\}$ be arbitrary. Then, using conditions \eqref{eq:6EVsplitdefs1a} and \eqref{eq:6EVsplitdefAs1b}, one obtains
\begin{equation}\label{eq:6EdomainstimateAAAAOPsq}
\begin{split}
|E^{(s^\one-1,+)}(m^+_j, \La^{(s^\one-1)}(m^+_j); \ve) - E| \le |E^{(s^\one-1,+)}(m^+_j, \La^{(s^\one-1)}(m_j); \ve) - E^{(s^\one-1,+)}(m^+_0, \La^{(s^\one-1)}(m^+_0); \ve)| \\
+ |E^{(s^\one-1,+)}(m^+_0, \La^{(s^\one-1)}(m^+_0); \ve) - E| < \delta^{(s^\one-2)}_0 + (\delta^{(s^\one-1)}_0)^{1/2} < 2 \delta^{(s^\one-2)}, \\
|E^{(s^\one-1,+)}(m^+_j, \La^{(s^\one-1)}(m^+_j); \ve) - E| \ge |E^{(s^\one-1,+)}(m^+_j, \La^{(s^\one-1)}(m^+_j); \ve) - E^{(s^\one-1,+)}(m^+_0, \La^{(s^\one-1)}(m^+_0); \ve)| \\
- |E^{(s^\one-1,+)}(m^+_0, \La^{(s^\one-1)}(m^+_0); \ve) - E| > (\delta^{(s^\one-1)}_0)^{1/2} - (\delta^{(s^\one-1)}_0)^{1/2}/2 > (\delta^{(s^\one-1)}_0)^4, \\
|E^{(s^\one-1,-)}(m^+_j, \La^{(s^\one-1)}(m^+_j); \ve) - E| \ge |E^{(s^\one-1,-)}(m^+_j, \La^{(s^\one-1)}(m^+_j); \ve) - E^{(s^\one-1,+)}(m^+_0, \La^{(s^\one-1)}(m^+_0); \ve) | \\
- |E^{(s^\one-1,+)}(m^+_0, \La^{(s^\one-1)}(m^+_0); \ve) - E| > (\delta^{(s^\one-1)}_0)^4.
\end{split}
\end{equation}
Due to part $(6)$ of Proposition~\ref{rem:con1smalldenomnn} applied to $H_{\La^{(s^\one-1)}(m^+_j),\ve}$ in the role of
$\hle$, one has
\begin{equation}\label{eq:6Hinvestimatestatement1NNNN}
|[(E - H_{\La^{(s')}(m),\ve})^{-1}](x,y)| \le S_{D(\cdot;\La^{(s')}(m)),T,\kappa_0,|\ve|;\La^{(s')}(m),\mathfrak{R}}(x,y)
\end{equation}
for $s' = s+q-1$ and any $m = m^+_j$, $j \in J^{(s+q-1)} \setminus \{0,j_0\}$. Similarly, \eqref{eq:6Hinvestimatestatement1NNNN} holds for any
$s' < s+q-1$. Note that
\begin{equation}\label{eq:6Em01pm}
\begin{split}
|E - E^{(s^\one-1,+)} \bigl( m^+_{j_0}, \La^{(s^\one-1)}(m^+_{j_0}); \ve \bigr)| < |E - E^{(s^\one-1,+)} \bigl( m^+_{0}, \La^{(s^\one-1)}(m^+_{0}); \ve \bigr) | \\
+ |E^{(s^\one-1,+)} \bigl(m^+_{0}, \La^{(s^\one-1)}(m^+_{0}); \ve \bigr) - E^{(s^\one-1,+)} \bigl( m^+_{j_0}, \La^{(s^\one-1)}(m^+_{j_0}); \ve \bigr) | \\
< (\delta^{(s^\one-1)}_0)^{1/2} + (\delta^{(s^\one-1)}_0)^{5/8}  < (\delta^{(s^\one-1)}_0)^{1/4}/2.
\end{split}
\end{equation}
Due to \eqref{eq:6aux00c01a1PSTATEIII} from Lemma~\ref{lem:6-9.1} applied to $H_{\La^{(s+q-1)}(m^\pm),\ve}$ in the role of $H_{\La',\ve}$, one has
\begin{equation}\label{eq:6Hinvestimatestatement1PaQ}
|[(E - H_{\La^{(s^\one-1)}(m^\pm) \setminus \{m^\pm\}})^{-1}](x,y)| \le s_{D(\cdot;\La^{(s^\one-1)}(m^\pm) \setminus \{m^\pm\}),T,\kappa_0,|\ve|;\La^{(s^\one-1)}(m^\pm) \setminus \{m^\pm\},\mathfrak{R}} (x,y).
\end{equation}
Recall that $|v(n) - v(m_0^+)| \ge 2 \delta_0^4$ for any $n \in \Lambda \setminus  \bigcup_{s',m} \La^{(s')}(m) $, due to condition $(vi)$ in Definition~\ref{def:7-6}. This implies $|E - v(n)| \ge \delta_0^4$ for any such $n$. Due to Proposition~\ref{prop:aux1}, \eqref{eq:6-9.Hinvestimatestatement1PQ} follows from \eqref{eq:6Hinvestimatestatement1NNNN} and \eqref{eq:6Hinvestimatestatement1PaQ}.

$(3)$ The estimates in \eqref{eq:6ACCOPQ} follow from Lemma~\ref{lem:2Qfunctionderiv}. Let $f_j$, $f$ be as in \eqref{eq:6QfformulaUP}. Assume for instance that \eqref{eq:6-9-13OPQ} holds. Due to part $(3)$ of Lemma~\ref{lem:6-9.1}, one has $\hat f^\pm := E - v(m^\pm) - \hat Q^{(s^\one-1)} \bigl( m^\pm, \La^{(s^\one-1)}(m^\pm); \ve,E \bigr) \in \mathfrak{F}^{(1)}_{\mathfrak{g}^{\pm,1}, \mathfrak{r}^{(\pm,1)},1/2}$,
with $\mathfrak{g}^{(\pm,1)} := (g_0^\pm) := \bigl( E^{(s^\one-1,+)} \bigl( m^\pm, \La^{(s^\one-1)}(m^\pm);\ve \bigr) \bigr)$,
$\mathfrak{r}^{(\pm,1)} := (2 \rho_0)$, and with $\tau^{(\hat f^\pm)} \ge \tau^\zero$. Due to \eqref{eq:6-9EVsplitdefs1a} in Definition~\ref{def:9-6}, this implies $\hat f^\pm \in \mathfrak{F}^{(1)}_{\mathfrak{g}^{(1)}, \mathfrak{r}^{(1)},1/2}$ with $\mathfrak{g}^{(1)} := (g_0) = (g_0^+)$, $ \mathfrak{r}^{(1)} := (\rho_0)$. Due to \eqref{eq:6ACCOPQ}, one has $|\partial^\alpha(f_1 - \hat f^+)|, |\partial^\alpha(f_2 - \hat f^-)| < (\delta_0^{(s^\one-1)}) (\min \tau^{(\tilde f^\pm)})^6$, $0 \le \alpha \le 2$. It follows from Lemma~\ref{lem:4stable} that $f_j \in \mathfrak{F}^{(1)}_{\mathfrak{g}^{(1)}, \mathfrak{r}^{(1)},3/4}$, $\tau^{(f_j)} > \tau^\zero/2$. Due to \eqref{eq:6-9-13OPQ}, $f_2 - f_1 > \tau^\one$. Due to \eqref{eq:6ACCOPQ}, one obtains $|b| < ((3/4) \min_j \tau^{(f_j)})^6$, $|\partial_u b^2| < ((3/4) \min_j \tau^{(f_j)})^6 |b|$, $|\partial^2_u b^2| < ((3/4) \min_j \lambda \tau^{(f_j)})^6$, since $\tau^{(f_j)} > \tau^\zero/2 > (\delta^{(s^\one-1)}_0)^{1/4}$. Thus, $f \in \mathfrak{F}^{(2,\pm)}_{\mathfrak{g}^{(2)},3/4}(f_1,f_2,b^2)$, $\tau^{(f)} \ge \tau^\one (\tau^\zero)^2/4$. The case when \eqref{eq:6-9-13OPQminus} holds is completely similar. This finishes part $(3)$.

$(4)$ Follows from the Schur complement formula.

$(5)$ It follows from  Definition~\ref{def:4a-functions} that $\chi(\ve,E) = 0$ if and only if $\chi^{(f)} = 0$. All statements follow from Lemma~\ref{lem:6-1ell}.

$(6)$ The proof is completely similar to the proof of parts $(3)$ and $(4)$ of Proposition~\ref{prop:8-5n}. One can see that the proof of this part has nothing to do with the fact that we use part $(5)$ of the Definition~\ref{def:aux1} with $\bar \La \neq \IZ^\nu$.
\end{proof}

\begin{defi}\label{def:6-6UP}
Assume that the classes of matrices $GSR^{[\mathfrak{s},s']}\bigl( \mathfrak{m},m^+,m^-, \La; \delta_0, \tau\bigr)$ are defined for $s^\one \le s' \le s^\one + q^\one - 1$, $q^\one > 0$, starting with $GSR^{[\mathfrak{s},s^\one]} \bigl( \mathfrak{m}, m^+,m^-,\La; \delta_0, \tau \bigr) := GSR^{[\mathfrak{s}]} \bigl( \mathfrak{m}, m^+,m^-, \La; \delta_0, \tau \bigr)$ being as in Definition~\ref{def:9-6}. Let $m^+_1$, $m^-_1 \in \La$. Assume there are subsets $\cM^{(s',+)} = \left\{ m^+_j : j \in J^{(s')} \right\}$, $\cM^{(s',-)} = \left\{ m^-_j : j \in J^{(s')} \right\}$, $\La^{(s')} (m_j^+) = \La^{(s')} (m_j^{-})$, $j \in J^{(s')}$, with $s \le s' \le s^\one + q^\one - 1$, and also subsets $\cM^{(s')}$, $\La^{(s')}(m)$, $m \in \cM^{(s')}$, $1 \le s'\le s+q-1$ such that the following conditions are valid:
\begin{enumerate}

\item[(i)] $m^\pm_1 \in \cM^{(s^\one+q^\one-1,\pm)}$, $($so, by convention, $1 \in J^{(s^\one+q^\one-1)}$ $)$, $m \in \La(m)$ for any $m$.

\item[(ii)] For any $m$, $H_{\La(m), \ve}$ belongs to one of the classes we have introduced before with $m$ being a principal point, $s(m) \le s^\one + q^\one - 1$ $($ for the notation $s(m)$, see Definitions~\ref{def:4-1},~\ref{def:7-6},~\ref{def:9-6} $)$. Furthermore, $H_{\La(m_1^+), \ve} \in GSR^{[\mathfrak{s}, s^\one+q^\one-1]}\bigl(\mathfrak{m},m_1^+,m_1^- ,\La(m_1^+); \delta_0, \tau\bigr)$.

\item[(iii)] For any $m,m'$, either $\La(m) \cap \La(m') = \emptyset$, or $\La(m) = \La(m')$, in which case $m,m'$ are the principal points for $H_{\La(m), \ve}$.

\item[(iv)] Let $\delta^{(s')}_0$, $R^{(s')}$ be as in Definition~\ref{def:4-1}. Then, $\bigl( m + B(R^{(s')}) \bigr) \subset \Lambda^{(s')}(m)$ for any $\Lambda^{(s')}(m)$, and $\bigl( m_0^\pm + B(R^{(s^\one+q^\one)}) \bigr) \subset \Lambda$.

\item[(v)] Let $E^{(s^\one, \pm)}(m^+, \La; \ve)$ be as in Proposition~\ref{prop:6-4}. Below in Proposition~\ref{prop:6-4UUP} we define inductively the functions $E^{(s^\one+\tilde q, \pm)}(\tilde {m}, \tilde \La; \ve)$. We require that for each $m^+_j \in \cM^{(s',+)}$, $m^+_j \notin \{ m^+_1, m^-_1 \}$, $m \in \cM^{(s')} (\La)$, $s\le s' < s+q$, and any $\ve \in (-\ve_{s-1}, \ve_{s-1})$, we have
\begin{align}
\label{eq:6EVsplitdefs1aFI} 3 \delta^{(s^\one+q^\one-1)}_0 \le |E^{(s^\one+q^\one-1,\pm)} \bigl( m^+_j, \La^{(s^\one+q^\one-1)}(m^+_j); \ve \bigr) - E^{(s^\one+q^\one-1,\pm)} \bigl( m^+_1, \La^{(s^\one+q^\one-1)}(m^+_1); \ve \bigr) | \\
\nn \le \delta^{(s^\one+q^\one-2)}_0, \\
3 \delta^{(s^\one+q^\one-1)}_0 \le |E^{(s^\one+q^\one-1,\mp)} \bigl( m^+_j, \La^{(s^\one+q^\one-1)}(m^+_j); \ve \bigr) - E^{(s^\one+q^\one-1,\pm)} \bigl( m_1^+, \La^{(s^\one+q^\one-1)}(m_1^+); \ve \bigr) | \label{eq:6EVsplitdefAs1bFI}, \\
\label{eq:6EVsplitdefsqFI} \frac{\delta^{(s')}_0}{2} \le |E^{(s',\pm)}\bigl(m^+_j, \La^{(s')}(m^+_j); \ve \bigr) - E^{(s+q-1,\pm )} \bigl( m^+_1, \La^{(s+q-1)}(m^+_1); \ve \bigr) | \le \delta^{(s'-1)}_0 \\
\nn \text{for $s \le s' < s+q-1$}, \\
\frac{\delta^{(s')}_0}{2} \le |E^{(s',\mp)} \bigl( m^+_j, \La^{(s')}(m^+_j); \ve \bigr) - E^{(s+q-1,\pm)} \bigl( m_1^+, \La^{(s+q-1)}(m_0^+); \ve \bigr) | \label{eq:6EVsplitdefAFI} \quad \text{for $s \le s' < s^\one + q^\one - 1$}, \\
\frac{\delta^{(s')}_0}{2} \le |E^{(s')} \bigl( m, \La^{(s')}(m); \ve \bigr) - E^{(s+q-1,+)} \bigl( m^+_1, \La^{(s+q-1)}(m^+_1); \ve \bigr) | \le \delta^{(s'-1)}_0, \quad s' \le s.
\end{align}

\item[(vi)] $|v(n) - v(m_1^+)| \ge 2 \delta_0^4$ for any $n \in \Lambda \setminus \bigcup_{s',m} \La^{(s')}(m)$.

\item[(vii)] Due to Proposition~\ref{prop:6-4UUP}, the functions
\begin{equation} \label{eq:6-10acOPQDEFAPUUP}
Q^{(s^\one+q^\one)}(m_1^\pm,\La; \ve, E) = \sum_{m',n' \in \La \setminus \{m^+_1,m_{1}^-\}} h(m_{1}^\pm,m';\ve) (E - H_{\La \setminus \{m^+-1,m^-_1\}})^{-1}(m',n') h(n',m^\pm_1;\ve)
\end{equation}
are well-defined for all $\ve \in (-\ve_{s-1}, \ve_{s-1})$,
\begin{equation} \label{eq:6EDomain}
\begin{split}
E \in \bigcup_\pm (E^{(s^\one + q^\one-1,\pm)} \bigl( m^+_1, \La^{(s^\one+q^\one-1)} (m^+_1); \ve \bigr) - 2 \delta^{(s^\one+q^\one-1)}_0, \\
E^{(s^\one+q^\one-1,\pm)} \bigl( m^+_1, \La^{(s^\one+q^\one-1)} (m^+_1); \ve \bigr) + 2 \delta^{(s^\one + q^\one - 1)}_0).
\end{split}
\end{equation}
We require that for these $(\ve,E)$,
\begin{equation} \label{eq:6-13GSRUUP}
v(m^+-1) + Q^{(s^\one + q^\one)}(m^+_1, \La,E) \ge v(m^-_1) + Q^{(s^\one + q^\one)}(m^-_1,\La; \ve, E) + \tau^\one.
\end{equation}
\end{enumerate}

Then we say that $\hle \in GSR^{[\mathfrak{s},s^\one+q]} \bigl( \mathfrak{m},m_1^+,m_1^-, \La; \delta_0, \tau^\one \bigr)$. We call $m^+_1,m^-_1$ the principal points. We set $s(m^\pm_1) = s^\one + q^\one$. We call $\La^{(s^\one + q^\one - 1)}(m^\pm_1)$ the $(s^\one + q^\one-1)$-set for $m^\pm_1$.
\end{defi}

\begin{prop}\label{prop:6-4UUP}
Using the notation from Definition~\ref{def:6-6UP}, the following statements hold:

$(1)$ Define inductively $D(x;\La) := D(x;\La \setminus \{m^+_1,m^-_1\}) := D(x;\La(m))$ if $x \in \La(m)$ with $m \notin \{m^+_1,m^-_1\}$, and $D(m^\pm_1;\La) := 4 \log (\delta^{(s^\one+q^\one)}_0)^{-1}$.

Then, $D(\cdot;\La) \in \mathcal{G}_{\La,T,\kappa_0}$, $D(\cdot;\La \setminus \{m^+_1,m^-_1\}) \in \mathcal{G}_{\La \setminus \{m^+_1,m^-_1\},T,\kappa_0}$.

$(2)$  Let $\cL^{(s^\one+q^\one-1,\pm)} := \cL_\mathbb{R} \bigl( E^{(s^\one+q^\one-1,\pm)} \bigl( m^+_1, \La^{(s^\one+q^\one-1)}(m^+_1); \ve \bigr), 2 \delta_0^{(s^\one+q^\one-1)} \bigr)$. For any $(\ve,E) \in \cL^{(s^\one+q^\one-1,\pm)}$, we have
\begin{equation}\label{eq:6-9.Hinvestimatestatement1PQUP}
|[(E - H_{\La \setminus \{m^+_1,m^-_1\},\ve})^{-1}](x,y)| \le s_{D(\cdot;\La), T, \kappa_0, |\ve|; \La \setminus \{m^+_1,m^-_1\}, \mathfrak{R}}(x,y).
\end{equation}

$(3)$ The functions
\begin{equation} \label{eq:6-10acUUP}
\begin{split}
Q^{(s^\one+q^\one)}(m^\pm_1,\La; \ve, E) = \sum_{m',n' \in \La \setminus \{m^+_1,m^-_1\}} h(m^\pm_1,m';\ve) (E - H_{\La \setminus \{m^+_1,m^-_1\}})^{-1} (m,n) h(n',m^\pm_1;\ve), \\
G^{(s^\one+q^\one)}(m^\pm_1,m^\mp_1, \La \setminus \{m^+_1,m^-_1\}; \ve, E) = h(m^\pm_1,m^\mp_1;\ve) \\
+ \sum_{m,n \in \La \setminus \{m^+_1,m^-_1\}} h(m^\pm_1,m;\ve) (E - H_{\La \setminus \{m^+_1,m^-_1\}})^{-1} (m,n) h(n,m^\mp_1;\ve)
\end{split}
\end{equation}
are well-defined and $C^2$-smooth in $\cL^{(s^\one+q^\one-1,+)} \cup \cL^{(s^\one+q^\one-1,-)}$,
\begin{equation} \label{eq:6-12acACCOPQnew1}
\begin{split}
& \big| \partial^\alpha_E Q^{(s^\one+q^\one)}(m_1^\pm, \La; \ve, E) - \partial^\alpha_E Q^{(s^\one+q^\one-1)} \bigl( m_1^\pm, \La^{(s^\one+q^\one-1)}(m_1^+); \ve,E \bigr) \big | \\
& \le 4 |\ve|^{3/2} \exp \left( -\kappa_0 R^{(s^\one+q^\one-1)} \right) \le |\ve| (\delta_0^{(s^\one+q^\one-1)})^{12}, \\
& \big| \partial^\alpha_E G^{(s^\one+q^\one)}(m^\pm_1, m^\mp_1, \La;\ve,E) - \partial^\alpha_E G^{(s^\one+q^\one-1)} \bigl( m_1^\pm,m^\mp_1, \La^{(s^\one+q^\one-1)}(m_1^+); \ve, E \bigr) \big| \\
& \le 4 |\ve|^{3/2} \exp \left( -\kappa_0 R^{(s^\one+q^\one-1)} \right) \le |\ve| (\delta_0^{(s^\one+q^\one-1)})^{12}, \\
& \big| \partial^\alpha_E  Q^{(s^\one+q^\one)}(m_1^\pm, \La; \ve, E) \big| \le |\ve|, \quad |E - v(m_1^\pm) - Q^{(s^\one+q^\one)}(m_1^\pm, \La; \ve, E)| < |\ve|, \\
& \big| \partial^\alpha_E G^{(s^\one+q^\one)}(m^\pm_1, m^\mp_1, \La;\ve,E)| \le 8 |\ve|^{1/2} \exp \left( -\frac{7\kappa_0}{8} |m_1^+ - m_1^-| \right) \le |\ve| (\delta_0^{(s^\one-1)})^{12}
\end{split}
\end{equation}
with $\alpha \le 2$. Furthermore, for $q^\one > 0$, set $\rho_0 = \delta^{(s^\one+q^\one-2)}_0$, $\rho_1 = \delta^{(s^\one+q^\one-1)}_0$,
$g_{0,\pm}(\ve,E) := g_{1,\pm}(\ve,E) := E^{(s^\one+q^\one-1,\pm)}(m^+_1, \La(m^+_1); \ve)$,
\begin{equation}\label{eq:6fformulaINTERIM}
\begin{split}
f_{1}(q^\one;\ve,E) = E - v(m^+_1) - Q^{(s^\one+q^\one)}(m^+_1,\La; \ve,E ), \quad f_{2}(q^\one;\ve,E) = E - v(m^-_1) - Q^{(s^\one+q^\one)}(m^-_1,\La; \ve,E ), \\
b^2(q^\one;\ve,E) = |G^{(s^\one+q^\one)}(m^\pm_1,m^\mp_1,\La; \ve, E)|^2, \quad f(q^\one;\ve,E) = f_1(q^\one;\ve,E) - \frac{b^2(q^\one;\ve,E)}{f_2(q^\one;\ve,E)}.
\end{split}
\end{equation}
Then $($ see Definition~\ref{def:4a-functions}$)$ $f(q^\one;\cdot) \in \mathfrak{F}^{(2)}_{\mathfrak{g}^{(2)}, \lambda(s,q^\one)}(f_1(q^\one;\cdot), f_2(q^\one;\cdot), b^2(q^\one;\cdot))$, $\tau^{(f_j(q^\one;\cdot))} > \tau^\zero(1-\xi(s,q^\one))$, $\tau^{(f)} \ge \tau^\one(\tau^\zero)^2 (1 - \xi(s,q^\one))^2$, with $\lambda(s,q) = (1 + 8^{-(s+q)}) \lambda(s,q-1)$, $\lambda(s,0) = 3/4$, $\lambda(s,q) < 1$ for any $s,q$, $1 - \xi(s,q) = (1 - \xi(s,q-1))(1 - 8^{-(s+q)})$, $\xi(s,0) = 0$, $1 - \xi(s,q) > 1/2$ for any $s,q$.

$(4)$ Let $(\ve,E) \in \cL^{(s^\one+q^\one-1,+)}\cup\cL^{(s^\one+q^\one-1,-)}$. Then, $E \in \spec H_{\La,\ve}$ if and only if $E$ obeys
\begin{equation} \label{eq:6-chiINTERIM}
\begin{split}
& \chi(\ve,E) := \bigl( E - v(m^+-1) - Q^{(s^\one+q^\one)}(m^+_1, \La; \ve, E) \bigr) \cdot \bigl( E - v(m^-_1) - Q^{(s^\one+q^\one)}(m^-_1,\La; \ve, E) \bigr) \\
& \qquad - G^{(s^\one+q^\one)}(m^+_1, m^-_1, \La; \ve, E)G^{(s^\one+q^\one)}(m^-_1, m^+_1, \La; \ve, E) = 0.
\end{split}
\end{equation}

$(5)$ For $\ve \in  (-\ve_{s-1}, \ve_{s-1})$, the equation
\begin{equation}\label{eq:8-13nn-3}
\chi(\ve,E) = 0
\end{equation}
has exactly two solutions $E = E^{(s^\one+q^\one, \pm)}(m_1^+, \La; \ve)$, $E^{(s^\one+q^\one, -)}(m_1^+, \La; \ve) < E^{(s^\one+q^\one, +)}(m_1^+, \La; \ve)$,
\begin{equation} \label{eq.6Eestimates1APINTERIM}
|E^{(s^\one+q^\one, \pm)}(m_1^+, \La; \ve) - E^{(s^\one+q^\one-1,\pm)}(m^\pm_1,\La^{(s^\one+q^\one-1)}(m^+_1);\ve)| < |\ve| (\delta^{(s^\one+q^\one-1)}_0)^3,
\end{equation}
\begin{equation} \label{eq.6Eestimates1APINTERIMEPM}
\begin{split}
[a_1 (\ve, E) + |b (\ve, E)| ]|_{E = E^{(s^\one+q^\one,+)} (m^+_1, \La; \ve)} \ge E^{(s^\one+q^\one,+)} (m^+_1, \La; \ve) \\
\ge \max ( a_1 (\ve, E),  a_2 (\ve, E) + |b (\ve, E)| )|_{E = E^{(s^\one+q^\one,+)} (m^+_1, \La; \ve)} , \\
[a_2 (\ve, E) - |b (\ve, E)| ]|_{E = E^{(s^\one+q^\one,+)} (m^+_1, \La; \ve)} \le E^{(s^\one+q^\one,-)} (m^+_1, \La; \ve) \\
\le \min ( a_2 (\ve, E),  a_1 (\ve, E) - |b (\ve, E)| )|_{E = E^{(s^\one+q^\one,-)} (m^+_1, \La; \ve)},
\end{split}
\end{equation}
where
\begin{equation}\begin{split}\nn
a_1 (\ve, E) = v(m_1^+) + Q^{(s^\one+q^\one)} (m^+_1, \La; \ve, E), \quad \quad a_2(\ve, E) = v(m_1^-) + Q^{(s^\one+q^\one)}(m_1^-, \La ; \ve, E), \\
b (\ve, E) = |b_1 (\ve,E)|, \quad b_1 (\ve, E) = G^{(s^\one+q^\one)} (m^+_1, m^-_1, \La; \ve, E).
\end{split}
\end{equation}

The functions $E^{(s^\one+q^\one, \pm)}(m^+_1, \La; \ve)$ are smooth on the interval $(-\ve_{s-1}, \ve_{s-1})$.

$(6)$
\begin{equation} \label{eq:6specHEEAAAINTERIM}
\begin{split}
\spec H_{\La,\ve} \cap \{E : \min_\pm |E - E^{(s^\one+q^\one,\pm)} (m^+_1,\La^{(s^\one+q^\one-1)} (m^+_1);\ve)| < 8 (\delta^{(s^\one+q^\one-1)}_1)^{1/4}\} \\
= \{ E^{(s^\one+q^\one,+)}(m_1^+, \La; \ve), E^{(s^\one+q^\one,-)}(m_1^+, \La; \ve) \}. \\
\end{split}
\end{equation}

If
\begin{equation}\label{eq:6EsplitspecINTERIM}
(\delta^{(s^\one+q^\one)}_1)^4 < \min_\pm |E - E^{(s^\one+q^\one-1,\pm)} \bigl( m^+_0, \La^{(s^\one+q^\one-1)}(m^+_1); \ve \bigr)|
< 3 \delta_0^{(s^\one+q^\one-1)}/2, \quad E \in \IR,
\end{equation}
then the matrix $(E - H_{\La,\ve})$ is invertible. Moreover,
\begin{equation}\label{eq:6inverINTERIM}
|[(E - H_{\La,\ve})^{-1}] (x,y)| \le s_{D(\cdot;\La),T,\kappa_0,|\ve|;k,\La,\mathfrak{R}}(x,y),
\end{equation}
where $D(x;\La)$ is as in part $(1)$.
\end{prop}

The proof of each statement in this proposition is completely similar to the proof of either a statement from Proposition~\ref{rem:con1smalldenomnn} or a statement from Proposition~\ref{prop:6-4}. We skip the proofs.

\bigskip

\textit{We need yet two more upgrades of the classes of matrices under consideration. We skip the proofs for the first upgrade and most of the proofs for the second one since they are completely similar to the proof of Proposition~\ref{rem:con1smalldenomnn} and Proposition~\ref{prop:6-4}, respectively. Here is the first one:}

\begin{defi}\label{def:9-6general}
Assume that the classes $GSR^{[\mathfrak{s}^{(h)}]} \bigl( \mathfrak{m}^{(h)},m^+,m^-, \La; \delta_0,\mathfrak{t}^{(h)} \bigr)$, $GSR^{[\mathfrak{s}^{(h)},s^{(h)}+q]} \bigl( \mathfrak{m}^{(h)},m^+,m^-, \La; \delta_0,\mathfrak{t}^{(h)} \bigr)$ are defined for all $h = 1, \dots, \ell$, $\ell \ge 2$, starting with $GSR^{[\mathfrak{s}^{(1)}]} \bigl( \mathfrak{m}^{(1)}, m^+,m^-,\La; \delta_0,\mathfrak{t}^{(1)} \bigr)$, $GSR^{[\mathfrak{s}^{(1)},s^\one+q]} \bigl( \mathfrak{m}^{(1)},m^+,m^-, \La; \delta_0,\mathfrak{t}^{(1)} \bigr)$ being as in Definition~\ref{def:9-6} and Definition~\ref{def:6-6UP}, respectively. Here, $\mathfrak{m}^{(h)} \subset \La$, $|\mathfrak{m}^{(h)}| = 2^{h+1}$, $\mathfrak{s}^{(h)} = (s^\zero,s^\one,\dots,s^{(h)})$, $s^{(k)} \in \mathbb{N}$, $s^{(k)} < s^{(k+1)}$, $\mathfrak{t}^{(h)} = (\tau^\zero,\dots,\tau^{(h)})$, $\tau^{(k)} > \tau^{(k+1)}>0$. Let $\hle$ be as in \eqref{eq:5-1N}--\eqref{eq:5-5N} and let $\delta^{(s')}_0$, $R^{(s')}$ be as in Definition~\ref{def:4-1}. Let $q$ be such that $\tau^{(\ell)} > (\delta^{(s^{(\ell)}+q-1)}_0)^{1/4}$. Let $m^+,m^- \in \La$. Assume that there are subsets $\cM \subset \La$, $\La(m) \subset \La$, $m \in \cM$, such that the following conditions hold
\begin{enumerate}

\item[(i)] $m^\pm \in \cM$, $m \in \La(m)$ for any $m$.

\item[(ii)] For any $m \in \cM$, $H_{\La(m), \ve}$ belongs to one of the classes we have introduced before with $s(m) \le s^{(\ell)} + q^{(\ell)} - 1$ $($ for the notation $s(m)$, see Definitions~\ref{def:4-1},~\ref{def:7-6},~\ref{def:9-6},~\ref{def:6-6UP} $)$. Furthermore, $H_{\La(m^\pm), \ve} \in GSR^{[\mathfrak{s}^{(\ell)}, s^{(\ell)}+q-1]} \bigl( \mathfrak{m}^{(\ell,\pm)}, \La(m^\pm); \delta_0,\mathfrak{t}^{(\ell)}\bigr)$ with some $\mathfrak{m}^{(\ell,\pm)} \subset \La(m^\pm)$, $m^\pm \in \mathfrak{m}^{(\ell,\pm)}$. Given $m \in \cM$ such that $H_{\La(m), \ve} \in GSR^{[\mathfrak{s}^{(\ell')}, s^{(\ell')}}+q'] \bigl( \mathfrak{m}^{(\ell')}, \La(m); \delta_0,\mathfrak{t}^{(\ell')} \bigr)$, we set $s(m) := s^{(\ell')} + q'$, which is the largest integer involved in the latter notation.

\item[(iii)] For any $m,m'$, either $\La(m) \cap \La(m') = \emptyset$, or $\La(m) = \La(m')$, in which case $m,m'$ are the principal points for $H_{\La(m), \ve}$. We use the notation $m' = \bullet m$ for the latter case. In the former case we say that $\bullet m$ does not exist and $\{m,\bullet m\} = \{m\}$. Finally, $\bullet m^+ \neq m^-$, that is, $\La(m^+) \neq \La(m^-)$.

\item[(iv)] Let $m \in \cM$. There exists a unique real-analytic function $E(m,\La(m);\ve)$, $\ve \in (-\ve_{s-1}, \ve_{s-1})$ such that
$E(m,\La(m);\ve)$ is a simple eigenvalue of $H_{\La(m), \ve}$ and $E(m,\La(m);0) = v(m)$. Furthermore, let $m \in \cM \setminus \{m^+,\bullet m^+,m^-,\bullet m^-\}$ be arbitrary. The following estimates hold:
\begin{align}
(\delta^{(s^{(\ell)}+q-1)}_0)^{1/2} \le \min_{m' \in \{m,\bullet m\}} |E \bigl( m^+, \La(m^+); \ve \bigr) - E \bigl( m', \La(m'); \ve \bigr)| \le \delta^{(s^{(\ell)}+q-2)}_0 \label{eq:6EVsplitdefs1aINT} \quad \text{if $s(m) = s^{(\ell)}+q-1$}, \\
|E \bigl( m^-, \La(m^-); \ve \bigr) - E \bigl( m^+, \La(m^+); \ve \bigr)| \le (\delta_0^{(s^{(\ell)}+q-1)})^{5/8} \label{eq:6-9EVsplitdefs1aINT} \\
\frac{\delta^{(s(m))}_0}{2} \le \min_{m' \in \{m,\bullet m\}} |E \bigl( m^+, \La(m^+); \ve \bigr) - E \bigl(m', \La(m'); \ve \bigr)| \le
\nn \delta^{(s(m)-1)}_0 \label{eq:6-9EVsplitdefs1aINT10} \quad \text{if $s(m) < s^{(\ell)}+q-1$.}
\end{align}

\item[(v)]
$$
\bigl(m + B(R^{(s(m))}\bigr)\subset \Lambda(m).
$$

\item[(vi)] $|v(n) - v(m_0)| \ge 2 \delta_0^4$ for any $n \in \Lambda \setminus \bigcup_{m \in \cM} \La(m)$.

\item[(vii)] Due to the inductive argument from Proposition~\ref{prop:6-4INT} below, for any $\ve \in (-\ve_{s-1}, \ve_{s-1})$ and any
\begin{equation} \label{eq:6EDomainINTERVAL}
E \in (E \bigl( m^+, \La(m^+);\ve \bigr) - (\delta^{(s^{(\ell)}+q-1)}_0)^{1/2}, E \bigl( m^+, \La(m^+);\ve \bigr) + (\delta^{(s^{(\ell)}+q-1)}_0)^{1/2}),
\end{equation}
the functions
\begin{equation} \label{eq:6-10acOPQDEFAPUUPINTERIM}
Q(m^\pm,\La; \ve, E) = \sum_{m',n' \in \La \setminus \{m^+,m^-\}} h(m^\pm,m';\ve) (E - H_{\La \setminus \{m^+,m^-\}})^{-1} (m',n') h(n',m^\pm;\ve)
\end{equation}
are well-defined. We require that for these $\ve,E$ and some $\tau^{(\ell+1)} > 0$, we have
\begin{equation} \label{eq:6-13GSRUUP-2}
v(m^+) + Q(m^+, \La,E) \ge v(m^-) + Q(m^-,\La; \ve, E) + \tau^{(\ell+1)}.
\end{equation}
\end{enumerate}

In this case we say that $\hle \in GSR^{[\mathfrak{s}^{(\ell+1)}]} \bigl( \mathfrak{m}^{(\ell+1)}, \La; \delta_0,\mathfrak{t}^{(\ell+1)} \bigr)$, $\mathfrak{m}^{(\ell+1)} = \bigcup_\pm \mathfrak{m}^{(\ell,\pm)}$, $\mathfrak{s}^{(\ell+1)} = (s^\one,\dots, s^{(\ell+1)})$, $s^{(\ell+1)} = s^{(\ell)}+q$, $\mathfrak{t}^{(\ell+1)} = (\tau^\zero,\dots, \tau^{(\ell+1)})$. We call $\mathfrak{m}^{(\ell+1)}$ the principal set for $\hle$ and $m^+,m^-$ the principal points for $\hle$. We set $s(m^\pm) = s^{(\ell+1)}$. We call $\La^{(s(m^\pm)-1)}(m^\pm)$ the $(s(m^\pm)-1)$-set for $m^\pm$.
\end{defi}

\begin{prop}\label{prop:6-4INT}
Using the notation from Definition~\ref{def:9-6general}, the following statements hold:

$(1)$ Define inductively $D(x;\La) = D(x;\La \setminus \{m^+,m^-\}) = D(x;\La \setminus \mathfrak{m}^{(\ell+1)}) = D(x;\La(m))$ if $x \in \La(m) \setminus \mathfrak{m}^{(\ell+1)}$, $D(x;\La) = D(x;\La \setminus \{m^+,m^-\}) = 4 \log (\delta^{(s^{(\ell+1)}-1)}_0)^{-1}$ if $x \in \mathfrak{m}^{(\ell+1)} \setminus\{m^+,m^-\}$, and $D(x;\La) = 4 \log (\delta^{(s^{(\ell+1)})}_0)^{-1}$ if $x \in \{m^+,m^-\}$. Then, $D(\cdot;\La \setminus \mathfrak{m}^{(\ell+1)}) \in \mathcal{G}_{\La \setminus \mathfrak{m}^{(\ell+1)},T,\kappa_0}$, $D(\cdot;\La \setminus \{m^+,m^-\}) \in \mathcal{G}_{\La \setminus \{m^+,m^-\},\IZ^\nu \setminus \{m^+,m^-\},T,\kappa_0}$, $D(\cdot;\La) \in \mathcal{G}_{\La,T,\kappa_0}$.

$(2)$ Set $g_1 = E \bigl( m^+, \La^{(s^{(\ell+1)}-1)}(m^+); \ve \bigr)$ and
\begin{equation} \label{eq:6-11domainnOPQUINT}
\cL^{(s^{(\ell+1)}-1)} := \Big\{ (\ve,E) : \ve \in (-\ve_{s^\zero-1},\ve_{s^\zero-1}), |E - g_1(\ve)| < \frac{(\delta_0^{(s^{(\ell+1)}-1)})^{1/2}}{2} \Big\}.
\end{equation}
For any $(\ve,E) \in \cL^{(s^{(\ell+1)}-1)}$, the matrix $(E - H_{\La \setminus \{m^+,m^-\},\ve})$ is invertible. Moreover,
\begin{equation}\label{eq:6-9.Hinvestimatestatement1PQINT}
|[(E - H_{\La \setminus \{m^+,m^-\},\ve})^{-1}](x,y)| \le s_{D(\cdot;\La \setminus \{m^+,m^-\}),T,\kappa_0,|\ve|;\La \setminus \{m+,m^-\},\mathfrak{R}}(x,y).
\end{equation}

$(3)$ Set for $(\ve,E) \in \cL^{(s^{(\ell+1)}-1)}$,
\begin{equation} \label{eq:6-10acUINT}
\begin{split}
Q^{(s^{(\ell+1)})}(m^\pm,\La; \ve, E) = \sum_{m,n \in \La \setminus \{m^+,m^-\}} h(m^\pm,m;\ve) (E - H_{\La \setminus \{m^+,m^-\}})^{-1} (m,n) h(n,m^\pm;\ve), \\
G^{(s^{(\ell+1)})}(m^\pm,m^\mp, \La; \ve, E) = h(m^\pm,m^\mp;\ve) + \sum_{m,n \in \La \setminus \{m^+,m^-\}} h(m^\pm,m;\ve) (E - H_{\La \setminus \{m^+,m^-\}})^{-1} (m,n) h(n,m^\mp;\ve), \\
\hat Q(m^\pm,\La(m^\pm); \ve, E) = \sum_{m,n \in \La(m^\pm) \setminus \{m^\pm\}} h(m^\pm,m;\ve) (E - H_{\La(m^\pm) \setminus \{m^\pm\},\ve})^{-1} (m,n) h(n,m^\pm;\ve).
\end{split}
\end{equation}
The functions in \eqref{eq:6-10acUINT} are well-defined and $C^2$-smooth. For $\alpha \le 2$, we have
\begin{equation} \label{eq:6ACCOPQINT}
\begin{split}
& \big| \partial^\alpha_E Q^{(s^{(\ell+1)})}(m^\pm, \La; \ve, E) - \partial^\alpha_E \hat Q \bigl( m^\pm, \La(m^\pm); \ve,E \bigr) \big| \le 4 |\ve|^{3/2} \exp \left( -\kappa_0 R^{(s^{(\ell+1)}-1)} \right) \le |\ve| (\delta_0^{(s^{(\ell+1)}-1)})^{12}, \\
& \big| \partial^\alpha_E G^{(s^{(\ell+1)})}(m^+, m^-, \La;\ve,E) \big| \le 4 |\ve|^{3/2} \exp \left( -\frac{7\kappa_0}{8} |m^+ - m^-| \right) \le 4 |\ve|^{3/2} \exp \left( -\kappa_0 R^{(s^{(\ell+1)}-1)} \right) \le |\ve| (\delta_0^{(s^{(\ell+1)}-1)})^{12}.
\end{split}
\end{equation}
Furthermore, set $\rho_0 = \delta^{(s^{(\ell+1)}-1)}$, $\rho_j = \rho_0$, $g_j = g_0$, $j = 1, \dots, \ell$,
\begin{equation}\label{eq:6QfformulaUPINT}
\begin{split}
f_1(\ve,E) = E - v(m^+) - Q^{(s^{(\ell+1)})}(m^+,\La; \ve,E ), \quad f_2(\ve,E) = E - v(m^-) - Q^{(s^{(\ell+1)})}(m^-,\La; \ve,E ), \\
b^2(\ve,E) = |G^{(s^{(\ell+1)})}(m^\pm,m^\mp,\La; \ve, E)|^2, \quad f(\ve,E) = f_1(\ve,E) -\frac{b^2(\ve,E)}{f_2(\ve,E)}.
\end{split}
\end{equation}
Then, $f \in \mathfrak{F}^{(\ell+1)}_{\mathfrak{g}^{(\ell+1)}, 1-4^{-\ell}}(f_1,f_2,b^2)$, $\tau^{(f_j)} > \tau^{[\ell]}/2$, $\tau^{(f)} \ge \tau^{[\ell+1]}$, where $\tau^{[j+1]} = \tau^{(j+1)}(\tau^{[j]})^2/4$, $j \ge 0$, $\tau^{[0]} := \tau^\zero$; see Definition~\ref{def:4a-functions}.

$(4)$ Let $(\ve,E) \in \cL^{(s^{(\ell+1)}-1)}$. Then, $E \in \spec H_{\La,\ve}$ if and only if $E$ obeys
\begin{equation} \label{eq:6-13NNNNOPQUINT}
\chi(\ve,E) := \bigl( E - v(m^+) - Q^{(s^{(\ell+1)})}(m^+, \La; \ve, E) \bigr) \cdot \bigl( E - v(m^-) - Q^{(s^{(\ell+1)})}(m^-,\La; \ve, E)\bigr) - |G^{(s^{(\ell+1)})}(m^+, m^-, \La; \ve, E)|^2 = 0.
\end{equation}

$(5)$ Let $f$ be as in part $(3)$ and let $\chi^{(f)}$ be as in Definition~\ref{def:4a-functions}. Then, $\chi(\ve,E) = 0$ if and only if $\chi^{(f)} = 0$. For $\ve \in (-\ve_{s-1}, \ve_{s-1})$, the equation
\begin{equation}\label{eq:8-13nnINT}
\chi^{(f)}(\ve,E) = 0
\end{equation}
has exactly two solutions $E(m^+, \La; \ve) > E(m^-, \La; \ve)$, which obey
\begin{equation} \label{eq.6Eestimates1APN1INT}
|E(m^\pm, \La; \ve) - E(m^\pm,\La(m^\pm);\ve)| < 4 (\delta^{(s^{(\ell+1)}-1)}_0)^{1/8}.
\end{equation}
The functions $E(m^\pm, \La; \ve)$ are $C^2$-smooth on the interval $(-\ve_{s^\zero-1}, \ve_{s^\zero-1})$. The following estimates hold,
\begin{equation} \label{eq:6-13acnq0N1INT}
\begin{split}
|\partial^\alpha_E \chi^{(f)}| \le 8 \quad \text{for $\alpha \le 2$}, \quad \partial^2_E \chi^{(f)} > 1/8, \\
\partial^\alpha_E \chi^{(f)}|_{\ve,E(m^-, \La; \ve)} < -(\tau^{(f)})^2, \quad \partial^\alpha_E \chi^{(f)}|_{\ve,E(m^+, \La; \ve)} > (\tau^{(f)})^2, \\
E(m^+, \La; \ve) - E(m^-, \La; \ve) > \frac{1}{8}[ -\partial^\alpha_E \chi^{(f)}|_{\ve,E(m^-, \La; \ve)} + \partial^\alpha_E \chi^{(f)}|_{\ve,E(m^+, \La; \ve)}], \\
-\partial^\alpha_E \chi^{(f)}|_{\ve,E(m^-, \La; \ve)}, \partial^\alpha_E \chi^{(f)}|_{\ve,E(m^+, \La; \ve)} > \frac{1}{2^{14}} \bigl(E(m^+, \La; \ve) - E(m^-, \La; \ve)\bigr)^2, \\
|\chi^{(f)}(\ve,E)| \ge \frac{1}{8} \min \bigl( (E - E(m^+, \La; \ve))^2, (E - E(m^-, \La; \ve))^2 \bigr),
\end{split}
\end{equation}
\begin{equation} \label{eq.6Eestimates1APPFINALPM}
\begin{split}
[ a_1 (\ve, E) + |b (\ve, E)| ]|_{E = E (m^+, \La; \ve)} \ge E (m^+, \La; \ve) \\
\ge \max ( a_1 (\ve, E),  a_2 (\ve, E) + |b (\ve, E)| )|_{E = E (m^+, \La; \ve)}, \\
[ a_2 (\ve, E) - |b (\ve, E)| ]|_{E = E (m^+, \La; \ve)} \le E (m^-, \La; \ve) \\
\le \min ( a_2 (\ve, E),  a_1 (\ve, E) - |b (\ve, E)| )|_{E = E (m^-, \La; \ve)},
\end{split}
\end{equation}
where
\begin{equation}\begin{split}\nn
a_1 (\ve, E) = v(m^+) + Q^{(s^{(\ell+1)})}(m^+, \La; \ve, E), \quad \quad a_2 (\ve, E) = v(m^-) + Q^{(s^{(\ell+1)})}(m^-, \La; \ve, E), \\
b (\ve, E) = |b_1 (\ve,E)|, \quad b_1 (\ve, E) = G^{(s^{(\ell+1)})}(m^+, m^-, \La; \ve, E).
\end{split}
\end{equation}

$(6)$
\begin{equation} \label{eq:6specHEEN1INT}
\begin{split}
\spec H_{\La,\ve} \cap \Big\{ |E - E(m^+,\La(m^+);\ve,E)| < \frac{(\delta^{(s^{(\ell+1)}-1)}_0)^{1/2}}{2} \Big\} \\
= \{ E(m^+, \La; \ve), E(m^-, \La; \ve)\}, \\
E(m^\pm, \La; 0) = v(m^\pm).
\end{split}
\end{equation}
Furthermore, let
\begin{equation}\label{eq:6EsplitspecconddomainqUINT}
(\delta^{(s^{(\ell+1)})}_0)^4 < \min_\pm |E - E(m^\pm, \La; \ve)| < 2 \delta^{(s^{(\ell+1)}-1)}_0, E \in \IR.
\end{equation}
Then the matrix $(E - H_{\La,\ve})$ is invertible. Moreover,
\begin{equation}\label{eq:5inverseestiMATEqINT}
|[(E - H_{\La,\ve})^{-1}](x,y)| \le S_{D(\cdot;\La),T,\kappa_0,|\ve|;k,\La,\mathfrak{R}}(x,y).
\end{equation}
\end{prop}

\bigskip

\bigskip

\textit{Here is the second and last upgrade of the classes of matrices.}

\begin{defi}\label{def:6-6UPFIN}
Let $\ell \in \mathbb{N}$ be fixed. Assume that the classes $GSR^{[\mathfrak{s}^{(\ell+1)},s^{(\ell+1)}+q']} \bigl( \mathfrak{m}^{(\ell+1)},m^+,m^-, \La'; \delta_0,\mathfrak{t}^{(\ell+1)} \bigr)$ are defined for all $q' = 0, \dots, q-1$, starting with
$GSR^{[\mathfrak{s}^{(\ell+1)},s^{(\ell+1)}]} \bigl( \mathfrak{m}^{(\ell+1)},m^+,m^-, \La'; \delta_0,\mathfrak{t}^{(\ell+1)} \bigr) :=
GSR^{[\mathfrak{s}^{(\ell+1)}]} \bigl( \mathfrak{m}^{(\ell+1)},m^+,m^-, \La'; \delta_0,\mathfrak{t}^{(\ell+1)} \bigr)$ being as in Definition~\ref{def:9-6general}. Let $\La$ and $m^+, m^- \in \La$ be given.

Assume that there are subsets $\cM \subset \La$, $\La(m) \subset \La$, $m \in \cM$, such that the following conditions hold
\begin{enumerate}

\item[(i)] $m^\pm \in \cM$, $m \in \La(m)$ for any $m$, $\La(m^+) = \La(m^-)$.

\item[(ii)] For any $m \in \cM$,  $H_{\La(m), \ve}$ belongs to one of the classes we have introduced before with $s(m) \le s^{(\ell)}+ q^{(\ell)} - 1$ $($ for the notation $s(m)$, see Definitions~\ref{def:4-1},~\ref{def:7-6},~\ref{def:9-6},~\ref{def:6-6UP}, ~\ref{def:9-6general} $)$. Furthermore, $H_{\La(m^+), \ve} \in GSR^{[\mathfrak{s}^{(\ell+1)}, s^{(\ell+1)}+q-1]} \bigl( \mathfrak{m}^{(\ell+1)}, m^+,m^-, \La(m^+); \delta_0,\mathfrak{t}^{(\ell+1)} \bigr)$ with some $\mathfrak{m}^{(\ell+1)} \subset \La(m^+)$.

\item[(iii)] For any $m,m'$, either $\La(m) \cap \La(m') = \emptyset$, or $\La(m) = \La(m')$, in which case $m,m'$ are the principal points for $H_{\La(m), \ve}$. We use the notation $m' = \bullet m$ for the latter case. In particular, $\bullet m^+ = m^-$. In the former case we say that $\bullet m$ does not exist and $\{m,\bullet m\} = \{m\}$.

\item[(iv)] Let $m \in \cM$. There exists a unique real-analytic function $E(m,\La(m);\ve)$, $\ve \in (-\ve_{s-1}, \ve_{s-1})$ such that
$E(m,\La(m);\ve)$ is a simple eigenvalue of $H_{\La(m), \ve}$ and $E(m,\La(m);0) = v(m)$. Furthermore, let $m \in \cM \setminus \{m^+,m^-\}$ be arbitrary. The following estimates hold:
\begin{equation}\nn
\begin{split}
3 \delta^{(s^{(\ell+1)}+q-1)}_0 \le \min_{m'' \in \{m^+, m^-\}} \min_{m' \in \{m, \bullet m\}} |E \bigl( m'', \La(m''); \ve \bigr) - E \bigl( m', \La(m'); \ve \bigr)| \le \\
\max_{m'' \in \{m^+, m^-\}} \min_{m' \in \{m, \bullet m\}} |E \bigl( m'', \La(m''); \ve \bigr) - E \bigl( m', \La(m'); \ve \bigr)| \le
\delta^{(s^{(\ell+1)}+q-2)}_0 \quad \text{if $s(m) = s^{(\ell+1)}+q-1$}, \\
\frac{\delta^{(s(m))}_0}{2} \le \min_{m'' \in \{m^+, m^-\}} \min_{m' \in \{m, \bullet m\}} |E \bigl( m'', \La(m''); \ve \bigr) - E \bigl( m', \La(m'); \ve \bigr)| \le \\
\max_{m'' \in \{m^+, m^-\}} \min_{m' \in \{m, \bullet m\}} |E \bigl( m'', \La(m''); \ve \bigr) - E \bigl( m', \La(m'); \ve \bigr)| \le \delta^{(s(m)-1)}_0 \quad \text{if $s(m) < s^{(\ell)}+q-1$.}
\end{split}
\end{equation}

\item[(v)]
\be\nn
\begin{split}
\bigl( m + B(R^{(s(m))} \bigr) \subset \Lambda(m) \quad \text{for any $\Lambda(m)$}, \\
\bigl( m^\pm + B(R^{(s^{(\ell+1)}+q)} \bigr) \subset \Lambda.
\end{split}
\end{equation}

\item[(vi)] $|v(n) - v(m_0)| \ge 2 \delta_0^4$ for any $n \in \Lambda \setminus \bigcup_{m \in \cM} \La(m)$.

\item[(vii)] Due to the inductive argument, for any $\ve \in  (-\ve_{s^\zero-1}, \ve_{s^\zero-1})$ and any
\begin{equation} \label{eq:6EDomainFIN}
E \in \bigcup_{\pm} (E \bigl( m^\pm, \La(m^+);\ve \bigr) - 2 \delta^{(s^{(\ell+1)}+q-1)}_0, E \bigl( m^\pm, \La(m^+);\ve \bigr) + 2 \delta^{(s^{(\ell+1)}+q-1)}_0),
\end{equation}
the functions
\begin{equation} \label{eq:6-10acOPQDEFAPUUPFIN}
Q^{(s^{(\ell+1)}+q)}(m^\pm,\La; \ve, E) = \sum_{m',n' \in \La \setminus \{m^+,m^-\}} h(m^\pm,m';\ve) (E - H_{\La \setminus \{m^+,m^-\}})^{-1} (m',n') h(n',m^\pm;\ve)
\end{equation}
are well-defined. We require that
\begin{equation} \label{eq:6-13GSRUUPFIN}
v(m^+) + Q^{(s^{(\ell+1)}+q)}(m^+, \La,E) \ge v(m^-) + Q^{(s^{(\ell+1)}+q)}(m^-,\La; \ve, E) + \tau^{(\ell+1)}.
\end{equation}
\end{enumerate}

In this case we say that $\hle \in GSR^{[\mathfrak{s}^{(\ell+1)}, s^{(\ell+1)}+q]} \bigl( \mathfrak{m}^{(\ell+1)}, m^+,m^-,\La; \delta_0,\mathfrak{t}^{(\ell+1)} \bigr)$. We call $m^+,m^-$ the principal points. We set $s(m^\pm) = s^{(\ell+1)}+q$. We call $\La^{(s(m^\pm)-1)}(m^\pm)$ the $(s(m^\pm)-1)$-set for $m^\pm$.
\end{defi}

\begin{theorem}\label{th:6-4FIN}
Let $\hle \in GPR^{[\mathfrak{s}^{(\ell+1)},s^{(\ell+1)}+q]} \bigl( \mathfrak{m}^{(\ell+1)}, m^+,m^-,\La; \delta_0,\mathfrak{t}^{(\ell+1)} \bigr)$. The following statements hold:

$(1)$ Define inductively $D(x;\La) = D(x;\La \setminus \{m^+,m^-\}) = D(x;\La \setminus \mathfrak{m}^{(\ell+1)}) = D(x;\La(m))$ if $x \in \La(m) \setminus \mathfrak{m}^{(\ell+1)}$, $D(x;\La) = D(x;\La \setminus \{m^+,m^-\}) = 4 \log (\delta^{(s^{(\ell+1)}+q-1)}_0)^{-1}$ if $x \in \mathfrak{m}^{(\ell+1)} \setminus \{m^+,m^-\}$, and $D(x;\La) = 4 \log (\delta^{(s^{(\ell+1)}+q)}_0)^{-1}$ if $x \in \{m^+,m^-\}$. Then, $D(\cdot;\La \setminus \mathfrak{m}^{(\ell+1)}) \in \mathcal{G}_{\La \setminus \mathfrak{m}^{(\ell+1)},T,\kappa_0}$, $D(\cdot;\La \setminus \{m^+,m^-\}) \in \mathcal{G}_{\La \setminus \{m^+,m^-\},\IZ^\nu \setminus \{m^+,m^-\},T,\kappa_0}$, $D(\cdot;\La) \in \mathcal{G}_{\La,T,\kappa_0}$.

$(2)$ Let $\cL^{(s^{(\ell+1)}+q-1,\pm)} := \cL_\mathbb{R} \bigl( E \bigl( m^\pm, \La(m^+); \ve \bigr), 2 \delta_0^{(s^{(\ell+1)}+q-1)} \bigr)$. For any $(\ve,E) \in \cL^{(s^{(\ell+1)}+q-1,\pm)}$,
\begin{equation}\label{eq:6-9.FIN}
|(E - H_{\La \setminus \{m^+,m^-\},\ve})^{-1}(x,y)| \le s_{D(\cdot;\La \setminus \{m^+,m^-\}),T,\kappa_0,|\ve|;\La \setminus \{m+,m^-\},\mathfrak{R}}(x,y).
\end{equation}

$(3)$ The functions
\begin{equation} \label{eq:6-10acFIN}
\begin{split}
Q^{(s^{(\ell+1)}+q)}(m^\pm,\La; \ve, E) = \\
\sum_{m,n \in \La \setminus\{m^+,m^-\}} h(m^\pm,m;\ve) (E - H_{\La \setminus \{m^+,m^-\}})^{-1} (m,n) h(n,m^\pm;\ve), \\
G^{(s^{(\ell+1)}+q)}(m^\pm,m^\mp,\La; \ve, E) = h(m^\pm,m^\mp;\ve) + \\
\sum_{m,n \in \La \setminus \{m^+,m^-\}} h(m^\pm,m;\ve) (E - H_{\La \setminus \{m^+, m^-\}})^{-1} (m,n) h(n,m^\mp;\ve)
\end{split}
\end{equation}
are well-defined and and $C^2$-smooth in the domain $\cL^{(s^{(\ell+1)}+q-1,+)} \cup \cL^{(s^{(\ell+1)}+q-1,-)}$,
\begin{equation} \label{eq:6-12acACCOPQNEEW1}
\begin{split}
& \big| \partial^\alpha_E Q^{(s^{(\ell+1)}+q)}(m^\pm,\La; \ve, E) - \partial^\alpha_E Q^{(s^{(\ell+1)}+q-1)} \bigl( m^\pm, \La(m^+); \ve,E \bigr) \big| \\
& \le 4 |\ve|^{3/2} \exp \left( -\kappa_0 R^{(s^{(\ell+1)}+q-1)} \right) \le |\ve| (\delta_0^{(s^{(\ell+1)}+q-1)})^{12}, \\
& \big| \partial^\alpha_E G^{(s^{(\ell+1)}+q)}(m^\pm,m^\mp,\La; \ve, E) - \partial^\alpha_E G^{(s^{(\ell+1)}+q-1)}(m^\pm,m^\mp, \La(m^+); \ve, E) \big | \\
& \le 4 |\ve|^{3/2} \exp \left( -\kappa_0 R^{(s^{(\ell+1)}+q-1)} \right) \le |\ve| (\delta_0^{(s^{(\ell+1)}+q-1)})^{12}, \\
& \big| \partial^\alpha_E Q^{(s^{(\ell+1)}+q)}(m^\pm, \La; \ve, E) \big| \le |\ve|, \quad |E - v(m_0^\pm) - Q^{(s^{(\ell+1)}+q)}(m^\pm, \La; \ve, E)| < |\ve|, \\
& |\partial^\alpha_E G^{(s^{(\ell+1)}+q)}(m^+, m^-, \La;\ve,E) \big| \le 8 |\ve|^{3/2} \exp \left( -\frac{7\kappa_0}{8} |m^+ - m^-| \right) \le |\ve|^{3/2} \exp \left( -\kappa_0 R^{(s^{(\ell+1)}-1)} \right) \\
&\le |\ve| (\delta_0^{(s^{(\ell+1)}-1)})^{12}.
\end{split}
\end{equation}
Furthermore, set $\rho_0 = \delta^{(s^{(\ell+1)}+q-1)}$, $\rho_j = \rho_0$, $g_j = g_0$, $j = 1, \dots, \ell$,
\begin{equation}\label{eq:6QfformulaUPINTq}
\begin{split}
f_1(\ve,E) = E - v(m^+) - Q^{(s^{(\ell+1)}+q)}(m^+,\La; \ve,E ), \quad f_2(\ve,E) = E - v(m^-) - Q^{(s^{(\ell+1)}+q)}(m^-,\La; \ve,E ), \\
b^2(\ve,E) = |G^{(s^{(\ell+1)}+q)}(m^\pm,m^\mp,\La; \ve, E)|^2, \quad f(\ve,E) = f_1(\ve,E) - \frac{b^2(\ve,E)}{f_2(\ve,E)}
\end{split}
\end{equation}
Then, $f \in \mathfrak{F}^{(\ell+1)}_{\mathfrak{g}^{(\ell+1)}, 1-4^{-\ell}}(f_1,f_2,b^2)$, $\tau^{(f_j)} > \tau^{[\ell]}/4$, $\tau^{(f)} \ge \tau^{[\ell+1]}/4$, where $\tau^{[\ell+1]}$ is the same as in Proposition~\ref{prop:6-4INT}.

$(4)$ Let $(\ve,E) \in \cL^{(s^{(\ell+1)}+q-1,\pm)}$. Then, $E \in \spec H_{\La,\ve}$ if and only if $E$ obeys
\begin{equation} \label{eq:6-13FIN}
\begin{split}
& \chi(\ve,E) := \bigl( E - v(m^+) - Q^{(s^{(\ell+1)}+q)}(m^+, \La; \ve, E) \bigr) \cdot \bigl( E - v(m^-) - Q^{(s^{(\ell+1)}+q)}(m^-,\La; \ve, E) \bigr) \\
& \qquad - |G^{(s^{(\ell+1)}+q)}(m^+, m^-, \La; \ve, E)|^2 = 0.
\end{split}
\end{equation}

$(5)$  Let $f$ be as in part $(3)$ and let $\chi^{(f)}$ be as in Definition~\ref{def:4a-functions}. Then, $\chi(\ve,E) = 0$ if and only if $\chi^{(f)} = 0$. For $\ve \in  (-\ve_{s^\zero-1}, \ve_{s^\zero-1})$, the equation
\begin{equation}\label{eq:8-13nnINT-2}
\chi^{(f)}(\ve,E) = 0
\end{equation}
has exactly two solutions $E(m^+, \La; \ve) > E(m^-, \La; \ve)$, which obey
\begin{equation} \label{eq.6Eestimates1APN1FIN}
|E(m^\pm, \La; \ve) - E(m^\pm,\La(m^+);\ve)| < 4|\ve| (\delta^{(s^{(\ell+1)}+q-1)}_0)^{1/8},
\end{equation}
\begin{equation} \label{eq.6Eestimates1APFINALPM}
\begin{split}
[ a_1 (\ve, E) + |b (\ve, E)| ]|_{E = E (m^+, \La; \ve)} \ge E (m^+, \La; \ve) \\
\ge \max ( a_1 (\ve, E),  a_2 (\ve, E) + |b (\ve, E)| )|_{E = E (m^+, \La; \ve)}, \\
[ a_2 (\ve, E) - |b (\ve, E)| ]|_{E = E (m^+, \La; \ve)} \le E (m^-, \La; \ve) \\
\le \min ( a_2 (\ve, E),  a_1 (\ve, E) - |b (\ve, E)| )|_{E = E (m^-, \La; \ve)},
\end{split}
\end{equation}
where
\begin{equation}
\begin{split}\nn
a_1 (\ve, E) = v(m^+) + Q^{(s^{(\ell+1)}+q)}(m^+, \La; \ve, E), \quad \quad a_2 (\ve, E) = v(m^-) + Q^{(s^{(\ell+1)}+q)}(m^-, \La; \ve, E), \\
b (\ve, E) = |b_1 (\ve,E)|, \quad b_1 (\ve, E) = G^{(s^{(\ell+1)}+q)}(m^+, m^-, \La; \ve, E).
\end{split}
\end{equation}

The functions $E(m^\pm, \La; \ve)$ are $C^2$-smooth on the interval $(-\ve_{s_0-1}, \ve_{s_0-1})$ and obey the estimates \eqref{eq:6-13acnq0N1INT}.

$(6)$
\begin{equation} \label{eq:6specHEEN1FIN}
\begin{split}
\spec H_{\La,\ve} \cap \Big\{ |E - E(m^+,\La(m^+);\ve,E)| < \frac{(\delta^{(s^{(\ell+1)}+q-1)}_0)^{1/2}}{2} \Big\} = \{ E(m^+, \La; \ve), E(m^-, \La; \ve) \}, \\
E(m^\pm, \La;0) = v(m^\pm).
\end{split}
\end{equation}
Furthermore, assume
\begin{equation}\label{eq:6EsplitspecconddomainqUFIN}
(\delta^{(s^{(\ell+1)}+q)}_0)^4 < \min_\pm |E - E(m^\pm, \La; \ve)| < 2 \delta^{(s^{(\ell+1)}+q-1)}_0, E \in \IR.
\end{equation}
Then the matrix $(E - H_{\La,\ve})$ is invertible. Moreover,
\begin{equation}\label{eq:5inverseestiMATEqFIN}
|[(E - H_{\La,\ve})^{-1}](x,y)| \le s_{D(\cdot;\La),T,\kappa_0,|\ve|;k,\La,\mathfrak{R}}(x,y).
\end{equation}

$(7)$ Let $\vp^{(\pm)}(\La; \ve) := \vp^{(\pm)}(\cdot,\La; \ve)$ be the eigenvector corresponding to $E(m^\pm, \La; \ve)$ and normalized by $\vp^{(\pm)}(m^\pm,\La; \ve) = 1$. Then,
\begin{equation} \label{eq:6-17evdecay}
\begin{split}
|\vp^{(\pm)}(n, \La; \ve)| \le |\ve|^{1/2} \sum_{m \in \mathfrak{m}^{(\ell)}} \exp \left( -\frac{7}{8} \kappa_0|n-m| \right), \quad \text{ $n \notin \mathfrak{m}^{(\ell)}$}, \\
|\vp^{(\pm)}(m, \La; \ve)| \le 1 + \sum_{0 \le t < s^{(\ell+1)}+q} 4^{-t} \quad \text{for any $m \in \mathfrak{m}^{(\ell)}$.}
\end{split}
\end{equation}
For any $n \in \La(m^+) $, we have
\begin{equation} \label{eq:6-21ACCUPSFINCON}
|\vp^{(\pm)}(n,\La;\ve) -\vp^{(\pm)}(n,\La(m^+);\ve)| \le 2 |\ve| (\delta_0^{(s^{(\ell)}+q-1)})^5.
\end{equation}
\end{theorem}

\begin{proof}
The proof of each of the statements $(1)$--$(6)$ is completely similar to the proof of either a statement from Proposition~\ref{rem:con1smalldenomnn} or a statement from Proposition~\ref{prop:6-4}, and we omit them. Let us prove $(7)$. We discuss the cases $\ell \le 2$. For $\ell > 2$, the proof is completely similar. Let $\ell = 1$. We follow for this case the notation from Proposition~\ref{rem:con1smalldenomnn}. In particular $\mathfrak{m}^{(1)} = (m^+_0,m^-_0)$. Due to part $(7)$ in Proposition~\ref{rem:con1smalldenomnn}, the eigenvectors $\vp^{(s^\one,\pm)}(\La; \ve)$, normalized by $\vp^{(s^{(1)},\pm)}(m^\pm_0, \La; \ve) = 1$, obey \eqref{eq:6-17evdecay} with $\ell = 1$, $q = 0$. Let $\vp^{(s^{(1)}-1)}(\La^{(s^{(1)}-1)}(m^+_0);\ve)$ be the vector defined in part $(6)$ of Proposition~\ref{prop:4-4} with $H_{\La^{(s^{(1)}-1)}(m^+), \ve}$ in the role of $\hle$. Set $\widetilde{\vp}(n) = \vp^{(s^{(1)},+)}(n, \La; \ve)$, $n \in \La^{(s^{(1)}-1)}(m^+_0)$. Recall that $m^+_0 + B(R^{(s^\one-1)}) \subset \La^{(s^{(1)}-1)}(m^+_0)$ and $m^-_0 \notin \La^{(s^{(1)}-1)}(m^+_0)$. Therefore, using \eqref{eq:6-17evdecay}, one obtains
\begin{equation} \label{eq:6-17AAAAevcloseBBB}
\| (E^{(s^\one,+)}(\La; \ve) - H_{\La^{(s^{(1)}-1)}(m^+_0), \ve}) \widetilde{\vp} \| \le \exp \left( -\kappa_0 R^{(s^\one-1)} \right).
\end{equation}
It follows from part $(4)$ of Proposition~\ref{prop:4-4} that
\begin{equation} \label{eq:6-17AAAAevcloseAA}
\spec H_{\La^{(s^{(1)}-1)}(m^+_0)} \cap \{ |E^{(s^{(1)}-1)}(m^+_0,\La^{(s^{(1)}-1)}(m^+_0);\ve) - E| < \delta^{(s^\one-1)}_0\} = \{ E^{(s^{(1)}-1)}(m^+_0,\La^{(s^{(1)}-1)}(m^+_0);\ve) \}.
\end{equation}
Clearly, $\|\vp^{(s^{(1)}-1)}(\cdot,\La^{(s^{(1)}-1)}(m^+_0);\ve)\|, \|\widetilde{\vp}\| \ge 1$. Combining \eqref{eq:6-17AAAAevcloseBBB}, \eqref{eq:6-17AAAAevcloseAA} with standard perturbation theory arguments, one concludes that there exists $\zeta$ with $|\zeta| = 1$ such that
\begin{equation} \label{eq:6-17AAAAevclose}
\| \zeta \frac{\vp^{(s^{(1)}-1)}(\La^{(s^{(1)}-1)}(m^+_0);\ve)}{\|\vp^{(s^{(1)}-1)}(\La^{(s^{(1)}-1)}(m^+_0);\ve)\|} - \frac{\widetilde{\vp}}{\|\widetilde{\vp}\|} \| \le 2 \frac{\exp(-\kappa_0 R^{(s^\one-1)})}{\delta^{(s^\one-1)}_0} < \exp \left( -\frac{\kappa_0}{2} R^{(s^\one-1)} \right).
\end{equation}
Note that $\|\vp^{(s^{(1)}-1)}(\cdot,\La^{(s^{(1)}-1)}(m^+_0);\ve)\| \le 2$. Since $\vp^{(s^{(1)}-1)}(m^+_0,\La^{(s^{(1)}-1)}(m^+_0);\ve) = 1$, $\widetilde{\vp}(m^+_0) = 1$, one concludes that $\|\vp^{(s^{(1)}-1)}(\La^{(s^{(1)}-1)}(m^+_0);\ve) - \widetilde{\vp}\| \le \exp (-\frac{\kappa_0}{2} R^{(s^\one-1)})$, as claimed. This finishes part $(7)$ for $\ell = 1$, $q = 0$. The case $\ell = 1$, $q > 0$ is similar.

Let $\ell = 2$, $q = 0$. Using the notation from Definition~\ref{def:9-6}, assume that \eqref{eq:6-9-13OPQ} holds, $m^+ := m^+_0$, $m^-_ := m^+_{j_0}$. Recall that $\mathfrak{m}^{(2)} = \mathfrak{m} = ((m^+_0,m^-_0), (m^+_{j_0},m^-_{j_0}))$. Due to Proposition~\ref{prop:6-4},
\eqref{eq:6-9.Hinvestimatestatement1PQ} holds. As above, \eqref{eq:6-17evdecay} follows from \eqref{eq:6-9.Hinvestimatestatement1PQ} and Lemma~\ref{lem:auxweight1}. As in the proof of part $(7)$ of Proposition~\ref{rem:con1smalldenomnn}, one obtains
\begin{equation}\label{eq:6Hinvestimatest1PQPM}
\begin{split}
Res[(E - H_{\La})^{-1}(n,m^\pm)]|_{E = E^{(s^\one,\pm)}(m^+_0, \La; \ve)} = - \alpha^\pm \sum_{x \in \La \setminus \{m^+,m^-\}} \\
(E^{(s^\one,\pm)}(m^+, \La; \ve) - H_{\La \setminus \{m^+,m^-\}})^{-1}(n,x) [h(x,m^\pm;\ve) + h(x,m^\mp;\ve) \beta^\pm], \quad n \in \La \setminus \{m^+,m^-\}, \\
Res[(E - H_{\La})^{-1}(m^\pm,m^\pm)]|_{E = E^{(s^\one,\pm)}(m^+, \La; \ve)} = \alpha^\pm, \\
Res[(E - H_{\La})^{-1}(m^\pm,m^\mp)]|_{E = E^{(s^\one,\pm)}(m^+, \La; \ve)} = \alpha^\pm \iota^\pm
\end{split}
\end{equation}
with $0 < |\alpha^\pm| \le 1$, $|\beta^\pm|, |\iota^\pm| \le 1$. In particular,
\begin{equation}\label{eq:6drevectors}
\begin{split}
\hle \vp^{(s^\one,\pm)}(\La; \ve) = E^{(s^\one,\pm)}(m^+, \La; \ve) \vp^{(s^\one,\pm)}(\La; \ve), \\
\vp^{(s^\one,\pm)}(\La; \ve)(m^\pm) = 1, \quad \vp^{(s^\one,\pm)}(\La; \ve)(m^\mp) = \iota^\pm, \\
\vp^{(s^\one,\pm)}(\La; \ve) := (\alpha^\pm)^{-1} \bigl(Res[(E - H_{\La})^{-1}(n,m^\pm)]|_{E = E^{(s^\one,\pm)}(m^+_0, \La; \ve)} \bigr)_{n \in \La}.
\end{split}
\end{equation}
The rest of the arguments for part $(7)$ is completely similar to case $\ell = 1$.
\end{proof}

\begin{remark}\label{rem:6.principaldecay}
Assume $\hle \in GSR^{[\mathfrak{s}^{(\ell+1)}, s^{(\ell+1)}+q]} \bigl( \mathfrak{m}^{(\ell+1)}, m^+,m^-,\La; \delta_0,\mathfrak{t}^{(\ell+1)} \bigr)$. Definitions~\ref{def:9-6general} and \ref{def:6-6UPFIN} do not require any upper estimate for the quantity $\diam (\mathfrak{m}^{(\ell+1)})$. This estimate is needed for an effective application of the estimate \eqref{eq:6-17evdecay} on $|\vp^{(\pm)}(n, \La; \ve)|$ from Theorem~\ref{th:6-4FIN}. In applications we always assume the following condition,
\begin{equation}\label{eq:6princsetdiam}
\begin{split}
|m^+ - m^-| < R^{(s^{(\ell+1)}+1)}/4, \\
\mathfrak{m}^{(\ell+1)} \subset \bigcup_{+,-} \bigl( m^\pm + B(R^{(s^{(\ell+1)}-1)}/4) \bigr).
\end{split}
\end{equation}
Although the first line implies the second one, it is convenient to keep it this way for the sake of referring to them. Recall that due to Definition~\ref{def:6-6UPFIN}, $\bigl( m^\pm + B(R^{(s^{(\ell+1)}+q)} \bigr) \subset \Lambda$. Therefore, \eqref{eq:6-17evdecay} combined with \eqref{eq:6princsetdiam} yields
\begin{equation} \label{eq:6-17evdecayrefined}
|\vp^{(\pm)}(n, \La; \ve)| \le |\ve|^{1/2} 2^{\ell+2} \exp \left( -\frac{7}{8} \kappa_0 R^{(s^{(\ell+1)}+q)} \right), \quad \text{ $n \in \La \setminus \bigcup_{+,-} \bigl( m^\pm + B(R^{(s^{(\ell+1)}+q)})$.}
\end{equation}
\end{remark}

Let $H_{\La_j, \ve} \in GSR^{[\mathfrak{s}^{(\ell+1)}, s^{(\ell+1)}+q]} \bigl( \mathfrak{m}^{(\ell+1)}, m^+,m^-,\La_j; \delta_0, \mathfrak{t}^{(\ell+1)} \bigr)$, $j = 1,2$, with the same principal set $\mathfrak{m}^{(\ell+1)}$ and with the same principal points $m^+,m^-$. We denote by $v(n,j)$ the diagonal entries of $H_{\La_j,\ve}$. We assume that $v(n,1) = v(n,2)$ for $n \in \La_1 \cap \La_2$. Let $E \bigl(m^\pm, \La_j; \ve \bigr)$ be the eigenvalue defined in Theorem~\ref{th:6-4FIN} with $H_{\La_j, \ve}$ in the role of $\hle$, $j = 1,2$.

\begin{corollary}\label{cor:6.twolambdas1}
Assume that condition \eqref{eq:6princsetdiam} holds for $\La = \La_j$, $j = 1,2$. Then,
\begin{equation} \label{eq:6.twolambdas1}
\big|E \bigl(m^\pm, \La_1; \ve \bigr) -  E \bigl( m^\pm, \La_2; \ve \bigr) \big| < |\ve| (\delta_0^{(s^{(\ell+1)}+q)})^5.
\end{equation}
\end{corollary}

\begin{proof}
The proof is similar to the proof of Corollary~\ref{cor:5.twolambdas1}. However, since the eigenvalues $E \bigl( m^\pm, \La_j; \ve \bigr)$ are almost double degenerate, some additional arguments are required. Let $\vp^{(\pm)}(\La_j; \ve)$ be the vector defined in part $(7)$ of Theorem~\ref{th:6-4FIN} with $H_{\La_j, \ve}$ in the role of $\hle$. Set $\widetilde{\vp}^{\pm}(\La_2; \ve)(n) = \vp^{(\pm)}(\La_1; \ve)(n)$ if $n \in \La_2 \cap \La_1$, $\widetilde{\vp}^{\pm}(\La_2; \ve)(n) = 0$ otherwise. It follows from Remark~\ref{rem:6.principaldecay} that
\begin{equation} \label{eq:6-17AAAAevcloseE1E2}
\| (E(m^\pm, \La_1; \ve) - H_{\La_2, \ve}) \widetilde{\vp}^\pm \| \le \exp \left( -\frac{7\kappa_0}{8} R^{(s^{(\ell+1)}+q)} \right).
\end{equation}
Since $\| \widetilde{\vp}^\pm \|\ge |\widetilde{\vp}^\pm(m^\pm)| = 1$, one has
\begin{equation} \label{eq:6-17AAAAevclose-2}
\dist (E \bigl( m^\pm,\La_1; \ve \bigr), \spec H_{\La_2, \ve}) \le \exp \left( -\frac{7\kappa_0}{8} R^{(s^{(\ell+1)}+q)} \right).
\end{equation}
Since the principal set $\mathfrak{m}^{(\ell+1)}$ is the same for both $\La_j$,  one can use induction, like in the proof of Corollary~\ref{cor:5.twolambdas1}, to verify that in fact
\begin{equation} \label{eq:6-17AAAAevclose1}
\dist (E \bigl( m^\pm, \La_1; \ve \bigr), \{ E \bigl( m^+,\La_2; \ve \bigr), E \bigl( m^-,\La_2; \ve \bigr) \}) \le \exp \left( -\frac{7\kappa_0}{8} R^{(s^{(\ell+1)}+q)} \right).
\end{equation}
If $E \bigl( m^+, \La_2; \ve \bigr) - E \bigl( m^-, \La_2; \ve \bigr) < \exp (-\frac{\kappa_0}{3} R^{(s^{(\ell+1)}+q)})$, then we are done. Assume $E \bigl( m^+, \La_2; \ve \bigr) - E \bigl( m^-, \La_2; \ve \bigr) \ge \exp (-\frac{\kappa_0}{3} R^{(s^{(\ell+1)}+q)})$. Assume $|E\bigl(m^+, \La_1; \ve \bigr) - E\bigl(m^-, \La_2; \ve \bigr)| < \exp (-\frac{7\kappa_0}{8} R^{(s^{(\ell)}+q)})$. Since $(\vp^{(+)}(\La_1; \ve), \vp^{(-)}(\La_1; \ve)) = 0$, it follows from Remark~\ref{rem:6.principaldecay} that $(\widetilde{\vp}^{+}(\La_2; \ve), \widetilde{\vp}^{-}(\La_2; \ve))| < \exp (-\frac{7\kappa_0}{8} R^{(s^{(\ell)}+q)})$. Since $\|\widetilde{\vp}^{+}(\La_2; \ve)\| \ge 1$, combined with \eqref{eq:6-17AAAAevcloseE1E2} this implies $|(\spec H_{\La_2, \ve}) \cap \{ |E - E \bigl( m^-,\La_2; \ve \bigr)| < \exp (-\frac{\kappa_0}{2} R^{(s^{(\ell)}+q)}) \}| \ge 2$. However, $E \bigl( m^+, \La_2; \ve \bigr)$ is the only eigenvalue of $H_{\La_2, \ve}$ different from $E \bigl(m^-, \La_2; \ve \bigr)$ that may belong to $\{ |E - E \bigl( m^-, \La_2; \ve \bigr)| < \exp (-\frac{\kappa_0}{2} R^{(s^{(\ell)}+q)}) \}$. This contradicts the assumption $|E \bigl(m^+ \La_2; \ve \bigr) - E \bigl(m^-, \La_2; \ve \bigr)| \ge \exp (-\frac{\kappa_0}{3} R^{(s^{(\ell)}+q)})$. Thus, $|E \bigl( m^\pm, \La_1; \ve \bigr) - E \bigl( m^-, \La_2; \ve \bigr)| < \exp (-\frac{7\kappa_0}{8} R^{(s^{(\ell)}+q)})$ is impossible. Similarly, $|E \bigl( m^\pm, \La_1; \ve \bigr) - E \bigl( m^+,\La_2; \ve \bigr)| < \exp (-\frac{7\kappa_0}{8} R^{(s^{(\ell)}+q)})$ is impossible. Since $E \bigl(m^+, \La_j; \ve \bigr) > E \bigl( m^-, \La_j; \ve \bigr)$, the statement follows from \eqref{eq:6-17AAAAevclose1}.
\end{proof}

\smallskip

Using the notation from Theorem~\ref{th:6-4FIN}, assume that the functions $h(m, n,\ve)$, $m, n \in \La$, depend also on some parameter $k \in(k_1,k_2)$, that is, $h(m,n;\ve) = h(m, n;\ve,k)$. Assume that $H_{\La,\varepsilon,k} := \bigl( h(m, n; \ve,k) \bigr)_{m, n \in \La} \in   GSR^{[\mathfrak{s}^{(\ell+1)}, s^{(\ell+1)}+q]} \bigl( \mathfrak{m}^{(\ell+1)}, m^+,m^-,\La; \delta_0,\mathfrak{t}^{(\ell+1)} \bigr)$ for all $k$. Let $Q^{(s^{(\ell+1)}+q)}(m^\pm,\La; \ve,k, E)$ etc.\ be the functions introduced in Theorem~\ref{th:6-4FIN} with $H_{\La,\varepsilon,k}$ in the role of $\hle$.

\begin{lemma}\label{lem:7differentiationFIN}
$(1)$ If $h(m,n;\ve,k)$ are $C^t$-smooth functions of $k$, then $Q^{(s^{(\ell+1)}+q)}(m^\pm,\La; \ve, E)$ etc.\ are $C^t$-smooth functions of all variables involved.

$(2)$ Assume also that $h(m,n;\ve,k)$ are $C^2$-smooth functions that for $m \neq n$ obey $|\partial^\alpha h(m,n;\ve,k)| \le B_0 \exp (-\kappa_0 |m - n|)$ for $|\alpha| \le 2$. Furthermore, assume that $|\partial^\alpha h(m,m;\ve,k)| \le B_0 \exp (\kappa_0 |m - m^+|^{1/5})$ for any $m \in \La$, $0 < |\alpha| \le 2$. Then, for $|\alpha| \le 2$, we have
\begin{equation}\label{eq:7Hinvestimatestatement1kvard}
\begin{split}
|\partial^\alpha (E - H_{\La \setminus \{ m^+, m^- \}, k})^{-1}] (x,y)| \le (3 B_0)^\alpha \mathfrak{D}^\alpha_{D(\cdot; \La \setminus \{ m^+, m^- \}), T, \kappa_0, |\ve|; \La \setminus \{ m^+, m^- \} } (x,y), \\
|\partial^\alpha Q^{(s+q)}(m^\pm, \La; \ve, k, E)| \le (3 B_0)^\alpha |\ve| \mathfrak{D}^\alpha_{D(\cdot; \La \setminus \{ m^+, m^- \} ), T, \kappa_0, |\ve|; \La \setminus \{ m^+, m^- \} } (m^\pm, m^\pm) < (3 B_0)^\alpha |\ve|^{3/2}, \\
|\partial^\alpha G^{(s+q)}(m^\pm, m^\mp, \La; \ve, k, E)| \le (3 B_0)^\alpha \mathfrak{D}^\alpha_{D(\cdot; \La \setminus \{ m^+, m^- \}), T, \kappa_0, |\ve|; \La \setminus \{ m^+, m^- \} } (m^\pm, m^\mp) \\
< (3B_0)^\alpha |\ve|^{1/2} \exp (-\kappa_0 |m^+-m^-/16|),
\end{split}
\end{equation}
\begin{equation}\label{eq:7Hinvestimatestatement1kvardE}
|\partial^\alpha E(m^\pm,\La; \ve,k, E) - \partial^\alpha v(m^\pm,k)| < (3B_0)^\alpha |\ve|^{3/2}.
\end{equation}
\end{lemma}

The proof of this statement is completely similar to the proof of Lemma~\ref{lem:5differentiation} and we skip it.

\section{Matrices with Inessential Resonances Associated with  $1$-Dimensional Quasi-Periodic Schr\"odinger Equations}\label{sec.7}

Let $c(n)$, $n \in \zv\setminus \{0\}$ obey
\begin{equation}\label{eq:7-5-4}
\begin{split}
\overline{c(n)} & = c(-n)\ ,\quad n \in \zv, \\
|c(n)| & \le \exp(-\kappa_0|n|)\ , \quad n \in \zv,
\end{split}
\end{equation}
where $0 < \kappa_0 \le 1/2$ is a constant.

Fix an arbitrary $\gamma \ge 1$. Given $\gamma - 1 \le |k| \le \gamma$ and $\epsilon > 0$, set $\lambda = 256 \gamma$ and consider $\ve$ with $|\ve| = \lambda^{-1} \epsilon$. With
\begin{equation} \label{eq:7-5-7}
\begin{split}
v(n; k) & = \lambda^{-1} (n \omega + k)^2\ , \quad n \in \zv\ ,\\
h_0(n, m) & = \lambda^{-1} c(n - m), \\
h(n, m; \ve, k) & = v(n; k)\ \text{if}\ m = n, \\
h(n, m; \ve, k) & = \ve\, h_0(n, m)\ \text{if}\ m \not= n,
\end{split}
\end{equation}
consider $H_{\ve, k} = \bigl( h(m, n; \ve, k) \bigr)_{m, n \in \zv}$. This is consistent with the notation in \eqref{eq:2-1}--\eqref{eq:2-4} of Section~\ref{sec.3} with
\begin{equation}\label{eq:7-5-7BBB}
B_1 = \lambda^{-1}.
\end{equation}
We denote by $H_{\La'; \ve, k}$ the submatrices $\bigl( h(m, n; \ve, k) \bigr)_{m,n \in \La'}$, $\La' \subset \zv$. We assume that the vector $\omega$ satisfies the following Diophantine condition,
\begin{equation}\label{eq:7-5-8}
|n \omega| \ge a_0 |n|^{-b_0}, \quad n \in \zv \setminus \{0\},
\end{equation}
with some $0 < a_0 < 1$, $\nu < b_0 < \infty$. Just for the sake of normalization of some estimates in this section, we assume that $\| \omega \| \le 1$, so that $|m \omega| \le |m|$ for any $m \in \IZ^\nu$.

Let $a_0$, $b_0$ be as in \eqref{eq:7-5-8}. Set $b_1 = 32 b_0$, $\beta_1 = b_1^{-1} = (32 b_0)^{-1}$. Fix an arbitrary $R_1$
with $\log R_1 \ge \max (\log (100 a_0^{-1}) , 2^{34}\beta_1^{-1} \log \kappa_0^{-1})$. Fix also $k \in \IR$. Set
\begin{equation}\label{eq:A.1A}
R^\one = R_1, \quad \delta_0^4 := \delta^\zero_0 = (R^\one)^{-\frac{1}{\beta_1}}, \quad \delta_0^{(u-1)} = \exp \bigl( - (\log R^{(u-1)})^2 \bigr), \quad u = 2, \dots, \quad R^{(u)} := \bigl( \delta_0^{(u-1)} \bigr)^{-\beta_1}.
\end{equation}
\textit{Let us remark here that the definition \eqref{eq:A.1A} is consistent with \eqref{eq:3basicparameters}. In particular, $\log \delta_0^{-1} > 2^{32} \beta_1^{-1} \log \kappa_0^{-1}$. Another remark is that, due to the Diophantine condition, one has}
\begin{equation}\label{eq:diphnores}
|m \omega| \ge a_0  |m|^{-b_0} \ge a_0 ( 48 R^{(u)})^{-b_0} > (R^{(u)})^{-2b_0} = (\delta_0^{(u-1)})^{1/16} \; \text{ if $0 < |m| \le 48 R^{(u)}$}.
\end{equation}
Define
\begin{equation}\label{eq:7K.1}
\begin{split}
k^\pm_m = - \frac{m \omega}{2} \pm \sigma(m) \; \text{ with } \; \sigma(m) = 32 (\delta^\esone_0)^{1/6} \; \text{ if $12 R^\esone < |m| \le 12 R^\es$ and } \; \sigma(0) = 32 (\delta^{(0)}_0)^{1/6}, \\
k^\pm_{m,s} = k^\pm_m \pm 64 \sum_{r \le s-1, \quad (\delta^{(r)}_0)^{1/2} \le \sigma(m)} (\delta^\ar_0)^{1/2}, \quad s \ge 1, \quad  k^\pm_{m,0} := k^\pm_m,
\end{split}
\end{equation}
where $R^\zero := 0$. Note the following identities,
$$
k^\pm_{-m} = - k^\mp_m, \quad k^\pm_{-m,s} = - k^\mp_{m,s}.
$$

\begin{lemma}\label{lem:7.setKs}
$(1)$ For $|m| \le 12 R^\one$, the intervals $(k_{m,2}^-, k_{m,2}^+)$  are disjoint. We denote by $I_j(s)$ the connected components of $\IR \setminus \bigcup_{0 < |m'| \le 12 R^\es} (k_{m',s+1}^-, k_{m',s+1}^+)$.

$(2)$ For $s \ge 2$, each $I_j(s)$ is a subinterval of some $I_k(s-1)$.

$(3)$ For $j \neq k$, $\dist(I_j(s),I_k(s)) \ge 64 (\delta^\esone_0)^{1/6}$.
\end{lemma}

\begin{proof}
All statements follow readily from the definitions \eqref{eq:7K.1}.
\end{proof}

\begin{lemma}\label{lem:A.2}
\begin{itemize}

\item[(1)] Let $m \in \zv$, $0 < \delta < 1/16$ be arbitrary. If $|v(m,k) - v(0,k)| < \delta^2$, then $\min (|m \omega| ,|2 k + m  \omega|) \le 32 \delta$ if $\gamma \le 4$ and $\min (|m \omega| ,|2 k + m  \omega|) \le 256 \delta^2$ if $\gamma > 4$.

\item[(2)] If $\min (|m \omega|, |2 k + m \omega|) < \delta<1$, then $|v(m,k) - v(0,k)| \le  \delta$.

\end{itemize}

Assume $s \ge 2$ and $k \in \IR \setminus \bigcup_{0 < |m'| \le 12 R^\es} (k_{m',s}^-, k_{m',s}^+)$. Then,
\begin{itemize}

\item[(3)] If $\min (|m \omega|, |2 k + m \omega|) < 32(\delta^{(s-1)}_0)^{1/2}$, then $k + m \omega \in \IR \setminus \bigcup_{|m'| \le 12 R^\es} (k_{m',s-1}^-, k_{m',s-1}^+)$.  Moreover, if in addition $\sgn (k + m\omega) = \sgn k$, then $k, k + m\omega$ belong to the same connected component of $\IR \setminus \bigcup_{|m'| \le 12 R^\es} (k_{m',s-1}^-, k_{m',s-1}^+)$. In particular, if $|v(m,k) - v(0,k)| < \delta^{(s-1)}_0$, then $k + m \omega \in \IR \setminus \bigcup_{|m'| \le 12 R^\es} (k_{m',s-1}^-, k_{m',s-1}^+)$.

\item[(4)] If $0 < |m_1 - m_2| \le 12 R^\es$, then $\max_j |v(m_j,k_1) - v(0,k_1)| \ge (\delta^{(s-1)}_0)^{1/2}$ for any $|k_1 - k| < (\delta^{(s-1)}_0)^{1/2}$. In particular, if $0 < |m_2| \le 12 R^\es$, then $|v(m_2,k_1) - v(0,k_1)| \ge (\delta^{(s-1)}_0)^{1/2}$.

\end{itemize}
\end{lemma}

\begin{proof}
$(1)$ One has $|v(m,k) - v(0,k)| = \lambda^{-1}|m \omega| \cdot |2 k + m \omega|$. Hence $\min (\lambda^{-1/2} |m\omega|, \lambda^{-1/2} |2k + m\omega|) < \delta$. So, if $\gamma \le 4$, $\lambda \le 2^{10}$ and the claim holds. Assume now $\gamma > 4$. Assume for instance $\lambda^{-1/2} |2k + m\omega| < \delta$. In this case, $|m\omega| > 2|k| - \delta \lambda^{1/2} > 2 \gamma - 2 - \gamma^{1/2} > \gamma$. Hence, $|2k + m\omega| < \frac{\lambda}{\gamma} |v(m,k) - v(0,k)| < 256 \delta^2$. If $\lambda^{-1/2} |m\omega| < \delta$, then $|2k + m\omega| > 2|k| - \delta \lambda^{1/2} > 2\gamma - 2 - \gamma^{1/2} > \gamma$. As before it follows that $|m\omega| < 256 \delta^2$.

$(2)$ Assume that $|m \omega| < \delta$. Then, since $\lambda > 4 \max(|k|, 1)$, one has
$$
|v(m,k) - v(0,k)| \le \lambda^{-1} (2|k| + \delta)\delta < \delta.
$$
This verifies $(2)$ in this case. The verification in the second possible case is similar.

$(3)$ Assume that $|m \omega| < 32(\delta^{(s-1)}_0)^{1/2}$, $k \in \IR \setminus \bigcup_{0< |m'| \le 12 R^\es} (k_{m',s}^-, k_{m',s}^+)$. Recall that $\sigma(m') \ge 32(\delta^\esone_0)^{1/6} > 32 (\delta^\esone_0)^{1/2}$ if $0 < |m'| \le 12 R^\es$. Hence,
$$
k + m \omega \in \IR \setminus \bigcup_{0 < |m'| \le 12 R^\es} (k_{m',s}^- + 32(\delta^{(s-1)}_0)^{1/2}, k_{m',s}^+ - 32(\delta^{(s-1)}_0)^{1/2}) \subset \IR \setminus \bigcup_{0 < |m'| \le 12 R^\es} (k_{m',s-1}^-, k_{m',s-1}^+).
$$
Assume that $|(k + m \omega) - (-k)| \le 32(\delta^\esone_0)^{1/2}$. Since $-\bigcup_{0 < |m'| \le 12 R^\es} (k_{m',s}^-, k_{m',s}^+) = \bigcup_{0 < |m'| \le 12 R^\es} (k_{m',s}^-, k_{m',s}^+)$, one has $-k \in \IR \setminus \bigcup_{0 < |m'| \le 12 R^\es} (k_{m',s}^-, k_{m',s}^+)$. Therefore, $k + m \omega \in \IR \setminus \bigcup_{0 < |m'| \le 12 R^\es} (k_{m',s-1}^-, k_{m',s-1}^+)$.  Moreover, if $\sgn (k+m\omega)=\sgn k$, then $k, k + m\omega$ belong to the same connected component of $\IR \setminus \bigcup_{|m'| \le 12 R^\es} (k_{m',s-1}^-, k_{m',s-1}^+)$.
This finishes the proof of the first statement in $(3)$. The last statement in $(3)$ follows from the first one with part $(1)$ of the current lemma taken into account.

$(4)$ Recall that due to the Diophantine condition, $|(m_2 - m_1) \omega| \ge a_0 (1 + |m_2 - m_1|)^{-b_0} \ge \sigma (m_2 - m_1)$. Let $k \in \IR \setminus \bigcup_{0 < |n| \le 12 R^\es} (k^-_{n,s}, k^+_{n,s})$. We prove $(4)$ first for $k_1 = k$. Assume that $\max_j |v(m_j,k) - v(0,k)| < 9 (\delta^{(s-1)}_0)^{1/2}$. Then, due to part $(1)$, one has $\min (|m_j \omega|, |2 k + m_j \omega|) \le 32 \cdot 3 (\delta^{(s-1)}_0)^{1/4} < 128 (\delta^{(s-1)}_0)^{1/4}$, $j = 1, 2$. If $|m_j \omega| < 128 (\delta^{(s-1)}_0)^{1/4}$, $j = 1, 2$, then $256 (\delta^{(s-1)}_0)^{1/4} > |(m_2 - m_1) \omega| > \sigma (m_2 - m_1)$. Due to \eqref{eq:7K.1}, this implies $|m_2 - m_1| > 12 R^\es$, contrary to the assumption in $(4)$. Similarly, if $|2 k + m_j \omega| < 128 (\delta^{(s-1)}_0)^{1/4}$, $j = 1, 2$, then $|m_2 - m_1| > 12 R^\es$. Assume now that, for instance, $|m_1 \omega| < 128 (\delta^{(s-1)}_0)^{1/4}$ and $|2 k + m_2 \omega| < 128 (\delta^{(s-1)}_0)^{1/4}$. Then, $|2 k + (m_2 - m_1) \omega| < 256 (\delta^{(s-1)}_0)^{1/4} < \sigma ((m_2 - m_1))$, since $|m_2 - m_1| \le 12 R^\es$. Hence,
\begin{equation}\nn
\begin{split}
k \in \left( -\frac{m' \omega}{2} - 128 (\delta^{(s-1)}_0)^{1/4}, -\frac{m' \omega}{2} + 128 (\delta^{(s-1)}_0)^{1/4} \right) \subset \\
\left( -\frac{m' \omega}{2} - \frac{\sigma (m')}{2}, -\frac{m' \omega}{2} + \frac{\sigma(m')}{2} \right) \subset ( k_{m',s}^-, k_{m',s}^+ )
\end{split}
\end{equation}
with $m' = m_2 - m_1$. Combined with the assumption $k \in \IR \setminus \bigcup_{0 < |m'| \le 12 R^\es} (k_{m',s}^-, k_{m',s}^+)$, this implies $|m'| > 12 R^\es$. This contradicts $|m_2 - m_1| \le 12 R^\es$. Let now $|k_1 - k| < (\delta^{(s-1)}_0)^{1/2}$. It follows from the above arguments that in any event $|k+m_j\omega| \le \lambda + 1$. In particular, $\lambda^{-1} (|k+m_j\omega| + |k_1+m_j\omega|) \le 5$. Hence,
\begin{equation}\nn
\begin{split}
|v(m_j,k_1) - v(m_j,k)| = \lambda^{-1}|k - k_1|(|k+m_j\omega + k_1 + m_j \omega|) \le 5(\delta^{(s-1)}_0)^{1/2}, \\
|v(0,k_1) - v(0,k)| = \lambda^{-1}|k - k_1|(|k| + |k_1|) \le 3(\delta^{(s-1)}_0)^{1/2},
\end{split}
\end{equation}
and the statement follows.
\end{proof}

\begin{remark}\label{rem:11.1}
$(1)$ For any $\La \subset \IZ^\nu$, the matrix $H_{\La,\ve,k}$ obeys conditions \eqref{eq:2-2}--\eqref{eq:2-4} from Section~\ref{sec.3}. Due to statement $(1)$ of Lemma~\ref{lem:A.2}, $H_{\La,\ve,k} \in \cN^{(1)}(n_0, \Lambda, \delta'_0)$ with $\delta'_0 = \delta'_0(\La, n_0, k) := \lambda^{-1} [ \min_{m \in \La\setminus \{n_0\}} \min (|(m - n_0) \omega|, |2 k + (m - n_0) \omega|)]^2$, provided $k \notin \frac{\omega}{2} \IZ^\nu$ and $\ve$ is sufficiently small.

$(2)$ For the rest of this work we use the notation $\gamma$, $\lambda$ without reference to $\gamma \ge 1$, $\gamma-1 \le |k|\le \gamma$, $\lambda = 256\gamma$. It is convenient for technical reasons not to assume here that $\gamma$ is an integer.
\end{remark}

We will repeatedly use the following basic properties of the matrices $H_{\La,\ve,k}$.

\begin{lemma}\label{lem:basicshiftprop}
Let $\La \subset \zv$, $m \in \zv$ be arbitrary.
\begin{itemize}

\item[(1)] Consider the map $S : \La \rightarrow m + \La$, $S(n) = n + m$, $n \in \La$. Given $\psi(\cdot) \in \IC^\La$, set $S^*(\psi)(n') = \psi(n' - m)$, $n' \in (m + \La)$. The map $S^* : \psi \rightarrow S^*(\psi)$ is a unitary operator, which conjugates $H_{m + \La, \ve, k}$ with $H_{\La, \ve, k + m \omega}$.

\item[(2)] Consider the map $\cS : \La \rightarrow -\La$, $\cS(n) = -n$, $n \in \La$. Given $\psi(\cdot) \in \IC^\La$, set $\cS^*(\psi)(n') = \psi(-n')$, $n' \in -\La$. The map $\cS^* : \psi \rightarrow \cS^*(\psi)$ is a unitary operator, which conjugates $H_{\La, \ve, k}$ with $\overline{H_{-\La, \ve, -k}}$.

\item[(3)] Using the notation from $(1)$ and $(2)$, one has for any $n_0 \in \La$, $k \notin \frac{\omega}{2} \IZ^\nu$ and sufficiently small $\ve$,
\begin{equation}\label{eq:11QEbasic}
\begin{split}
Q(S(n_0), S(\La); \ve, k, E) & = Q(n_0, \La; \ve, k + m \omega, E), \\
E^{(1)}(S(n_0), S(\La); \ve, k) & = E^{(1)}(n_0, \La; \ve, k + m \omega), \\
Q(\cS(n_0), \cS(\La); \ve, k, E) & = \overline{Q(n_0, \La; \ve, -k, E)}, \\
E^{(1)}(\cS(n_0), \cS(\La); \ve, k) & = E^{(1)}(n_0, \La; \ve, -k).
\end{split}
\end{equation}

\end{itemize}

Assume that $H_{\La,\ve,k}$ with some given subsets $\cM^{(s')} \subset \La$, $s' = 1, \dots, s-1$, $\La^{(s')}(m) \subset \La$, $m \in \cM^{(s')}$ belongs to $\cN^{(s)}(n^\zero, \La, \delta_0)$ {\rm (}resp., $\hle \in GSR^{[\mathfrak{s}^{(\ell+1)}]}\bigl( \mathfrak{m}^{(\ell+1)}, \La; \delta_0,\mathfrak{t}^{(\ell+1)}\bigr)${\rm )}. Let $m_0$ be arbitrary. Then,
\begin{itemize}

\item[(4)] The matrix $H_{\La - m_0, \ve, k + m_0 \omega}$ with the subsets $\cM^{(s')} - m_0 \subset \La - m_0$, $s' = 1, \dots, s-1$, $\La^{(s')} (m - m_0) := \La^{(s')} (m) - m_0 \subset \La - m_0$, $m \in \cM^{(s')}$ belongs to $\cN^{(s)} (n^\zero - m_0, \La - m_0, \delta_0)$ {\rm (}resp., $GSR^{[\mathfrak{s}^{(\ell+1)}]}\bigl( \mathfrak{m}^{(\ell+1)}, \La; \delta_0, \mathfrak{t}^{(\ell+1)} \bigr)${\rm )}. Furthermore, let $E^{(s)}(n^\zero, \La; \ve, k)$ {\rm (}resp., $ E(m^\pm, \La; \ve, k)${\rm )} be defined as in Proposition~\ref{prop:4-4} {\rm (}resp., Theorem~\ref{th:6-4FIN}{\rm )} with $H_{\La, \ve, k}$ in the role of $\hle$. Then, $E^{(s)}(n^\zero, \La; \ve, k) = E^{(s)}(n^\zero - m_0, \La - m_0; \ve, k + m_0 \omega)$ {\rm (}resp., $E(m^+, \La; \ve, k) = E(m^+ - m_0, \La - m_0; \ve, k + m_0 \omega)${\rm )}.

\item[(5)] The matrix $H_{-\La, \ve, -k}$ with the subsets $-\cM^{(s')} \subset -\La$, $s' = 1, \dots, s-1$, $\La^{(s')}(-m) := -\La^{(s')}(m) \subset -\La$, $m \in \cM^{(s')}$ belongs to $\cN^{(s)}(-n^\zero, -\La, \delta_0)$ {\rm (}resp., $\hle \in GSR^{[\mathfrak{s}^{(\ell+1)}]}\bigl( \mathfrak{m}^{(\ell+1)}, \La; \delta_0,\mathfrak{t}^{(\ell+1)}\bigr)${\rm )}. Furthermore, $E^{(s)}(n^\zero, \La; \ve, k) = E^{(s)}(-n^\zero, -\La; \ve, -k)$ {\rm (}resp., $E(m^\pm, \La; \ve, k) = E(-m^\pm, -\La; \ve, -k)${\rm )}.

\end{itemize}
\end{lemma}

\begin{proof}
$(1)$ Both statements follow from the definition of the matrices $H_{\La, \ve, k}$.

$(2)$ The statements follow from the definition of the matrices $H_{\La, \ve, k}$ and the symmetry $v(n, k) = v(-n, -k)$, $n \in \zv$, $k \in \IR$.

$(3)$ For $k \notin \frac{\omega}{2} \IZ^\nu$, any $\La'$, $n'_0 \in \La'$, and sufficiently small $|\ve|$, $H_{\La', \ve, k} \in \cN^{(1)}(n'_0, \Lambda', \delta'_0)$ with $\delta'_0 = \delta'_0(\La', n'_0, k) := \lambda^{-1} [\min_{m \in \La'\setminus \{n_0\}} \min (|(m - n'_0) \omega|, |2 k + (m - n'_0) \omega|)]^2$. In particular, all functions in \eqref{eq:11QEbasic}  are well-defined for sufficiently small $|\ve|$. The identities in \eqref{eq:11QEbasic} follow from Proposition~\ref{prop:4-4} and $(1)$, $(2)$ of the present lemma.

$(4) \& (5)$ Assume that $H_{\La, \ve, k}$ with some given subsets $\cM^{(s')} \subset \La$, $s' = 1, \dots, s-1$, $\La^{(s')}(m) \subset \La$, $m \in \cM^{(s')}$ belongs to $ \cN^{(s)} (n^\zero, \La, \delta_0)$. Let $m_0$ be arbitrary. We will verify that $H_{\La - m_0, \ve, k + m_0 \omega} \in \cN^{(s)}(n^\zero - m_0, \La - m_0, \delta_0)$. The proof goes by induction over $s = 1, 2, \dots$. Note first that $v(n,k)=  v(n - m_0, k + m_0 \omega)$ for any $k, n, m_0$. Furthermore, due to part $(1)$, $H_{\La, \ve, k}$ and $H_{\La - m_0, \ve, k + m_0 \omega}$ are unitarily conjugate. In particular, these matrices have the same eigenvalues. Secondly, recall that since $H_{\La, \ve, k} \in \cN^{(s)} (n^\zero, \La, \delta_0)$, $E^{(s)}(n^\zero, \La; k, \ve)$ is the only eigenvalue of $H_{\La, \ve, k}$ which is analytic in $\ve$ and obeys $E^{(s)} (n^\zero, \La; k, 0) = v(n^\zero, k)$. If $H_{\La - m_0, \ve, k + m_0 \omega} \in \cN^{(s)} (n^\zero - m_0, \La - m_0, \delta_0)$, then $E^{(s)} (n^\zero - m_0, \La - m_0; k + m_0 \omega, \ve)$ is the only eigenvalue of $H_{\La - m_0, \ve, k + m_0 \omega}$ that is analytic in $\ve$ and obeys $E^{(s)} (n^\zero - m_0, \La - m_0; k + m_0 \omega, 0) = v(n^\zero - m_0, k + m_0 \omega) = v(n^\zero, k)$. Since the matrices have the same eigenvalues, these two functions are equal. We use these remarks for an induction argument over $s = 1, 2, \dots$. If $H_{\La, \ve, k} \in \cN^{(1)} (n^\zero, \La, \delta_0)$, then $H_{\La - m_0, \ve, k + m \omega} \in \cN^{(1)} (n^\zero - m_0, \La - m_0, \delta_0)$; see Definition~\ref{def:4-1}. Assume that the statement holds for $s' = 1, 2, \dots, s-1$ in the role of $s$. Clearly, conditions $(a)$, $(b)$, $(d)$, $(f)$ of Definition~\ref{def:4-1} hold for $H_{\La - m_0, \ve, k + m_0 \omega}$ since they hold for $H_{\La, \ve, k}$. Condition $(c)$ of Definition~\ref{def:4-1} holds due to the inductive assumption applied to each $H_{\La^\ar, \ve, k}$ with $r \le s-1$. Due to the previous remarks, we see that condition $(e)$ of Definition~\ref{def:4-1} holds for $H_{\La - m_0, \ve, k + m_0 \omega}$ since it holds for $H_{\La, \ve, k}$. This finishes the proof of $(4)$ in case $H_{\La, \ve, k} \in \cN^{(s)} (n^\zero, \La, \delta_0)$. The proof of $(5)$ in case $H_{\La,\ve,k} \in \cN^{(s)} (n^\zero, \La, \delta_0)$ is completely similar $($ of course, one should again use the fact that the matrices have the same eigenvalues and are self-adjoint $)$. The proof of both $(4)$ and $(5)$ in case $\hle \in GSR^{[\mathfrak{s}^{(\ell+1)}]}\bigl( \mathfrak{m}^{(\ell+1)}, \La; \delta_0,\mathfrak{t}^{(\ell+1)}\bigr)$ is completely similar.
\end{proof}

Given two sets $\La', \La'' \subset \IZ^\nu$, we introduce the following relation:
\begin{equation}\label{eq:7RA.1}
\La' \between \La'' \quad \text { if $\La' \cap \La'' \neq \emptyset$ and $\La' \cap (\IZ^\nu \setminus\La'') \neq \emptyset$}.
\end{equation}

Set
\begin{equation}\label{eq:A.1}
\begin{split}
\La^\one_k (0) = B(2 R^\one), \quad k \in \IR \setminus \bigcup_{0 < |m'| \le 12 R^\one} (k_{m',0}^-, k_{m',0}^+), \\
\cM^{(1)}_{k,1} = \{ m : |v(m,k) - v(0,k)| \le \delta_0/16\}, \quad k \in \IR \setminus \bigcup_{0 < |m'| \le 12 R^\two} (k_{m',1}^-, k_{m',1}^+), \\
\La^\one_k (m) = m + \La^\one_{k + m \omega} (0), \quad m \in \cM^{(1)}_{k,1}, \\
\La^\two_k (0) = B(3 R^\two) \setminus \Bigl( \bigcup_{m' \in \cM^{(1)}_{k,1} : \La^\one_{k} (m') \between B(3 R^\two))} \La^\one_{k} (m') \Bigr), \\
\cM^{(2)}_{k,2} = \{ m : |v(m,k) - v(0,k)| \le 3 \delta^\one_0/4 \}, \quad k \in \IR \setminus \bigcup_{0 < |m'| \le 12 R^{(3)}} (k_{m',2}^-, k_{m',2}^+), \\
\La^\two_k (m) = m + \La^\two_{k + m \omega} (0), \quad m \in \cM^{(2)}_{k,2}, \quad k \in \IR \setminus \bigcup_{0 < |m'| \le 12 R^{(3)}} (k_{m',2}^-, k_{m',2}^+), \\
\cM^{(s-1)}_{k,s-1} = \{ m : |v(m,k) - v(0,k)| \le 3 \delta^{(s-2)}_0/4\}, \quad k \in \IR \setminus \bigcup_{0 < |m'| \le 12 R^{(s-1)}} (k_{m',s-1}^-, k_{m',s-1}^+),\\
\cM^{(s')}_{k,s-1} = \{ m : |v(m,k) - v(0,k)| \le (3\delta^{(s'-1)}_0/4) - \sum_{s' < s'' \le s-1} \delta^{(s''-1)}_0, \\
m \notin \bigcup_{s' < s'' \le s-1} \bigcup_{m'' \in \cM^{(s'')}_{k,s-1}} \La^{(s'')}_k (m'') \}, \quad 1 < s' \le s-2, \\
\cM^{(1)}_{k,s-1} = \{ m : |v(m,k) - v(0,k)| \le (\delta^{(0)}_0/16) - \sum_{1 < s'' \le s-1} \delta^{(s''-1)}_0, \\
m \notin \bigcup_{1 < s'' \le s-1} \bigcup_{m'' \in \cM^{(s'')}_{k,s-1}} \La^{(s'')}_k(m'')\}, \quad k \in \IR \setminus \bigcup_{0 < |m'| \le 12 R^{(s-1)}} (k_{m',s-1}^-, k_{m',s-1}^+), \\
\La^{(s')}_k(m) = m + \La^{(s')}_{k + m \omega} (0), \quad m \in \cM^{(s')}_{k,s-1}, \quad k \in \IR \setminus \bigcup_{0 < |m'| \le 12 R^{(s-1)}} (k_{m',s-1}^-, k_{m',s-1}^+).
\end{split}
\end{equation}
\begin{equation}\nn
\begin{split}
\La^\es_k (0) = B(3 R^\es) \setminus \Bigl( \bigcup_{r = 1, \dots, s-1} \bigcup_{m' \in \cM^{(r)}_{k,s-1} : \La^{\ar}_{k} (m') \between B(3 R^\es))} \La^\ar_{k} (m') \Bigr), \\
k \in \IR \setminus \bigcup_{0 < |m'| \le 12 R^{(s)}} (k_{m',s-1}^-, k_{m',s-1}^+).
\end{split}
\end{equation}

\begin{remark}\label{rem:7.1oinout}
It follows from the definitions in \eqref{eq:A.1} that

$(a)$ $0 \in \cM^{(s-1)}_{k,s-1}$.

$(b)$ $\cM^{(s')}_{k,s-1} \cap \cM^{(s'')}_{ k,s-1} = \emptyset$ for any $s' < s'' \le s-1$.

$(c)$ Due to \eqref{eq:A.1}, for any $r$, we have $B(2 R^\ar) \subset \La^{(r)}_k(0) \subset B(3 R^\ar)$. In particular, $\La^{(s-1)}_k(0) \subset \La^{(s)}_k(0)$. Furthermore, we use the notation $\La^{(s')}_{k}(m)$ and not $\La^{(s')}_{k,s-1}(m)$. This is because if $m \in \cM^{(s')}_{k,s_1}$ for some $s_1 < s$, the set $m + \La^{(s')}_{k + m \omega}(0)$ is still the same.

$(d)$ If $s' < s-1$, $m \in \cM^{(s')}_{k,s-1}$, then
$$
(3 \delta^{(s')}_0/4) - \sum_{s'+1 < s'' \le s-1} \delta^{(s''-1)}_0 < |v(m,k) - v(0,k)| \le (3 \delta^{(s'-1)}_0/4) - \sum_{s' < s'' \le s-1} \delta^{(s''-1)}_0.
$$
Furthermore, it follows from the definition of the set $\cM^{(s-1)}_{k,s-1}$ and part $(4)$ of Lemma~\ref{lem:A.2} that for any  $m \in \cM^{(s-1)}_{k,s-1} \setminus \{0\}$ with $|m| \le 12 R^\es$, we have
$$
(\delta^{(s-1)}_0)^{1/2} < |v(m,k) - v(0,k)| \le 3 \delta^{(s-2)}_0/4.
$$
\end{remark}

\begin{lemma}\label{lem:A.3}
Let $s \ge 2$ and $k \in \IR \setminus \bigcup_{0 < |m'| \le 12 R^\esone} (k_{m',s-1}^-, k_{m',s-1}^+)$. Then,
\begin{itemize}

\item[(1)] If $m_i \in \cM^{(s')}_{k,s-1}$, $1 \le s' \le s-1$, $i = 1, 2$ and $m_1 \neq m_2$, then $|m_1 - m_2| > 12 R^{(s')}$, $\dist(\La^{(s')}_k(m_1), \La^{(s')}_k(m_2)) > 6 R^{(s')}$.

\item[(2)] Assume that for some $m_1, m_2 \in \zv$, $s_1 < s_2$, we have
$$
|v(m_i,k) - v(0,k)| \le (3 \delta^{(s_i-1)}_0/4) - \sum_{s_i < s'' \le s-1} \delta^{(s''-1)}_0, \quad i = 1, 2.
$$
Then,
\begin{equation}\label{eq:11centeratm'state}
|v(m_1 - m_2, k + m_2 \omega) - v(0, k + m_2 \omega)| < 3 \delta^{(s_1-1)}_0/4 - \sum_{s_1 < s'' \le s_2-1} \delta^{(s''-1)}_0.
\end{equation}

\item[(3)] Assume that for every $2 \le s' \le s-1$, the following condition holds:

$(\mathfrak{S}_{s'})$ If $k \in \IR \setminus \bigcup_{0 < |m'| \le 12 R^{(s')}} (k_{m',s'-1}^-, k_{m',s'-1}^+)$, $m_1 \in \cM^{(s_1)}_{k,s'-1}$, $s_1 \le s'-1$, $|m_1| < 12 R^{(s')}$, then either $\La^{(s_1)}_{k} (m_1) \subset \La^{(s')}_k(0)$ or $\La^{(s_1)}_{k} (m_1) \cap \La^{(s')}_k(0) = \emptyset$.

Then, for every $s \ge 2$, the following statement holds. Assume that for some $s_1 \le s-1$, $|m_1| < 12 R^\es$, we have
$$
|v(m_1,k) - v(0,k)| \le (3 \delta^{(s_1-1)}_0/4) - \sum_{s_1 < s'' \le s-1} \delta^{(s''-1)}_0.
$$
Then:

    either $(\alpha)$ $m_1 \in \La^{(s_2)}_k(m_2)$ for some $s_1 < s_2 \le s-1$, $m_2 \in \cM^{(s_2)}_{k,s-1}$,

    or $(\beta)$ $m_1 \in \cM^{(s_1)}_{k,s-1}$ and $\La^{(s_1)}_k(m_1)) \cap \La^{(s_2)}_k(m_2) = \emptyset$ for any $m_2 \in \cM^{(s_2)}_{k,s-1}$ $m_2 \neq m_1$ with $s_1 \le s_2 \le s-1$.

In case $(\alpha)$, one has $m_1 + \La^{(s_1)}_{k + m_1 \omega} (0) \subset \La^{(s_2)}_k (m_2)$.

\item[(4)] The condition $(\mathfrak{S}_{s'})$ holds for each $s' = 2, \dots, s$.

\end{itemize}
\end{lemma}

\begin{proof}
(1) Since $k \in \IR \setminus \bigcup_{0 < |m'| \le 12 R^\esone} (k_{m',s-1}^-, k_{m',s-1}^+) \subset \IR \setminus \bigcup_{0 < |m'| \le 12 R^{(s')}} (k_{m',s'}^-, k_{m',s'}^+)$, part $(4)$ of Lemma~\ref{lem:A.2} applies. Therefore, $|m_1 - m_2| > 12 R^{(s')}$. It follows from the definition of the sets $\La^\ar_k(m)$ that $\La^\ar_k(m) \subset (m + B(3 R^\ar))$ for any $m,r$. Thus the second statement in $(1)$ also holds.

$(2)$ We have
\begin{equation}\label{eq:11centeratm'}
\begin{split}
|v(m_1 - m_2, k + m_2 \omega) - v(0, k + m_2 \omega)| \le |v(m_1,k) - v(0,k)| + |v(0,k) - v(m_2,k)| \\
< 3 \delta^{(s_1-1)}_0/4 - \sum_{s_1 < s'' \le s-1} \delta^{(s''-1)}_0 + 3 \delta^{(s_2-1)}_0/4 - \sum_{s_2 < s'' \le s-1} \delta^{(s''-1)}_0 \\
< 3 \delta^{(s_1-1)}_0/4 - \sum_{s_1 < s'' \le s_2-1} \delta^{(s''-1)}_0,
\end{split}
\end{equation}
as claimed.

$(3)$ The proof goes via induction over $s = 2, 3, \dots$. Note first of all that due to part $(3)$ of Lemma~\ref{lem:A.2}, $k + m_1 \omega \in \IR \setminus \bigcup_{0 < |m'| \le 12 R^{(s_1)}} (k_{m',s_1-1}^-, k_{m',s_1-1}^+)$, so $\La^{(s_1)}_{k + m_1 \omega}(0)$ is well-defined in any event. Let $s=2$. The only possibility here is $s_1 = 1$ and there is no room for case $(\alpha)$. Due to part $(1)$ of the current lemma, one has $\dist(\La^{(1)}_k(m_1), \La^{(1)}_k(m_2)) > 6 R^{(1)}$ for any $m_2 \in \cM^{(1)}_{k,1}$, $m_2 \neq m_1$. This proves part $(3)$ for $s = 2$. It is important to note here, for the sake of the proof of $(4)$, that for $s = 2$, the proof of $(3)$ does not require any additional condition; in particular, the condition $(\mathfrak{S}_2)$ is not required. Let $s > 2$ be arbitrary. Assume that $(3)$ holds for any $s' < s$ in the role of $s$. Assume that $(\alpha)$ fails. Then, $m_1 \in \cM^{(s_1)}_{k,s-1}$, just due to the definition \eqref{eq:A.1}. With part $(1)$ of the current lemma taken into account, it suffices to consider $|m_1 - m_2| < 12 R^{(s_2)}$ with $s_1 < s_2 \le s-1$, $m_2 \in \cM^{(s_2)}_{k,s-1}$. Note that
\begin{equation}\label{eq:11centeratm'ABCDEF}
m_1 - m_2 \notin \La^{(s_2)}_{k + m_2 \omega}(0).
\end{equation}
Indeed, otherwise $m_1 \in (m_2 + \La^{(s_2)}_{k + m_2 \omega}(0)) = \La^{(s_2)}_k(m_2)$, contrary to the assumption that $(\alpha)$ fails for $m_1$. Note also that due to part $(3)$ in Lemma~\ref{lem:A.2}, $k + m_2 \omega \in \IR \setminus \bigcup_{0 < |m'| \le 12 R^{(s_2)}} (k_{m',s_2-1}^-, k_{m',s_2-1}^+)$. Since $|m_1 - m_2| < 12 R^{(s_2)}$ and \eqref{eq:11centeratm'state} holds, one can apply the inductive assumption for part $(3)$ of the current lemma to $k + m_2 \omega$ in the role of $k$, $(m_1 - m_2)$ in the role of $m_1$, and $s_2$ in the role of $s$. So, either $(\alpha)$ or $(\beta)$ hold. Consider first case $(\alpha)$, that is, assume that $m_1 - m_2 \in \La^{(s')}_{k + m_2 \omega}(m')$ for some $s_1 < s' \le s_2-1$, $m' \in \cM^{(s')}_{k + m_2 \omega, s_2-1}$. The inductive assumption for the very last statement in part $(3)$ implies that in this case one has
\begin{equation}\label{eq:11minlambdalambdain}
(m_1 - m_2) + \La^{(s_1)}_{k + m_1 \omega}(0) \subset \La^{(s')}_{k + m_2 \omega}(m').
\end{equation}
It follows from \eqref{eq:11centeratm'ABCDEF} that
\begin{equation}\label{eq:11centeratm'ABCFDE}
\La^{(s')}_{k + m_2 \omega}(m') \nsubseteqq \La^{(s_2)}_{k + m_2 \omega}(0).
\end{equation}
It follows from \eqref{eq:11centeratm'ABCFDE} and condition $(\mathfrak{S}_{s_2})$ that
\begin{equation}\label{eq:11centeratm'ABCXYZ}
\La^{(s')}_{k + m_2 \omega}(m') \cap \La^{(s_2)}_{k + m_2 \omega}(0) = \emptyset.
\end{equation}
Combining \eqref{eq:11minlambdalambdain} with \eqref{eq:11centeratm'ABCXYZ}, one obtains
\begin{equation}\label{eq:11minlambdalambdainAG}
(m_1 + \La^{(s_1)}_{k + m_1 \omega}(0)) \cap (m_2 + \La^{(s_2)}_{k + m_2 \omega}(0)) = \emptyset,
\end{equation}
which is what is claimed for $m_1$ in $(\beta)$. This finishes the proof if $(\alpha)$ holds for $m_1 - m_2$. Assume now that $(\beta)$ holds for $m_1 - m_2$, that is,
\begin{equation}\label{eq:11centeratm2a}
((m_1 - m_2) + \La^{(s_1)}_{k + m_1 \omega}(0)) \cap \La^{(s')}_{k + m_2 \omega}(m') = \emptyset
\end{equation}
for any  $s_1 \le s' \le s_2-1$, $m' \in \cM^{(s')}_{k + m_2 \omega, s_2-1}$. Since $|m_1 - m_2| < 12 R^{(s_2)}$ and \eqref{eq:11centeratm'state} holds, one has $(m_1 - m_2) \in \cM^{(s_1)}_{k + m_2 \omega, s_2-1}$. Since \eqref{eq:11centeratm'ABCDEF} holds, condition $(\mathfrak{S}_{s_2})$ implies
\begin{equation}\label{eq:11minlambdalambdainAGAG}
((m_1 - m_2) + \La^{(s_1)}_{k + m_1 \omega}(0)) \cap \La^{(s_2)}_{k + m_2 \omega}(0) = \emptyset.
\end{equation}
The relation \eqref{eq:11minlambdalambdainAGAG} implies \eqref{eq:11minlambdalambdainAG}. This finishes the inductive verification of the dichotomy in $(3)$. To finish part $(3)$, assume that $(\alpha)$ holds for $m_1$. So, $m_1 \in m_2 + \La^{(s_2)}_{k + m_2 \omega}(0)$ for some $s_1 < s_2 \le s-1$, $m_2 \in \cM^{(s_2)}_{k,s-1}$. Recall that \eqref{eq:11centeratm'state} holds. Due to the inductive assumption, either
$(\alpha)$ or $(\beta)$ holds for $(m_1 - m_2)$. Consider first the case $(\beta)$. Then, $(m_1 - m_2) \in \cM^{(s_1)}_{k + m_2 \omega, s_2-1}$. Since  $(m_1 - m_2) \in \La^{(s_2)}_{k + m_2 \omega}(0)$, due to condition $(\mathfrak{S}_{s_2})$, one has $(m_1 - m_2) + \La^{(s_1)}_{k + m_1 \omega}(0) \subset \La^{(s_2)}_{k + m_2 \omega}(0)$. This implies the second statement in part $(3)$ in this case. Consider now case $(\alpha)$, that is, $m_1 - m_2 \in \La^{(s')}_{k + m_2 \omega}(m')$ for some $s_1 < s' \le s_2-1$, $m' \in \cM^{(s')}_{k + m_2 \omega, s_2-1}$. Since $(m_1 - m_2) \in \La^{(s_2)}_{k + m_2 \omega}(0)$, due to condition $(\mathfrak{S}_{s_2})$, one has $\La^{(s')}_{k + m_2 \omega}(m') \subset \La^{(s_2)}_{k + m_2 \omega}(0)$. Furthermore, $|m_1 - m_2| < 12 R^{(s')} < 12 R^{(s_2)}$. Due to part $(2)$ of the current lemma, \eqref{eq:11centeratm'state} holds and one can apply the inductive assumption for the second statement in part $(3)$ of the current lemma with $m_1 - m_2$ in the role of $m_1$, $m'$ in the role of $m_2$, and $s_2$ in the role of $s$. Hence, $(m_1 - m_2) + \La^{(s_1)}_{k + m_1 \omega}(0) \subset \La^{(s')}_{k + m_2 \omega} (0)$. This finishes the inductive proof of $(3)$.

$(4)$ Once again the proof goes via induction over $s = 2, 3, \dots$. Let us verify $(\mathfrak{S}_2)$. The only possibility here is $s_1 = 1$ and $m_1 \in \cM^{(1)}_{k,1}$. Assume $\La_k^{(1)}(m_1) \cap \La_k^{(2)}(0) \neq \emptyset$. Then, clearly, $\La_k^{(1)}(m_1) \cap B(3 R^\two) \neq \emptyset$. Note that $\La_k^{(1)}(m_1) \cap (\IZ^\nu \setminus B(3 R^\two)) = \emptyset$. Indeed, otherwise $\La_k^{(1)}(m_1) \between B(3R^\two))$. Since
\begin{equation}\label{eq:A.1proflema1}
\La^\two_k(0) = B(3 R^\two) \setminus \Bigl( \bigcup_{m' \in \cM^{(1)}_{k,1} : \La^\one_{k}(m') \between B(3 R^\two))} \La^\one_{k}(m') \Bigr),
\end{equation}
that would imply $\La_k^{(1)}(m_1) \cap \La_k^{(2)}(0) = \emptyset$, contrary to the assumption. Let $m'$ be an arbitrary vector in the union in \eqref{eq:A.1proflema1}. Since $m_1$ is not included in the union in \eqref{eq:A.1proflema1}, $m' \neq m_1$. Due to part $(3)$ of the current lemma, $\La^\one_{ k}(m') \cap \La^\one_{k}(m_1) = \emptyset$. Combining this fact with $\La_k^{(1)}(m_1) \cap (\IZ^\nu \setminus B(3 R^\two)) = \emptyset$, one concludes that $\La_k^{(1)}(m_1) \subset \La^\two_k(0)$. This finishes the case $s = 2$. Assume that $(\mathfrak{S}_{s'})$ holds for any $s' \le s-1$. Let us verify $(\mathfrak{S}_{s})$. Using the notation from $(\mathfrak{S}_s)$, one can assume that $(m_1 + \La^{(s_1)}_{k + m_1 \omega}(0)) \cap \La^{(s)}_k(0) \neq \emptyset$. Since we assume that $(\mathfrak{S}_{s'})$ holds for any $s' \le s-1$, we can apply part $(3)$  to $m_1$. Consider first the case $(\beta)$. In this case, $m_1 \in \cM^{(s_1)}_{k,s-1}$, $(m_1+\La^{(s_1)}_{k + m_1 \omega}(0)) = \La^{(s_1)}_k(m_1)$. Recall that
\begin{equation}\label{eq:A.1proflema1AgA}
\La^\es_k(0) = B(3 R^\es) \setminus \Bigl( \bigcup_{r = 1, \dots, s-1} \bigcup_{m' \in \cM^{(r)}_{k,s-1} : \La^{\ar}_{k}(m') \between B(3 R^\es))} \La^\ar_{k}(m') \Bigr).
\end{equation}
Clearly, this implies $\La_k^{(s_1)}(m_1) \cap B(3 R^\es) \neq \emptyset$. Note that $\La_k^{(s_1)}(m_1) \cap (\IZ^\nu \setminus B(3 R^\es)) = \emptyset$. Indeed, otherwise $\La_k^{(s_1)}(m_1) \between B(3 R^\es))$. This would imply $\La_k^{(s_1)}(m_1) \cap \La_k^{(s)}(0) = \emptyset$, contrary to the assumption. Let $m'$ be an arbitrary vector in the union in \eqref{eq:A.1proflema1AgA}. Since $m_1$ is not included in the union in \eqref{eq:A.1proflema1AgA}, $m'\neq m_1$. Due to part $(3)$ of the current lemma, $\La^{(s')}_{ k}(m') \cap \La^{(s_1)}_{k}(m_1) = \emptyset$. Combining this fact with $\La_k^{(s_1)}(m_1) \cap (\IZ^\nu \setminus B(3 R^\es)) = \emptyset$, one concludes that $\La_k^{(s_1)}(m_1) \subset \La^\es_k(0)$. Consider now the case $(\alpha)$. Using the notation from case $(\alpha)$, one has $m_1 + \La^{(s_1)}_{k + m_1 \omega}(0) \subset \La^{(s_2)}_k(m_2)$. Note that for $m_2$, case $(\beta)$ takes place. Clearly $\La_k^{(s_2)}(m_2) \cap \La_k^{(s)}(0) \neq \emptyset$. Hence, $\La_k^{(s_2)}(m_2) \subset \La_k^{(s)}(0)$ and we are done.
\end{proof}

\begin{remark}\label{rem:7.1Rss1issue}
$(1)$ We remark here that in the proof of part $(3)$ in Lemma~\ref{lem:A.3}, we did not use the definition of the set $\La^{(s)}_k(0)$ from \eqref{eq:A.1}. We did use the definition of the sets $\La^{(s')}_k(m')$, $s' \le s-1$ from \eqref{eq:A.1}, part~$(1)$ of Lemma~\ref{lem:A.3} and condition $(\mathfrak{S}_{s-1})$ only. We will invoke this fact in Remark~\ref{rem:7.kawayfromzero}. We use the latter in Sections~\ref{sec.8} and \ref{sec.9}.

$(2)$ For technical reasons related to small values of $|k|$, we need to introduce for those $k$ some auxiliary sets $\La^\es_{k,sym}(m) \subset \La^\es_{k}(m)$, which give a very good ``approximation'' of $\La^\es_{k}(m)$ and at the same time obey $- \La^\es_{k,sym}(m) = \La^\es_{k,sym}(m)$; see Lemma~\ref{lem:7setLambdas}.
\end{remark}

\begin{lemma}\label{lem:12Nesreflection}
For any $k \in \IR$, $r$, $s$, we have $\cM^{(r)}_{-k} = -\cM^{(r)}_{k}$, $\La^\es_{-k}(-m) = -\La^\es_{k}(m)$.
\end{lemma}

\begin{proof}
One has $v(-m,-k) = v(m,k)$ for any $m$, $k$. This implies the first statement, $\cM^{(r)}_{-k}=-\cM^{(r)}_{k}$. Using this, one can easily verify the second statement using induction in $s$.
\end{proof}

To proceed with the definition of $\La^\es_{k,sym}(m)$, we need some combinatorics.

\begin{defi}\label{defi:5.twolambdas5}
Let $s > 0$ be an arbitrary integer. Let $A = (a_1, \ldots, a_n)$ be an arbitrary word over the alphabet $\{ 1, 2, \ldots, s \}$. We say that the word $A$ is correct if it has no sub-word $\tilde A = (a_j, \ldots, a_k)$ with $j < k$, $a_j = a_k$, and $\max_{j < i < k} a_i < a_j$. Otherwise, the word is called incorrect. By convention, each one letter word $A = (a_1)$ is correct. We denote by $\cA(s)$ the collection of all words over the alphabet $\{ 1, 2, \ldots, s \}$ and by $\cA_c(s)$ the collection of all correct words in $\cA(s)$. We also say that the word $A = (a_1,\dots,a_n)$ has length $n$.
\end{defi}

\begin{lemma}\label{lem:5.twolambdas2}
If $A = (a_1, \ldots, a_n) \in \cA_c(s)$, then $n \le 2^s - 1$.
\end{lemma}

\begin{proof}
The proof goes by induction on $s$. For $s = 1$, the only correct word is $A = (1)$. Assume that the statement holds for the alphabet $\{ 1, \ldots, s-1 \}$. If $a_j < s$ for every $j$, then $A \in \cA_c(s-1)$ and the statement holds due to the inductive assumption. Assume that $a_j = s$ for some $j$. Then, $a_k < s$ for every $k \neq j$ since otherwise $A \notin \cA_c(s)$. Let $A_1 = (a_1, \ldots, a_{j-1})$, $A_2 = (a_{j+1}, \ldots, a_n)$. Then, clearly, $A_1, A_2 \in \cA_c(s-1)$. Due to the inductive assumption $j-1 \le 2^{s-1} - 1$, $n-j \le  2^{s-1} - 1$. Hence, $n \le  2^{s} - 1$.
\end{proof}

\begin{lemma}\label{lem:5.twolambdas3}
Suppose $A = (a_1, \ldots, a_n) \notin \cA_c(s)$. Let $\tilde A = (a_j, \ldots, a_k)$ be a minimal length incorrect sub-word of $A$. Then, $a_j = a_k$, $a_i < a_j$ for any $j < i < k$ and $k-j \le 2^{a_j} - 1$.
\end{lemma}

\begin{proof}
Obviously, a minimal length incorrect sub-word $\tilde A = (a_j, \ldots, a_k) \notin \cA_c(s)$ exists. Due to the minimality, the words $A_1 = (a_{j+1}, \ldots, a_{k})$ and $A_2 = (a_{j}, \ldots, a_{k-1})$ are correct. On the other hand, $\tilde A$ has a sub-word $(a_{j + \ell}, \ldots, a_{k - m})$ such that $a_{j + \ell} = a_{k - m}$, $a_i < a_{j + \ell}$ for any $j + \ell < i < k - m$. Since both $A_1$ and $A_2$ are correct, $\ell=0$ and $m=0$. So, $a_j = a_k$, $a_i < a_j$ for any $j < i < k$. In particular, $A_1 \in \cA_c(a_j)$. Due to Lemma~\ref{lem:5.twolambdas2}, $k - j \le 2^{a_j} - 1$.
\end{proof}

\begin{defi}\label{defi:5.twolambdas6}
$(1)$ Consider arbitrary subsets $\La', \La'' \subset \IZ^\nu$. Assume that $\La' \cap \La'' \neq \emptyset$, $\La' \nsubseteq \La''$, $\La'' \nsubseteq \La'$. In this case, we say that $\La'$ and $\La''$ are chained. A sequence  $\La^{(\ell)}$, $\ell = 1, \ldots, n$ with $n\ge 2$ is called a chain if $\La^{(\ell)}$ and  $\La^{(\ell + 1)}$ are chained for every $\ell = 1, \ldots, n-1$.

$(2)$ Let $\mathfrak{L}$ be a system of sets $\La \subset \IZ^\nu$. Let $t(\La)$ be a function $\La \in \mathfrak{L}$ with values in $\mathbb{N}$.
We say that $(\mathfrak{L},t)$ is a proper subtraction system if the following conditions hold:
$(i)$ For any $a \in \mathbb{N}$, $R_a:=\min_{\La', \La'' \in \mathfrak{L} : t(\La') = a, \; t(\La'') = a, \; \La' \neq \La''} \; \dist (\La', \La'') > 0$, $(ii)$ Let $\La \in \mathfrak{L}$ be arbitrary, $a=t(\La)+1$. There exist subsets $\Xi_j \subset \La$, $j = 1, \dots$ such that $\diam(\Xi_j) < 2^{-a} R_a$, $\La = \cup_j \Xi_j$, and if for some for some $\La' \in \mathfrak{L}$, $\La \cap \La'\neq \emptyset$, then $\Xi_{j} \cap \La'\neq \emptyset$ for any $j$.

$(3)$ Let $(\mathfrak{L},t)$ be a proper subtraction system.  Given an arbitrary set $\La_{0,0}, \subset \IZ^\nu$, we set
\begin{equation}\label{eq:5.twolambdas5NEW}
\La_{0,\ell} = \La_{0,\ell-1} \setminus \Bigl(  \bigcup_{\La \in \mathfrak{L} : \La \nsubseteq \La_{0,\ell-1}} \La\Bigr).
\end{equation}
\end{defi}

\begin{lemma}\label{lem:5.lambdachains}Let $(\mathfrak{L},t)$ be a proper subtraction system.

$(1)$ Let $\La, \La' \in \mathfrak{L}$, $\La \cap \La' \neq \emptyset$. Let $a = t(\La) + 1$. For any $x \in \La$, we have $\dist (x,\La') < 2^{-a} R_a$.

$(2)$ Let $a \in \mathbb{N}$ and let $\La^{(\ell)}$, $\ell = 1, \ldots, n$ be a chain, $\La^{(\ell)} \in \mathfrak{L}$, $t(\La^{(\ell )}) < a$, $\ell = 1, \ldots, n$. Then, $\dist(\La^{(1)}, \La^{(n)}) < (n-1) 2^{-a} R_a$.
\end{lemma}

\begin{proof}
$(1)$ Let $\Xi_j \subset \La$, $j = 1, \dots$ be as in $(ii)$ of part $(2)$ of Definition~\ref{defi:5.twolambdas6}. Since $\La \cap \La' \neq \emptyset$, one has $\Xi_j \cap \La' \neq \emptyset$ for any $j$. Given $x \in \La$, there exists $j$ such that $x \in \Xi_j$. Since $\diam (\Xi_j) < 2^{-a} R_a$, the claim follows.

$(2)$ The proof goes by induction in $n = 2, \dots$. For $n = 2$, the claim is clear since $\La^{(1)} \cap \La^{(2)} \neq \emptyset$. Let $n > 2$. Assume the claim holds for any chain with $n - 1$ sets. Clearly, $\La^{(\ell)}$, $\ell = 1, \dots, n-1$ is a chain. Hence, $\dist(\La^{(1)}, \La^{(n-1)}) < (n-2) 2^{-a} R_a$. Therefore there exist $x \in \La^{(1)}, y \in \La^{(n-1)}$ such that $|x-y| < (n-2) 2^{-a} R_a$. By part $(1)$ of the current lemma, $\dist (y, \La^{(n)}) < 2^{-a} R_a$. Hence, $\dist (x, \La^{(n)}) \le |x-y| + \dist (y, \La^{(n)}) < (n-1) 2^{-a} R_a$.
\end{proof}

\begin{lemma}\label{lem:5.twolambdas4}
Let $\La_{0, \ell}$ be as in \eqref{eq:5.twolambdas5NEW}. Let
$$
\mathfrak{N}_\ell = \big\{ \La\in\mathfrak{L} : \La \cap \La_{0, \ell - 1} \neq \emptyset, \; \La \cap \bigl( \IZ^\nu \setminus \La_{0, \ell - 1} \bigl) \neq \emptyset \big\}.
$$

$(1)$ Assume that $\La_{0, \ell} \neq \La_{0, \ell - 1}$, $\ell = 1, \ldots, \ell_1$. Let $\La \in \mathfrak{N}_{\ell_1}$. There exists a chain $\La^{(\ell)}$, $\ell = 0, \ldots, \ell_1$ such that $\La^{(0)} = \La_{0,0}$, $\La^{(\ell_1)} = \La$, $\La^{(\ell)} \in \mathfrak{N}_\ell$, $\ell = 1, \ldots, \ell_1$, and in particular, $\La^{(\ell)}\neq \La^{(\ell')}$ if $\ell<\ell'$.

$(2)$ Assume that $s = \sup_{\La \in \mathfrak{L}}t(\La) < \infty$. There exists $\ell_0 < 2^s$ such that $\La_{0, \ell} = \La_{0, \ell_0}$, for any $\ell \ge \ell_0$.

$(3)$ Let $\ell_0$ be such such that $\La_{0, \ell_0+1}= \La_{0, \ell_0}$. Then, for any $\La \in\mathfrak{L}$, we have either $\La \subset \La_{0, \ell_0}$ or $\La \subset \Bigl( \IZ^\nu \setminus \La_{0, \ell_0} \Bigr)$.
\end{lemma}

\begin{proof}
$(1)$ The proof of the first statement goes by induction over $\ell_1 = 1, 2, \ldots$. Assume that $\La_{0,1} \neq \La_{0,0}$. Then, $\mathfrak{N}_1 \neq \emptyset$. Just by the definition, $\La \in \mathfrak{N}_{1}$ if and only if $\La$ is chained with $\La_{0,0}$. Assume that the statement holds for any $\ell' = 1, 2, \ldots, \ell_1-1$ in the role of $\ell_1$. Assume that $\La \in \mathfrak{N}_{\ell_1}$. One has the following cases.

$(\alpha)$: $\La \cap \bigl( \IZ^\nu \setminus \La_{0, \ell_1 - 2} \bigl) \neq \emptyset$. Since $\La \in \mathfrak{N}_{\ell_1}$, one has
$\La \cap \La_{0, \ell_1 - 1} \neq \emptyset$, $\La \cap \La_{0, \ell_1 - 2} \neq \emptyset$. Together with the assumption of the case, this implies $\La \in \mathfrak{N}_{\ell_1 - 1}$, which in turn implies $\La \cap \La_{0, \ell_1 - 1} = \emptyset$, which contradicts $\La \in \mathfrak{N}_{\ell_1}$. Thus, this case is impossible.

$(\beta)$: $\La \cap \bigl( \bigcup_{\La' \in \mathfrak{N}_{\ell_1 - 1}} \La' \setminus \La_{0, \ell - 2} \bigl) \neq \emptyset$. In this case, there exists $\La' \in \mathfrak{N}_{\ell_1 - 1}$ such that $\La \cap \La' \neq \emptyset$. Note that $\La \subset \La'$ is impossible, since in this case one would have $\La \cap \La_{0, \ell_1 - 1} = \emptyset$, contrary to the assumption that $\La \in \mathfrak{N}_{\ell_1}$. Assume that $\La' \subset \La$. Since $\La \in \mathfrak{N}_{\ell_1 - 1}$, this would imply $\La \cap \La_{0, \ell_1 - 2}\neq \emptyset$ and $\La \cap \bigl( \IZ^\nu \setminus \La_{0, \ell_1 - 2} \bigr) \neq \emptyset$. This means $\La \in \mathfrak{N}_{\ell_1-1}$.
This is again impossible, since in this case one would have $\La \cap \La_{0, \ell_1 - 1} = \emptyset$. Thus, $\La$ is chained with $\La'$. Applying the inductive assumption to $\La'$, one obtains the statement for $\La$.
 Assume $\La^{(\ell)} = \La^{(\ell')}$, $\ell<\ell'$. Then $\La^{(\ell)} \in \mathfrak{N}_{\ell}$ and at the same time  $\La^{(\ell)} \in \mathfrak{N}_{\ell'}$. This is inconsistent with the definition of the sets $\La_{0, \ell - 1}$ and $\mathfrak{N}_{\ell}$.

$(2)$ Assume that $\La_{0, \ell} \neq \La_{0, \ell - 1}$, $\ell = 1, \ldots, \ell_1$ for some $\ell_1 \ge 2^s$. Due to part $(1)$ of the current lemma, there exists a chain $\La^{(\ell)}$, $\ell = 0, \ldots, \ell_1$ such that $\La^{(\ell)} \in \mathfrak{N}_\ell$, $\ell = 1, \ldots, \ell_1$. Consider the word $A = (a_1, \dots, a_{\ell_1})$, $a_j = t(\La^{(j)})$ over the alphabet $\{1, \dots, s\}$. Since $\ell_1 \ge 2^s$, due to Lemma~\ref{lem:5.twolambdas2}, $A \notin \cA_c(s)$. Due to Lemma~\ref{lem:5.twolambdas3}, $A$ has a sub-word $\tilde A = (a_j, \ldots, a_k)$ such that $a_j = a_k$, $a_i < a_j$ for any $j < i < k$ and $k-j \le 2^{a_j} - 1$. Due to part $(2)$ of Lemma~\ref{lem:5.lambdachains}, there exist $x \in \La^{(j+1)}, y \in \La^{(k-1)}$ with $|x-y| < (k-j-2) 2^{-a_j} R_{a_j}$. Since $\La^{(j)} \cap \La^{(j+1)} \neq \emptyset$, due to part $(1)$ of Lemma~\ref{lem:5.lambdachains} one has $\dist (x, \La^{(j)}) < 2^{-a_j} R_{a_j}$. Similarly, $\dist (y, \La^{(k)}) < 2^{-a_j} R_{a_j}$. Hence  $\dist (\La^{(j)}, \La^{(k)}) < (k-j) 2^{-a_j} R_{a_j} < R_{a_j}$. On the other hand, due to part $(1)$ of the current lemma, $\La^{(j)} \neq \La^{(k)}$. Since $t(\La^{(j)}) = t(\La^{(k)}) = a_j$, this contradicts the definition of the quantities $R_a$. Thus, there exists $\ell_0 < 2^s$ such that $\La_{0, \ell_0} = \La_{0, \ell_0 + 1}$. It follows from the definition \eqref{eq:5.twolambdas5NEW} that $\La_{0, \ell} = \La_{0, \ell_0}$, for any $\ell \ge \ell_0$.

$(3)$ This follows from the definition \eqref{eq:5.twolambdas5NEW}.
\end{proof}

Set $\cS(n) = -n$, $n \in \IZ^\nu$.

\begin{lemma}\label{lem:7mdeltaSm}
Let $s \ge 2$ and $k \in \IR \setminus \bigcup_{0 < |m'| \le 12 R^\esone} (k_{m',s-1}^-, k_{m',s-1}^+)$. Assume $|k| < \delta^{(s-2)}_0$.

$(1)$ If $|v(m,k) - v(0,k)| < \delta$, with $ \delta^{(s-2)}_0/2 \le \delta < 1/64$, then $|v(\cS(m),k) - v(0,k)| < 4 \delta/3$.

$(2)$ Let $s' < s$, $m_j \in \cM^{(s')}_{k,s-1}$, $j = 1, 2$, and assume that $\cS(m_1) \neq m_2$. Then, $\dist(\cS(\La_k^{(s')}(m_1)), \La_k^{(s')}(m_2)) > 6 R^{(s')}$.
\end{lemma}

\begin{proof}
$(1)$  Note first of all that $\gamma = 1$, $\lambda = 256$. Since $|v(m,k) - v(0,k)| < \delta$, it follows from $(1)$ in Lemma~\ref{lem:A.2} that $\min (|m \omega| ,|2 k + m  \omega|) \le 32 \delta^{1/2}$. In particular,  $|m  \omega| \le 32 \delta^{1/2} + 2 |k| < 32 \delta^{1/2} + 2 \delta$. One has
\begin{equation}\label{eq:8Tmineqverif}
\begin{split}
|v(\cS(m),k) - v(0,k)| = |v(-m,k) - v(0,k)| = \lambda^{-1} |m \omega| |2 k - m \omega|\le
\lambda^{-1}|m \omega|( |2 k + m \omega| + 4|k|) \\
= |v(m,k) - v(0,k)| + 4 \lambda^{-1} |m\omega||k| < \delta + \frac{1}{64} (32 \delta^{1/2} + 2 \delta) (2 \delta) < 4 \delta/3.
\end{split}
\end{equation}

$(2)$ One has $|v(m_j,k) - v(0,k)| < 3 \delta^{(s'-1)}_0/4$, $j = 1, 2$. Note that $3 \delta^{(s'-1)}_0/4 \ge \delta^{(s-2)}_0/2$, since we assume $s' < s$. Due to part $(1)$, one also obtains $|v(\cS(m_1),k) - v(0,k)| < \delta^{(s'-1)}_0$. Due to part $(4)$ of Lemma~\ref{lem:A.2}, one has $|\cS(m_1) - m_2| > 12 R^{(s')}$, since $S(m_1) \neq m_2$. This implies $\dist (\cS(\La_k^{(s')}(m_1)), \La_k^{(s')}(m_2)) > 6 R^{(s')}$.
\end{proof}

\begin{defi}\label{defi:7.LLLL}
Assume $s \ge 2$, $|k| < \delta^{(s-2)}_0$. It follows from \eqref{eq:diphnores} and \eqref{eq:7K.1}
that $k \in \IR \setminus \bigcup_{0 < |m'| \le 12 R^\esone} (k_{m',s-1}^-, k_{m',s-1}^+)$. Let $\mathfrak{L}'$ be the collection of all sets $\La(m) := \La_k^{(s')}(m) \cup \cS(\La_k^{(s')}(m))$, $1 \le s' \le s-1$, $m \in \cM^{(s')}_{k,s-1}$. We say that $\La(m_1) \backsim \La(m_2)$ if $s_1 = s_2$, and either $m_1 = m_2$ or $\cS(m_1) = m_2$. Clearly, this is an equivalence relation on $\mathfrak{L}'$. Let $\mathfrak{M}$ be the set of equivalence classes. Clearly, each class has at most two elements in it. For each $\mathfrak{m} \in \mathfrak{M}$, set $\La(\mathfrak{m}) = \bigcup_{\La(m_1) \in \mathfrak{m}} \La(m_1)$. Set $\mathfrak{L} = \{ \La(\mathfrak{m}) : \mathfrak{m} \in \mathfrak{M} \}$. Let $\La(\mathfrak{m}) \in \mathfrak{L}$, $ \La^{(s')}(m) \cup \cS(\La^{(s')}(m)) \in \mathfrak{m}$. Set $t(\La(\mathfrak{m})) = s'$. This defines an $\mathbb{N}$-valued function on $\mathfrak{L}$. Set also $p_\mathfrak{m} = \{m,\cS(m)\}$. Clearly, the set $p_\mathfrak{m}$ depends only on $\mathfrak{m}$.
\end{defi}

\begin{lemma}\label{lem:7.lLL}
Using the notation from Definition~\ref{defi:7.LLLL}, the following statements hold.

$(1)$ For any $\La(\mathfrak{m}_j) \in \mathfrak{L}$, $j = 1, 2$, such that $t(\La(\mathfrak{m}_1)) = t(\La(\mathfrak{m}_2))$, $\mathfrak{m}_1 \neq \mathfrak{m}_2$, we have $\dist (\La(\mathfrak{m}_1), \La(\mathfrak{m}_2)) \ge 6R^{(t(\La(\mathfrak{m}_1)))}$.

$(2)$ For any $\mathfrak{m}$, we have
\begin{equation}\label{eq:7Lmsets}
\bigcup_{m \in p_\mathfrak{m}} \bigl( (m + B(2 R^{(t(\La(\mathfrak{m})))})) \bigr) \subset \La(\mathfrak{m}) \subset \bigcup_{m \in p_\mathfrak{m}} \bigl( (m + B(3 R^{(t(\La(\mathfrak{m})))})) \bigr).
\end{equation}
Furthermore, $\La(\mathfrak{m}) = \Xi(\mathfrak{m}) \cup \cS(\Xi(\mathfrak{m}))$, where $\diam (\Xi(\mathfrak{m})) \le 6 R^{(t(\La(\mathfrak{m})))}$.

$(3)$ If $\mathfrak{m}_1 \neq \mathfrak{m_2}$, then $\La(\mathfrak{m}_1) \neq \La(\mathfrak{m}_2)$.

$(4)$ The pair $(\mathfrak{L},t)$ is a proper subtraction system.

$(5)$ For any $\mathfrak{m}$, we have $\La(\mathfrak{m}) =\cS(\La(\mathfrak{m}))$.
\end{lemma}

\begin{proof}
$(1)$ Let $ \La^{(s')}(m_j) \cup \cS(\La^{(s')}(m_j)) \in \mathfrak{m_j}$, $j = 1, 2$. Since $\mathfrak{m_1} \neq \mathfrak{m_2}$, $p_{\mathfrak{m_1}} \cap p_{\mathfrak{m_2}} = \emptyset$. Therefore, $\dist(\La^{(s')}(m_1), \La^{(s')}(m_2)) > 6 R^{(s')}$,  $\dist (\cS(\La^{(s')}(m_1)), \cS(\La^{(s')}(m_2))) > 6 R^{(s')}$. Furthermore, due to part $(2)$ of Lemma~\ref{lem:7mdeltaSm}, $\dist(\cS(\La^{(s')}(m_1)), \La^{(s')}(m_2)) > 6 R^{(s')}$, $\dist(\cS(\La^{(s')}(m_2)), \La^{(s')}(m_1)) > 6 R^{(s')}$. This implies the statement in $(1)$.

$(2)$ Let $ \La^{(s')}(m') \cup \cS(\La^{(s')}(m')) \in \mathfrak{m}$. One has
\begin{equation}\label{eq:8Lmsets1}
\bigl(m' + B(2R^{(t(\La(\mathfrak{m})))})\bigr) \subset \La^{(s')}(m') \subset \bigl(m' + B(3 R^{(t(\La(\mathfrak{m})))})\bigr).
\end{equation}
Furthermore, $\{m', \cS(m')\} = p_\mathfrak{m}$. This implies the first statement in $(2)$. The second statement in $(2)$ follows from Definition~\ref{defi:7.LLLL}.

$(3)$ Let $\mathfrak{m}_1 \neq \mathfrak{m_2}$. If $t(\La(\mathfrak{m}_1)) = t(\La(\mathfrak{m}_2))$, then $(3)$ follows from $(1)$. If $t(\La(\mathfrak{m}_1)) \neq t(\La(\mathfrak{m}_2))$, then $(3)$ follows from $(2)$.

$(4)$ Assume that $t(\La') = t(\La'')$, $\La' \neq \La''$. It follows from $(3)$ and $(1)$ that $\dist (\La', \La'') \ge R^{(t(\La'))}$. So, $(i)$ from part $(2)$ of Definition~\ref{defi:5.twolambdas6} holds with $R_a \ge R^{(a)}$. Let $\La(\mathfrak{m})$ be arbitrary, and set $a = t(\La(\mathfrak{m})) + 1$. Due to part $(2)$, one has $\La(\mathfrak{m}) = \Xi(\mathfrak{m}) \cup \cS(\Xi(\mathfrak{m}))$ with $\diam (\Xi(\mathfrak{m})) \le 6 R^{(t(\La(\mathfrak{m})))} = 6 R^{(a-1)} < 2^{-a} R^{(a)} \le 2^{-a} R_a$. Furthermore, let $\La(\mathfrak{m}')$ be arbitrary. Assume $\La(\mathfrak{m}) \cap \La(\mathfrak{m}') \neq \emptyset$. Once again, due to part $(2)$, one has $\La(\mathfrak{m}') = \Xi(\mathfrak{m}') \cup \cS(\Xi(\mathfrak{m}'))$. This implies $\Xi(\mathfrak{m}) \cap \La(\mathfrak{m}') \neq \emptyset$ and $\cS(\Xi(\mathfrak{m})) \cap \La(\mathfrak{m}') \neq \emptyset$. Hence, $(ii)$ from part $(2)$ of Definition~\ref{defi:5.twolambdas6} holds as well. This finishes the proof of $(4)$.

$(5)$ This follows readily from the definition of the sets $\La(\mathfrak{m})$.
\end{proof}

Assume $|k| < \delta^{(s-2)}_0$. For $\ell = 1, 2, \ldots$, set
\begin{equation} \label{eq:7.twolambdas5}
\mathfrak{B}(s,0) :=  B(3 R^\es), \quad \mathfrak{B}(s,\ell) = \mathfrak{B}(s, \ell-1)  \setminus \Bigl( \bigcup_{\mathfrak{m} \in \mathfrak{M} : \La(\mathfrak{m}) \between \mathfrak{B}(s,\ell-1)} \La(\mathfrak{m}) \Bigr).
\end{equation}

\begin{lemma}\label{lem:7setLambdas}
$(1)$ There exists $\ell_0 < 2^s$ such that $\mathfrak{B}(s,\ell) = \mathfrak{B}(s, \ell+1)$ for any $\ell \ge \ell_0$.

$(2)$ For any $\La \in \mathfrak{L}$, we have either $\La \subset \mathfrak{B}(s,\ell_0)$ or $\La \subset \Bigl( \IZ^\nu \setminus
\mathfrak{B}(s, \ell_0) \Bigr)$.

$(3)$ Set $\La^{(s)}_{k,sym}(0) = \mathfrak{B}(s,\ell_0)$. Then, for any $\La_k^{(s')}(m)$, we have either $\La_k^{(s')}(m) \cap \La^{(s)}_{k,sym}(0) = \emptyset$ or $\La_k^{(s')}(m) \subset \La^{(s)}_{k,sym}(0)$.

$(4)$ $\cS(\mathfrak{B}(s,\ell)) = \mathfrak{B}(s,\ell)$ for any $\ell$. In particular, $\cS(\La^{(s)}_{k,sym}(0)) = \La^{(s)}_{k,sym}(0)$.

$(5)$ For any $\ell \ge 1$, we have
\begin{equation}\label{eq:7.twolambdas5NEW}
\{ n \in \mathfrak{B}(s,\ell-1)) : \dist (n, \IZ^\nu \setminus \mathfrak{B}(s,\ell-1)) \ge 6 R^\esone \} \subset \mathfrak{B}(s,\ell) \subset \mathfrak{B}(s,\ell-1)).
\end{equation}
In particular, $B(2 R^\es)  \subset \La^{(s)}_{k,sym}(0) \subset B(3 R^\es)$.

$(6)$ $\La^{(s)}_{k,sym}(0) \subset \La^{(s)}_{k}(0)$.
\end{lemma}

\begin{proof}
Parts $(1)$, $(2)$ follow from Lemma~\ref{lem:5.twolambdas4}.

Let $\La_k^{(s')}(m)$ be such that $\La_k^{(s')}(m) \cap \mathfrak{B}(s,\ell_0) \neq \emptyset$. Let $\mathfrak{m}$ be the equivalence class containing  $\La_k^{(s')}(m) \cup \cS(\La_k^{(s')}(m))$. Then, just by definition, $\La_k^{(s')}(m) \subset \La(\mathfrak{m})$. In particular, $\La(\mathfrak{m}) \cap \mathfrak{B}(s,\ell_0) \neq \emptyset$. This implies $\La(\mathfrak{m}) \subset \mathfrak{B}(s,\ell_0)$. Therefore, $\La_k^{(s')}(m) \subset \mathfrak{B}(s,\ell_0)$. This finishes the proof of $(3)$.

To verify $(4)$ note that $\cS(\mathfrak{B}(s,0)) = \mathfrak{B}(s,0)$. Combining this with part $(5)$ of Lemma~\ref{lem:7.lLL}, one obtains $T(\mathfrak{B}(s,\ell)) = \mathfrak{B}(s,\ell)$ for any $\ell$, as claimed.

Consider an arbitrary $\La(\mathfrak{m})$. It follows from \eqref{eq:7Lmsets} and Definition~\ref{defi:7.LLLL} that there exists $m$ such that
$$
\La(\mathfrak{m}) \subset \bigl( (m + B(6 R^{(t(\La(\mathfrak{m})))})) \bigr) \cup \mathcal{S}\bigl( (m + B(6 R^{(t(\La(\mathfrak{m})))})).
$$
Assume that  $\La(\mathfrak{m}) \between \mathfrak{B}(s,\ell-1)$. Due to part $(4)$ of the current lemma, $\cS(\mathfrak{B}(s,\ell-1)) = \mathfrak{B}(s,\ell-1)$. Hence,
$$
\{ n \in \mathfrak{B}(s,\ell-1)) : \dist (n, \IZ^\nu \setminus \mathfrak{B}(s,\ell-1)) \ge 6 R^\esone \} \subset \mathfrak{B}(s,\ell-1) \setminus \La(\mathfrak{m}).
$$
This implies \eqref{eq:7.twolambdas5NEW}. The second statement in $(5)$ follows from \eqref{eq:7.twolambdas5NEW} since $\ell_0 < 2^s$.

Statement $(6)$ follows from the definition of the sets $\La^{(s)}_{k,sym}(0)$, $\La^{(s)}_{k}(0)$.
\end{proof}

\begin{prop}\label{prop:A.3}
Let $s \ge 1$ and $k \in \IR \setminus \bigcup_{0 < |m| \le 12 R^\es} (k_{m,s-1}^-, k_{m,s-1}^+)$, $\delta_0 := (\delta_0^\zero)^{1/2}$. Let $\ve_0$, $\ve_{s}$ be as in Definition~\ref{def:4-1}. For $\ve \in (-\ve_{s},\ve_{s})$, the following statements hold.

\begin{itemize}

\item[(1)] For $s = 1$ and any $0 < |m| \le 12 R^{(1)}$, $|k_1 - k| < \delta(1) := 2\delta^{(0)}_0$, we have $|v(m,k_1) - v(0,k_1)| \ge \delta_0$. If $s \ge 2$, $0 < |m| \le 12 R^{(s)}$, $m \notin \bigcup_{1 \le r \le s-1} \bigcup_{m' \in \cM^{(r)}_{k,s-1}} \La^{(r)}_{k}(m')$, then $|v(m,k) - v(0,k)| \ge \delta_0^4$.

\item[(2)] The matrix $H_{\La^\es_k(0), \ve, k}$ belongs to $\cN^{(s)}(0, \La^\es_k(0), \delta_0)$. If $s \ge 2$ and $|k|<\delta^{(s-2)}_0$, then the matrix $H_{\La^\es_{k,sym}(0), \ve, k}$ belongs to $\cN^{(s)}(0, \La^\es_{k,sym}(0), \delta_0)$. We introduce an additional notation $\La^\es_{k,a}(0)$, which means $\La^\es_{k}(0)$ if $|k| \ge \delta^{(s-2)}_0$, and either of $\La^\es_k(0)$, $\La^\es_{k,sym}(0)$ if $|k| < \delta^{(s-2)}_0$. For $s \ge 2$, the subsets from Definition~\ref{def:4-1} are as follows: $\cM^{(r)}_{k}(\La^\es_{k,a}(0)) := \cM^{(r)}_{k,s-1} \cap \La^\es_{k,a}(0)$, $\La^\ar_k(m')$, $m' \in \cM^{(r)}_{k,s-1} (\La^\es_{k,a}(0))$, $r = 1, \dots, s-1$.

\item[(3)] Assume that $k \in \IR \setminus \bigcup_{|m| \le 12 R^{(s)}} (k_{m,s}^-, k_{m,s}^+)$. Then for any $m \in \cM^{(s)}_{k,s}$, the matrix $H_{\La^{(s)}_{k,a}(m), \ve, k}$ with the subsets $m + \cM^{(r)}_{k + m \omega}(\La^{(s)}_{k + m \omega,a }(0))$, $\La^\ar_k(m') := m' + \La^\ar_{k + m' \omega}(0)$, $r = 1, \dots, s-1$ belongs to $\cN^{(s)}(m, \La^\es_{k,a}(m), \delta^\zero_0)$.

\item[(4)]  For $|k - k_1| < \delta(1)$, the matrix $H_{\La^\one_{k}(0), \ve, k_1}$ belongs to the class $\cN^{(1)}(0, \La^\one_k(0), \delta_0)$. For $s \ge 2$, $k_1 \in (k-\delta(s), k+\delta(s))$ with $\delta(s) := 2\delta^{(s-2)}_0$, the matrix $H_{\La^\es_{k,a}(0), \ve, k_1}$ with the subsets $\cM^{(r)}_{k} (\La^\es_{k,a}(0))$, $\La^\ar_k(m')$, $m' \in \cM^{(r)}_{k} (\La^\es_{k,a}(0))$, $r = 1, \dots, s-1$  belongs to the class $\cN^{(s)} (0, \La^\es_{k,a}(0), \delta_0)$. Let $Q^{(s)} (0, \La^{(s)}_{k,a}(0); \ve, k_1, E)$, $E^{(s)}(0, \La^{(s)}_{k,a}(0); \ve, k_1)$ be defined as in Proposition~\ref{prop:4-4} with $H_{\La^{(s)}_{k,a}(0),\ve,k_1}$ in the role of $\hle$. The following estimates hold for $s = 1$, $|k_1 - k| < \delta^\zero_0/4$ or $s \ge 2$, $|k_1 - k| < \delta(s)/8$:
\begin{equation}\label{eq:7kk1compderiv}
|\partial^\alpha_{k_1} E^{(s)}(0, \La^{(s)}_{k,a}(0); \ve, k_1) - \partial^\alpha_{k_1} v(0,k_1)| < |\ve|^{17/16}, \quad \alpha \le 2.
\end{equation}
\begin{equation}\label{eq:7kk1compderivloverin}
\begin{split}
(\sgn k_1) \partial^\alpha_{k_1} E^{(1)}(0, \La^{(1)}_{k,a}(0); \ve, k_1)\ge \frac{7|k_1|}{4\lambda}, \quad 0<\alpha\le 2\\
(\sgn k_1) \partial^\alpha{k_1} E^{(s)}(0, \La^{(s)}_{k,a}(0); \ve, k_1)\ge \frac{7|k_1|}{4\lambda} - \sum_{s' \ge 1 : |k| > \delta^{(s')}_0/2} |\ve| (\delta^{(s')}_0)^5, \quad s\ge 2, \quad 0<\alpha\le 2,
\end{split}
\end{equation}
\begin{equation}\label{eq:7kk1comp}
|E^{(s)}(0, \La^{(s)}_{k,a}(0);\ve,k_1) - E^{(s)}(0, \La^{(s)}_{k,a}(0);\ve,k)| < 3 |k - k_1|.
\end{equation}
Furthermore, if $k_2 \in (k - \delta(s), k + \delta(s))$ and $k_2 \in \IR \setminus \bigcup_{|m| \le 12 R^\es} (k_{m,s-1}^-, k_{m,s-1}^+)$, so that the current proposition applies to $k_2$, then
\begin{equation}\label{eq:7kk1comp1}
|E^{(s)}(0, \La^{(s)}_{k,a}(0);\ve,k_1) - E^{(s)}(0, \La^{(s)}_{k_2,a}(0);\ve,k_1)| \le 3 |\ve| (\delta^\esone_0)^5.
\end{equation}

\item[(5)] Let $k_1 \in (k-\delta(s), k+\delta(s))$. Let $Q^{(s')} (0, \La^{(s')}_k(0); \ve, k_1, E)$, $E^{(s')}(0, \La^{(s')}_k(0); \ve, k_1)$ be defined as in Proposition~\ref{prop:4-4} with $H_{\La^{(s')}_k(0), \ve, k_1}$ in the role of $\hle$. Then, for $|\alpha| \le 2$,
\begin{equation}\label{eq:7Ederiss1compA}
|\partial^\alpha_{k_1} E^{(s)}(0, \La^{(s)}_k(0); \ve, k_1) - \partial^\alpha_{k_1} E^{(s-1)}(0, \La^{(s-1)}_k(0); \ve, k_1)| \le |\ve| (\delta^\esone_0)^5.
\end{equation}
Here, $E^{(0)}(m',\La'; k', \ve) := v(m',k')$, as usual.

\item[(6)] Let $0 < k < k' \le \gamma $, $k , k' \in \IR \setminus \bigcup_{|m| \le 12 R^\es} (k_{m,s}^-, k_{m,s}^+)$. Define $k \thicksim _s k'$ if $k, k'$ are in the same connected component of $\IR \setminus \bigcup_{0 < |m| \le 12 R^\es} (k_{m,s}^-, k_{m,s}^+)$, $k \nsim_s k'$ otherwise.

Then,
\begin{equation}\label{eq:7Ederivlower}
\begin{split}
E^{(s)}(0, \La^{(s)}_{k'}(0); \ve, k') - E^{(s)}(0, \La^{(s)}_{k}(0); \ve, k) \le \frac{9k'}{4\lambda} (k' - k)  + 3 |\ve| (\delta^\es_0)^{5} \;  \\
\text{for any $0 < k < k' \le \gamma $ if $s=1$, and for $k' - k < \delta^{(s-2)}_0$ if $s\ge 2$}, \\
\\
E^{(s)}(0, \La^{(s)}_{k'}(0); \ve, k') - E^{(s)}(0, \La^{(s)}_{k}(0); \ve, k) \\
\ge \begin{cases} \frac{7}{8\lambda} ((k')^2 - k^2) - 3 |\ve| (\delta^\zero_0)^{4} & \text{if $s = 1$}, \\ \frac{7}{8\lambda} ((k')^2 - k^2) - 3 |\ve| (\delta^\es_0)^{4} & \text{if $s \ge 2$ and $k \thicksim _s k'$}, \\ \frac{7}{8\lambda} ((k')^2 - k^2) - 8 |\ve| \sum_{s' \le s-1 : \min(k' - k,k) > \delta^{(s')}_0} (\delta^{(s')}_0)^{4} & \text{if $s\ge 2$ and $k \nsim_s k'$}. \end{cases}
\end{split}
\end{equation}

\item[(7)]
\begin{equation}\label{eq:7Ecomparison}
\begin{split}
E^{(s)}(0, \La^{(s)}_{k}(0); \ve, k) = E^{(s)}(0, \La^{(s)}_{-k}(0); \ve, -k), \\
E^{(1)}(0, \La^{(1)}_{k}(0); \ve, k_1) = E^{(1)}(0, \La^{(1)}_{k}(0); \ve, -k_1) \quad \text{if $|k|,|k_1|<\delta^{(0)}_0$}, \\
E^{(s)}(0, \La^{(s)}_{k,sym}(0); \ve, k_1) = E^{(s)}(0, \La^{(s)}_{k,sym}(0); \ve, -k_1) \quad \text{if $s \ge 2$, $|k|, |k_1| < \delta^{(s-2)}_0/2$}.
\end{split}
\end{equation}

\end{itemize}
\end{prop}

\begin{proof}
In the proofs below we verify the statements for $H_{\La^\es_{k}(0), \ve, k}$. The verification for $H_{\La^\es_{k,sym}(0), \ve, k}$ is completely similar. Let $k \in \IR \setminus \bigcup_{|m| \le 12 R^{(1)}} (k_{m,0}^-, k_{m,0}^+)$ and suppose $|k_1 - k| < \delta(1)$. Consider $m$ satisfying $0 < |m| \le 12 R^\one$. It follows from part $(4)$ of Lemma~\ref{lem:A.2} that $|v(m,k_1) - v(0,k_1)| \ge (\delta_0^\zero)^{1/2} = \delta_0$. This verifies the first statement in $(1)$; see Definition~\ref{def:4-1}. The second statement in $(1)$ follows immediately from the definition of the sets $\cM^{(r)}_{k,s-1}$.

The proof of parts $(2)$--$(7)$ goes by induction in $s = 1, 2, \dots$. Let $s = 1$ and let  $k \in \IR \setminus \bigcup_{|m| \le 12 R^\one} (k_{m,0}^-, k_{m,0}^+)$. First of all, part $(5)$ is due to \eqref{eq:3Hinvestimatestatement1kvardE} from Lemma~\ref{lem:5differentiation}, part $(7)$ is due to part $(5)$ of Lemma~\ref{lem:basicshiftprop} and Lemma~\ref{lem:12Nesreflection}. It follows from part $(1)$ that for $|k_1 - k| < \delta(1)$, the matrix $H_{\La^\one_k(0), \ve, k_1}$ belongs to the class $\cN^{(1)}(0, \La^\one_k(0), \delta_0)$. This gives the base of the induction for part $(2)$ and for the first statement in part $(4)$. The second statement in part $(4)$ is due to \eqref{eq:3Hinvestimatestatement1kvardE} from Lemma~\ref{lem:5differentiation}. Taking into account that $\partial^\alpha_k h(m,n;k,\ve) = 0$ if $m \neq n$, and
$$
|\partial^\alpha_k h(m,m;k,\ve)| = 2 \lambda^{-1} |k + m \omega| < 8 \exp (|m|^{1/5}),
$$
so that $B_0 = 8$ in the notation of the lemma, one obtains the estimate \eqref{eq:7kk1compderiv}. Assume $|k| \le \delta^\zero_0/2$. Note that $\La^{(1)}_{k}(0) = -\La^{(1)}_{k}(0)$. Due to parts $(4)$ and $(7)$ of the current proposition, the function  $E^{(1)}(0, \La^{(1)}_{k}(0); \ve, k_1)$ is well defined, $C^2$-smooth, obeys $E^{(1)}(0, \La^{(1)}_{k}(0); \ve, k_1) = E^{(1)}(0, \La^{(1)}_{k}(0); \ve, -k_1)$ for $|k_1| < \delta^\zero_0$ and $\partial^2_{k_1} E^{(1)}(0, \La^{(1)}_{k}(0); \ve, k_1) > 7/4\lambda$ with $\lambda=256$. This implies \eqref{eq:7kk1compderivloverin}. Assume  $1 \ge k > \delta^\zero_0$. Note that $\lambda = 256$ in this case. Since $E^{(1)}(0, \La^{(1)}_k(0); \ve, k_1) = E^{(1)}(0, \La^{(1)}_{k_1}(0); \ve, k_1)$, one obtains using \eqref{eq:7kk1compderiv},
\begin{equation}\label{eq:7kk1compderivintegr}
\begin{split}
 \partial_{k_1} E^{(1)}(0, \La^{(s)}_k(0); \ve, k_1)
\ge \partial_{k_1} v(0,k_1) - |\ve|^{17/16}\\
 = (2/\lambda) k_1 - |\ve|^{17/16} > \frac{7k_1}{4\lambda} > (\delta_0)^2
\end{split}
\end{equation}
for $|k_1 - k| < \delta^\zero_0/4$ in case $k > 0$. A similar estimate holds for $k \ge 1$ and for $k < 0$. So, \eqref{eq:7kk1compderivloverin} holds in any event. The estimate \eqref{eq:7kk1comp} follows from \eqref{eq:7kk1compderiv}. The estimate \eqref{eq:7kk1comp1} is trivial for $s = 1$ since $\La^{(1)}_{k,sym}(0)=\La^{(1)}_k(0) = \La^{(1)}_{k_2}(0)$. This finishes part $(4)$ for $s = 1$.

Assume that $k \in \IR \setminus \bigcup_{|m| \le 12 R^{(2)}} (k_{m,2}^-, k_{m,2}^+)$. Then, due to Lemma~\ref{lem:A.2},  $k + m \omega \in \IR \setminus \bigcup_{|m| \le 12 R^{(2)}} (k_{m,1}^-, k_{m,1}^+)$ for any  $m \in \cM^{(2)}_{k,2}$. Therefore, due to part $(2)$ of the current lemma with $s = 1$ and part $(4)$ of Lemma~\ref{lem:basicshiftprop}, the matrix $H_{\La^{(1)}_k(m), \ve, k}$ belongs to the class $\cN^{(1)}(m, \La^\one_k(m), \delta_0)$. This is the base of the induction for part $(3)$.

We will now verify $(6)$. The upper estimate follows from \eqref{eq:7kk1compderiv}. Let us verify the lower estimate. Let $[k'_j, k_j'']$ be the connected components of the set $\IR \setminus \bigcup_{0 < |m| \le 12 R^{(1)}} (k_{m,2}^-, k_{m,2}^+)$, enumerated so that $k''_j < k'_{j+1}$. Assume $k'_i \le k < k' \le k''_i$ for some $i$. Assume also that $k' - k > \delta^\zero_0$. Set $\theta_t = k + t \delta^\zero_0$, $t = 0, \dots, t'-1$, where $t' = [(\delta^\zero_0)^{-1} (k' - k)] - 1$, $\theta_{t'} = k'$. Combining \eqref{eq:7kk1compderivloverin} with \eqref{eq:7kk1comp1}, one obtains
\begin{equation}\label{eq:7Ekkcompar2A}
E^{(1)}(0, \La^{(1)}_{\theta_r}(0); \ve, \theta_r) - E^{(1)}(0, \La^{(1)}_{\theta_{r-1}}(0); \ve, \theta_{r-1}) \ge \frac{7}{8\lambda} (\theta_r^2 - \theta_{r-1}^2) - 3 |\ve| (\delta^\zero_0)^5.
\end{equation}
Adding up \eqref{eq:7Ekkcompar2A} over $r = 1, \dots, t'$, one obtains
\begin{equation}\label{eq:7Ekkcompar2AB}
E^{(1)}(0, \La^{(1)}_{k'}; \ve,k') - E^{(1)}(0, \La^{(1)}_{k}(0); \ve,k) \ge \frac{7}{8\lambda} ((k')^2 - k^2) - 3 t' |\ve| (\delta^\zero_0)^5.
\end{equation}
Recall that in $(6)$ we assume $k' \le \gamma$. So, $t' \le (\delta^\zero_0)^{-1}$. Hence, \eqref{eq:7Ekkcompar2AB} implies in particular the lower estimate in \eqref{eq:7Ederivlower}. The argument for the case $k' - k \le \delta^\zero_0$ is completely similar. Consider now an arbitrary  $k+\gamma \ge k' > k$. Recall that $k'_{j+1} - k''_j \ge \min_{|m| < 12 R^\one} \sigma(m)> 4 \ve_0^{1/2}$. Let $[k'_j,k''_j]$ be arbitrary such that $k''_j \ge k$. It follows from the definitions in \eqref{eq:7K.1}  that $(-\delta^\zero_0, \delta^\zero_0) \subset [k'_\ell, k''_\ell]$ for some $\ell$. Since $k$ belongs to one of the $[k'_m, k''_m]$, one concludes that $k''_j \ge \delta^\zero_0 > \ve_0^{1/2}$. Due to part $(5)$, one has
\begin{equation}\label{eq:7Ekkcompar}
E^{(1)}(0, \La^{(1)}_{k'_{j+1}}(0); \ve, k'_{j+1}) - E^{(1)}(0, \La^{(1)}_{k_j''}(0); \ve, k''_{j}) \\
\ge [v(0,k'_{j+1}) - v(0,k''_j)] - 2 |\ve| > \frac{7}{8\lambda} ((k'_{j+1})^2 - (k''_{j})^2).
\end{equation}

Combining the estimates \eqref{eq:7Ekkcompar} with the estimates \eqref{eq:7Ekkcompar2AB}, and taking into account $k - k' \le 1$,
one concludes that
\begin{equation}\label{eq:7Ekkcompar3}
E^{(1)}(0, \La^{(1)}_{k'}(0); \ve, k') - E^{(1)}(0, \La^{(1)}_{k}(0); \ve, k)  \ge \frac{7k}{8\lambda} (k'^2 - k^2) - 12 |\ve| (\delta^\zero_0)^4.
\end{equation}
So, the lower estimate in \eqref{eq:7Ederivlower} holds in any event.

This finishes the case $s=1$.

Let $s \ge 2$ be arbitrary. Once again, part $(5)$ is due to \eqref{eq:4-17AAAA} from Proposition~\ref{prop:4-4} and part $(7)$ is due to part $(5)$ of Lemma~\ref{lem:basicshiftprop}. Assume that statements $(2)$--$(4)$ hold for any $s' = 1, \dots, s-1$ in the role of $s$. We will now verify that $H_{\La^\es_k(0), \ve, k} \in \cN^{(s)}(0, \La^\es_k(0), \delta_0)$. Condition $(a)$ of Definition~\ref{def:4-1} holds. Due to the definition, one has $\cM^{(r)}_{k} \cap  \cM^{(s)}_{k} = \emptyset$ if $r < s$. Due to Lemma~\ref{lem:A.3}, the second part in $(b)$ of Definition~\ref{def:4-1} holds. To verify condition $(c)$ of Definition~\ref{def:4-1}, note that $k \in \IR \setminus \bigcup_{0 < |m'| \le 12 R^\es} (k_{m',s-1}^-, k_{m',s-1}^+) \subset \IR \setminus \bigcup_{0 < |m'| \le 12 R^{(s')}} (k_{m',s'}^-, k_{m',s'}^+)$ for any $s' \le s-1$. In particular, due to the inductive assumption, part $(3)$ of the current proposition applies with $s'$ in the role of $s$. This implies condition $(c)$ of Definition~\ref{def:4-1} for $s' = s-1$. Let $s' < s-1$, $m \in \cM^{(s')}_{k}(\La^\es_k(0))$. Then, $|v(m,k) - v(0,k)| < \delta^{(s'-1)}_0$. Part $(3)$ of Lemma~\ref{lem:A.2} applies. So, $k + m \omega \in \IR \setminus \bigcup_{0 < |m'| \le 12 R^{(s')}} (k_{m',s'-1}^-, k_{m',s'-1}^+)$. Therefore, the inductive assumptions apply to $k + m \omega$ in the role of $k$ and $s'$ in the role of $s$. In particular, $H_{\La^{(s')}_{k + m \omega}(0), \ve, k + m \omega} \in \cN^{(s')}(0, \La^{(s')}_{k + m \omega}(0), \delta_0)$. Due to part $(4)$ of Lemma~\ref{lem:basicshiftprop}, this implies condition $(c)$ of Definition~\ref{def:4-1} for $s'$.

Recall that $m' + B(2 R^\ar) \subset \La^\ar(m')$ for any $m'$ and $r$. Therefore, condition $(d)$ in Definition~\ref{def:4-1} holds.

Condition $(f)$ in Definition~\ref{def:4-1} follows readily from the definition of the sets $\cM^{(r)}_{k}$.

We will now verify condition $(e)$ in Definition~\ref{def:4-1}. Let $s' \le s-1$ be arbitrary. Using the inductive assumption, and combining the estimate of part $(5)$ with $s'' = s', \dots, s-1$ in the role of $s$, one obtains
\begin{equation}\label{eq:12.condedefi51a}
|E^{(s-1)}(0, \La^{(s-1)}_k(0); \ve, k) - E^{(s')}(0, \La^{(s')}_k(0); \ve, k)| \le 2 |\ve| (\delta^{(s')}_0)^5.
\end{equation}
Let $m \in \cM^{(s')}_k$ be arbitrary. Due to part $(4)$ of Lemma~\ref{lem:basicshiftprop}, one has
\begin{equation}\label{eq:12.condedefi51aXY}
E^{(s')}(m, \La^{(s')}_k(m); \ve, k) = E^{(s')}(0, \La^{(s')}_{k + m \omega}(0); \ve, k + m \omega).
\end{equation}
Let us verify first the lower estimate in condition $(e)$ in Definition~\ref{def:4-1}. Consider the case $s' < s-1$. Recall that due to $(d)$ in Remark~\ref{rem:7.1oinout},
\begin{equation}\label{eq:12.condedefi51aABCD}
11 \delta^{(s')}_0/16 < (3 \delta^{(s')}_0/4) - \sum_{s'+1 < s'' \le s-1} \delta^{(s''-1)}_0 < |v(m,k) - v(0,k)| \le (3 \delta^{(s'-1)}_0/4) - \sum_{s' < s'' \le s-1} \delta^{(s''-1)}_0.
\end{equation}
Note that $k + m \omega \in \IR \setminus \bigcup_{0 < |m'| \le 12 R^{(s')}} (k_{m',s'}^-, k_{m',s'}^+)$, due to Lemma~\ref{lem:A.2}. Note also that $|k + m \omega|, |k| \in \IR \setminus \bigcup_{0 < |m'| \le 12 R^{(s')}} (k_{m',s'}^-, k_{m',s'}^+)$. By Lemma~\ref{lem:7.setKs}, $|k|, |k + m\omega|$ belong to the same connected component of $\IR \setminus \bigcup_{|m'| \le 12 R^\es} (k_{m',s-1}^-, k_{m',s-1}^+)$. Using the inductive assumption for parts $(6)$, $(7)$ of the current proposition, \eqref{eq:12.condedefi51aABCD} and the fact that $k, k + m\omega$ belong to the same connected component of $\IR \setminus \bigcup_{|m'| \le 12 R^\es} (k_{m',s-1}^-, k_{m',s-1}^+)$, one obtains
\begin{equation}\label{eq:12.condedefi51b}
\begin{split}
|E^{(s')}(0, \La^{(s')}_{k + m \omega}(0); \ve, k + m \omega) - E^{(s')}(0, \La^{(s')}_k(0); \ve, k)| \ge \frac {7}{8\lambda} | (k + m \omega)^2 - k^2 | - 12 |\ve| (\delta^{(s')}_0)^{4} \\
= \frac{7}{8}|v(m,k) - v(0,k)| - 12 |\ve| (\delta^{(s')}_0)^{4} \ge \frac{77}{128} \delta^{(s')}_0 -12|\ve| (\delta^{(s')}_0)^{4}.
\end{split}
\end{equation}
Combining \eqref{eq:12.condedefi51a} with \eqref{eq:12.condedefi51aXY} and \eqref{eq:12.condedefi51b}, one obtains
\begin{equation}\label{eq:12.condedefi51c}
\begin{split}
|E^{(s-1)}(0, \La^{(s-1)}_k(0); \ve, k) - E^{(s')}(m, \La^{(s')}_k(m); \ve, k)| \\
\ge \frac{77}{128} \delta^{(s')}_0 - 12|\ve| (\delta^{(s')}_0)^{4} - 2 |\ve| (\delta^{(s')}_0)^5 > \frac{\delta^{(s')}_0}{2}.
\end{split}
\end{equation}
This verifies the lower estimate in condition $(e)$ in Definition~\ref{def:4-1} for $s' < s-1$. The derivation of the upper estimate is completely similar and we omit it. This finishes the verification of condition $(e)$ in Definition~\ref{def:4-1} for $s' < s-1$. The verification in case $s' = s-1$ is completely similar. So, we have $H_{\La^\es_k(0), \ve, k} \in \cN^{(s)}(0, \La^\es_k(0), \delta_0)$, that is, part $(2)$ of the proposition holds.

The verification of part $(3)$ is completely similar to the one in case $s' = 1$.

We will now verify the first statement in $(4)$, that is, for $k_1 \in (k - \delta(s), k + \delta(s))$, the matrix $H_{\La^\es_k(0), \ve, k_1}$ with the subsets $\cM^{(r)}_{k}(\La^\es_k(0))$, $\La^\ar_k(m')$, $m' \in \cM^{(r)}_{k} (\La^\es_k(0))$, $r = 1, \dots, s-1$ obeys conditions $(a)$--$(f)$ of Definition~\ref{def:4-1}. Conditions $(a)$, $(b)$, $(d)$ hold for obvious reasons. Let $s' < s-1$, $m \in \cM^{(s')}_{k}(\La^\es_k(0))$. Then, as we explained above, $k + m \omega \in \IR \setminus \bigcup_{0 < |m'| \le 12 R^{(s')}} (k_{m',s'-1}^-, k_{m',s'-1}^+)$, and the inductive assumptions apply to $k + m \omega$ in the role of $k$ and $s'$ in the role of $s$. Since $|(k_1 + m \omega) - (k + m \omega)| < \delta(s) < \delta(s')$, $H_{\La^{(s')}_{k + m \omega}(0), \ve, k_1 + m \omega} \in \cN^{(s')} (0, \La^{(s')}_{k + m \omega}(0), \delta_0)$. Due to part $(4)$ of Lemma~\ref{lem:basicshiftprop}, $H_{\La^{(s')}_{k}(m), \ve, k_1} \in \cN^{(s')}(m, \La^{(s')}_{k}(m), \delta_0)$, that is, condition $(c)$ of Definition~\ref{def:4-1} holds. The verification of condition $(e)$ is completely similar to the one we did for $H_{\La^\es_k(0), \ve, k}$. Thus the first statement in part $(4)$ holds.

The estimate \eqref{eq:7kk1compderiv} is due to \eqref{eq:3Hinvestimatestatement1kvardE} from Lemma~\ref{lem:5differentiation}. Let us verify \eqref{eq:7kk1compderivloverin}. Assume $|k| \le \delta^{(s-2)}_0/2$. Recall that $\La^{(s)}_{k,sym}(0)=-\La^{(s)}_{k,sym}(0)$. Due to parts $(4)$ and $(7)$ of the current proposition, the function  $E^{(s)}(0, \La^{(s)}_{k,sym}(0); \ve, k_1)$ is well defined, $C^2$-smooth, obeys $E^{(s)}(0, \La^{(s)}_{k,sym}(0); \ve, k_1)=E^{(s)}(0, \La^{(1)}_{k,sym}(0); \ve, -k_1)$ for $|k_1| < \delta^{(s-2)}_0$ and $\partial^2_{k_1} E^{(s)}(0, \La^{(s)}_{k,sym}(0); \ve, k_1) > 7/4$. This implies \eqref{eq:7kk1compderivloverin} for $|k| < \delta^{(s-2)}_0/2$. For $|k| > \delta^{(s-2)}_0/2$, \eqref{eq:7kk1compderivloverin} follows from the inductive assumption regarding \eqref{eq:7kk1compderivloverin}
with $s-1$ in the role of $s$ combined with part $(5)$.

The estimate \eqref{eq:7kk1comp} follows from \eqref{eq:7kk1compderiv}. The estimate \eqref{eq:7kk1comp1} is due to Corollary~\ref{cor:5.twolambdas1}. This finishes the proof of part $(4)$.

Let us verify part $(6)$. The upper estimate follows from \eqref{eq:7kk1compderiv} and \eqref{eq:7kk1comp1}. Let us verify  the lower estimate. Let $[k'_j, k_j'']$ be the the connected components of the set $\IR \setminus \bigcup_{|m| \le 12 R^{(s)}} (k_{m,s+1}^-, k_{m,s+1}^+)$, enumerated so that $k''_j < k'_{j+1}$. If $k, k' \in [k'_j, k_j'']$ for some $j$, then the proof goes just as for $s = 1$ with use of \eqref{eq:7kk1compderivloverin}. So, assume $k \in [k'_\ell, k_\ell'']$, $k' \in [k'_m, k_m'']$, $\ell < m$. Note first of all that
\begin{equation}\label{eq:7EkkcomparsarbPr}
\begin{split}
E^{(s)}(0, \La^{(s)}_{k''_{\ell}}(0); \ve, k''_{\ell}) - E^{(s)}(0, \La^{(s)}_{k}(0); \ve, k) \ge \frac{7}{8\lambda} ((k''_{\ell})^2 - k^2) - 12 |\ve| (\delta^{(s)}_0)^{4}, \\
E^{(s)}(0, \La^{(s)}_{k'}(0); \ve, k') - E^{(s)}(0, \La^{(s)}_{k'_{m}}(0); \ve, k'_{m}) \ge \frac{7}{8\lambda} ((k')^2 - (k'_{m})^2) - 12 |\ve| (\delta^{(s)}_0)^{4}.
\end{split}
\end{equation}
Using part $(5)$ and the inductive assumption for part $(6)$ with $(s-1)$ in the role of $s$, one obtains
\begin{equation}\label{eq:7Ekkcomparsarb}
\begin{split}
E^{(s)}(0, \La^{(s)}_{k'_{m}}(0); \ve, k'_{m}) - E^{(s)}(0, \La^{(s)}_{k_\ell''}(0); \ve, k''_{\ell}) \\
\ge [E^{(s-1)}(0, \La^{(s-1)}_{k'_{m}}(0); \ve, k'_{m}) - E^{(s-1)}(0, \La^{(s-1)}_{k''_{\ell}}(0); \ve, k''_{\ell})] - 2 |\ve| (\delta^\esone_0)^5 \\
\ge \frac{7}{8\lambda} ((k'_{m})^2 - (k''_{\ell})^2) - 26 |\ve| \sum_{s' \le s-1 : k'_m-k''_\ell > \delta^{(s')}_0} (\delta^{(s')}_0)^{4} - 2 |\ve| (\delta^\esone_0)^5.
\end{split}
\end{equation}
Combining \eqref{eq:7EkkcomparsarbPr} with \eqref{eq:7Ekkcomparsarb}, one obtains the lower estimate in part $(6)$.
\end{proof}

\begin{remark}\label{rem:7.kawayfromzero}
$(0)$ Using the notation from the last proposition, let $\La^{(s,\mathbf{1})}_{k}(0)$ be such that for any $\La^{(s')}_{k}(m)$ with $s' < s$, we have either $\La^{(s')}_{k}(m) \subset \La^{(s,\mathbf{1})}_{k}(0)$ or $\La^{(s')}_{k}(m) \cap \La^{(s,\mathbf{1})}_{k}(0) = \emptyset$. Assume also that $B(R^\es) \subset \La^{(s,\mathbf{1})}_{k}(0)$. Then, Proposition~\ref{prop:A.3} applies with $\La^{(s,\mathbf{1})}_{k}(0)$ in the role of $\La^{(s)}_{k}(0)$. In particular, $H_{\La^{(s,\mathbf{1})}_k(m), \ve, k} \in \cN^{(s)}(m, \La^{(s,\mathbf{1})}_k(m), \delta^\zero_0)$.

$(1)$ Here we want to remark again that the condition $|k| > (\delta^\zero_0)^{1/2}$ has not been used anywhere except for part $(6)$ of the last proposition.

$(2)$ Let $12 R^\esone < |m^\zero| \le 12 R^\es$ be arbitrary. Assume that $k \in \IR \setminus \bigcup_{0 < |m| \le 12 R^\es, \; m \neq m^\zero} (k_{m,s-1}^-, k_{m,s-1}^+)$. Then, obviously, Proposition~\ref{prop:A.3} applies with $s - 1$ in the role of $s$. Furthermore, let $\cM^{(s')}_{k,s-1}$, $\La^{(s')}_k(m)$ be defined as in \eqref{eq:A.1}. Due to part $(2)$ of Remark~\ref{rem:7.1Rss1issue}, part $(3)$ of Lemma~\ref{lem:A.3} applies. Therefore, conditions $(a)$--$(d)$ in Definition~\ref{def:4-1} hold. The derivation of \eqref{eq:12.condedefi51b}, \eqref{eq:12.condedefi51c} for $s' \le s-1$, $m \neq m^\zero$ goes the same way as in the proof of  Proposition~\ref{prop:A.3}. Assume that $m^\zero \in \cM^\esone_{k,s-1}$. Then, for $|k_1 - k| < 2 \delta^{(s-2)}_0$, we have
\begin{equation}\label{eq:7.condeupperEXCL}
|E^{(s-1)}(m^\zero, \La^{(s-1)}_{k}(m^\zero); \ve, k_1) - E^{(s-1)}(0, \La^{(s-1)}_k(0); \ve, k_1)| \le 3 | |k_1 + m^\zero \omega| - |k_1| | + (\delta^{(s-1)}_0)^5
\end{equation}
and
\begin{equation}\label{eq:7.condPROPmzero}
|E^{(s-1)}(m^\zero, \La^{(s-1)}_{k}(m^\zero); \ve, k_1) - E^{(s-1)}(0, \La^{(s-1)}_k(0); \ve, k_1)| \ge \frac{7}{8\lambda} | (k_1 + m^\zero \omega)^2 - k_1^2 | - 12|\ve|(\delta^{(s-1)}_0)^4.
\end{equation}

$(3)$ Note that Proposition~\ref{prop:A.3E0} applies to $k=0$. For the proof of Theorem~A in Section~\ref{sec.11}, we also need to consider the matrices $(H_{\La',\ve, 0}-E)$ with $-\ve_0^{1/2} < E < 0$; see \eqref{eq:2epsilon0} in Definition~\ref{def:4-1}. The analysis of these matrices goes almost word for word as the one for the matrices in Proposition~\ref{prop:A.3}. Moreover, the same subsets $\La^\es_0(0)$ can be employed. In Proposition~\ref{prop:A.3E0} we just state the result needed for the proof of Theorem~A. We omit the proof the proposition.
\end{remark}

\begin{prop}\label{prop:A.3E0}.
Let $-\ve^{1/2} < E < 0$ be arbitrary. For each $s = 1, 2, \dots$, the matrix $(H_{\La^\es_0(0), \ve, k}-E)$ belongs to $\cN^{(s)}(0, \La^\es_0(0), \delta_0)$.
\end{prop}

\section{Matrices with an Ordered Pair of Resonances Associated with $1$--Dimensional Quasi-Periodic Schr\"odinger Equations}\label{sec.8}

\begin{defi}\label{def:simplestresonance}
Let $s \ge 1$, $q \ge 0$  $n_0 \in \IZ^\nu$, $0 < |n_0| \le 12 R^{(1)}$ if $s = 1$, and $12 R^\esone < |n_0| \le 12 R^{(s)}$ if $s \ge 2$. Assume that
\begin{equation} \label{eq:8kn0cases}
(k_{n_0} - 2 \sigma(n_0),  k_{n_0,} + 2 \sigma(n_0)) \subset \IR \setminus \bigcup_{0 < |m'| \le 12 R^{(s)}, \; m' \neq n_0} (k_{m',s-1}^-, k_{m',s-1}^+)
\end{equation}
with $k_{n_0} = -n_0 \omega/2$ and $\sigma(n_0)$ as defined in \eqref{eq:7K.1}. We set  $\cR^{(s,s)}(\omega, n_0) := (k_{n_0} - 2 \sigma(n_0), k_{n_0,} + 2 \sigma(n_0))$.
\end{defi}

\begin{remark}\label{rem:8.ksetintersect}
$(1)$ The intersection of $\cR^{(s,s)}(\omega,n_0)$ and $\IR \setminus \bigcup_{0 < |m'| \le 12 R^{(s)}} (k_{m',s-1}^-, k_{m',s-1}^+)$ is a non-empty set $\cK^{(s)}_{n_0} := (k_{n_0} - 2 \sigma(n_0),  k_{n_0,} + 2 \sigma(n_0)) \setminus (k_{n_0,s-1}^-, k_{n_0,s-1}^+)$. In particular, Proposition~\ref{prop:A.3} applies to $k \in \cK^{(s)}_{n_0}$. For technical reasons, we need to verify that in fact Proposition~\ref{prop:A.3} applies on a slightly bigger set; see part $(3)$ of Lemma~\ref{lem:8A.3} below.

$(2)$ If $k \in \cR^{(s,s)}(\omega,n_0)$, then $-k \in \cR^{(s,s)}(\omega,-n_0)$.

$(3)$ Since $|n_0| \le 12 R^\es$, one has due to \eqref{eq:diphnores}, $|k_{n_0}| > \frac12 (\delta_0^{(s-1)})^{1/16} $.
\end{remark}

\begin{lemma}\label{lem:8A.2}
Let $k \in \cR^{(s,s)}(\omega, n_0)$, $0 < |m| \le 12 R^\es$, $m \neq n_0$. Then, $| |k + m \omega| - |k| | > (\delta_0^\esone)^{1/16}/2$.
\end{lemma}

\begin{proof}
Assume $k_{n_0} > 0$, $k + m \omega \le 0$. Then,
\begin{equation} \label{eq:8kn0casesA1}
| |k + m \omega| - |k| | = |2 k + m \omega| \ge |(m - n_0) \omega| - 2 |k - k_{n_0}| \ge (\delta_0^\esone)^{1/16} - 4 \sigma(n_0) > (\delta_0^\esone)^{1/16}/2.
\end{equation}
The verification for the rest of the cases is similar.
\end{proof}

\begin{lemma}\label{lem:8A.3}
Let $k \in \cR^{(s,s)}(\omega,n_0)$. Then,
\begin{itemize}

\item[(0)] $n_0 \in \cM^\esone_{k,s-1}$.

\item[(1)] The subsets in \eqref{eq:A.1} are well-defined. For $|k'-k| < 2 \delta^{(s-2)}_0$, each matrix $H_{\La^{(r)}_k(m), \ve, k'}$, $r \le s-1$ belongs to the class $\cN^{(r)}(m, \La^\ar_k(m), \delta^\zero_0)$.

\item[(2)] Let $m_0^+ = 0$, $m^-_0 = n_0$. For $|k' - k| < (\delta^{(s-1)}_0)^{1/6}$, the matrices $H_{\La^{(s')}_k(m), \ve, k'}$ obey all conditions stated in Definition~\ref{def:8-1a} $($ except for the fact that the set $\La$ is not defined $)$.

\item[(3)] Assume that $|k - k_{n_0}| > (\delta^\esone)^{7/8}$. Let  $\La^\es_k(0)$ be as in \eqref{eq:A.1}. Then, $H_{\La^\es_k(0), \varepsilon, k} \in \cN^{(s)} \bigl( 0, \La^\es_k(0); \delta_0 \bigr)$.

\end{itemize}
\end{lemma}

\begin{proof}
Clearly, $k \in \IR \setminus \bigcup_{0 < |m'| \le 12 R^{(s-1)}} (k_{m',s-1}^-, k_{m',s-1}^+)$. One has
\begin{equation}\label{eq:8-5-4UUU}
\begin{split}
|n_0 \omega| < 2 |k| + 1, \\
|v(n_0,k) - v(0,k)| = 2\lambda^{-1} |n_0 \omega| \cdot |k - k_{n_0}| < 4 \lambda^{-1}(2|k|+1) \sigma(n_0) < 256 (\delta^\esone_0)^{1/6} < 3 \delta^{(s-2)}_0/4,
\end{split}
\end{equation}
which means $n_0 \in \cM^\esone_{k,s-1}$, due to the definitions in \eqref{eq:A.1}.

Part $(1)$ is due to part $(2)$ of Remark~\ref{rem:7.kawayfromzero} after Proposition~\ref{prop:A.3}.

To prove $(2)$, note that all conditions except \eqref{eq:4-3AAAAAmnotm0} and \eqref{eq:4-3AAAAA} are due to part $(2)$ of Remark~\ref{rem:7.kawayfromzero}. Furthermore, due to part $(2)$ of Remark~\ref{rem:7.kawayfromzero}, one has for any $m \in \cM^\esone_{k,s-1}$,
\begin{equation}\label{eq:8.condPROP}
|E^{(s-1)}(m, \La^{(s-1)}_{k}(m); \ve, k') - E^{(s-1)}(0, \La^{(s-1)}_k(0); \ve, k')| \ge \frac{7}{8\lambda} | (k' + m \omega)^2 - (k')^2 | - 12 |\ve| (\delta^{(s-1)}_0)^4.
\end{equation}
Take here $m \neq 0,n_0$. Then, combining \eqref{eq:8.condPROP} with  with
Lemma~\ref{lem:8A.2}, one obtains condition \eqref{eq:4-3AAAAAmnotm0}. Once again, due to part $(2)$ of Remark~\ref{rem:7.kawayfromzero}, one has
\begin{equation}\label{eq:8.condeupperEXCL}
\begin{split}
|E^{(s-1)}(n_0, \La^{(s-1)}_{k}(n_0); \ve, k') - E^{(s-1)}(0, \La^{(s-1)}_k(0); \ve, k')| \\
\le  3| |k' + n_0 \omega| - |k'| | + (\delta^{(s-1)}_0)^5 \le 6\sigma(n_0)+ 6 |k' - k_{n_0}| + (\delta^{(s-1)}_0)^5
< (\delta^{(s-1)}_0)^{1/8},
\end{split}
\end{equation}
as required in condition \eqref{eq:4-3AAAAA}.

Assume that $|k - k_{n_0}| > (\delta^\esone)^{7/8}$. To prove part $(3)$, we need to verify the lower estimate in the first line in condition \eqref{eq:4-3sge3} in Definition~\ref{def:4-1} only. Assume for instance, $|k| \le 1$. In this case, $\lambda = 256$. Recall that due to \eqref{eq:diphnores}, $|n_0 \omega| > (\delta_0^{(s-1)})^{1/16}$. Due to part $(2)$ of Remark~\ref {rem:7.kawayfromzero}, one has
\begin{equation}\label{eq:8.condedefi51b}
\begin{split}
|E^{(s-1)}(n_0, \La^{(s-1)}_{k}(n_0); \ve, k) - E^{(s-1)}(0, \La^{(s-1)}_k(0); \ve, k)| \ge \frac{7}{8\times 256} | (k + n_0 \omega)^2 - k^2 | -
(\delta^{(s-1)}_0)^5 > 3 \delta^{(s-1)}_0 \\
= \frac{7}{8\times 256} 2 |n_0 \omega| |k - k_{n_0}| - (\delta^{(s-1)}_0)^5 > 3 \delta^{(s-1)}_0,
\end{split}
\end{equation}
as required. The case $|k| \ge 1$ is completely similar.
\end{proof}

\begin{remark}\label{rem:8.ksetintersectA}
From this point to the end of Proposition~\ref{prop:8.1}, we always assume that $k \in \cR^{(s,s)}(\omega,n_0)$; and moreover,
\begin{equation}\label{eq:8realPR}
|k - k_{n_0}| \le (\delta^\esone)^{3/4}.
\end{equation}
On the set $(\delta^\esone)^{7/8} < |k - k_{n_0}| \le (\delta^\esone)^{3/4}$, we will be able to apply both Proposition~\ref{prop:A.3} and
Proposition~\ref{prop:8.1}.
\end{remark}

Let $T$ be the reflection map $T(n) = -n + n_0$. For $s > 1$, due to Lemma~\ref{lem:8A.3}, the subsets in \eqref{eq:A.1} are well-defined, and each matrix $H_{\La^{(r)}_k(m), \ve, k}$, $r \le s-1$ belongs to the class $\cN^{(r)}(m, \La^\ar_k(m), \delta^\zero_0)$. Assume that $n_0 \in \La^\es_k(0)$. Then, $H_{\La^\es_k(0), \varepsilon, k} \in \widehat{OPR^{(s)}} \bigl( 0, n_0, \La^\es_k(0); \delta_0 \bigr)$. We will now re-define the set $\La^\es_k(0)$ so that $H_{\La^\es_k(0), \varepsilon, k} \in OPR^{(s)} \bigl( 0, n_0, \La^\es_k(0); \delta_0, \tau^{(0)} \bigr)$. To this end we will define the set $\La^\es_k(0)$ so that $T(\La^\es_k(0)) = \La^\es_k(0)$, where $T(n) = -n + n_0$. \textit{Provided that $k \not= k_{n_0}$, this symmetry will imply condition \eqref{eq:5-13NNNN1} in Definition~\ref{def:5-2a} with some $\tau^{(0)} = \tau^{(0)}(k) > 0$}. For $s = 1$, set
\begin{equation}\label{eq:8A.1definsets}
\La^\one_k(0) = B(3 R^\one) \cup T(B(3R^\one)).
\end{equation}
For $s > 1$, the ``new'' set $\La^\es_k(0)$ will be a ``relatively small perturbation'' of the set
$$
\mathfrak{B}(n_0,s) := B(3 R^\es) \cup T(B(3R^\es)).
$$

\begin{lemma}\label{lem:8mdeltaTm}
$(1)$ If $|v(m,k) - v(0,k)| < \delta$, with $(\delta^{(s-1)}_0)^{1/2}/4 \le \delta\le 1/256$, then $|v(T(m),k) - v(0,k)| < 4 \delta/3$.

$(2)$ Let $s \ge 2$, $1 \le s' \le s-1$, $m_j \in \cM^{(s')}_{k,s-1}$, $j = 1, 2$, and assume that $T(m_1) \neq m_2$. Then, $\dist(T(\La^{(s')}(m_1)), \La^{(s')}(m_2)) > 6 R^{(s')}$.
\end{lemma}

\begin{proof}
$(1)$ Since $k \in \cR^{(s,s)}(\omega,n_0)$, one has $|n_0 \omega| < 2 |k| + 1$. Since $|v(m,k) - v(0,k)| < \delta$, it follows from $(1)$ in Lemma~\ref{lem:A.2} that $|m \omega| < 2 |k| + 1$. One has
\begin{equation}\label{eq:8Tmineqverif-2}
\begin{split}
|v(T(m),k) - v(0,k)| \le |v(T(m),k) - v(m,k)| + |v(m,k) - v(0,k)| \\
\le \lambda^{-1}(|n_0 \omega| + 2 |m \omega|) |2 k + n_0 \omega| + \delta < 8 \lambda^{-1} (|k| + 1) (\delta^\esone)^{3/4} + \delta < 4 \delta/3.
\end{split}
\end{equation}

$(2)$ One has $|v(m_j,k) - v(0,k)| < 3 \delta^{(s'-1)}/4 < (\delta^{(s'-1)})^{1/2}/4$, $j = 1, 2$. Due to part $(1)$, one also obtains $|v(T(m_1),k) - v(0,k)| < (\delta^{(s'-1)})^{1/2}$. Due to part $(4)$ of Lemma~\ref{lem:A.2}, one has $|T(m_1) - m_2| > 12 R^{(s')}$, since $T(m_1) \neq m_2$. This implies $\dist(T(\La^{(s')}(m_1)), \La^{(s')}(m_2)) > 6 R^{(s')}$.
\end{proof}

\begin{defi}\label{defi:8.LLL}
Let $\mathfrak{L}'$ be the collection of all sets $\La(m) := \La^{(s')}(m) \cup T(\La^{(s')}(m))$, $1 \le s' \le s-1$, $m \in \cM^{(s')}_{k,s-1}$. We say that $\La(m_1) \approx \La(m_2)$ if $s_1 = s_2$, and either $m_1 = m_2$ or $T(m_1) = m_2$. Clearly, this is an equivalence relation on $\mathfrak{L}'$. Let $\mathfrak{M}$ be the set of equivalence classes. Clearly, each class has at most two elements in it. For each $\mathfrak{m} \in \mathfrak{M}$, set $\La(\mathfrak{m}) = \bigcup_{\La(m_1) \in \mathfrak{m}} \La(m_1)$. Set $\mathfrak{L} = \{ \La(\mathfrak{m}) : \mathfrak{m} \in \mathfrak{M} \}$. Let $\La(\mathfrak{m}) \in \mathfrak{L}$, $ \La^{(s')}(m) \cup T(\La^{(s')}(m)) \in \mathfrak{m}$. Set $t(\La(\mathfrak{m})) = s'$. This defines an $\mathbb{N}$-valued function on $\mathfrak{L}$. Set also $p_\mathfrak{m} = \{m,T(m)\}$. Clearly, the set $p_\mathfrak{m}$ depends only on $\mathfrak{m}$.
\end{defi}

\begin{lemma}\label{lem:8.lLL}
Using the notation from Definition~\ref{defi:8.LLL}, the following statements hold.

$(1)$ For any $\La(\mathfrak{m}_j) \in \mathfrak{L}$, $j = 1, 2$, such that $t(\La(\mathfrak{m}_1)) = t(\La(\mathfrak{m}_2))$, $\mathfrak{m}_1 \neq \mathfrak{m}_2$, we have $\dist (\La(\mathfrak{m}_1), \La(\mathfrak{m}_2)) \ge R^{(t(\La(\mathfrak{m}_1)))}$.

$(2)$ For any $\mathfrak{m}$, we have
\begin{equation}\label{eq:8Lmsets}
\bigcup_{m \in p_\mathfrak{m}} \bigl( (m + B(2 R^{(t(\La(\mathfrak{m})))})) \bigr) \subset \La(\mathfrak{m}) \subset \bigcup_{m \in p_\mathfrak{m}} \bigl( (m + B(3 R^{(t(\La(\mathfrak{m})))})) \bigr).
\end{equation}
Furthermore, $\La(\mathfrak{m}) = \Xi(\mathfrak{m}) \cup T(\Xi(\mathfrak{m}))$, where $\diam (\Xi(\mathfrak{m})) \le 6 R^{(t(\La(\mathfrak{m})))}$.

$(3)$ If $\mathfrak{m}_1 \neq \mathfrak{m_2}$, then $\La(\mathfrak{m}_1) \neq \La(\mathfrak{m}_2)$.

$(4)$ The pair $(\mathfrak{L},t)$ is a proper subtraction system; see Definition~\ref{defi:5.twolambdas6}.

$(5)$ For any $\mathfrak{m}$, we have $\La(\mathfrak{m}) = T(\La(\mathfrak{m}))$.
\end{lemma}

\begin{proof}
$(1)$ Let $ \La^{(s')}(m_j) \cup T(\La^{(s')}(m_j)) \in \mathfrak{m_j}$, $j = 1, 2$. Since $\mathfrak{m_1} \neq \mathfrak{m_2}$, $p_{\mathfrak{m_1}} \cap p_{\mathfrak{m_2}} = \emptyset$. Therefore, $\dist(\La^{(s')}(m_1), \La^{(s')}(m_2)) > 6 R^{(s')}$,  $\dist(T(\La^{(s')}(m_1)), T(\La^{(s')}(m_2))) > 6 R^{(s')}$. Furthermore, due to part $(2)$ of Lemma~\ref{lem:8mdeltaTm}, $\dist(T(\La^{(s')}(m_1)), \La^{(s')}(m_2)) > 6 R^{(s')}$, $\dist(T(\La^{(s')}(m_2)), \La^{(s')}(m_1)) > 6 R^{(s')}$. This implies the statement in $(1)$.

$(2)$ Let $ \La^{(s')}(m') \cup T(\La^{(s')}(m')) \in \mathfrak{m}$. One has
\begin{equation}\label{eq:8Lmsets1-2}
\bigl(m' + B(2R^{(t(\La(\mathfrak{m})))})\bigr) \subset \La^{(s')}(m') \subset \bigl(m' + B(3 R^{(t(\La(\mathfrak{m})))})\bigr).
\end{equation}
Furthermore, $\{m', T(m')\} = p_\mathfrak{m}$. This implies the first statement in $(2)$. The second statement in $(2)$ follows from Definition~\ref{defi:8.LLL}.

$(3)$ Let $\mathfrak{m}_1 \neq \mathfrak{m_2}$. If $t(\La(\mathfrak{m}_1)) = t(\La(\mathfrak{m}_2))$, then $(3)$ follows from $(1)$. If $t(\La(\mathfrak{m}_1)) \neq t(\La(\mathfrak{m}_2))$, then $(3)$ follows from $(2)$.

Now we will verify $(4)$. Assume that $t(\La') = t(\La'')$, $\La' \neq \La''$. It follows from $(3)$ and $(1)$ that $\dist (\La', \La'') \ge R^{(t(\La'))}$. So, $(i)$ from part $(2)$ of Definition~\ref{defi:5.twolambdas6} holds with $R_a \ge R^{(a)}$. Let $\La(\mathfrak{m})$ be arbitrary, $a = t(\La(\mathfrak{m}) + 1$. Due to part $(2)$, one has $\La(\mathfrak{m}) = \Xi(\mathfrak{m}) \cup T(\Xi(\mathfrak{m}))$ with $\diam (\Xi(\mathfrak{m})) \le 6 R^{(t(\La(\mathfrak{m})))} = 6 R^{(a-1)} < 2^{-a} R^{(a)} \le 2^{-a} R_a$. Furthermore, let $\La(\mathfrak{m}')$ be arbitrary. Assume $\La(\mathfrak{m}) \cap \La(\mathfrak{m}') \neq \emptyset$. Once again, due to part $(2)$, one has $\La(\mathfrak{m}) = \Xi(\mathfrak{m}') \cup T(\Xi(\mathfrak{m}'))$. This implies $\Xi(\mathfrak{m}) \cap \La(\mathfrak{m}') \neq \emptyset$ and $T(\Xi(\mathfrak{m})) \cap \La(\mathfrak{m}') \neq \emptyset$. Hence, $(ii)$ from part $(2)$ of the Definition~\ref{defi:5.twolambdas6} holds as well. This finishes $(4)$.

Part $(5)$ follows from the definition of the sets $\La(\mathfrak{m})$.
\end{proof}

Set
\begin{equation} \label{eq:8.twolambdas5}
\begin{split}
\mathfrak{B}(n_0,s) := B(3 R^\es) \cup T(B(3 R^\es)), \\
\mathfrak{B}(n_0,s,\ell) = \mathfrak{B}(n_0, s, \ell-1)  \setminus \Bigl( \bigcup_{\mathfrak{m} \in \mathfrak{M} : \La(\mathfrak{m}) \between \mathfrak{B}(n_0,s,\ell-1)} \La(\mathfrak{m})\Bigr)
\end{split}
\end{equation}
for $\ell = 1, 2, \ldots$.

\begin{lemma}\label{lem:8setLambdas}
$(1)$ There exists $\ell_0 < 2^s$ such that $\mathfrak{B}(n_0,s,\ell) = \mathfrak{B}(n_0, s, \ell+1)$ for any $\ell \ge \ell_0$.

$(2)$ For any $\La \in \mathfrak{L}$, we have either $\La \subset \mathfrak{B}(n_0,s,\ell_0)$ or $\La \subset \Bigl( \IZ^\nu \setminus
\mathfrak{B}(n_0, s, \ell_0) \Bigr)$.

$(3)$ Set $\La^{(s,\mathbf{1})}_k(0) = \mathfrak{B}(n_0,s,\ell_0)$. Then, for any $\La^{(s')}(m)$, we have either $\La^{(s')}(m) \cap \La^{(s,\mathbf{1})}_k(0) = \emptyset$ or $\La^{(s')}(m) \subset \La^{(s,\mathbf{1})}_k(0)$.

$(4)$ $T(\mathfrak{B}(n_0,s,\ell)) = \mathfrak{B}(n_0,s,\ell)$ for any $\ell$. In particular, $T(\La^{(s,\mathbf{1})}_k(0)) = \La^{(s,\mathbf{1})}_k(0)$.

$(5)$ For any $\ell \ge 1$, we have
\begin{equation}\label{eq:8.twolambdas5NEW}
\{ n \in \mathfrak{B}(n_0,s,\ell-1)) : \dist (n, \IZ^\nu \setminus \mathfrak{B}(n_0,s,\ell-1)) \ge 6 R^\esone \} \subset
\mathfrak{B}(n_0,s,\ell) \subset \mathfrak{B}(n_0,s,\ell-1)).
\end{equation}
In particular, $B(2 R^\es) \cup (n_0 + B(2 R^\es)) \subset \La^{(s,\mathbf{1})}_k(0) \subset B(3 R^\es) \cup (n_0 + B(3 R^\es))$.
\end{lemma}

\begin{proof}
Parts $(1)$, $(2)$ follow from Lemma~\ref{lem:5.twolambdas4}.

Let $\La^{(s')}(m)$ be such that $\La^{(s')}(m) \cap \mathfrak{B}(n_0,s,\ell_0) \neq \emptyset$. Let $\mathfrak{m}$ be the equivalence class containing  $\La^{(s')}(m) \cup T(\La^{(s')}(m))$. Then, just by definition, $\La^{(s')}(m) \subset \La(\mathfrak{m})$. In particular, $\La(\mathfrak{m}) \cap \mathfrak{B}(n_0,s,\ell_0) \neq \emptyset$. This implies $\La(\mathfrak{m}) \subset \mathfrak{B}(n_0,s,\ell_0)$. Therefore, $\La^{(s')}(m) \subset \mathfrak{B}(n_0,s,\ell_0)$. This finishes the proof of $(3)$.

To verify $(4)$ note that $T(\mathfrak{B}(n_0,s,0)) = \mathfrak{B}(n_0,s,0)$. Combining this with part $(5)$ of Lemma~\ref{lem:8.lLL}, one obtains $T(\mathfrak{B}(n_0,s,\ell)) = \mathfrak{B}(n_0,s,\ell)$ for any $\ell$, as claimed.

It follows from \eqref{eq:8Lmsets} that $\diam (\La(\mathfrak{m})) \le 6 R^\esone$ for any $\mathfrak{m}$. Combining this with \eqref{eq:8.twolambdas5}, one obtains \eqref{eq:8.twolambdas5NEW}. The second statement in $(5)$ follows from \eqref{eq:8.twolambdas5NEW} since $\ell_0<2^s$.
\end{proof}

\begin{remark}\label{rem:7.Laksetsambig}
$(1)$ If $k \in \cR^{(s,s)}(\omega,n_0)$, then $-k \in \cR^{(s,s)}(\omega,-n_0)$, $|-k - k_{-n_0}| = |k - k_{n_0}|$, and $\La^{(s,\mathbf{1})}_{-k}(0) = -\La^{(s,\mathbf{1})}_k(0)$.

$(2)$ Let $(\delta^\esone)^{7/8} < |k - k_{n_0}| \le (\delta^\esone)^{3/4}$. Due to part $(3)$ in Lemma~\ref{lem:8A.3}, the subset $\La^\es_k$ is well-defined by \eqref{eq:A.1}; moreover, $H_{\La^\es_k(0), \varepsilon, k} \in \cN^{(s)} \bigl( 0, \La^\es_k(0); \delta_0 \bigr)$. The notation $\La^{(s,\mathbf{1})}_k(0)$ is introduced to avoid ambiguity.
\end{remark}

\begin{prop}\label{prop:8.1}
Assume that  $k \in \cR^{(s,s)}(\omega,n_0)$, $|k - k_{n_0}| \le (\delta^\esone)^{3/4}$.

Assume $k_{n_0} > 0 $. Let $\ve_0$, $\ve_{s}$ be as in Definition~\ref{def:4-1}. Let $\ve \in (-\ve_{s},\ve_{s})$.
\begin{itemize}

\item[(1)] If $s = 1$, then for any $0 < |m| \le 12 R^{(1)}$, $m \neq n_0$, and any $|k_1 - k| < \delta(1) := |\ve| (\delta^{(0)}_0)^5$, we have $|v(m,k_1) - v(0,k_1)| \ge \delta_0$. If $s \ge 2$, $0 < |m| \le 12 R^{(s)}$, $m \notin \bigcup_{1 \le r \le s-1} \bigcup_{m' \in \cM^{(r)}_{k,s-1}} \La^{(r)}_{k}(m')$, then $|v(m,k) - v(0,k)| \ge \delta_0/2$.

\item[(2)] For any $k_{n_0} < k' \le k_{n_0} + (\delta^\esone)^{3/4}$, we have $H_{\La^{(s,\mathbf{1})}_k(0), \ve, k'} \in OPR^{(s)} \bigl( 0, n_0, \La^{(s,\mathbf{1})}_k(0); \delta_0, \tau^\zero \bigr)$, $\tau^\zero = [\min(2 \ve_0^{3/4}, k_{n_0}/256)] |k' - k_{n_0}|$. For any $k_{n_0} - (\delta^\esone)^{3/4} \le k' < k_{n_0}$, we have $H_{\La^{(s,\mathbf{1})}_k(0), \ve, k'} \in OPR^{(s)} \bigl( n_0, 0, \La^{(s,\mathbf{1})}_k(0); \delta_0, \tau^\zero \bigr)$.

\item[(3)] For $k_{n_0} < k' \le k_{n_0} + (\delta^\esone)^{3/4}$, we denote by $E^{(s,\pm)}(0, \La^{(s,\mathbf{1})}_k(0); \ve, k')$ the functions defined in Proposition~\ref{prop:8-5n} with $H_{\La^{(s,\mathbf{1})}_k(0), \ve, k'}$ in the role of $\hle$. Similarly, for $k_{n_0} - (\delta^\esone)^{3/4} \le k' < k_{n_0}$, we denote by $E^{(s,\pm)}(n_0, \La^{(s,\mathbf{1})}_k(0); \ve, k')$ the functions defined in Proposition~\ref{prop:8-5n}. Then, with $k^\zero:=\min(\ve_0^{3/4}, k_{n_0}/512)$, one has
\begin{equation}\label{eq:8Ekderivatives}
\begin{split}
\partial_\theta E^{(s,+)}(0, \La^{(s,\mathbf{1})}_k(0); \ve, k_{n_0} + \theta) > (k^\zero)^2 \theta, \quad \theta > 0, \\
\partial_\theta E^{(s,-)}(0, \La^{(s,\mathbf{1})}_k(0); \ve, k_{n_0} + \theta) < -(k^\zero)^2 \theta, \quad \theta > 0,
\end{split}
\end{equation}
\begin{equation}\label{eq:8Esymmetry}
E^{(s,\pm)}(0, \La^{(s,\mathbf{1})}_k(0); \ve, k_{n_0} + \theta) = E^{(s,\pm)}(n_0, \La^{(s,\mathbf{1})}_k(0); \ve, k_{n_0} - \theta), \quad \theta > 0,
\end{equation}
\begin{equation}\label{eq:8Efirstder}
|\partial_\theta E^{(s,\pm)}(0, \La^{(s,\mathbf{1})}_k(0); \ve, k_{n_0} + \theta) | \le 2,
\end{equation}
\begin{equation}\label{eq:9Ekk1EG}
|E^{(s,\pm)}(0, \La^{(s,\mathbf{1})}_k(0); \ve, k_1) - E^{(s, \pm)}(0, \La^{(s,\mathbf{1})}_{k_1}(0); \ve, k_1)| < |\ve|(\delta^\es_0)^5, \quad 0 < |k_1 - k_{n_0}| < |\ve| (\delta^\esone)^{3/4}
\end{equation}
\begin{equation}\label{eq:9Esplit}
E^{(s,+)}(0, \La^{(s,\mathbf{1})}_k(0); \ve, k') - E^{(s, -)}(0, \La^{(s,\mathbf{1})}_k(0); \ve, k') > (k^\zero |k' - k_{n_0}|)^2/2.
\end{equation}

\item[(4)]
\begin{equation}\label{eq:8kk1comp1}
|E^{(s,\pm)}(0, \La^{(s,\mathbf{1})}_k(0); \ve, k') - E^{(s-1)}(0, \La^{(s-1)}_k(0); \ve, k')| \le 4 |\ve|(\delta^\esone_0)^{1/8}.
\end{equation}
Here, $E^{(0)}(m',\La';\ve,k') := v(m',k')$, as usual. In particular,
\begin{equation}\label{eq:8kk1comp1EpEm}
|E^{(s,+)}(0, \La^{(s,\mathbf{1})}_k(0); \ve, k) - E^{(s,+)}(0, \La^{(s,\mathbf{1})}_{k'}(0); \ve, k')| \le 9|\ve|(\delta^\esone_0)^{1/8}.
\end{equation}

\item[(5)] Let $k \in \IR \setminus \bigcup_{|m| \le 12 R^\es, \; m \neq n_0} (k_{m,s}^-, k_{m,s}^+)$, $(\delta^\esone)^{7/8} < |k - k_{n_0}| \le (\delta^\esone)^{3/4}$. Then, $H_{\La^{(s,\mathbf{1})}_k(0), \varepsilon, k} \in \cN^{(s)} \bigl( 0, \La^{(s,\mathbf{1})}_k(0); \delta_0 \bigr)$. Furthermore,
$$
E^{(s)}(0, \La^{(s,\mathbf{1})}_k(0); \ve, k) = \begin{cases} E^{(s,+)}(0, \La^{(s,\mathbf{1})}_k(0);\ve,k) \quad \text {if $(\delta^\esone)^{7/8} < k - k_{n_0} \le (\delta^\esone)^{3/4}$}, \\ E^{(s,-)}(0, \La^{(s,\mathbf{1})}_k(0); \ve, k) \quad \text {if $(\delta^\esone)^{7/8} < k_{n_0} - k \le (\delta^\esone)^{3/4}$}.
\end{cases}
$$

\item[(6)] For any $k_{n_0} < k' \le k_{n_0} + (\delta^\esone)^{3/4}$, we have $H_{\La^{(s,\mathbf{1})}_k(0), \ve, k'} \in OPR^{(s)} \bigl(0, n_0, \La^{(s,\mathbf{1})}_k(0); \delta_0, \tau^\zero \bigr)$. For any $k_{n_0} - (\delta^\esone)^{3/4} \le k' < k_{n_0}$, we have $H_{\La^{(s,\mathbf{1})}_k(0), \ve, k'} \in OPR^{(s)} \bigl(n_0,0, \La^\es_k(0); \delta_0, \tau^\zero\bigr)$. Furthermore,
$$
E^{(s,\pm)}(0, \La^{(s,\mathbf{1})}_k(0); \ve, k') = E^{(s,\pm)}(0, \La^{(s,\mathbf{1})}_k(0); \ve, -k').
$$

\end{itemize}
\end{prop}

\begin{proof}
The verification of part $(1)$ goes the same way as in Proposition~\ref{prop:A.3}.

It follows from $(2)$ in Lemma~\ref{lem:8A.3} that $H_{\La^{(s,\mathbf{1})}_k(0), \ve, k'} \in \widehat{OPR^{(s)}} \bigl( 0, n_0, \La^{(s,\mathbf{1})}_k(0); \delta_0 \bigr)$. Now we will verify \eqref{eq:5-13NNNN1} in Definition~\ref{def:5-2a}. \textit{As we mentioned before, the symmetry $T(\La^{(s,\mathbf{1})}_k(0)) = \La^{(s,\mathbf{1})}_k(0)$ plays a crucial role for that matter}. Set for convenience $\La = \La^{(s,\mathbf{1})}_k(0)$, $m_0^+ = 0$, $m_0^- = n_0$, $\La_{m^+_0,m_0^-} = \La^{(s,\mathbf{1})}_k(0) \setminus \{0, n_0\}$. As in \eqref{eq:5-10ac}, consider the functions
\begin{equation} \label{eq:8-10acbasicfunctions}
\begin{split}
K^{(s)}(m, n, \La; \ve, k_{n_0} + \theta, E) = (E - H_{\La_{m^+_0,m_0^-}, \ve, k_{n_0} + \theta})^{-1}(m,n), \quad m, n \in \La_{m^+_0,m_0^-}, \\
Q^{(s)}(m_0^\pm, \La; \ve, k_{n_0} + \theta, E) \\
= \sum_{m',n' \in \La_{m^+_0,m_0^-}} h(m_0^\pm, m'; \ve, k_{n_0} + \theta) K^{(s)}(m', n'; \La; \ve, k_{n_0} + \theta, E) h(n', m_0^\pm; \ve, k_{n_0} + \theta), \\
G^\es(m^\pm_0, m^\mp_0, \La; \ve, k_{n_0} + \theta, E) = h(m_0^\pm, m_0^\mp; \ve, k_{n_0} + \theta) \\
+ \sum_{m', n' \in \La_{m^+_0,m_0^-}} h(m_0^\pm, m'; \ve, k_{n_0} + \theta) K^{(s)}(m', n'; \La; \ve, k_{n_0} + \theta, E) h(n', m_0^\mp; \ve, k_{n_0} + \theta)
\end{split}
\end{equation}
with $|\theta| < (\delta_0^\esone)^{3/4}$,
\begin{equation}\label{eq:8.domainOPAAAA}
\begin{split}
|\ve| < \ve_0, \quad |E - v(m_0^+)| < \delta_0/4, \quad \text { in case $s = 1$}, \\
|\ve| < \ve_{s-2} := \ve_0 - \sum_{1 \le s' \le s-2} \delta^{(s')}_0, \quad \big| E - E^{(s-1)}(m^+_0, \La^{(s-1)}_k(m_0^+); \ve, k_{n_0} + \theta) \big| < \rho_0 := 2 \delta^{(s-1)}_0, \quad s \ge 2.
\end{split}
\end{equation}
One can estimate $\partial_\theta Q^{(s)}(m_0^\pm, \La; \ve, k_{n_0} + \theta, E)$, $\partial^2_{\theta, \theta} Q^{(s)}(m_0^\pm, \La; \ve, k_{n_0} + \theta, E)$ using Lemma~\ref{lem:5differentiation}. Like in the proof of Proposition~\ref{prop:A.3}, taking into account that
\begin{equation} \label{eq:8-10partialderivv}
\begin{split}
\partial_\theta h(m, n; \ve, k_{n_0} + \theta) = 0, \quad \text{ if $m \neq n$}, \\
|\partial_\theta h(m, m; \ve, k_{n_0} + \theta)| = 2\lambda^{-1} |k_{n_0} + \theta + m \omega| < 8 \exp (|m|^{1/5}), \\
\partial^2_\theta h(m, m; \ve, k_{n_0} + \theta) = 2\lambda^{-1},
\end{split}
\end{equation}
one obtains for $\alpha \le 2$,
\begin{equation}\label{eq:8hHinvestimatestatement1kvard}
\begin{split}
|\partial^\alpha_{\theta} Q^{(s)}(m_0^\pm, \La; \ve, k_{n_0} + \theta, E)| < |\ve|^{4/3} = (\epsilon \lambda^{-1})^{4/3}, \\
|\partial^\alpha_{\theta} Q^{(s)}(m_0^\pm, \La; \ve, k_{n_0} + \theta, E) - \partial^\alpha_{\theta} Q^{(s-1)}(m_0^\pm, \La(m_0^\pm); \ve, k_{n_0} + \theta, E)| < |\ve| (\delta^\esone_0)^5, \\
|\partial^\alpha_{\theta} G^\es(m^\pm_0, m^\mp_0, \La; \ve, k_{n_0} + \theta, E) < |\ve|^{4/3} \exp \Big( -\frac{\kappa_0}{4} |n_0| \Big) = (\epsilon \lambda^{-1})^{4/3} \exp \Big( -\frac{\kappa_0}{4} |n_0| \Big).
\end{split}
\end{equation}

Now we invoke the symmetry $T(\La) = \La$. Note that $T(m^\pm_0) = m^\mp_0$. In particular, $T(\La_{m^+_0, m_0^-}) = \La_{m^+_0, m_0^-}$. Note also that $h(m, n; \ve, k) = \ve c(n - m)$ for $m \neq n$, that is, it does not depend on $k$. Finally, note that for $k' = k_{n_0} + \theta$, we have $-(k' + n_0 \omega) = (k_{n_0} - (k'-k_{n_0})) = k_{n_0} - \theta$. Using Lemma~\ref{lem:basicshiftprop}, one obtains
\begin{equation} \label{eq:8-basicfunctionsrelat}
\begin{split}
K^{(s)}(m, n, \La; \ve, k_{n_0} + \theta, E) = K^{(s)}(T(m), T(n), \La; \ve, k_{n_0} - \theta, E), \quad m, n \in \La_{m^+_0, m_0^-}, \\
Q^{(s)}(m_0^\pm, \La; \ve, k_{n_0} + \theta, E) = Q^{(s)}(m_0^\mp, \La; \ve, k_{n_0} - \theta, E), \\
G^\es(m^\pm_0, m^\mp_0, \La; \ve, k_{n_0} + \theta, E) = G^\es(m^\mp_0, m^\pm_0, \La; \ve, k_{n_0} - \theta, E).
\end{split}
\end{equation}
Note also that $v(m^-_0, k_{n_0} + \theta) = v(m^+_0, k_{n_0} - \theta)$.

Assume first $k_{n_0} > 1$. Note that $k_{n_0} \le \gamma = \lambda/256 \le k_{n_0} + 1 \le 2 k_{n_0}$. For $-(\delta^\esone)^{3/4} < \theta < (\delta^\esone)^{3/4}$, we have
\begin{equation}\label{eq:8hHinvestimatestatement1kvard1}
\begin{split}
1/64 > 2 \lambda^{-1} k_{n_0} + 2 (\delta^\esone)^{3/4} + \ve_0^{4/3} > \partial_{\theta} (v(m^+_0, k_{n_0} + \theta) + Q^{(s)}(m_0^+, \La; \ve, k_{n_0} + \theta, E)) \\
> 2 \lambda^{-1} k_{n_0} - 2 (\delta^\esone)^{3/4} - \ve_0^{4/3} > 1/512 - 2 (\delta^\esone)^{3/4}- \ve_0^{4/3} > 1/1024.
\end{split}
\end{equation}
So, for $(\delta^\esone)^{3/4}>\theta > 0$, we have
\begin{equation}\label{eq:8hHinvestimatestatement1kvard2}
v(m^+_0, k_{n_0} + \theta) + Q^{(s)}(m_0^+, \La; \ve, k_{n_0} + \theta, E) - v(m^-_0, k_{n_0} + \theta) - Q^{(s)}(m_0^-, \La; \ve, k_{n_0} + \theta, E) > \theta/512.
\end{equation}
Taking here $\theta = k' - k_{n_0}$, one concludes that condition \eqref{eq:5-13NNNN1} in Definition~\ref{def:5-2a} holds with $H_{\La^{(s,\mathbf{1})}_k(0), \ve, k'}$ in the role of $\hle$. Thus, $H_{\La^{(s,\mathbf{1})}_k(0), \ve, k'} \in OPR^{(s)} \bigl( 0, n_0, \La^\es_k(0); \delta_0, \tau^\zero \bigr)$ if $k_{n_0} < k' \le k_{n_0} + (\delta^\esone)^{3/4}$. If $k_{n_0} - (\delta^\esone)^{3/4} \le k' < k_{n_0}$, one just has to switch the roles of $0$ and $n_0$.

Assume now $1 \ge k_{n_0} > 256 (2 (\delta^\esone)^{3/4} + \ve_0^{4/3})$. Note that in this case, $\lambda = 256$. Like above, one concludes that
\begin{equation}\label{eq:8hHinvestimatestatement1kvard2AGU}
\begin{split}
1/128 > \partial_{\theta} (v(m^+_0, k_{n_0} + \theta) + Q^{(s)}(m_0^+, \La; \ve, k_{n_0} + \theta, E)) > 2 \lambda^{-1} k_{n_0} - 2 (\delta^\esone)^{3/4} - \ve_0^{4/3} \ge 2 \ve_0^{4/3}, \quad |\theta| < \delta^\esone)^{3/4}, \\
v(m^+_0, k_{n_0} + \theta) + Q^{(s)}(m_0^+, \La; \ve, k_{n_0} + \theta, E) - v(m^-_0, k_{n_0} + \theta) - Q^{(s)}(m_0^-, \La; \ve, k_{n_0} + \theta, E) \\
> 2 \ve_0^{3/4} \theta, \quad 0 < \theta < \delta^\esone)^{3/4}.
\end{split}
\end{equation}

Finally, assume $0 < k_{n_0} \le 256 (2 (\delta^\esone_0)^{3/4} + \ve_0^{3/4})$. Once again, $\lambda = 256$. Find $r$ such that $(\delta^{(r)}_0)^{1/2} < k_{n_0} \le (\delta^{(r-1)}_0)^{1/2}$. Note first of all that in this case $($ see Remark~\ref{rem:8.ksetintersect} $)$,
\begin{equation}\label{eq:8hHinvestimatsmallkbasics}
|n_0| > 12 R^\ar, \quad\quad\quad s>r, \quad\quad k_{n_0} > (\delta^{(s-1)}_0)^{1/4}.
\end{equation}
Let for instance $k > k_{n_0}$, so that $m^+_0 = 0$. Recall that due to parts $(2)$ and $(4)$ of Proposition~\ref{prop:A.3}, the function $Q^{(r)}(0, \La^{(r)}_k; \ve, k_{n_0} + \theta, E)$ is well defined for $0 < k_{n_0} \le \delta^{(r-1)}_0/2$, $|\theta| < \delta^{(r-1)}_0$. Furthermore, due to \eqref{eq:8hHinvestimatestatement1kvard}, one has
\begin{equation}\label{eq:8hHinvestimatsmallk}
\partial^2_{\theta} (v(0, k_{n_0} + \theta) + Q^{(r)}(0, \La^{(r)}_k; \ve, k_{n_0}+\theta, E)) > 2 \lambda^{-1} - (\epsilon \lambda^{-1})^{4/3} > 1/256.
\end{equation}
Recall also that $v(0,k') + Q^{(r)}(0, \La^{(r)}_k; \ve, k', E) = v(0, -k') + Q^{(r)}(0, \La^{(r)}_k; \ve, -k', E)$. For $k_{n_0} + \theta >0$, this implies
\begin{equation}\label{eq:8hHinvestimatestatsmallk1}
\partial_{\theta} (v(0, k_{n_0} + \theta) + Q^{(r)}(0, \La^{(r)}_k; \ve, k_{n_0}+\theta, E)) > (k_{n_0} + \theta)/256.
\end{equation}
Now, we invoke the second estimate in \eqref{eq:8hHinvestimatestatement1kvard}, applied to $Q^{(r_1)}(0, \La^{(r_1)}_k; \ve, k_{n_0} + \theta, E))$, with $r_1$ in the role of $s$, running $r_1 = r+1, \dots, s$. This yields
\begin{equation}\label{eq:8hHinvestimatestatsmallk1r1}
|\partial_{\theta}  Q^{(r)}(0, \La^{(r)}_k; \ve, k_{n_0} + \theta, E)) - \partial_\theta Q^{(s)}(0, \La^{(s)}_k; \ve, k_{n_0} + \theta, E)| < 2 |\ve| (\delta^{(r)}_0)^5 < (k_{n_0} + \theta)/512.
\end{equation}
Combining \eqref{eq:8hHinvestimatestatsmallk1} with \eqref{eq:8hHinvestimatestatsmallk1r1}, one obtains
\begin{equation}\label{eq:8hHinvestimatestatsmallk2}
\begin{split}
\partial_{\theta} (v(0, k_{n_0} + \theta) + Q^{(s)}(0, \La^{(s)}_k; \ve, k_{n_0} + \theta, E)) > k_{n_0}/512, \quad  |\theta| < (\delta^\esone)^{3/4}, \\
v(m^+_0, k_{n_0} + \theta) + Q^{(s)}(m_0^+, \La; \ve, \theta, E) - v(m^-_0, k_{n_0} + \theta) - Q^{(s)}(m_0^-, \La; \ve, \theta, E) \\
> k_{n_0} \theta/256, \quad 0 < \theta < (\delta^\esone)^{3/4}.
\end{split}
\end{equation}
This finishes the proof of part $(2)$.

To prove part $(3)$, recall that due to Proposition~\ref{prop:8-5n}, $E^{(s,\pm)}(0, \La^{(s,\mathbf{1})}_k(0); \ve,k')$, $k' = k_{n_0} + \theta$, $\theta > 0$, are the solutions of the equation
\begin{equation} \label{eq:8Eequation}
\begin{split}
& \chi(\ve, \theta, E) := \bigl( E - v(m_0^+, k_{n_0} + \theta) - Q^\es(m^+_0, \La; \ve, k_{n_0} + \theta, E) \bigr) \cdot \bigl( E - v(m_0^-, k_{n_0} + \theta) - Q^\es(m_0^-, \La ; \ve, k_{n_0} + \theta, E) \bigr) \\
& \qquad - \big| G^\es(m^+_0, m^-_0, \La; \ve, k_{n_0} + \theta, E) \big|^2 = 0.
\end{split}
\end{equation}

To proceed with the verification of \eqref{eq:8Ekderivatives}--\eqref{eq:8kk1comp1}, note first of all that part $(4)$, which is \eqref{eq:8kk1comp1}, is due to \eqref{eq.5Eestimates1AP} from Proposition~\ref{prop:8-5n} combined with the fact that $H_{\La^{(s,\mathbf{1})}_k(0), \ve, k'} \in OPR^{(s)} \bigl( 0, n_0, \La^{(s,\mathbf{1})}_k(0); \delta_0, \tau^\zero \bigr)$.

To verify \eqref{eq:8Ekderivatives}, we invoke Lemma~\ref{lem:4zetatheta} with $\ell = 1$, $a_1 = v(m_0^+, k_{n_0} + \theta) + Q^\es(m^+_0, \La; \ve, k_{n_0} + \theta, E)$, $a_2 = v(m_0^+, k_{n_0} - \theta) + Q^\es(m^+_0, \La; \ve, k_{n_0} - \theta, E)$, $b = G^\es(m^+_0, m^-_0, \La; \ve, k_{n_0} + \theta, E)$, $\theta < (\delta^\esone)^{3/4}$. For that we first invoke part $(9)$ of Lemma~\ref{4.fcontinuedfrac}. Let us verify the validity of the needed conditions. First of all we need the conditions $|E - a_i|, b^2, |\partial_{E,\theta}^{\alpha_1,\alpha_2} a_i|, |\partial_{E,\theta}^{\alpha_1,\alpha_2} b^2| < 1/64$ from Definition~\ref{def:4a-functions}. The condition $|E - a_i| < 1/64$ is due to Proposition~\ref{prop:8-5n}. The rest is due to \eqref{eq:8hHinvestimatestatement1kvard} $($ note that $|\partial^\alpha_\theta v(m_0^+, k_{n_0} + \theta)| < 1/256$ $)$. Now we turn to the conditions of $(9)$ in Lemma~\ref{4.fcontinuedfrac}. Due to \eqref{eq:8hHinvestimatestatement1kvard1}, \eqref{eq:8hHinvestimatestatement1kvard2AGU},
\eqref{eq:8hHinvestimatestatsmallk2}, one has $\partial_\theta a_1 < -\min(\ve_0^{3/4}, k_{n_0}/512) = -k^\zero$. Due to \eqref{eq:8hHinvestimatestatement1kvard}, $|\partial^\alpha_{\theta} |b|^2| < \exp(-\frac{\kappa_0}{4} |n_0|) < \exp(-\frac{\kappa_0}{4} R^\esone) < (\delta^{(s-1)}_0)^2$, $\alpha \le 2$. Combining this with $|k_{n_0}| > (\delta^{(s-1)}_0)^{1/2}$, one concludes that $|\partial^2_\theta b^2| < (k^\zero)^2/8$. Recall also that $a_2(\ve,E,\theta) = a_1(\ve,E,-\theta)$, $|b(\ve,E,\theta) = |b(\ve,E,-\theta)|$. In particular, $\chi(\ve,E,\theta) = \chi(\ve,E,-\theta)$. Thus, all conditions needed for $(9)$ in Lemma~\ref{4.fcontinuedfrac} hold. Hence, \eqref{eq:4a-thetaest} from $(9)$ of Lemma~\ref{4.fcontinuedfrac} holds. Now \eqref{eq:8Ekderivatives} follows from Lemma~\ref{lem:4zetatheta}.

The identity \eqref{eq:8Esymmetry} is due to the symmetry \eqref{eq:8-basicfunctionsrelat}. The estimate \eqref{eq:9Esplit} follows from \eqref{eq:8Ekderivatives}.

To verify \eqref{eq:8Efirstder} we invoke the following Feynman formula, well-known in the general perturbation theory of Hermitian matrices. Let $H(\theta) = (H(m,n;\theta)_{m,n \in \La}$, $|\La| < \infty$ be a real analytic Hermitian matrix-function of $\theta \in (\theta_1,\theta_2)$. Let $E_m(\theta)$, $\psi_m(n;\theta)$, $m,n \in \La$, be a real analytic parametrization of the eigenvalues and normalized eigenvectors of $H(\theta)$, respectively. Due to Rellich's Theorem, such a parametrization always exists. Then,
\begin{equation}\label{eq:8Feynman}
\partial_\theta E_m = \sum_{n \in \La}| \psi_m(n;\theta)|^2 \partial_\theta H(n,n;\theta).
\end{equation}
Thus,
\begin{equation}\label{eq:8Feynman1}
\partial_\theta E^{(s,\pm)}(0, \La^{(s,\mathbf{1})}_k(0); \ve, k_{n_0} + \theta) = \| \varphi^{(s,\pm)} \|^{-2} \sum_{n \in \La} |\varphi^{(s,\pm)}(n;\ve,\theta)|^2 \partial_\theta v(n,k_{n_0} + \theta),
\end{equation}
where $\varphi^{(s,\pm)}(n;\ve,\theta)$ stands for the eigenvector corresponding to $E^{(s,\pm)}(0, \La^{(s,\mathbf{1})}_k(0); \ve, k_{n_0} + \theta)$. Due to \eqref{eq:5evdecaY} from part $(7)$ of Proposition~\ref{rem:con1smalldenomnn}, one has $\| \varphi^{(s,\pm)} \| \ge 1$ and
\begin{equation}\label{eq:8evdecaY}
|\vp^{(s,\pm)}(n; \ve,\theta)| \le |\ve|^{1/3} [\exp(-\frac{7\kappa_0}{8} |n - 0|) + \exp(-\frac{7\kappa_0}{8} |n - n_0|)], \quad n \notin \{0,n_0\}.
\end{equation}
Note that $|\partial_\theta v(n,k_{n_0} + \theta)| = 2 \lambda^{-1} |k_{n_0} + \theta + n\omega| = 2 \lambda^{-1} |-k_{n_0} + \theta + (n - n_0)\omega| \le 1 + |n - n_0|$. Similarly, $|\partial_\theta v(n,k_{n_0} + \theta)| \le 1 + |n|$. Combining \eqref{eq:8Feynman1} and \eqref{eq:8evdecaY} with these estimates and taking into account that $|\ve| < \ve_0$, one obtains \eqref{eq:8Efirstder}.

The estimate \eqref{eq:9Ekk1EG} follows from Corollary~\ref{cor:6.twolambdas1}.

Assume $(\delta^\esone)^{7/8} < |k - k_{n_0}| \le (\delta^\esone)^{3/4}$. The proof of $H_{\La^{(s,\mathbf{1})}_k(0), \varepsilon, k} \in \cN^{(s)} \bigl( 0, \La^{(s,\mathbf{1})}_k(0); \delta_0 \bigr)$ goes the same way as the proof of part $(3)$ of Lemma~\ref{lem:8A.3}; see also part $(0)$ of Remark~\ref{rem:7.kawayfromzero}. To finish $(5)$, consider for instance the case $k_{n_0} + (\delta^\esone)^{7/8} < k \le k_{n_0} + (\delta^\esone)^{3/4}$. Recall that due to \eqref{eq:5specHEE} from Proposition~\ref{prop:8-5n}, one has
\begin{equation} \label{eq:8specHEE}
\begin{split}
\spec H_{\La^{(s,\mathbf{1})}_k(0), \ve, k} \cap \{ |E - E^\esone(0, \La^\esone_k(0); \ve, k)| < 8 (\delta^\esone_0)^{1/4} \} \\
= \{ E^{(s,+)}(0, \La^{(s,\mathbf{1})}_k(0); \ve, k), E^{(s, -)}(0, \La^{(s,\mathbf{1})}_k(0); \ve, k) \}, \\
E^{(s,+)}(0, \La^{(s,\mathbf{1})}_k(0); 0, k) = v(0,k), \quad , \quad E^{(s,-)}(0, \La^{(s,\mathbf{1})}_k(0); 0, k) = v(n_0,k).
\end{split}
\end{equation}
On the other hand, $E^{(s)}(0, \La^{(s,\mathbf{1})}_k(0); \ve, k)$ is the only eigenvalue of $H_{\La^{(s,\mathbf{1})}_k(0), \ve, k}$, which is an analytic function of $\ve$ and obeys $E^{(s)}(0, \La^{(s,\mathbf{1})}_k(0); 0, k) = v(0,k)$. Hence $E^{(s,+)}(0, \La^{(s,\mathbf{1})}_k(0); \ve, k) = E^{(s)}(0, \La^{(s,\mathbf{1})}_k(0); \ve, k)$. This finishes the proof of $(5)$ in this case. The proof for the second possible case is completely similar.

The arguments for part $(6)$ are completely similar to the arguments for part $(7)$ of Proposition~\ref{prop:A.3} and we skip them.
\end{proof}

\begin{remark}\label{rem:8.komega}
Assume that  $k = k_{n_0}$. Assume for instance $k_{n_0} > 0$. Proposition~\ref{prop:8.1} says that for any $k_{n_0} < k' \le k_{n_0} + (\delta^\esone)^{3/4}$, we have $H_{\La^{(s,\mathbf{1})}_k(0), \ve, k'} \in OPR^{(s)} \bigl( 0, n_0, \La^{(s,\mathbf{1})}_k(0); \delta_0, \tau^\zero \bigr)$, and  for any $k_{n_0} - (\delta^\esone)^{3/4} \le k' < k_{n_0}$, we have $H_{\La^{(s,\mathbf{1})}_k(0), \ve, k'} \in OPR^{(s)} \bigl( n_0, 0, \La^{(s,\mathbf{1})}_k(0); \delta_0, \tau^\zero \bigr)$. It does not say anything about $k' = k_{n_0}$. The only reason for that is that the expression on the left-hand side of \eqref{eq:8hHinvestimatestatement1kvard2} vanishes for $\theta = 0$. In fact, due to the symmetry \eqref{eq:8-basicfunctionsrelat}, one has
\begin{equation}\label{eq:8hHinvestimatestatement1kvard20}
v(m^+_0, k_{n_0}) + Q^{(s)}(m_0^+, \La; \ve, k_{n_0}, E) = v(m^-_0, k_{n_0}) + Q^{(s)}(m_0^-, \La; \ve, k_{n_0}, E).
\end{equation}
In Proposition~\ref{prop:8.10} we analyze the case $k' = k_{n_0}$ via the limit $k' \rightarrow k_{n_0}$ with $k' \neq k_{n_0}$.
\end{remark}

\begin{prop}\label{prop:8.10}
Let $\ve \in (-\ve_{s},\ve_{s})$.
\begin{itemize}

\item[(1)] The limits
\begin{equation}\label{eq:8kk1comp1lim}
E^{(s,\pm)}(0, \La^{(s)}_{k_{n_0}}(0); \ve, k_{n_0}):=\lim_{k_1\rightarrow k_{n_0}} E^{(s,\pm)}(0, \La^{(s)}_{k_{n_0}}(0); \ve, k_1)
\end{equation}
exist. Moreover,
\begin{equation} \label{eq:8specHEEAAA}
\begin{split}
\spec H_{\La^{(s)}_{k_{n_0}}(0), \ve, k_{n_0}} \cap \{ E : |E - E^{(s-1)}(0,\La^{(s-1)}_{k_{n_0}}(0); \ve, k_{n_0})| < 8 (\delta^{(s-1)}_0)^{1/4} \} \\
= \{ E^{(s,+)}(0, \La^{(s)}_{k_{n_0}}(0); \ve, k_{n_0}), E^{(s,-)}(0, \La^{(s)}_{k_{n_0}}(0); \ve, k_{n_0}) \}.
\end{split}
\end{equation}
Finally, $E^{(s,+)}(0, \La^{(s)}_{k'}(0); \ve, k_{n_0}) \ge E^{(s,-)}(0, \La^{(s)}_{k_{n_0}}(0); \ve, k_{n_0})$.

\item[(2)] $E = E^{(s,\pm)}(0, \La^{(s)}_{k_{n_0}}(0); \ve, k_{n_0})$ obeys the following equation,
\begin{equation} \label{eq:8Eequation0}
E - v(0, k_{n_0}) - Q^\es(0,\La^{(s)}_{k_{n_0}}(0); \ve, E) \mp \big| G^\es(0,n_0,\La^{(s)}_{k_{n_0}}(0); \ve, E) \big| = 0,
\end{equation}
where
\begin{equation} \label{eq:8-basicfunctions0}
\begin{split}
Q^\es(0,\La^{(s)}_{k_{n_0}}(0); \ve, E) := Q^{(s)}(m_0^+, \La; \ve, k_{n_0} , E), \\
G^\es(0,n_0,\La^{(s)}_{k_{n_0}}(0); \ve, E) := G^\es(m^+_0, m^-_0, \La; \ve, k_{n_0}, E);
\end{split}
\end{equation}
see \eqref{eq:8-10acbasicfunctions}.

\end{itemize}
\end{prop}

\begin{proof}
We will consider the case $s \ge 2$. For $s = 1$, the argument is completely similar. Let, for instance, $0 < k_{n_0} < k' \le k_{n_0} + (\delta^\esone)^{3/4}$. By Proposition~\ref{prop:8.1}, $H_{\La^{(s)}_{k_{n_0}}(0), \ve, k'} \in OPR^{(s)} \bigl( 0, n_0, \La^{(s)}_{k_{n_0}}(0); \delta_0, \tau^\zero \bigr)$. Due to part $(3)$ of Proposition~\ref{prop:8-5n}, one has
\begin{equation} \label{eq:8specHEEAAA1}
\begin{split}
\spec H_{\La^{(s)}_{k_{n_0}}(0), \ve, k'} \cap \{ E : |E - E^{(s-1)}(0,\La^{(s-1)}_{k_{n_0}}(0); \ve, k')| < 8 (\delta^{(s-1)}_0)^{1/4} \} \\
= \{ E^{(s,+)}(0, \La^{(s)}_{k_{n_0}}(0); \ve, k'), E^{(s,-)}(0, \La^{(s)}_{k_{n_0}}(0); \ve, k') \}.
\end{split}
\end{equation}
Due to Remark~\ref{rem:7.kawayfromzero}, $E^{(s-1)}(0,\La^{(s-1)}_{k_{n_0}}(0); \ve, k')$ is a $C^2$-smooth function of $k'$, $|k'- k_{n_0}| < 2 \delta^{(s-2)}_0$. Clearly, $H_{\La^{(s)}_{k_{n_0}}(0), \ve, k'}$ is a $C^2$-smooth matrix-function of $k'$. Finally, recall that
$E^{(s,+)}(0, \La^{(s)}_{k_{n_0}}; \ve, k')>E^{(s,-)}(0, \La^{(s)}_{k_{n_0}}(0); \ve, k')$. Combining all that, one concludes that part $(1)$ is valid
$($ of course, only continuity of the functions involved matters here $)$.

To prove part $(2)$, recall that $E^{(s,\pm)}(0, \La^{(s)}_{k_{n_0}}(0); \ve,k')$ are the only two solutions of the equation \eqref{eq:8Eequation} with $k' = k_{n_0} + \theta$. Recall also that the functions $Q^{(s)}(m_0^\pm, \La; \ve, k_{n_0} + \theta, E)$, $G^\es(m^\pm_0, m^\mp_0, \La; \ve, k_{n_0} + \theta, E)$ defined in \eqref{eq:8-10acbasicfunctions} are $C^2$-smooth in the domain $|\theta| < (\delta_0^\esone)^{3/4}$, $\big| E - E^{(s-1)}(m^+_0, \La^{(s-1)}_k(m_0^+); \ve, k_{n_0} + \theta) \big| < 2 \delta^{(s-1)}_0$.
Taking also into account \eqref{eq:8hHinvestimatestatement1kvard20}, one concludes that $E = E^{(s,\pm)}(0, \La^{(s)}_{k_{n_0}}(0); \ve, k_{n_0})$ obeys the following equation,
\begin{equation} \label{eq:8Eequation01}
(E - v(0, k_{n_0}) - Q^\es(0,\La^{(s)}_{k_{n_0}}(0); \ve, E))^2 - \big| G^\es(0,n_0,\La^{(s)}_{k_{n_0}}(0); \ve, E) \big|^2 = 0.
\end{equation}
Recall now that due to part $(2)$ of Proposition~\ref{prop:8-5n}, one has with $k' = k_{n_0} + \theta$,
\begin{equation} \label{eq:8-13acnq0quasi}
\begin{split}
E^{(s,\pm)}(0, \La^{(s)}_{k_{n_0}}(0); \ve, k')\ge v(n_0, k_{n_0}) + Q^\es(n_0,\La^{(s)}_{k_{n_0}}(0); \ve,\theta, E^{(s,\pm)}(0, \La^{(s)}_{k_{n_0}}(0); \ve, k')) \\
+ | G^\es(0,n_0,\La^{(s)}_{k_{n_0}}(0); \ve, E^{(s,\pm)}(0, \La^{(s)}_{k_{n_0}}(0); \ve, k')) \big|.
\end{split}
\end{equation}
Combining \eqref{eq:8Eequation01} with \eqref{eq:8-13acnq0quasi}, one concludes that $E = E^{(s,+)}(0, \La^{(s)}_{k_{n_0}}(0); \ve, k_{n_0})$ obeys \eqref{eq:8Eequation0}. The argument for $E = E^{(s,-)}(0, \La^{(s)}_{k_{n_0}}(0); \ve, k_{n_0})$ is similar.
\end{proof}

\begin{remark}\label{rem:8.6sharperssoneestimate}
For $|k - k_{n_0}| \rightarrow 0$, we need a stronger version of the estimate \eqref{eq:8kk1comp1EpEm} in Proposition~\ref{prop:8.1}. For that we invoke Remark~\ref{rem:7.sharperssoneestimate} from Section~\ref{sec.5}.
\end{remark}

\begin{corollary}\label{cor:8.sharperesesone}
Using the notation from Proposition~\ref{prop:8.1}, the following estimate holds,
\begin{equation}\label{eq:8kk1comp1sharperrr}
|E^{(s,+)}(0, \La^{(s,\mathbf{1})}_{k_{n_0}}(0); \ve, k_1) -E^{(s,-)}(0, \La^{(s,\mathbf{1})}_{k_{n_0}}(0); \ve, k_1)| \le 2 |\ve| \exp(-\frac{\kappa_0}{2} |n_0|),
\end{equation}
provided $|k_1 - k_{n_0}|$ is small enough. In particular, using the notation from Proposition~\ref{prop:8.10}, one has
\begin{equation}\label{eq:8kk1comp1sharperrrlim}
|E^{(s,+)}(0, \La^{(s,\mathbf{1})}_{k_{n_0}}(0); \ve, k_{n_0}) - E^{(s,-)}(0, \La^{(s,\mathbf{1})}_{k_{n_0}}(0); \ve,k_{n_0} )| \le 2 |\ve| \exp(-\frac{\kappa_0}{2} |n_0|).
\end{equation}
\end{corollary}

\begin{proof}
We consider the case $s \ge 2$. The case $s = 1$ is completely similar. Using the subsets $\La^{(s')}_{k_{n_0}}(m)$ with $s' < s-1$, one can define a subset $\La'(0)$ so that the following conditions hold: $(i)$ $H_{\La'(0), \varepsilon, k} \in \cN^{(s-1)} \bigl( 0, \La'(0); \delta_0 \bigr)$ if $|k - k_{n_0}| < (\delta^\esone)^{3/4}$, $(ii)$ $\La'(0) = -\La'(0)$, $(iii)$ $\La'(0) \supset B(R)$, where $|n_0|/8 < R < |n_0|/4$. Set $\La'(n_0) = n_0 + \La'(0)$. Then, $H_{\La'(n_0), \varepsilon, k} \in \cN^{(s-1)} \bigl( n_0, \La'(n_0); \delta_0 \bigr)$ if $|k - k_{n_0}| < (\delta^\esone)^{3/4}$. Furthermore, let $E^{(s-1)}(0, \La'(0); \ve, k)$, $E^{(s-1)}(n_0, \La'(n_0); \ve, k)$ be the corresponding eigenvalues. Due to Lemma~\ref{lem:basicshiftprop}, one has  $E^{(s-1)}(0, \La'(0); \ve, k_{n_0} + \theta) = E^{(s-1)}(n_0, \La'(n_0); \ve, k_{n_0} - \theta)$. In particular,
\begin{equation} \label{eq:8-3AAAAANEW}
\big| E^{(s-1)}(0, \La'(0); \ve, k_{n_0} + \theta) - E^{(s-1)}(n_0, \La'(n_0); \ve, k_{n_0} + \theta) \big| \le \exp(-R),
\end{equation}
provided $\theta > 0$ is small enough. Thus both conditions mentioned in Remark~\ref{rem:7.sharperssoneestimate} hold. This implies the claim.
\end{proof}

\section{Matrices with Ordered Pairs of Resonances Associated with $1$--Dimensional Quasi-Periodic Schr\"odinger Equations: General Case}\label{sec.9}

Let us start with the following

\begin{lemma}\label{lem:9A.2}
Assume that $q \ge 1$ and $k \in \IR \setminus \bigcup_{0 < |m'| \le 12 R^{(s+q-1)}, \; m' \neq n_0} (k_{m',s+q-1}^-, k_{m',s+q-1}^+)$. Then,

$(0)$ $|k| > (\delta^\esone_0)^{1/16}/4$.

$(1)$ Let $\delta < \delta^\esone_0$. If $|v(m,k) - v(0,k)| < \delta$, then either $(\mathfrak{a})$ $|m \omega| < 2^{12} \delta^{15/16}$, \quad $|2 k + m \omega| > \delta^{1/16}/4$, \quad $|(k + m \omega) - k_{n_0}| < 2^{12} \delta^{15/16} + |k - k_{n_0}|$, or $(\mathfrak{b})$  $|2 k + m \omega| <  2^{12} \delta^{15/16}$, \quad $|m \omega| > \delta^{1/16}/4$, \quad $|(k + (m - n_0) \omega) - k_{n_0}| < 2^{12} \delta^{15/16} + 3 |k - k_{n_0}|$. In both cases, $| |k + m \omega| - |k_{n_0}| | < 2^{12} \delta^{15/16} + 3 |k - k_{n_0}|$. Finally, $n_0 \neq \pm 2m$.

$(1)'$ Suppose $|v(m,k) - v(0,k)| < \delta < \delta^\esone_0$. In case $(\mathfrak{a})$, we have
\begin{equation}\label{eq:9vofmmplusn}
\begin{split}
|v(m + n_0, k) - v(0, k)| \le 2^{6}\delta^{15/16} + |k - k_{n_0}|, \\
|(m + n_0) \omega| > (\delta^\esone_0)^{1/16}/2, \quad |v(m - n_0, k) - v(0, k)| > (\delta^\esone_0)^{1/16}/256,
\end{split}
\end{equation}
and we have case $(\mathfrak{b})$ for $m + n_0$. In case $(\mathfrak{b})$, we have
\begin{equation}\label{eq:9vofmmplusnB}
\begin{split}
|v(m - n_0, k) - v(0, k)| \le 2^{6}\delta^{15/16} + |k - k_{n_0}|, \\
|(2 k + (m - n_0) \omega| > (\delta^\esone_0)^{1/16}/2, \quad |v(m + n_0, k) - v(0, k)| > (\delta^\zero_0)^{1/16}/256,
\end{split}
\end{equation}
and we have case $(\mathfrak{a})$ for $m - n_0$.

$(2)$ Assume $|v(m,k) - v(0,k)| < \delta^{(s'-1)}_0$, $1 \le s' \le s-1$. Then, $k + m \omega \in \IR \setminus \bigcup_{0 < |m'| \le 12 R^{(s-1)}} \quad (k_{m',s'}^-, k_{m',s'}^+)$.

$(3)$ Assume $|v(m,k) - v(0,k)| < \delta^{(s'-1)}_0$, $s \le s' \le s+q-1$. Then, in case $(\mathfrak{a})$, one has $|(k + m \omega) - k| < 4(\delta^{(s'-1)}_0)^{15/16}$, $k + m \omega \in \IR \setminus \bigcup_{0 < |m'| \le 12 R^{(s')}, \; m' \neq n_0} (k_{m',s'-1}^-, k_{m',s'-1}^+)$. In case $(\mathfrak{b})$, one has $|(k + m \omega) - (-k)| < 4(\delta^{(s'-1)}_0)^{15/16}$, $k + m \omega \in \IR \setminus \bigcup_{0 < |m'| \le 12 R^{(s')}, \; m' \neq -n_0} (k_{m',s'-1}^-, k_{m',s'-1}^+)$.

$(4)$ If $1 \le s' \le s-1$, $0 < |m_1 - m_2| \le 12 R^{(s')}$, then $\max |v(m_i,k) - v(0,k)| \ge (\delta^{(s'-1)}_0)^{1/2}$. If $s \le s' \le s+q-1$, $|v(m_i,k) - v(0,k)| < (\delta^{(s'-1)}_0)^{1/2}$, $i = 1, 2$ and $0 < |m_1 - m_2| \le  12 R^{(s')}$, then $m_1 - m_2 \in \{n_0, -n_0\}$. Furthermore, if we have case $(\mathfrak{a})$ for $m_1$, then we have case $(\mathfrak{b})$ for $m_2$.

$(4)'$ Assume $|k - k_{n_0}| < 2 \sigma(n_0)$. Then, $k \in \IR \setminus \bigcup_{0 < |m'| \le 36 R^{(s)}, \; m' \neq n_0} (k_{m',s+q-1}^-, k_{m',s+q-1}^+)$. Furthermore, if $|v(m_i,k) - v(0,k)| < (\delta^{(s-1)}_0)^{1/2}$, $i = 1, 2$ and $0 < |m_1 - m_2| \le  36 R^{(s)}$, then $m_1 - m_2 \in \{n_0, -n_0\}$.

$(5)$ Assume $|k - k_{n_0}| > (\delta^\esone)^{7/8}$. If $|v(m_i,k) - v(0,k)| < \delta^{(s'-1)}_0$, $s \le s' \le s+q-1$, $i = 1, 2$, and $m_1 \neq m_2$, then $|m_1 - m_2| > 12 R^{(s')}$.
\end{lemma}

\begin{proof}
$(0)$ Recall that due to \eqref{eq:diphnores}, we have $|k_{n_0}| > (\delta^{(s-1)})^{1/16}/2$. This implies $(0)$.

$(1)$ Assume $|v(m,k) - v(0,k)| < \delta < \delta^{(s-1)}_0$. Due to part $(1)$ of Lemma~\ref{lem:A.2}, one has $\min (|m \omega| ,|2 k + m  \omega|) \le 32\delta^{1/2}$ if $\gamma\le 4$, $\min (|m \omega| ,|2 k + m  \omega|) \le 256\delta$ if $\gamma\ge 4$. Consider first the case $\gamma \le 4$. Then $\lambda \le 1024$. Assume $|m \omega| \le 32 \delta^{1/2}$. Then, using $(0)$, one obtains $|2 k + m \omega| > 2 |k| - |m \omega| > (\delta^{(s-1)}_0)^{1/16}/4$, $|m \omega| = \lambda |v(m,k) - v(0,k)| |2 k + m \omega|^{-1} < 2^{12} \delta (\delta^\zero_0)^{-1/16} =  2^{12}\delta^{15/16}$, as claimed in $(1)$. Furthermore, $|(k + m \omega) - k_{n_0}| < 2^{12} \delta^{15/16} + |k - k_{n_0}|$.  This establishes all inequalities in case $(\mathfrak{a})$. The estimation in case $(\mathfrak{b})$ is completely similar. Note that $(\mathfrak{a})$ and $(\mathfrak{b})$ obviously exclude each other. Finally, assume $n_0 = \pm 2m$. Due to \eqref{eq:diphnores}, one has $|n_0 \omega| > (\delta^\esone_0 )^{1/16} > \delta^{1/2}$. Therefore, we cannot have case $(\mathfrak{a})$. So, we must have case $(\mathfrak{b})$. Then, $|4 k + \iota n_0 \omega)| < \delta^{1/2}$ for some $\iota \in \{-1,1\}$. Recall that $|2 k + n_0 \omega| < 64 (\delta^\esone_0)^{1/6}$. Hence, $|n_0 \omega| < \delta^{1/2} + 64 (\delta^\esone_0)^{1/6} < 65 (\delta^\esone_0)^{1/6}$. This contradiction with \eqref{eq:diphnores} proves that $n_0 = \pm 2m$ is impossible. Consider now the case $\gamma \ge 4$. Assume $|m \omega| < 256 \delta$. Then, $|2 k + m \omega| > 2 |k| - |m \omega| > 3 k/2 > \lambda$, $|m \omega| = \lambda |v(m,k) - v(0,k)| |2 k + m \omega|^{-1} < \delta$. Furthermore, $|(k + m \omega) - k_{n_0}| <\delta + |k - k_{n_0}|$. This establishes all inequalities in case $(\mathfrak{a})$. The estimation in case $(\mathfrak{b})$ is completely similar. The proof of the rest is completely similar to the case $\gamma\le 4$.

$(1)'$ Recall first of all that $\lambda^{-1} |k_{n_0}| \le 1/256$, $\lambda^{-1} |k_{n_0}| > 1/512$ if $|k_{n_0}| \ge 1$, $\lambda = 256$ in case $|k_{n_0}| < 1$. In case $(\mathfrak{a})$, one has
\begin{equation}\nn
\begin{split}
|v(m + n_0, k) - v(0, k)| = 2 \lambda^{-1} |m \omega + n_0 \omega| |k + (m \omega/2) - k_{n_0}| \le (2^{6} \delta^{15/16} + |k - k_{n_0}|), \\
|v(m - n_0, k) - v(0, k)| = 2 \lambda^{-1} |m \omega - n_0 \omega| |k + (m \omega/2) - k_{n_0} +2k_{n_0} | \\
\ge 2 \lambda^{-1} ((\delta^\esone_0)^{1/16} - 2^{12}  \delta^{15/16})(2|k_{n_0}| -  2^{12}\delta^{15/16} - |k - k_{n_0}|) > (\delta^\esone_0)^{1/16}/256,
\end{split}
\end{equation}
as claimed in \eqref{eq:9vofmmplusn}. Furthermore, $|(m + n_0) \omega| \ge |n_0 \omega| - |m \omega| > (\delta^{(s-1)})^{1/16} - \delta^{1/2} > (\delta^{(s-1)})^{1/16}/2$. Clearly, we have case $(\mathfrak{b})$ for $m + n_0$. The verification of \eqref{eq:9vofmmplusn} in case $(\mathfrak{b})$ goes in a similar way.

$(2)$ Note that $k \in \IR \setminus \bigcup_{0 < |m'| \le 12 R^{(s-1)}} (k_{m',s-1}^-, k_{m',s-1}^+)$ since $n_0 > 12 R^\esone$. Therefore, this part follows from part  $(3)$ in Lemma~\ref{lem:A.2} .

$(3)$ Assume $|v(m,k) - v(0,k)| < \delta^{(s'-1)}_0$, $s \le s' \le s+q-1$. In case $(\mathfrak{a})$, one has $|(k + m \omega) - k| = |m \omega| < 4 (\delta^{(s'-1)}_0)^{15/16}$. This implies, in particular, $k + m \omega \in \IR \setminus \bigcup_{0 < |m'| \le 12 R^{(s')}, \; m' \neq n_0} (k_{m',s'-1}^-, k_{m',s'-1}^+)$; see the definition \eqref{eq:7K.1}. In case $(\mathfrak{b})$, one has $|(k + m \omega) - (-k)| = |2 k + m \omega| < 4(\delta^{(s'-1)}_0)^{15/16}$. Note that $-k \in \IR \setminus \bigcup_{0 < |m'| \le 12 R^{(s+q)}, \quad m' \neq -n_0} (k_{m',s'-1}^-, k_{m',s'-1}^+)$. This implies, in particular,  $k + m \omega \in \IR \setminus \bigcup_{0 < |m'| \le 12 R^{(s')}, \quad m' \neq -n_0} (k_{m',s'-1}^-, k_{m',s'-1}^+)$. This finishes part $(3)$.

$(4)$ The proof of the first statement in $(4)$ goes the same way as the proof of part $(4)$ in Lemma~\ref{lem:A.2} since $k \in \IR \setminus \bigcup_{0 < |m'| \le 12 R^{(s-1)}} (k_{m',s'-1}^-, k_{m',s'-1}^+)$. To prove the second statement, assume, for instance, that we have case $(\mathfrak{a})$ for $m_1$ and case $(\mathfrak{b})$ for $m_2$. In this case,
\begin{equation}\label{eq:9sqnonres}
\begin{split}
k \in \left( -\frac{m'\omega}{2} - (\delta^{(s'-1)}_0)^{1/4} , -\frac{m'\omega}{2} + (\delta^{(s'-1)}_0)^{1/4} \right) \subset \\
\left( -\frac{m'\omega}{2} - \frac{\sigma(m')}{2} , -\frac{m'\omega}{2} + \frac{\sigma(m')}{2} \right) \subset ( k_{m',s'-1}^-, k_{m',s'-1}^+),
\end{split}
\end{equation}
where $m' = m_2 - m_1$. Assume $|m'| \le 12 R^{(s')}$. Then the only possibility is $m' = n_0$, that is, $m_2 - m_1 = n_0$. The proof for the rest of the cases is similar.

$(4)'$ Since $|k - k_{n_0}| < 2 \sigma(n_0)$, the first statement in $(4)'$ follows from \eqref{eq:7K.1} combined with \eqref{eq:diphnores}.
$($ If $q > 1$, then $R^{(s+q-1)} > 48 R^{(s)}$ and therefore $k \in \IR \setminus \bigcup_{0 < |m'| \le 48 R^{(s)}, \; m' \neq n_0} (k_{m',s+q-1}^-, k_{m',s+q-1}^+)$ without any additional condition for $k$. $)$ Applying the same arguments as in part $(4)$, one obtains the second statement in $(4)'$.

$(5)$ Assume $|k - k_{n_0}| > (\delta^\esone)^{7/8}$, $|v(m_i,k) - v(0,k)| < (\delta^{(s'-1)}_0)$, $s \le s' \le s+q-1$, $i = 1, 2$ and $m_1 \neq m_2$. Assume first that $s < s' \le s+q-1$. To prove the statement in this case we again repeat the arguments from the proof of part $(4)$ in Lemma~\ref{lem:A.2}. Note that in this case, $(\delta^{(s'-1)}_0)^{1/2} < (\delta^\esone)^{7/8}$. So, the only possibility is $|m_1 - m_2| > 12 R^{(s')}$. This proves the statement for $s' > s$. Consider now the case $s' = s$. If we have case $(\mathfrak{a})$ for both $m_1, m_2$, then $2^{13} (\delta^{(s-1)}_0)^{15/16} > |(m_2 - m_1) \omega|$. Due to \eqref{eq:diphnores}, this implies $|m_2 - m_1| > 48 R^\es$. Similarly, if we have case $(\mathfrak{b})$ for both $m_1, m_2$, then $|m_2 - m_1| > 48 R^\es$. Assume now, for instance, that we have case $(\mathfrak{a})$ for  $m_1$ and case $(\mathfrak{b})$ for $m_2$. Then, $|m_1 \omega| < 2^{12} (\delta^{(s-1)}_0)^{15/16}$, $|2 k + m_2 \omega| < 2^{12} (\delta^{(s-1)}_0)^{15/16}$. Assume also that $|m_2 - m_1|\le 12 R^\es$. Then $m_2 = m_1 + n_0$, as before. But this implies
$$
|2 k + m_2 \omega| = |2 k + m_1 \omega + n_0 \omega| \ge 2 |k - k_{n_0}| - |m_1 \omega| > 2 (\delta^{(s-1)}_0)^{7/8} - 2^{12} (\delta^{(s-1)}_0)^{15/16} > (\delta^{(s-1)}_0)^{7/8}.
$$
This contradiction implies $|m_2 - m_1| > 12 R^\es$, as claimed.
\end{proof}

\begin{remark}\label{rem:9.1abcases}
In this section and later in this work, when we refer to the cases $(\mathfrak{a})$ and $(\mathfrak{b})$, we mean cases $(\mathfrak{a})$ and $(\mathfrak{b})$ of Lemma~\ref{lem:9A.2}.
\end{remark}

In this section we use the same notation as in Section~\ref{sec.8}. We always assume that $k \in \cR^{(s,s)}(\omega,n_0)$. In particular,
\begin{equation}\label{eq:9realPR}
|k - k_{n_0}| \le 2 \sigma(n_0) = 64(\delta^\esone)^{1/6}.
\end{equation}

\begin{defi}\label{defi:8reflectionchoiceq}
For $q \ge 0$, let $\cR^{(s,s+q)}(\omega,n_0)$ be the set of $k \in \IR \setminus \bigcup_{0 < |m'| \le 12 R^{(s+q)}, \; m \neq n_0} (k_{m',s+q-1}^-, k_{m',s+q-1}^+)$, $0 < |k - k_{n_0}| \le 2 \sigma(n_0)$. For $1 \le r \le s-1$, $k' \in \cR^{(s,s)}(\omega,n_0)$, let $\La^{(r)}_{k'}(0)$ be the sets from Proposition~\ref{prop:A.3} $($ see also \eqref{eq:A.1} for the definitions $)$. For $0 < |k' - k_{n_0}| \le (\delta^\esone_0)^{3/4}$, let $\La^{(s,\mathbf{1})}_{k'}(0)$ be the set from Proposition~\ref{prop:8.1}.

Let $q \ge 2$. Assume that the sets $\La^{(s',\mathbf{1})}_{k'}(0)$ are already defined for all $s \le s' \le s+q-1$, provided
\begin{equation}\label{eq:9kintcond}
0 < |k' - k_{n_0}| < (\delta^\esone_0)^{3/4} - 4 \sum_{s-1 \le t \le s'-1} (\delta^{(t)}_0)^{15/16}
\end{equation}
and $k' \in \cR^{(s,s')}(\omega,n_0)$. Assume also that the sets $\La^{(s')}_{k'}(0)$ are already defined for all $1 \le s' \le s+q-1$, provided
\begin{equation}\label{eq:9koutcond}
(\delta^\esone_0)^{7/8} + 4 \sum_{s-1 \le t \le s'} (\delta^{(t)}_0)^{15/16} < |k' - k_{n_0}| < 2 \sigma(n_0)
\end{equation}
and $k' \in \cR^{(s,s')}(\omega,n_0)$. Assume the same for $-n_0$ in the role of $n_0$. Assume also that $\La^{(s')}_{k'}(0) \subset \La^{(s'')}_{k'}(0)$ if $s' < s'' \le s+q-1$, and $\La^{(s',\mathbf{1})}_{k'}(0) \subset \La^{(s'',\mathbf{1})}_{k'}(0)$ if $s \le s' < s'' \le s+q-1$.

$(\mathbb{A})$ Assume that for $k' = k$, we have \eqref{eq:9kintcond} with $s' = s + q$ and $k' \in \cR^{(s,s+q-1)}(\omega,n_0)$. Assume also that $k \notin \frac{\omega}{2}\zv$, so that $|k + m \omega| \neq |k_{n_0}| >0$, provided $n_0 \notin \{2m, -2m\}$.

$(1)$ Let $m$ be such that $|v(m,k) - v(0,k)| < 3 \delta^{(s+q-2)}_0/4$, $|m| < 12 R^{(s+q)}$. Assume also that $v(m,k) > v(0,k_{n_0})$ if $v(0,k) > v(0,k_{n_0})$ (resp., assume that $v(m,k) < v(0,k_{n_0})$ if $v(0,k) < v(0,k_{n_0})$). Then we say that $m \in \cM^{(s+q-1,+)}_{k,s+q-1}$ (resp., $m \in \cM^{(s+q-1,-)}_{k,s+q-1}$). Combining \eqref{eq:9kintcond} with $(1)$, $(3)$ from Lemma~\ref{lem:9A.2} and with the fact that $|k + m \omega| \neq |k_{n_0}| > 0$, unless $n_0 \notin \{2m, -2m\}$, one concludes that $\La^{(s+q-1,\mathbf{1})}_{k+m\omega}(0)$ is well-defined. We set $\La^{(s+q-1)}_{k}(m) = m + \La^{(s+q-1,\mathbf{1})}_{k+m\omega}(0)$.

$(2)$ Given $s \le s' \le s+q-2$, assume that for any $s' < s'' \le s+q-1$, the sets $\cM^{(s'',+)}_{k,s+q-1}$, $\La^{(s'')}_{k,s+q-1}(m'')$ are already defined. Let $m$ be such that $|v(m,k) - v(0,k)| \le (3 \delta^{(s'-1)}_0/4) - \sum_{s' < s'' \le s+q-1} \delta^{(s''-1)}_0$, $|m| < 12 R^{(s+q)}$. Assume also that $m \notin \bigcup_{s' < s'' \le s+q-1} \bigcup_{m'' \in \cM^{(s'',+)}_{k,s+q-1}} \La^{(s'')}_{k}(m'')$. Then we say that $m \in \cM^{(s',+)}_{k,s+q-1}$ if $v(m,k) > v(0,k_{n_0})$, $v(0,k) > v(0,k_{n_0})$, respectively, $m \in \cM^{(s',-)}_{k,s+q-1}$ if $v(m,k) < v(0,k_{n_0})$, $v(0,k) < v(0,k_{n_0})$. We set $\La^{(s')}_{k}(m) = m + \La^{(s',\mathbf{1})}_{k+m\omega}(0)$. As in $(1)$ above, $\La^{(s',\mathbf{1})}_{k+m\omega}(0)$ is well-defined.

$(3)$ Given $s' < s$, assume that for any $s' < s'' \le s+q-1$, the sets $\cM^{(s'')}_{k,s+q-1}$, $\La^{(s'')}_{k,s+q-1}(m'')$ are already defined. Let $m$ be such that $|v(m,k) - v(0,k)| \le (3 \delta^{(s'-1)}_0/4) - \sum_{s' < s'' \le s+q-1} \delta^{(s''-1)}_0$, $|m| < 12 R^{(s+q)}$. Assume also that $m \notin \bigcup_{s' < s'' \le s+q-1} \bigcup_{m'' \in \cM^{(s'')}_{k,s+q-1}} \La^{(s'')}_{k}(m'')$. Then we say that $m \in \cM^{(s')}_{k,s+q-1}$ and we set $\La^{(s')}_{k}(m) = m + \La^{(s')}_{k+m\omega}(0)$.

$(4)$ For $s \le s'\le s+q-1$, we enumerate the points of $\cM^{(s',\pm)}_{k,s+q-1}$ as $m^\pm_j$, $j \in J^{(s')}$. Set
\begin{equation}\label{def:9mpm}
\begin{split}
m_j^- = \begin{cases} m_j^+ + n_0 & \text{if $v(0,k) > v(0,k_{n_0})$, $\sgn(k+m^+_j\omega) = - \sgn (n_0 \omega)$}, \\
m^+_j - n_0 & \text{if $v(0,k) > v(0,k_{n_0})$, $\sgn (k + m_j^+ \omega) = \sgn (n_0 \omega)$}, \end{cases} \\
m_j^+ = \begin{cases} m_j^- + n_0 & \text{if $v(0,k) < v(0,k_{n_0})$, $\sgn (k+m_j^-\omega) = - \sgn (n_0 \omega)$}, \\
m_j^- - n_0 & \text{if $v(0,k) < v(0,k_{n_0})$, $\sgn (k + m_j^- \omega) = \sgn (n_0 \omega)$}, \end{cases} \\
\La^{(s')}_{k}(m_j^-) = \La^{(s')}_{k}(m_j^+), \quad \cM^{(s',\mathbf{1})}_{k,s+q-1} := \cM^{(s',+)}_{k,s+q-1} \cup \cM^{(s',-)}_{k,s+q-1}.
\end{split}
\end{equation}

$(\mathbb{B})$ Let $k$ be as in \eqref{eq:9koutcond} with $s' = s+q-1$. Then we define $\cM^{(s')}_{k,s+q-1}$, $\La^{(s')}_{k}(m'')$ just as in \eqref{eq:A.1}.
\end{defi}

\begin{remark}\label{rem:9.1inout}
$(1)$ In the last definition and for the rest of this work, we do not use the notation $\La^{(s+q-1,\mathbf{1})}_{k}(m)$ for any $m$ except $m = 0$. This simplifies the statements in what follows. For $m = 0$, we use the notation $\La^{(s+q-1,\mathbf{i})}_{k}(m)$, which includes both possibilities. We use also the notation $\cM^{(s',\mathbf{i})}_{k,s+q-1}$, which includes all possibilities. None of that will produce ambiguity since in the proofs, we always specify the cases the notation applies to.

$(2)$ If $r \le s-1$, then $B(2 R^\ar) \subset \La^{(r)}_k(0) \subset B(3 R^\ar)$. This property holds due to \eqref{eq:A.1}. For $r = s$, we have
\begin{equation}\label{eq:9lambdasetssize}
\begin{split}
B(2 R^\ar) \cup (n_0 + B(2 R^\ar)) \subset \La^{(r,\mathbf{1})}_k(0) \subset B(3 R^\ar) \cup (n_0 + B(3 R^\ar)) \quad \text {if $\sgn (k) = - \sgn (n_0 \omega)$}, \\
B(2 R^\ar) \cup (-n_0 + B(2 R^\ar)) \subset \La^{(r,\mathbf{1})}_k(0) \subset B(3 R^\ar) \cup (-n_0 + B(3 R^\ar)) \quad \text {if $\sgn (k) = \sgn (n_0 \omega)$}, \\
B(2 R^\ar) \subset \La^{(r)}_k(0) \subset B(3 R^\ar) \quad \text{if part $(\mathbb{B})$ in Definition~\ref{defi:8reflectionchoiceq} applies.}
\end{split}
\end{equation}

For $r = s$, the first two relations in \eqref{eq:9lambdasetssize}, addressing the case $|k - k_{n_0}| \le (\delta^\esone)^{3/4}$, are due to part~$(5)$ of Lemma~\ref{lem:8setLambdas}. The third one, addressing the case  $|k - k_{n_0}| > (\delta^\esone)^{7/8}$, is due to part~$(c)$ of Remark~\ref{rem:7.1oinout}. For $r > s$, we will establish \eqref{eq:9lambdasetssize} inductively in the corresponding domains of $k$. Note that \eqref{eq:9lambdasetssize} implies in particular $\La^{(r,\mathbf{i})}_k(0) \subset \La^{(r',\mathbf{i})}_k(0)$, $\La^{(r)}_k(0) \subset \La^{(r')}_k(0)$ for any $s \le r < r'$.

$(3)$  $\cM^{(s_1,+)}_{ k,s+q-1} \cap \cM^{(s_2,+)}_{ k,s+q-1} = \emptyset$, $\cM^{(s_1)}_{ k,s+q-1} \cap \cM^{(s_2+)}_{ k,s+q-1} = \emptyset$, $\cM^{(s_1)}_{ k,s+q-1} \cap \cM^{(s_2)}_{ k,s+q-1} = \emptyset$ if $v(0,k) > v(0,k_{n_0})$, $s_1 < s_2$. Respectively, $\cM^{(s_1,-)}_{ k,s+q-1} \cap \cM^{(s_2,-)}_{ k,s+q-1} = \emptyset$, $\cM^{(s_1)}_{ k,s+q-1} \cap \cM^{(s_2-)}_{k,s+q-1} = \emptyset$, $\cM^{(s_1)}_{ k,s+q-1} \cap \cM^{(s_2)}_{ k,s+q-1} = \emptyset$ if $v(0,k) < v(0,k_{n_0})$, $s_1 < s_2$.

$(4)$ $0 \in \cM^{(s+q-1,+)}_{k,s+q-1}$ if $v(0,k) > v(0,k_{n_0})$; $0 \in \cM^{(s+q-1,-)}_{k,s+q-1}$ if $v(0,k) < v(0,k_{n_0})$. For notational convenience, we assume that $0 \in J^{(s+q-1)}$.

$(5)$ Let $s \le s'$, $v(m,k) > v(0,k_{n_0})$, $v(0,k) > v(0,k_{n_0})$, $m \in \cM^{(s',+)}_{k,s+q-1}$, or $v(m,k) < v(0,k_{n_0})$, $v(0,k) < v(0,k_{n_0})$, $m \in \cM^{(s',-)}_{k,s+q-1}$, or let $s' < s$, $m \in \cM^{(s')}_{k,s+q-1}$. Then,
\begin{equation}\label{eq:9lambdasetscC}
(3 \delta^{(s')}_0/4) - \sum_{s'+1 < s'' \le s+q-1} \delta^{(s''-1)}_0 < |v(m,k) - v(0,k)| \le (3 \delta^{(s'-1)}_0/4) - \sum_{s' < s'' \le s+q-1} \delta^{(s''-1)}_0.
\end{equation}

$(6)$ For any $|n| < 5 R^{(s+q)}$ such that $n \notin \bigl( \bigcup_{1 \le s' \le s+q} \bigcup_{m \in \cM^{(s',i)}_{k,s+q-1}} \La^{(s')}_k(m)\bigr)$, we have $|v(n,k) - v(0,k)| \ge (\delta_0)^4$.
\end{remark}

\begin{remark}\label{rem:9.1abcasesQ1}
In Lemma~\ref{lem:9A.2} and Definition~\ref{defi:8reflectionchoiceq} we assume that $k$ belongs to the complement of $\bigcup_{0 < |m'| \le 12 R^{(s+q-1)}, \; m' \neq n_0} (k_{m',s+q-1}^-, k_{m',s+q-1}^+)$ instead of the complement of $\bigcup_{0 < |m'| \le 12 R^{(s+q)}, \; m' \neq n_0} (k_{m',s+q-1}^-, k_{m',s+q-1}^+)$ because of the further development in Section~\ref{sec.10}. The latter condition is needed only in Proposition~\ref{prop:9.1}.
\end{remark}

\begin{lemma}\label{lem:9mdeltaTmCor}
Assume that \eqref{eq:9lambdasetssize} holds for all $r \ge s$.

$(1)$  $m \in \La^{(s')}_{k}(m)$ for any $s'$ and any $m \in \cM^{(s')}_{k,s+q-1}$.

$(2)$ Let $k$ be as in \eqref{eq:9kintcond} with $s' = s+q$. Then $\cM^{(s_1)}_{k,s+q-1} \cap \cM^{(s_2)}_{k,s+q-1} = \emptyset$ for any $1 \le s_1 < s_2 < s$, $\cM^{(s_1)}_{k,s+q-1} \cap \cM^{(s_2,\mathbf{1})}_{k,s+q-1} = \emptyset$ for any $1 \le s_1 < s \le s_2$, $\cM^{(s_1,\mathbf{1})}_{k,s+q-1} \cap \cM^{(s_2,\mathbf{1})}_{k,s+q-1} = \emptyset$ for any $s \le s_1 < s_2$.
\end{lemma}

\begin{proof}
$(1)$. The statement follows from \eqref{eq:9lambdasetssize} and the definition \eqref{def:9mpm}.

$(2)$ The statement follows from the condition $m \notin \bigcup_{s' < s'' \le s+q-1} \bigcup_{m'' \in \cM^{(s'')}_{k,s+q-1}} \La^{(s'')}_{k}(m'')$ in Definition~\ref{defi:8reflectionchoiceq} and part $(1)$ of the current lemma.
\end{proof}

\begin{lemma}\label{lem:9A.3}
Assume that \eqref{eq:9lambdasetssize} holds for all $r \ge s$.

$(1)$ Assume that for some $m_1$, we have
\begin{equation}\label{eq:9afrprelim}
|v(m_1,k) - v(0,k)| \le (3 \delta^{(s'-1)}_0/4) - \sum_{s' < s'' \le s+q-1} \delta^{(s''-1)}_0.
\end{equation}
Let  $m_2 \in \cM^{(s')}_{k,s+q-1}$. Then, either $m_1 \in \La^{(s')}_k(m_2)) \cap \cM^{(s')}_{k,s^\one-1}$ or
\begin{equation}\label{eq:9m1m2cases}
|m_1 - m_2| \ge \begin{cases} 36 R^{(s)} & \text{if $s' = s$ and $\La^{(s')}_k(m_2)$ is defined as in $(\mathbb{A})$ of Definition~\ref{defi:8reflectionchoiceq}}, \\ 12 R^{(s')} & \text{otherwise}. \end{cases}
\end{equation}

$(1)'$ For any $s'$ and any $m_1$, $m_2$, either $\La^{(s')}_{k,s+q-1}(m_1) = \La^{(s')}_{k,s+q-1}(m_2)$, or  $\dist(\La^{(s')}_{k,s+q-1}(m_1), \La^{(s')}_{k,s+q-1}(m_2)) > 5 R^{(s')}$.

$(2)$ Assume that for some $1 \le s_1 < s_2 \le s+q-1$, $m_1, m_2$, we have
\begin{equation}\label{eq:9afr}
|v(m_i,k) - v(0,k)| \le (3 \delta^{(s_i-1)}_0/4) - \sum_{s_i < s'' \le s+q-1} \delta^{(s''-1)}_0, \quad i = 1, 2.
\end{equation}
Then,
\begin{equation}\label{eq:11centeratm'state0}
|v(m_1 - m_2, k + m_2 \omega) - v(0, k + m_2 \omega)| < 3 \delta^{(s_1-1)}_0/4 - \sum_{s_1 < s'' \le s_2 - 1} \delta^{(s''-1)}_0.
\end{equation}

$(3)$ Assume that for any $s \le s' \le s+q-1$, the following condition holds:

$(\mathfrak{S}_{s'})$ Let $k \in \IR \setminus \bigcup_{0 < |m'| \le 12 R^{(s'-1)}, \; m' \neq n_0} (k_{m',s'-1}^-, k_{m',s'-1}^+)$, $m_1 \in \cM^{(s_1)}_{k,s'-1}$, $s_1 \le s'-1$ , $|m_1| < 12 R^{(s')}$. Then, either $\La^{(s_1)}_{k}(m_1) \subset \La^{(s')}_k(0)$ or $\La^{(s_1)}_{k}(m_1) \cap \La^{(s')}_k(0) = \emptyset$.

Then, the following statement holds.

Assume that for some $s \le s_1 \le s+q-1$, $|m_1| < 5 R^{(s+q)}$, we have
\begin{equation}\label{eq:9afrYYZ}
|v(m_1,k) - v(0,k)| \le (3 \delta^{(s_1-1)}_0/4) - \sum_{s_1 < s'' \le s+q-1} \delta^{(s''-1)}_0.
\end{equation}
If
\begin{equation}\label{eq:9kintcondCOND}
0 < |k - k_{n_0}| < (\delta^\esone_0)^{3/4} - \sum_{s-1 \le t \le s+q-1} (\delta^{(t)}_0)^{15/16},
\end{equation}
then assume also that $v(m_1,k) > v(0,k_{n_0})$ if $v(0,k) > v(0,k_{n_0})$, and respectively, $v(m_1,k) < v(0,k_{n_0})$ if $v(0,k) < v(0,k_{n_0})$. Then,

    either $(\alpha)$ $m_1 \in \La^{(s_2)}_k(m_2)$ for some $s_1 < s_2 \le s+q-1$, $m_2 \in \cM^{(s_2)}_{k,s+q-1}$,

    or $(\beta)$ $m_1 \in \cM^{(s_1)}_{k,s+q-1}$ and $\La^{(s_1)}_k(m_1)) \cap \La^{(s_2)}_k(m_2) = \emptyset$ for any $m_2 \in \cM^{(s_2)}_{k,s-1}$ $m_2 \neq m_1$ with $s_1 \le s_2 \le s-1$.
\end{lemma}

\begin{proof}
$(1)$ Assume $s' \ge s$. Consider the case when \eqref{eq:9kintcondCOND} holds, so that part $(\mathbb{A})$ in Definition~\ref{defi:8reflectionchoiceq} applies. Assume for instance $v(0,k) > v(0,k_{n_0})$. Due to Definition~\ref{defi:8reflectionchoiceq}, one can assume that $\La^{(s')}_{k}(m_2) = m_2 + \La^{(s',\mathbf{1})}_{k + m_2 \omega}(0)$, with $|k + m_2 \omega| > |n_0 \omega|/2$,
$|v(m_2,k) - v(0,k)|) \le 3 \delta^{(s'-1)}_0/4$. Due to part $(4)$ of Lemma~\ref{lem:9A.2}, either $|m_1 - m_2| > C(s') R^{(s')}$, or $m_1 - m_2 \in \{0,n_0,-n_0\}$, where $C(s') = 36$ if $s' = s$ and $C(s') = 12$ otherwise. Assume $m_1 - m_2 \in \{0, n_0, -n_0\}$. If $m_1 = m_2$, then we are done. Assume  $m_1 - m_2 \in \{n_0, -n_0\}$. Note that due to part $(1)$ in Lemma~\ref{lem:9A.2}, one has $||k + m_1 \omega| - |k_{n_0}|| < 2^{12} (\delta^{(s'-1)}_0)^{15/16} + 3 |k - k_{n_0}| < (\delta^{(s'-1)}_0)^{1/2}$. Consider the case $\sgn(k + m_2 \omega) = - \sgn (n_0 \omega)$. Note that in this case, $|k + m_2 \omega - n_0 \omega| > |n_0 \omega|$. This implies $m_1 \neq m_2 - n_0$, that is, $m_1 = m_2 + n_0$. Due to \eqref{def:9mpm}, \eqref{eq:9lambdasetssize}, $m_2 + n_0 \omega \in \La^{(s')}_{k}(m_2) \cap \cM^{(s',\mathbf{1})}_{k,s+q-1}$. The proof in case when \eqref{eq:9kintcondCOND} holds and $\sgn(k + m_2 \omega) = \sgn (n_0 \omega)$ is similar. This finishes the case when part $(\mathbb{A})$ applies and $s' \ge s$. The verification for the rest of the cases follows straight from parts $(4)$, $(5)$ of Lemma~\ref{lem:9A.2}.

$(1)'$ This part follows from part $(1)$ of the current lemma combined with \eqref{eq:9lambdasetssize} and \eqref{def:9mpm}.

$(2)$ The proof goes word for word as the proof of $(2)$ in Lemma~\ref{lem:A.3}.

$(3)$ With part $(1)$ of the current lemma in mind, the proof goes word for word as the proof of $(3)$ in Lemma~\ref{lem:A.3}.
\end{proof}

\begin{lemma}\label{lem:9mdeltaTm}
Let $0 < |k - k_{n_0}| < (\delta^{(s+q-1)}_0)^{1/16}$.

$(1)$ If $|v(m,k) - v(0,k)| < \delta$ with $\delta \ge (\delta^{(s-1)}_0)^{1/32}$, then $|v(T(m),k) - v(0,k)| < 4 \delta/3$.

Assume also that \eqref{eq:9lambdasetssize} holds for all $s \le r \le s+q-1$.

$(2)$ Let $s' < s$, $m_j \in \cM^{(s')}_{k,s+q-1}$, $j = 1, 2$. Then, either $T(m_1) = m_2$ or $\dist(T(\La_k^{(s')}(m_1)), \La_k^{(s')}(m_2)) > 6 R^{(s')}$. Let $s \le s'$, $m_j \in \cM^{(s')}_{k,s+q-1}$, $j = 1, 2$. Then, either $T(\La_k^{(s')}(m_1)) \cap \La_k^{(s')}(m_2) \neq \emptyset$ or $\dist(T(\La_k^{(s')}(m_1)), \La_k^{(s')}(m_2)) > 5 R^{(s')}$. In the former case, $(\La_k^{(s')}(m_1) \cap \cM^{(s')}_{k,s+q-1}) = \La_k^{(s')}(m_{2}) \cap \cM^{(s')}_{k,s+q-1}$.
\end{lemma}

\begin{proof}
$(1)$ The proof of this part is completely similar to the proof of part $(1)$ of Lemma~\ref{lem:8mdeltaTm}.

$(2)$ It follows from $(1)$ of the current lemma that $|v(T(m_1),k) - v(0,k)|) \le \delta^{(s'-1)}_0$. Applying part $(1)$ of Lemma~\ref{lem:9A.3} to $T(m_1)$ and $m_2$, one obtains the statement.
\end{proof}

\begin{defi}\label{defi:9.LLL}
Assume $0 < |k - k_{n_0}| < (\delta^{(s+q-1)}_0)^{1/16}$. Using the notation from Definition~\ref{defi:8reflectionchoiceq}, assume that for any $s \le s'\le s+q-1$, condition $(\mathfrak{S}_{s'})$ holds. Let $\mathfrak{L}'$ be the collection of all sets $\La(m) := \La^{(s')}_k(m) \cup T(\La^{(s')}_k(m))$, $1 \le s' \le s+q-1$, $m \in \cM^{(s')}_{k,s+q-1}$. We say that $\La(m_1) \approx \La(m_2)$ if $s_1 = s_2$, and $\La(m_1) \cap \La(m_2) \neq \emptyset$. It follows from part $(3)$ of Lemma~\ref{lem:9A.3} and part $(2)$ of Lemma~\ref{lem:9mdeltaTm} that this is indeed an equivalence relation on $\mathfrak{L}'$. Let $\mathfrak{M}$ be the set of equivalence classes. It follows from part $(2)$ of Lemma~\ref{lem:9mdeltaTm} that each class has at most two elements in it. For each $\mathfrak{m} \in \mathfrak{M}$, set $\La(\mathfrak{m}) = \bigcup_{\La(m_1) \in \mathfrak{m}} \La(m_1)$. Set $\mathfrak{L} = \{\La(\mathfrak{m}) : \mathfrak{m} \in \mathfrak{M}\}$. Let $\La(\mathfrak{m}) \in \mathfrak{L}$, $ \La^{(s')}_k(m^+_j) \cup T(\La^{(s')}_k(m^+_j)) \in \mathfrak{m}$. Set $t(\La(\mathfrak{m})) = s'$. This defines an $\mathbb{N}$-valued function on $\mathfrak{L}$. Set also $p_\mathfrak{m} = \bigcup_{\La^{(s')}_k(m)) \in \mathfrak{m}} \La^{(s')}_k(m)) \cap \cM^{(s')}_{k,s^\one-1}$.
\end{defi}

In the next lemma we use Definition~\ref{defi:5.twolambdas6} from Section~\ref{sec.7}.

\begin{lemma}\label{lem:9.lLL}
Using the notation from Definition~\ref{defi:9.LLL}, assume in addition that condition \eqref{eq:9lambdasetssize} holds. Then,

$(1)$ For any $\La(\mathfrak{m}_j) \in \mathfrak{L}$, $j = 1, 2$, such that $t(\La(\mathfrak{m}_1)) = t(\La(\mathfrak{m}_2))$, $\mathfrak{m}_1 \neq \mathfrak{m}_2$, we have $\dist (\La(\mathfrak{m}_1), \La(\mathfrak{m}_2)) > R^{(t(\La(\mathfrak{m}_1)))}$.

$(2)$ For any $\mathfrak{m}$,
\begin{equation}\label{eq:9Lmsets}
\bigcup_{m \in p_\mathfrak{m}} \bigl( (m + B(2 R^{(t(\La(\mathfrak{m})))})) \bigr) \subset\La(\mathfrak{m}) \subset \bigcup_{m \in p_\mathfrak{m}} \bigl( (m + B(3 R^{(t(\La(\mathfrak{m})))})) \bigr).
\end{equation}

$(3)$ If $\mathfrak{m}_1 \neq \mathfrak{m_2}$, then $\La(\mathfrak{m}_1) \neq \La(\mathfrak{m}_2)$.

$(4)$ The pair $(\mathfrak{L},t)$ is a proper subtraction system.

$(5)$ For any $\mathfrak{m}$, $\La(\mathfrak{m}) = T(\La(\mathfrak{m}))$.
\end{lemma}

\begin{proof}
Note first of all that part $(2)$ is just condition \eqref{eq:9lambdasetssize}, which we assume in this lemma. The proof of parts $(1)$, $(3)$, $(4)$, $(5)$ goes word for word as the proof of the parts $(1)$, $(3)$, $(4)$, $(5)$ of Lemma~\ref{lem:8.lLL}. The only detail that has to be mentioned regarding $(3)$ is that $T(n_0) = 0$ and $T(B(R)) = n_0 - B(R) = n_0 + B(R)$ for any $R$.
\end{proof}

Using the notation from Definition~\ref{defi:9.LLL}, assume in addition that condition \eqref{eq:9lambdasetssize} holds. For $\ell = 1, 2, \ldots$, set
\begin{equation} \label{eq:9.twolambdas5}
\begin{split}
\mathfrak{B}(n_0,s+q) := B(3 R^{(s+q)}) \cup (n_0 + B(3 R^{(s+q)})), \\
\mathfrak{B}(n_0,s+q,\ell) := \mathfrak{B}(n_0,s+q,\ell-1) \setminus \Bigl( \bigcup_{\mathfrak{m} \in \mathfrak{M} : \La(\mathfrak{m}) \between \mathfrak{B}(n_0,s+q,\ell-1)} \La(\mathfrak{m}) \Bigr).
\end{split}
\end{equation}

\begin{lemma}\label{lem:9setLambdas}
Using the notation from the definition \eqref{eq:9.twolambdas5}, the following statements hold.

$(1)$ There exists $\ell_0 < 2^{s+q}$ such that $\mathfrak{B}(n_0,s+q,\ell) = \mathfrak{B}(n_0,s+q,\ell+1)$ for any $\ell \ge \ell_0$.

$(2)$ For any $\La \in \mathfrak{L}$, either $\La \subset \mathfrak{B}(n_0,s+q,\ell_0)$ or $\La \subset \Bigl( \IZ^\nu \setminus
\mathfrak{B}(n_0,s+q,\ell_0) \Bigr)$.

$(3)$ Set $\La^{(s+q)}_k(0) = \mathfrak{B}(n_0,s+q,\ell_0)$. Then for any $\La^{(s')}(m)$, either $\La^{(s')}(m) \cap \La^{(s+q)}_k(0) = \emptyset$ or $\La^{(s')}(m) \subset \La^{(s+q)}_k(0)$.

$(4)$ $T(\mathfrak{B}(n_0,s+q,\ell)) = \mathfrak{B}(n_0,s+q,\ell)$ for any $\ell$. In particular, $T(\La^{(s+q)}_k(0)) = \La^{(s+q)}_k(0)$.

$(5)$ For any $\ell \ge 1$,
\begin{equation} \label{eq:9.twolambdas5NEW}
\{ n \in \mathfrak{B}(n_0,s+q,\ell-1)) : \dist (n, \IZ^\nu \setminus \mathfrak{B}(n_0,s,\ell-1)) \ge 3 R^{(s+q-1)} \} \subset
\mathfrak{B}(n_0,s+q,\ell) \subset \mathfrak{B}(n_0,s+q,\ell-1)).
\end{equation}
\end{lemma}

\begin{proof}
The proof of parts $(1)$--$(5)$ goes word for word as the corresponding proof of parts $(1)$--$(5)$ of Lemma~\ref{lem:8setLambdas}.
\end{proof}

\begin{lemma}\label{lem:9conditionLambda}
Assume $k \in \cR^{(s,s+q-1)}(\omega,n_0)$, $q \ge 1$. Assume also that $0 < |k - k_{n_0}| < (\delta^{(s+q-1)}_0)^{1/16}$.

$(1)$ Definition~\ref{defi:8reflectionchoiceq} and  definition \eqref{eq:9.twolambdas5} inductively define the sets $\La^{(s')}_k(0)$ for $s' = s, \dots, s+q$, so that condition $(\mathfrak{S}_{s'})$ holds for any $s \le s' \le s+q-1$ and condition \eqref{eq:9lambdasetssize} holds.

$(2)$  $\La^{(s_1)}_k(m_1) \cap \La^{(s_2)}_k(m_2) = \emptyset$ for any $m_i \in \cM^{(s_i)}_{k,s+q-1}$, $i = 1, 2$, unless $s_1 = s_2$, $\La^{(s_1)}_k(m_1) = \La^{(s_2)}_k(m_2)$.
\end{lemma}

\begin{proof}
$(1)$ For $q = 1$, the condition $(\mathfrak{S}_{s})$ holds due to part $(3)$ in Lemma~\ref{lem:8setLambdas}. Therefore, part $(3)$ of Lemma~\ref{lem:9A.3} applies with $q = 1$. Furthermore, Definition~\ref{defi:9.LLL} applies and Lemma~\ref{lem:9.lLL} applies. This defines $\La^{(s+1)}_k(0)$ via \eqref{eq:9.twolambdas5}, and Lemma~\ref{lem:9setLambdas} applies. Due to part $(3)$ in Lemma~\ref{lem:9setLambdas}, the condition $(\mathfrak{S}_{s+1})$ holds. These arguments define  $\La^{(s+q')}_k(0))$, $q' = 1, \dots, q$. So, part $(1)$ of the current lemma holds.

$(2)$ Since condition $(\mathfrak{S}_{s'})$ holds for $s' \ge s$, part $(2)$ follows from part $(3)$ of Lemma~\ref{lem:9A.3}.
\end{proof}

\begin{defi}\label{defi:9.LLLIII}
Using the notation from Definition~\ref{defi:8reflectionchoiceq}, assume that $(\delta^{(s+q')}_0)^{1/16} \le |k - k_{n_0}| < (\delta^{(s+q'-1)}_0)^{1/16}$ for some $q' \le q$. Define $\La^{(s')}_k(0)$ for $s \le s' \le s+q'$ via Lemma~\ref{lem:9setLambdas}. If $q' < q$, define $\La^{(s')}_k(m)$ for $s+q' < s' \le s+q-1$, and $\La^{(s')}_k(0)$ for $s+q' < s' \le s+q$, inductively as in \eqref{eq:A.1}, that is, by setting
\begin{equation}\label{eq:9A.1LLLL}
\begin{split}
\La^{(s')}_k(0) = B(3 R^{(s')}) \setminus \Bigl( \bigcup_{r \le s'-1} \bigcup_{m' \in \cM^{(r)}_{k,s'-1} : \La^{\ar}_{k}(m') \between B(3 R^{(s')}))} \La^\ar_{k}(m') \Bigr), \\
\La^{(s')}_k(m) = m + \La^{(s')}_{k + m \omega} (0).
\end{split}
\end{equation}
\end{defi}

\begin{lemma}\label{lem:9conditionLambdaAAA}
Assume $k \in \cR^{(s,s+q-1)}(\omega,n_0)$, $q \ge 1$. Assume also that $0 < |k - k_{n_0}| < 2 \sigma(n_0)$.

$(1)$ Definition~\ref{defi:8reflectionchoiceq}, Lemma~\ref{lem:9setLambdas} and Definition~\ref{defi:9.LLLIII} inductively define the sets $\La^{(s')}_k(0)$ for $s' = s, \dots, s+q$ so that condition $(\mathfrak{S}_{s'})$ holds for any $s \le s' \le s+q-1$. If $|k - k_{n_0}| < (\delta^{(r)}_0)^{1/16}$, \eqref{eq:9lambdasetssize} holds. If $|k - k_{n_0}| \ge (\delta^{(r)}_0)^{1/16}$, then
\begin{equation}\label{eq:9lambdasetssizeAG}
B(2 R^\ar) \subset \La^{(r)}_k(0) \subset B(3 R^\ar).
\end{equation}

$(2)$ $\La^{(s_1)}_k(m_1)) \cap \La^{(s_2)}_k(m_2) = \emptyset$ for any $m_i \in \cM^{(s_i)}_{k,s+q-1}$, $i = 1, 2$, unless $s_1 = s_2$, $\La^{(s_1)}_k(m_1) = \La^{(s_2)}_k(m_2)$.
\end{lemma}

\begin{proof}
To prove both statements, we only need to verify condition $(\mathfrak{S}_{s'})$. Due to Lemma~\ref{lem:9conditionLambda}, this condition holds if $|k - k_{n_0}| < (\delta^{(s'-1)}_0)^{1/16}$. Assume $(\delta^{(s+q_1)}_0)^{1/16} \le |k - k_{n_0}| < (\delta^{(s+q_1-1)}_0)^{1/16}$ for some $q_1 < q$. The verification goes by induction, starting with $s' = s + q_1$. Assume that condition $(\mathfrak{S}_{s'})$ holds for any $s+q_1 \le s' \le q-1$. Then, part $(3)$ of Lemma~\ref{lem:9A.3} applies. Since $\La^{(s+q)}_k(0)$ is defined via \eqref{eq:9A.1LLLL}, the verification of condition $(\mathfrak{S}_{s+q})$ goes the same way as in the proof of part $(4)$ of Lemma~\ref{lem:A.3}.
\end{proof}

\begin{remark}\label{rem:9.1abcasesQ1AG}
In Lemma~\ref{lem:9conditionLambdaAAA}, we assume that $k \in \IR \setminus \bigcup_{0 < |m'| \le 12 R^{(s+q-1)}, \; m' \neq n_0} (k_{m',s+q-1}^-, k_{m',s+q-1}^+)$, instead of $k \in \IR \setminus \bigcup_{0 < |m'| \le 12 R^{(s+q)}, \; m' \neq n_0} (k_{m',s+q-1}^-, k_{m',s+q-1}^+)$.
\end{remark}

\begin{prop}\label{prop:9.1}
\begin{itemize}

\item[(I)]
Set
\begin{equation}\label{eq:9kintcondNN}
\mathcal{I}(s,q) := \{ k' : |k' - k_{n_0}| < (\delta^\esone_0)^{3/4} - \sum_{s-1 \le t \le s+q} (\delta^{(t)}_0)^{31/32} \}.
\end{equation}

Assume that $k \in \cR^{(s,s+q)}(\omega,n_0) \cap \mathcal{I}(s,q) $. Let $\ve_0$, $\ve_{s}$ be as in Definition~\ref{def:4-1}. Let $\ve \in (-\ve_{s},\ve_{s})$.
\begin{itemize}

\item[(1)] If $|k| > |k_{n_0}|$, then for any $k' \in \mathcal{I}(s,q)$ with $|k'| > |k_{n_0}|$, $|k'-k| < \delta_0^{(s+q-1)}$, one has $H_{\La^{(s+q)}_k(0),\ve,k'} \in OPR^{(s,s+q)} \bigl( 0,n_0, \La^{(s+q)}_k(0); \delta_0, \tau^\zero \bigr)$,  $\tau^\zero = \tau^\zero(k') = |k_{n_0}| | |k'| - |k_{n_0}| |$. If $|k| < |k_{n_0}|$, then for any $k' \in \mathcal{I}(s,q)$ with $|k'| < |k_{n_0}|$, $|k' - k| < \delta_0^{(s+q-1)}$, one has $H_{\La^{(s+q)}_k(0),\ve,k'}\in OPR^{(s,s+q)}\bigl(n_0,0, \La^{(s+q)}_k(0); \delta_0,\tau^\zero\bigr)$.

\item[(2)] Let $E^{(s+q,\pm)}(\La^{(s+q)}_k(0); \ve,k')$ be the functions defined in Proposition~\ref{rem:con1smalldenomnn} with $H_{\La^{(s+q)}_k(0),\ve,k'}$ in the role of $\hle$. Assume that $k_{n_0} > 0$. Then, with $k^\zero:=\min(\ve_0^{3/4}, k_{n_0}/512)$, one has
\begin{equation}\label{eq:9Ekderivatives}
\begin{split}
\partial_\theta E^{(s+q,+)}(0, \La^{(s+q,\mathbf{1})}_k(0); \ve, k_{n_0} + \theta) > (k^\zero)^2 \theta, \quad \theta > 0, \\
\partial_\theta E^{(s+q,-)}(0, \La^{(s+q,\mathbf{1})}_k(0); \ve, k_{n_0} + \theta) < -(k^\zero)^2 \theta, \quad \theta > 0,
\end{split}
\end{equation}
\begin{equation}\label{eq:9Esymmetry}
E^{(s+q,\pm)}(0, \La^{(s+q,\mathbf{1})}_k(0); \ve, k_{n_0} + \theta) = E^{(s+q,\pm)}(n_0, \La^{(s+q,\mathbf{1})}_k(0); \ve, k_{n_0} - \theta), \quad \theta > 0,
\end{equation}
\begin{equation}\label{eq:9Efirstder}
|\partial_\theta E^{(s+q,\pm)}(0, \La^{(s+q)}_k(0); \ve, k_{n_0} + \theta)| \le 2,
\end{equation}
\begin{equation}\label{eq:9kk1La}
|E^{(s+q,\pm)}(0, \La^{(s+q)}_k(0);\ve,k_1) - E^{(s+q,\pm)}(0, \La^{(s+q-1)}_{k_1}(0);\ve,k_1)| \le |\ve| (\delta^{(s+q)}_0)^5.
\end{equation}

If $0 < |k - k_{n_0}| < (\delta^{(s+q-1)}_0)^{1/16}/2$,
\begin{equation}\label{eq:9Esymmetry-2}
E^{(s+q,\pm)}(0, \La^{(s+q)}_k(0); \ve, k_{n_0} + \theta) = E^{(s+q,\pm)}(n_0, \La^{(s+q)}_k(0); \ve, k_{n_0} - \theta), \quad 0 < \theta < (\delta^{(s+q-1)}_0)^{1/16}/2.
\end{equation}

\end{itemize}

\item[(II)] Assume that $k \in \cR^{(s,s+q)}(\omega,n_0)$ and
\begin{equation}\label{eq:9koutcondState}
(\delta^\esone_0)^{7/8} + \sum_{s-1 \le t \le s+q-1} (\delta^{(t)}_0)^{31/32} < |k - k_{n_0}| < 2 \sigma(n_0).
\end{equation}
Then, $H_{\La^{(s+q,\mathbf{1})}_k(0), \varepsilon, k} \in \cN^{(s+q)} \bigl( 0, \La^{(s+q,\mathbf{1})}_k(0); \delta_0 \bigr)$. Furthermore,
$$
E^{(s+q)}(0, \La^{(s+q,\mathbf{1})}_k(0); \ve, k) = \begin{cases} E^{(s+q,+)}(0, \La^{(s+q,\mathbf{1})}_k(0);\ve,k) & \text{if $k > k_{n_0}$}, \\ E^{(s+q,-)}(0, \La^{(s+q,\mathbf{1})}_k(0); \ve, k) & \text{if $k < k_{n_0}$}. \end{cases}
$$

\end{itemize}
\end{prop}

\begin{proof}
The proof of $(I)$ is completely similar to the proof of Proposition~\ref{prop:8.1} and we omit it. The proof of $(II)$ is completely similar to the proof of Proposition~\ref{prop:A.3} and we omit it as well.
\end{proof}

We also need the following version of Proposition~\ref{prop:8.10}.

\begin{prop}\label{prop:9.10}
Let $\ve \in (-\ve_{s},\ve_{s})$.

$(1)$ The limits
\begin{equation}\label{eq:9kk1comp1lim}
E^{(s+q,\pm)}(0, \La^{(s+q)}_{k_{n_0}}(0); \ve, k_{n_0}):=\lim_{k_1\rightarrow k_{n_0}} E^{(s+q,\pm)}(0, \La^{(s+q)}_{k_{n_0}}(0); \ve, k_1)
\end{equation}
exist. Furthermore,
\begin{equation} \label{eq:9specHEEAAA}
\begin{split}
\spec H_{\La^{(s+q)}_{k_{n_0}}(0), \ve, k_{n_0}} \cap \{ E : \min_\pm |E - E^{(s+q-1,\pm )}(0,\La^{(s+q-1)}_{k_{n_0}}(0); \ve, k_{n_0})| < 8 (\delta^{(s+q-1)}_0)^{1/4} \} \\
= \{ E^{(s,+)}(0, \La^{(s+q)}_{k_{n_0}}(0); \ve, k_{n_0}), E^{(s+q,-)}(0, \La^{(s+q)}_{k_{n_0}}(0); \ve, k_{n_0}) \},
\end{split}
\end{equation}
\begin{equation}\label{eq:9kk1comp1limap1}
|E^{(s+q,\pm)}(0, \La^{(s+q)}_{k_{n_0}}(0); \ve, k_{n_0}) - E^{(s+q-1,\pm)}(0, \La^{(s+q)}_{k_{n_0}}(0); \ve, k_{n_0})| \le |\ve| \delta_0^{(s+q-1)},
\end{equation}
$E^{(s+q,+)}(0, \La^{(s+q)}_{k_{n_0}}(0)(0); \ve, k_{n_0}) \ge E^{(s+q,-)}(0, \La^{(s+q)}_{k_{n_0}}(0); \ve, k_{n_0})$.

$(1)'$ Let $\min_\pm |E - E^{(s+q-1,\pm)} \bigl( 0, \La^{(s+q-1)}_{k_{n_0}}(0); \ve,k_{n_0} \bigr)| < 2 \delta_0^{(s+q-1)}$. The matrix $(E - H_{\La^{(s+q)}_{k_{n_0}}(0) \setminus \{0, n_0\},\ve,k_{n_0}})$ is invertible. Moreover,
\begin{equation}\label{eq:9Hinvestimatestatement1PQ}
\begin{split}
|[(E - H_{\La^{(s+q)}_{k_{n_0}}(0) \setminus \{0,n_0\},\ve,k_{n_0}})^{-1}](m,n)| \\
\le \begin{cases} 3 |\ve|^{1/2} \exp(-\frac{7}{8} \kappa_0 |m-n| + 8 \kappa_0 \log \delta_0^{-1} (\min (\mu^{(s+q,0)}(m),\mu^{(s+q,0)}(n) )^{1/5}) & \text{if $m\neq n$}, \\ 2 \exp(8 \kappa_0 \log \delta_0^{-1} (\mu^{(s+q,0)}(m))^{1/5}) & \text{if $m = n$},
\end{cases}
\end{split}
\end{equation}
$\mu^{(s+q,0)}(m) := \dist(m, \zv \setminus [\La^{(s+q)}_{k_{n_0}}(0) \setminus \{0, n_0\}])$.

$(2)$ $E = E^{(s+q,\pm)}(0, \La^{(s+q)}_{k_{n_0}}(0); \ve, k_{n_0})$ obeys the following equation,
\begin{equation} \label{eq:9Eequation0}
E - v(0, k_{n_0}) - Q^{(s+q)}(0,\La^{(s+q)}_{k_{n_0}}(0); \ve, E)  \mp \big| G^{(s+q)}(0,n_0,\La^{(s+q)}_{k_{n_0}}(0); \ve, E) \big| = 0,
\end{equation}
where
\begin{equation} \label{eq:9-10acbasicfunctions}
\begin{split}
Q^{(s+q)}(0,\La^{(s+q)}_{k_{n_0}}(0); \ve, E) \\
= \sum_{m',n' \in \La^{(s+q)}_{k_{n_0}}(0) \setminus \{0,n_0\}} h(m_0^\pm, m'; \ve, k_{n_0}) [(E - H_{\La^{(s+q)}_{k_{n_0}}(0) \setminus \{0,n_0\}, \ve,k_{n_0}})^{-1}] (m',n') h(n', m_0^\pm; \ve, k_{n_0} ), \\
G^{(s+q)}(0,n_0,\La^{(s+q)}_{k_{n_0}}(0); \ve, E) = h(m_0^\pm, m_0^\mp; \ve, k_{n_0} ) \\
+ \sum_{m', n' \in \La^{(s+q)}_{k_{n_0}}(0) \setminus \{0,n_0\}} h(m_0^\pm, m'; \ve, k_{n_0}) [(E - H_{\La^{(s+q)}_{k_{n_0}}(0) \setminus \{0,n_0\},\ve,k_{n_0}})^{-1}] (m',n') h(n', m_0^\mp; \ve, k_{n_0}).
\end{split}
\end{equation}
\end{prop}

\begin{proof}
The proof of $(1)$  goes just like the proof of $(1)$ in Proposition~\ref{prop:8.10}. Let us verify $(1)'$. Let, for instance, $k_{n_0} > 0$. Due to part $(I)$ of Proposition~\ref{prop:9.1}, one has $H_{\La^{(s+q)}_{k_{n_0}}(0),\ve,k'} \in OPR^{(s,s+q)} \bigl( 0,n_0, \La^{(s+q)}_{k_{n_0}}(0); \delta_0,\tau^\zero \bigr)$ for any $0 < k' - k_{n_0} < \delta_0^{(s+q-1)}$. Due to part $(2)$ of Proposition~\ref{rem:con1smalldenomnn},
\begin{equation}\label{eq:9Hinvestimatestatement1PQk1}
|[(E - H_{\La^{(s+q)}_{k_{n_0}}(0) \setminus \{0,n_0\},\ve,k'})^{-1}](m,n)| \le s_{D(\cdot;\La^{(s+q)}_{k_{n_0}}(0)\setminus\{0,n_0\}),T,\kappa_0,|\ve|;\La^{(s+q)}_{k_{n_0}}(0) \setminus \{0,n_0\},\mathfrak{R}}(m,n);
\end{equation}
see \eqref{eq:3Hinvestimatestatement1PQ}. It follows from Lemma~\ref{lem:auxweight1} that
\begin{equation}\label{eq:9Hinvestimatestatement1PQk2}
\begin{split}
s_{D(\cdot;\La^{(s+q)}_{k_{n_0}}(0) \setminus \{0,n_0\}),T,\kappa_0,|\ve|;\La^{(s+q)}_{k_{n_0}}(0) \setminus \{0,n_0\},\mathfrak{R}}(m,n) \le \\
\begin{cases} 3 |\ve|^{1/2} \exp(-\frac{7}{8} \kappa_0 |m - n| + 8 \kappa_0 \log \delta_0^{-1} (\min (\mu^{(s+q,0)}(m),\mu^{(s+q,0)}(n) )^{1/5}) & \text{if $m \neq n$}, \\ 2 \exp (8 \kappa_0 \log \delta_0^{-1} (\mu^{(s+q,0)}(m))^{1/5}) & \text{if $m = n$}. \end{cases}
\end{split}
\end{equation}
Taking $k' \rightarrow k_{n_0}$ in \eqref{eq:9Hinvestimatestatement1PQk1}, one obtains \eqref{eq:9Hinvestimatestatement1PQ}. This verifies $(1)'$. The verification of $(2)$ goes just like the one for $(2)$ in Proposition~\ref{prop:8.10}.
\end{proof}

\section{Matrices with a Graded System of Ordered Pairs of Resonances Associated with $1$--Dimensional Quasi-Periodic Schr\"odinger Equations}\label{sec.10}

\begin{defi}\label{def:10.simplestresonance}
Using the notation from Proposition~\ref{prop:9.1}, let $q \ge 2$, $n_1 \in \IZ^\nu$, $12R^{(s+q-1)} < |n_1| \le 12R^{(s+q)}$ be such that
\begin{equation}\label{eq:10kintcondNN}
\begin{split}
(k_{n_1} - 2 \sigma(n_1), k_{n_1} + 2 \sigma(n_1)) \cap (k_{n_{0}} - 2 \sigma(n_0), k_{n_{0}} + \sigma(n_0)) \neq \emptyset, \\
(k_{n_1} - 2 \sigma(n_1), k_{n_1} + 2 \sigma(n_1)) \subseteq \IR \setminus \bigcup_{R^\es \le |m'| \le 12 R^{(s+q)}, \quad m' \notin \{n_0,n_1\}} (k_{m',s+q-1}^-, k_{m',s+q-1}^+).
\end{split}
\end{equation}
Set $s^\zero := s$, $s^\one := s + q$, $\mathfrak{s}^\one = (s^\zero,s^\one)$. Let $\cR^{(\mathfrak{s}^{(1)},s^{(1)})}(\omega,n_1)$ be the non-empty set in the first line of \eqref{eq:10kintcondNN}.
\end{defi}

\begin{lemma}\label{lem:10A.3}
Let $k \in \cR^{(\mathfrak{s}^{(1)},s^{(1)})}(\omega,n_1)$.

$(1)$ The subsets $\cM^{(s')}_{k,s^{(1)}-1}$, $\La^{(s')}_k(m)$, $s' \le s^{(1)}-1$ from Definition~\ref{defi:8reflectionchoiceq} are well-defined. Furthermore, $H_{\La^{(s')}_k(m),\ve,k} \in \cN^{(s')}(m,\La^\ar_k(m),\delta^\zero_0)$ , $s' \le s^\zero-1$. If
\begin{equation}\label{eq:10koutcondm1inn0}
|k_{n_1} - k_{n_0}| < (\delta^{(s^\zero-1)}_0)^{3/4} - 4 \sum_{s^\zero - 1 \le t \le s^\one-2} (\delta^{(t)}_0)^{15/16} - 2 \sigma(n_1),
\end{equation}
then $H_{\La^{(s')}_k(m^+_j),\ve,k} \in OPR^{(s')}(m^+_j,m^-_j,\La^\ar_k(m^+_j),\delta^\zero_0,\tau^\zero)$, $\tau^\zero = \delta^\zero_0 |k - k_{n_0}|$ for any $s^\zero \le s' \le s^{(1)}-1$ and any $m^+_j$.

$(2)$ $|k - k_{n_0}| > (\delta^{(s^\one-1)}_0)^{1/15}$, $\tau^\zero > (\delta^{(s^\one-1)}_0)^{1/14}$.

$(3)$ Assume \eqref{eq:10koutcondm1inn0} holds. Then conditions $(i)$--$(iv)$ and $(vi)$ from Definition~\ref{def:7-6} hold. Furthermore,  $n_1 \in \cM^{(s^\one-1)}_{k,s^\one-1}$, that is, $n_1 \in \{m^+_{j_0},m_{j_0}^-\}$ for some $j_0 \in J^{(s^\one-1)}$. Conditions \eqref{eq:6EVsplitdefs1a}--\eqref{eq:6-9EVsplitdefs1a} from Definition~\ref{def:9-6} hold. Finally, assume that $|k - k_{n_1}| > (\delta^{(s^\one-1)}_0)^{7/8}$.  Let $\La^{(s^\one)}_k(0)$ be as in Definition~\ref{defi:9.LLLIII}. Then, $H_{\La^{(s^\one)}_k(0),\ve,k} \in OPR^{(s^\one)}(0,n_0,\La^{(s+q)}_k(0),\delta^\zero_0,\tau^\zero)$, just as in Proposition~\ref{prop:9.1}.
\end{lemma}

\begin{proof}
$(1)$ Clearly, $k \in \IR \setminus \bigcup_{0 < |m'| \le 12 R^{(s^\one-1)},\quad m' \neq m^\zero} (k_{m',s^\one-1}^-,k_{m',s^\one-1}^+)$. So, part Proposition~\ref{prop:9.1} applies to $q-1$ in the role of $q$. This implies all statements in $(1)$.

$(2)$ Using \eqref{eq:diphnores} and \eqref{eq:7K.1}, one obtains $|k - k_{n_0}| > |k_{n_1} - k_{n_0}| - 2 (\delta^{(s^\one-1)}_0)^{1/6} > (\delta^{(s^\one-1)}_0)^{1/15}$, $\tau^\zero = \delta^\zero_0 |k - k_{n_0}| > (\delta^{(s^\one-1)}_0)^{1/14}$ since $0 < |n_1 - n_0| < 13 R^{(s^\one)}$; see \eqref{eq:diphnores}.

$(3)$ Assume \eqref{eq:10koutcondm1inn0} holds. Conditions (i)--(iv) and (vi) from Definition~\ref{def:7-6} hold due to Lemma~\ref{lem:9conditionLambda} and Remarks~\ref{rem:9.1inout} and \ref{rem:9.1abcasesQ1AG}. One has $|k - k_{n_0}| \le |k_{n_1} - k_{n_0}| + |k - k_{n_1}| < (\delta^{(s^\one-1)}_0)^{3/4}$, $||k + m_{0,1}\omega| - |k_{n_0}|| \le ||k_{n_1}| - |k_{n_0}|| + ||k| - |k_{n_1}|| < (\delta^{(s^\one-1)}_0)^{3/4}$. Hence, $|v(n_0,k) - v(0,k)| < 2 (\delta^{(s^\one-1)}_0)^{3/4} 2 (|k_{n_0} + 1|) < 3 (\delta^{(s^\one-2)}_0)/4$. Similarly, $|v(n_1+n_0,k) - v(0,k)| < 3 (\delta^{(s^\one-2)}_0)/4$ if it is case $(\mathfrak{a})$ for $m_1$ and $|v(n_1-n_0,k) - v(0,k)| < 3 (\delta^{(s^\one-2)}_0)/4$ if it is case $(\mathfrak{b})$. This implies $n_1 \in \cM^{(s^\one-1)}_{k,s^\one-1}$. Let $m \in \cM^{(s^\one-1)}_{k,s^{(1)}-1}$, $\La^{(s^\one-1)}_{k}(m) = m + \La^{(s^\one-1)}_{k+m\omega}(0)$. Assume for instance that it is case $(\mathfrak{a})$. Recall that $|m| < 12 R^{(s+q)}$ and in particular $|m\omega| > (\delta^{(s^\one-1)}_0)^{1/16}$; see \eqref{eq:diphnores}.
Using \eqref{eq:9Ekderivatives} and \eqref{eq:9kk1La} with $q-1$ in the role of $q$, one obtains
\begin{equation}\label{eq:10.condPROP}
\begin{split}
|E^{(s^\one-1,+)}(m, \La^{(s^\one-1)}_{k}(m);\ve,k) - E^{(s^\one-1,+)}(0, \La^{(s^\one-1)}_k(0);\ve,k)| \\
= |E^{(s^\one-1,+)}(0, \La^{(s^\one-1)}_{k+m\omega}(0);\ve,k+m\omega) - E^{(s^\one-1,+)}(0, \La^{(s^\one-1)}_k(0);\ve,k)| \\
\ge |E^{(s^\one-1,+)}(0, \La^{(s^\one-1)}_{k+m\omega}(0);\ve,k+m\omega) - E^{(s^\one-1,+)}(0, \La^{(s^\one-1)}_{k+m\omega}(0);\ve,k)| \\
- |E^{(s^\one-1,+)}(0, \La^{(s^\one-1)}_{k+m\omega}(0);\ve,k) - E^{(s^\one-1,+)}(0, \La^{(s^\one-1)}_k(0);\ve,k)| \\
\ge \delta_0^\zero (|k+m\omega|-|k||)^2 - (\delta^{(s^\one-1)}_0)^5 = \delta_0^\zero (|m\omega|)^2 - (\delta^{(s^\one-1)}_0)^5 > (\delta^{(s^\one-1)}_0)^{1/7} > \delta^{(s^\one-1)}_0.
\end{split}
\end{equation}
This implies \eqref{eq:6EVsplitdefs1a}. The verification in case $(\mathfrak{b})$ is similar. The verification of \eqref{eq:6EVsplitdefAs1b}, \eqref{eq:6-9EVsplitdefs1a} is similar. Finally, assume that $|k - k_{n_1}| > (\delta^{(s^\one-1)}_0)^{7/8}$. Then an estimation like \eqref{eq:10.condPROP} works for $m = n_{1}$. This implies $H_{\La^{(s^\one)}_k(0),\ve,k} \in OPR^{(s^\one)}(0, n_0, \La^{(s+q)}_k(0),\delta^\zero_0,\tau^\zero)$, just as in Proposition~\ref{prop:9.1}.
\end{proof}

Set
\begin{equation}\label{eq:10.mapT1}
T_1(n) = n_1 - n, \quad n \in \zv.
\end{equation}

\begin{lemma}\label{lem:10mdeltaTm}
Assume $|k - k_{n_1}| < (\delta^{(s^\one-1)}_0)^{3/4 }$.

$(1)$ If $|v(m,k) - v(0,k)| < \delta$ with $\delta \ge (\delta^{(s^\one-1)}_0)^{1/2}/4$, then $|v(T_1(m),k) - v(0,k)| < 4 \delta/3$.

$(2)$ Let $m_j \in \cM^{(s')}_{k,s^\one-1}$, $j=1,2$. Then either $T_1(\La^{(s')}(m_1)) \cap \La^{(s')}(m_2) \neq \emptyset$ or $\dist(T_1(\La^{(s')}(m_1)),\La^{(s')}(m_2)) > 5 R^{(s')}$. In the former case, $T_1 \bigl(\La^{(s')}(m_1) \cap \cM^{(s')}_{k,s^\one-1} \bigr) = \La^{(s')}(m_2)) \cap \cM^{(s')}_{k,s^\one-1}$.
\end{lemma}

\begin{proof}
$(1)$ The proof of this part is completely similar to the proof of part $(1)$ of Lemma~\ref{lem:8mdeltaTm}.

$(2)$ It follows from $(1)$ of the current lemma that $|v(T_1(m_1),k) - v(0,k)|) \le \delta^{(s'-1)}_0$. Applying part $(1)$ of Lemma~\ref{lem:9A.3} to $T_1(m_1)$ and $m_2$, one obtains the statement.
\end{proof}

\begin{defi}\label{defi:10.LLL}
Assume that
\begin{equation}\label{eq:9kintcondCONDDEF}
\begin{split}
0 < |k - k_{n_0}| < (\delta^\esone_0)^{3/4} -\sum_{s-1 \le t \le s+q-1} (\delta^{(t)}_0)^{31/32}, \\
|k - k_{n_1}| < (\delta^{(s^\one-1)}_0)^{3/4}.
\end{split}
\end{equation}
Let $\mathfrak{L}'$ be the collection of all sets $\La(m) := \La^{(s')}(m) \cup T_1(\La^{(s')}(m))$, $1 \le s' \le s+q-1$, $m \in \cM^{(s')}_{ k,s+q-1}$. We say that $\La(m_1) \approx \La(m_2)$ if $s_1 = s_2$ and $\La(m_1) \cap \La(m_2) \neq \emptyset$. It follows from  part $(3)$ of Lemma~\ref{lem:9A.3} and part $(2)$ of Lemma~\ref{lem:10mdeltaTm} that this is indeed an equivalence relation on $\mathfrak{L}'$. Let $\mathfrak{M}$ be the set of equivalence classes. It follows from this definition and part $(2)$ of Lemma~\ref{lem:10mdeltaTm} that each class has at most two elements in it. For each $\mathfrak{m} \in \mathfrak{M}$, set $\La(\mathfrak{m}) = \bigcup_{\La(m_1) \in \mathfrak{m}} \La(m_1)$. Set $\mathfrak{L} = \{ \La(\mathfrak{m}) : \mathfrak{m} \in \mathfrak{M} \}$. Let $\La(\mathfrak{m}) \in \mathfrak{L}$, $\La^{(s')}(m) \in \mathfrak{m}$. Set $t(\La(\mathfrak{m})) = s'$. This defines an $\mathbb{N}$-valued function on $\mathfrak{L}$. Set also $p_\mathfrak{m} = \bigcup_{\La^{(s')}_k(m)) \in \mathfrak{m}} \La^{(s')}_k(m) \cap \cM^{(s')}_{k,s^\one-1}$.
\end{defi}

\begin{lemma}\label{lem:10.lLL}
$(1)$ For any $\La(\mathfrak{m}_j) \in \mathfrak{L}$, $j=1,2$, such that $t(\La(\mathfrak{m}_1)) = t(\La(\mathfrak{m}_2))$, $\mathfrak{m}_1 \neq \mathfrak{m}_2$, we have $\dist (\La(\mathfrak{m}_1), \La(\mathfrak{m}_2)) > R^{(t(\La(\mathfrak{m}_1)))}$.

$(2)$ For any $\mathfrak{m}$, we have
\begin{equation}\label{eq:10Lmsets}
\bigcup_{m \in p_\mathfrak{m}} \bigl((m + B(2R^{(t(\La(\mathfrak{m})))}) \subset \La(\mathfrak{m}) \subset \bigcup_{m \in p_\mathfrak{m}} \bigl((m + B(3R^{(t(\La(\mathfrak{m})))}).
\end{equation}

$(3)$ The pair $(\mathfrak{L},t)$ is a proper subtraction system.

$(4)$ For any $\mathfrak{m}$, we have $\La(\mathfrak{m}) = T_1(\La(\mathfrak{m}))$.
\end{lemma}

\begin{proof}
The proof is completely similar to the proof of Lemma~\ref{lem:9.lLL}.
\end{proof}

Set
\begin{equation} \label{eq:10.twolambdas5}
\begin{split}
\mathfrak{B}(n_1) := B(3R^{(s^\one)}) \cup (n_1 + B(3R^{(s^\one)})), \\
\mathfrak{B}(n_1,\ell) = \mathfrak{B}(n_1,\ell-1) \setminus \Bigl( \bigcup_{\mathfrak{m} \in \mathfrak{M} : \La(\mathfrak{m}) \between \mathfrak{B}(n_1,\ell-1)} \La(\mathfrak{m})\Bigr),
\end{split}
\end{equation}
$\ell = 1, 2, \ldots$.

\begin{lemma}\label{lem:10setLambdas}
$(1)$ There exists $\ell_0 < 2^{s^\one}$ such that $\mathfrak{B}(n_1,\ell) = \mathfrak{B}(n_1,\ell-1)$ for any $\ell \ge \ell_0$.

$(2)$ For any $\La \in \mathfrak{L}$, we have either $\La \subset \mathfrak{B}(n_1,\ell_0)$ or $\La \subset \Bigl( \IZ^\nu \setminus \mathfrak{B}(n_1,\ell_0) \Bigr)$.

$(3)$ Set $\La^{(s^\one)}_k(0) = \mathfrak{B}(n_1,\ell_0)$. Then, for any $\La^{(s')}(m)$, we have either $\La^{(s')}(m) \cap \La^{(s^\one)}_k(0) = \emptyset$ or $\La^{(s')}(m) \subset \La^{(s^\one)}_k(0)$.

$(4)$ $T_1(\mathfrak{B}(n_1,\ell)) = \mathfrak{B}(n_1,\ell)$ for any $\ell$. In particular, $T_1(\La^{(s^\one)}_k(0)) = \La^{(s^\one)}_k(0)$.

$(5)$ For any $\ell \ge 1$, we have
\begin{equation}\label{eq:9.twolambdas5NEW-2}
\{ n \in \mathfrak{B}(n_1,\ell-1)) : \dist(n,\IZ^\nu \setminus \mathfrak{B}(n_1,\ell-1)) \ge 3 R^{(s^\one-1)} \} \subset \mathfrak{B}(n_1,\ell) \subset \mathfrak{B}(n_1,\ell-1)).
\end{equation}
\end{lemma}

\begin{proof}
The proof of parts $(1)$--$(5)$ goes word for word as the proof of parts $(1)$--$(5)$ of Lemma~\ref{lem:8setLambdas}.
\end{proof}

\begin{prop}\label{prop:10.1}
$(I)$ Let $k$ be as in \eqref{eq:9kintcondCONDDEF}. Set $\mathfrak{m}^{(1)} := \{ 0, n_0, n_1, n_1-n_0 \}$, $\mathfrak{s}^{(1)} = (s^\zero,s^\one)$. If $|k| > |k_{n_1}|$, then $H_{\La^{(s^\one)}_k(0), \ve, k} \in GSR^{(\mathfrak{s}^{(1)})} \bigl( \mathfrak{m}^{(1)},0,n_1, \La^{(s^\one)}_k(0); \delta_0,\mathfrak{t}^\one \bigr)$, $\mathfrak{t}^\one = (\tau^\zero,\tau^\one)$, $\tau^{(r)} = \tau^{(r)}(k) = |k_{n_r}| ||k| - |k_{n_r}||$. If $|k| < |k_{n_1}|$, then $H_{\La^{(s^\one)}_k(0), \ve, k'} \in GSR^{(\mathfrak{s}^{(1)})} \bigl( \mathfrak{m}^{(1)},n_1,0, \La^{(s^\one)}_k(0); \delta_0,\mathfrak{t}^\one \bigr)$.

$(II)$ Let $k \in \IR \setminus \big[ \bigcup_{R^\es \le |m'| \le 12 R^{(s+q)}, \quad m' \notin \{n_0,n_1\}} (k_{m',s+q-1}^-,k_{m',s+q-1}^+) \cup \{k_{n_0},k_{n_1}\}$. One can define $\La^{(s^\one)}_k(0)$ so that $H_{\La^{(s^\one)}_k(0), \ve, k}$ belongs to one of the classes introduced in Sections~\ref{sec.3}, \ref{sec.5}, \ref{sec.6} with $0$ being either the principal point or one of the two principal points.

$(III)$ Let $k$ be as in part $(II)$. There exists a unique real-analytic function $E(0,\La^{(s^\one)}_k(0);\ve,k)$ of $\ve \in (-\ve_{s-1}, \ve_{s-1})$ such that $E(0,\La^{(s^\one)}_k(0);\ve,k)$ is a simple eigenvalue of $H_{\La^{(s^\one)}_k(0), \ve, k}$ and $E(0,\La^{(s^\one)}_k(0);\ve,k) = v(0,k)$. Moreover,
\begin{equation}\label{eq:10Esymmetry}
E(0,\La^{(s^\one)}_k(0);\ve,k) = E(0,\La^{(s^\one)}_{-k}(0);\ve,-k),
\end{equation}
\begin{equation}\label{eq:10Ekk1EG}
\begin{split}
(k^\zero)^2 (k - k_1)^2 - 3 |\ve| (\delta^{(s^{(1)})}_0)^{4} - 10 |\ve| \sum_{s'\ge s \delta^{(s')}_0 < \min (k-k_1,k_1)} (\delta^{(s')}_0)^{4} < E(0,\La^{(s^\one)}_k(0);\ve,k) - E(0,\La^{(s^\one)}_{k_1}(0);\ve,k_1) \\
< \frac{2k}{\lambda} (k - k_1) + \sum_{k_1 < k_{n_\ell} < k} 2 |\ve| (\delta^{(s^{(\ell)}-1)}_0)^{1/8} + 2 |\ve| (\delta^{(s^\one)}_0)^5, \quad 0 < k_1 < k,\quad \gamma-1\le k_1 \le  \gamma.
\end{split}
\end{equation}
where $k^\zero:=\min(\ve_0^{3/4}, k_{n_0}/512)$ and $\gamma$ is the same as in the Definition \eqref{eq:7-5-7}.
\end{prop}

\begin{proof}
The proof of $(I)$ is completely similar to the proof of Proposition~\ref{prop:8.1} and we omit it. The proof of $(II)$ is completely similar to the proof of Proposition~\ref{prop:A.3} and we omit it as well. The existence of $E(0,\La^{(s^\one)}_k(0);\ve,k)$, its analyticity and uniqueness follows from
part $(II)$. The proof of \eqref{eq:10Ekk1EG}
$(III)$ is a simple combination
of the \eqref{eq:8Ekderivatives} from Proposition~\ref{prop:8.1} and \eqref{eq:7Ederivlower} from Proposition~\ref{prop:A.3}.
\end{proof}

\begin{defi}\label{def:10.finitereset} Set
\begin{equation}\label{eq:10K.1}
\begin{split}
\mathcal{I}_n = ( k_n - (\delta^\es_0)^{3/4}, k_n + (\delta^\es_0)^{3/4} ) \quad \text{if $12R^\esone < |n| \le 12 R^\es$}, \\
\mathcal{R}(k) = \{ n \in \zv \setminus \{0\} : k \in \mathcal{I}_n\}, \quad \mathcal{G} = \{ k : |\mathcal{R}(k)| < \infty \}.
\end{split}
\end{equation}
Let $k \in \mathcal{G}$ be such that $|\mathcal{R}(k)| > 0$. We enumerate the points of $\mathcal{R}(k)$ as $n^{(\ell)}(k)$, $\ell = 0,\dots,\ell(k)$, $1+\ell(k) = |\mathcal{R}(k)|$, so that $|n^{(\ell)}(k)| < |n^{(\ell+1)}(k)|$; see Lemma~\ref{lem:10resetdiscr1} below. Let $s^{(\ell)}(k)$ be defined so that $12 R^{(s^{(\ell)}(k)-1)} < n^{(\ell)}(k) \le 12 R^{(s^{(\ell)}(k))}$, $\ell = 0,\dots,\ell(k)$. Set
\begin{equation}\label{eq:10mjdefi}
\begin{split}
T_{m}(n) = m - n, \quad m,n \in \mathbb{Z}^\nu, \\
\mathfrak{m}^{(0)}(k) = \{ 0,n^{(0)}(k)\}, \quad \mathfrak{m}^{(\ell)}(k) = \mathfrak{m}^{(\ell-1)}(k) \cup T_{n^{(\ell)}(k)} (\mathfrak{m}^{(\ell-1)}(k)), \quad \ell = 1,\dots,\ell(k).
\end{split}
\end{equation}
\end{defi}

\begin{lemma}\label{lem:10resetdiscr1}
Assume $m_1 \in \mathcal{R}(k)$. Let $12 R^{(s_1-1)} < |m_1| \le 12 R^{(s_1)}$. Then,

$(1)$ $|m_1\omega| > (\delta^{(s_1-1)})^{1/16}$, $|k| > (\delta^{(s_1-1)})^{1/16}/2$.

$(2)$ $\sgn(k) = -\sgn (m_1\omega)$.

$(3)$ If $m_2 \in \mathcal{R}(k)$, $m_1 \neq m_2$, then $|m_1| \neq |m_2|$. If $|m_1| < |m_2|$, then, in fact, $|m_2| > R^{(s_1+1)}/2$.
\end{lemma}

\begin{proof}
Part $(1)$ follows from \eqref{eq:diphnores}. Part $(2)$ follows from $(1)$, see \eqref{eq:10K.1}. To prove $(3)$, one can assume that $|m_2| \ge |m_1|$. Since $|2k + m_i\omega| < 2(\delta^{(s_i)}_0)^{3/4}$, $i=1,2$ and $|m_2| \ge |m_1|$, one has $|m_1\omega - m_2\omega| < 4(\delta^{(s_1)}_0)^{3/4}$. It follows from \eqref{eq:diphnores} that $|m_1 - m_2| > R^{(s_1+1)}$. This implies $(3)$.
\end{proof}

\begin{lemma}\label{lem:10resetn0}
Let $n^{(0)} \in \zv \setminus \{0\}$. Then, $(1)$ $n^\zero \in \mathcal{R}(k_{n^\zero})$, $(2)$ $|n^\zero| = \max_{m \in \mathcal{R}(k_{n^\zero})} |m|$. In particular, $k_{n^\zero} \in \mathcal{G}$, $s^{(\ell(k_{n^\zero}))}(k_{n^\zero}) = s(n^\zero)$, where $12 R^{(s(n^\zero)-1)} < |m_1| \le 12 R^{(s(n^\zero))}$.
\end{lemma}

\begin{proof} Statement $(1)$ is obvious. Assume $m \in \mathcal{R}(k_{n^\zero})$, $m \neq n^\zero$, $|m| > |n^\zero|$. Then $|(m - n^\zero)\omega| = 2 |k_{n^\zero} - k_{m}| < (\delta^\es_0)^{3/4}$, where $12 R^\esone < |m| \le 12 R^\es$. This contradicts \eqref{eq:diphnores} since $|m - n^\zero| \le 2 |m| < 24 R^\es$. This proves the first statement in $(2)$. The other statements in $(2)$ follow from this one.
\end{proof}

In the next theorem we finalize the results on matrices associated with quasi-periodic Schr\"{o}dinger equations. We skip the proofs since they are completely similar to those we have done before.

\begin{thmd}
$(I)$ Let $k \in \mathcal{G} \setminus \frac{\omega}{2} (\mathbb{Z}^\nu\setminus \{0\})$ be such that $|\mathcal{R}(k)| > 0$. Let $\mathfrak{m}^{(\ell)}(k)$, $s^{(\ell)}(k)$, $\ell(k)$ be as in Definition~\ref{def:10.finitereset}. Given $q \ge 0$, there exists $\La^{(s^{(\ell(k))}(k)+q)}_k(0) \subset \mathbb{Z}^\nu$ such that $H_{\La^{(s^{(\ell(k))}(k)+q)}_k(0), \ve, k} \in GSR^{[\mathfrak{s}^{(\ell(k))}(k),s^{(\ell)}(k)+q]} \bigl( \mathfrak{m}^{(\ell(k))}(k),m^+(k),m^-(k), \La^{(s^{(\ell)}(k)+q)}_k(0); \delta_0,\mathfrak{t}^{(\ell(k))}(k)\bigr)$, $\mathfrak{t}^{(\ell)}(k) = (\tau^\zero(k),\dots,\tau^{(\ell)}(k))$, $\tau^{(r)}(k) = |k_{n_r}| ||k| - |k_{n_r}||$, $m^+(k) = 0$, $m^-(k) = n^{(\ell(k))}(k)$ if $|k| > |k_{n^{(\ell)}(k)}|$, $m^-(k) = 0$, $m^+(k) = n^{(\ell(k))}(k)$ if $|k| < |k_{n^{(\ell)}(k)}|$.

$(II)$ For each $k \in \mathcal{G}$ and each $s$, there exists $\La^{(s)}_k(0)$ such that $H_{\La^{(s)}_k(0), \ve, k} \in \cN^{(s+q)} \bigl( 0, \La^{(s)}_k(0); \delta_0 \bigr)$ if $\mathcal{R}(k) = \emptyset$, $H_{\La^{(s)}_k(0), \ve, k} \in GSR^{[\mathfrak{s}^{(\ell(k))}(k),s+q]} \bigl( \mathfrak{m}^{(\ell(k))}(k),m^+(k),m^-(k), \La^{(s^{(\ell)}(k)+q)}_k(0); \delta_0,\mathfrak{t}^{(\ell(k))}(k) \bigr)$ if $\mathcal{R}(k) \neq \emptyset$ and $s = s^{(\ell)}(k)+q$. Moreover, $\La^{(s-1)}_k(0)$ serves the role of the $(s-1)$-subset in the corresponding definition, see the definitions in Sections~\ref{sec.3}, \ref{sec.5}, \ref{sec.6}. Let $-\ve_0^{1/2} < E < 0$ be arbitrary. For each $s = 1, 2, \dots$, the matrix $(H_{\La^\es_0(0), \ve, k}-E)$ belongs to $\cN^{(s)}(0, \La^\es_0(0), \delta_0)$, see Proposition~\ref{prop:A.3E0}.

$(III)$ Let $k \in \mathcal{G}$. There exists a unique real-analytic function $E(0,\La^{(s)}_k(0);\ve,k)$ of $\ve \in (-\ve_{0}/2, \ve_{0}/2)$ such that
$E(0,\La^{(s)}_k(0);\ve,k)$ is a simple eigenvalue of $H_{\La^{(s)}_k(0), \ve, k}$ and $E(0,\La^{(s)}_k(0);\ve,k) = v(0,k)$. Moreover, the following conditions hold:
\begin{equation}\label{eq:10Eksquaire}
|E(0,\La^{(s)}_k(0);\ve,k) - v(0,k)| < \ve^{1/2},
\end{equation}
\begin{equation}\label{eq:10EsymmetryT}
E(0,\La^{(s)}_k(0);\ve,k) = E(0,\La^{(s)}_{-k}(0);\ve,-k),
\end{equation}
\begin{equation}\label{eq:10Ekk1EGT}
\begin{split}
(k^\zero)^2 (k - k_1)^2 - 3 |\ve| (\delta^{(s)}_0)^{4} - 10 |\ve| \sum_{\delta^{(s')}_0 < \min (k-k_1,k_1)} (\delta^{(s')}_0)^{4} < E(0,\La^{(s)}_k(0);\ve,k) - E(0,\La^{(s)}_{k_1}(0);\ve,k_1) \\
< \frac{2k}{\lambda} (k - k_1) + \sum_{k_1 < k_{n} < k, \quad s(n) \le s} 2 |\ve| (\delta^{(s(n)-1)}_0)^{1/8} + 2 |\ve| (\delta^{(s)}_0)^5, \quad 0 < k_1 < k,\quad \gamma-1\le k_1 \le  \gamma.
\end{split}
\end{equation}
where $s(n)$ is defined via $12 R^{(s(n)-1)} < |n| \le 12 R^{(s(n))}$,
$k^\zero:=\min(\ve_0^{3/4}, k_{n^\zero}/512)$ and $\gamma$ is the same as in the Definition \eqref{eq:7-5-7}.

$(IV)$ Let $n^{(0)} \in \zv \setminus \{0\}$ and $s \ge s^{(\ell(k_{n^\zero}))}$. Assume, for instance, $k_{n^\zero} > 0$.

$(1)$ The limits
\begin{equation}\label{eq:10kk1comp1lim}
E^\pm(0, \La^{(s)}_{k_{n^\zero}}(0); \ve, k_{n^\zero}) := \lim_{k_1 \rightarrow k_{n^\zero} \pm 0} E(0, \La^{(s)}_{k_{n^\zero}}(0); \ve, k_1)
\end{equation}
exist,
\begin{equation}\label{eq:10kk1comp1gapsize}
0 \le E^+(0, \La^{(s)}_{k_{n^\zero}}(0); \ve, k_{n^\zero}) - E^-(0, \La^{(s)}_{k_{n^\zero}}(0); \ve, k_{n^\zero}) \le 2|\ve|\exp \Big(-\frac{\kappa_0}{2} |n^\zero| \Big)
\end{equation}
Furthermore,
\begin{equation}\label{eq:10specHEEAAA}
\begin{split}
\spec H_{\La^{(s)}_{k_{n^\zero}}(0), \ve, k_{n^\zero}} \cap \{ E : \min_{\pm} |E - E^\pm (0,\La^{(s-1)}_{k_{n^\zero}}(0); \ve, k_{n^\zero})| < 8 (\delta^{(s-1)}_0)^{1/4} \} \\
= \{ E^+(0, \La^{(s)}_{k_{n^\zero}}(0); \ve, k_{n^\zero}),E^-(0, \La^{(s)}_{k_{n^\zero}}(0); \ve, k_{n^\zero}) \}.
\end{split}
\end{equation}
\begin{equation}\label{eq:10kk1comp1limapp}
|E^\pm(0, \La^{(s)}_{k_{n^\zero}}(0); \ve, k_{n^\zero}) - E^\pm(0, \La^{(s-1)}_{k_{n^\zero}}(0); \ve, k_{n^\zero})| \le |\ve|\delta^{(s-1)}_0,
\end{equation}
$E^{+}(0, \La^{(s)}_{k_{n^\zero}}(0); \ve, k_{n^\zero}) \ge E^{-}(0, \La^{(s)}_{k_{n^\zero}}(0); \ve, k_{n^\zero})$.

$(2)$ Provided $\min_{\pm} |E - E^\pm \bigl(0, \La^{(s)}_{k_{n^\zero}}(0); \ve,k_{n^\zero} \bigr)| < 2 \delta_0^{(s-1)}$, the matrix $(E - H_{\La^{(s)}_{k_{n^\zero}}(0) \setminus \{0, n^\zero\},\ve,k_{n^\zero}})$ is invertible. Moreover,
\begin{equation}\label{eq:10Hinvestimatestatement1PQ}
\begin{split}
|[(E - H_{\La^{(s)}_{k_{n^\zero}}(0) \setminus \{0,n^\zero\},\ve,k_{n^\zero}})^{-1}](m,n)| \\
\le \begin{cases} |\ve|^{1/2} \exp(-\frac{31}{32} \kappa_0 |m-n| + 8 \kappa_0 \log \delta_0^{-1} (\min (\mu^{(s)}(m),\mu^{(s+q)}(n) )^{1/5}) & \text{if $m \neq n$ and $|m-n| > 2 |n^\zero|$}, \\
2 \exp (8 \kappa_0 \log \delta_0^{-1}( \mu^{(s)}(m))^{1/5}) & \text{if $m = n$}, \end{cases}
\end{split}
\end{equation}
$\mu^{(s)}(m) := \dist(m,\zv \setminus \La^{(s)}_{k_{n^\zero}}(0))$. Finally, if
$$
E\in\bigl( E^- \bigl(0, \La^{(s)}_{k_{n^\zero}}(0); \ve,k_{n^\zero} \bigr)+\delta, E^+\bigl(0, \La^{(s)}_{k_{n^\zero}}(0); \ve,k_{n^\zero} \bigr)-\delta\bigr),
$$
$\delta>0$, then
\begin{equation}\label{eq:10Hinvestimatestatement1PQrep}
\begin{split}
|[(E - H_{\La^{(s)}_{k_{n^\zero}}(0),\ve,k_{n^\zero}})^{-1}](m,n)|\le \begin{cases} \exp(-\frac{1}{2} \kappa_0 |m-n|) & \text{if $|m-n| > 8\max (|n^\zero|, \log \delta^{-1})$}, \\ \delta^{-1} & \text{for any $m,n$.} \end{cases}
\end{split}
\end{equation}

$(3)$ $E = E^\pm(0, \La^{(s)}_{k_{n^\zero}}(0); \ve, k_{n^\zero})$ obeys the following equation
\begin{equation}\label{eq:10Eequation0}
E - v(0, k_{n^\zero}) - Q^{(s)}(0,\La^{(s)}_{k_{n^\zero}}(0); \ve, E)  \mp \big| G^{(s)}(0,n^\zero,\La^{(s)}_{k_{n^\zero}}(0); \ve, E) \big| = 0,
\end{equation}
where
\begin{equation} \label{eq:10-10acbasicfunctions}
\begin{split}
Q^{(s)}(0,\La^{(s)}_{k_{n^\zero}}(0); \ve, E) \\
= \sum_{m',n' \in \La^{(s)}_{k_{n^\zero}}(0)\setminus\{0,n^\zero\}} h(0, m'; \ve, k_{n^\zero}) [(E - H_{\La^{(s)}_{k_{n^\zero}}(0) \setminus \{0,n^\zero\},\ve,k_{n^\zero}})^{-1}](m',n') h(n',0; \ve, k_{n_0} ), \\
G^{(s)}(0,n^\zero,\La^{(s)}_{k_{n^\zero}}(0); \ve, E) = h(0, n^\zero; \ve, k_{n^\zero} ) \\
+ \sum_{m', n' \in \La^{(s)}_{k_{n^\zero}}(0) \setminus \{0,n^\zero\}} h(0, m'; \ve, k_{n^\zero}) [(E - H_{\La^{(s)}_{k_{n^\zero}}(0) \setminus \{0,n^\zero\},\ve,k_{n^\zero}})^{-1}](m',n') h(n', n^\zero; \ve, k_{n^\zero}).
\end{split}
\end{equation}

$(V)$  If
$$
E \in \bigl( E \bigl(0, \La^{(s)}_{0}(0); \ve,0 \bigr)-\ve_0^{1/2}/2, E \bigl(0, \La^{(s)}_{0}(0); \ve,0 \bigr)-\delta\bigr),
$$
$0 < \delta < \ve_0^{1/2}/2$, then
\begin{equation}\label{eq:10Hinvestimatestatkzero}
|[(E - H_{\La^{(s)}_{0}(0),\ve,0})^{-1}](m,n)|\le \begin{cases} \exp(-\frac{1}{2} \kappa_0 |m-n|) & \text{if $|m-n| > 8\max (|n^\zero|, \log \delta^{-1})$}, \\ \delta^{-1} & \text{for any $m,n$,} \end{cases}
\end{equation}
see Proposition~\ref{prop:A.3E0}.
\end{thmd}

\section{Proof of the Main Theorems}\label{sec.11}

Consider the Schr\"{o}dinger operator
\begin{equation} \label{eq:17-1}
\big[ H y \big](x) := -y''(x) + V(x) y(x), \quad x \in \IR^1,
\end{equation}
where $V(x)$ is a quasi-periodic function,
\begin{align}
V(x) & = \sum_{n \in \zv \setminus \{ 0 \}} c(n) e^{2 \pi i n \omega x}\ , \quad x \in \IR^1, \label{eq:17-2} \\[6pt]
\omega & = \omap \in \IR^\nu \ ,\quad \nu \ge 2, \label{eq:17-3}
\end{align}
with
\begin{equation}\label{eq:17-4}
\begin{split}
\overline{c(n)} & = c(-n), \quad n \in \zv \setminus \{ 0 \}, \\
|c(n)| & \le \ve \exp(-\kappa_0|n|), \quad n \in \zv \setminus \{ 0 \},
\end{split}
\end{equation}
where $\ve, \kappa_0 > 0$.

We denote by $\widehat{f}(k)$ the Fourier transform of a function $f(x)$,
\begin{equation}\label{eq:11Fourier}
\widehat{f}(k) := \int_\mathbb{R} e^{-2 \pi i k x} f(x) \, dx,
\end{equation}
$x, k \in \IR$. Let $\cS(\IR)$ be the space of Schwartz functions $f(x)$, $x \in \IR $. Let $g(k)$ be a measurable function that, for any $a > 0$, decays faster than $|k|^{-a}$ as $|k| \rightarrow \infty$. Let $\psi = \check{g}$ be its inverse Fourier transform. Then $\psi$ belongs to the domain of $H$ and the following identity holds:
\begin{equation} \label{eq:17-1FtransH}
\widehat{H\psi}(k) =  (2 \pi)^2k^2 \widehat{\psi}(k) +\sum_{m \in \IZ^\nu \setminus \{0\}} c(-m) \widehat{\psi} (k+m\omega).
\end{equation}
In particular, this identity holds for any $f \in \cS(\IR)$.  Set $H_{ k} = \bigl(h(m, n; k)\bigr)_{m, n \in \zv}$, where
\begin{equation} \label{eq:12-a7}
\begin{split}
h(n, m; k) & = (2\pi )^2 (n\omega + k)^2,\quad \text{if}\ m = n, \\
h(n, m;k) & = c(n-m), \quad \text{if}\ m \not= n, \\
\end{split}
\end{equation}
Clearly, for each $k$, the matrix $H_{k}$ defines a self-adjoint operator in $\ell^2(\zv)$. Due to \eqref{eq:12-a7}, one has for any $m, n, \ell \in \zv$,
\begin{equation} \label{eq:12cocyclic}
H_{ k+\ell\omega}(m,n) = H_{k}(m+\ell,n+\ell).
\end{equation}

Let $k > 0$ be arbitrary. If $k \ge 3/4$, pick an arbitrary $\gamma \ge 1$ such that $\gamma - 1/4 \le |k| \le \gamma-1/2$. If $0 < k < 3/4$, set $\gamma = 1$. For $k < 0$, we pick the same $\gamma$ as for $|k|$. Define $\tilde {H}_{k} = \bigl(\tilde {h}(m, n;k)\bigr)_{m, n \in \zv}$ similarly to \eqref{eq:7-5-7} from Section~\ref{sec.7}, that is, set $\lambda = 256 \gamma$ and
\begin{equation} \label{eq:11-5-7}
\begin{split}
v(n; k) & = \lambda^{-1}(n \omega + k)^2\ , \quad n \in \zv\ ,\\
\tilde {h}(n, m;k) & = v(n; k)\ \text{if}\ m = n, \\
\tilde {h}(n, m;k) & =  \lambda^{-1}(2\pi )^{-2}c(n - m),\ \text{if}\ m \not= n.
\end{split}
\end{equation}
Define also  $\tilde {H}_{\epsilon,k} = \bigl(\tilde {h}(m, n;\epsilon, k)\bigr)_{m, n \in \zv}$ with
$\tilde {h}(n, n;\epsilon,k)=\tilde {h}(n, n;k)$, $\tilde {h}(n, m;\epsilon,k)=\epsilon \tilde {h}(n, m;k)$, if $m\neq n$.

\begin{proof}[Proof of  Theorem C]
Using the notation from Theorem D, let $k \in \mathcal{G} \setminus \frac{\omega}{2} \mathbb{Z}^\nu$. Note first of all that due to \eqref{eq:A.1A}, the set $\mathfrak{G}$ in Theorem C obeys $\mathfrak{G} \subset \mathcal{G}$. This is because the intervals $\mathcal{I}_n$
in Definition~\ref{def:10.finitereset} are smaller than the intervals $\mathfrak{J}_n$; see \eqref{eq:1K.1}, \eqref{eq:A.1A}, \eqref{eq:diphnores}. By Theorem~D, there exists $\ve_0 = \ve_0(\kappa_0,\omega)$ such that if $|\epsilon| < 2$ and $|\ve| := \lambda^{-1} \epsilon < \ve_0$, then for each $s$, there exists $\La^{(s)}_k(0)$ such that $\tilde {H}_{\La^{(s)}_k(0),\epsilon,k}$ belongs to one of the classes introduced in Sections~\ref{sec.3}, \ref{sec.5}, \ref{sec.6} with $0$ being either the principal point or one of the two principal points. Moreover, $\La^{(s-1)}_k(0)$ serves the role of the $(s-1)$-subset in the corresponding definition in Sections~\ref{sec.3}, \ref{sec.5}, \ref{sec.6}. Assume for instance $\tilde H_{\La^{(s)}_k(0),\epsilon,k} \in GSR^{[\mathfrak{s}^{(\ell(k))}(k),s]} \bigl( \mathfrak{m}^{(\ell(k))}(k),m^+(k),m^-(k), \La^{(s^{(\ell(k))}(k)+q)}_k(0); \delta_0, \mathfrak{t}^{(\ell(k))}(k) \bigr)$, with $s = s^{(\ell(k))}(k)+q$, $q = 1,\dots$, $m^+(k) = 0$; see the notation in Theorem~D. Let $\tilde {E}(0,\La^{(s)}_k(0);\epsilon,k)$ be the eigenvalue from part $(III)$ of Theorem D. Set $\tilde {E}(0,\La^{(s)}_k(0);k) = \tilde{E} (0, \La^{(s)}_k(0); 1, k)$. Now we invoke Theorem~\ref{th:6-4FIN} from Section~\ref{sec.6}. Recall that due to part $(5)$ of Theorem~\ref{th:6-4FIN}, one has
\begin{equation} \label{eq.11Eestimates1APN1FIN}
|\tilde{E} (0,\La^{(s)}_k(0);k) - \tilde{E} (0,\La^{(s-1)}_k(0);k)| < 4 \ve (\delta^{(s-1)}_0)^{1/8}.
\end{equation}
Therefore the limit
\begin{equation} \label{eq.11Eestimates1APN1FIN1}
\tilde {E}(k) = \lim_{s \rightarrow \infty} \tilde {E}(0,\La^{(s)}_k(0);k)
\end{equation}
exists. Furthermore, using the notation of part $(7)$ of Theorem~\ref{th:6-4FIN}, denote by $\vp^{(+)}(\La^{(s)}_k;k) := \vp^{(+)} (\cdot, \La^{(s)}_k; k)$ the eigenvector corresponding to $\tilde {E}(0,\La^{(s)}_k(0);k)$ and normalized by $\vp^{(+)}(0,\La^{(s)}_k;k)) = 1$. Due to part $(7)$ of Theorem~\ref{th:6-4FIN}, one has
\begin{equation} \label{eq:11-17evdecay}
\begin{split}
|\vp^{(+)}(n,\La^{(s)}_k;k)| \le |\ve|^{1/2} \sum_{m \in \mathfrak{m}^{(\ell(k))}(k)} \exp \Big( -\frac{7}{8} \kappa_0 |n-m| \Big), \quad \text{ $n \notin \mathfrak{m}^{(\ell(k))}(k)$}, \\
|\vp^{(+)}(m,\La^{(s)}_k;k)| \le 1 + \sum_{0 \le t < s} 4^{-t} \quad \text{for any $m \in \mathfrak{m}^{(\ell(k))}(k)$},
\end{split}
\end{equation}
\begin{equation} \label{eq:11-21ACCUPSFINCON}
|\vp^{(+)}(n,\La^{(s)}_k;k) -\vp^{(+)}(n,\La^{(s-1)}_k;k)| \le 2 |\ve| (\delta_0^{(s-1)})^5.
\end{equation}
It follows from \eqref{eq:11-17evdecay} and \eqref{eq:11-21ACCUPSFINCON} that for each $n \in \mathbb{Z}^\nu$, the limit
\begin{equation} \label{eq:11philimfin}
\vp(n;k) = \lim_{s \rightarrow \infty} \vp^{(+)}(n,\La^{(s)}_k;k)
\end{equation}
exists and obeys $\vp(0 ;k) = 1$,
\begin{equation} \label{eq:11-17evdecay1}
\begin{split}
|\vp(n ;k)| \le \ve^{1/2} \sum_{m \in \mathfrak{m}^{(\ell)}} \exp \Big( -\frac{7}{8} \kappa_0 |n-m| \Big), \quad \text{ $n \notin \mathfrak{m}^{(\ell(k))}(k)$}, \\
|\vp(m;k)| \le 2 \quad \text{for any $m \in \mathfrak{m}^{(\ell(k))}(k)$}.
\end{split}
\end{equation}
It follows also from \eqref{eq:11-17evdecay} and \eqref{eq:11-21ACCUPSFINCON} that
\begin{equation} \label{eq:11philimH-2}
\tilde {H}_{k} \vp(k) = \tilde {E}(k) \vp(k).
\end{equation}
Note that $H_{k} = \lambda (2\pi )^2 \tilde {H}_{k}$. This implies
\begin{equation} \label{eq:11philimH-3}
H_{k} \vp(k) = E(k) \vp(k).
\end{equation}
with $E(k) = \lambda (2\pi )^2 \tilde {E}(k)$. This finishes the proof of part $(1)$ of Theorem~C.

$(2)$ It follows from \eqref{eq:10EsymmetryT} and \eqref{eq:10Ekk1EGT} in Theorem~D that
\begin{equation}\label{eq:11EsymmetryT}
E(k) = E(-k),
\end{equation}

\begin{equation}\label{eq:11Ekk1EGT}
\begin{split}
(k^\zero)^2 (k - k_1)^2  - 10 |\ve| \sum_{\delta^{(s')}_0 < \min (k-k_1,k)} (\delta^{(s')}_0)^{4} < E(k) - E(k_1) \\
< \frac{2k}{\lambda} (k - k_1) + 2|\ve|\sum_{k_1 < k_{n} < k}(\delta^{(s(n)-1)}_0)^{1/8} , \quad 0 < k_1 < k,\quad \gamma-1\le k_1 \le  \gamma.
\end{split}
\end{equation}
where $s(n)$ is defined via $12 R^{(s(n)-1)} < |n| \le 12 R^{(s(n))}$, $k^\zero := \min(\ve_0^{3/4}, k_{n^\zero}/512)$, and $\gamma$ is the same as in the definition \eqref{eq:7-5-7}. Note that the quantity $\delta(n)$ in \eqref{eq:1Ekk1EGT} of Theorem~C obeys $\delta(n) > 2 (\delta^{(s(n) - 1)}_0)^{1/8}$. It follows from the first inequality in \eqref{eq:11Ekk1EGT} that
\begin{equation}\label{eq:11Ekk1EGT1}
E(k) - E(k_1) > \frac{(k^\zero)^2 (k - k_1)^2}{2}.
\end{equation}
Thus, \eqref{eq:1Ekk1EGT} in Theorem~C follows from \eqref{eq:11Ekk1EGT}. Finally, due to Lemma~\ref{lem:basicshiftprop}, one has $\vp^{(\pm)}(n,\La^{(s)}_{-k};-k) = \overline{\vp^{(\pm)}(-n,\La^{(s)}_k;k)}$. This implies $\vp(n;-k) = \overline{\vp(-n;k)}$, as claimed. This finishes the proof of part $(2)$.

$(3)$ We apply Theorem D. Let $n^{(0)} \in \zv \setminus \{0\}$ and $s > s^{(\ell(k_{n^\zero}))}$. Assume for instance that $k_{n^\zero} > 0$. Using \eqref{eq:10Ekk1EGT}, one has for $0 < \theta < \delta^{(s-1)}_0$,
\begin{equation}\label{eq:11kk1comp1lim}
|E^\pm(0, \La^{(s)}_{k_{n^\zero}}(0);k_{n^\zero}) - E(0, \La^{(s)}_{k_{n^\zero}}(0);k_{n^\zero}\pm \theta)| < 2 (|k_{n^\zero}|+1) \theta + 2 |\ve| (\delta^{(s)}_0)^5,
\end{equation}
since the sum on the right-hand side of \eqref{eq:10Ekk1EGT} is over the empty set. Due to \eqref{eq:10kk1comp1limapp},
\begin{equation}\label{eq:11kk1comp1limapp}
|E^\pm(0, \La^{(s)}_{k_{n^\zero}}(0);k_{n^\zero}) - E^\pm(0, \La^{(s-1)}_{k_{n^\zero}}(0);k_{n^\zero})| \le \ve \delta^{(s-1)}_0.
\end{equation}

Therefore the limit
\begin{equation} \label{eq.11Eestimates1APN1FIN1-2}
E^\pm(k_{n^\zero}) = \lim_{s \to \infty} E^\pm(0, \La^{(s)}_{k_{n^\zero}}(0);k_{n^\zero})
\end{equation}
exists,
\begin{equation}\label{eq:11kk1comp1limapp11}
|E^\pm(k_{n^\zero}) - E^\pm(0, \La^{(s-1)}_{k_{n^\zero}}(0);k_{n^\zero})| \le 2 \ve\delta^{(s-1)}_0.
\end{equation}
Due to \eqref{eq:10Ekk1EGT}, one obtains also
\begin{equation}\label{eq:11kk1comp1limapp-2}
|E^\pm(k_{n^\zero}) - E(k_{n^\zero}\pm\theta)| \\
\le 2(k_{n_0}+1) \theta + \sum_{n : \text{$k_{n}$ is between $k_{n^\zero}$ and $k_{n^\zero} \pm \theta$}} 2\ve (\delta^{(s(n)-1)}_0)^{1/8}.
\end{equation}

Assume now that $E^+(0;k_{n^\zero}) > E^-(0;k_{n^\zero}) > 0$. Let $E^+(0;k_{n^\zero}) > E > E^-(0;k_{n^\zero}) > 0$. Let $s > s^{(\ell(k_{n^\zero}))}$ be large enough so that $\sigma(E) := \min (E^+(0;k_{n^\zero}) - E, E - E^-(0;k_{n^\zero}) > \delta^{(s)}_0$. Then, due to \eqref{eq:11kk1comp1limapp11}, one has $E^+(0, \La^{(s)}_{k_{n^\zero}}(0);k_{n^\zero}) - E, E - E^-(0, \La^{(s)}_{k_{n^\zero}}(0);k_{n^\zero}) > \sigma(E) - \rho_s$, where $\rho_s \to 0$ as $s \to \infty$. Due to \eqref{eq:10specHEEAAA} in Theorem~D, the matrix $(E - H_{\La^{(s)}_{k_{n^\zero}}(0), \ve, k_{n^\zero}})$ is invertible, moreover $\|(E - H_{\La^{(s)}_{k_{n^\zero}}(0), \ve, k_{n^\zero}})^{-1}\|\le 2\sigma(E)^{-1}$, provided $\sigma(E)/2> \rho_s$. Since $B(0,R^\es)\subset \La^{(s)}_{k_{n^\zero}}(0)$ and $R^\es\rightarrow +\infty$ with $s\rightarrow +\infty$, one has
$$
\|[(E - H_{\La^{(s)}_{k_{n^\zero}}(0), \ve, k_{n^\zero}})-(E - H_{\ve, k_{n^\zero}})]f\|\rightarrow 0
$$
for any $f$ supported on a finite subset of $\mathbb{Z}^\nu$. Due to part $(1)$ of Lemma~\ref{lem:13.1A}, $(E - H_{\ve, k_{n^\zero}})$ is invertible. Due to part $(3)$ of Lemma~\ref{lem:13.1A}, $(E - H_{\ve, k})$ is invertible for any $k$ as claimed in part $(3)$ of Theorem~C.
\end{proof}

\begin{lemma}\label{lem:13.1A}
$(1)$ Let $A$, $A_s$, $s=1,\dots$ be self-adjoint operators acting in the Hilbert space $\mathcal{L}$, $\mathcal{L}_s$ respectively, $\mathcal{L} \supset \mathcal{L}_s$. Let $\mathcal{D}_{A}$, $\mathcal{D}_{A_s}$ be the domains of the operators $A$ and $A_s$, respectively. Assume that $(a)$ each $A_s$ is invertible, and moreover $B := \sup_s \|A_s^{-1}\| < \infty$, $(b)$ there exists a dense set $\mathcal{D} \subset \mathcal{D}_{A}$ such that for any $f \in \mathcal{D}$, there exists $s_f$ such that $f \in \mathcal{D}_{A_s}$ for $s \ge s_f$ and $\|(A - A_s)f\| \rightarrow 0$ as $s \rightarrow \infty$. Then $A$ is invertible, and $\|A^{-1}\|\le B$.

$(2)$ Using the notations of $(1)$, assume in addition that the following conditions hold: $(c)$ the set $\mathcal{D}$ contains an orthonormal basis $\{g_n\}_{n \in \mathbb{N}}$ of the space $\mathcal{L}$, $(d)$ $\sup_s |\la A_s^{-1} g_m, g_n \ra| \le \rho(m,n)$ with $S^2 := \sup_m \sum_{n} \rho(m,n)^2 < \infty$. Then $|\la A^{-1} g_m, g_n\ra|\le \rho(m,n)$ for any $m,n$.

$(3)$ Assume that for some $k_0$, $E \in \IR$ the operator $(E - H_{k_0})$ is invertible. Then $(E - H_{k})$ is invertible for every $k$.
\end{lemma}

\begin{proof}
(1) One has $\|A_s f\| \ge B^{-1}\|f\|$ for any $f \in \mathcal{D}_{A_s}$. This implies $\|A f\| \ge B^{-1}\|f\|$ for any $f \in \mathcal{D}$. Since $\mathcal{D} \subset \mathcal{D}_{A}$ is dense, the statement follows.

$(2)$ Recall that the set $\{x = (x_n) \in \ell^2(\mathbb{N}) : |x_n| \le \epsilon(n)\}$ is $\|\cdot\|$-compact, provided $\sum_n \epsilon(n)^2 < \infty$. With $m$ being fixed, consider the sequence $A_s^{-1}g_m$, $s = 1, 2, \ldots$. Therefore, it follows from the condition $(d)$ in $(2)$ that this sequence
has a $\|\cdot\|$-convergent subsequence. Using a standard diagonalization argument, one concludes that there exists a subsequence $s_j$ such that $h_m := \lim_{j \rightarrow \infty} A_{s_j}^{-1}g_m$ exists in the $\|\cdot\|$-sense for every $m$. Let $m,n$, $\varepsilon > 0$ be arbitrary. Find $j_0$ such that $\|h_m - A_{s_{j_0}}^{-1}g_m\| < \varepsilon$ and $\|A g_n - A_{s_{j_0}}g_n\| < \varepsilon$. Then one has
\begin{equation} \label{eq:11inverseapproxi}
\begin{split}
|\la A h_m,g_n\ra-\la g_m,g_n\ra|=|\la h_m,A g_n\ra-\la g_m,g_n\ra|\le |\la A_{s_{j_0}}^{-1}g_m,A_{s_{j_0}}g_n\ra -\la g_m, g_n \ra| \\
+ \|h_m-A_{s_{j_0}}^{-1}g_m\|\|A g_n\| + \|A_{s_{j_0}}^{-1}g_m\| \|A g_n-A_{s_{j_0}}g_n\| \le \varepsilon \|A g_n\| + S\varepsilon.
\end{split}
\end{equation}
Hence, $A h_m = g_m$, that is, $h_m = A^{-1}g_m$. Due to condition $(d)$ one has
\begin{equation} \label{eq:11inverseapproxi1}
|\la A^{-1}g_m,g_n\ra|=|\la h_m,g_n\ra|=|\lim_{j\rightarrow \infty} \la A_{s_j}^{-1}g_m,g_n\ra |\le \rho(m,n),
\end{equation}
as claimed.

$(3)$  Recall that
\begin{equation} \label{eq:12cocyclicProof}
H_{k_0 + \ell \omega}(m,n) = H_{k_0}(m + \ell,n + \ell)
\end{equation}
for any $\ell$. Given $t \in \mathbb{Z}^\nu$ and $f(\cdot)\in \ell^2(\mathbb{Z}^\nu)$, set $U_t f(n) := f(n-t)$, $n \in \mathbb{Z}^\nu$. Clearly, $U_t$ is a unitary operator. Furthermore, $U_t (a(m,n))_{m,n \in \mathbb{Z}^\nu} U_t^{-1} = (a(m+t,n+t))_{m,n \in \mathbb{Z}^\nu}$ for any self-adjoint operator $A = (a(m,n))_{m,n \in \mathbb{Z}^\nu}$ whose domain contains the standard basis vectors $e_n$, $n \in \mathbb{Z}^\nu$. Combining this with \eqref{eq:12cocyclicProof} one concludes that $H_{k_0 + \ell\omega}$ is unitarily conjugated to $H_{ k_0}$. In particular,  $\|(E - H_{k_0 + \ell\omega})^{-1} = \|(E - H_{k_0})^{-1}\|$ for any $\ell$.
Given $k$, there exists a sequence $\ell_s$ such that $(k_0 + \ell_s) \omega \rightarrow k$. Then $\|[(E - H_{k_0 + \ell_s \omega}) - (E - H_{k})]f\| \rightarrow 0$ for any $f$ supported on a finite subset of $\mathbb{Z}^\nu$. Therefore the statement follows from part $(1)$.
\end{proof}

To prove Theorem~A we need the following lemma.

\begin{lemma}\label{lem:13.1}
$(1)$ Assume that for some $E \in \IR$, there exist $\gamma(E) > 0$, $B(E) < \infty$ such that for any $k$, $x, y \in \zv$, we have
\begin{equation} \label{eq:13-7}
|(E - H_{k})^{-1}(x,y)| \le B(E) \exp(-\gamma(E) |x-y|).
\end{equation}
Then, $(E - H)$ is invertible.

$(2)$ Let $n^\zero \in \zv \setminus \{0\}$. Assume, for instance, $k_{n^\zero} > 0$. Let
$$
E^\pm(k_{n^\zero}) = \lim_{k \to k_{n^\zero} \pm 0, \; k \in \mathfrak{G} \setminus \mathcal{K}(\omega)} E(k), \quad \text{ for $k_m > 0$,}
$$
as in Theorem~A. Assume $E^-(k_{n^\zero})<E^\pm(k_{n^\zero})$. Let $E \in (E^-(k_{n^\zero}) + \delta,
E^+(k_{n^\zero}) - \delta)$, $\delta > 0$ arbitrary. Then, for every $k$, we have
\begin{equation}\label{eq:11Hinvestimatestatement1PQreprep2}
|[(E - H_{k})^{-1}](m,n)|\le \begin{cases} \exp(-\frac{1}{2} \kappa_0 |m-n|) &  \text{if $|m-n| > 8\max (|n^\zero|, \log \delta^{-1})$}, \\
\delta^{-1} & \text{for any $m,n$.}
\end{cases}
\end{equation}

$(3)$ For every $E \in (E(0) -\ve_0^{1/2}/2, E(0))$ and every $k$, we have
\begin{equation}\label{eq:11Hinvestimatestatementkzero}
|[(E - H_{k})^{-1}](m,n)|\le \begin{cases} \exp(-\frac{1}{2} \kappa_0 |m-n|) &  \text{if $|m-n| > 8\max (|n^\zero|, \log \delta^{-1})$}, \\
\delta^{-1} & \text{for any $m,n$.}
\end{cases}
\end{equation}
\end{lemma}

\begin{proof}
$(1)$ For any $k$, we have
\begin{equation} \label{eq:13-a}
\begin{split}
\sum_{x \in \zv} |(E - H_{k})(m,x)| |(E - H_{k})^{-1}(x,n)| \\
\lesssim (k + m\omega)^2 B(E) \exp(-\gamma(E) |m-n|) + 4^\nu \ve_0 B(E) (\gamma_1(E))^{-\nu} \exp(-\gamma_1(E) |m-n|), \quad m \neq n, \\
\sum_{x \in \zv} |(E - H_{k})(m,x)| |(E - H_{k})^{-1}(x,m)| \\
\lesssim (k + m\omega)^2 B(E) + 4^\nu  B(E) (\gamma_1(E))^{-\nu} \exp(-\gamma_1(E) |m-n|), \\
\sum_{x \in \zv} (E - H_{k})(m,x) (E - H_{k})^{-1}(x,n) = \delta_{m,n},
\end{split}
\end{equation}
where $\delta_{m,n}$ is the Kronecker symbol, $\gamma_1(E) = \frac{1}{2} \min (\kappa_0, \gamma(E))$. In particular, for any $k$ and for any bounded $\psi : \zv \to \IC$, we have
\begin{equation} \label{eq:13-b}
\sum_{x \in \zv} (E - H_{k})(m,x) (E - H_{k})^{-1}(x,n) \psi(n) = \psi(m), \quad m \in \zv,
\end{equation}
and the series converges absolutely.

Let $f \in \cS(\IR^1)$ be arbitrary. Set
\begin{equation} \label{eq:13-8}
g(k) = \sum_{n \in \zv} (E - H_{k})^{-1}(0,n) \widehat{f} (k + n\omega).
\end{equation}
Note that due to the identity \eqref{eq:12cocyclic}, one has for any $k$ and any $m, n, \ell \in \zv$,
\begin{equation} \label{eq:12cocyclicinv}
(E - H_{k + \ell \omega})^{-1}(m,n) = (E - H_{k})^{-1}(m + \ell, n + \ell).
\end{equation}
Using \eqref{eq:12cocyclicinv}, one obtains for any $k$ and any $m \in \zv$,
\begin{equation} \label{eq:13-8b}
\begin{split}
g(k + m\omega) = \sum_{n \in \zv} (E - H_{k + m\omega})^{-1}(0,n) \widehat{f} (k + m\omega + n\omega) \\
= \sum_{n \in \zv} (E - H_{ k})^{-1}(m,n+m) \widehat{f} (k + m\omega + n\omega) = \sum_{n \in \zv} (E - H_{k})^{-1}(m,n) \widehat{f} (k + n\omega).
\end{split}
\end{equation}
Combining this with \eqref{eq:13-b}, one obtains for any $k$,
\begin{equation} \label{eq:13-8b-2}
\sum_{m \in \zv} (E - H_{k})(0,m) g(k + m\omega) = \widehat{f}(k).
\end{equation}
It follows from the definition \eqref{eq:13-8} that for any $h \in L^2(\IR)$, one has
\begin{equation} \label{eq:13-9}
\Big| \int_\IR g(k) h(k) \, dk \Big| \lesssim 2^\nu B(E) (\gamma(E))^{-\nu} \| \widehat{f} \|_{2} \| h \|_2.
\end{equation}
Hence, $\| g \|_2 \lesssim 2^\nu B(E) (\gamma(E))^{-\nu} \| \widehat{f} \|_{2}$. In particular, there exists the inverse Fourier transform $\psi := \check{g} \in L^2(\IR)$, $\| \psi \|_2 = \| g \|_2 \le M(E) \|f\|_2$, where $M(E)$ is a constant. Furthermore, since $\widehat{f} \in \cS(\IR)$, one obtains using condition \eqref{eq:13-7}, $\lim_{|k| \to \infty} |k|^a |g(k)| = 0$ for any $a > 0$. Therefore, \eqref{eq:17-1FtransH} holds. Combining \eqref{eq:17-1FtransH} with \eqref{eq:13-8b-2}, one obtains
\begin{equation} \label{eq:17-1FtransHa}
\widehat{[(E - H) \psi]}(k) = \widehat{f}(k).
\end{equation}
So,
\begin{equation} \label{eq:17-1FtransHaIn}
(E - H) \psi = f, \quad \| \psi \|_2 \le M(E) \| f \|_2.
\end{equation}
Since $f \in \cS(\IR)$ is arbitrary, $(E - H)$ is invertible.

$(2)$ Let
$$
E \in \bigl( E^- \bigl(0, \La^{(s)}_{k_{n^\zero}}(0); \ve,k_{n^\zero} \bigr)+\delta, E^+\bigl(0, \La^{(s)}_{k_{n^\zero}}(0); \ve,k_{n^\zero} \bigr)-\delta\bigr),
$$
$\delta > 0$. Then due to \eqref{eq:10Hinvestimatestatement1PQrep} from part $(IV)$ of Theorem~D, one has
\begin{equation}\label{eq:11Hinvestimatestatement1PQreprep}
|[(E - H_{\La^{(s)}_{k_{n^\zero}}(0),\ve,k_{n^\zero}})^{-1}](m,n)| \le \begin{cases} \exp(-\frac{1}{2} \kappa_0 |m-n|) & \text{if $|m-n| > 8\max (|n^\zero|, \log \delta^{-1})$}, \\ \delta^{-1} & \text{for any $m,n$.} \end{cases}
\end{equation}
Since $\|(E - H_{\La^{(s)}_{k_{n^\zero}}(0),\ve,k_{n^\zero}})f - (E - H_{k_{n^\zero}})f\| \rightarrow 0$ as $s \rightarrow \infty$, for any $f$ supported on a finite subset of $\mathbb{Z}^\nu$, part $(2)$ of Lemma~\ref{lem:13.1A} applies. Thus,
\begin{equation}\label{eq:11Hinvestimatestatement1PQreprep5}
|[(E - H_{k_{n^\zero}})^{-1}](m,n)| \le \begin{cases} \exp(-\frac{1}{2} \kappa_0 |m-n|) & \text{if $|m-n| > 8\max (|n^\zero|, \log \delta^{-1})$}, \\ \delta^{-1} & \text{for any $m,n$.} \end{cases}
\end{equation}
Let $U_t f(n) := f(n-t)$, $n \in \mathbb{Z}^\nu$ be the unitary operator defined in the proof of Lemma~\ref{lem:13.1A}. Then, as we saw in the
proof of Lemma~\ref{lem:13.1}, $U_t(E- H_{k})U_t^{-1}=(E-H_{k+t\omega})$. This implies $U_t (E- H_{k})^{-1} U_t^{-1} = (E - H_{k+t\omega})^{-1}$. Hence,
\begin{equation}\label{eq:11Hinvestimatestatement1PQreprep1}
\begin{split}
|(E - H_{k_{n^\zero}+t\omega})(m,n)| = |(E - H_{k_{n^\zero}})^{-1}(m+t,n+t)|\\
\le \begin{cases} \exp(-\frac{1}{2} \kappa_0 |m-n|) & \text{if $|m-n| > 8\max (|n^\zero|, \log \delta^{-1})$}, \\ \delta^{-1} & \text{for any $m,n$.} \end{cases}
\end{split}
\end{equation}
Given $k$, there exists a sequence $\ell_s$ such that $(k_{n^\zero} + \ell_s) \omega \rightarrow k$. Then $\|[(E - H_{k_n^\zero + \ell_s \omega}) - (E - H_{k})]f\| \rightarrow 0$ for any $f$ supported on a finite subset of $\mathbb{Z}^\nu$. Therefore the statement follows from part $(2)$ of Lemma~\ref{lem:13.1A}.

$(3)$ The proof is completely similar to the proof of $(2)$ with part $(V)$ of Theorem~D being invoked.
\end{proof}

\begin{proof}[Proof of Theorem A]
Given $k \in \IR$ and $\vp(n) : \zv \rightarrow \mathbb{C}$ such that $|\varphi(n)| \le C_\varphi |n|^{-\nu-1}$, where $C_\varphi$ is a constant, set
\begin{equation} \label{eq:11-5}
y_{\vp, k}(x) = \sum_{n \in \zv}\, \vp(n) e \bigl( (n\omega + k)x \bigr).
\end{equation}
The function $y_{\vp, k}(x)$ satisfies equation \eqref{eq:1-1} if and only if
\begin{equation} \label{eq:111-6}
(2 \pi)^2 (n\omega + k)^2 \vp(n) + \sum_{m \in \zv\setminus \{0\}}\, c(n - m) \vp(m) = E \vp(n)
\end{equation}
for any $n \in \zv$. Let $E(k)$ and $(\vp(n; k))_{n \in \zv}$ be as in Theorem~C. Then,
$$
\psi(k, x) = \sum_{n \in \zv} \vp(n; k) e((n\omega + k)x)
$$
obeys equation \eqref{eq:1-1} with $E = E(k)$, that is,
\begin{equation}\label{eq:11.sco}
H \psi \equiv - \psi''(k,x) + V(x) \psi(k,x) = E(k) \psi(k, x).
\end{equation}
Due to $(1)$ and $(2)$ from Theorem~C, conditions $(a)$--$(c)$ in Theorem~A hold. Due to part $(2)$ of Lemma~\ref{lem:13.1}, one has
$$
\spec H \cap (E^-(k_m), E^+( k_m)) = \emptyset \text{ if $E^-(k_m) < E^+(k_m)$} .
$$
Due to part $(2)$ in Lemma~\ref{lem:13.1} one has
$$
\spec H \cap (E(0) - \ve_0^{1/2}/2, E(0)) = \emptyset.
$$
It is well-known that
$$
\spec H \subset  [0, \infty) + \{ \ve V(x) : x \in \mathbb{R} \}.
$$
It is easy to see that  $|V(x)|\le (4\kappa_0^{-1})^\nu$. Hence, $|\ve| |V(x)|<\ve_0^{1/2}/4$ for any $x$ and any $|\ve|<\ve_0$, see the definition of $\ve_0$ in \eqref{eq:2epsilon0} from Definition~\ref{def:4-1}. Thus, one concludes that
$$
\spec H \subset  [E(0), \infty) \setminus \bigcup_{m \in \zv \setminus \{0\} : E^-(k_m) < E^+(k_m)} (E^-(k_m), E^+( k_m)).
$$

Recall the following well-known general fact in the spectral theory of Sturm-Liouville equations. If for some $E \in \mathbb{R}$, there exists a bounded smooth function which obeys equation \eqref{eq:1-1}, that is,
\begin{equation}\label{eq:11.sco-2}
- y'' + V(x) y(x) = E y(x),
\end{equation}
then $E \in \spec H$. For any $k \in \mathcal{G} \setminus \mathcal{K} (\omega)$, the function $\psi(k,x)$ is bounded. Hence, $E(k) \in \spec H$. It follows from \eqref{eq:11Ekk1EGT} that $E(k)$ is continuous at each point $k \in \mathcal{G} \setminus \mathcal{K} (\omega)$. It follows also from \eqref{eq:11Ekk1EGT} that $E(k)$ is monotone for $k \in \mathcal{G} \setminus \mathcal{K} (\omega)$, $k > 0$. Recall also that $E(-k) = E(k)$. Finally, due to
\eqref{eq:10Eksquaire}, $E(k) \to \infty$ when $k \to \infty$. One concludes that
$$
\overline{\{ E(k) : k \in \mathcal{G} \setminus \mathcal{K}(\omega) \}} = [E(0), \infty) \setminus \bigcup_{m \in \zv \setminus \{0\} : E^-(k_m) < E^+(k_m)} (E^-(k_m), E^+(k_m)).
$$

Hence,
$$
\spec H \supset [\uE, \infty) \setminus \bigcup_{m \in \zv \setminus \{0\} : E^-(k_m) < E^+(, k_m)} (E^-(k_m), E^+( k_m)).
$$
This finishes the proof of Theorem~A.
\end{proof}

\begin{proof}[Proof of part $(1)$ of Theorem~B]
Let $n^\zero \in \zv \setminus \{0\}$ be arbitrary. We assume that $k_{n^\zero} = - \frac{n^\zero \omega}{2} > 0$. The case $k_{n^\zero} = - \frac{n^\zero \omega}{2} < 0$ is similar. Due to \eqref{eq:10kk1comp1gapsize} of part $(4)$ of Theorem~D, one has
\begin{equation}\label{eq:11kk1comp1gapsize}
0 \le E^+(0, \La;k_{n^\zero}) - E^-(0, \La;k_{n^\zero}) \le 2\ve\exp \Big(-\frac{\kappa_0}{2} |n^\zero| \Big),
\end{equation}
where $\La = \La^{(s)}_{k_{n^\zero}}(0)$, $s = s^{(\ell(k_{n^\zero}))}(k_{n^\zero})$, $\ell = \ell(k_{n^\zero})$. It follows now from \eqref{eq.11Eestimates1APN1FIN1} that
\begin{equation}\label{eq:11kk1comp1gapsizelim}
0 \le E^+(k_{n^\zero}) - E^-(k_{n^\zero}) \le 2\ve\exp \Big(-\frac{\kappa_0}{2}|n^\zero| \Big),
\end{equation}
as claimed in part $(1)$ of Theorem~B.
\end{proof}

To prove part $(2)$ of Theorem~B, it is convenient to establish a few lemmas first.

\begin{lemma}\label{lem:11.gapinequalities}
Using the notation from the proof of part $(1)$ of Theorem~B, for any $n^\zero$, the Fourier coefficient $c(n^\zero)$ obeys the following estimate,
\begin{equation} \label{eq:11gapsineqstatement}
\begin{split}
|c(n^\zero)| \le \ve_0^{-1} \exp( \kappa_0 |n^\zero|) (E^+(\La;k_{n^\zero}) - E^-(\La;k_{n^\zero})) \\
+ \sum_{m', n' \in \La \setminus \{ 0,n^\zero \}} |c(m')| s_{D(\cdot;\La \setminus \{0,n^\zero\}),T,\kappa_0,\ve;\La \setminus \{0,n^\zero\}, \mathfrak{R}}(m',n') |c(n'- n^\zero)|.
\end{split}
\end{equation}
Here, as usual, $T = 4 \kappa_0 \log \delta_0^{-1}$.
\end{lemma}

\begin{proof}
We use equation \eqref{eq:10Eequation0} from part $(IV)$ of Theorem~D,
\begin{equation} \label{eq:11Eequation0}
[E - v(0, k_{n^\zero}) - Q^{(s)}(0,\La; \ve, k_{n^\zero}, E)  \mp \big| G^{(s)}(0,n^\zero,\La; \ve, k_{n^\zero}, E) \big|]|_{E = E^\pm(\La; \ve, k_{n^\zero}} = 0,
\end{equation}
where
\begin{equation} \label{eq:11-10acbasicfunctions}
\begin{split}
Q^{(s)}(0,\La; k_{n^\zero}+\theta, E) \\
= \sum_{m', n' \in \La \setminus \{0,n^\zero\}} h(0, m';k_{n^\zero}+\theta) [(E - H_{\La \setminus \{0, n^\zero\}, k_{n^\zero} + \theta})^{-1}](m',n') h(n',0;k_{n_0}+\theta), \\
G^{(s)}(0,n^\zero,\La;k_{n^\zero}+\theta, E) = h(0, n^\zero;k_{n^\zero}+\theta) \\
+ \sum_{m', n' \in \La \setminus \{0,n^\zero\}} h(0, m';k_{n^\zero}+\theta) [(E - H_{\La \setminus \{0,n^\zero\},k_{n^\zero}+\theta})^{-1}](m',n') h(n', n^\zero;k_{n^\zero}+\theta),
\end{split}
\end{equation}
$0 \le \theta \le (\delta^\es_0)^{3/4}$; see the notation in the proof of part $(1)$ above. Set
\begin{equation} \label{eq:11-10acbasicfunctions-2}
\begin{split}
a_1 = v(0,k_{n^\zero}+\theta) + Q^{(s)}(0,\La;k_{n^\zero}+\theta, E), \\
a_2 = v(0,k_{n^\zero}-\theta) + Q^{(s)}(0,\La;k_{n^\zero}-\theta, E), \\
b = G^{(s)}(0,n_0,\La;k_{n^\zero}+\theta, E),
\end{split}
\end{equation}
$f_i := E - a_i$, $i = 1,2$, $f = f_1 - |b|^2 f_2^{-1}$. Due to Proposition~\ref{prop:6-4INT}, $f \in \mathfrak{F}^{(\ell)}_{\mathfrak{g}^{(\ell - 1)},1} (f_1,f_2,b^2)$, provided $\theta > 0$. Now we invoke Lemma~\ref{4.fcontinuedfrac}. Due to part $(2)$ of that lemma, the functions $\mu^{(f_j)}$, $\chi^{(f_j)}$ are $C^2$-smooth, $|\partial^\alpha \mu^{(f_j)}|, |\partial^\alpha \chi^{(f_j)}| \le 1$, $|\alpha| \le 2$. It follows from \eqref{eq:11Eequation0} that
\begin{equation} \label{eq.11Eestimates1APFINALPM1}
[ f_1 (k_{n^\zero}, E) - |b (E)| ]|_{E=E^+ (\La;k_{n^\zero})} - [ f_1 (k_{n^\zero}, E) + |b (k_{n^\zero}, E)| ]|_{E=E^- ( \La;k_{n^\zero})} = 0.
\end{equation}
Hence,
\begin{equation} \label{eq.11Eestimates1APFINALB1}
|b(k_{n^\zero}, E)| |_{E = E^+(\La;k_{n^\zero})} \le |f_1 (k_{n^\zero}, E) |_{E = E^+(\La;k_{n^\zero})} - f_1 (k_{n^\zero}, E) |_{E = E^-(\La;k_{n^\zero})}|.
\end{equation}
Recall that $\chi^{(f_j)} = \mu^{(f_j)}f_j$. One has
\begin{equation} \label{eq.11Eestimates1APFINALB2}
\begin{split}
|\mu^{(f_1)}f_1 (k_{n^\zero}, E) |_{E = E^+(\La;k_{n^\zero})} - \mu^{(f_1)}f_1 (k_{n^\zero}, E) |_{E = E^-(\La; k_{n^\zero})}| \le E^+(\La;k_{n^\zero}) - E^-(\La;k_{n^\zero}), \\
|\mu^{(f_1)}f_1 (k_{n^\zero}, E) |_{E = E^+(\La;k_{n^\zero})} - \mu^{(f_1)}|_{E = E^+(\La;k_{n^\zero})} f_1(k_{n^\zero},E)
|_{E = E^-(\La;k_{n^\zero})}| \\
\le E^+(\La;k_{n^\zero}) - E^-(\La;k_{n^\zero}) + |\mu^{(f_1)}|_{E = E^+(\La;k_{n^\zero})} - \mu^{(f_1)}|_{E = E^-(\La; k_{n^\zero})}| \sup |f_1 (k_{n^\zero},E)| \\
\le 2 (E^+(\La;k_{n^\zero}) - E^-(\La;k_{n^\zero})).
\end{split}
\end{equation}
Recall also that due to part $(4)$ in Lemma~\ref{4.fcontinuedfrac}, one has $ |\mu^{(f_i)}| \ge 2^{-2^{\ell-1}+1} \tau^{(f_{i})}$. Let $n^{(\ell')}(k_{n^\zero})$, $\ell' = 0, \dots, \ell = \ell(k_{n^\zero})$ be as in Definition~\ref{def:10.finitereset}. Here $n^{(\ell)}(k_{n^\zero}) = n^\zero$, see Definition~\ref{def:10.finitereset} and Lemma~\ref{lem:10resetn0}. By Theorem~D, $\tau^{(f_{1})} = \tau^{(\ell)}(k_{n^\zero}) = |k_{n^{(\ell-1)}}| ||k_{n^\zero}|-|k_{n^{(\ell-1)}}||$. This implies $|\mu^{(f_i)}| > \ve_0 \exp(-\kappa_0 |n^\zero|)$. Therefore,
\begin{equation} \label{eq.11Eestimates1APFINALB4}
|b(k_{n^\zero},E)| |_{E = E^+(\La;k_{n^\zero})} \le \ve_0^{-1} \exp(\kappa_0 |n^\zero|) (E^+(\La;k_{n^\zero}) - E^-(\La; k_{n^\zero})).
\end{equation}
Recall that
\begin{equation} \label{eq:11-10acbasicfunctionsBineq}
b(k_{n^\zero},E) = c(n^\zero) + \sum_{m', n' \in \La \setminus \{0,n^\zero\}} c(m') [(E - H_{\La \setminus \{ 0, n^\zero \}, k_{n^\zero}})^{-1}] (m',n') c(n'- n^\zero)
\end{equation}
and also that
\begin{equation}\label{eq:11Hinvestimatestatement1PQ}
|[(E - H_{\La \setminus \{0,n^\zero\},k_{n^\zero}})^{-1}](m,n)| \le s_{D(\cdot; \La \setminus \{0,n^\zero\}), T, \kappa_0, \ve; \La \setminus \{0,n^\zero\}, \mathfrak{R}}(m,n).
\end{equation}
Therefore, \eqref{eq:11gapsineqstatement} follows from \eqref{eq.11Eestimates1APFINALB4}.
\end{proof}

It is convenient to introduce the following notation: $\La'(n^\zero) = \La \setminus \{0,n^\zero\}$, $\La(n^\zero) = \zv \setminus \{0,n^\zero\}$,
\begin{equation} \label{eq:11-10snotation}
s(n^\zero;m',n') = \begin{cases} s_{D(\cdot;\La(n^\zero)),T,\kappa_0,\ve;\La(n^\zero),\mathfrak{R}}(m',n'), \quad \text{if $m', n' \in \La'(n^\zero)$}, \\
0, \quad \text{if $m' \in \La(n^\zero) \setminus \La'(n^\zero)$, or $n' \in \La(n^\zero) \setminus \La'(n^\zero)$, or both.}
\end{cases}
\end{equation}
We re-write \eqref{eq:11gapsineqstatement} in the following form,
\begin{equation} \label{eq:11-10acbasicfunctionsBineq1}
\begin{split}
|c(n^\zero)| \le \ve_0^{-1} \exp (\kappa_0 |n^\zero|) (E^+(\La;k_{n^\zero}) - E^-(\La;k_{n^\zero})) \\
+ \sum_{m', n' \in \La(n^\zero)} |c(m')| s(n^\zero;m',n') |c(n'- n^\zero)|.
\end{split}
\end{equation}

In the next lemma we recall the main properties of the sum $s(n^\zero;m',n')$ from Section~\ref{sec.2}, stated in a form convenient for our goals.

\begin{lemma}\label{lem:11.sumsproperties}
Let $s(n^\zero;m',n')$ be as in \eqref{eq:11-10snotation}.

$(1)$
\begin{equation}\label{eq:11auxtrajectweightO}
\begin{split}
s(n^\zero;m,n) \le \sum_{\gamma \in \Ga_{n^\zero}(m, n)} w_{n^\zero} (\gamma), \\
w_{n^\zero} (\gamma) = \Big[ \prod w(n_j,n_{j+1}) \Big] \exp \Big( \sum_{1 \le j \le k} D_{n^\zero}(n_j) \Big).
\end{split}
\end{equation}
Here $w(m,n) := |c(n-m)|$, $\Ga_{n^\zero}(m, n)$ stands for a set of trajectories $\gamma = (n_1,\dots,n_k)$, $k := k(\gamma) \ge 1$, $n_j \in \La(n^\zero)$, $n_1 = m$, $n_k = n$, $n_{j+1} \neq n_j$, $D_{n^\zero}(x) > 0$, $x \in \zv \setminus \{0,n^\zero\}$. Moreover, the following conditions hold:

$(i)$ $D_{n^\zero}(x) \le T \mu_{n^\zero}(x)^{1/5}$ for any $x$ such that $D_{n^\zero}(x) \ge 4 T \kappa_0^{-1}$, where $\mu_{n^\zero}(x) = \min (|x|,|x-n^\zero|)$, $T = 4 \kappa_0 \log \delta_0^{-1}$,

$(ii)$
\begin{equation}\label{eq:11auxtrajectweight5NNNNN}
\begin{split}
\min (D_{n^\zero}(n_{i}), D_{n^\zero}(n_{j})) \le T \| (n_{i}, \dots, n_{j}) \|^{1/5} \\
\text{for any $i < j$ such that $\min (D_{n^\zero}(n_{i}), D_{n^\zero}(n_{j})) \ge 4 T \kappa_0^{-1}$}, \quad \text{unless $j = i + 1$}.
\end{split}
\end{equation}
Moreover,
\begin{equation}\label{eq:11auxtrajectweight5NNNNN1}
\begin{split}
\text{if $\min (D_{n^\zero}(n_{i}), D_{n^\zero}(n_{i+1})) > T |(n_{i} - n_{i+1})|^{1/5}$}, \quad \text{for some $i$, then }, \\
\min (D_{n^\zero}(n_{j'}), D_{n^\zero}(n_{i})) \le T \| (n_{j'}, \dots, n_{i}) \|^{1/5}, \quad \min (D_{n^\zero}(n_{i}), D_{n^\zero}(n_{j''})) \le T \| (n_{i}, \dots, n_{j''}) \|^{1/5}, \\
\min (D_{n^\zero}(n_{j'}), D_{n^\zero}(n_{i+1})) \le T \|(n_{j'}, \dots, n_{i+1}) \|^{1/5}, \quad \min (D_{n^\zero}(n_{i+1}), D_{n^\zero}(n_{j''})) \le T \| (n_{i+1}, \dots, n_{j''}) \|^{1/5}, \\
\text{for any $j' < i < i+1 < j''$.}
\end{split}
\end{equation}

$(2)$ Assume that for all $n \in \zv$, we have $|c(n)| \le \tilde \ve \exp (- \tilde \kappa |n|)$ with $\tilde \ve \le \ve_0$, $\tilde \kappa \ge \kappa_0$. Let $\gamma = (n_1, \dots, n_{k}) \in \Ga_{n^\zero} := \bigcup_{m,n} \Ga_{n^\zero}(m, n)$. Set $M = 4 T \kappa_0^{-1}$, $\bar D(\gamma) = \max_j D(n_j)$, $t_D(\gamma) := \frac{\log \bar D(\gamma)}{\log M}$, $\vartheta_t = \sum_{0 < s \le t} 2^{-5s}$. Then,
\begin{equation}\label{eq:11auxtrajectweight2}
w_{n^\zero}(\gamma) \le \begin{cases}  \tilde \ve^{k(\gamma)-1} \exp (- \tilde \kappa \| \gamma \| + k (\gamma) M^5) & \text{if $t_D(\gamma) \le 5$}, \\
\tilde \ve^{k(\gamma)-1} \exp (- \tilde \kappa (1 - \vartheta_{t_D(\gamma)+1}) \| \gamma \| + 2 \bar D(\gamma)) & \text{if $t_D(\gamma) > 5$}. \end{cases}
\end{equation}
Furthermore, $\bar D(\gamma) \le 2 T [\min (|n_1|,|n^\zero-n_k|)^{1/5} + \| \gamma \|^{1/5}]$.
\end{lemma}

In the next lemma we establish an estimate similar to \eqref{eq:11auxtrajectweight2} under a slightly weaker condition on $|c(n)|$, and also an estimate for the sum of such terms.

\begin{lemma}\label{lem:11scaleddecayestimate}
Let $\tilde \ve \le \ve_0$, $\tilde \kappa \ge 2 \kappa_0$, $R_1 \ge 2^{30} (\kappa_0^{-1}T)^2$. Set $R_t = 5 R_{t-1}/4$, $\rho_{t-1} = 2^{-10} t^{-2}$, $t = 2,\dots$, $\sigma_t = \sum_{1 \le \ell \le t} \rho_\ell$. Assume
\begin{equation} \label{eq:11scaledconditionaldecay}
|c(p)| \le \begin{cases} \tilde \ve \exp ( - \tilde \kappa |p|) & \text{if $ 0 < |p| \le R_2$}, \\ \tilde \ve \exp ( - \frac{15}{16} (1 - \sigma_{3t}) \tilde \kappa |p|) & \text{if $R_{t-1} < |p| \le R_t$, $3 \le t \le t_0$.}\end{cases}
\end{equation}
For $t \ge 1$, let $\Gamma^{(t)}_{n^\zero}$ be the set of trajectories $\gamma = (n_1, \dots, n_{k}) \in \Ga_{n^\zero}$ with $\| \gamma \| \le 2 R_t$ and $\max_j |n_{j+1} - n_j| \le R_{t+1}$. Then, for any $\gamma \in \Gamma^{(t)}_{n^\zero}$ with $t \le t_0-1$, we have
\begin{equation} \label{eq:11gamscaledconditionaldecay}
w_{n^\zero}(\gamma) \le \tilde \ve^{k(\gamma)-1} \exp \Big( -\frac{15}{16} (1 - \sigma_{3t+4}) \tilde \kappa \| \gamma \| + 2 \bar D(\gamma) + k(\gamma)M \Big).
\end{equation}
Furthermore,
\begin{equation} \label{eq:11gamscaledconditionaldecaysum1statement}
\sum_{\gamma \in \Ga_{n^\zero}(m, n) : k(\gamma) \ge 2, \| \gamma \| \le R_{t_0}} w_{n^\zero}(\gamma) \le \exp(- 2 T (\min (|m|,|n^\zero-n|)^{1/5}) \exp \Big( -\frac{15}{16} (1 - \sigma_{3t_0+2}) \tilde \kappa |n-m| \Big).
\end{equation}
\end{lemma}

\begin{proof}
The proof of \eqref{eq:11gamscaledconditionaldecay} goes by induction in $t = 1, 2, \dots$. Let $\gamma = (n_1,\dots,n_{k}) \in \Gamma^{(1)}_{n^\zero}$. Then, in particular, $\max_j |n_{j+1}-n_j| \le R_{2}$. Due to \eqref{eq:11scaledconditionaldecay}, one has $w(n_j,n_{j+1}) \le \tilde \ve \exp(-\tilde \kappa |n_j-n_{j+1}|)$. Hence \eqref{eq:11auxtrajectweight2} applies. Note that $1 - \vartheta_{t_D(\gamma)+1} > 15/16$. This implies \eqref{eq:11gamscaledconditionaldecay} for $t = 1$ in both cases in \eqref{eq:11auxtrajectweight2}.

Let $\Gamma^{(t)}_{n^\zero,0}$ be the set of trajectories $\gamma = (n_1,\dots,n_{k}) \in \Gamma^{(t)}_{n^\zero}$ with $\max_j |n_{j+1}-n_j| \le R_t$, $\Gamma^{(t)}_{n^\zero,1} = \Gamma^{(t)}_{n^\zero} \setminus \Gamma^{(t)}_{n^\zero,0}$.

Let $\gamma = (n_1,\dots,n_{k}) \in \Gamma^{(t)}_{n^\zero,1}$. Then there exists $j_0$ such that $|n_{j_0+1} - n_{j_0}| > R_t$. Note that $|n_{j+1} - n_{j}| < R_t$ for any $j \neq j_0$, since $\| \gamma \| \le 2 R_t$. Let $\gamma_1 = (n_1,\dots,n_{j_0})$, $\gamma_2 = (n_{j_0+1},\dots,n_k)$. Note that $\| \gamma_1 \| + \| \gamma_2 \| < R_t < 2 R_{t-1}$ since $\| \gamma \| \le 2 R_t$ and $|n_{j_0+1} - n_{j_0}| > R_t$. Therefore $\gamma_1,\gamma_2 \in \Gamma^{(t-1)}_{n^\zero}$. Hence, the inductive assumption applies,
\begin{equation} \label{eq:11gamscaledconditionaldecay2}
\begin{split}
w_{n^\zero}(\gamma_i) \le \tilde \ve^{k(\gamma_i)-1} \ve \exp \Big(-\frac{15}{16} (1 - \sigma_{3t+1}) \tilde \kappa \| \gamma_i \| + 2 \bar D(\gamma_i) + k(\gamma_i)M \Big), \quad i=1,2, \\
w_{n^\zero}(\gamma) = w_{n^\zero}(\gamma_1) |c(n_{j_0+1} - n_{j_0})| w_{n^\zero}(\gamma_2) \\
\le \tilde \ve^{k(\gamma)-1} \exp \Big( -\frac{15}{16} (1 - \sigma_{3t+1}) \tilde \kappa (\|\gamma_1\| + \|\gamma_2\|) -\frac{15}{16} (1 - \sigma_{3t+3}) \tilde \kappa |n_{j_0+1} - n_{j_0}| + 2 \bar D(\gamma_1) + 2 \bar D(\gamma_2) + k(\gamma)M \Big).
\end{split}
\end{equation}
Let $D(n_{j_i}) = \bar D(\gamma_i)$, $i=1,2$. We have the following cases:

$(a)$ Assume $j_1 < j_0$. In this case due to \eqref{eq:11auxtrajectweight5NNNNN}, one has
\begin{equation}\label{eq:11auxtrajectweight5NUPP}
2 \min (\bar D(\gamma_1), \bar D(\gamma_2)) \le 2 T \|\gamma \|^{1/5} < \frac{\rho_{3t+4}}{4} \tilde \kappa (\|\gamma_1\| + \|\gamma_2\| + |n_{j_0+1} - n_{j_0}|)
\end{equation}
since $\|\gamma\| > R_t \ge 2^{30} (\kappa_0^{-1}T)^2 (5/4)^{t-1}$, $t \ge 2$. Combining \eqref{eq:11gamscaledconditionaldecay2} with \eqref{eq:11auxtrajectweight5NUPP}, one obtains \eqref{eq:11gamscaledconditionaldecay}.

$(b)$ Assume $j_0+1 < j_2$. Similarly to case $(a)$, one verifies \eqref{eq:11gamscaledconditionaldecay}.

$(c)$ Assume $j_1 = j_0$, $j_0+1 = j_2$. Let $\gamma'_1 = (n_1,\dots,n_{j_0-1})$, $\gamma'_2 = (n_{j_0+2},\dots,n_k)$. Once again, applying the inductive assumption, one obtains
\begin{equation} \label{eq:11gamscaledconditionaldecay7}
\begin{split}
w_{n^\zero}(\gamma'_i) \le \tilde \ve^{k(\gamma_i)-1} \exp \Big( -\frac{15}{16} (1 - \sigma_{3t+1}) \tilde \kappa \|\gamma'_i\| + 2 \bar D(\gamma'_i) + k(\gamma_i)M \Big), \quad i=1,2, \\
w_{n^\zero}(\gamma) = w_{n^\zero}(\gamma'_1) \exp(D_{n^\zero}(n_{j_0})) |c(n_{j_0+1} - n_{j_0})| \exp(D_{n^\zero} (n_{j_0+1})) w_{n^\zero}(\gamma'_2) \\
\le \tilde \ve^{k(\gamma)-1} \exp \Big(-\frac{15}{16} (1 - \sigma_{3t+1}) \tilde \kappa (\|\gamma_1\| + \|\gamma_2\|) -\frac{15}{16} (1 - \sigma_{3t+3}) \tilde \kappa |n_{j_0+1} - n_{j_0}| \Big) \times \\
\exp(2 \bar D(\gamma_1) + 2 \bar D(\gamma_2) + D_{n^\zero}(n_{j_0}) + D_{n^\zero}(n_{j_0+1}) + k(\gamma)M).
\end{split}
\end{equation}
One has $2 D_{n^\zero}(\gamma'_1) = 2 \min (D_{n^\zero}(\gamma'_1), D_{n^\zero}(n_{j_0})) < 2 T \|\gamma\|^{1/5} < \rho_{3t+4} \tilde \kappa (\|\gamma_1\| + \|\gamma_2\| + |n_{j_0+1} - n_{j_0}| )/4$. Similarly, $2 D_{n^\zero}(\gamma'_2) < \rho_{3t+4} \tilde \kappa (\|\gamma_1\| + \|\gamma_2\| + |n_{j_0+1} - n_{j_0}|)/4$. Therefore, \eqref{eq:11gamscaledconditionaldecay} follows from \eqref{eq:11gamscaledconditionaldecay7}.

Thus, \eqref{eq:11gamscaledconditionaldecay} holds for $\gamma \in \Gamma^{(t)}_{n^\zero,1}$ in any event. Let $\gamma = (n_1,\dots,n_{k}) \in \Gamma^{(t)}_{n^\zero,0}$. Assume $\|(n_1,\dots,n_{k})\| \le 2 R_{t-1}$. Recall that $\max_j |n_{j+1} - n_j| \le R_t$ since $\gamma = (n_1,\dots,n_{k}) \in \Gamma^{(t)}_{n^\zero,0}$. Hence, in this case the inductive assumption applies and even a stronger estimate than \eqref{eq:11gamscaledconditionaldecay} holds. Assume $\|(n_1,\dots,n_{k})\| > 2 R_{t-1}$. Then there exists $j_0$ such that $\|(n_1,\dots,n_{j_0})\| \le 2 R_{t-1}$, $\|(n_1,\dots,n_{j_0+1})\| > 2 R_{t-1}$. Let $\gamma_1 = (n_1,\dots,n_{j_0})$, $\gamma_2 = (n_{j_0+1},\dots,n_k)$. Note that $\|\gamma_2\| = \|\gamma\| - \|(n_1,\dots,n_{j_0+1})\| < 2 R_t - 2 R_{t-1} < 2 R_{t-1}$ since $R_t < 2 R_{t-1}$. Therefore $\gamma_1,\gamma_2 \in \Gamma^{(t-1)}_{n^\zero}$. Hence, the inductive assumption applies,
\begin{equation} \label{eq:11gamscaledconditionaldecay10}
\begin{split}
w_{n^\zero}(\gamma_i) \le \tilde \ve^{k(\gamma_i)-1} \exp \Big( -\frac{15}{16} (1 - \sigma_{3t+1}) \tilde \kappa \|\gamma_i\| + 2 \bar D(\gamma_i) + k(\gamma_i) M \Big), \quad i=1,2, \\
w_{n^\zero}(\gamma) = w_{n^\zero}(\gamma_1) |c(n_{j_0+1} - n_{j_0})| w_{n^\zero}(\gamma_2) \\
\le \tilde \ve^{k(\gamma)-1} \exp \Big( -\frac{15}{16} (1 - \sigma_{3t+1}) \tilde \kappa (\|\gamma_1\| + \|\gamma_2\|) - \frac{15}{16} (1 - \sigma_{3t+3}) \tilde \kappa |n_{j_0+1} - n_{j_0}| + 2 \bar D(\gamma_1) + 2 \bar D(\gamma_2) + k(\gamma)M \Big).
\end{split}
\end{equation}
Let $D(n_{j_i}) = \bar D(\gamma_i)$, $i=1,2$. We have the following cases:

$(\alpha)$ Assume $j_1 < j_0$. In this case, due to \eqref{eq:11auxtrajectweight5NNNNN}, one has
\begin{equation}\label{eq:11auxtrajectweight5NUPPal}
2 \min (\bar D(\gamma_1), \bar D(\gamma_2)) \le 2 T \| \gamma \|^{1/5} < \frac{\rho_{3t+4}}{4} \tilde \kappa (\|\gamma_1\| + \|\gamma_2\| + |n_{j_0+1} - n_{j_0}|).
\end{equation}
Combining \eqref{eq:11gamscaledconditionaldecay10} with \eqref{eq:11auxtrajectweight5NUPPal}, one obtains \eqref{eq:11gamscaledconditionaldecay}.

$(\beta)$ Assume $j_0+1 < j_2$. Similarly to case $(\alpha)$, one verifies \eqref{eq:11gamscaledconditionaldecay}.

$(\gamma)$ Assume $j_1 = j_0$, $j_0+1 = j_2$. Let $\gamma'_1 = (n_1,\dots,n_{j_0-1})$, $\gamma'_2 = (n_{j_0+2},\dots,n_k)$. Once again, applying the inductive assumption, one obtains
\begin{equation} \label{eq:11gamscaledconditionaldecay7be}
\begin{split}
w_{n^\zero}(\gamma'_i) \le \tilde \ve^{k(\gamma'_i)-1} \exp \Big(-\frac{15}{16} (1 - \sigma_{3t+1}) \tilde \kappa \|\gamma'_i\| + 2 \bar D(\gamma'_i) + k(\gamma_i') M \Big), \quad i=1,2, \\
w_{n^\zero}(\gamma) = w_{n^\zero}(\gamma'_1) |c(n_{j_0-1} - n_{j_0})|  \exp(D_{n^\zero}(n_{j_0})) |c(n_{j_0+1} - n_{j_0})| \exp(D_{n^\zero} (n_{j_0+1})) \\
|c(n_{j_0+2} - n_{j_0+1})| w_{n^\zero}(\gamma'_2) \le \tilde \ve^{k(\gamma)-1} \exp(-\frac{15}{16} (1 - \sigma_{3t+1}) \tilde \kappa (\|\gamma'_1\| + \|\gamma'_2\|)) \\
\times \exp(- \frac{15}{16} (1 - \sigma_{3t+3}) \tilde \kappa (|n_{j_0} - n_{j_0-1}| + |n_{j_0+1} - n_{j_0}| + |n_{j_0+2} - n_{j_0+1}|)) \times \\
\exp(2 \bar D(\gamma'_1) + 2 \bar D(\gamma'_2) +  D_{n^\zero}(n_{j_0}) +  D_{n^\zero}(n_{j_0+1}) + k(\gamma)M).
\end{split}
\end{equation}
One has $2 \bar D_{n^\zero}(\gamma'_1) = 2 \min (\bar D_{n^\zero}(\gamma'_1), D_{n^\zero}(n_{j_0})) < 2 T \|\gamma\|^{1/5} < \rho_{3t+4} \tilde \kappa (\|\gamma'_1\| + \|\gamma'_2\| + |n_{j_0+2} - n_{j_0-1}| )/4$. Similarly, $2 \bar D_{n^\zero}(\gamma'_2) < \rho_{3t+4} \tilde \kappa (\|\gamma_1\| + \|\gamma_2\| + |n_{j_0+1} - n_{j_0}| )/4$. Therefore, \eqref{eq:11gamscaledconditionaldecay} follows from \eqref{eq:11gamscaledconditionaldecay7be}.

Thus \eqref{eq:11gamscaledconditionaldecay} holds in any event. Recall that $\bar D(\gamma) \le 2 T [(\min (|m|, |n^\zero-n|)^{1/5} + \|\gamma\|^{1/5}]$, $\gamma \in \Ga_{n^\zero}(m, n)$; see Lemma~\ref{lem:11.sumsproperties}. Set $w'_{n^\zero}(\gamma) = \exp(- 2 T (\min (|m|,|n^\zero-n|)^{1/5}) w_{n^\zero}(\gamma)$. Note that $2 T \|\gamma\|^{1/5} < \rho_{3t+5} \tilde \kappa \|\gamma\|/4$ if $\|\gamma\| \ge R_t$. Recall also the elementary estimate of Lemma~\ref{lem:2gammasum}: for any $\alpha, k > 0$,
\begin{equation}\label{eq:11auxtrajectweight21a}
\sum_{\gamma \in \Ga(m,n;k,\zv)} \exp(-\alpha \|\gamma\|) < (8 \alpha^{-1})^{(k-1)\nu}.
\end{equation}
Set $\Ga^{(t)}_{n^\zero}(m, n) = \Ga_{n^\zero}(m, n)\cap\Ga^{(t)}_{n^\zero}$. Note that if $\gamma \in \Ga_{n^\zero}$ and $\|\gamma\| \le R_{t+1}$, then $\gamma \in \Ga^{(t)}_{n^\zero}$ and \eqref{eq:11gamscaledconditionaldecay} applies. Finally, $\|\gamma\| \ge |n-m|$ for any $\gamma \in \Ga_{n^\zero}(m, n)$. Taking all that into account, one obtains
\begin{equation} \label{eq:11gamscaledconditionaldecaysum1}
\begin{split}
\sum_{\gamma \in \Ga_{n^\zero}(m, n) : k(\gamma) \ge 2, \|\gamma\| \le R_{t_0}} w'_{n^\zero}(\gamma) \le \sum_{t \le t_0-1} \sum_{\gamma \in \Ga_{n^\zero}(m, n), k(\gamma) \ge 2, R_t \le \|\gamma\| \le R_{t+1}} w'_{n^\zero}(\gamma) \\
\le e^M \sum_{t \le t_0-1} \sum_{\gamma \in \Ga_{n^\zero}(m, n),k(\gamma) \ge 2, R_t \le \|\gamma\| \le R_{t+1}} (e^M \tilde \ve)^{k(\gamma)-1} \exp(-\frac{15}{16} (1 - \sigma_{3t+4}) \tilde \kappa \|\gamma\| + \rho_{3t+5} \tilde \kappa \|\gamma\|/4) \\
\le e^M \sum_{k \ge 1} (e^M \tilde \ve)^{k-1} \exp(-\frac{15}{16} (1 - \sigma_{3t_0+2}) \tilde \kappa |n-m|) \sum_{k(\gamma) = k} \exp(-\rho_{3t_0+2} \tilde \kappa \|\gamma\|/4) \\
\le e^M \exp(-\frac{15}{16} (1 - \sigma_{3t_0+2}) \tilde \kappa |n-m|) \sum_{k \ge 1} (e^M \tilde \ve)^{k-1}(8 \alpha^{-1})^{(k-1)\nu} \le \exp(-\frac{15}{16} (1 - \sigma_{3t_0+2}) \tilde \kappa |n-m|).
\end{split}
\end{equation}
Here, in the last step, $\alpha = \rho_{3t+5}\tilde \kappa/4$, and we have used $\tilde \ve < \ve_0$.
\end{proof}

\begin{proof}[Proof of part $(2)$ of Theorem~B]
Set $R_1 = 2^{30}(\kappa_0^{-1}T)^2$, $R_t = 5R_{t-1}/4$, $\rho_{t-1} = 2^{-10} t^{-2}$, $t = 2,\dots$, $\sigma_t = \sum_{1 \le \ell \le t} \rho_\ell$ as in Lemma~\ref{lem:11scaleddecayestimate}. Set also $\ve^\zero_0 := \exp(-2 R_2)$. One can see that $\ve^\zero_0 < \ve_0^4$. Assume that
\begin{equation} \label{eq:11theoremBgapcondition}
E^+(\La;k_{n^\zero}) - E^-(\La;k_{n^\zero}) \le \ve^\zero \exp(-\kappa^\zero |n^\zero|), \quad \text{ for all $n^\zero \in \zv \setminus \{0\}$},
\end{equation}
where $\ve^\zero < \ve^\zero_0$, $\kappa^\zero > 4 \kappa_0$. We re-write \eqref{eq:11-10acbasicfunctionsBineq1} in the following form,
\begin{equation} \label{eq:11-10acbasicfunctionsBineq2}
|c(n^\zero)| \le (\ve^\zero)^{3/4} \exp(-3 \kappa^\zero |n^\zero|/4) + \sum_{m', n' \in \La(n^\zero)} |c(m')| s(n^\zero;m',n') |c(n'- n^\zero)|.
\end{equation}

Assume that with some $(\ve^\zero)^{1/2} < \hat \ve \le \ve^\zero_0$ and $\kappa_0 \le \hat \kappa \le \kappa^\zero/2$, we have $|c(p)|\le \hat \ve \exp(-\hat \kappa |p|)$ for $|p| > 0$. Set $\tilde \ve = \hat \ve/2$, $\tilde \kappa = 7 \hat \kappa/6$. We claim that in this case, in fact,
\begin{equation} \label{eq:11scaledconditionaldecayY}
|c(p)| \le \begin{cases} \tilde \ve \exp(-\tilde \kappa |p|) & \text {if $0 < |p| \le R_2$}, \\ \tilde \ve \exp(-\frac{15}{16} (1 - \sigma_{3t}) \tilde \kappa |p|) & \text {if $R_{t-1} < |p| \le R_t$, $t \ge3$}.
\end{cases}
\end{equation}
It is important to note here that $\frac{15}{16} (1 - \sigma_{3t}) \tilde \kappa > (\frac{15}{16})^2 \frac{7}{6} \hat \kappa := L \hat \kappa$, with $L > 1$. This allows one to iterate the argument and Theorem~B follows.

The verification of the claim goes by induction in $t$, starting with the first line in \eqref{eq:11scaledconditionaldecayY}, and then with the help of Lemma~\ref{lem:11scaleddecayestimate}. The idea is to run $n^\zero$ in \eqref{eq:11-10acbasicfunctionsBineq2} and to combine the inequalities which we have for different $n^\zero$. To this end it is convenient to replace $n^\zero$ in the notation. To verify the first line in \eqref{eq:11scaledconditionaldecayY}, we invoke \eqref{eq:11gamscaledconditionaldecaysum1statement} from Lemma~\ref{lem:11scaleddecayestimate} with $\hat \ve$ in the role of $\tilde \ve$ and $\hat \kappa$ in the role of $\tilde \kappa$. Note that condition \eqref{eq:11scaledconditionaldecay} of Lemma~\ref{lem:11scaleddecayestimate} holds for any $t$ for trivial reasons. So,
\begin{equation}\label{eq:11auxtrajectweightO5B}
\begin{split}
\sum_{m,n \in \La_{p}} |c(m)| s(p;m,n) |c(p-n)| \\
\le \hat \ve^2 \sum_{m,n \in \La_{p}} \exp(-\hat \kappa |m|) \exp \Big( -\frac{15}{16} (1 - \sigma_{3t_0+2}) \hat \kappa |n-m| + 2 T (\min(|m|,|n-p|))^{1/5} \Big) \exp(-\hat \kappa |p-n|).
\end{split}
\end{equation}
Using the elementary estimates of Lemma~\ref{lem:2gammasum}, one obtains from \eqref{eq:11auxtrajectweightO5B} that
\begin{equation}\label{eq:11auxtrajectweightO5B1}
\sum_{m,n \in \La_{p}} |c(m)| s(p;m,n) |c(p-n)| \le \hat \ve^{3/2}/4 \le \tilde \ve (\ve^\zero_0)^{1/2}/2 < (\tilde \ve /2) \exp(-R_2).
\end{equation}
It follows from \eqref{eq:11-10acbasicfunctionsBineq2} combined with \eqref{eq:11auxtrajectweightO5B1} that for any $|p| > 0$, we have
\begin{equation}\label{eq:11auxtrajectweightO5B2}
|c(p)| < (\ve^\zero)^{3/4} \exp(-3 \kappa^\zero |p|/4) + (\tilde \ve/2) \exp(-R_2) < \tilde \ve \exp(-\tilde \kappa R_2).
\end{equation}
This verifies the first line in \eqref{eq:11scaledconditionaldecay}.

Assume now that for some $\ell \ge 2$, \eqref{eq:11scaledconditionaldecay} holds for any $0 < |p| \le R_{t}$ and any $t \le \ell$. Let $|q| > R_\ell$ be arbitrary. For $t \ge 1$, let $\Gamma^{(t)}_{q}$ be the set of trajectories $\gamma = (n_1,\dots,n_{k}) \in \Ga_{q}$ with $\|\gamma\| \le 2 R_t$ and $\max_j |n_{j+1} - n_j| \le R_{t+1}$. Let $\Ga^{(t)}_q(m,n) = \Ga^{(t)}_{q} \cap \Ga_{q}(m,n)$. We have
\begin{equation}\label{eq:11auxtrajectweightO5}
\begin{split}
\sum_{m,n \in \La_{q}} |c(m)| s(q;m,n) |c(q-n)| \le \sum_{m,n} |c(m)| |c(q-n)| \sum_{\gamma \in \Ga_{q}(m,n)} w_{q} (\gamma) \\
\le \Sigma_1 + \Sigma_2 + \Sigma_3 := \sum_{m,n : |m|,|n-q| \le R_\ell} |c(m)| |c(q-n)| \sum_{\gamma \in \Ga^{(\ell-1)}_{q}(m,n)} w_{q} (\gamma) \\
\sum_{m,n} |c(m)| |c(q-n)| \sum_{\gamma \in \Ga_{q}(m,n), \quad \|\gamma\| > 2 R_{\ell-1}} w_{q} (\gamma) + \Sigma_3.
\end{split}
\end{equation}
Here the sum $\Sigma_3$ is over the cases when $\|\gamma\| \le 2 R_{\ell-1}$ and either $\max (|m|,|q-n|) > R_\ell$ or $\gamma = (n_1,\dots,n_{k})$ obeys $\max_j |n_{j+1}-n_j| > R_{\ell}$, or both.

Using \eqref{eq:11gamscaledconditionaldecaysum1statement} from Lemma~\ref{lem:11scaleddecayestimate} with $\ell$ in the role of $t_0$ and the inductive assumption, one obtains
\begin{equation}\label{eq:11auxtrajectweightO6}
\begin{split}
\Sigma_1 \le \tilde \ve^2 \sum_{m,n} \exp \Big( -\frac{15}{16} (1 - \sigma_{3 \ell}) \tilde \kappa |m| \Big) \exp \Big( -\frac{15}{16} (1 - \sigma_{3 \ell + 2}) \tilde \kappa |m-n| + 2 T (\min(|m|,|q-n|))^{1/5} \Big) \times \\
\exp \Big( -\frac{15}{16} (1 - \sigma_{3\ell}) \tilde \kappa|q-n| \Big).
\end{split}
\end{equation}
Note that $2 T (\min(|m|,|q-n|))^{1/5}) < \frac{1}{4} \rho_{3 \ell + 3} (|m| + |m-n| + |q-n|)$ since $|q| > R_\ell$. Estimating the sum in \eqref{eq:11auxtrajectweightO6}, one obtains $($ see Lemma~\ref{lem:2gammasum} $)$
\begin{equation}\label{eq:11auxtrajectweightO6a}
\Sigma_1 \le \tilde \ve^{3/2} \exp \Big( -\frac{15}{16} (1 - \sigma_{3\ell+2} - \frac{1}{4} \rho_{3\ell+3}) \tilde \kappa |q| \Big).
\end{equation}

To estimate the sum $\Sigma_2$, we use Lemma~\ref{lem:11scaleddecayestimate} with $\hat \ve$ in the role of $\tilde \ve$ and $\hat \kappa$ in the role of $\tilde\kappa$:
\begin{equation}\label{eq:11auxtrajectweightO15}
\begin{split}
\Sigma_2 \le \sum_{m,n} |c(m)| |c(q-n)| \sum_{t \ge \ell-1} \sum_{\gamma \in \Ga_{q}(m, n), k(\gamma) \ge 2, R_t \le \|\gamma\| \le R_{t+1}} w_{q}(\gamma) \le \sum_{m,n} |c(m)| |c(q-n)| \times \\
\exp(2 T (\min(|m|,|q-n|))^{1/5}) e^M \big[ \sum_{\gamma \in \Ga_{q}(m, n),k(\gamma) \ge 2, 2 R_{\ell-1} \le \|\gamma\| \le R_{\ell}} + \sum_{t \ge \ell} \sum_{\gamma \in \Ga_{q}(m,n),k(\gamma) \ge 2, R_t \le \|\gamma\| \le R_{t+1}} \big] \\
(e^M \hat \ve)^{k(\gamma) - 1} \exp \Big( -\frac{15}{16} (1 - \sigma_{3t+4}) \hat \kappa \|\gamma\| + 2 T \|\gamma\|^{1/5} \Big) \le \hat \ve^2 \sum_{m,n} \exp(-\hat \kappa (|m| + |q-n|) + 2 T (\min(|m|,|q-n|))^{1/5}) \times \\
\exp \Big( -\frac{15}{16} (1 - \sigma_{3\ell+2}) \hat \kappa \times (2 R_{\ell-1}) \Big) \le \hat \ve^{3/2} \exp \Big( -\frac{15}{16} (1 - \sigma_{3\ell+2} - \frac{1}{4} \rho_{3\ell+3}) \hat \kappa \times (2R_{\ell-1}) \Big) \\
< \hat\ve^{3/2} \exp \Big( -\frac{15}{16} (1 - \sigma_{3\ell+2} - \frac{1}{4} \rho_{3\ell+3}) \tilde \kappa R_{\ell+1} \Big).
\end{split}
\end{equation}

Let us now estimate $\Sigma_3$. Given $r,s \in \La_q$ with $|s-r| > R_\ell$, denote by $\Ga_{q;r,s}$ the set of trajectories $\gamma = (n_1,\dots,n_{k}) \in \bar \Ga_{q}$ with $\|\gamma\| \le 2 R_{\ell-1}$ and such that
\begin{equation}\label{eq:11gammarsplit}
\gamma = \gamma' \cup \gamma'',
\end{equation}
where $r$ is the endpoint of $\gamma'$ and $s$ is the starting point of $\gamma''$. Note that since $\|\gamma\| \le 2 R_{\ell-1}$, one has $|n_{j+1} - n_j| \le R_\ell$ for all $j$ with one exception when $n_{j+1} = s$, $n_j = r$. In particular, the inductive assumption applies to $\gamma'$, $\gamma''$. Denote by $\Sigma'_3$ the part of sum $\Sigma_3$ with $\gamma \in \Ga_{q;r,s}$ and with $|m|, |q-n| \le R_\ell$. Then just as in the above derivations, one obtains
\begin{equation}\label{eq:11auxtrajectweightO16}
\Sigma'_3 \le \tilde \ve^{3/2} \sum_{|r-s| > R_\ell} \exp \Big( -\frac{15}{16} (1 - \sigma_{3\ell+2}) \tilde \kappa |r| \Big) |c(r-s)| \exp \Big( -\frac{15}{16} (1 - \sigma_{3 \ell + 2}) \tilde \kappa|q-s| \Big).
\end{equation}
The estimation of the rest of the sum $\Sigma_3$ is similar. One has
\begin{equation}\label{eq:11auxtrajectweightO17}
\begin{split}
\Sigma_3 \le \tilde \ve^{3/2} \sum_{|r-s| > R_\ell} \exp \Big( -\frac{15}{16} (1 - \sigma_{3\ell+2}) \tilde \kappa |r| \Big) |c(r-s)| \exp \Big( -\frac{15}{16} (1 - \sigma_{3 \ell + 2}) \tilde \kappa |q-s| \Big) \\
+ 2 \tilde \ve^{3/2} \sum_{|r| > R_\ell} |c(r)| \exp \Big( -\frac{15}{16} (1 - \sigma_{3\ell+2}) \tilde \kappa |q-r| \Big) \\
\le \tilde \ve^{3/2} \sum_{R_\ell < |r-s| \le R_{\ell+1}} \exp \Big( -\frac{15}{16} (1 - \sigma_{3\ell+2}) \tilde \kappa|r| \Big) |c(r-s)| \exp \Big( -\frac{15}{16} (1 - \sigma_{3 \ell + 2}) \tilde \kappa |q-s| \Big) \\
+ 2 \tilde \ve^{3/2} \sum_{R_\ell < |r| \le R_{\ell+1}} |c(r)| \exp \Big( -\frac{15}{16} (1 - \sigma_{3 \ell + 2}) \tilde \kappa |q-r| \Big) \\
\hat \ve^{3/2} \exp \Big( -\frac{15}{16} (1 - \sigma_{3 \ell + 2} - \frac{1}{4} \rho_{3 \ell + 3}) \tilde \kappa R_{\ell + 1} \Big).
\end{split}
\end{equation}

Now we invoke \eqref{eq:11-10acbasicfunctionsBineq2}. For $|q| > R_\ell$, one obtains
\begin{equation} \label{eq:11-10acbasicfunctionsBineq2F1}
\begin{split}
|c(q)| \le (\ve^\zero)^{3/4} \exp(- \kappa^\zero |q|/2) + \sum_{1 \le i \le 3} \Sigma_i \le (\ve^\zero)^{3/4} \exp(-\tilde\kappa|q|) \\
+ \tilde \ve^{3/2} \exp \Big( -\frac{15}{16} (1 - \sigma_{3 \ell + 2} - \frac{1}{4} \rho_{3 \ell + 3}) \tilde \kappa |q| \Big) + 2 \hat \ve^{3/2} \exp \Big( -\frac{15}{16} (1 - \sigma_{3 \ell + 2} - \frac{1}{4} \rho_{3 \ell + 3}) \tilde \kappa R_{\ell + 1} \Big) \\
+ \tilde \ve^{3/2} \sum_{R_\ell < |r-s| \le R_{\ell + 1}} \exp \Big( -\frac{15}{16} (1 - \sigma_{3 \ell + 2}) \tilde \kappa |r| \Big) |c(r-s)| \exp \Big( -\frac{15}{16} (1 - \sigma_{3 \ell + 2}) \tilde \kappa |q-s| \Big) \\
+ 2 \tilde \ve^{3/2} \sum_{R_\ell < |r| \le R_{\ell + 1}} |c(r)| \exp \Big( -\frac{15}{16} (1 - \sigma_{3 \ell + 2}) \tilde \kappa |q-r| \Big).
\end{split}
\end{equation}

Here we have replaced $\kappa^\zero/2$ by $\tilde \kappa < \kappa^\zero/2$. Now we consider $R_\ell < |q| \le R_{\ell+1}$. We replace $R_{\ell+1}$ in the exponent by a smaller quantity $|q|$ and we obtain a self-contained system of inequalities for $|c(q)|$ with $R_\ell < |q| \le R_{\ell+1}$. This allows us to iterate  \eqref{eq:11-10acbasicfunctionsBineq2F1}. It is convenient to replace the multiple sums via summation over trajectories $\gamma = (n_0,\dots,n_k) \in \Ga(0,q)$. Set
\begin{equation} \label{eq:11-10acbasicfunctionsBineq2F2}
\begin{split}
\ve' = \tilde \ve^{1/4}, \quad \kappa'= \frac{15}{16} (1 - \sigma_{3\ell+2} - \frac{1}{4} \rho_{3\ell+3}) \tilde \kappa, \quad w'(m,n) = \ve' \exp(-\kappa' |m-n|), \\
w'((n_0,\dots,n_k) = \prod w'(n_j,n_{j+1}).
\end{split}
\end{equation}
Iterating \eqref{eq:11-10acbasicfunctionsBineq2F1} $N$ times, one obtains
\begin{equation} \label{eq:11-10acbasicfunctionsBineq2F3}
\begin{split}
|c(q)| \le \tilde \ve^{3/2} (\sum_{0 \le k \le N} 4^k \ve'^k) \exp \Big( -\frac{15}{16} (1 - \sigma_{3 \ell + 2} - \frac{1}{4} \rho_{3 \ell + 3}) \tilde \kappa R_{\ell + 1} \Big) \\
+ \tilde \ve \sum_{1 \le k \ge 3N} 4^{k} \ve'^k \sum_{\gamma \in \Ga(0,q) : k(\gamma) = k} w'(\gamma) + 4^N \ve'^N.
\end{split}
\end{equation}
Taking here $N$ large enough and evaluating the sums over $\gamma$ as before, one obtains \eqref{eq:11scaledconditionaldecay}.
\end{proof}

\end{document}